\documentclass[11pt]{amsart}
\usepackage{amsfonts}
\usepackage{fullpage, amsmath, amsthm,amsfonts,amssymb,stmaryrd, mathrsfs,amscd}
\usepackage{graphicx}
\usepackage{hyperref}
\makeatletter
\usepackage[pdftex]{color}
\usepackage{rotating}
\usepackage[all]{xy}
\usepackage{relsize}

\newtheorem{thm}[subsubsection]{Theorem}
\newtheorem{prop}[subsubsection]{Proposition}
\newtheorem{lem}[subsubsection]{Lemma}

\newtheorem{lem-def}[subsubsection]{Lemma-Definition}
\newtheorem{cor}[subsubsection]{Corollary}
\newtheorem{theorem}[subsubsection]{Theorem}
\newtheorem{proposition}[subsubsection]{Proposition}
\newtheorem{lemma}[subsubsection]{Lemma}

\newtheorem{conjecture}[subsubsection]{Conjecture}

\theoremstyle{definition}
\newtheorem{example}[subsubsection]{Example}
\newtheorem{remark}[subsubsection]{Remark}
\theoremstyle{definition}
\newtheorem{hypothesis}[subsubsection]{Hypothesis}
\newtheorem{definition}[subsubsection]{Definition}

\newtheorem{ex}[subsubsection]{Example}

\newtheorem{notation}[subsubsection]{Notation}
\newtheorem{construction}[subsubsection]{Construction}

\newtheorem{convention}[subsubsection]{Convention}
\newtheorem{rmk}[subsubsection]{Remark}
\theoremstyle{definition}
\newtheorem{dfn}[subsubsection]{Definition}

\numberwithin{equation}{subsection}

\input xy
\xyoption{all}

\newcommand{\quash}[1]{}  
\newcommand{\nc}{\newcommand}
\nc{\on}{\operatorname}
\newcommand{\bbzeta}{\boldsymbol{\zeta}}

\DeclareFontEncoding{OT2}{}{} 

\newcommand{\res}{\mathrm{res}}

\newcommand{\cl}{\mathrm{cl}}

\DeclareMathOperator{\Ker}{Ker}
\def\bbb{\mathbf{b}}
\def\bba{\mathbf{a}}
\def\bbc{\mathbf{c}}
\def\bbe{\mathbf{e}}
\def\AAA{\mathbb{A}}
\def\BB{\mathbb{B}}
\def\CC{\mathbb{C}}
\def\DD{\mathbb{D}}

\def\FF{\mathbb{F}}
\def\GG{\mathbb{G}}

\def\LL{\mathbb{L}}

\def\NN{\mathbb{N}}

\def\QQ{\mathbb{Q}}
\def\RR{\mathbb{R}}
\def\SSS{\mathbb{S}}

\def\XX{\mathbb{X}}

\def\ZZ{\mathbb{Z}}

\def\bfB{\mathbf{B}}

\def\calA{\mathcal{A}}

\def\calE{\mathcal{E}}
\def\calF{\mathcal{F}}

\def\calH{\mathcal{H}}

\def\calL{\mathcal{L}}

\def\calO{\mathcal{O}}

\def\calV{\mathcal{V}}

\def\bfJ{\boldsymbol{J}}

\newcommand{\Corr}{\mathrm{Corr}}

\newcommand{\fraka}{{\mathfrak a}}

\newcommand{\frakg}{{\mathfrak g}}

\newcommand{\frakp}{{\mathfrak p}}

\newcommand{\scrC}{\mathscr{C}}
\newcommand{\scrS}{\mathscr{S}}

\newcommand{\frakH}{{\mathfrak H}}

\newcommand{\frakK}{{\mathfrak K}}

\newcommand{\frakM}{{\mathfrak M}}

\newcommand{\frakS}{{\mathfrak S}}

\newcommand{\Sat}{\mathrm{Sat}}

\newcommand{\rmP}{\mathrm{P}}
\newcommand{\der}{\mathrm{der}}

\newcommand{\bA}{{\mathbb A}}
\newcommand{\bB}{{\mathbb B}}
\newcommand{\bC}{{\mathbb C}}
\newcommand{\bD}{{\mathbb D}}

\newcommand{\bF}{{\mathbb F}}
\newcommand{\bG}{{\mathbb G}}
\newcommand{\bH}{{\mathbb H}}

\newcommand{\bL}{{\mathbb L}}
\newcommand{\bM}{{\mathbb M}}
\newcommand{\bN}{{\mathbb N}}

\newcommand{\bQ}{{\mathbb Q}}
\newcommand{\bR}{{\mathbb R}}
\newcommand{\bS}{{\mathbb S}}
\newcommand{\bT}{{\mathbb T}}

\newcommand{\bV}{{\mathbb V}}

\newcommand{\bX}{{\mathbb X}}

\newcommand{\bZ}{{\mathbb Z}}

\newcommand{\mA}{{\mathcal A}}

\newcommand{\mC}{{\mathcal C}}
\newcommand{\mD}{{\mathcal D}}
\newcommand{\mE}{{\mathcal E}}
\newcommand{\mF}{{\mathcal F}}

\newcommand{\mH}{{\mathcal H}}

\newcommand{\mL}{{\mathcal L}}

\newcommand{\mN}{{\mathcal N}}
\newcommand{\mO}{{\mathcal O}}
\newcommand{\mP}{{\mathcal P}}

\newcommand{\mT}{{\mathcal T}}

\newcommand{\mV}{{\mathcal V}}

\def\rmH{\mathrm{H}}

\newcommand{\leftone}{\leftarrow}
\newcommand{\rightone}{\to}

\newcommand{\Rep}{\mathrm{Rep}}
\nc{\Tate}{\mathrm{Tate}}
\newcommand{\bfSh}{\mathbf{Sh}}
\nc{\al}{{\alpha}} \nc{\be}{{\beta}} \nc{\ga}{{\gamma}}
\nc{\ve}{{\varepsilon}} \nc{\Ga}{{\Gamma}} \nc{\la}{{\lambda}}
\nc{\La}{{\Lambda}}

\nc{\ad}{{\on{ad}}}
\newcommand{\Ad}{{\on{Ad}}}
\nc{\aff}{{\on{aff}}}
\nc{\Aff}{{\mathbf{Aff}}}
\newcommand{\Aut}{{\on{Aut}}}
\nc{\Bun}{{\on{Bun}}}
\newcommand{\cha}{{\on{char}}}
\nc{\Coh}{{\on{Coh}}}

\nc{\diag}{{\on{diag}}}
\nc{\dR}{{\on{dR}}}
\newcommand{\End}{{\on{End}}}
\nc{\Fl}{{\calF\ell}}
\newcommand{\Gal}{{\on{Gal}}}
\newcommand{\Gr}{{\on{Gr}}}
\newcommand{\Hom}{{\on{Hom}}}
\newcommand{\id}{{\on{id}}}
\nc{\Id}{{\on{Id}}}

\nc{\Ind}{{\on{Ind}}}
\nc{\inv}{{\on{Inv}}}
\nc{\Iso}{{\on{Isom}}}
\newcommand{\Lie}{{\on{Lie}}}
\nc{\Nm}{{\on{Nm}}}

\nc{\pf}{{\on{pf}}}
\newcommand{\pr}{{\on{pr}}}
\nc{\rec}{{\on{rec}}}
\nc{\Iw}{\on{Iw}}
\newcommand{\Res}{{\on{Res}}}
\newcommand{\s}{{\on{sc}}}

\newcommand{\Spec}{{\on{Spec}}}

\nc{\tr}{{\on{tr}}}
\newcommand{\loc}{\mathrm{loc}}
\newcommand{\Tr}{{\on{Tr}}}

\newcommand{\Min}{\mathsf{Min}}

\newcommand{\GL}{{\on{GL}}}
\nc{\GSp}{{\on{GSp}}} \nc{\GU}{{\on{GU}}} \nc{\SL}{{\on{SL}}}
\nc{\SU}{{\on{SU}}} \nc{\SO}{{\on{SO}}}

\nc{\Ql}{{\overline{\bQ}_\ell}}

\nc{\fg}{\frakg}
\nc{\fp}{\frakp}

\nc{\rat}{\overline{\bQ}}
\nc{\triv}{{\bf{1}}}

\newcommand{\longto}{\longrightarrow}

\newcommand{\std}{\mathrm{std}}
\newcommand{\et}{\mathrm{et}}

\newcommand{\PGL}{\on{PGL}}

\nc{\dm}{/\!\!/}

\def\xcoch{\mathbb{X}_\bullet}
\def\xch{\mathbb{X}^\bullet}

\nc{\wt}{\mathrm{wt}}
\nc{\Sh}{\on{Sh}}
\nc{\Sht}{\on{Sht}}
\nc{\wSh}{\widetilde{\Sht}}
\nc{\Sph}{\on{Sph}}
\nc{\Fr}{\on{Frob}}
\nc{\Fp}{{^\sigma\mE}}
\nc{\und}{\underline}
\nc{\mmu}{{\mu_\bullet}}
\nc{\nnu}{{\nu_\bullet}}
\nc{\Hk}{{\on{Hk}}}
\nc{\lhk}{\Hk^{\on{loc}}}
\newcommand{\IC}{\on{IC}}
\nc{\bp}{{\mathbf{p}}}
\newcommand{\cris}{\mathrm{cris}}
\nc{\MV}{{\bM\bV}}

\newcommand{\xdashrightarrow}[2][]{\ext@arrow 3359 \rightarrowfill@@{#1}{#2}}
\def\rightarrowfill@@{\arrowfill@@\relax\relbar\shortrightarrow}
\def\arrowfill@@#1#2#3#4{%
  $\m@th\thickmuskip0mu\medmuskip\thickmuskip\thinmuskip\thickmuskip
   \relax#4#1
   \xleaders\hbox{$#4#2$}\hfill
   #3$%
}

\begin{document}
\title{Cycles on Shimura varieties via Geometric Satake}
\author{Liang Xiao} 
\address{Liang Xiao, Department of Mathematics, University of Connecticut, Storrs, 341 Mansfield Road, Unit 1009, Storrs, CT 06269.}
\email{liang.xiao@uconn.edu}\date{\today}
\author{Xinwen Zhu}
\address{Xinwen Zhu, Department of Mathematics, Caltech, 1200 East California Boulevard, Pasadena, CA 91125.}
\email{xzhu@caltech.edu}\date{\today}
\begin{abstract}
We construct (cohomological) correspondences between mod $p$ fibers of \emph{different} Shimura varieties and describe the fibers of these correspondences by studying irreducible components of affine Deligne-Lusztig varieties. In particular, in the case of correspondences from a Shimura set to a Shimura variety, we obtain a description of the basic Newton stratum of the latter, and show that the irreducible components of the basic Newton stratum generate all the Tate classes in the middle cohomology of the Shimura variety, under a certain genericity condition. Along the way, we also determine the set of irreducible components of the affine Deligne-Lusztig variety associated to an unramified twisted conjugacy class.
\end{abstract}
\thanks{
L.X. was  partially supported by Simons Collaboration grant \#278433 and NSF grant DMS-1502147. X.Z. was partially supported by NSF grant DMS-1303296/1535464 and
DMS-1602092 and the Sloan Fellowship.}
\maketitle

\setcounter{tocdepth}{2}
\tableofcontents

\section{Introduction}
The geometry of Shimura varieties and of moduli spaces of shtukas lies at the heart of the Langlands program. It encodes the deep relation between the world of automorphic representations and the world of Galois representations. In this article, we focus on the relations between the geometry of mod $p$ fibers of \emph{different} Shimura varieties. These relations usually provide a \emph{canonical} Jacquet-Langlands transfer between various Hecke modules in a geometric way, as firstly noticed in the case of modular curves by Ribet and Serre (\cite{Ribet}, \cite{serre}) and then further developed by Ghitza and Helm (\cite{Gh1, Gh2, helm-PEL, helm}) in some other cases.  Tian and one of the authors (L.X.) took a slightly different viewpoint and constructed characteristic $p$ cycles on the special fiber of some quaternionic and unitary Shimura varieties (\cite{tian-xiao1, tian-xiao2, HTX}), parameterized by (the special fiber of) another Shimura variety, as predicted by the Langlands conjecture and the Tate conjecture.
In another direction, the recent work of V. Lafforgue (\cite{La}) on the Langlands correspondence over global function fields also made use of the relation of cohomology between different moduli of shtukas in an essential way. 

In this work, using the geometric Satake for $p$-adic groups (\cite{Z}) we develop a general strategy that relates the geometry and ($\ell$-adic) cohomology of the mod $p$ fibers of different Shimura varieties, which unifies and generalizes some aspects of the above mentioned works. 
In particular, we give a description of the basic Newton stratum of a Hodge type Shimura variety when its dimension is exactly the half of the dimension of the Shimura variety. Moreover, we prove that the cycle classes of the irreducible components of the basic locus generate Tate classes of the mod $p$ fiber of the Shimura variety, under a certain genericity condition. Note that in this article, we only consider the case when Shimura varieties have good reduction at $p$ and the case when the basic locus is middle dimensional in the ambient Shimura variety, and therefore exclude the modular curve case. We hope to generalize our approach in the future to also take this case into account.

\subsection{The main theorem}
\label{S:intro}
Let us be more precise about the question we study in this paper. 
Let $(G,X)$ be a Shimura datum. We fix an open compact subgroup $K\subset G(\bA_f)$ sufficiently small, and let
\[\mathbf{Sh}_K(G,X)=G(\bQ)\backslash X\times G(\bA_f)/K\]
denote the corresponding Shimura variety.  Recall that $\mathbf{Sh}_K(G,X)$ has a canonical model defined over the reflex field $E=E(G,X)\subset \bC$ (e.g. see \cite{Milnebook}). Let $d=\dim\mathbf{Sh}_K(G,X)$ denote its dimension. To motivate our results, we assume that $\mathbf{Sh}_K(G,X)$ is projective in the following discussions.

Let $\mH_K:=C_c^\infty(K\backslash G(\bA_f)/K,\Ql)$ denote the $\Ql$-coefficient\footnote{Note that it is possible to work with $\QQ_\ell$-coefficients instead in many statements below.} Hecke algebra of $G$ as usual, and let $\pi_f$ be an irreducible $\mH_K$-module.
It cuts out a ``submotive" $\mathbf{Sh}_K(G,X)[\pi_f]$ of $\mathbf{Sh}_K(G,X)$. More precisely, its $\ell$-adic realization can be written as
\[W^i(\pi_f,\Ql)= \Hom_{\mH_K}\big(\pi_f, \on{H}^i(\mathbf{Sh}_K(G,X)_{\overline \QQ},\Ql)\big),\]
which is a natural representation of $\Gal(\overline \bQ/E)$. Here $\overline\bQ$ denotes the algebraic closure of $\bQ$ in $\bC$. 

Now let $p$ be an unramified prime for $(G,X,K)$, namely, $K = K^pK_p \subset  G(\AAA_f^p)  \times G(\QQ_p) $ with $K_p$ a hyperspecial open compact subgroup of $G(\bQ_p)$. 
In this case, it is expected that $\mathbf{Sh}_K(G,X)$ should have a good reduction $\overline \scrS$ over $\FF_v$, where $v$ is a place of $E$ over $p$ with residue field $\FF_v$. Therefore, the motive  $\mathbf{Sh}_K(G,X)[\pi_f]$ should also have good reduction $\overline\scrS[\pi_f]$ at $v$. In particular, $W^i(\pi_f,\Ql)$ is a $\phi_v$-module, where $\phi_v$ denotes the \emph{geometric} Frobenius element of $\Gal(\overline\FF_v/\FF_v)$.  If $i$ is even, let 
\[T^{i}(\pi_f,\Ql):=\bigcup_{j\geq 1}W^{2i}(\pi_f,\Ql)(i)^{\phi_v^j}\]
denote the space of Tate classes for the motive $\overline{\mathscr S}[\pi_f]$. 
According to the Tate conjecture, these classes should come from algebraic cycles on $\overline{\mathscr S}$ via the cycle class map. One of the main goals of this work is to verify this for ``generic" $\pi_f$ in many cases when $(G,X)$ is of Hodge type.

To describe our results more precisely, let us first recall the (naive) conjectural description of 
$$W(\pi_f,\Ql):=\bigoplus_i W^i(\pi_f,\Ql)(d/2)$$  
as a $\phi_v$-module. Here we fix once and for all a square root $\sqrt{p}$, or equivalently a half Tate twist  $\Ql(\frac{1}{2})$. 

Let
$(\hat G,\hat B,\hat T, \hat X)$ be the Langlands dual group of $G$ defined over $\Ql$, equipped with a pinning (in particular a Borel subgroup $\hat B$ and a maximal torus $\hat T$). Note that our assumption on $p$ implies that the action of $\Gal(\overline{\bQ}_p/\bQ_p)$ on $(\hat G,\hat B,\hat T, \hat X)$ factors through $\Gal(\FF_{p^m}/\FF_p)$ for some finite field $\FF_{p^m}$ containing $\FF_v$, and therefore one can define the local Langlands dual group ${^L}G_{p}:=\hat G\rtimes \Gal(\FF_{p^m}/\FF_p)$. Let $\xch(\hat T)=\Hom(\hat T,\bG_m)$ denote the weight lattice of $\hat T$.
Let $\{\mu\}$ be the conjugacy class of the one parameter subgroups associated to the Shimura datum, regarded as a dominant weight (with respect to $\hat B$) in $\xch(\hat T)$ in the usual way.
We write $\mu^*$ for $-w_0(\mu)$ where $w_0$ is the longest element in the Weyl group of $\hat G$, and let
$V_{\mu^*}$
be the highest weight representation of $\hat G$ of highest weight $\mu^*$, which extends canonically to a representation
\[r_{\mu^*}: \hat G\rtimes \Gal(\FF_{p^m}/\FF_v) \to \GL(V_{\mu^*}).\]
Let $\on{rec}(\pi_{f,p}): \langle \phi_p\rangle\to {^L}G_{p}$  denote the unramified Langlands parameter of the $p$-component $\pi_{f,p}$ of $\pi_f$ (which depends on our choice of the half Tate twist). Then $\phi_v$ acts on $V_{\mu^*}$ as $r_{\mu^*}(\on{rec}(\pi_{f,p})(\phi_v))$. This action gives a well-defined isomorphism class in the Grothendieck group of $\phi_v$-modules, denoted by $[V_{\mu^*}]$.
Then the (naive) Langlands conjecture claims that in the Grothendieck group of $\phi_v$-modules,
\begin{equation}\label{E: Lang Conj}
[W(\pi_f,\Ql)]= a_G(\pi_f)[V_{\mu^*}],
\end{equation} 
where $a_G(\pi_f)$ is a certain multiplicity of $\pi_f$.

\begin{rmk}
\label{R:Gphi/G}
Here is a more canonical description of $W(\pi_f,\Ql)$ as a $\phi_v$-module. We  consider the conjugation action of $\hat G$ on the non-neutral connected component $\hat G\phi_p\subset {^L}G_p$. The quotient stack $[\hat G\phi_p/\hat G]$ can be thought as the stack of unramified Langlands parameters, and the $\hat G$-structure on $V_{\mu^*}$ defines a vector bundle on $[\hat G\phi_p/\hat G]$ by descent, denoted by $\widetilde {V_{\mu^*}}$. 
There is a tautological automorphism $\ga_{\on{taut}}$ of this vector bundle: it acts on the fiber of $\widetilde{ {V_{\mu^*}}}$ over $\ga\phi_p$ as $r_{\mu^*}((\ga\phi_p)^{[v:p]})$. The Langlands parameter of $\pi_{f,p}$ defines a point of this stack. Then $W(\pi_f,\Ql)$ equipped with the action of $\phi_v$ is canonically identified with the fiber of $\widetilde{ V_{\mu^*}}$ over this point, equipped with the tautological action.
\end{rmk}

With the above description of $W(\pi_f,\Ql)$ at hand, we can also give a description of $T^{d/2}(\pi_f,\Ql)$. 
For $\la\in\xch(\hat T)$ and a representation $V$ of $\hat G$, let $V(\la)$ denote the $\la$-weight subspace of $V$ (with respect to $\hat T$).
We define the following lattice
\[\Lambda^{\Tate_p}:=\Big\{\la\in \xch(\hat T)\; \Big|\; \sum_{i=0}^{m-1} \phi_p^i(\la)\in \xcoch(Z_G)\Big\}\subset \xch(\hat T).\]
Here, $Z_G$ denotes the center of $G$, and $\xcoch(Z_G)$ denotes the coweight lattice of $Z_G$. If we denote by $Z_G^\circ$ the neutral connected component of $Z_G$ and $\widehat{Z_G^\circ}$ its dual group, then there is a surjective map $\hat T\to \widehat{Z_G^\circ}$ and therefore 
$\xcoch(Z_G)=\xch(\widehat{Z_G^\circ})$ is a sublattice of $\xch(\hat T)$. For a representation $V$ of $\hat G$, we define the following subspace
\[V^{\Tate_p}=\bigoplus_{\la\in\Lambda^{\Tate_p}} V(\la).\]
This space is usually large. For example, in the case $G$ is an odd unitary group of signature $(i,n-i)$ over a quadratic imaginary field, the dimension of this space at an inert prime $p$ is $\begin{pmatrix}(n-1)/2\\ \lfloor i/2\rfloor\end{pmatrix}$ (see Remark \ref{classification of minuscule} for more examples).

We have the following observation (see Lemma \ref{L:strongly general again} for the proof).
\begin{lem}
\label{L:strongly general}
Let $\ga\phi_p\in \hat T\rtimes \Gal(\FF_{p^m}/\FF_p) \subset \hat G\rtimes \Gal(\FF_{p^m}/\FF_p)$ be an element whose image in $\widehat{Z_G^\circ} \rtimes \Gal(\FF_{p^m}/\FF_p)$ has finite order. Then 
\begin{equation}
\label{E:VTate inclusion}
V_{\mu^*}^{\Tate_p}\subseteq \bigcup_{j\geq 1} (V_{\mu^*})^{r_{\mu^*}((\ga\phi_p)^{j[v:p]})},\end{equation}
and if $\ga$ is strongly general with respect to $V_{\mu^*}$ in the sense of Definition~\ref{D: general parameter}, the above inclusion is an isomorphism.
\end{lem}

Now, up to conjugacy, we may assume that $\on{rec}(\pi_{f,p})(\phi_p)=\ga_p\phi_p$ for some $\ga_p\in\hat T$ satisfying the condition of the lemma above. Therefore, if \eqref{E: Lang Conj} holds, then
\[a_G(\pi_f)V_{\mu^*}^{\Tate_p}\subset T^{d/2}(\pi_f,\Ql)\subset W(\pi_f,\Ql)\]
and if the local parameter of $\pi_{f,p}$ is general enough, the above inclusion should be an isomorphism. 

One of the main goals of our work is to give a uniform construction of those algebraic cycles on $\overline{\mathscr S}$ predicted by the Tate conjecture and the Langlands conjecture as above.
To state our results, we need the following.

\begin{lem}
\label{L:inner form}
Assume that there is a prime $p$ such that $G_{\bQ_p}$ is unramified and $V_{\mu^*}^{\Tate_p}\neq0$. Then there exists a (unique) inner form $G'$ of $G$, such that $G'(\bA_f)\simeq G(\bA_f)$ and $G'_{\bR}$ is compact modulo center.
\end{lem}
This is a combination of Corollary \ref{inner mu} and Proposition \ref{unramified basic}.

Now we further assume that $(G,X)$ is of Hodge type and $p>2$. Let $v$ be a place of $E$ over $p$. Let $\mO_{E,(v)}$ denote the localization of the ring of integers of $E$ at $v$, $\FF_v$ the residue field at $v$, and $\mO_{E,v}$ the completion of $\mO_{E,(v)}$.
Then by the work of Kisin \cite{Ki1} and Vasiu \cite{Vasiu}, there is a canonical smooth integral model of $\mathbf{Sh}_K(G,X)$ over $\mO_{E,(v)}$, denoted by $\mathscr S_K(G,X)$. It will be more convenient in the sequel to change the notation from $\overline\scrS$ to $\Sh_{\mu}$ to denote (the perfection of) its mod $p$ fiber. There is a natural Newton stratification of $\Sh_{\mu}$. Let $\Sh_{\mu,b}$ denote the basic Newton stratum, on which $\mH_K$ acts by (cohomological) correspondences.\footnote{The action of the Hecke algebra at $p$ is subtle and can be ignored at the first pass.}

We need one more notation. For a (not necessarily irreducible) algebraic variety $Z$ of dimension $d$ over an algebraically closed field, let ${\rm H}^{\rm BM}_{2d}(Z)$ denote the $(-d)$-Tate twist of the top degree Borel-Moore homology (with $\Ql$-coefficients), which is the vector space with a basis labeled by the $d$-dimensional irreducible components of $Z$. Now for
$X$ is a smooth (but not necessarily proper) variety of dimension $d+r$ defined over a finite field $\FF_q$, and a (not necessarily irreducible) projective subvariety $Z \subseteq X_{\overline \FF_q}$ of dimension $d$, there is the cycle class map
\[\on{cl}: {\rm H}^{\rm BM}_{2d}(Z)\to T^{r}(X,\Ql):=\bigcup_{j\geq 1}{\rm H}_{c}^{2r}(X_{\overline \FF_q}, \Ql(r))^{\phi_q^j}.\]

Our main theorem is as follows. (For a more precise form, see  Theorem \ref{T:main theorem again}). 
\begin{thm}
\label{T:main theorem}
Assume that $(G,X)$ is of Hodge type and the center $Z_G$ is connected. Let $K\subset G(\bA_f)$ be a (small enough) open compact subgroup.
Let $p>2$ be a prime, such that $K_p$ is hyperspecial and $V_{\mu^*}^{\Tate_p}\neq 0$. 
Let $G'$ be the inner form of $G$ as in Lemma~\ref{L:inner form}. Then:
\begin{enumerate}
\item $\Sh_{\mu,b}$ is pure of dimension $\frac{d}{2}$. In particular, $d$ is always an even integer. 
There is an $\mH_K$-equivariant isomorphism
$${\rm H}^{\rm BM}_{d}(\Sh_{\mu, b, \overline \FF_v})\cong C(G'(\bQ)\backslash G'(\bA_f)/K,\Ql)\otimes V_{\mu^*}^\Tate.$$
Here (a chosen) isomorphism $G'(\bA_f)\simeq G(\bA_f)$, we regard $K$ as an open compact subgroup of $G'(\bA_f)$, and use $C(G'(\bQ)\backslash G'(\bA_f)/K,\Ql)$ to denote the space of $\Ql$-valued functions on the finite set $G'(\bQ)\backslash G'(\bA_f)/K$.

\item 

Let $\pi_f$ be an irreducible module of $\calH_K$, and let 
$${\rm H}^{\rm BM}_{d}(\Sh_{\mu,b, \overline \FF_v})[\pi_f]=\Hom_{\mH_K}\big(\pi_f,{\rm H}^{\rm BM}_{d}(\Sh_{\mu,b, \overline \FF_v})\big)\otimes \pi_f$$ be the $\pi_f$-isotypical component. Then the cycle class map
\[\on{cl}:{\rm H}^{\rm BM}_{d}(\Sh_{\mu, b, \overline \FF_v})\to {\rm H}^{d}_{c}\big(\Sh_{\mu,\overline \FF_v},\overline\bQ_\ell(d/2)\big)\]
restricted to ${\rm H}^{\rm BM}_{d}(\Sh_{\mu, b, \overline \FF_v})[\pi_f]$ is injective if the Langlands parameter of $\pi_{f,p}$ (the component of $\pi_f$ at $p$) is general with respect to $V_{\mu^*}$.

\item Assume that $\mathbf{Sh}_K(G,X)$ is a Kottwitz arithmetic variety (i.e. those considered in \cite{Kolambda}), or assume that $G_\der$ is simply-connected and anisotropic, and there is a place $p\neq p'$ such that $\pi_{f,p'}$ is an unramified twist of the Steinberg representation. Then the $\pi^p_f$-isotypical component of the cycle class map $\on{cl}$ is surjective onto 
$$\sum_{\pi_{p}}T^{d/2}(\pi_p\pi_f^p,\Ql)\otimes\pi_{p}\pi^p_f$$ 
if the Langlands parameters of $\{\pi_{p}\}$ appearing in the above sum are all strongly general with respect to $V_{\mu^*}$. 
In particular, the Tate conjecture holds for these $\pi^p_f$.
\end{enumerate}
\end{thm}
We refer to Definition \ref{D: general parameter} and Remark \ref{R: general and strongly general} below for the precise definition and some examples of ``general" and ``strongly general" with respect to $V_{\mu^*}$.

\begin{rmk}
\begin{enumerate}
\item We do not need to assume that the Shimura variety $\mathbf{Sh}_K(G,X)$ is compact in Part (1) and (2) of the theorem. Shimura varieties considered in Part (3) are compact.

\item

Note that from the dimension formula of the Newton strata of mod $p$ fiber of Shimura varieties of Hodge type (\cite{Ham2, ZhangChao2}), it is always the case that $2\dim \Sh_{\mu,b}\leq \dim \Sh_{\mu}$. It turns out that the equality achieves if and only if $V_{\mu^*}^{\Tate_p}\neq 0$.

\item One may replace the sum in Part (3) by a single representation $\pi_f$ if the following Conjecture \ref{introconj: S=T} holds. We also conjecture Part (3) of the theorem should be true for general Shimura varieties under the same condition on the Langlands parameter at $p$. 

\item Part (2) of
Theorem~\ref{T:main theorem} may be viewed as a geometric realization of Jacquet-Langlands correspondence.
\end{enumerate} 

\end{rmk}

The theorem will be obtained as a combination of another four sets of theorems, which will be discussed below. The first two are of geometric nature, and the last two are of representation theoretical nature. We believe each of them is of independent interests.

\subsection{Irreducible components of some affine Deligne-Lusztig varieties}
The first ingredient is the study of irreducible components of certain affine Deligne--Lusztig varieties (ADLVs) for unramified elements. We switch to the local setting and will treat the case of equal and mixed characteristics uniformly.\footnote{This result is new even in the equal characteristic.} 
Let $F$ be a non-archimedean local field with finite residue field $k = \bF_q = \bF_{p^m}$ and ring of integers $\calO$. We temporarily drop the restriction $p\neq 2$.
Let $L/F$ denote the completion of the maximal unramified extension of $F$, $\mO_L\subset L$ its ring of integers, and $\bar k$ the residue field. The $q$-power arithmetic Frobenius in $\Gal(L/F)$ is denoted by $\sigma$.

Let $G$ be an unramified reductive group over $\mO$.  For simplicity, we assume that its center $Z_G$ is connected in this subsection.
 Let $(\hat G, \hat B, \hat T, \hat X)$ be the pinned Langlands dual group as before, now equipped with an action of $\sigma$. Taking invariants gives a (not necessarily connected) reductive group with a ``pinning" $(\hat G^\sigma, \hat B^\sigma, \hat T^\sigma, X)$. As usual, one can define the set of dominant weights $\xch(\hat T)_\sigma^+\subset \xch(\hat T)_\sigma=\xch(\hat T^\sigma)$ of $\hat T^\sigma$ (with respect to $\hat B^\sigma$), which is canonically bijective to $\xch(\hat T)_\sigma/W_0$, where $W_0$ denotes the relative Weyl group of $G$.

Let $B(G)$ denote the set of  $\sigma$-conjugacy classes of $G(L)$. In this paper, we will mainly consider the subset $B(G)_{\on{unr}}$ of unramified elements of $B(G)$.
Recall that an element $b\in B(G)$ is called \emph{unramified} if its twisted centralizer $J_b$ is a Levi subgroup of $G$, i.e. the $F$-rank of $J_b$ coincides with the $F$-rank of $G$. There is a canonical bijection (see \S \ref{S: unram element})
\[B(G)_{\on{unr}}\cong \xch(\hat T)_\sigma^+,\quad b\mapsto \la_b,\]
which composed with the natural injection $\xch(\hat T)_\sigma^+\hookrightarrow \xch(\hat T)^{\sigma,+}\times \xch(Z(\hat G)^\sigma)$ gives the Kottwitz' invariants of unramified elements.
Note that if $b$ is unramified, $J_b$ is an unramified reductive group (but without a preferred choice of $\mO$-structure). Let $\bH\bS_b$ denote the set of hyperspecial subgroups of $J_b(F)$. Since we have assumed that $Z_G$ is connected, $J_b(F)$ acts transitively on this set by conjugation.

Let $\Gr=LG/L^+G$ denote the affine Grassmannian of $G$, where as usual $L^+G$ (resp. $LG$) denotes the jet group (resp. loop group) of $G$.
For a conjugacy class of cocharacters $\mu$ of $G$ and $b\in B(G)$, let
\[X_\mu(b):= \big\{ gL^+G\in \Gr_G\; \big|\; g^{-1}b\sigma(g)\in \overline{L^+G\varpi^\mu L^+G}\,\big\}\]
denote the corresponding \emph{affine Deligne-Lusztig variety (ADLV)}, where  $\varpi\in\mO$ is any choice of uniformizer.
This is a closed subvariety of $\Gr_G$, (perfectly) locally of finite type, equipped with an action of $J_b(F)$. The general dimension formula for ADLVs in the affine Grassmannian (\cite{Ham, Z}) specialized to the case when $b$ is unramified, implies that $X_\mu(b)$ is equidimensional of dimension $\langle \rho, \mu-\la_b\rangle$. Here, as usual $\rho$ denotes the half sum of positive coroots for $(\hat G,\hat B,\hat T)$, which descends to a linear functional on $\xch(\hat T)_\sigma$. 
For a character $\la\in\xch(\hat T)$, let $\la_\sigma$ denote its image in $\xch(\hat T)_\sigma$. 

The following is the first part of our main results on the geometry of affine Deligne-Lusztig varieties (see Theorem~\ref{C: irr comp up to J} (3)).
\begin{thm}
\label{T: ADLV}
Assume that $Z_G$ is connected. 
Let $\mu$ be a dominant cocharacter of $G$.
Let $b$ be an unramified element.
Assume that $k\neq \bF_2$ if any of the simple factors of $J_\tau$ contains a factor of type $\mathsf{B}_n,\mathsf{C}_n,\mathsf{F}_4$, and $k\neq \bF_2,\bF_3$ if  it contains $\mathsf{G}_2$. Then  there is a canonical $J_b(F)$-equivariant bijection between the set of irreducible components of $X_\mu(b)$ and the set
\begin{equation}
\label{E: parameterization irr} \bigsqcup_{\la\in \xch(\hat T),\ \la_\sigma=\la_b}\MV_\mu(\la) \times \bH \bS_b,
\end{equation}
where  $\MV_\mu(\la)$ is the set of so-called Mirkovi\'c-Vilonen (MV) cycles (which gives a basis of $V_\mu(\la)$ via the geometric Satake). In particular, the $J_b(F)$-orbits of the set of irreducible components of $X_\mu(b)$ give a basis of $V_\mu|_{\hat G^\sigma}(\la_b)$.
\end{thm}

\begin{rmk}
\label{R:after local main 1}
(i) We believe that the restriction on the field $k$ in the theorem is not necessary. 

(ii) The above parameterization of the set of irreducible components of $X_\mu(b)$ is independent of the choice of $b$ in the following sense.
Let $b',b$ be two representatives of an unramified $\sigma$-conjugacy class. Since $\la_b=\la_{b'}$, the set $\bigsqcup_{\la\in \xch(\hat T),\ \la_\sigma=\la_b}\MV_\mu(\la) $ is independent of the choice of $b$ or $b'$. Now a choice of $g$ such that $b'=g^{-1}b\sigma(b)$ induces an isomorphism $X_\mu(b)\cong X_{\mu}(b')$ and $\bH\bS_b\cong \bH\bS_{b'}$, and the parameterization in the theorem is compatible with these isomorphisms.

(iii) Note that $\bH\bS_b$ embeds in the Bruhat-Tits building of $J_b$ as some vertices. This is the point of view adapted in most literatures (\cite{vollaard, VW, RTW, HP, goertz-he}) to parameterize irreducible components of ADLVs.
\end{rmk}

Corresponding to each $(\bba,x)\in  \MV_\mu(\la)\times \bH\bS_b$, there is an irreducible component $X_\mu^{\bba,x}(b)\subset X_\mu(b)$, under the above parameterization. This is a (perfectly) projective variety with an action of $J_{b,x}(\mO)$, the hyperspecial subgroup of $J_b$ corresponding to $x$. The second part of our main results on the geometry of affine Deligne-Lusztig varieties is a construction of (an open subset of) $X_{\mu}^{\bba,x}(b)$. Let us briefly describe it and refer to Theorem \ref{C: unique comp} for details.

Let $\bba\in\MV_\mu(\la)$. Under the assumption that $Z_G$ is connected, there is a way to choose a particular pair of dominant weights $(\tau_\bba, \nu_\bba)$ of $\hat T$ (with respect to $\hat B$) so that we can write $\la=\tau_\bba+\sigma(\nu_\bba)-\nu_\bba$ (see Lemma~\ref{L: finding best tau}(2)). Then up to $\sigma$-conjugacy, we can choose $b=\varpi^{\tau_\bba}$. In this case, $J_b(F)\cap G(\mO_L)$ is a hyperspecial subgroup, giving a point $x_0\in \bH\bS_b$. This $X_{\mu}^{\bba, x_0}(b)$ contains an open subset $\mathring{X}_{\mu}^{\bba,x_0}(b)$, which fits into the following Cartesian diagram
\begin{equation}
\label{E:irr components of ADLV open}
\xymatrixcolsep{5pc}\xymatrix{\mathring{X}_\mu^{\bba,x_0}(b)\ar[r]\ar[d]& \mathring{\Gr}_{(\nu_\bba,\mu)|\tau_\bba+\sigma(\nu_\bba)}^{0,\bba}\ar[d]\\
\mathring{\Gr}_{\nu_\bba}\ar^-{1\times b\sigma}[r]&\mathring{\Gr}_{\nu_\bba}\times \mathring{\Gr}_{\tau_\bba+\sigma(\nu_\bba)}
.}\end{equation}
Here $\mathring{\Gr}_{\nu_\bba}$ and $\mathring{\Gr}_{\tau+\sigma(\nu_\bba)}$ are Schubert cells in $\Gr$, i.e. $L^+G$-orbits through $\varpi^{\nu_\bba}$ and $\varpi^{\tau+\sigma(\nu_\bba)}$, and $\mathring{\Gr}_{(\nu_\bba,\mu)|\tau+\sigma(\nu_\bba)}^{0,\bba}$ is a certain $L^+G$-stable subset in $\mathring{\Gr}_{\nu_\bba}\times \mathring{\Gr}_{\tau+\sigma(\nu_\bba)}$ related to the chosen $\bba$ (see \eqref{E: open Satake cycle}).  We refer to \S \ref{S: irr comp of ADLVs} for more detailed discussions.

Note that the definition of $\mathring{X}_\mu^{\bba, x_0}(b)$ is similar to the definition of finite Deligne-Lusztig varieties. Indeed, in some cases (e.g. $b$ is basic and $\nu_\bba$ is minuscule), $\mathring{X}_\mu^{\bba, x_0}(b)$ is a finite (closed) Deligne-Lusztig variety, on which the action of $J_{b,x_0}(\mO)$ factors through $J_{b,x_0}(\mO/\varpi)$.\footnote{So far most literatures on irreducible components of ADLVs only focus on these cases and we believe our result is the first work that systematically goes beyond these cases.} However in general for some $\bba$, we cannot choose $\nu_\bba$ to be minuscule (even if $\mu$ is minuscule). Then $X_\mu^{\bba, x_0}(b)$ belongs to a new class of varieties, on which the action of $J_{b,x_0}(\mO)$ does not factor through $J_{b,x_0}(\mO/\varpi)$. For example, if 
$$(G,\mu,b)=(\on{U}_n, (1^20^{n-2}), 1),\quad n>4,$$  
there are many $\bba$'s such that the action of $J_{b,x_0}(\mO)\simeq \on{U}_n(\mO)$ on $X_\mu^{\bba, x_0}(b)$'s factors through $\on{U}_n(\mO/\varpi^2)$ but not through $\on{U}_{n}(\mO/\varpi)$ (in particular they are not finite Deligne-Lusztig varieties and are \emph{different} varieties in mixed and equal characteristics). We refer to \cite{XZ} for a more detailed study of this case. These varieties probably deserve further study.

\begin{rmk}
\label{R:after local main 2}
Motivated by the above theorem, Miaofen Chen and one of the authors (X.Z.) made a conjecture that gives a similar description of the $J_b(F)$-orbits of the set of irreducible components of $X_\mu(b)$ for all $b\in B(G)$. Namely, for each $\sigma$-conjugacy class $b\in B(G)$, one can attach an element $\la_b\in\xch(\hat T)_\sigma$, which is the ``best integral approximation of" the Newton point $\nu_b$ of $b$.
\begin{conjecture}
The $J_b(F)$-orbits of the set of irreducible components of $X_\mu(b)$ canonically give a basis of
$(V_{\mu}|_{\hat G^\sigma})(\la_b)$. 
\end{conjecture}
Besides the cases studied in this paper,
this conjecture has also been verified recently by Hamacher  and Viehmann \cite{HV} in some other special cases (which have no overlap with what studied in this paper). Probably combining their techniques and the method of the current paper will give an approach of the conjecture in general and give a method to describe the irreducible components of general $X_\mu(b)$.
\end{rmk}

Now back to the global situation as in Theorem \ref{T:main theorem}. We assume $p\neq 2$, and revert $\mu$ to $\mu^*$. It turns out under the assumption $V_{\mu^*}^{\Tate_p} \neq 0$, the basic element $b\in B(G_{\bQ_p},\mu^*)$ is unramified and $V_{\mu^*}^{\Tate_p}=V_{\mu^*}|_{\hat G^\sigma}(\la_b)$ (see \S \ref{S: unram element}). 
Then Part (1) of Theorem \ref{T:main theorem} follows easily from Theorem \ref{T: ADLV} together with the Rapoport--Zink uniformization in this case (see Corollary \ref{C: RZ uniformization basic}): there is an isomorphism (which depends on some auxiliary choices)
\begin{equation}
\label{E: T RZ uniformization}
G'(\QQ) \big \backslash X_{\mu^*}(b) \times G'(\AAA_f^{p}) / K^p \simeq \Sh_{\mu,b,\overline\bF_v}.
\end{equation}

In addition, the parametrization in \eqref{E: parameterization irr} induces a decomposition $X_{\mu^*}(b)=\cup_{\bba}X_{\mu^*}^{\bba}(b)$ into unions of certain irreducible components of $X_{\mu^*}(b)$, where $\bba\in \MV_{\mu^*}(\tilde\la)$ and $\tilde\la\in\xch(\hat T)$ is a lift of $\la_b$. This in turn induces a corresponding decomposition $\Sh_{\mu,b,\overline\bF_v}=\cup_\bba \Sh_{\mu,b,\overline\bF_v}^\bba$ (which is independent of auxiliary choices).
Then for the dimension reason, the cycle class map (the Gysin map) induces
\begin{equation}
\label{E: T cycle map}
\on{Gys}_\bba: {\rm H}^{\rm BM}_{d}(\Sh_{\mu,b,\overline\bF_v}^\bba)\cong 
C(G'(\bQ)\backslash G'(\bA_f)/K,\Ql)\to \on{H}_c^{d}(\Sh_{\mu,\overline\bF_v}, \Ql(d/2)).
\end{equation}
Dually, there is a restriction map
\[\Res_\bbb:\on{H}_c^{d}(\Sh_{\mu,\overline\bF_v}, \Ql(d/2))\to \on{H}_c^d(\Sh_{\mu,b,\overline\bF_v}^\bbb, \Ql(d/2))\cong C(G'(\bQ)\backslash G'(\bA_f)/K,\Ql).\]
The way we prove Part (2) of Theorem~\ref{T:main theorem} is to compute the intersection pairing of these cycles, or equivalently the composition of the Gysin and restriction maps $\Res_\bbb\circ \on{Gys}_\bba$.
We will show that the composition 
\[\Res_\bbb\circ\on{Gys}_\bba\in \End (C(G'(\bQ)\backslash G'(\bA_f)/K,\Ql))\]
is induced by an element $f_{\bba,\bbb}$ in the local unramified Hecke algebra $C_c^\infty(K_p\backslash G(\bQ_p)/K_p)$, i.e.
\[\Res_{\bbb}\circ\on{Gys}_{\bba}(f)=\int_{G(\bQ_p)} f(xy^{-1})f_{\bba,\bbb}(y)dy, \quad f\in C(G'(\bQ)\backslash G'(\bA_f)/K).\]
It seems in general it is very difficult to compute these elements $\{f_{\bba,\bbb}\}$ directly,\footnote{For some explicit computations in some special cases, we refer to \cite{{tian-xiao1}, HTX} and \cite{XZ}.} especially in the case when some $X_{\mu^*}^{\bba, x}$ are not finite Deligne-Lusztig varieties. Therefore, we proceed along a different way, which makes use of the geometric Satake isomorphism in an essential way. This will be explained in the next subsection.

\subsection{Cohomological correspondences between the moduli of local shtukas}
Now we explain the key ingredient in the proof of our main theorem. We first work in the local setting and keep the notations from the previous subsection. We introduce some moduli spaces attached to $G$. For a perfect $\bF_q$-algebra $R$, let $D_R:=\Spec W_\mO(R)$ denote the disk and $D_R^*:=\Spec W_\mO(R)[1/\varpi]$ the punctured disk parameterized by $\Spec R$. The $q$-Frobenius of $R$ induces an automorphism of $D_R$, denoted by $\sigma$.

Let
$\Hk^{\loc}=[L^+G\backslash \Gr]$ denote the local Hecke stack, whose $R$-points classify two $G$-torsors $\mE_1,\mE_2$ on $D_R$, and a modification $\beta: \mE_1\dashrightarrow \mE_2$ between them (i.e. $\beta$ is an isomorphism between $\mE_1$ and $\mE_2$ when restricted to $D_R^*$). Let $\Sht^\loc$ denote the moduli of local $G$-shtukas ``with singularities at the closed point of $D$", i.e. an $R$-point of $\Sht^\loc$ is an $R$-point $(\mE_1,\mE_2,\beta)$ of $\Hk^\loc$, together with an isomorphism $\mE_2\simeq {^\sigma}\mE_1:=\sigma^*\mE_1$. 
Then there is a natural morphism 
$$\varphi^\loc: \Sht^\loc\to \Hk^\loc.$$ 
Although $\Hk^\loc$ and $\Sht^\loc$ are not algebraic (since they are quotient of ind-schemes by \emph{infinite} dimensional groups), the categories of perverse sheaves $\on{P}(\Hk^\loc_{\overline\bF_q})$ and $\on{P}(\Sht^\loc_{\overline\bF_q})$ on these moduli spaces make sense. Roughly speaking, this is because these moduli spaces can be approximated by algebraic stacks so one can define the categories of perverse sheaves on them as certain filtered colimits of categories of perverse sheaves on these approximations. We refer to \S \ref{SS:Perv(Hk)} and \S \ref{SS:Perv(Sloc)} for the rigorous treatments. In the introduction, it is harmless to pretend that these stacks were finite dimensional quotients so one can naively define the categories.

We also need the moduli space $\Hk(\Sht^\loc)$, which classify two $G$-shtukas $(\mE_1,\mE_2\simeq {^\sigma}\mE_1,\beta)$ and $(\mE'_1,\mE'_2\simeq{^\sigma}\mE'_1,\beta')$ and a modification $\al:\mE_1\dashrightarrow \mE'_1$ such that $\beta'\al=\sigma^*(\al)\beta$. There are two natural projections 
\begin{equation}
\label{E: Hk for Local Sht in Intro}
\xymatrix{ &\Hk(\Sht^\loc)\ar_{\overleftarrow{h}^\loc}[dl]\ar^{\overrightarrow{h}^\loc}[dr]&\\
\Sht^\loc&&\Sht^\loc,
}
\end{equation}
remembering the underlying local $G$-shtukas. One can regard $\Hk(\Sht^\loc)$ as the Hecke correspondence of $\Sht^\loc$.
There is a category $\on{P}^{\on{Corr}}(\Sht^\loc_{\overline\bF_q})$, whose objects are as in $\on{P}(\Sht^\loc_{\overline\bF_q})$, but whose morphisms roughly speaking are given by cohomological correspondences supported on $\Hk(\Sht^\loc)$. Again, one needs to approximate $\Hk(\Sht^\loc)$ by algebraic stacks to make sense of this category. The actual definition of this category is somehow involved, and we refer to \S \ref{S: moduli of res loc Sht} and \S \ref{S: cat PCorrSht} for a detailed construction of this category. Here we pretend that one can naively define this category and describe (in the naive sense) objects and their endomorphism ring in this category.

\begin{ex}
For a dominant coweight $\mu$, let $\Sht^\loc_\mu\subset\Sht^\loc$ denote the closed substack classifying those $(\mE_1,\mE_2\simeq{^\sigma}\mE_1,\beta)$ such that the relative position of the modification $\beta$ is bounded by $\mu$, and for two coweights $\mu_1,\mu_2$,  let 
$$\Sht^\loc_{\mu_1\mid\mu_2}=(\overleftarrow{h}^\loc)^{-1}(\Sht^\loc_{\mu_1})\cap(\overrightarrow{h}^\loc)^{-1}(\Sht^\loc_{\mu_2})\subset\Hk(\Sht^\loc).$$ 

Note that $\Sht^\loc_0$ is isomorphic to the classifying stack $[\Spec k/G(\mO)]$ of the profinite group $G(\mO)$ and $\Sht^\loc_{0\mid 0}\cong [G(\mO)\backslash G(F)/G(\mO)]$. Therefore, the correspondence
\[[G(\mO)\backslash \Spec k]\leftarrow [G(\mO)\backslash G(F)/G(\mO)]\rightarrow [\Spec k/G(\mO)],\]
is embedded in \eqref{E: Hk for Local Sht in Intro}.
Let $\delta_{\mathbf{1}}\in\on{P}^{\on{Corr}}(\Sht^\loc)$ be the ``constant sheaf" on $\Sht^\loc_0$. Then by definition elements of $\End_{\on{P}^{\on{Corr}}(\Sht^\loc)}(\delta_{\mathbf{1}})$ should be given by cohomological correspondences supported on the discrete space $\Sht^\loc_{0\mid 0}$. Giving such a correspondence should be the same as giving a (compactly supported) function on the double coset space $G(\mO)\backslash G(F)/G(\mO)$. It follows that  $\End_{\on{P}^{\on{Corr}}(\Sht^\loc)}(\delta_{\mathbf{1}})$ should be isomorphic to the spherical Hecke algebra $C_c^\infty(G(\mO)\backslash G(F)/G(\mO))$. This is indeed the case, once the precise definitions are given (see Proposition \ref{P: endo of unit}).
\end{ex}

On the dual group side, let $\on{Rep}(\hat G)$ denote the category of finite dimensional representations of $\hat G$ (with $\Ql$-coefficients), or equivalently the category of coherent sheaves on the classifying stack $\bfB \hat G$ over $\Ql$. Let $[\hat G\sigma/\hat G]$ denote the stack of unramified Langlands parameters over $\Ql$, as introduced in Remark~\ref{R:Gphi/G}.\footnote{Careful readers may have noticed that we changed the geometric Frobenius $\phi_p$ to the arithmetic one $\sigma$. See Remark~\ref{R:Gphi/G versus Gsigma/G} and \ref{R: geom Sat, arith Frob v.s. geom Frob} for the discussion.} There is a natural map $\hat G\sigma/\hat G\to \bfB \hat G$ and every representation $V$ of $\hat G$, regarded as a vector bundle on $\bfB \hat G$, pulls back to a vector bundle on $\hat G\sigma/\hat G$, which is nothing but the previously introduced $\widetilde V$.
Let $\on{Coh}^{\hat G}_{fr}(\hat G\sigma)$ denote the full subcategory of coherent sheaves on the stack $[\hat G\sigma/\hat G]$ generated by those $\widetilde{V}$. Note that for $V, W\in \on{Rep}(\hat G)$, the space $\Hom(\widetilde V,\widetilde W)$ of homomorphisms from $\widetilde V$ to $\widetilde W$ as coherent sheaves on $[\hat G\sigma/\hat G]$ is a module over the ring of regular functions on $[\hat G\sigma/\hat G]$
$$\bfJ:= \Gamma\big([\hat G\sigma/\hat G],\mO\big)=\Ql [\hat G]^{c_\sigma(\hat G)}.$$
Here $c_\sigma(\hat G)$ denotes the action of $\hat G$ on $\hat G$ by $g\bullet h=gh\sigma(g)^{-1}$. 

Here is our main local theorem (see Theorem~\ref{T:Spectral action}), which is already new in equal characteristic.
\begin{thm}
\label{T:periodic geometric Satake}
\begin{enumerate}
\item There exists a functor $S: \Coh_{fr}^{\hat G}(\hat G\sigma)\to \on{P}^{\on{Corr}}(\Sht^\loc)$, such that the following diagram is commutative
\[\xymatrix{
\on{Rep}(\hat G)\ar[r]^\cong \ar[d] & \on{P}(\Hk^\loc_{\overline\bF_q})\ar[d]\\
\on{Coh}_{fr}^{\hat G}(\hat G\sigma)\ar[r]^-S & \on{P}^{\on{Corr}}(\Sht^\loc_{\overline\bF_q}),
}
\]
where the top horizontal equivalence is the geometric Satake correspondence, the left vertical functor is $V\mapsto \widetilde V$ (or pullback of coherent sheaves along $[\hat G\sigma/\hat G]\to \bfB \hat G$), and the right vertical functor is the pullback of sheaves along $\varphi^\loc: \Sht^\loc\to \Hk^\loc$.

\item Let $\mathbf{1}$ denote the trivial representation. Then $S(\widetilde{\mathbf{1}})=\delta_{\mathbf{1}}$ and the map 
$$S: \bfJ=\End_{\Coh_{fr}^{\hat G}(\hat G\sigma)}(\widetilde{\mathbf{1}})\to \End_{\on{P}^{\on{Corr}}(\Sht^\loc)}(\delta_{\mathbf{1}})=C_c(G(\mO)\backslash G(F)/G(\mO))$$
coincides with the classical Satake isomorphism.
\end{enumerate} 
\end{thm}

\begin{rmk}
The key point of the theorem is a construction of maps between $\Hom$ spaces in the two categories.
It follows from the theorem that the ring $\bfJ$ acts on every $S(\widetilde{V})\in\on{P}(\Sht^\loc)$ as cohomological correspondences. These are (local version of) V. Lafforgue's $S$-operators, as we shall see below. But the theorem gives more. Namely, the bigger algebra
$\End(\widetilde V)$ acts on $S(\widetilde V)$. 
\end{rmk}
\begin{remark}
\label{R:Gphi/G versus Gsigma/G}
We mention that there is a subtlety when identifying $\bfJ$ with the spherical Hecke algebra at $p$.  This leads the usual switch from $\mu$ to $\mu^*$ for Galois representations arising from the cohomology of Shimura varieties, and is also related to the canonical isomorphism of stacks $[\hat G \sigma / \hat G] \cong [\hat G \sigma^{-1} / \hat G]$. See
Remark \ref{R: geom Sat, arith Frob v.s. geom Frob}.
\end{remark}

We apply the above theorem in the following way. We come back to the global setting as before. Let $(G,X)$ be a Shimura datum of Hodge type and let $p$ be an unramified prime (i.e. $K_p\subset G(\bQ_p)$ is hyperspecial).  Let $\underline G$ be the reductive $\bZ_p$-model of $G$ determined by $K_p$,
and let $\Sh_{\mu}$ denote (the perfection of) the mod $p$ fiber of the corresponding Shimura variety as before. 
There is a universal local $\underline G$-shtuka on $\Sh_{\mu}$ (coming from the crystalline $\underline G$-torsor on $\Sh_\mu$), which induces a morphism
\[\loc_p: \Sh_{\mu}\to \Sht_\mu^\loc.\]
In the Siegel case, $\loc_p$ is nothing but the perfection of the morphism sending an abelian variety to its underlying $p$-divisible group.

Next assume that $(G',X')$ is another Shimura datum such that $G(\bA_f)\simeq G'(\bA_f)$. Then the level structure $K\subset G(\bA_f)$ can be transported to a level structure of $G'$. Let $\{\mu'\}$ denote the conjugacy class of the Shimura cocharacters of $(G',X')$. Let $v\mid p$ be a place of the composite of the reflex field $EE'$.
Let $\Sh_{\mu'}$ denote its mod $p$ fiber over $v$. We assume that there exists a perfect ind-scheme $\Sh_{\mu'\mid\mu}$ that fits into the following diagram with both squares Cartesian
\begin{equation}
\label{introE: exotic Hecke}
\xymatrix{
\Sh_{\mu'} \ar[d]_{\loc_p}& \ar[l] \ar[r]\Sh_{\mu'\mid \mu} \ar[d] &\Sh_{\mu'} \ar[d]^{\loc_p}\\
\Sht_{\mu'}^\loc & \ar[l] \Sht^\loc_{\mu'\mid\mu}\ar[r]& \Sht_{\mu}^\loc,
}
\end{equation}
where $\Sht^\loc_{\mu'\mid\mu}=(\overleftarrow{h}^\loc)^{-1}(\Sht^\loc_{\mu'})\cap(\overrightarrow{h}^\loc)^{-1}(\Sht^\loc_\mu)\subset\Hk(\Sht^\loc)$.
We expect such diagram always exists for any such pair $(G,X)$ and $ (G',X')$, at the unramified prime $p$ if
\[\Hom_{\Coh^{\hat G}_{fr}(\hat G\sigma_p)}(\widetilde{V_{\mu'}},\widetilde{V_{\mu}})\neq 0.\]
Although it seems possible to establish this diagram in greater generality, in this paper, we will only focus on some special cases, which suffices for our applications. 
\begin{itemize}
\item
We show in Proposition~\ref{P:Shmu mu} that such diagram exists when $(G,X)=(G',X')$, in which case $\Sh_{\mu'\mid \mu}$ is the perfection of the mod $p$ fiber of a natural integral model of the $p$-power Hecke correspondences of $\mathbf{Sh}_K(G,X)$.
\item
We establish this diagram using Rapoport-Zink uniformization when $\mathbf{Sh}_{K}(G',X')$ is a Shimura set (\S\ref{SS:exotic Hecke G1 discrete} and Proposition~\ref{exotic Hecke}).
\item
We also prove the existence of this diagram in some PEL case (see \S\ref{SS:exotic Hecke PEL} and Proposition~\ref{P:exotic Hecke PEL}).
\end{itemize}
We leave the general cases to a future work. 

\begin{rmk}
As just mentioned, if $(G,X)=(G',X')$, $\Sh_{\mu\mid\mu}$ is just the mod $p$ fiber of the Hecke correspondence of $\Sh_\mu$.
In general, $\Sh_{\mu'\mid \mu}$ can be regarded as ``exotic Hecke correspondences" between mod $p$ fibers of \emph{different} Shimura varieties. These correspondences cannot be lifted to characteristic zero, and give a large class of characteristic $p$ cycles on Shimura varieties. 
\end{rmk}

Now, by pulling back morphisms in $\on{P}^{\on{Corr}}(\Sht^\loc)$ along vertical maps in the above diagram gives the following.  (See \S\ref{SS: coh corr Sh}--Proposition~\ref{P: S=T for Sh set} for the proof.)
\begin{thm}
\label{introT: spectral action}
\begin{enumerate}
\item
Let $(G_i,X_i), i=1,2,3$ be a collection of Shimura data, with all $G_i(\bA_f)$ isomorphic to each other. We fix a common level structure $K$, and let $p$ be an unramified prime. In addition, assume that for each pair $(G_i,X_i),(G_j,X_j)$, the Cartesian diagram \eqref{introE: exotic Hecke} exists. Let $\{\mu_i\}$ denote the conjugacy class of Shimura cocharacters of $(G_i,X_i)$ and $d_i=\dim \Sh_K(G_i,X_i)$. Let $V_i=V_{\mu_i}$ be the corresponding highest weight representation of $\hat G$, and $\widetilde{V_i}$ the corresponding vector bundle on $[\hat G\sigma_p/\hat G]$. Let $\mH^p$ denote the prime-to-$p$ Hecke algebra.
Choose a place $v\mid p$ of the composite of all reflex fields.
Then there is a natural action of $\bfJ$ on $\on{H}^*_c(\Sh_{\mu_i,\overline\bF_v},\Ql(d_i/2))$, and a
canonical $\bfJ$-equivariant map 
\[\Hom(\widetilde{V_i},\widetilde{V_j})\to \Hom_{\bfJ\otimes\mH^p}\big(\on{H}_c^{*+d_i}(\Sh_{\mu_i,\overline\bF_v},\Ql(d_i/2)),\on{H}_c^{*+d_j}(\Sh_{\mu_j,\overline\bF_v},\Ql(d_j/2))\big),\] 
which is compatible with the natural compositions on both sides.
In particular, there is a natural action of the algebra $\End(\widetilde{V_{\mu}})$ on $\on{H}_c^{*}(\Sh_{\mu, \overline \FF_v},\Ql)$.

\item When $\Sh_{\mu,\overline \FF_v}$ is a Shimura set, the action of $\bfJ=\End(\widetilde{V_{\mu}})$ on 
$$\on{H}_c^{*}(\Sh_{\mu, \overline \FF_v},\Ql)\cong C_c(G(\bQ)\backslash G(\bA_f)/K, \Ql)$$ 
coincides with the usual Hecke algebra action.
\end{enumerate}
\end{thm}

\begin{rmk}
(1) The above theorem in particular gives some geometric Jacquet-Langlands transfer, as first studied by  Helm \cite{helm-PEL} in some special cases.

(2) The action of $\bfJ\subset \End(\widetilde{V_{\mu}})$ on $\on{H}_c^{*}(\Sh_{\mu},\Ql)$ in the theorem is the Shimura variety analogue of V. Lafforgue's $S$-operators. However, the action of the larger algebra  $\End(\widetilde{V_{\mu}})$ on $\on{H}_c^{*}(\Sh_{\mu},\Ql)$ is new, even in the function field case (where instead of Shimura varieties one considers the moduli of shtukas). It is easy to deduce the congruence relation conjecture (known as the Blasius-Rogawski conjecture) from the existence of such action and the conjecture below.\footnote{In equal characteristic, V. Lafforgue deduced the congruence relation for moduli of shtukas from the $S=T$ theorem via some ingenious but complicated manipulations in tensor categories. But one can obtain a much simpler proof by using the action of the larger algebra $\End(\widetilde{V_{\mu}})$ on the cohomology.} We will discuss this in another occasion.
\end{rmk}

The following conjecture is the analogue of V. Lafforgue's $S=T$ theorem.
\begin{conjecture}
\label{introconj: S=T}
Under the Satake isomorphism,
the action of $\bfJ$ in the above theorem coincides with the usual Hecke algebra action on $\on{H}_c^{*}(\Sh_{\mu, \overline \FF_v},\Ql)$. 
\end{conjecture}

We cannot prove this conjecture in general at the moment except the case of Shimura sets as mentioned above. Fortunately, this suffices for the application to Theorem \ref{T:main theorem}.

\medskip
Now, we come back to Theorem \ref{T:main theorem}. Recall that we assume $V^{\Tate_p}_{\mu^*}\neq 0$, which implies that there exists $G'$ as in Lemma \ref{L:inner form}. In particular, we can choose a (weak) Shimura datum $(G',X')$, such that $\Sh_{K}(G',X')=G'(\bQ)\backslash G'(\bA_f)/K$ is a Shimura set. In this set the Shimura cocharacter for $(G',X')$ is central, denoted by $\tau\in\xcoch(Z_{G'})=\xcoch(Z_G)$. Then $V_{\tau^*}$ is a $1$-dimensional representation of $\hat G$.

As we mentioned above, the diagram \eqref{introE: exotic Hecke}  exists for the pair $(G,X)$ and $(G',X')$. Indeed, in this case the diagram $\Sh_\tau\leftarrow \Sh_{\tau\mid \mu}\to \Sh_\mu$ is essentially given by the Rapoport-Zink uniformization \eqref{E: T RZ uniformization}.
Therefore we obtain a map
\begin{equation}
\label{E: }
C_c(G'(\bQ)\backslash G'(\bA_f)/K)\otimes_{C_c(G(\bZ_p)\backslash G(\bQ_p)/G(\bZ_p))} \Hom(\widetilde{V_{\tau}},\widetilde{V_{\mu}})\to \on{H}_c^{d}(\Sh_{\mu,\overline\bF_v},\bQ_\ell(d/2)),
\end{equation}
which turns out to encode all cycle class maps \eqref{E: T cycle map}.\footnote{There is a subtlety regarding the choice of $\tau$ which we ignore here. See Remark \ref{R: subtle depends on bbb}.} To summarize, the intersection matrix $(f_{\bba,\bbb})$, via the Satake isomorphism, now is expressed as the following pairing 
\[\Hom(\widetilde{V_\tau},\widetilde{V_\mu})\otimes_{\bfJ} \Hom(\widetilde{V_\mu},\widetilde{V_\tau})\to \Hom(\widetilde{V_\tau},\widetilde{V_\tau})=\bfJ\]
on the dual group side. So to prove Part (2) of Theorem \ref{T:main theorem}, it remains to understand this pairing, in particular its determinant, as a divisor over 
$$\Spec \bfJ= \hat G\sigma/\!\!/\hat G,$$
the GIT quotient of $\hat G\sigma$ by $\hat G$.
This turns out to be an interesting question in representation theory, which we discuss in the next subsection.

\subsection{Generalized Chevalley restriction map}
\label{Intro: Gen Chevalley}
We now come to another ingredient needed for the main theorem. This will be established in a companion paper \cite{XZch}.
We work on the dual group side. 
Recall that we assume that $Z_G$ is connected so $\hat G_\der$ is simply-connected. 

For a representation $V$ of $\hat G$, we consider the global section  
$$\bfJ(V)=\Gamma([\hat G\sigma/\hat G],\widetilde{V})=(\mO_{\hat G}\otimes V)^{c_\sigma\hat G},$$ 
which can be regarded as the space of $\hat G$-equivariant maps $\hat G\to V$. It follows by definition that
\[\Hom(\widetilde V,\widetilde W)\cong \bfJ(V^*\otimes W).\]
Therefore, our goal is to understand the natural pairing
\begin{equation}
\label{IntroE: pairing}
\bfJ(V)\otimes \bfJ(V^*)\to \bfJ.
\end{equation}
for a general representation $V$ of $\hat G$. 

We have the following theorem. To state it, we need a few notations. Let $S\subset T$ be the split maximal subtorus of $T$, so its dual group  is $\hat S=\hat T/(\sigma-1)\hat T$, and let $\Phi_{\on{rel}}^\vee$ be the \emph{relative} coroot system of $(G,S)$. Note that $\Phi_{\on{rel}}^\vee\subset\xch(\hat S)=\xch(\hat T)^\sigma$.
Let $W_0$ denote Weyl group of $\Phi^\vee_\mathrm{rel}$; it is the subgroup of the absolute Weyl group $W$ of $G$ fixed by $\sigma$. For a character $\la\in\xch(\hat S)$, we write $e^\la$ for the function on $\hat S$ defined by $\la$.

\begin{thm}
\label{T:intro Chevalley restriction}
\begin{enumerate}
\item $\bfJ(V)$ is a finite projective $\bfJ$-module.
\item The determinant of the pairing \eqref{IntroE: pairing} (which makes sense thanks to (1)) is a divisor on $\hat G\sigma/\!\!/\hat G\cong \hat S/\!\!/W_0$ defined by the function
\[\prod_{\al'\in\Phi_{\on{rel}}^\vee, \frac{\al'}{2}\not\in\Phi_{\on{rel}}^\vee} (e^{\al'}-1)^{\zeta_{\al'}}\prod_{\al', \frac{\al'}{2}\in\Phi_{\on{rel}}^\vee} (e^{\al'/2}+1)^{\zeta_{\al'}},\]
where 
$$\zeta_{\al'}= \sum_{n\geq 1} \dim V|_{\hat G^\sigma}(n\al'_\sigma),$$ 
and $\al'_\sigma$ denotes a coroot of the absolute system $(G, T)$ such that the sum of its $\sigma$-orbits is $\al'$.
\end{enumerate}
\end{thm}

Now we can give the promised definition of the ``generality" of Satake parameters. Recall that the Frobenius action on $\hat G$ factors through the Galois group $\Gal(\FF_{p^m}/\FF_p)$.
\begin{dfn}
\label{D: general parameter}
Let $V$ a representation of $\hat G$.
An element $\ga\sigma\in \hat G\sigma/\!\!/\hat G$ is called \emph{general with respect to $V$} (or $V$-general for short) if it does not belong to the divisor in the above theorem. It is \emph{strongly general with respect to $V$} if, for every dominant coweight $\lambda$ of $\hat T^{\sigma}$ appearing in $V|_{\hat G^\sigma}$, which does not factor through $\widehat{Z_{ G}^\circ}^{\sigma}$, we have $\lambda((\gamma \phi_p)^{mn}) \neq 1$ for every $n>1$. 
\end{dfn}

\begin{rmk}
\label{R: general and strongly general}
Note that if $\ga\sigma$ is away from the discriminant divisor, i.e. it is regular semisimple,
then it is general with respect to any representation $V$ of $\hat G$. 
However, for some ``small" representation $V$ of $\hat G$, the $V$-general condition might be weaker than regular semisimple condition. For example, let $\hat G=\GL_{2n+1}$ on which $\sigma$ acts via the standard outer automorphism, and let $V$ be the standard representation of $\GL_n$. Let $\{\al_1,\ldots,\al_{2n+1}\}$ be the eigenvalues of $\ga$. Then $\ga\sigma$ is $V$-general if and only if $\al_i\neq \al_{2n+2-i}$ for $i\neq n+1$, while $\ga\sigma$ is regular semisimple if $\al_i\neq \al_j$ for all $i\neq j$.
Note that also in this case, $\ga\sigma$ is strongly general with respect to $V$ if any of its power is general with respect to $V$.

For another example, $\hat G=\bG_m \times \prod_{i=1}^f \GL_n$ with $\sigma$ permuting the $\GL_n$-factors.
Let $V = \chi \otimes \boxtimes_{i=1}^f\wedge^{a_i}\std$ be the tensor product of exterior powers of the standard representation of $\GL_n$. The condition $V^{\Tate} \neq 0$ is equivalent to $\sum_{i} a_i \equiv 0 \bmod n$.
Then a local Langlands parameter is $f$ diagonal matrices $\{\ga_i\}$ and a scalar $c$. We write $\ga_i=\on{diag}\{\al_{i1},\ldots,\al_{in}\}$, and put $\beta_j: = \prod_i \alpha_{ij}$. Then unless $V$ is one-dimensional, the $V$-regular condition is equivalent to $\beta_j \neq \beta_{j'}$ whenever $j \neq j'$.
\end{rmk}

\subsection{A Jacquet-Langlands transfer}
Finally, we explain the ingredient needed to prove Part (3) of Theorem \ref{T:main theorem}. First, we restrict ourselves to those $(G,X)$ considered by Kottwitz \cite{Kolambda}. For a Hecke module $\pi_f$, and a $\bQ$-rational representation $\xi$ of $G$, let $m_{G'}(\pi_f\otimes \xi_\bC)$ denote the multiplicity of $\pi_f\otimes\xi_\bC$ appearing in $C^\infty(G'(\bQ)\backslash G'(\bA)/K)$, and $a_G(\pi,\xi)$ be the integer defined in \cite{Kolambda} (denoted by $a(\pi_f)$, see middle of p.p. 657 of \emph{loc. cit.}). This is the ``multiplicity" of $\pi_f$ appearing in the cohomology of $\bfSh_K(G,X)$.

We have the following Jacquet-Langlands type result, proved by a simple comparison of the trace formulas. 
For a more general statement, see Theorem \ref{JL:mult}.
\begin{thm} For the $\pi_f$ and $\xi$ above, we have an equality $a_G(\pi_f,\xi)=m_{G'}(\pi_f\otimes\xi_\bC)$.
\end{thm} 
This formula should follow from the forthcoming work of Kaletha-Minguez-Shin on the endoscopic classification of inner forms of unitary group. However, our proof does not relies on any knowledge of Arthur's multiplicity formula, nor heavy machinery of the endoscopy theory. Indeed, following the idea of Kottwitz \cite{Kolambda}, we pseudo-stabilize the trace formulas on $G$ and $G'$ (instead of on the quasi-split inner form), and compare them directly.

Now it is easy to deduce Part (3) from Part (2), together with \cite{Kolambda} and the above theorem. As a corollary, $a_G(\pi,\xi)$ is always a non-negative integer if $G'$ as in Lemma \ref{L:inner form} exists. On the other hand, the computation in Example \ref{JL:ex} (ii) shows that if $\dim \bfSh_K(G,X)$ is even, such $G'$ exists.
Combining with \cite{Kolambda}, we obtain the following result which seems to be new in this generality.
\begin{cor}
Let $\bfSh_K(G,X)$ be a Kottwitz Shimura variety. For an algebraic representation $\xi$ of $G_\bC$, let $\mL_\xi$ denote the local system of $\bC$-vector spaces over $\bfSh_K(G,X)$.
If $\dim \bfSh_K(G,X)$ is even, then the cohomology ${\rm H}^i(\bfSh_K(G,X)_{\bC},\mL_\xi)=0$ for all odd $i$. 
\end{cor}

\begin{rmk}For those $G$ arising from a division algebra of dimension $9$ over a quadratic imaginary field, the vanishing of $H^1$ was first observed Rapoport-Zink under an assumption at finite places and proved by Rogawski in general \cite[Theorem 15.3.1]{Ro}. When $G$ is compact at all but one infinite place, this corollary is also a consequence of a result of Clozel (cf. \cite[Theorem 3.3, Theorem 3.5]{Cl}). The above argument is different from the one in \cite{Cl}. 
\end{rmk}

If $(G,X)$ is general, the endoscopy is presented to complicate the picture. But one can still compare the Lefschetz trace formula for $\bfSh_K(G,X)$ and the Arthur-Selberg trace formula for $G'$. As usual, imposing a local condition at $p'\neq p$ will greatly simplify the situation, and one can prove something similar to the above theorem in this case.
We refer to \S \ref{S: JL for Shimura} for precisely statements and details.

\subsubsection*{Acknowledgements}
We thank Joel Kamnitzer,  Robert Kottwitz, Sug Woo Shin, and Yihang Zhu for useful discussions, and Yakov Varshavsky for sending to us his unpublished preprint.

\section{Preliminaries on automorphic representations}
\label{Sec:auto-repn}
In this section, we establish a special Jacquet-Langlands type formula, based on results from \cite{KoEllSing,Kolambda}. 

In this section, we use $\cdot^D$ to denote the Pontryagin dual of abelian groups. For a torus $T$ over $F$, let $\xch(T)$ denote its character group (over $\overline F$) and $\xcoch(T)$ its cocharacter group.
If $G$ is a reductive group over a field $F$, $G_\ad$ denotes its adjoint form, $G_\der$ its derived group, and $G_\s$ the simply-connected cover of $G_\der$. If $T\subset G$ is a maximal torus, its image in $G_\ad$ is denoted by $T_\ad$, and its pre-images in $G_\der$ and $G_\s$ are denoted by $T_\der$ and $T_\s$ respectively. The center of $G$ is denoted by $Z_G$.
Let $(\hat{G},\hat{B},\hat{T},\hat X)$ denote its dual group with a pinning over $\Ql$. The center of $\hat{G}$ is denoted by $Z(\hat{G})$.
The $F$-structure on $G$ induces an action $\Gal(\overline F/F)\to \Aut(\hat G,\hat B,\hat T,\hat X)$. 

\subsection{A Galois cohomology computation}
We recall a result of Kottwitz (cf. \cite[\S 1, \S 2]{KoEllSing}).
\begin{thm}[Kottwitz]
\label{T: thm of Kott}
(i) If $F$ is a local field of characteristic zero, and $G$ is a connected reductive group over $F$, then there is a canonical map
\[\al_G: \on{H}^1(F,G)\to \pi_0( Z(\hat{G})^{\Ga_F})^D,\]
which is an isomorphism if $F$ is a $p$-adic field.

(ii) If $F$ is a global field, then there is an exact sequence of pointed set
$$\on{H}^1(F,G)\longrightarrow \bigoplus_v \on{H}^1(F_v,G)\xrightarrow{\oplus_v\al_{G,v}} \pi_0(Z(\hat{G})^{\Ga_{F}})^D.$$
\end{thm}
We do not need the precise construction of $\al_G$. The following properties of $\al_G$ suffice for our purposes.
\begin{itemize}
\item If $G=T$ is a torus, then $\al_T$ is the given by
\begin{equation}\label{JL:TN}
\on{H}^1(F, T)\cong \on{H}^1(F,\xch(T))^D\cong \pi_0(\hat{T}^{\Ga_F})^D,
\end{equation}
where the first isomorphism is the Tate-Nakayama duality, and the second isomorphism is induced by the exponential sequence $0\to \xch(T)\to \Lie \hat{T}\to \hat{T}\to 0$.
\item If $T\subset G$ is a maximal torus, then the following diagram is commutative
\begin{equation}\label{JL:comp}
\begin{CD}
\on{H}^1(F,T)@>\al_T>> \pi_0(\hat{T}^{\Ga_F})^D\\
@VVV@VVV\\
\on{H}^1(F,G)@>\al_G>>  \pi_0(Z(\hat{G})^{\Ga_F})^D,
\end{CD}
\end{equation}
where the right vertical map is induced by $Z(\hat{G})\subset \hat{T}$.
\end{itemize}
Note that if $G=G_\ad$ is an adjoint group, $\pi_0(Z(\hat{G})^{\Ga_F})=Z(\hat{G})^{\Ga_F}$.

We will be interested in using the above result to construct inner forms of a reductive group. For us, an inner form of a connected reductive group $G$ over a field $F$ is a cohomology class of $\on{H}^1(F,G_\ad)$, which as usual can be represented as $(G',\Psi)$, where $G'$ is an $F$-form of $G$, and 
$$\Psi:G\to G'$$ (usually called an inner twist) is a $\Ga_F$-stable $G_\ad(\overline F)$-orbit of $\overline F$-isomorphisms. Recall that in general,  given an $F$-form $G'$ of $G$, 
there may be more than one way to upgrade it to an inner form $\Psi: G\to G'$ (if such a $\Psi$ exists). I.e., the fibers of $\on{H}^1(F,G_\ad)\to\on{H}^1(F,\Aut(G))$ may consist of more than one element. However, if $G'$ is quasi-split, or if $F=\bR$ and $G'$ is compact modulo center, then the fiber of  is a singleton (if non-empty). So we can talk about the quasi-split inner form (or the compact modulo center inner form) of $G$.

Now let $G$ and $G'$ be two connected reductive group over $\bQ$, and assume that $G\otimes\bA_f\simeq G'\otimes\bA_f$. Note that they share the same pinned Langlands dual group $(\hat G, \hat B, \hat T, \hat X)$. In addition, the action of $\Gal(\overline\bQ_\ell/\bQ_\ell)$ on $(\hat G, \hat B, \hat T, \hat X)$ induced by $G$ and $G'$ are the same at every finite place $\ell$, and therefore the action of $\Gal(\overline\bQ/\bQ)$ on $(\hat G, \hat B, \hat T, \hat X)$ induced by $G$ and $G'$ are the same. It follows that we can choose an inner twist $\Psi: G\to G'$ so that $(G',\Psi)$  defines an element $H^1(\bQ,G_\ad)$, which is trivial when restricted to all finite places. Then the class $[(G'_\bR,\Psi_\bR)]\in\on{H}^1(\bR, G_\ad)$ maps to zero under the composition $\on{H}^1(\bR, G_\ad\otimes \bR)\to (Z(\hat G_\s)^{\Ga_\bR})^D\to (Z(\hat G_\s)^{\Ga_\bQ})^D$. Conversely, let $G$ be a connected reductive group over $\bQ$, and let $\Psi_\bR: G_\bR\to G'_\bR$ be an inner form of $G_\bR$. Then
\begin{lem}\label{JL:inner form}
If the restriction of the character $\al_{G_\ad,\infty}[(G'_\bR,\Psi_\bR)]$ to $Z(\hat{G}_\s)^{\Ga_{\bQ}}\subset Z(\hat{G}_\s)^{\Ga_{\bR}}$ is trivial,
then there exists a unique (up to isomorphism) inner form $(G',\Psi)$ of $G$, which restricts to $[(G'_\bR,\Psi_\bR)]$ at the infinite place, and which is trivial at all finite places, i.e. $\Psi$ is $G_\ad(\overline\bA_f)$-conjugate to an isomorphism $G\otimes\bA_f\simeq G'\otimes\bA_f$. 
\end{lem}
\begin{proof}The uniqueness follows from Hasse's principle for adjoint groups (cf. \cite[Theorem 6.22]{PR}). The existence is a direct consequence the above theorem applied to $G_\ad$.
\end{proof}

As we shall see below, there are many examples of $(G, G'_\bR,\Psi_\bR)$ that satisfy the assumption of the lemma.

\subsubsection{}
We prove a result that makes Lemma \ref{JL:inner form} more useful.
For the purpose, we first assume that $F$ is real and
compute $\al_{G}$ in some cases.
Let $G$ be a real reductive group, with $\frakg$ its Lie algebra. Let $\bS=\Res_{\bC/\bR}\bG_m$. 
Let $X$ be a $G(\bR)$-conjugacy class of homomorphisms $h:\bS\to G$ such that for some (and therefore every) $h\in X$,
\begin{enumerate}
\item[(W)] under $\Ad\circ h:\bS\to\GL(\fg)$, $\fg$ acquires a real Hodge structure of weight zero, i.e. the weight homomorphism $w_h:\bG_m\subset \bS\stackrel{h}{\to} G$ factors through the center of $G$, and
\item[(P)] $\ad h(i)$ is a Cartan involution of $G_\ad$ (note that by (W), $\ad h(-1)=1$ so $\ad h(i)$ is an involution).
\end{enumerate}
These two conditions in particular imply that in the inner class of $G$, there is $G_c$ that is compact modulo center. Indeed,
\begin{equation}\label{compact form}
G_c=\{g\in G(\bC)\mid h(i)\iota{g}=g h(i)\},
\end{equation}
where $\iota$ is the conjugation with respect to the real structure of $G$. Therefore, $G_c$ determines a class $[G_c]\in H^1(\bR, G_\ad)$.

Let us fix the isomorphism $\bS_\bC=\bG_m\times\bG_m$ given as follows:
For any $\bC$-algebra $R$,  
\begin{equation}
\label{E: Delignetorus}
(\bC\otimes_\bR R)^\times\simeq R^\times\times R^\times,\quad z\otimes r\mapsto (zr, \bar{z}r).
\end{equation}
Let $\bG_m\to \bS_\bC$ be the inclusion of the first factor.
Then the map $\{h\}$ induces a homomorphism
\begin{equation}\label{JL:coch}
\mu_h: \bG_m\to \bS_\bC\to G\otimes_\bR\bC.
\end{equation}
As $h$ vary in $X$, $\mu_h$ form a conjugacy classes of 1-parameter subgroups of $G$ over $\bC$. In the sequel, $\mu_h$ is also denoted by $\mu$ for simplicity. Let $\mu_\ad$ denote the composition of $\mu$ with the projection $G\to G_\ad$. We regard $\mu_\ad$ as a character of $\hat{T}_\s$, and by restriction to $Z(\hat{G}_\s)^{\Ga_\bR}$, it defines an element $\bar{\mu}_{\ad}\in (Z(\hat{G}_\s)^{\Ga_\bR})^D$.

\begin{lem}
\label{L: Hodge cocharacter vs inner form}
We have $\alpha_{G_\ad}[G_c]= \bar{\mu}_{\ad}$.
\end{lem}
\begin{proof}
We fix an $h$ and therefore a cocharacter $\mu=\mu_h$ (not just a conjugacy class). Let $h_\ad$ denote the composition of $h$ with the projection $G\to G_\ad$.  We can assume that $h:\bS\to T\subset G$, where $T$ is a fundamental Cartan subgroup of $G$. Note that $[G_c]$ is represented by the conjugation by $h_\ad(i)$, where 
the element $h_\ad(i)$ also represents a cocycle in $H^1(\bR, T_\ad)$, denoted by $[h_\ad(i)]$. It is enough to prove that under the isomorphism
\[\al_T: H^1(\bR,T_\ad) \simeq \pi_0(\hat{T}^{\Ga_\bR})^D,\]
$[h_\ad(i)]=[\mu_\ad]$.

Note that as element in $T_\ad(\bC)$,
\[h_\ad(i)= (h_\ad)_\bC(i,-i)=h_\ad(-1,1)=\mu_\ad(-1),\]
where the first and the third identities are by definition, and the second identity follows from (W).

Note that $T_\ad$ is a compact torus, and the complex conjugation acts on $\xch(T_\ad)$ by $-1$. Therefore, 
$$\left\{\begin{array}{cl}H^1(\bR,T_\ad)&\simeq (T_\ad)_2\simeq \{t\in T_\ad\mid t^2=1\}, \\ H^1(\bR,\xch(T_\ad))&\simeq \xcoch(\hat{T}_\s)/2, \\  \hat{T}_\s^{\Ga_\bR}&=(\hat{T}_\s)_2=\{t\in \hat{T}_\s\mid t^2=1\}, \end{array}\right.$$ 
The pairing $H^1(\bR, T_\ad)\times H^1(\bR,\xch(T_\ad))\to \bQ/\bZ$ is given by 
$$(T_\ad)_2\times \xch(T_\ad)/2\to \bZ/2,\quad (t,\bar{\breve\la})\mapsto \breve\la(t),$$
where $\breve\la\in\xch(T_\ad)$ and $\bar{\breve\la}$ denote its image in $\xch(T_\ad)/2$.
The isomorphism $H^1(\bR,\xch(T_\ad))\simeq \pi_0(\hat{T}_\s^{\Ga_{\bR}})$ is given by 
$$\xcoch(\hat{T}_\s)/2\simeq (\hat{T}_\s)_2,\quad \bar{\breve\la}\mapsto \breve\la (-1).$$

Putting all above together, we see that for $\bar{\breve\la}\in \xcoch(\hat{T}_\s)/2\simeq \pi_0(Z(\hat{G}_\s))$, 
$$\mu_{\ad,1}(\bar{\breve\la})=\mu_\ad(\breve\la(-1)),$$ 
and
$$([h_\ad(i)],\bar{\breve\la})=([\mu_\ad(-1)],\bar{\breve\la})=\breve\la(\mu_\ad(-1)).$$
Therefore $[h_\ad(i)]=\bar{\mu}_{\ad}$ as an element in $\pi_0(\hat{T}_\s^{\Ga_\bR})^D$.
\end{proof}

Alternatively, we may regard $G$ as an inner form of $G_c$, whose class  $[G]\in \on{H}^1(\bR,G_{c,\ad})$ under the map $\al_{G_c}: \on{H}^1(\bR,G_{c,\ad})\to (Z(\hat G_\s)^{\Ga_\bR})^D$ then is given by $-\bar{\mu}_{\ad}$.

\begin{cor}
\label{C: position of two real}
Let $(G,X)$ and $(G',X')$ be two pairs, consisting of a real reductive group and a conjugacy class of homomorphisms from $\bS$. Assume that both pairs satisfy (W) and (P) as above, and that the corresponding compact (modulo center) inner forms are isomorphic $G_c\simeq G'_c$. 
Then there is a canonical inner twist $\Psi:G\to G'$ such that the image of $[(G',\Psi)]\in \on{H}^1(\bR,G)$ under $\al_G$ is $\bar{\mu}_{\ad}-\bar{\mu}'_{\ad}$, where $\bar{\mu}_{\ad}$ and $\bar{\mu}'_{\ad}$ are the characters of $Z(\hat G)^{\Ga_\bR}$ corresponding to $(G,X)$ and $(G',X')$ as defined above.
\end{cor}
\begin{proof}
We choose $h:\bS\to G$, and $h':\bS\to G'$. Let $T\subset G$ (resp. $T'\subset G'$) be a fundamental Cartan containing $h(\bC^\times)$ (resp. $h'(\bC^\times)$). Let $G_c$ (resp. $G'_c$) be the subgroup of $G(\bC)$ (resp. $G'(\bC)$) as defined in \eqref{compact form}. By assumption, we can choose an isomorphism $\iota: G_c\simeq G'_c$ sending $T$ to $T'$. Then we have the inner twist $G_\bC\cong (G_c)_\bC\stackrel{\iota}{\simeq}(G'_c)_\bC\cong G'_\bC$, which is independent of any choices.
Its cohomology class in $\on{H}^1(\bR,G_\ad)$ comes from a class in $\on{H}^1(\bR,T_\ad)$, which in turn can be represented by $\iota^{-1}(h'_{\ad}(i))h_\ad(i)=\iota^{-1}(\mu'_\ad(-1))\mu_\ad(-1)$. As argued in Lemma \ref{L: Hodge cocharacter vs inner form}, under $\al_G$, this is $\bar{\mu}_\ad-\bar{\mu}'_\ad$.
\end{proof}

Combining with Lemma \ref{JL:inner form},
\begin{cor}\label{inner mu}
Let $G$ be a connected reductive group defined over $\bQ$ and assume that there exists a $G(\bR)$-conjugacy class $X$ of homomorphisms $h:\bS\to G_{\bR}$ satisfying (W) and (P). Let $\mu=\mu_h$ be as in \eqref{JL:coch}. Let $(G'_\bR, X')$ be another pair, consisting of a real algebraic group and a conjugacy class of homomorphisms $h':\bS\to G'_\bR$, whose compact inner form is isomorphic to the one for $G_\bR$. Then there is a canonical inner twist $\Psi_\bR:G_\bR\to G'_\bR$. Let $\mu'=\mu_{h'}$. If $\mu_\ad|_{Z(\hat{G}_\s)^{\Ga_{\bQ}}}=\mu'_\ad|_{Z(\hat{G}_\s)^{\Ga_{\bQ}}}$, then up to isomorphism there exists a unique inner form $(G',\Psi)$ of $G$ that is trivial at all finite places and restricts to $(G'_\bR,\Psi_\bR)$ at the infinite place.

In particular, if $\mu_\ad|_{Z(\hat{G}_\s)^{\Ga_{\bQ}}}=1$, then up to isomorphism there exists a unique inner form $(G',\Psi)$ of $G$ such that $G'_{\bR}$ is compact modulo center and $\Psi$ is $G_\ad(\overline\bA_f)$-conjugate to an isomorphism $G\otimes\bA_f\simeq G'\otimes\bA_f$.
\end{cor}

\begin{rmk}
In Proposition \ref{unramified basic}, we will give a condition at a finite prime $p$ that ensures the last assumption of the above corollary hold.
\end{rmk}

\begin{ex}\label{JL:ex}
We give two examples. The first example is well known.

(i) $G= D^\times$, where $D$ is a quaternion algebra over a totally real field $F$, regarded as a reductive group over $\bQ$. Write $G_\bR\simeq (\GL_2)^r \times (\bH)^s$. In this case $Z(\hat{G}_\s)^{\Ga_\bR}\simeq (\bZ/2)^{r+s}$, and $Z(\hat{G}_\s)^{\Ga_\bQ}\simeq\bZ/2$ embeds diagonally into $(\bZ/2)^{r+s}$. There is a standard $h:\bS\to (G_\ad)_\bR$ given as
\[h(a+bi)=\big(\big(\begin{smallmatrix}a& b\\-b&a\end{smallmatrix}\big)^r,1^s\big).\]
Then  it is easy to check $\mu_{\ad,1}(a_1,\ldots,a_r,a_{r+1},\ldots,a_{r+s})=a_1\cdots a_r$. In particular, the value of $\mu_\ad$ on the only non-trivial element $(-1,\ldots,-1)\in Z(\hat{G})^{\Ga_\bQ}$ is $(-1)^r$. We see that if $r'\equiv r \mod 2$, there is a quaternion algebra $D'$ over $F$, that is isomorphic to $D$ at all the finite places, and split at $r'$ real places. Then the group $G'=D'^\times$ is isomorphic to $G$ at all finite places as in Corollary \ref{inner mu}.  In particular, if $r$ is even, we have $G'$ that is compact modulo center at the infinity place.

(ii) Let $G$ be a unitary (similitude) group (or its inner form) over a totally real field $F$, associated to a CM extension $E/F$. We regard $G$ as a reductive group over $\bQ$. Write $(G_\ad)_\bR=\prod_i PU(p_i,q_i)$. Note that 
$$Z(\hat{G}_\s)^{\Ga_\bR}=\left\{\begin{array}{cc} \{1\} & p_i+q_i \mbox{ is odd,}\\ \prod_i \bZ/2 & p_i+q_i \mbox{ is even.} \end{array} \right.$$ In the former case, the assumption of Lemma \ref{JL:inner form} is empty so there is a unitary group $G'$, isomorphic to $G$ at all finite places, and $(G'_\ad)_{\bR}=\prod_i PU(p'_i,q'_i)$ for any signature $\{(p'_i,q'_i)\}$. In particular, there is $G'$ such that $G'_\bR$  is compact modulo center. In the latter case,  
$Z(\hat{G}_\s)^{\Ga_{\bQ}}\simeq \bZ/2$ embeds diagonally into $\prod_i\bZ/2$. There is a standard $h:\bS\to (G_\ad)_\bR$ given as
\[h(z)=(\on{diag}\{(z/\bar{z})^{p_1},1^{q_1}\},\cdots,\on{diag}\{(z/\bar{z})^{p_r},1^{q_r}\}),\]
and it is easy to check
\[\mu_\ad:Z(\hat{G})^{\Ga_\bR}\to \bZ/2,\quad \mu_\ad(a_1,\ldots,a_r)=\prod a_i^{p_i}.\]
In particular, the value of $\mu_\ad$ on the only non-trivial element $(-1,\ldots,-1)\in Z(\hat{G})^{\Ga_\bQ}$ is $(-1)^{\sum p_i}=(-1)^{\sum p_iq_i}$. We see that if $\sum p_iq_i=\sum p'_iq'_i$, we have $G'$ as in Corollary \ref{inner mu} that is isomorphic to $G$ at all finite places, and is of signature $\{p'_i,q'_i)\}$ at the infinite place. In particular, if $\sum p_iq_i$ is even, there $G'$, isomorphic to $G$ at all finite places, and is compact modulo center at the infinite places.
\end{ex}

\subsection{A Jacquet-Langlands transfer in the case without endoscopy}
Let $G$ be a connected reductive group over $\bQ$. We fix an inner twist $\Psi: G\to G'$ over $\overline\bQ$, and an identification $\theta: G(\bA_f)\simeq G'(\bA_f)$ so that 
\begin{equation}\label{JL:aux}
\Psi=\on{Int}(h)\theta
\end{equation} for some $h\in G'_{\ad}(\overline\bA_f)$.

We need to use some trace formulas in a simple way. Here we briefly explain our choice of measures.  As usual, we use the Tamagawa measure for connected reductive groups over $\bQ$. We fix compatible Haar measures $dx_\infty$ on $G(\bR)$ and on $G'(\bR)$, and a Haar measure $dx^\infty$ on $\theta: G(\bA_f)\simeq G'(\bA_f)$ so that the Tamagawa measure $dx$ decomposes as $dx^\infty dx_\infty$.  Let $A_G=A_{G'}$ be the maximal split subtorus of $Z_G=Z_{G'}$. We use the canonical Haar measure on $A_G(\bR)^0$,  coming from a basis of the $\bZ$-lattice $\xch(A_G)$. 

We fix a few more notations. Let $\fraka_G=\Hom(\xch(G)_\bQ,\bR)$, where $\xch(G)_\bQ$ is the group of $\bQ$-rational characters of $G$, and let $H_G: G(\bA)\to \fraka_G$ be the map given by $(H_G(g),\la)=\log |\la(g)|$. It induces an isomorphism $H_G: A_G(\bR)^0\simeq \fraka_G$. Let $G(\bA)^1=\ker H_G$ and $G(\bR)^1=G(\bR)\cap G(\bA)^1$. Then we have the isomorphism $G(\bA)^1\times A_G(\bR)^0\cong G(\bA)$ given by multiplication
and $G(\bQ)\backslash G(\bA)^1$ has finite volume.
For a (quasi)character $\chi: A_G(\bR)^0\to \bC^\times$, let $L^2_\chi(G(\bQ)\backslash G(\bA))$ denote the space of square integrable functions on $G(\bQ)\backslash G(\bA)^1$, regarded as a representation of $G(\bA)$ on which $G(\bA)^1$ acts via right translation and $A_G(\bR)^0$ acts through $\chi$. Similarly, we have  $L^2_\chi(G'(\bQ)\backslash G'(\bA))$. Using $A_G(\bR)^0\simeq \fraka_G$, we can also regard $\chi$ as a character of $\fraka_G$.

Given a smooth function $h$ on $G(\bA)$ (resp. on $G(\bR)$), compactly supported modulo $A_G(\bR)^0$, and transforming by $\chi^{-1}$ under $A_G(\bR)^0$, and a representation $\pi$ of $G(\bA)$ (resp. of $G(\bR)$), on which $A_G(\bR)^0$ acts through $\chi$,
the trace $\tr(h| \pi)$ is understood as 
$$\tr(h\cdot (\chi\circ H_G)| \pi\otimes (\chi^{-1}\circ H_G)),$$ 
where $h\cdot (\chi\circ H_G)$ is regarded as a compactly supported function on $G(\bA)/A_G(\bR)^0$ (resp. $G(\bR)/A_G(\bR)^0$) that acts on the $G(\bA)/A_G(\bR)^0$ (resp. $G(\bR)/A_G(\bR)^0$)-representation $\pi\otimes (\chi^{-1}\circ H_G)$. We have similar notations for $G'$.

We further assume that:
\begin{enumerate}
\item $G$ and $G'$ are anisotropic modulo center over $\bQ$; 
\item $G_\der$ is simply-connected;
\item $(A_G)_\bR$ is the maximal split subtorus in the center of $G_\bR$;
\item $G_\bR$ has a compact modulo center inner form;
\item The endoscopy of $G$ is trivial in the sense that for every $\ga\in G(\bQ)$, the group $\frakK(I_{\ga}/\bQ)$ as defined in \cite[\S 4]{KoEllSing} is trivial.
\end{enumerate}

Let $\xi$ be an irreducible algebraic representation of $G$ (defined over some number field). Let $\xi^*$ denote the dual representation of $\xi$. Let $\chi_\xi$ denote the restriction of $\xi^*_\bC$ to $A_G(\bR)^0$.
Let $\Pi_\xi$ be the set of equivalent classes of irreducible representations of $G_\bR$ with the same central and infinitesimal characters as $\xi^*_\bC$, and let $\Pi_\xi^0$  be the subset of discrete representations. It forms a single L-packet. For $\pi_\xi^0\in\Pi_\xi^0$, let $f_{\pi_\xi^0}$ denote a pseudo coefficient normalized so that $\tr(f_{\pi_\xi^0}|\pi_\xi^0)=1$. Let 
$$f_\xi=\frac{1}{|\Pi_\xi^0|}\sum_{\pi_\xi^0} f_{\pi_\xi^0}.$$
We similarly have $f'_\xi$.

Now let $\pi_f$ be an irreducible $G(\bA_f)=G'(\bA_f)$-module. Following Kottwitz (\cite{Kolambda}) let $a_{G,\xi}(\pi_f)$ be the number\footnote{In \cite{Kolambda}, this number is simply denoted by $a(\pi_f)$. In addition, Kottwitz define $a(\pi_f)$ using $f_{\pi_\xi^0}$ for a single $\pi_\xi^0$ instead of the average $f_\xi$. But as explained in \emph{loc. cit.}, this does not affect $a(\pi_f)$.} defined by
\[a_{G,\xi}(\pi_f)=\sum_{\pi_\infty\in\Pi_\xi}m_G(\pi_f\otimes\pi_\infty)\tr(f_\xi|\pi_\infty),\]
where $m_G(\pi_f\otimes \pi_\infty)$ denotes the multiplicity of $\pi_f\otimes \pi_\infty$ in $L^2_{\chi_\xi}(G(\bQ)\backslash G(\bA))$. Similarly, we have $a_{G',\xi}(\pi_f)$. Note that if $G'_\bR$ is compact modulo center, $a_{G',\xi}(\pi_f)=m_{G'}(\pi_f\otimes \xi_\bC)$ is the automorphic multiplicity of $\pi_f\otimes \xi_\bC$ in $L^2_{\chi_\xi}(G'(\bQ)\backslash G'(\bA))$.

\begin{thm}\label{JL:mult}
We have $a_{G,\xi}(\pi_f)=a_{G',\xi}(\pi_f)$. In particular, if $G'_\bR$ is compact modulo center, then $a_{G,\xi}(\pi_f)=m_{G'}(\pi_f\otimes\xi_\bC)$.
\end{thm}
\begin{proof}
We use $\theta$ to identify $G(\bA_f)$ and $G'(\bA_f)$.
It is enough to prove that the following two distributions on $G(\bA_f)\cong G'(\bA_f)$ are equal: for $h\in C_c^\infty(G(\bA_f))$, 
\begin{equation}\label{JL:comparison}
\sum a_{G,\xi}(\pi_f)\tr(h\mid \pi_f)=\sum a_{G',\xi}(\pi_f)\tr(h\mid \pi_f).
\end{equation}
By the trace formula for the compact quotients $G(\bQ)\backslash G(\bA)^1$, the left hand side of \eqref{JL:comparison} is equal to
\begin{equation}\label{JL:tr of G'}
\tr(hf_{\xi}\mid L^2_{\chi_\xi}(G(\bQ)\backslash G(\bA)))= \sum_{\ga\in G(\bQ)/\sim} \tau(I_{\ga})O_{\ga}(h f_{\xi}).
\end{equation}
Let us explain the notations in the above formula: 
\begin{itemize}
\item $G(\bQ)/\sim$ denotes the set of conjugacy classes of $G(\bQ)$, 
\item $\tau(I_\ga)=\on{vol}(I_\ga(\bQ)\backslash I_\ga(\bA)/A_G(\bR)^0)$ is the Tamagawa number for the centralizer $I_\ga$ of $\ga$ in $G$; 
\item $O_\ga$ denotes the usual orbital integral. 
\end{itemize}
As $\frakK(I_\ga/ \bQ)$ is trivial by assumption, by the argument of \cite[\S 9]{KoEllSing}, one can pseudo-stabilize \eqref{JL:tr of G'} as
\[\tau(G)\sum_{\ga\in G(\bQ)/^{\on{st}}{\sim}}SO_\ga(hf_\xi),\]
where $G(\bQ)/^{\on{st}}{\sim}$ denote the set of stable conjugacy classes of $G(\bQ)$.
Likewise, the right hand side of \eqref{JL:comparison}
can be written as
\begin{equation}\label{JL:tr of G}
\tau(G')\sum_{\ga\in G'(\bQ)/^{\on{st}}\sim} SO_\ga(h f'_\xi).
\end{equation}

Recall that $\tau(G)=\tau(G')$ by \cite{KoTamagawa}.
Also recall the stable orbital integrals of $f_\xi$, \cite[Lemma 3.1]{Kolambda}.
\[
SO_\ga(f_\xi)=\left\{\begin{array}{ll} \tr \xi_\bC(\ga)\cdot\on{vol}(I(\bR)/A_G(\bR)^0)^{-1}\cdot e(I)& \ga \mbox{ elliptic semisimple } \\ 0 & \ga \mbox{ non elliptic, } \end{array}\right.
\]
where $I$ denotes the inner for of the centralizer of $\ga$ in $G$ that is anisotropic modulo $A_G$. In particular, $SO_\ga(f_\xi)=SO_{\ga'}(f'_\xi)$ if $\ga$ and $\ga'$ are associated.

Now the theorem follows from the following lemma.
\end{proof}
\begin{lem}The stable conjugacy classes of $\ga\in G(\bQ)$ with $\ga_\infty$ elliptic semisimple are bijective to stable conjugacy classes of $\ga'\in G'(\bQ)$ with $\ga'_\infty$ elliptic semisimple.
\end{lem}
\begin{proof}
Using \cite[Theorem 6.6]{KoEllSing}, and our assumption $\frakK(I_\ga/\bQ)=0$, it is enough to show that there is a bijection induced by $\psi$ between stable conjugacy classes in $G(\bA)$ (i.e. elements $\ga\in G(\bA)$ up to $G(\overline\bA)$-conjugacy) with $\ga_\infty$ semisimple elliptic and stable conjugacy classes $\ga'\in G'(\bA)$ with $\ga'_\infty$ semisimple elliptic. Namely, given such $\ga\in G(\bA)$, we claim that there is $\ga'$ from $G'(\bA)$ such that $\ga'$ and $\psi(\ga)$ are conjugate under $G'(\overline\bA)$ and vice versa. Indeed, for every finite place $v$, we have $\psi_v=\on{Int}(h_v)\circ \theta_v$ by \eqref{JL:aux}. Then we take $\ga'_v=\theta_v(\ga_v)$. For $v=\infty$, we can assume that $\ga_\infty\in T(\bR)\subset G(\bR)$, where $T$ is a fundamental Cartan subgroup of $G_{\bR}$. Then there exists some $h_\infty\in G'_{\ad}(\bC)$ such that $\theta_{\infty}=\on{Int}(h_\infty)\circ\psi_\infty$ induces an embedding $T_\bR\to G'_{\bR}$. Then we take $\ga'_\infty=\theta_\infty(\ga_\infty)$. Clearly, we can reverse the construction to get $\ga\in G(\bA)$ from $\ga'\in G'(\bA)$ with the condition that $\ga'_\infty$ is semisimple elliptic.
\end{proof}

\subsection{A Jacquet-Langlands transfer for Shimura varieties}
\label{S: JL for Shimura}
We retain notations as in the previous subsection, but now only assume (1)-(3). 
In addition we assume that $(G,X)$ is a Shimura datum of Hodge type and let $\mathbf{Sh}_K(G,X)$ denote the corresponding Shimura variety (with some fixed level $K$). For an algebraic representation $\xi$ of $G$, let $\mL_\xi$ denote the $\Ql$-local system on $\mathbf{Sh}_K(G,X)$.
Then it is projective variety.  Let $p$ be an unramified prime. Let $v$ be a place of the reflex field over $p$, and let $\phi_v$ denote the geometric Frobenius. Assume that the degree of $\bF_v/\bF_p$ is $d$.
Let $f^p\in C_c^\infty(G(\bA_f^p))$ be a test function away from $p$.
Recall that by the work of Kottwitz (\cite{KoAnnArbor}) and Kisin-Shin-Zhu (\cite{KSZ}), the Lefschetz trace formula for the Shimura variety gives
\begin{equation}
\label{E: Lef trace for Sh}
\tr(\phi_v^j\times f^p\mid \on{H}^*(\mathbf{Sh}_K(G,X),\mL_\xi))=\sum_{\mE}\iota(G,H)ST_e(h),\quad j\gg 0
\end{equation}
where $\mE$ is the set of isomorphism classes of elliptic endoscopic triple $(H,s,\eta)$ of $G$, such that $H$ is unramified at $p$, and $h=h_ph^ph_\infty$ is certain test function on $H(\bA)$ constructed in \cite[\S 7]{KoAnnArbor}.

Assume that $G$ has an inner form $(G',\Psi)$ that is trivial at all finite places and $G'\otimes\bR$ is compact modulo center. 
The geometric side of the Arthur-Selberg trace formula for $G'$
\begin{equation}\label{E: AS trace of cpt}
\sum_{\mE}\iota(G,H)ST_e(f'^{H}),
\end{equation}
where $f'$ is a test function on $G'(\bA)$, and $f'^{H}$ is its (usual) transfer to $H(\bA)$. As before, we identify $G(\bA_f)\simeq G'(\bA_f)$ via $\theta$. 
\begin{lem}
If we choose the test function $f'=f'_{p,j}f'^p\Theta_\xi$ such that 
\begin{itemize}
\item $f'^p=f^p$,
\item $f'_{p,j}$ is the element in the spherical Hecke algebra $H(G(\bQ_p), K_p)$ which corresponds $\ga\phi_p\mapsto \tr((\ga\phi_p)^{dj}, V_{-\mu})$ under the Satake isomorphism (see \S \ref{S: geom v.s. classic Sat} for a review of the Satake isomorphism),
\item and $\Theta_\xi$ is the character of $\xi_{\bC}$,
\end{itemize}
then in the two stable trace formula, the parts corresponding to the endoscopic triple $(G^*, 1, \id)$ coincide.
\end{lem}
\begin{proof}
This follows from the explicit construction of $h$ in \eqref{E: Lef trace for Sh} as in \cite[\S 7]{KoAnnArbor}.
\end{proof}

Now we further assume $G_\der$ and $G'_\der$ have no simple factors that are the groups considered in the previous subsection. This in particular implies that for every finite place $v$ of $\bQ$, $G_{\der,\bQ_v}\simeq G'_{\der,\bQ_v}$ has  \emph{no} anisotropic factor.

Now for another place $p'\neq p$, if $f_{p'}$ is the Euler-Poincare function as constructed in \cite[Section 2]{KoTamagawa}, more precisely, the pullback of the Euler-Poincare function from $\bar{G}_{\bQ_{p'}}$, the quotient of $G_{\bQ_{p'}}$ by the maximal split torus in the center of $G_{\bQ_{p'}}$, then as explained in \cite[Lemma A.7, Lemma A.12]{KS16}, the trace of $f^p$ on an irreducible representation $\pi^{p}$ of $G(\bA_f^p)$ appearing in $\on{H}^*(\mathbf{Sh}_K(G,X),\mL_\xi))$ (or in the spectrum side of the trace formula of $G'$) is not zero only when $\pi_{p'}$ is an unramified character twist of the Steinberg representation. In addition, in both stable trace formulas, only the term corresponding to the endoscopic triple $(G^*, 1, \id)$ is non-vanishing. 

Let us write $\on{H}^*(\mathbf{Sh}_K(G,X),\mL_\xi)=\sum_{\pi_f} \pi^K_f\otimes W(\pi_f)$ as before. For $\pi_f$ such that $\pi_{f,p'}$ is a unramified character twist of the Steinberg representation, we choose $f=f_{p'}f^{p'}$ such that $\tr(f\mid \pi_f)=1$ and its trace on any other representation appearing in $\on{H}^*(\mathbf{Sh}_K(G,X),\mL_\xi)$ vanishes. Now, let $S_{\on{bad}}$ be the finite set of primes such that if $p\not\in S_{\on{bad}}$, then $K_p$ is hyperspecial and $f_p$ is the unite of the unramified Hecke algebra at $p$. Then by the comparison of the stabilized Lefschetz trace formula at  $p\not\in S_{\on{bad}}$, and the stable trace formula of $G'$, we have
\[\tr(f^p\mid \pi_f^p)\tr(\phi_v^j|W(\pi_f))=m_{G'}(\pi_f\otimes\xi_\bC)\tr(f^p\mid \pi_f)\tr(f'_{p,j}\mid \pi_{f,p})\]
Therefore, we obtain

\begin{cor}
\label{C: comparison, stable case}
Let $\pi_f$ be an irreducible smooth $\theta: G(\bA_f)\simeq G'(\bA_f)$ representation, which is an unramified character twist of the Steinberg at some place $p'$. Let $p\not\in S_{\on{bad}}$. Then in the Grothendieck group of $\phi_v$-modules,
\[[W(\pi_f)]=m_{G'}(\pi_f\otimes\xi_\bC)[V_{\mu^*}],\]
where $[V_{\mu^*}]$ is the $\phi_v$-module defined using the Satake parameter of $\pi_{f,p}$ as in \cite[p.p. 656-657]{Kolambda} (see also the introduction).
\end{cor}

\section{Complements on the geometric Satake}\label{S: geom Sat}
The geometric Satake isomorphism is the main tool for this paper. In this section, we recall its content in a form we need in the sequel. We will also need some results on the crystal structures on Mirkovi\'c-Vilonen cycles.
Our presentation is parallel for both equal and mixed characteristic, following \cite{Z}. 

Throughout this section, we will use the following notations. Let $F$ be a local field (of either equal or mixed characteristic), with ring of integers $\calO$ and residue field $k =\FF_q$. Note that in either equal or mixed characteristic case, there is a unique map $W(k)\to \mO$ that induces the identity map of the residue fields.

Let $L$ be the completion of the maximal unramified extension of $F$, and $\mO_L$ its ring of integers with residue field $\bar{k} = \overline \FF_q$. We fix a uniformizer $\varpi$ of $F$.
Let $G$ be an \emph{unramified} reductive group over $\calO$. We denote by $T$ the abstract Cartan subgroup of $G$. Recall that it is defined as the quotient of a Borel subgroup $B\subset G$ by its unipotent radical. It turns out that $T$ is independent of the choice of $B$ up to a canonical isomorphism. Let $S\subset T$ denote the maximal split subtorus. 
When we need to embed $T$ (or $S$) into $G$ as a (split) maximal torus, we will state it explicitly.

Let $\xch=\xch(T)$ denote the weight lattice, i.e. the free abelian group $\Hom(T_{\overline F},\bG_m)$, and let $\xcoch=\xcoch(T)$ denote the coweight lattice, i.e. the dual of $\xch(T)$. Let $\Phi\subset \xch$ (resp. $\Phi^\vee\subset \xcoch$) denote the set of roots (resp. coroots). The Borel subgroup $B\subset G$ determines the semi-group of dominant coweights $\xcoch^+\subset \xcoch$ and the subset of positive roots $\Phi^+\subset\Phi$, which turn out to be independent of the choice of $B$. Let $\Delta\subset \Phi^+$ be the set of simple roots and $\Delta^\vee$ the corresponding set of simple coroots. The quadruple $(\xch,\Delta,\xcoch,\Delta^\vee)$ is called the \emph{based root datum} of $G$. The  $q$-power (arithmetic) Frobenius $\sigma$ acts on $(\XX^\bullet, \Delta, \XX_\bullet, \Delta^\vee)$, preserving $\xcoch^+$. Let ${\xcoch}_\sigma$ denote the $\sigma$-coinvariants of $\xcoch$, and let ${\xcoch}_\sigma^+$ denote the image of the map $\xcoch^+\to {\xcoch}_\sigma$. For $\mu\in \xcoch$, its image in ${\xcoch}_{\sigma}$ is denoted by $\mu_\sigma$.

Recall that there is a partial order $\preceq$ on $\xcoch$: $\la\preceq \mu$ if $\mu-\la$ is a non-negative integral linear combinations of simple coroots; we say $\la\prec\mu$ if $\la\preceq \mu$ but $\la\neq \mu$. The restriction of the partial order to $\xcoch^+$ sometimes is called the \emph{Bruhat order} on $\xcoch^+$. 
We will usually denote by $2\rho\in\xch$ the sum of all positive roots, and put $\rho=\frac{1}{2}(2\rho)\in \xch(T)\otimes\bQ$. 

Let $W$ denote the absolute Weyl group, and $W_0=W^\sigma$ the relative Weyl group of $G$.
There is a canonical bijection $\xcoch/W\cong \xcoch(T)^+$.

Let $\Aff_k^\pf$ denote the category of perfect $k$-algebras. For a perfect $k$-algebra $R$, we define the ring of Witt vectors in $R$ with coefficients in $\mO$ as
\begin{equation}\label{ramified Witt vector}
W_\mO(R):= W(R)\hat\otimes_{W(k)}\mO:= \varprojlim_n W_{\mO,n}(R),\quad W_{\mO,n}(R):=W(R)\otimes_{W(k)} \mO/\varpi^n.
\end{equation}
Consider the usual Techm\"uller lifting $[-]:
R\to W_\mO(R)$ sending $ r\mapsto [r]$. 
Note that if $\cha F=\cha k$, then $W_\mO(R)\simeq R[[\varpi]]$, and $r\mapsto [r]$ a ring homomorphism. We sometimes write
\begin{equation}\label{discs}
D_{n,R}=\Spec W_{\mO,n}(R),\quad D_R=\Spec W_\mO(R),\quad D_R^*=\Spec W_\mO(R)[1/\varpi],
\end{equation}
thought as families of (punctured) discs parameterized by $\Spec R$.

We refer to \S \ref{ASS:perfect AG} for generalities of perfect algebraic geometry.

\subsection{Geometry of Schubert varieties}
\subsubsection{Jet groups and loop groups}
\label{S: Jet and loop}
Let $H$ is an affine group scheme of finite type defined over $\mO$, we denote by $L^+H$ (resp. $LH$) the \emph{jet group} (resp. \emph{loop group}). As  a presheaf, we have
\begin{equation}\label{E:jet and loop}
L^+H(R)=H(W_\mO(R))\quad \textrm{(resp.} \ LH(R)=H(W_\mO(R)[1/\varpi]) \ ).
\end{equation}
It is represented by an affine group scheme (resp. ind-scheme). 
For $n\geq 0$, let $L^nH$ be the $n$th jet group, i.e. 
\begin{equation}\label{E:rjet}
L^nH(R)=H(W_{\mO,n}(R)).
\end{equation} 
Then $L^nH$ is represented by the perfection of an algebraic $k$-group (the usual Greenberg realization), and $L^+H=\underleftarrow\lim L^nH$. Finally, let
\begin{equation}\label{E:cong subgroup}
L^+H^{(n)}=\ker(L^+H\to L^nH)
\end{equation}
be the \emph{$n$th principal congruence subgroup} of $L^+H$.

For $m>n$, let $\pi_{m,n}: L^mH\to L^nH$ denote the natural projection.  The kernel of $\pi_{m,n}$ is denoted by $L^{m-n}H^{(n)}$ so that $L^+H^{(n)}$ can be identified with the inverse limit $\varprojlim_{m \to \infty} L^{m-n}H^{(n)}$.

\subsubsection{Affine Grassmannians}
We denote by $\Gr:=\Gr_G$ the affine Grassmannian of $G$ over $k$. As a presheaf,

\[\Gr(R)= \left\{(\mE,\beta)\left | \begin{array}{l}\mE \mbox{ is a } G\mbox{-torsor on } D_R\textrm{ and }  \\
 \beta: \mE|_{D_R^*}\simeq \mE^0|_{D_R^*} \mbox{ is a trivialization}\end{array}\right.\right\},\]
where $\mE^0$ denotes the trivial $G$-torsor.
It is known that affine Grassmannian $\Gr$ is represented as the inductive limit of subfunctors $\Gr=\underrightarrow\lim X_i$, with $X_i\to X_{i+1}$ closed embedding, and $X_i$ being perfections of projective varieties (\cite{BL, fa, Z, BS}). In addition, it makes sense to define the category of perverse sheaves (with $\Ql$-coefficients) on $\Gr$ as in direct limit 
$$\on{P}(\Gr)=\underrightarrow\lim \on{P}(X_i),$$ 
where the connecting functor is given by the pushforward along the closed embedding $X_i\to X_{i+1}$ (which is independent of the choice of the presentation $\Gr=\underrightarrow\lim X_i$ up to a canonical equivalence). 

\subsubsection{Bruhat decomposition} 
\label{SS: Bruhat decomposition}
Let $K$ be an algebraically closed field containing $\bar k$.
Recall that the Cartan decomposition induces a canonical bijection
\begin{equation}\label{Cartan decomp2} 
G(W_\mO(K)) \big\backslash G(W_{\mO}(K)[1/\varpi]) \big/G(W_{\mO}(K))\cong \xcoch(T)/W\cong \bX_\bullet(T)^+.
\end{equation}
The $G(W_\mO(K))$-double coset corresponding to $\mu\in\bX_\bullet(T)^+$ can be concretely realized as follows. We fix embeddings $T\subset B\subset G$.
Each coweight $\mu\in\xcoch(T)$ defines a map $\mu:L^\times\to T(L)\subset G(L)$, and we denote by $\varpi^\mu=\mu(\varpi)$ the image of the uniformizer $\varpi$ in $G(L)=LG(\bar k)$. Although $\varpi^\mu$ depends on the choice of the embedding $T\subset G$ and the choice of the uniformizer, the double coset $G(W_{\mO}(K))\varpi^\mu G(W_{\mO}(K))$ does not, and corresponds to $\mu$ under the parameterization \eqref{Cartan decomp2}.

Now let $\mE_1$ and $\mE_2$ be two $G$-torsors over $D_{K}=\Spec W_{\mO}(K)$, and let 
$\beta: \mE_1|_{D_{K}^*}\simeq \mE_2|_{D_{K}^*}$
be an isomorphism over $D_{K}^*$. Such $\beta$ is usually called a \emph{modification} from $\mE_1$ to $\mE_2$, and denoted by
\[\beta: \mE_1\dashrightarrow\mE_2.\footnote{The dotted arrow reminds that there may not be an actual morphism of $G$-torsors over $D_R$.}
\]

We attach to $\beta$ an element
\[\inv(\beta)\in \xcoch(T)^+,\]
which we call  the \emph{relative position} of $\beta$, as follows: by choosing isomorphisms $\phi_1:\mE_1\simeq \mE^0$ and $\phi_2:\mE_2\simeq \mE^0$, one obtains an automorphism of the trivial $G$-torsor $\phi_2\beta\phi_1^{-1}\in \Aut(\mE^0|_{D_{K}^*})$ and therefore an element in $G(W_{\mO}(K)[1/\varpi])$. Different choices of $\phi_1$ and $\phi_2$ will modify this element by left and right multiplication by elements from $G(W_{\mO}(K))$. Therefore, via the bijection \eqref{Cartan decomp2}, we attach $\beta$ a dominant coweight $\inv(\beta) \in \xcoch(T)^+$. Note that when $G=\GL_n$, $\inv(\beta)$ is just the Hodge polygon of the $n\times n$-matrix determined by  $\beta$.

Now let $\mE_1$ and $\mE_2$ be two $G$-torsors over $D_R$, and 
$\beta: \mE_1|_{D_R^*}\simeq \mE_2|_{D_R^*}$ a modification.
Then for each geometric point $x\in \Spec R$, by base change we obtain $(\mE_1|_{D_x},\mE_2|_{D_x},\beta_x: \mE_1|_{D_x^*}\simeq \mE_2|_{D_x^*})$. We write $\inv(\beta) \preceq \mu$ (resp. $\inv(\beta) = \mu$) for a dominant coweight $\mu$ if $\inv(\beta_x) \preceq \mu$ (resp. $\inv(\beta_x) =\mu$) for all geometric points $x \in \Spec R$.

\subsubsection{Schubert varieties}
We define the \emph{(spherical) Schubert} variety $\Gr_{\mu}$ as the closed subset
$$\Gr_{\mu}=\left\{(\mE,\beta)\in \Gr\mid \inv(\beta)\preceq \mu\right\},$$
of $\Gr$. 
It contains the \emph{Schubert cell}
\[\mathring{\Gr}_\mu:=\left\{(\mE,\beta)\in \Gr\mid \inv(\beta)=\mu\right\}=\Gr_{\mu}\setminus \cup_{\la\prec\mu}\Gr_{\la}\]
as an open dense subset\footnote{In some literature Schubert varieties are denoted by $\Gr_{\preceq\mu}$ or $\overline{\Gr}_\mu$ while Schubert cells are denoted by $\Gr_\mu$. We hope our notation will not cause confusions.} . We will use the same notation to denote the image of $\varpi^\mu\in G(L)=LG(\bar k)\to \Gr(\bar k)$. Then $\mathring{\Gr}_\mu$
is the $L^+G$-orbit through $\varpi^\mu$. Note that there is a natural projection
\begin{equation}\label{E: Hodge to fil}
\pr_\mu: \mathring{\Gr}_\mu\to (\bar G/\bar P_\mu)^\pf, \qquad g\varpi^\mu \bmod L^+G\mapsto \bar{g} \bmod \bar P^\pf_\mu,
\end{equation}
where $P_\mu$ is the parabolic subgroup of $G$ generated by root subgroups $U_{\al}$ for those $\langle\al,\mu\rangle\leq 0$, and taking bar means reduction modulo $\varpi$.

The Schubert variety $\Gr_\mu$ is the perfection of a projective variety defined over the field of definition of the cocharacter $\mu$ (which is a finite extension of $k$).

\subsubsection{Convoluted product of affine Grassmannians}

For a sequence $\mu_\bullet = (\mu_1, \dots, \mu_t)$ of dominant coweights, we define 
$\Gr_{\mmu}$ 
as the presheaf which classifies the isomorphism classes of modifications
\begin{equation}
\label{E:sequence of isogenies}
\xymatrix{
\calE_t \ar@{-->}[r]^-{\beta_t} &\calE_{t-1} \ar@{-->}[r]^-{\beta_{t-1}} & \cdots \ar@{-->}[r]^-{\beta_2} & \calE_1 \ar@{-->}[r]^-{\beta_1} & \calE_0=\calE^0
}
\end{equation}
of $G$-torsors such that each $\beta_i$ has relative position $\preceq \mu_i$ at all geometric points.  For a non-negative integer $n$ (which allows to be $\infty$),
Let $\Gr_\mmu^{(n)}$ denote the natural $L^nG$-torsor over $\Gr_\mmu$ classifying the chain of modifications of $G$-torsors over $D_R$ as above together with an isomorphism $\epsilon_n:\mE_t|_{D_{n,R}}\simeq \mE^0|_{D_{n,R}}$. When $\mmu = \mu$, $\Gr_\mu^{(\infty)}\subset LG$ is   the preimage of $\Gr_\mu$ in $LG$, and
\begin{equation}
\label{E: LnG-torsor over Grmu}
\Gr_\mu^{(n)}= \Gr^{(\infty)}_\mu/L^+G^{(n)}.
\end{equation}

Note that for every $s\leq t$, there is a projection
\begin{equation}
\label{E: projs}
\pr_s:\Gr_{\mmu}\to \Gr_{\mu_1,\ldots,\mu_s}
\end{equation} 
by forgetting $(\mE_t\stackrel{\beta_t}{\dashrightarrow}\cdots\dashrightarrow \mE_{s+1}\stackrel{\beta_{s+1}}{\dashrightarrow})$ in \eqref {E:sequence of isogenies}, whose fibers over geometric points (non-canonically) are isomorphic to $\Gr_{\mu_{s+1},\ldots,\mu_t}$. This construction induces an isomorphism
\begin{equation}
\label{E: conv vs twisted}
\Gr_{\mu_1,\ldots,\mu_t}\cong \Gr_{\mu_1,\ldots,\mu_s}\tilde\times\Gr_{\mu_{s+1},\ldots,\mu_t}:= \Gr^{(\infty)}_{\mu_1,\ldots,\mu_s}\times^{L^+G}\Gr_{\mu_{s+1},\ldots,\mu_t},
\end{equation}
where $ \Gr^{(\infty)}_{\mu_1,\ldots,\mu_s}$ is the $L^+G$-torsor over $\Gr_{\mu_1,\ldots,\mu_s}$ classifying a point $(\mE_s\dashrightarrow\cdots\dashrightarrow\mE_0=\mE^0)$ of $\Gr_{\mu_1,\ldots,\mu_s}$ together with an isomorphism $\mE_s\simeq \mE^0$. So $\Gr_{\mu_1,\ldots,\mu_s}\tilde\times\Gr_{\mu_{s+1},\ldots,\mu_t}$ is the twisted product of $\Gr_{\mu_1,\ldots,\mu_s}$ and $\Gr_{\mu_{s+1},\ldots,\mu_t}$ (see \cite[\S A.1.3]{Z} for a more detailed discussion of the twisted product construction).
Therefore, $\Gr_{\mmu}\cong \Gr_{\mu_1}\tilde\times\Gr_{\mu_2}\tilde\times\cdots\tilde\times\Gr_{\mu_t}$ is called the \emph{twisted product} of $\Gr_{\mu_1},\ldots,\Gr_{\mu_t}$.

In later sections, we will make use of the following definition.
\begin{definition}
\label{D:acceptable pair}
For a sequence of dominant coweights $\mmu$, we say a non-negative integer $m$ is \emph{$\mmu$-large}
if $m\geq \langle |\mmu|, \al_h\rangle$, where $\al_h$ is the highest root. 
\end{definition}
For example, if $\mu$ is a dominant minuscule coweight, then $1$ is $\mu$-large. Also, if $m$ is $\mmu$-large, then $m$ is also $\sigma(\mmu)$-large and $\mu^*_\bullet$-large. Moreover, if $m$ is $\mmu$-large and $m'$ is $\nu_\bullet$-large, then $m+m'$ is $(\mmu,\nu_\bullet)$-large.

This notion is used the following observation. 
\begin{lem}
\label{L: m mu large}
If $m$ is $\mmu$-large, then the left action of $L^+G^{(m+n)}$ on $\Gr_\mmu^{(n)}$ is trivial for every non-negative integer $n$, or equivalently, the left action of $L^+G$ on $\Gr_\mmu^{(n)}$ factors through the action of $L^{m+n}G$.  
\end{lem}
\begin{proof}
First we consider the case $\mmu=\mu$ is single coweight. The lemma follows from
\[\varpi^{-\mu} L^+G^{(m)}\varpi^{\mu}\subseteq L^+G^{(m-\langle\mu,\al_h\rangle)}.\]
To deal with the general case, we can choose a sequence of non-negative integers $m_\bullet=(m_1,\ldots,m_t)$ such that $m=\sum m_i$ and $m_i$ is $\mu_i$-large. Since $m':=\sum_{i\geq 2}m_i$ is $(\mu_2, \dots, \mu_t)$-large, by induction there is a canonical isomorphism
\begin{equation}
\label{E:Hecke factorization induction}
\Gr_{\mu_1}^{(\infty)} \times^{L^+G^{(m'+n)}} \Gr_{\mu_2, \dots, \mu_t}^{(n)}\cong\Gr_{\mu_1}^{(\infty)}  / L^+G^{(m'+n)} \times \Gr^{(n)}_{\mu_2, \dots, \mu_t}.
\end{equation}
On the one hand, by \eqref{E: conv vs twisted}, $\Gr^{(n)}_{\mmu}$ is the quotient of by the diagonal action of $L^{m'+n}G$ on the left hand side of \eqref{E:Hecke factorization induction}. On the other hand, $L^+G^{(m+n)}$ acts trivially on $\Gr_{\mu_1}^{(\infty)}  / L^+G^{(m'+n)}=\Gr_{\mu_1}^{(m'+n)}$ by the discussion in the single coweight case. Therefore, $L^+G^{(m+n)}$ acts trivially on $\Gr_{\mmu}^{(n)}$.
\end{proof}

\subsubsection{The convolution map} 
Let 
\begin{equation}
\label{E: convolution map}
m_\mmu:\Gr_{\mmu}\to \Gr
\end{equation}
denote the \emph{convolution map}. It sends the sequence of modifications \eqref{E:sequence of isogenies} to the composition $(\calE_t, \beta_1\circ \cdots \circ \beta_t) \in \Gr$. Both \eqref{E: projs} and \eqref{E: convolution map} are perfectly proper morphisms.
It is known that $m_\mmu$ is a semi-small map (see \cite[Lemma 4.4]{MV}, \cite[Lemma 9.3]{NP} and \cite[Proposition 2.3]{Z}), i.e.
\begin{equation}\label{semismall}
\dim \Gr_\mmu\times_{\Gr}\Gr_{\mmu}=\dim \Gr_\mmu.
\end{equation}
This is equivalent to the fact that $(m_\mmu)_*\IC$ is a direct sum of simple perverse sheaves on $\Gr$, each of which is isomorphic to the intersection cohomology sheaf on some $\Gr_\nu$.

For $\la_\bullet$ another sequence of dominant coweights, we define
\begin{equation*}
\Gr_{\la_\bullet|\mmu}^0:= \Gr_{\la_\bullet}\times_{\Gr}\Gr_{\mmu}
\end{equation*}
where the fiber product is taken with respect to the convolution maps $m_{\la_\bullet}$ and $m_\mmu$. Explicitly, $\Gr_{\la_\bullet|\mmu}^0$ classifies the following commutative diagram of modifications of $G$-torsors
\begin{equation}
\label{E:corr}
\xymatrix{
\calE'_s \ar@{-->}[r]^-{\beta'_s}\ar@{=}[d] &\calE'_{s-1} \ar@{-->}[r]^-{\beta'_{s-1}} & \cdots \ar@{-->}[r]^-{\beta'_2} & \calE'_1 \ar@{-->}[r]^-{\beta'_1} & \calE'_0\ar@{=}[r]&\calE^0\ar@{=}[d]\\
\mE_t \ar@{-->}[r]^-{\beta_t} &\calE_{t-1} \ar@{-->}[r]^-{\beta_{t-1}} & \cdots \ar@{-->}[r]^-{\beta_2} & \calE_1 \ar@{-->}[r]^-{\beta_1} & \calE_0\ar@{=}[r]&\calE^0,
}
\end{equation}
where the top row defines a point of $\Gr_{\la_\bullet}$ and the bottom row defines a point of $\Gr_{\mu_\bullet}$.

Let $h^{\leftarrow}_{\la_\bullet}$ and $h^{\rightarrow}_{\mmu}$ denote the natural maps from $\Gr_{\la_\bullet|\mmu}^0$ to $\Gr_{\la_\bullet}$ and $\Gr_\mmu$ respectively. For the reason that will be clear in the sequel, we call
\begin{equation}\label{E: Sat corr}
\Gr_{\la_\bullet}\stackrel{h^{\leftarrow}_{\la_\bullet}}{\longleftarrow} \Gr^0_{\la_\bullet\mid \mmu}\stackrel{h^{\rightarrow}_{\mmu}}{\longrightarrow}\Gr_{\mmu}
\end{equation}
a \emph{Satake correspondence}. Note that $\Gr^0_{\la_\bullet\mid \mmu}=\Gr^0_{\mmu\mid\la_\bullet}$.

\begin{ex}
If $\la_\bullet=\la$ is a single dominant coweight, then $\Gr_{\mmu|\la}^0=m_\mmu^{-1}(\Gr_\la)$, where $m_\mmu:\Gr_\mmu\to \Gr$ is the convolution map.
\end{ex}

We need a corollary of \eqref{semismall}. Recall that for a variety $X$ over $k$, we denote by 
$\on{H}^{\on{BM}}_i(X_{\bar k})$ the $i$th \emph{Borel-Moore homology} of $X_{\bar k}$ (see Notation \ref{N: BM homology} for our convention).

\begin{prop} 
\label{SS:cycles of geometric Satake}
\begin{enumerate}
\item We have
$$\dim \Gr_{\la_\bullet|\mmu}^0\leq \frac{1}{2}\big(\dim \Gr_{\la_\bullet}+\dim \Gr_\mmu\big)=\langle\rho,|\la_\bullet|+|\mmu|\rangle.$$

\item There is a canonical isomorphism
\[\Hom_{\on{P}(\Gr)} \big((m_{\la_\bullet})_*\IC_{\la_\bullet},(m_\mmu)_*\IC_{\mmu}\big)\cong  \on{Corr}_{\Gr_{\la_\bullet|\mmu}^0}\big((\Gr_{\la_\bullet},\IC_{\la_\bullet}),(\Gr_{\mmu},\IC_\mmu)\big).\]
Under the above isomorphism, the composition
\begin{multline}\label{III: comp}
\Hom \big((m_{\kappa_\bullet})_*\IC_{\kappa_\bullet},(m_{\lambda_\bullet})_*\IC_{\lambda_\bullet}\big)\otimes \Hom \big((m_{\lambda_\bullet})_*\IC_{\lambda_\bullet},(m_{\mu_\bullet})_*\IC_{\mu_\bullet}\big)\\
\to\Hom\big((m_{\kappa_\bullet})_*\IC_{\kappa_\bullet},(m_{\mu_\bullet})_*\IC_{\mu_\bullet}\big)
\end{multline}
corresponds to 
\begin{multline*}\on{Corr}_{\Gr_{\kappa_\bullet|\lambda_\bullet}^0}\big((\Gr_{\kappa_\bullet},\IC),(\Gr_{\lambda_\bullet},\IC)\big) \otimes \on{Corr}_{\Gr_{\lambda_\bullet|\mu_\bullet}^0}\big((\Gr_{\lambda_\bullet},\IC),(\Gr_{\mu_\bullet},\IC)\big)\\
\to \on{Corr}_{\Gr_{\kappa_\bullet|\mu_\bullet}^0}\big((\Gr_{\kappa_\bullet},\IC),(\Gr_{\mu_\bullet},\IC)\big),
\end{multline*}
obtained by the pushforward of the composition of cohomological correspondences along the perfectly proper morphism
\[\on{Comp}:\Gr_{\kappa_\bullet|\lambda_\bullet}^0\times_{\Gr_{\lambda_\bullet}}\Gr_{\lambda_\bullet|\mu_\bullet}^0\to \Gr_{\kappa_\bullet|\mu_\bullet}^0.\]

\item There is a canonical isomorphism 
$$ \Hom_{\on{P}(\Gr)} \big((m_{\la_\bullet})_*\IC_{\la_\bullet},(m_\mmu)_*\IC_{\mmu}\big)\cong \on{H}^{\on{BM}}_{\langle2\rho,|\la_\bullet|+|\mmu|\rangle}(\Gr_{\la_\bullet|\mmu}^0).$$
\end{enumerate} 
\end{prop}

\begin{proof}
These facts are well-known.  But due to the importance, we sketch the arguments here.

(1) By the semismallness, the fiber of $m_\mmu$ over the point $\varpi^\nu$ is $\leq \langle\rho, |\mmu|-\nu\rangle$. It follows that 
$$\dim \Gr^0_{\la_\bullet\mid \mmu}\times_{\Gr}\mathring{\Gr}_{\nu}\leq \langle\rho, |\lambda_\bullet|-\nu\rangle + \langle\rho, |\mu_\bullet|-\nu\rangle + \dim \mathring{\Gr}_\nu =\langle \rho, |\la_\bullet|+|\mmu|\rangle$$ for any $\nu$.
The dimension estimate of $\Gr^0_{\la_\bullet\mid \mmu}$ then follows.  

(2) We write the diagram
\[
\xymatrix{
\Gr_{\la_\bullet|\mmu}^0  \ar[d]_{h^{\leftarrow}_{\la_\bullet}} \ar[r]^{h^{\rightarrow}_{\mmu}}
& \Gr_\mmu \ar[d]_{m_\mmu}
\\
\Gr_{\lambda_\bullet} \ar[r]^-{m_{\lambda_\bullet}} & \Gr.
}
\]
Then by proper base change,
\begin{eqnarray*}
\Hom_{\on{P}(\Gr)} ((m_{\la_\bullet})_*\IC,(m_\mmu)_*\IC) & \cong & \Hom_{\on{D}^b_c(\Gr_{\mmu})} ((m_\mmu)^*(m_{\la_\bullet})_*\IC,\IC) \\
& \cong & \Hom_{\on{D}^b_c(\Gr_\mmu)} ((h^{\rightarrow}_{\mmu})_*(h^{\leftarrow}_{\la_\bullet})^*\IC,\IC)\\
& \cong & \Hom_{\on{D}^b_c(\Gr_{\la_\bullet\mid\mmu}^0)} ((h^{\leftarrow}_{\la_\bullet})^*\IC,(h^{\rightarrow}_{\mmu})^!\IC)\\
& = & \on{Corr}_{\Gr_{\la_\bullet|\mmu}^0}((\Gr_{\la_\bullet},\IC),(\Gr_{\mmu},\IC)).
\end{eqnarray*}
The last assertion of (2) follows from the definition of the composition of cohomological correspondences (see Appendix \ref{Sec:cohomological correspondence}).

(3) 
Let $\mathring{\Gr}_{\la_\bullet}\subset \Gr_{\la_\bullet}$ (resp. $\mathring{\Gr}_{\mmu}\subset \Gr_{\mmu}$) denote the open subset formed by the twisted product of the open Schubert cells, and let $\mathring{\Gr}_{\la_\bullet\mid\mmu}^0=(h^{\leftarrow}_{\la_\bullet})^{-1}(\mathring{\Gr}_{\la_\bullet})\cap (h^{\rightarrow}_{\mmu})^{-1}(\mathring{\Gr}_{\mmu})$. By restriction to the above open subsets, we have
\begin{equation}
\label{E: res of Sat corr}
\on{Corr}_{\Gr_{\la_\bullet|\mmu}^0}((\Gr_{\la_\bullet},\IC),(\Gr_{\mmu},\IC))
\to \on{Corr}_{\mathring{\Gr}_{\la_\bullet|\mmu}^0}((\mathring{\Gr}_{\la_\bullet},\Ql[\langle2\rho,|\la_\bullet|\rangle]),(\mathring{\Gr}_{\mmu},\Ql[\langle2\rho,|\mmu|\rangle])).
\end{equation}
Note that the latter space is canonically isomorphic to
\[\on{H}^{\on{BM}}_{\langle2\rho,|\la_\bullet|+|\mmu|\rangle}(\mathring{\Gr}^{0}_{\la_\bullet\mid\mmu})\cong\on{H}^{\on{BM}}_{\langle2\rho,|\la_\bullet|+|\mmu|\rangle}(\Gr^{0}_{\la_\bullet\mid\mmu}),\]
where the isomorphism follows from Part (1). It remains to show that \eqref{E: res of Sat corr} is an isomorphism.
First, if $\la_\bullet=\la$ is a single coweight, this is standard. For example, see \cite[Corollary 5.1.5]{Z16} for a proof. The general case can be reduced to the single coweight case as follows. Let $u\in \Hom(\IC_\la,(m_{\la_\bullet})_*\IC)$, considered as a cohomological correspondence by Part (2). Assume that its restriction to $\mathring{\Gr}_{\la_\bullet\mid\mmu}^0$ is zero. Then it is easy to see that for every $v\in \Hom_{\on{P}(\Gr)} ((m_{\la_\bullet})_*\IC,(m_\mmu)_*\IC)$, the composition $u\circ v$, when regarded as a cohomological correspondence, becomes zero when restricted to $\mathring{\Gr}^0_{\la\mid\mmu}$. Therefore $u\circ v$ is zero by the single coweight case. This implies that $u=0$. Therefore, \eqref{E: res of Sat corr} is injective. On the other hand,
\[\dim\Hom_{\on{P}(\Gr)} ((m_{\la_\bullet})_*\IC,(m_\mmu)_*\IC)=\sum_\la\dim \Hom(\IC_\la,(m_{\la_\bullet})_*\IC)\dim \Hom(\IC_\la, (m_{\mu_\bullet})_*\IC)\]
and
\[\dim\on{H}^{\on{BM}}_{\langle2\rho,|\la_\bullet|+|\mmu|\rangle}(\Gr^{0}_{\la_\bullet\mid\mmu})=\sum_\la\dim\on{H}^{\on{BM}}_{\langle2\rho,\la+|\la_\bullet|\rangle}(\Gr^{0}_{\la\mid\la_\bullet})\dim\on{H}^{\on{BM}}_{\langle2\rho,\la+|\mu_\bullet|\rangle}(\Gr^{0}_{\la\mid\mu_\bullet}).\]
It follows by the single coweight case that the two spaces have the same dimension. Therefore, \eqref{E: res of Sat corr} is an isomorphism.
\end{proof}
\begin{dfn}
\label{D: SatCyc}
An irreducible component of $\Gr_{\la_\bullet|\mmu}^0$ of dimension $\langle\rho,|\la_\bullet|+|\mmu|\rangle$ is called a \emph{Satake cycle}.
The set of Satake cycles in $\Gr_{\la_\bullet|\mmu}^0$ is denoted by $\bS_{\la_\bullet|\mmu}$. For $\bba\in \bS_{\la_\bullet|\mmu}$, the cycle labelled by $\bba$ is denoted by $\Gr_{\la_\bullet|\mmu}^{0,\bba}$.
\end{dfn}
\begin{rmk}
\begin{enumerate}
\item Since $\Gr_{\la_\bullet\mid \mmu}^0=\Gr_{\mmu\mid \la_\bullet}^0$, $\bS_{\la_\bullet\mid\mmu}=\bS_{\mmu\mid\la_\bullet}$.
\item
In general, $\Gr_{\lambda_\bullet|\mu_\bullet}^{0}$ is not equi-dimensional (see \cite{haines}). So the union of $\Gr_{\lambda_\bullet|\mu_\bullet}^{0, \bba}$ for all $\bba \in \SSS_{\lambda_\bullet|\mmu}$ need not be the whole $\Gr_{\lambda_\bullet|\mu_\bullet}^{0}$.

\item Later, we will identify the vector space appearing in Part (2) with another vector space, using the geometric Satake isomorphism, which justifies the name. 
\end{enumerate}
\end{rmk}

Now, let $\nu_\bullet$ and $\xi_\bullet$ be two other (not necessarily non-empty) sequences of dominant coweights of length $a$ and $b$ respectively.
We define 
\begin{equation}
\label{E: cond on Grnuzeta}\Gr^0_{\nu_\bullet; \la_\bullet\mid \mmu; \xi_\bullet}:= \Gr_{\nu_\bullet}\tilde\times \Gr_{\la_\bullet\mid \mmu}^0\tilde\times \Gr_{\xi_\bullet}.
\end{equation}
This is a closed subscheme of $\Gr^0_{(\nu_\bullet,\la_\bullet,\xi_\bullet)\mid (\nu_\bullet, \mmu, \xi_\bullet)}$, classifying those diagrams \eqref{E:corr} that induce
\[
\mE_i=\mE'_i,\ \ \ 0\leq i\leq a,\quad \mE_{t-j}=\mE'_{t-j}, \ \ \ 0\leq j\leq b,
\]
or more explicitly, the diagrams
\[
\xymatrix@C=14pt{
\calE'_{a+b+s} \ar@{}[rrd]|{\cdots \ \cdots} \ar@{-->}[r] \ar@{=}[d] &\cdots \ar@{-->}[r] & \calE'_{a+s+1} \ar@{-->}[r] \ar@{=}[d] & \calE'_{a+s}\ar@{=}[d] \ar@{-->}[r] & \calE'_{a+s-1} \ar@{-->}[r] & \cdots \ar@{-->}[r] & \calE'_{a+1} \ar@{-->}[r] & \calE'_{a} \ar@{-->}[r]\ar@{=}[d] & \calE'_{a-1}\ar@{}[rrd]|{\cdots \ \cdots}  \ar@{-->}[r]\ar@{=}[d] & \cdots \ar@{-->}[r] & \calE'_0 \ar@{=}[d]\ar@{=}[r] &\mE^0\ar@{=}[d]
\\
\calE_{a+b+t} \ar@{-->}[r]& \cdots \ar@{-->}[r] & \calE_{a+t+1} \ar@{-->}[r]  & \calE_{a+t} \ar@{-->}[r] & \calE_{a+t-1} \ar@{-->}[r]& \cdots \ar@{-->}[r] & \calE_{a+1} \ar@{-->}[r] & \calE_{a} \ar@{-->}[r] & \calE_{a-1} \ar@{-->}[r] & \cdots \ar@{-->}[r]& \calE_0  \ar@{=}[r]& \mE^0,
}
\]

To finish this subsection, we prove the following lemma, which will only be used in \S\ref{Sec:S=T}. Recall the notion of cohomologically smooth morphisms in Definition \ref{D: coh smooth}.
\begin{lem}
\label{L: reversing modification}
Assume that $n$ is sufficiently large. Then $gL^+G^{(n)} \mapsto g^{-1}L^+G$ induces a cohomologically smooth morphism 
$\Gr_{\mu}^{(n)}\to \Gr_{\mu^*}$.
\end{lem}
\begin{proof}
Note that the fiber over each $g$ is isomorphic to $g L^+G g^{-1} / L^+G^{(n)} \cong L^+ G / g^{-1} L^+ G^{(n)} g$, which is a $L^+G$-homogeneous space and therefore is perfectly smooth. 
This observation in fact already implies that $\on{rev}: \mathring{\Gr}^{(n)}_{\mu}\to \mathring{\Gr}_{\mu^*}$ is perfectly smooth of relative dimension $n\dim G$, where $\mathring{\Gr}_{\mu}^{(n)}$ is the pre-image of the Schubert cell $\mathring{\Gr}$ in $\Gr_\mu^{(n)}$. So the lemma holds if $\mu$ is minuscule.

If $\mu=\mu^*$ is quasi-minuscule, there is a canonical resolution $\widetilde\Gr_{\mu}\to \Gr_{\mu}$ as from \cite[Lemma 2.12]{Z} (note that $\Gr_{\mu}$ was denoted by $\Gr_{\leq \mu}$ there), with geometrically connected fibers.

Now, now let $\mu=\mu_1+\cdots+\mu_r$, where each $\mu_i$ is minuscule or quasi-minuscule. Let $\widetilde\Gr_{\mu_i}=\Gr_{\mu_i}$ if $\mu_i$ is minuscule. We can form the convolution product $\widetilde\Gr_{\mmu}:=\widetilde\Gr_{\mu_1}\tilde\times\cdots\tilde\times\widetilde{\Gr}_{\mu_r}$, which is a resolution of $\Gr_\mu$. Write $\widetilde\Gr_{\mu^*_\bullet}:=\Gr_{\mu^*_r}\tilde\times\cdots\tilde\times\Gr_{\mu^*_1}$.
As in the proof of \cite[Lemma 2.23, Lemma 2.24]{Z}, the base change of $\Gr_{\mu}^{(n)}\to \Gr_{\mu^*}$ along this resolution is $\widetilde\Gr_{\mmu}^{(n)}\to \widetilde\Gr_{\mu^*_\bullet}$, which by Lemma \ref{L: criterion coh smooth} (1) is cohomological smooth. Therefore, by Lemma \ref{L: criterion coh smooth} (2), $\Gr_{\mu}^{(n)}\to \Gr_{\mu^*}$ is cohomologically smooth.

Finally, $\Gr$ is the union of Schubert varieties $\Gr_\mu$ with $\mu$ a sum of minuscules and quasi-minuscules. Therefore, the lemma holds for all $\mu$. 
\end{proof}

\begin{rmk}
We expect that the map $\Gr_\mu^{(n)} \to \Gr_{\mu^*}$ is in fact perfectly smooth. This is the case in equal characteristic, but we are not able to prove it in mixed characteristic.
\end{rmk}

\subsection{Geometry of semi-infinite orbits} We also recall the geometry of semi-infinite orbits in the affine Grassmannians and Mirkovi\'c-Vilonen cycles. We need to fix embeddings $T\subset B\subset G$.

\subsubsection{Mirkovi\'c-Vilonen cycles}
For $\lambda$ a coweight of $G$, we use
\[
S_\lambda = LU  \varpi^\lambda L^+G / L^+G
\]
to denote the semi-infinite $LU$-orbit on the affine Grassmannian.

Note that there is a natural projection
\begin{equation}
\label{E:triviallization of U(O) torsor}
\xymatrix@R=0pt{
\pi_\la:LU \ar[r] & S_\la
\\
u \ar@{|->}[r] &  \varpi^\la u,
}
\end{equation}
realizing $LU$ as an $L^+U$-torsor over $S_\la$. In addition, $S_\la=\varpi^\la S_0\to \Gr$ is a locally closed embedding, as explained in \cite[\S 2.2.1]{Z}. 

Recall the following basic result of Mirkovi\'c-Vilonen cycles. 
\begin{theorem}
\label{T:weight space interpretation}
For $\lambda$ and $\mu$ two coweights of $G$ with $\mu$ dominant, every irreducible component of the intersection $S_\lambda \cap \Gr_\mu$ is of dimension $\langle \rho, \lambda+ \mu \rangle$. In addition, the number of its irreducible components equals to the dimension of the $\la$-weight space $V_\mu(\la)$ of the irreducible representation $V_\mu$ of $\hat G$ of highest weight $\mu$.
\end{theorem}
\begin{proof}For the proof, see \cite[Lemma 2.17.4, Proposition 5.4.2]{GHKR} when $\cha F >0$ and \cite[Corollary 2.8]{Z} when $\cha F  =0$. See also \cite[Theorem~3.2(a)]{MV} for the proof of the dimension formula.
\end{proof}

\begin{notation}
Irreducible components of $S_\lambda \cap \Gr_\mu$ are often referred to as the \emph{Mirkovi\'c-Vilonen cycles}. We will denote by $\bM\bV_\mu(\la)$ the set of irreducible components of $S_\la\cap\Gr_\mu$, and for $\bbb\in \MV_\mu(\la)$, write $(S_\la\cap\Gr_\mu)^\bbb$ the irreducible component (a.k.a. MV cycle) labeled by $\bbb$.
\end{notation}

The following are some additional facts about MV cycles.
\begin{lem}
\label{L: basic prop of MV}
\begin{enumerate}
\item The point $\varpi^\la\in (S_\la\cap\Gr_\mu)^\bbb$ for every $\bbb\in \MV_\mu(\la)$.
\item The variety $(S_\la\cap\Gr_\mu)^\bbb$ is stable under the action of $L^+B$.
\item The set $\MV_\mu(\la)$ is independent of the choice of the Borel $B$ up to canonical isomorphisms.
\end{enumerate}
\end{lem}
\begin{proof} Recall that $S_\la$ is the attractor for the $\bG_m$-action on $\Gr$ given by $2\rho: \bG_m\to T\to L^+T$. (1) and (2) follow. For (3), let $B_1$ and $B_2$ be two Borel subgroups, and choose $g\in G(\mO_L)$ such that $gB_1g^{-1}=B_2$. For a chosen Borel subgroup $B$, we use $S_\la^B$ and $\MV^B_\mu(\la)$ denote the corresponding semi-infinite orbit and set of MV cycles.
Then the action of $g$ induces an isomorphism $(S^{B_1}_\la\cap \Gr_\mu)\cong (S^{B_2}_\la\cap \Gr_\mu)$ and therefore a bijection $\MV_\mu^{B_1}(\la)\cong \MV_\mu^{B_2}(\la)$. By (2), this bijection is independent of the choice of $g$. 
\end{proof}

\subsubsection{Some explicit description of the MV cycles}
Let us recall explicit descriptions of MV cycles in some special cases.
\begin{enumerate}
\item[(1)] First, assume that $G=\on{PGL}_2$. Let us identify $\xcoch$ with $\bZ$ in the usual way (so that $1\in\bZ$ corresponds to the dominant cocharacter $t\mapsto \begin{pmatrix}t&\\ &1\end{pmatrix}$). Then $S_\ell\cap \Gr_m$ is non-empty if and only if $|\ell|\leq m$ and $2\mid m-\ell$. In this case, there is a natural isomorphism
\begin{equation}\label{E: MV sl2}
(\bA^{\frac{\ell+m}{2}})^\pf\simeq S_\ell\cap \Gr_m,\quad (u_{\frac{\ell-m}{2}},\ldots,u_{\ell-1})\mapsto \begin{pmatrix}1& \sum\limits_{\frac{\ell-m}{2}\leq i\leq \ell-1} [u_i]\varpi^i\\ &1 \end{pmatrix}\begin{pmatrix}\varpi^\ell& \\ & 1\end{pmatrix} \mod L^+G.
\end{equation}
In particular, $S_\ell\cap \Gr_m$ is irreducible.

\item[(2)] Next, we assume that $G$ is general, but $\mu$ is a minuscule or quasi-minuscule cocharacter of $G$. Recall that a dominant coweight $\mu$ of $G$ is called (quasi-)minuscule if all (non-zero) weights of the irreducible representation $V_\mu$ of $\hat G$ are in a single orbit under the Weyl group. Recall that if $G$ is a simple group not of type $A$, then the unique quasi-minuscule (but non-minuscule) coweight is the unique \emph{short} dominant coroot. The following statements summarize results of \cite[\S 6-\S 7]{NP} and \cite[\S 2.2.2]{Z}.

If $\mu$ is a minuscule coweight of $G$, then $\Gr_\mu=\mathring{\Gr}_\mu$ and $S_\la\cap\Gr_\mu$ is non-empty if and only if $\la=w\mu$ for some $w\in W$. In this case
\begin{equation}\label{E: MV min}
S_\la\cap \Gr_\mu=L^+U\varpi^\la L^+G/L^+G\cong  L^+U \varpi^\la L^+U/L^+U
\end{equation}
is irreducible, and the projection \eqref{E: Hodge to fil} maps it isomorphically to $(\bar{U}w\bar P_{\mu}/\bar P_{\mu})^\pf$.

Next, we consider the case when $\mu$ is quasi-minuscule. Then $\mu$ is a dominant coroot. Let $\Delta_\mu$ denote the set of simple coroots that lie in the $W$-orbits of $\mu$.
Then $\Gr_\mu=\mathring{\Gr}_\mu\sqcup \Gr_0$ and $S_\la\cap\Gr_\mu$ is non-empty if and only if $\la\in W\mu$ or $\la=0$. 

If $\la=w\mu$, then
\begin{equation}\label{E: MV quasiminI}
S_\la\cap \Gr_\mu=L^+U\varpi^\la L^+G/L^+G\cong  L^+U \varpi^\la L^+U/L^+U
\end{equation}
is irreducible, and the projection \eqref{E: Hodge to fil} maps $S_\la\cap \Gr_\mu$ onto $(\bar{U}w\bar P_{\mu}/\bar P_{\mu})^\pf$. In addition, 
\begin{itemize}
\item if $\la=w\mu\succ 0$ is a positive coroot,  then $S_\la\cap \Gr_\mu=\pr_\mu^{-1}(\bar{U}w\bar P_{\mu}/\bar P_{\mu})^\pf)$, and
\item  if $\la=w\mu\prec0$ is a negative coroot, then $\pr_\mu: S_\la\cap \Gr_\mu\to (\bar{U}w\bar P_{\mu}/\bar P_{\mu})^\pf$ is an isomorphism.
\end{itemize}

If $\la=0$, then
\[S_0\cap \Gr_\mu=\bigcup_{\al\in \Delta_\mu} (S_0\cap \Gr_\mu)^{\al},\]
where $(S_0\cap \Gr_\mu)^{\al}$ is an irreducible component of $(S_0\cap\Gr_\mu)$ whose $\bar k$-points can be written as
\begin{equation}\label{E: MV quasiminII}
(S_0\cap \Gr_\mu)^{\al}=\Gr_0\,\bigsqcup\, \bigg(\pr_{\mu}^{-1}\Big(\bigsqcup_{w\in W/W_\mu, w\mu\preceq -\al} (\bar Uw\bar P_\mu)^\pf \Big)\,\big\backslash \bigsqcup_{w\mu\preceq -\al} (S_{w\mu}\cap \Gr_\mu)\bigg).
\end{equation}
In this way, we establish a canonical bijection $\MV_\mu(0)\cong \Delta_\mu$.
\end{enumerate}

\subsubsection{Convoluted product of MV cycles}
Note that it makes sense to talk about convolutions of semi-infinite orbits.
Let $\la_\bullet = (\lambda_1, \dots, \lambda_n)$ be a sequence of coweights of $G$. One can define
\[S_{\la_\bullet}:=S_{\la_1}\tilde\times S_{\la_2}\tilde\times\cdots\tilde\times S_{\la_n}\subset \Gr\tilde\times \Gr\tilde\times\cdots\tilde\times\Gr.\]

The formula
\begin{equation}\label{fact for}
(\varpi^{\la_1}x_1,\ldots,\varpi^{\la_n}x_n)\mapsto (\varpi^{\la_1}x_1,\varpi^{\la_1+\la_2}(\varpi^{-\la_2}x_1\varpi^{\la_2})x_2,\ldots)
\end{equation}
with $x_i \in LU$
defines an isomorphism 
\begin{equation}\label{factorizationI}
\tilde m: S_{\la_\bullet}\xrightarrow{ \ \cong\ } S_{\la_1}\times S_{\la_1+\la_2}\times\cdots\times S_{|\la_\bullet|},
\end{equation}
so that the following diagram is commutative
\[
\xymatrix{
S_{\la_\bullet}\ar[r]^-{\tilde m} \ar[d] & S_{\la_1}\times S_{\la_1+\la_2}\times\cdots\times S_{|\la_\bullet|}
\ar[d]
\\
\Gr\tilde\times \Gr\tilde\times\cdots\tilde\times\Gr \ar[r]^-{\tilde m} & \Gr^n,
}\]
where $\tilde m$ on the convoluted affine Grassmannian is the product of the convolution map for the first $i$ factors.

Let $\mmu=(\mu_1, \dots, \mu_n)$ be a sequence of dominant coweights.
Note that each $S_{\la_i}\cap\Gr_{ \mu_i}$ is $L^+U$-invariant,
so it also makes sense to convolve the $L^+U$-varieties $S_\la\cap\Gr_{ \mu}$, and there is a canonical identification
\begin{equation}\label{factorizationII}
(S_{\la_1}\cap\Gr_{ \mu_1})\tilde\times\cdots\tilde\times (S_{\la_n}\cap \Gr_{ \mu_n})\cong S_{\la_\bullet}\cap \Gr_{\mmu}.
\end{equation}
The set of irreducible components of $(S_{\la_1}\cap\Gr_{ \mu_1})\tilde\times\cdots\tilde\times (S_{\la_n}\cap \Gr_{ \mu_n})$ is given by $\prod_i\MV_{\mu_i}(\la_i)$, and for $\bbb_\bullet:=(\bbb_1,\ldots,\bbb_n)\in\prod_i\MV_{\mu_i}(\la_i)$, the irreducible component labeled by $\bbb_\bullet$ is
\[(S_{\la_1}\cap\Gr_{ \mu_1})^{\bbb_1}\tilde\times\cdots\tilde\times (S_{\la_n}\cap \Gr_{ \mu_n})^{\bbb_n}.\]

Here is the relationship between MV cycles and Satake cycles.
\begin{lem}
\label{L:Sat vs MV}
Let $\nu,\mu,\nu+\la$ be dominant coweights.
There is a unique injective map
\[i_\nu^{\MV}: \bS_{(\nu,\mu)\mid \la+\nu}\to \MV_{\mu}(\la),\]
such that for every $\bba\in \bS_{(\nu,\mu)\mid \la+\nu}$
\begin{equation}
\label{E:description Z(S)}
 \Gr_{(\nu, \mu)\mid\la+\nu}^{0, \bba} \cap S_{\nu, \lambda} = (S_\nu \cap \Gr_\nu) \tilde \times (S_{\lambda} \cap \Gr_\mu)^{i^{\MV}_\nu(\bba)}
\end{equation}
as subschemes of $\Gr\tilde\times \Gr$.

In addition, this map is compatible with the bijection $\MV_\mu^{B_1}(\la)\cong \MV_\mu^{B_2}(\la)$ in Lemma \ref{L: basic prop of MV}, for different choices of Borel subgroups.
\end{lem}
\begin{proof}
To reduce the notation load, we denote $ \Gr_{(\nu, \mu)\mid\la+\nu}^{0, \bba} \cap S_{\nu, \lambda}$ by $Z$ in the proof.

Let $\pr_1:\Gr\tilde\times\Gr\to \Gr$ denote the projection to the first factor and $m:\Gr\tilde\times\Gr\to \Gr$ the convolution map. Then we have an isomorphism 
$(\pr_1,m):\Gr\tilde\times\Gr\simeq \Gr\times \Gr$, given by 
\begin{equation}
\label{E:E2 to E1 to E0}
 (\mE_2\stackrel{\beta_1}{\dashrightarrow} \mE_1\stackrel{\beta_0}{\dashrightarrow} \mE_0=\mE^0)\longmapsto (\mE_1\stackrel{\beta_1}{\dashrightarrow}\mE_0 = \mE^0), \ (\mE_2\stackrel{\beta_0\beta_1}{\dashrightarrow} \mE_0=\mE^0).
 \end{equation}
Note that by \eqref{factorizationI}, $S_{\nu,\la}\cap (\Gr_{\nu}\times \Gr_{\la+\nu})=(\Gr_{\nu}\cap S_\nu)\times (\Gr_{\la+\nu}\cap S_{\la+\nu})$ may be viewed as an open subspace of $\Gr_\nu\times\Gr_{\lambda+\nu}$. As  $\Gr_{(\nu, \mu)\mid\la+\nu}^{0, \bba}\subset \Gr_{\nu}\times \Gr_{\lambda+\nu}$ is closed, the intersection
$Z$ is an open subset of $\Gr_{(\nu, \mu)\mid\lambda + \nu}^{0, \bba}$, and therefore is irreducible of dimension $\langle \rho, \la+\mu+2\nu\rangle$ if the intersection is non-empty.

On the other hand by \eqref{factorizationII}, $Z\subset (\Gr_\nu\cap S_\nu)\tilde\times(\Gr_{\mu}\cap S_\la)$, when regarded as subschemes in $\Gr\tilde\times\Gr$.
It would then follow that if $Z$ is nonempty, it is equal to an irreducible component of $(S_\nu \cap \Gr_\nu) \tilde \times (S_\lambda \cap \Gr_\mu)$, which must be of the form  $(S_\nu \cap \Gr_\nu) \tilde \times (S_{\lambda} \cap \Gr_\mu)^{i^{\MV}_\nu(\bba)}$ for some $i^{\MV}_\nu(\bba)\in \MV_\mu(\la)$. This defines the desired map in the lemma, and it is clear that this map is compatible with the bijection $\MV_\mu^{B_1}(\la)\cong \MV_\mu^{B_2}(\la)$ in Lemma \ref{L: basic prop of MV}, for different choices of Borel subgroups.

We shall show that $m_{\nu,\mu}^{-1}(\varpi^{\lambda+\nu}) \cap Z$ is non-empty, where $m_{\nu,\mu}: \Gr_\nu\tilde\times\Gr_\mu\to \Gr$ is the convolution map.
First note by the definition of $\Gr_{(\nu, \mu)\mid\lambda + \nu}^{0, \bba}$, $m_{\nu,\mu}^{-1}(\varpi^{\lambda+\nu}) \cap \Gr_{(\nu, \mu)\mid\lambda + \nu}^{0, \bba}$ is a (non-empty irreducible) variety of dimension $\langle \rho, \mu- \lambda \rangle$. 
Since $\Gr_\nu = \sqcup_{\nu'\preceq \nu}(S_{\nu'}\cap \Gr_{\nu})$, it suffices to show that, for $\nu'\prec \nu$, the intersection
\begin{equation}
\label{E:smaller dim intersection}
(\pr_1)^{-1} \big( S_{\nu'}\cap \Gr_{\nu}\big) \cap m_{\nu,\mu}^{-1}(\varpi^{\lambda+\nu})
\end{equation}
has dimension $< \langle \rho, \mu- \lambda \rangle$. 

Note that
$m_{\nu,\mu}^{-1}(\varpi^{\la+\nu})$
classifies sequence of modifications \eqref{E:E2 to E1 to E0} such that $\inv(\beta_0)\preceq \nu,\ \inv(\beta_1)\preceq \mu$ and the composition $\beta_1 \circ \beta_0=\varpi^{\lambda+\nu}$ if $\calE_2$ is properly trivialized.
Rearranging the maps gives
\[
\xymatrix@C=40pt{
\calE_1 \ar@{-->}@/^10pt/[rr]^-{\beta_1^{-1}} \ar@{-->}[r]_-{\beta_0} & \calE_0 \ar@{-->}[r]_-{\varpi^{-\la-\nu}} & \calE_2.
}
\]
So $m_{\nu,\mu}^{-1}(\varpi^{\la+\nu})\cong \Gr_{\mu^*} \cap \varpi^{-\lambda-\nu} \Gr_\nu$, and under this isomorphism,
\[
\eqref{E:smaller dim intersection} \cong \Gr_{\mu^*} \cap \varpi^{-\la-\nu}(S_{\nu'}\cap\Gr_{ \nu}) \subseteq \Gr_{\mu^*} \cap S_{-\la-\nu+\nu'}
\]
is of dimension $<\langle\rho, \mu-\la\rangle$.
This concludes the proof of the lemma.
\end{proof}

More generally
\begin{lem}
\label{L: decomp MV into Satake}
Let $\mu_1,\mu_2,\mu$ be dominant.
There is a natural injective map
\[S_{(\mu_1,\mu_2)\mid \mu}\times \MV_\mu\to \MV_{\mu_1}\times \MV_{\mu_2}.\]
\end{lem}
\begin{proof}
Consider $m_{\mu_1,\mu_2}:\Gr_{\mu_1,\mu_2}\to \Gr$. Let $\bba\in\MV_\mu(\la)$ and $(S_\la\cap\Gr_{\mu})^{\bba}$ the corresponding irreducible component. Let $\bbb\in\bS_{(\mu_1,\mu_2)\mid \mu}$. 
As $\Gr^{0,\bbb}_{(\mu_1,\mu_2)\mid \mu}|_{\mathring{\Gr}_{\mu}}\to \mathring{\Gr}_\mu$ is a fibration with (non-canonically) isomorphic fibers of dimension $\langle\rho,\mu_1+\mu_2-\mu\rangle$, we see  that $C:=m_{\mu_1,\mu_2}^{-1}(S_\la\cap\Gr_{\mu})^{\bba}\cap \Gr^{0,\bbb}_{(\mu_1,\mu_2)\mid \mu}|_{\mathring{\Gr}_{\mu}}$ is irreducible of dimension $\langle\rho,\mu_1+\mu_2+\la\rangle$. On the other hand, by \eqref{factorizationII} and Theorem \ref{T:weight space interpretation}, every irreducible component of $m_{\mu_1,\mu_2}^{-1}(S_\la)\cap\Gr_{\mu_1,\mu_2}$ is of dimension $\langle\rho,\mu_1+\mu_2+\la\rangle$, and is of the form $(S_{\nu_1}\cap\Gr_{\mu_1})^{\mathbf{c}_1}\tilde\times(S_{\nu_2}\cap\Gr_{\mu_2})^{\mathbf{c}_2}$ for unique $\mathbf{c}_i\in \MV_{\mu_i}(\nu_i)$. Therefore, $C$ is an open subset of exactly one of such an irreducible components. We thus construct the desired map given by $(\bba,\bbb)\mapsto(\mathbf{c}_1,\mathbf{c}_2)$. It is clear that this is an injective map
\end{proof}
\begin{rmk}
(1) It is easy to see from the proof that under the above map, the image of $S_{(\mu_1,\mu_2)\mid \mu}\times \MV_\mu(\mu)$ is contained in $\MV_{\mu_1}(\mu_1)\times \MV_{\mu_2}(\mu-\mu_1)$. Therefore, it recovers the map  from Lemma \ref{L:Sat vs MV}.

(2) As we shall see from the proof of Proposition \ref{comb decom} below, this map is in fact a bijection.
\end{rmk}

\subsubsection{Relation with Levi subgroups}
\label{SS:theta M}
Now let $M\subset G$ be a standard Levi subgroup, with the set of simple roots $\Delta_M\subset \Delta$. Let $U^M\subset U$ be the unipotent radical of the parabolic subgroup $P_M=MB$. Then $U/U^M$ is isomorphic to the unipotent radical $U_M$ of the Borel subgroup $B_M=M\cap B$. If $\Delta_M=\{\al_i\}$ is a simple root (in which case we call $M$ the \emph{standard Levi corresponding to $\al_i$}), then $U_M\cong U_{\al_i}$.
Recall that $T$ is a common maximal torus of $M$ and $G$, so semi-infinite orbits on $\Gr_M$ (with respect to $U_M$) are also parameterized by $\xcoch(T)$. 
Let $S_{M,\la}$ be a semi-infinite orbit of $\Gr_M$ corresponding to $\la$. 
Then there is a natural projection
\[\theta_M: S_\la\to S_{M,\la},\quad \varpi^\la LU/L^+U\to \varpi^\la LU_M/L^+U_M.\]

\subsection{Crystal structures on the set of MV cycles}
\label{S:geom realization of G-crystal}
In this subsection, we endow a $\hat G$-crystal structure on the set of MV cycles via the Littelmann paths.\footnote{Unfortunately, in later section of the paper, the terminology $G$-crystal will appear with a completely different meaning. Namely, it will be crystals in the sense of Grothendieck. The use of $\hat G$-crystal will be confined in this and the next section. We hope no confusion will arise.} The main result we need for later application is Proposition \ref{P: geom prop of MV cycle}.

The $\hat G$-crystal structure on $\MV_\mu$ was first discovered by Braverman and Gaitsgory \cite{BG}, using the geometric Satake equivalence.
Our approach is different, and closely follows \cite{NP}. The connection between Littelmann paths (or some variants) and  MV cycles has also been extensively studied in literature (e.g. see \cite{GL05, GL11, BaGa}). In fact, some of the following results are contained or can be easily deduced from the above mentioned references (we thank J. Kamnitzer to point out this). However, as our proof of Proposition \ref{P: geom prop of MV cycle} makes use of the construction of the $\hat G$-crystal structure on $\MV_\mu$, we decide to make the exposition self-contained.

We start by recalling the definition of $\hat G$-crystals.

\begin{dfn}
Recall that $\Delta$ (resp. $\Delta^\vee$) denotes the set of simple roots (resp. simple coroots) of $G$.

\begin{enumerate}
\item A \emph{(normal) $\hat G$-crystal}\footnote{We will not discuss non-normal crystals in this paper.} is a finite set $\bB$, equipped with a map $\on{wt}:\bB\to \xcoch(T)$,\footnote{In this section, since we chose to work with $\hat G$-crystals, the target of the weight map is the weight lattice of $\hat G$ which is the same as the coweight lattice of $G$. We shall try to make clear the distinction of weights of $G$ and $\hat G$.}  and operators $e_{\al},f_{\al}:\bB\to \bB\cup\{0\}$ for each $ \al\in\Delta$, such that 
\begin{itemize}
\item[(a)]
for every $\bbb\in\bB$, either $e_\al\bbb=0$ or $\on{wt}(e_\al \bbb)=\on{wt}(\bbb)+\al^\vee$, and either $f_\al\bbb=0$ or $\on{wt}(f_\al \bbb)=\on{wt}(\bbb)-\al^\vee$,
\item[(b)]
for all $\bbb, \bbb' \in \bB$ one has $\bbb' = e_\al \cdot \bbb$ if and only if $\bbb = f_\al \cdot \bbb'$, and
\item[(c)]
if $\varepsilon_\al,\phi_\al:\bB\to\bZ$, $ \al\in \Delta$ are the maps defined by
\[\varepsilon_\al(\bbb)=\max\{n\mid e_\al^n\bbb\neq 0\} \quad \textrm{and} \quad \phi_\al(\bbb)=\max\{n\mid f_\al^n\bbb\neq 0\},\]
then  we require $\phi_\al(\bbb)-\varepsilon_\al(\bbb)=\langle \alpha,\on{wt}(\bbb)\rangle$.
\end{itemize}
For $\lambda \in \XX_\bullet(T)$, we use $\BB(\lambda)$ to denote the set of elements with weight $\lambda$ for $\hat G$, called the \emph{weight space} with weight $\lambda$ for $\hat G$.

\item Let $\bB_1$ and $\bB_2$ be two $\hat G$-crystals, a morphism $\psi:\bB_1\to\bB_2$ is a map of the underlying sets compatible with $\on{wt}$, $ e_\al$, and $f_\al$ (whenever it makes sense).

\item A $\hat G$-crystal $\bB$ is called a \emph{highest weight crystal} of highest weight $\la$ of $\hat G$ if there exists $\bbb\in \bB$, satisfying $e_\al\bbb=0$ for all $\al$, $\on{wt}(\bbb)=\la$, and $\bB$ is generated from $\bbb$ by operators $f_\al$ (for all $\al \in \Delta$). In this case, such $\bbb$ is unique and is denoted by $\bbb_\la$.

\item Let $\BB$ be a $\hat G$-crystal. The dual $\hat G$-crystal has underlying set $\bbb^*$ for each $\bbb \in \BB$ (setting $0^* = 0$), whose maps are given by 
\[
\wt(\bbb^*) = -\wt(\bbb), \quad e_\al(\bbb^*) = (f_\al\bbb)^*, \quad \textrm{and} \quad f_\al(\bbb^*) = (e_\al\bbb)^*.
\]

\item Let $\bB_1$ and $\bB_2$ be two $\hat G$-crystals. The tensor product $\bB_1\otimes\bB_2$ is the $\hat G$-crystal with underlying set $\bB_1\times\bB_2$, and $\on{wt}(\bbb_1\otimes \bbb_2)=\on{wt}(\bbb_1)+\on{wt}(\bbb_2)$. Following \cite{kashiwara} and \cite{joseph}, the operators $e_\al$ and $f_\al$ are defined by
\begin{align*}
e_\al(\bbb_1\otimes\bbb_2)&=\left\{\begin{array}{ll} e_\al\bbb_1\otimes \bbb_2 & \mbox{ if }  \phi_\al(\bbb_1)\geq \varepsilon_\al(\bbb_2) \\ \bbb_1\otimes e_\al\bbb_2 & \mbox{ otherwise, } \end{array}\right.
\\
f_\al(\bbb_1\otimes\bbb_2)&=\left\{\begin{array}{ll} f_\al\bbb_1\otimes \bbb_2 & \mbox{ if }\phi_\al(\bbb_1)> \varepsilon_\al(\bbb_2)   \\ \bbb_1\otimes f_\al\bbb_2 & \mbox{ otherwise. } \end{array}\right.
\end{align*}
We have 
\begin{align}\label{E: tensor epsilon and phi}
\varepsilon_\al(\bbb_1\otimes \bbb_2) & = \max\{ \varepsilon_\al(\bbb_1), \varepsilon_\al(\bbb_2)- \langle \alpha,\on{wt}(\bbb_1)\rangle \} \quad \textrm{and}
\\ \nonumber \phi_\al(\bbb_1\otimes \bbb_2) &= \max\{ \phi_\al(\bbb_2), \phi_\al(\bbb_1) + \langle \alpha,\on{wt}(\bbb_2)\rangle \}.
\end{align}
Taking tensor product of $\hat G$-crystals is associative, making the category of $\hat G$-crystals a monoidal category. 

\item A family of highest weight $\hat G$-crystals $\{\bB_\la|\la\in\xcoch(T)^+\}$ is called \emph{closed} if for every $\la,\mu$, there is a map of $\hat G$-crystals $\bB_{\la+\mu}\to \bB_\la\otimes\bB_\mu$ sending $\bbb_{\la+\mu}$ to $\bbb_\la\otimes\bbb_\mu$.  Given such a family, and a sequence of dominant weights $\mmu$ of $G$, we put $\BB_\mmu = \BB_{\mu_1} \otimes \cdots \otimes \BB_{\mu_n}$.

\end{enumerate}
\end{dfn}

Recall the following theorem of Joseph (\cite{joseph}).
\begin{thm}[Joseph]
\label{T: joseph}
Assume that $\hat G$ is simply-connected. Then there is a unique (up to unique isomorphism) family of closed highest weight $\hat G$-crystals.
\end{thm}
In general, passing to the adjoint group induces a map $\pr: \xcoch(T)\to \xcoch(T_\mathrm{ad})$, and $\hat G_{\mathrm{ad}}$ is simply-connected. Let $\bB_\mu=\bB_{\pr(\mu)}$. Then $\{\bB_\mu\}$ defines a family of closed highest weight $\hat G$-crystals.

\subsubsection{Weight subsets as homomorphisms for $\hat G$-crystals} 
\label{SS:relation to the G-crystals}
Let $\{\BB_\la\}$ be the family of a closed highest weight $\hat G$-crystals as constructed above.
Let $\lambda \in \XX_\bullet(T)$ be a weight of $\hat G$. If $\nu \in \XX_\bullet(T)^+$ is a dominant weight of $\hat G$ such that $\lambda + \nu$ is also dominant, we have a natural map
\begin{equation}
\label{E:inj crystal}
i_\nu^\bB:\Hom(\bB_{\la+\nu},\bB_\nu\otimes\bB)\to \bB(\la),
\end{equation}
such that $\bbb=i^\BB_\nu(\psi)$ is the unique element in $\bB$ satisfying $\psi(\bbb_{\la+\nu})=\bbb_\nu\otimes\bbb$. (Note that $e_\al(\bbb_{\la +\nu}) = 0$ for all $\al \in \Delta$ forces $\psi(\bbb_{\la+\nu})$ to take such form.) 
For $\bbb\in\bB(\la)$, it is straightforward to check from the definition that
\begin{align}
\label{E:condition to realize as a homomorphism crystal}
\bbb \in \mathrm{Im}i_\nu^\BB & \textrm{ if and only if }\phi_\al(\bbb_{\nu+\lambda}) = \phi_\al(\bbb_\nu \otimes \bbb) \textrm{ for all }\al \in \Delta,
\\
\nonumber & \textrm{ if and only if }\phi_\al(\bbb) \leq \phi_\al(\bbb) - \varepsilon_\al(\bbb) +\phi_\al(\bbb_\nu)  \textrm{ for all }\al \in \Delta,
\\
\nonumber &\textrm{ if and only if }\langle\al,\nu\rangle\geq \varepsilon_\al(\bbb) \textrm{ for all }\al \in \Delta.
\end{align}

\subsubsection{Littelmann path and $\hat G$-crystals}
\label{S: littelmann path model}
Given $\hat G$, there are several ways to construct a family of closed highest weight $\hat G$-crystals (although such family is unique up to a unique isomorphism if $\hat G$ is simply-connected,  by Theorem \ref{T: joseph}).
Here we recall one construction using Littelmann paths. This construction is easily related to the set of MV cycles.
Our presentation follows \cite{NP}\footnote{Littelmann himself considers more general paths, e.g. what he called LS paths.}.
Let 
$$\Min\subset \xch(\hat T)^+\setminus \{0\}=\xcoch(T)^+\setminus\{0\}$$ be the set of minimal elements with respect to the Bruhat order. If the group $G$ is simple of type $\on{A}$, $\Min$ is the set of minuscule weights of $\hat G$. In general, it is the union of minuscule weights and the quasi-minuscule weights of $\hat G$.

\begin{dfn}
Let $\ga_i: [0,1]\to \xch(\hat T)\otimes\bR, \ i=1,2$ be two piecewise linear maps. The \emph{concatenation} of $\ga_1$ and $\ga_2$, denoted by $\ga_1*\ga_2$ is the map given by
\[(\ga_1*\ga_2)(t)=\left\{\begin{array}{ll}\ga_1(2t) & 0\leq t\leq 1/2, \\ \ga_1(1)+\ga_2(2t-1) & 1/2\leq t\leq 1. \end{array}\right.\]

Two piecewise linear maps $\gamma_1, \gamma_2: [0,1] \to \xch(\hat T) \otimes \RR$ are called \emph{reparametrizations} of each other if there exists a strictly increasing piecewise linear function $r: [0,1] \to [0,1]$ such that $r(0) = 0$, $r(1) =1$, and $\gamma_1(t) = \gamma_2(r(t))$.
\end{dfn}

\begin{dfn}\label{LP}
We say a piecewise linear map $\ga: [0,1]\to \xch(\hat T)\otimes\bR$ an \emph{elementary Littelmann path} if it is one of the following forms:
\begin{enumerate}
\item $\ga(t)=\ga_\nu(t):=t\nu$, where $\nu\in W\mu$ for some $\mu\in \Min$;
\item $\ga(t)={^2}\ga_{\al}(t):=\left\{\begin{array}{ll} -t\al^\vee & 0\leq t\leq 1/2, \\ (t-1)\al^\vee & 1/2\leq t\leq 1, \end{array}\right.$ where $\al$ is a simple root such that $\al^\vee\in W\mu$ for some $\mu\in \Min$.
\end{enumerate}

A piecewise linear map $\ga: [0,1]\to \xch(\hat T)\otimes\bR$ is a called a \emph{Littelmann path} if after reparametrization, it is of the  form $\ga=\ga_1*\ga_2*\cdots *\ga_m$, where $\ga_i$ are elementary Littelmann paths.
The set of Littelmann paths up to reparametrization is denoted by $\Pi$.

A \emph{dominant Littelmann path} is a Littelmann path $\ga$ completely lying in the closure of the dominant Weyl chamber.
\end{dfn}

\begin{dfn}\label{path vs seq}
To a Littelmann path $\ga$, we attach a sequence $T(\ga)=\mu_\bullet=(\mu_1,\ldots,\mu_m)\subset \Min$ as follows: If $\ga$ is an elementary Littelmann path of type (1), let $T(\ga)=\mu\in \Min$ such that $\ga(1)\in W\mu$, if $\ga$ is of the type (2), let $T(\ga)=\mu \in \Min$ such that $-2\ga(1/2)\in W\mu$. If $\ga=\ga_1*\cdots*\ga_m$, let $T(\ga)=(T(\ga_1),\ldots,T(\ga_m))$. It is easy to see that $T(\ga)$ is well-defined. We denote $m$ by $\ell(\ga)$, and call it the length of the Littelmann path $\ga$. 
\end{dfn}

Recall that in \cite{Li}, Littelmann defined root operators on $\bZ\Pi$, the free $\bZ$-module generated by $\Pi$.\footnote{In fact, Littelmann defined such operators on the free $\bZ$-modules generated by the set of more general paths.} We recall its definition. Let $\al$ be a simple root of $G$, so the corresponding $\al^\vee$ is a simple root of $\hat G$. We now define the two root operators $e_\al,f_\al:\bZ\Pi\to\bZ\Pi$ associated to $\al$, starting with $e_\al$. 

Let $\ga=\ga_1*\cdots*\ga_m$, with $\ga_i$ elementary. After reparametrization, we can assume that $\ga_i(t)=\ga(\frac{i-1+t}{m})-\ga(\frac{i-1}{m})$.
Let $m_\al=\min_{t\in [0,1]}\{ \langle \alpha, \ga(t)\rangle \}$. If $m_\al>-1$, let $e_\al\ga=0$. If $m_\al\leq -1$, let $t_1$ be minimal among $t$ such that $ \langle\alpha, \ga(t)\rangle=m_\al$. Then there are three possibilities.

(i) If $t_1=\frac{i}{m}$ for some $i$, and in addition, $\langle\alpha, \ga(\frac{i-1}{m}) \rangle=m_\al+1$, we define 
$$e_\al(\ga)=\ga_1*\ga_2*\cdots*\ga_{i-1}*s_{\al^\vee}(\ga_i)*\ga_{i+1}*\cdots*\ga_m,$$ 
where $s_{\al^\vee}$ is the simple reflect on $\xch(\hat T)$ given by $\al^\vee$.

(ii) If $t_1=\frac{i}{m}$ is some $i$ but $\langle\alpha, \ga(\frac{i-1}{m}) \rangle=m_\al+2$ (in particular $\al^\vee$ is a simple root of $\hat G$ that can be conjugated into $\Min$ under $W$, and $\ga_i(t)=\ga_{-\al}(t)$), we define
$$e_\al(\ga)=\ga_1*\ga_2*\cdots*\ga_{i-1}*{^2}\ga_\al(t)*\ga_{i+1}*\cdots*\ga_m,$$ 

(iii) If $t_1=\frac{2i-1}{2m}$ for some $i$ (in particular $\al^\vee$ is a simple root of $\hat G$ that can be conjugated into $\Min$ under $W$, and $\ga_i(t)={^2}\ga_\al(t)$), we define
$$e_\al(\ga)=\ga_1*\ga_2*\cdots*\ga_{i-1}*\ga_\al(t)*\ga_{i+1}*\cdots*\ga_m,$$ 

The root operator $f_\al$ can be defined in a similar way, but we prefer to give another short definition.
\begin{dfn}
Let $\ga:[0,1]\to\xch(\hat T)\otimes\bR$ be a piecewise linear map. We define its dual $\ga^*$ as
\[\ga^*(t)=\ga(1-t)-\ga(1).\]
\end{dfn}
Then it is clear that the dual of a Littelmann path is a Littelmann path. Indeed, $\ga^*_\nu(t)=\ga_{-\nu}(t)$ and ${^2}\ga_\al^*(t)={^2}\ga_\al(t)$. We define $f_\al\ga=(e_\al\ga^*)^*$.

\medskip

Here are some simple properties of the root operators (\cite[\S 2]{Li}).
\begin{lem}
Let $\ga\in\Pi$. 
\begin{enumerate}
\item If $e_\al(\ga)\neq 0$, then $e_\al(\ga)(1)=\ga(1)+\al^\vee$ and $T(e_\al(\ga))=T(\ga)$. In addition, $f_\al e_\al(\ga)=\ga$. Similar statements hold for $f_\al$.
\item The path $\ga$ is dominant if and only if $e_\al(\ga)=0$ for all $\al$. 
\end{enumerate}
\end{lem}

Now let $\ga$ be a dominant Littelmann path, with $\ga(1)=\mu$ (not necessarily in $\Min$). Let $\bB_{\ga}$ denote the set of all Littelmann paths obtained by applying root operators to $\ga$. 
Recall that in \cite{Li}, Littelmann proved the following isomorphism theorem.

\begin{thm}[Littelmann]
\label{T: Littelmann isom}
Let $\ga_1$ and $\ga_2$ be two dominant Littelmann paths satisfying $\ga_1(1)=\ga_2(1)$. Then the assignment $\ga_1\to \ga_2$ extends to a unique isomorphism $i_{\ga_1,\ga_2}:\bB_{\ga_1}\to \bB_{\ga_2}$ compatible with the action of root operators.
\end{thm}

As a result, for every dominant weight $\mu$ of $\hat G$, one can canonically define a set $\bB_\mu:=\bB_{\ga}$ for any dominant Littelmann path $\ga$ satisfying $\ga(1)=\mu$. The following theorem can be found in \cite{joseph}.

\begin{thm}The collection $\{\bB_\mu\}$ defined as above together with the root operators form a closed family of highest weight $\hat G$-crystals.
\end{thm}

Next, we relate the Littelmann paths with the MV cycles.
\begin{prop}\label{MV vs path}
Given a sequence of dominant weights $\{\mu_\bullet\}\subset \Min$ of $\hat G$ and $\la\in\xch(\hat T)$, there is a canonical bijection between 
\[\{ \ga\in\Pi\mid  \ga(1)=\la,\ T(\ga)=\mu_\bullet\}\leftrightarrow\{\mbox{Irreducible components of } m_{\mmu}^{-1}(S_{\la})\cap \Gr_{\mu_\bullet}\}.\]
\end{prop}
\begin{proof} Recall that
\[m_{\mmu}^{-1}(S_\la)\cap \Gr_{ \mu_\bullet}=\bigsqcup_{\nu_\bullet,|\nu_\bullet|=\la} S_{\nu_\bullet}\cap\Gr_{ \mu_\bullet},\]
$$S_{\nu_\bullet}\cap\Gr_{ \mu_\bullet}= (S_{\nu_1}\cap\Gr_{\mu_1})\tilde\times\cdots\tilde\times(S_{\nu_n}\cap\Gr_{ \mu_n}).$$ 
By \eqref{E: MV min}, \eqref{E: MV quasiminI}, and \eqref{E: MV quasiminII}, we have
\begin{itemize}
\item if $\nu_i\in W\mu_i$, $S_{\nu_i}\cap \Gr_{\mu_i}$ is irreducible; 
\item if $\nu_i=0$, and $\mu_i$ is quasi-minuscule, then the irreducible components of $S_0\cap \Gr_{ \mu_i}$ canonically one-to-one correspond to $\Delta_{\mu_i}$, given by $ (S_0\cap \Gr_{\mu_i})^\al \leftrightarrow \alpha$; and 
\item if $\nu_i$ is not of the above forms, then $S_{\nu_i}\cap\Gr_{\mu_i}$ is empty. 
\end{itemize}
Therefore, given an irreducible component of $m_{\mmu}^{-1}(S_\la)\cap \Gr_{ \mu_\bullet}$, we construct a Littelmann path 
$\ga=\ga_1*\cdots*\ga_n$, where 
\begin{itemize}
\item $\ga_i(t)=\ga_{\nu_i}(t)$ if $\nu_i\neq 0$;
\item $\ga_i(t)={^2}\ga_{\al_i}(t)$ if $\nu_i=0$ and $(S_0\cap\Gr_{\mu_i})^{\al_i}$ with $ \al_i\in\Delta_{\mu_i}$ appears as a factor in the above factorization of the irreducible component. 
\end{itemize}
It is clear how to construct the inverse map, so we get the bijection as in the proposition.
\end{proof}

For $\ga$ a Littelmann path with $\la=\ga(1)$ and $T(\ga)=\mmu$, let 
$$S_\ga:=(m_{\mmu}^{-1}(S_\la)\cap \Gr_{\mu_\bullet})^\ga$$ denote the corresponding irreducible component. The following result is proved in \cite[Lemma 9.5]{NP} in the equal characteristic situation. But the proof works for the mixed characteristic situation as well.\footnote{There is one subtlety that needs one's attention: In the mixed characteristic situation, one cannot use the affine root group such as $U_{\al_i,-1}$ in \emph{loc. cit.}. However, in the argument one can replace affine root group such as $U_{\al_i,-1}$ by the ``compact open" subgroup $U_{\al_i,\geq -1}=\prod_{j\geq -1}U_{\al_i,j}$, which is defined in the mixed characteristic case.}
\begin{prop}\label{domin path}
Assume that $\la$ is dominant and $\ga(1)=\la$. Then $S_\ga\subset m_{\mmu}^{-1}(\Gr_{\la})$ if and only if $\ga$ is dominant.
\end{prop}

Here is the relation between root operators and MV cycles.
\begin{lem}\label{root operator mv}
Let $\ga\in \Pi$ and assume that $e_\al(\ga)\neq 0$. Then $S_\ga$ is contained in the closure of $S_{e_\al(\ga)}$.
\end{lem}
\begin{proof}It is enough to verify this for an elementary Littelmann path. But this is clear by the description of MV cycles in terms of minuscule and quasi-minuscule Schubert varieties given in \eqref{E: MV min}, \eqref{E: MV quasiminI}, and \eqref{E: MV quasiminII}.
\end{proof}

\begin{prop}\label{bijectionI}
Let $\ga$ be a dominant Littelmann path and $\ga(1)=\mu$. Let
$\bB_\ga(\la)=\{\delta\in\bB_\ga\mid \delta(1)=\la\}$.
There is a canonical bijection 
\[\phi_\ga:\bB_\ga(\la)\cong \MV_\mu(\la),\]
which is compatible with Littelmann's isomorphism theorem (Theorem \ref{T: Littelmann isom}), i.e., if $\ga_1$ and $\ga_2$ are two dominant Littelmann paths with $\ga_1(1)=\ga_2(1)=\mu$, then $\phi_{\ga_2}i_{\ga_1,\ga_2}=\phi_{\ga_1}$.
\end{prop}
\begin{proof}
We first construct the bijection as follows. Let $\mu_\bullet=T(\ga)$. 
By Proposition \ref{domin path}, $S_\ga\subset m_{\mmu}^{-1}(\Gr_{\mu})$. Therefore for any $\delta\in \bB_\ga$, $S_\delta\subset m_{\mmu}^{-1}(\Gr_{\mu})$ by Lemma \ref{root operator mv}. In other words,
\[m_{\mmu}(S_\delta)\subset S_{\delta(1)}\cap \Gr_{\mu}.\]
We claim that the closure $\overline{m_{\mmu}(S_{\delta})}$ is of dimension $\langle\rho,\delta(1)+\mu\rangle$ and therefore is the closure of an irreducible component of $S_{\delta(1)}\cap\Gr_{\mu}$. Indeed, if $e_\al(\delta)\neq 0$, then $m_\mmu(S_\delta)\subset m_\mmu(\bar{S}_{e_\al(\delta)})$ Lemma \ref{root operator mv}. But since $m_\mmu(S_{e_\al(\delta)})\subset S_{e_\al(\delta)(1)}=S_{\delta(1)+\al}$ does not intersect with $m_\mmu(S_{\delta})\subset S_{\delta(1)}$, we see that $\dim m_\mmu(\bar{S}_{e_\al(\delta)})\geq \dim m_{\mmu}(S_{\delta})+1$. But $\dim S_\ga=\langle2\rho,\mu\rangle$ and $\dim S_{w_0(\ga)}=0$, and for each $\delta \in \BB_\gamma$, there is a sequence of root operators in $\{e_\al\}$ sending $\delta$ to $\ga$ and a sequence of root operators in $\{f_\al\}$ sending $\delta$ to $w_0(\ga)$. The sum of the lengths of these two sequences is always $\langle2\rho,\mu\rangle$. This proves the claim.

The claim gives a well-defined map
\[\phi_\ga: \bB_\ga(\la)\rightarrow\MV_\mu(\la), \quad \delta\mapsto \overline{m_\mmu(S_\delta)}\cap S_\la.\]
The next claim is that this map is injective. Then $\phi_\ga$ is bijective because the two sets have the same cardinality: namely by \cite[Theorem 9.1]{Li} and Theorem \ref{T:weight space interpretation}, the cardinality of both sets is $\dim V_\mu(\la)$.

Note that the closure of $S_\ga$ is exactly an irreducible component of $m_\mmu^{-1}\Gr_{\mu}$, of dimension $\langle\rho,\mu+|\mmu|\rangle$. In other words, there is $\bba\in\bS_{\mmu\mid\mu}$ such that $\Gr_{\mmu\mid\mu}^{0,\bba}$ is the closure of $S_\ga$. We consider
\[m_\mmu:\Gr_{\mmu\mid\mu}^{0,\bba}|_{\mathring{\Gr}_\mu}\to\mathring{\Gr}_\mu,\]
which is a fibration, with fibers (non-canonically) isomorphic of each other, of dimension $\langle\rho,|\mmu|-\mu\rangle$. Now because the generic fibers of $m_\mmu:S_\delta\to m_\mmu(S_\delta)$ is of dimension $\langle\rho,|\mmu|-\mu\rangle$, $S_\delta\cap \Gr_{\mmu\mid\mu}^{0,\bba}|_{\mathring{\Gr}_\mu}$ is dense in $m_\mmu^{-1}(m_\mmu(S_\delta)\cap \mathring{\Gr}_\mu)$, by dimension reasons. This shows that if $\delta\neq \delta'$, $m_\mmu(S_\delta)\cap \mathring\Gr_\mu$ and $m_\mmu(S_{\delta'})\cap \mathring\Gr_\mu$ can only intersect at a subset of strict smaller dimension. Therefore $\phi_\ga$ is an injection.

Next, we show that the bijection $\phi_\ga$ is compatible with the Littelmann's isomorphism theorem.
Let us denote the irreducible component of $(S_\la\cap\Gr_{\mu})$ corresponding to $\phi_\ga(\delta)\in\MV_\mu(\la)$ by $(S_\la\cap \Gr_{\mu})^\delta$ instead of $(S_\la \cap\Gr_\mu)^{\phi_\ga(\delta)}$ for simplicity in the sequel.

Given $\mu\in\xcoch^+$, let $\ell(\mu)$ denote the minimal length among all dominant Littelmann paths $\ga$ with $\ga(1)=\mu$. We prove the compatibility by induction on $\ell(\mu)$. If $\ell(\mu)=0$. There is nothing to prove. If $\ell(\mu)=1$, let $\ga$ be an elementary Littelmann path with $\ga(1)=\mu$, and $\ga'$ be another dominant Littelmann path with $\ga'(1)=\mu$. Let $\delta\in \bB^\ga(\la)$ and $\delta'=i_{\ga,\ga'}(\delta)\in\bB^{\ga'}(\la)$. Let $\mmu=T(\ga')$. If $\mu$ is minuscule or quasi-minuscule and  $\la=w\mu$, then as each $S_{w\mu}\cap \Gr_\mu$ is irreducible, the proposition is true in this case. Therefore, we assume that $\mu$ is quasi-minuscule, and $\la=0$. By definition, the path $^{2}\ga_\al\in\bB_\ga(0)$ corresponds to the irreducible component $(S_0\cap\Gr_{\leq\mu})^\al$.

Let $\delta_\al\in \bB_{\ga'}$ be (the unique) path such that $\delta_\al(1)=\al$ is a simple coroot (which is conjugate to $\mu$ under $W$), then $m_\mmu:S_{\delta_\al}\to S_{\al}\cap\Gr_{\mu}$ is dominant. Under Littelmann's isomorphism theorem $f_\al(\delta_\al)$ maps to the path ${^2}\ga_\al\in\bB_{\ga}$. Therefore,
we need to show that the map $m_{\mmu}: \Gr_{\mmu}\to \Gr_{|\mmu|}$ restricts to a dominant map
\[m_\mmu:S_{f_{\al}(\delta_\al)}\to (S_0\cap \Gr_{\mu})^\al.\]
Note that $m_{\mmu}$ maps $S_{f_\al^2(\delta_\al)}$ to $(S_{-\al}\cap \Gr_{\mu})$, which should be in the closure $\overline{m_\mmu(S_{f_{\al}(\delta_\al)})}$. But $(S_0\cap \Gr_{\mu})^\al$ is the unique irreducible component whose closure contains $(S_{-\al}\cap \Gr_{\mu})$. Therefore, the compatibility holds in this case.

Now we assume that the case $\ell(\mu)\leq m$ has been proved. Let $\mu$ be dominant and $\ell(\mu)=m$. Let $\ga=\ga_1*\cdots*\ga_{m}$ be a dominant Littelmann path of length $m$ such that $\ga(1)=\mu$. Write $\mmu=(\mu_1,\ldots,\mu_m)=T(\ga)$. We can write $\ga$ as the concatenation of the path $\ga_{<m}*\ga_m$, where $\ga_{<m}=\ga_1*\cdots*\ga_{m-1}$. Let $\delta\in\bB_\ga$ and we also write $\delta=\delta_{<m}*\delta_m$, where $\delta_m$ is an elementary Littelmann path. Then it is easy to show that $\delta_{<m-1}\in \bB_{\ga_{<m}}$.

Let $\ga'$ be another dominant Littelmann path with $\ga'(1)=\mu$. We first assume that $\ga'$ is of the form
\[\ga'=\ga'_{<m}*\ga_m.\]
Let $\delta'\in\bB_{\ga'}$ be the path corresponding to $\delta$ under  Littelmann's isomorphism and similarly let $\delta'_{<m}\in \bB_{\ga'_{<m}}$ be the the path corresponding to $\delta_{<m}$ Then it is easy to see that $\delta'=\delta'_{<m}*\delta_m$.

Now, let $\mu_{<m}=\ga_{<m}(1)$ and $\la_{<m}=\delta_{<m}(1)$, and consider
\[\Gr_{\mu_{<m}}\tilde\times\Gr_{\mu_m}\to\Gr_{\mu_{<m}+\mu_m}.\]
By induction, $\overline{m_{\mmu}(S_{\delta_{<m}})}=\overline{m_{\mmu}(S_{\delta'_{<m}})}\subset \Gr_{\mu_{<m}}\cap\bar{S}_{\la_{<m}}$.
Then it is immediate to see that the bijections $\phi_{\ga}$ and $\phi_{\ga'}$ are compatible with Littelmann's isomorphism theorem.

Now let $\ga'$ be a general dominant Littelmann path with $\ga'(1)=\mu$. We can consider
\[\ga''=\ga'*\ga_m^**\ga_m.\]
We reduce to show that the bijections $\phi_{\ga'}$ and $\phi_{\ga''}$ are compatible with Littelmann's isomorphism theorem. But note that the isomorphism theorem in this case is given by $i_{\ga',\ga''}(\delta)= \delta*w_0(\ga^*_m)*w_0(\ga_m)\in\bB_{\ga''}$, where $w_0$ is the longest element in the Weyl group $W$, which acts on $\Pi$ via the construction of \cite[\S 8]{Li}. It is easy to the compatibility between $\phi_{\ga'}$ and $\phi_{\ga''}$ as well. \end{proof}
 
\begin{rmk}\label{mult version}
There is a multi-coweights version of Proposition \ref{bijectionI}, which generalizes Proposition \ref{MV vs path}.
Let $\mmu$ be a sequence of dominant coweights, and $\ga_\bullet$ a sequence of dominant Littelmann paths with $\ga_i(1)=\mu_i$. Let  
\[\bB_{\ga_\bullet}(\la)=\{\delta = \delta_1*\cdots*\delta_n\mid\delta(1)=\la,\ \delta_i\in\bB_{\ga_i}\},\]
and 
\[\bM\bV_{\mmu}(\la)=\bigsqcup_{\sum \la_i=\la} \Big(\prod_i\MV_{\mu_i}(\la_i) \Big).\]
Note that if $\ga'_\bullet$ be another set of dominant paths with $\ga'_i(1)=\mu_i$. Then Littelmann's isomorphism (Theorem~\ref{T: Littelmann isom}) induces $i_{\ga_\bullet,\ga'_\bullet}: \bB_{\ga_\bullet}(\la)\cong \bB_{\ga'_\bullet}(\la)$.
Now, there is a canonical bijection 
\[\phi_{\ga_\bullet}:\bB_{\ga_\bullet}(\la)\leftrightarrow\bM\bV_{\mmu}(\la).\]
such that $\phi_{\ga'_{\bullet}}i_{\ga_\bullet,\ga'_\bullet}=\phi_{\ga_\bullet}$. Namely, the irreducible component of $m^{-1}_\mmu(S_\la)$ corresponding to $\delta_1*\cdots*\delta_n$ is $(S_{\delta_1(1)}\cap\Gr_{\mu_1})^{\delta_1}\tilde\times\cdots\tilde\times (S_{\delta_n(1)}\cap\Gr_{\mu_n})^{\delta_n}$.
\end{rmk}

Let $\ga_\bullet$ and be a sequence dominant Littelmann paths with $\gamma_i(1) = \mu_i$. We denote
\[\bS_{\ga_\bullet\mid \mu}=\{\delta=\ga_1*\delta_2*\cdots*\delta_n\mid \delta(1)=\mu, \ \delta_i\in\bB_{\ga_i},\ \ga_1*\delta_2*\cdots*\delta_j \mbox{ dominant}, j=2,\ldots,n \}\]
Littelmann's decomposition formula says that
\[\bB_{\ga_1,\ga_2}=\bigsqcup_{\mu}\bigsqcup_{\ga\in \bS_{(\ga_1,\ga_2)\mid\mu}} \bB_{\ga}.\]

\begin{prop}\label{comb decom}
There is a canonical bijection
\[\bS_{(\ga_1,\ga_2)\mid \mu}\cong \bS_{(\mu_1,\mu_2)\mid \mu},\]
that induces the following diagram is commutative
\[\begin{CD}
\bS_{(\ga_1,\ga_2)\mid \mu}@>>> \bS_{(\mu_1,\mu_2)\mid \mu}\\
@VVV@VVV\\
\bB_{\ga_2}(\mu-\mu_1)@>\phi_{\ga_2}>>\MV_{\mu_2}(\mu-\mu_1)
\end{CD}
\]
where the left vertical arrow is the obvious one, and the right vertical arrow is from Lemma \ref{L:Sat vs MV}.
\end{prop}
\begin{proof}
We first construct a more general map $\bS_{\ga_\bullet\mid \mu}\to \bS_{\mmu\mid \mu}$ for a sequence of dominant Littelmann path, where $\mu_i=\ga_i(1)$.
By Proposition \ref{domin path}, for $\ga_1*\delta_2*\cdots*\delta_n\in \bS_{\ga_\bullet\mid\mu}$, 
$$(S_{\mu_1}\cap\Gr_{\mu_1})\tilde\times (S_{\delta_2(1)}\cap\Gr_{ \mu_2})^{\delta_2}\tilde\times\cdots\tilde\times (S_{\delta_n(1)}\cap\Gr_{ \mu_n})^{\delta_n}\subset m^{-1}_{\mmu}(\Gr_{ \mu})$$ 
(note that $S_{\mu_1}\cap\Gr_{ \mu_1}=(S_{\mu_1}\cap\Gr_{ \mu_1})^{\ga_1}$). As $(S_{\mu_1}\cap\Gr_{ \mu_1})\tilde\times (S_{\delta_2(1)}\cap\Gr_{ \mu_2})^{\delta_2}\tilde\times\cdots\tilde\times (S_{\delta_n(1)}\cap\Gr_{\mu_n})^{\delta_n}$ is irreducible of dimension $\langle\rho,|\mmu|+\mu \rangle$, it is an open subset of one irreducible component of $\Gr_{\mmu\mid\mu}^0$. This defines the required map. This map is clearly an injection. In addition, in the case $\ga_\bullet=(\ga_1,\ga_2)$, the diagram is commutative follows directly from the construction here and in Lemma \ref{L:Sat vs MV}. It remains to show that this map is a bijection.

Note that from the construction, the following diagram is also commutative
\[\begin{CD}
\bigsqcup_{\mu}\bigsqcup_{\ga\in \bS_{(\ga_1,\ga_2)\mid\mu}} \bB_{\ga}@>>>\bB_{\ga_1,\ga_2}\\
@V\phi_\ga VV@VV\phi_{(\ga_1,\ga_2)}V\\
\bigsqcup_{\mu}\bS_{(\mu_1,\mu_2)\mid\mu}\times \MV_\mu@>>>\MV_{\mu_1,\mu_2},
\end{CD}\]
where the bottom map is from Lemma \ref{L: decomp MV into Satake}, which is injective.
The top and the right vertical arrows are bijections. 
It follows that the left vertical arrow and the bottom arrows are also bijective. It in particular implies that each map $\bS_{(\ga_1,\ga_2)\mid \mu}\to \bS_{(\mu_1,\mu_2)\mid \mu}$ is a bijection for every $\mu$.
\end{proof}

Putting things together, we have proved the following theorem.
\begin{theorem}
[Braverman-Gaitsgory]
\label{T:Braverman-Gaitsgory}
The collection $\{\MV_\mu\}$ for all $\mu$ form a closed family of $\hat G$-crystals.\footnote{A priori, the $\hat G$-crystal structures on $\MV_\mu$ constructed in this way might be different from the one given in \cite{BG}. But it follows from \cite{BaGa} that they are the same.}
\end{theorem}

\subsubsection{}
The main result we need in \S \ref{S: irr comp of ADLVs} is as follows.
We recall that $L\bG_a$ admits a filtration by closed subgroups $L^+\bG_a^{(n)}$, where $L^+\bG_a^{(n)}(R)=\varpi^nW_\mO(R)$. In particular, $L^+\bG_a^{(0)}=L^+\bG_a$. If $U_\al$ is a root subgroup of $G$, then by choosing $x_{\al}:\bG_a\simeq U_\al$ over $\mO$ and transport of structures, we obtain a filtration $LU_\al=\cup_{n \in \ZZ} L^+U_\al^{(n)}$, which is independent of the choice of the isomorphism $x_\al$.

\begin{prop}\label{P: geom prop of MV cycle}
For each simple root $\alpha$, let $M_\alpha$ denote the standard Levi with roots $\{ \pm \alpha\}$. 
Then the map
\[\prod \theta_{M_\al}: S_\la\to \prod_i S_{M_\al,\la}\]
defined in \S\ref{SS:theta M}
induces a \emph{dominant} map
\[(S_\la\cap \Gr_\mu)^\bbb\to \prod_{\al} S_{M_\al,\la}\cap \Gr_{M_\al,\la+ \varepsilon_\al(\bbb)\al^\vee},\]
and its fibers are geometrically connected.
\end{prop}
\begin{rmk}
Note that this proposition refines \cite[Corollary 3.1]{BG}.
\end{rmk}
\begin{proof}
First, we assume that $\mu$ is minuscule. Then $\bbb\in \bM\bV_\mu(w\mu)$ for some $w \in W$. In this case, 
\[\varepsilon_\al(\bbb)=\max\{0,-\langle w\mu,\alpha\rangle\},\quad \phi_\al(\bbb)=\max\{0, \langle w\mu,\alpha\rangle\}.\]
Then the proposition follows from \eqref{E: MV min}.

Next, we assume that $\mu$ is quasi-minuscule. There are two cases. 
\begin{itemize}
\item
If $\bbb\in \bM\bV_\mu(w\mu)$ for some $w \in W$, then still
\[\varepsilon_\al(\bbb)=\max\{0,-\langle w\mu,\alpha\rangle\},\quad \phi_\al(\bbb)=\max\{0, \langle w\mu,\alpha\rangle\},\]
and the proposition follows from \eqref{E: MV quasiminI}.
\item
If $\bbb\in \bM\bV_\mu(0)$ corresponding to $\alpha\in\Delta_\mu$ under the canonical bijection
$\bM\bV_\mu(0) \cong \Delta_\mu$, then
\[\varepsilon_\al(\bbb)=\phi_\al(\bbb)=1, \quad \varepsilon_\beta(\bbb)=\phi_\beta(\bbb)=0, \ \beta\neq \al.\]
But it follows from \eqref{E: MV quasiminII} that $(S_0\cap\Gr_\mu)^{\al}$ maps surjectively onto $(S_{M_\al,0}\cap\Gr_{M_\al,\al^\vee})\times \prod_{\beta\neq \al}(S_{M_\beta,0}\cap\Gr_{M_\beta,0})$, with geometrically connected fibers.
\end{itemize}

In general, $(S_\la\cap\Gr_\mu)^\bbb$ is dominated by
\[(S_{\la_1}\cap\Gr_{\mu_1})^{\bbb_1}\tilde\times(S_{\la_2}\cap\Gr_{\mu_2})^{\bbb_2}\tilde\times\cdots\tilde\times(S_{\la_n}\cap\Gr_{\mu_n})^{\bbb_n}\to (S_\la\cap\Gr_\mu)^\bbb,\]
for some sequence of minimal elements $(\mu_1, \dots, \mu_n) \subset \Min$, some sequence of weights $\lambda_1, \dots, \lambda_n$, and some elements $\bbb_j\in \MV_{\mu_j}(\la_j)$ such that $\la=\sum \la_j$, and $\bbb=\bbb_1\otimes\cdots\otimes\bbb_n$. 

Consider the commutative diagram
\begin{equation}
\label{E:commutative diagram crystal resolution}
\xymatrix@C=10pt{
(S_{\la_1}\cap\Gr_{\mu_1})^{\bbb_1}\tilde\times(S_{\la_2}\cap\Gr_{\mu_2})^{\bbb_2}\tilde\times\cdots\tilde\times(S_{\la_n}\cap\Gr_{\mu_n})^{\bbb_n} \ar[d]
 \ar[r] &
 (S_\la\cap\Gr_\mu)^\bbb \ar[d]
\\
\displaystyle\prod_{\al}(S_{M_\al,\la_1}\cap \Gr_{M_\al, \la_1+\varepsilon_\al(\bbb_1)\al^\vee})\tilde\times\cdots\tilde\times (S_{M_\al,\la_n}\cap \Gr_{M_\al, \la_n+\varepsilon_\al(\bbb_n)\al^\vee}) \ar[r] & \displaystyle\prod_{\al} S_{M_\al,\la}\cap \Gr_{M_\al,\la+ \varepsilon_\al(\bbb)\al^\vee}.
}
\end{equation}
Using the minuscule and quasi-minuscule cases discussed above, the left vertical arrow is dominant and has geometric connected fibers.
Therefore, it is enough to show that for every $\al\in\Delta$, the $\alpha$-factor of the bottom horizontal map of \eqref{E:commutative diagram crystal resolution} is surjective and has geometrically connected fibers.
Indeed, the fibers of this map are all affine spaces, so they are geometrically connected. 

We now prove the surjectivity.
By \eqref{E: tensor epsilon and phi},
\[\varepsilon_\al(\bbb_1\otimes\cdots\otimes\bbb_n)=\max_j\{\varepsilon_\al(\bbb_j)-\langle\la_1+\cdots+\la_{j-1},\al\rangle \}.\]
Therefore, the union of the intervals
\[\bigcup_j\; \big[-\varepsilon_\al(\bbb_j)+\langle \la_1+\cdots+\la_{j-1},\al\rangle,\  \langle\la_1+\cdots+\la_j,\al\rangle-1\big]\]
contains $[-\varepsilon_\al(\bbb), \langle \la,\al\rangle-1]\cap\bZ$.
Now we conclude by using the description of $S_\ell\cap \Gr_{\mathrm{PGL}_2,m}$ in \eqref{E: MV sl2}.
\end{proof}

\subsection{The geometric Satake}
\subsubsection{The geometric Satake}
The following theorem of geometric Satake will play the central role in our construction.  When $\cha F >0$, this is a result of Lusztig, Drinfeld, Ginzburg, Mirkovi\'c-Vilonen (see \cite{MV}) and when $\cha F =0$, this is the main theorem of \cite{Z}. Let $\on{P}_{L^+G\otimes \bar k}(\Gr\otimes \bar k)$ be the Satake category: the category of $L^+G\otimes \bar k$-equivariant $\Ql$-coefficients perverse sheaves on $\Gr\otimes \bar k$. Let us recall the definition of this category in some detail, since similar constructions will repeatedly appear in \S \ref{SS:Perv(Hk)} and \S \ref{SS:Perv(Sloc)}.

We first recall that $\pi_0(\Gr\otimes \bar k)$ is canonically isomorphic to $\pi_1(G)$. Consider the decomposition $\Gr\otimes \bar k= \sqcup_{\bbzeta\in \pi_1(G)} \Gr_{\bbzeta}$ into connected components. For $\nu\in\xcoch(T)$, if its image under $\xcoch(T)\to \pi_1(G)$ is $\bbzeta$, we write $\nu\in \bbzeta$. Then 
\[\Gr_{\bbzeta}=\varinjlim_{\mu\in\bbzeta} \Gr_{\mu},\]
where $i_{\mu,\mu'}: \Gr_{\mu}\to \Gr_{\mu'}$ is a closed embedding if $\mu\preceq\mu'$.
Recall that if $m$ is $\mu$-large, the action of $L^+G$ on $\Gr_\mu$ factors through the action $L^mG$ (see Lemma \ref{L: m mu large}). 
Note that if $m'\geq m\geq \langle\mu,\al_h\rangle$, 
there is a canonical equivalence
\[\on{P}_{L^mG\otimes \bar k}(\Gr_\mu)\cong \on{P}_{L^{m'}G\otimes \bar k}(\Gr_\mu).\]
Then we define
\begin{equation}
\label{E: def of Sat cat}
\on{P}_{L^+G\otimes \bar k}(\Gr\otimes \bar k):= \bigoplus_{\bbzeta\in \pi_1(G)}\on{P}_{L^+G\otimes \bar k}(\Gr_{\bbzeta}),\quad \on{P}_{L^+G\otimes \bar k}(\Gr_{\bbzeta})=\varinjlim_{(\mu,m)}\on{P}_{L^mG\otimes \bar k}(\Gr_\mu),
\end{equation}
where the limit is taken over $\{(\mu,m)\mid \mu\in\bbzeta, m\geq \langle \mu, \al_h\rangle\}$, with the partial order given by $(\mu,m)\leq (\mu',m')$ if $\mu\preceq \mu'$ and $m\leq m'$, and the connecting functor given the fully faithful embedding
\[\on{P}_{L^mG\otimes \bar k}(\Gr_\mu)\cong \on{P}_{L^{m'}G\otimes \bar k}(\Gr_\mu)\xrightarrow{i_{\mu,\mu',*}} \on{P}_{L^{m'}G\otimes \bar k}(\Gr_{\mu'}).\]
Note that the limit is filtered (although not directly).

The above defined category has a natural monoidal structure, given by Lusztig's convolution product of sheaves. Let $\mA_1,\mA_2\in \on{P}_{L^+G\otimes \bar k}(\Gr\otimes \bar k)$, coming from $\Gr_{\mu_1}$ and $\Gr_{\mu_2}$ respectively, we denote by $\mA_1\tilde\boxtimes\mA_2$ the ``external twisted product'' of $\mA_1$ and $\mA_2$ on $\Gr_{\mu_1}\tilde\times\Gr_{\mu_2}$, whose pullback to $\Gr^{(n)}_{\mu_1}\times \Gr_{\mu_1}$ (for some $n$ that is $\mu_2$-large) is isomorphic to the external product of the pullback of $\mA_1$ to $\Gr^{(n)}_{\mu_1}$ and $\mA_2$ (this property uniquely characterizes $\mA_1\tilde\times\mA_2$ up to a unique isomorphism). 
For example, if $\mA_i=\IC_{\mu_i}$, then $\IC_{\mu_1}\tilde\boxtimes\IC_{\mu_2}$ is canonically isomorphic to the intersection cohomology sheaf of $\Gr_{\mu_1}\tilde\times\Gr_{\mu_2}$.
Similarly, one can define the $n$-folded ``external twisted product" $\mA_1\tilde\boxtimes\cdots\tilde\boxtimes\mA_n$ on $\Gr_{\mmu}$. If $\mA_i=\IC_{\mu_i}$, then this is just the intersection cohomology sheaf $\IC_{\mmu}$ of $\Gr_\mmu$.
Then the convolution product of $\mA_1$ and $\mA_2$ is defined as
\begin{equation}\label{E: Lus conv}
\mA_1\star\mA_2:=m_!(\mA_1\tilde\boxtimes\mA_2),
\end{equation}
where $m:\Gr_{\mu_1}\tilde\times\Gr_{\mu_2}\to \Gr$ is the convolution map (defined by \eqref{E: convolution map}).
A priori, this is an $(L^+G\otimes \bar k)$-equivariant $\ell$-adic complex on $\Gr\otimes \bar k$. But miraculously, it is a perverse sheaf.

This following theorem is usually referred as the geometric Satake isomorphism.

 \begin{thm}[Geometric Satake]
\label{T:geom Satake}
The Satake category is a semisimple tensor category.
The hypercohomology functor $\on{H}^*: \on{P}_{L^+G\otimes \bar k}(\Gr\otimes \bar k) \to \on{Vect}_{\Ql}$ lifts to a natural equivalence of tensor categories
\begin{equation}
\label{E:geometric Satake}
\on{H}^*= \on{H}^*(\Gr, -): \on{P}_{L^+G\otimes \bar k}(\Gr\otimes \bar k) \to \on{Rep}(\hat G_{\Ql}).
\end{equation}
In particular, it sends $\IC_{\mu_1}\star\cdots\star\IC_{\mu_n}=(m_{\mmu})_*\IC_{\mmu}$ to the tensor product of highest weight representations $V_\mmu$.
\end{thm}

\begin{rmk}
\label{R: birth of the dual group}
(1) Indeed, it is more canonical to \emph{define} the dual group $\hat G$ of $G$ as the Tannakian group of the Tannakian category $(\on{P}_{L^+G\otimes \bar k}(\Gr\otimes \bar k),\on{H}^*)$, and \emph{define} $V_\mu$ as $\on{H}^*(\Gr,\IC_\mu)$, equipped with the tautological action $\hat G=\Aut^\otimes(\on{H}^*)$. In the rest of the paper, we will take this point of view.

(2) As explained in \cite[\S 3]{Z16}, the Tannakian group is canonically equipped with a pinning $(\hat G, \hat B, \hat T, \hat X)$. We briefly recall the construction of $(\hat B, \hat T)$ and refer to \emph{loc. cit.} for more details.
The grading on the cohomology $\on{H}^*$ defines a cocharacter $\bG_m\to \hat G$, which is regular. Then its centralizer gives a maximal torus $\hat T$, and $\hat B\supset \hat T$ is the unique Borel in $\hat G$ such that this cocharacter is dominant with respect to this Borel.
\end{rmk}

From now on, we fix an inverse of the equivalence
\[\Sat: \on{Rep}(\hat G)\to \on{P}_{L^+G\otimes \bar k}(\Gr\otimes \bar k),\]
sending $V_\mu$ to $\IC_\mu$. 

We need the following consequence of the geometric Satake isomorphism.
\begin{cor}\label{C: rep to cor}
There is a canonical isomorphism
\begin{equation}\label{E: rep to cor}
\Sat: \Hom_{\hat G}(V_{\la_\bullet},V_\mmu)\cong  \on{Corr}_{\Gr_{\la_\bullet|\mmu}^0}((\Gr_{\la_\bullet},\IC),(\Gr_{\mmu},\IC)),
\end{equation}
such that the composition 
\[ \Hom_{\hat G}(V_{\kappa_\bullet},V_{\la_\bullet}) \otimes \Hom_{\hat G}(V_{\la_\bullet}, V_\mmu)\to \Hom_{\hat G}(V_{\kappa_\bullet}, V_{\mu_\bullet})\]
corresponding to
\[\begin{split}
\on{Corr}_{\Gr_{\kappa_\bullet|\lambda_\bullet}^0}((\Gr_{\kappa_\bullet},\IC),(\Gr_{\lambda_\bullet},\IC))\otimes \on{Corr}_{\Gr_{\lambda_\bullet|\mu_\bullet}^0}((\Gr_{\lambda_\bullet},\IC),(\Gr_{\mu_\bullet},\IC))\\
\to \on{Corr}_{\Gr_{\kappa_\bullet|\mu_\bullet}^0}((\Gr_{\kappa_\bullet},\IC),(\Gr_{\mu_\bullet},\IC))
\end{split}\]
is obtained as the pushforward (in the sense of \ref{ASS:pushforward correspondence}) of the composition of the cohomological correspondences along the perfectly proper morphism
\[\on{Comp}: \Gr_{\kappa_\bullet\mid\lambda_\bullet}^0\times_{\Gr_{\lambda_\bullet}^0}\Gr_{\lambda_\bullet\mid\mu_\bullet}^0\to \Gr_{\kappa_\bullet\mid \mu_\bullet}^0.\]
\end{cor}

\subsubsection{Galois-equivariant geometric Satake}
As explained in \cite{RiZh} and \cite[\S 5.5]{Z16}, the Galois group $\Gal(\bar k/k)$, acts on $\on{P}_{L^+G\otimes \bar k}(\Gr\otimes k)$ by tensor automorphisms: $\ga\in\Gal(\bar k/k)$ induces $\ga: \Gr\otimes \bar k\cong\Gr\otimes \bar k$, and the pullback $\ga^*: \on{P}_{L^+G\otimes \bar k}(\Gr\otimes \bar k)\to \on{P}_{L^+G\otimes \bar k}(\Gr\otimes \bar k)$ is a tensor functor. In addition, there is a canonical isomorphism of tensor functors $\al_\ga: \on{H}^*\circ\ga^*\cong \on{H}^*$.
By general nonsense, this defines an action of $\Gal(\bar k/k)$ on $\hat{G}$: for every $g\in \hat G$, regarded as an automorphism of the tensor functor $\on{H}^*$, $\ga(g)$ is the tensor automorphism of $\on{H}^*$ defined as
\begin{equation*}\label{gaaction}
\on{H}^*\stackrel{\al_\ga}{\longrightarrow}\on{H}^*\circ\ga^*\stackrel{g\,\circ\, \id}{\ \xrightarrow{\hspace*{1cm}}\ } \on{H}^*\circ\ga^*\stackrel{\al^{-1}_\ga}{\longrightarrow}\on{H}^*.
\end{equation*}
Note that since the Galois action preserves the grading, this action preserves $(\hat G,\hat B,\hat T)$. 
\begin{rmk}
However, as explained in \emph{loc. cit.}, this action does not preserve $\hat X$. In our case where $k$ is a finite field, this is a minor issue and can be fixed by choosing a half Tate twist $\Ql(1/2)$. We refer \emph{loc. cit.} for detailed discussion.
\end{rmk}
\begin{notation}
\label{N:sigma W}
For $\ga\in\Gal(\bar k/k)$ and $V$ a representation of $\hat G$, let $\ga V$ be the representation of $\hat G$ obtained as $\hat G\stackrel{\ga^{-1}}{\to}\hat G\to \GL(V)$, called the $\ga$-twisted of $V$. In particular, for a dominant coweight $\mu$ of $G$, $\ga V_\mu \cong V_{\ga(\mu)}$.
\end{notation}
It follows from the above discussion that
\begin{equation}
\label{E: Gal equiv Sat}
\ga^*\Sat(V)\cong \Sat(\ga V)
\end{equation}
for every $V\in \on{Rep}(\hat G)$.

Let $\on{P}^0_{L^+G}(\Gr)\subset \on{P}_{L^+G}(\Gr)$ denote the full subcategory of semisimple sheaves of pure of weight zero. For example, by choosing a half Tate twist $\Ql(1/2)$, the normalized intersection cohomology sheaf $\IC_\mu^{\rm N}$ on $\Gr_\mu$ is an object in $\on{P}_{L^+G}^0(\Gr)$, where $\IC_\mu^{\rm N}|_{\mathring{\Gr}_\mu}=\Ql[2\langle\rho,\mu\rangle](\langle\rho,\mu\rangle)$. By Corollary \ref{SS:cycles of geometric Satake}(3), and an argument as in \cite[Lemma 5.5.14 (1) ]{Z16}, $\on{P}^0_{L^+G}(\Gr)$ is a monoidal subcategory of $\on{P}_{L^+G}(\Gr)$ under the convolution product.

Let ${^L}G= \hat{G}\rtimes\Gal(\bar k/k)$ denote the usual (unramified) local Langlands dual group. We may regard ${^L}G$ as a pro-algebraic group and therefore it makes sense to define the category $\on{Rep}({^L}G)$ of algebraic representations of ${^L}G$. In particular, $\on{Rep}(\Gal(\bar k/k))$ is the category of representations of finite quotients of $\Gal(\bar k/k)$.  
\begin{thm}
\label{T: geom Sat upgraded}
The category $\on{P}_{L^+G}^0(\Gr)$ is a tensor subcategory of $\on{P}_{L^+G}(\Gr)$. By choosing a half Tate twist $\Ql(1/2)$, the cohomology functor
\[\bigoplus_i \on{H}^i(\Gr, -)(\tfrac{i}{2}): \on{P}_{L^+G}^0(\Gr)\to \on{Rep}(\Gal(\bar k/k)),\]
lifts to an equivalence of categories
\[\on{P}_{L^+G}^0(\Gr)\cong \on{Rep}({^L}G),\]
such that the pullback functor $\on{P}_{L^+G}^0(\Gr)\to\on{P}_{L^+G\otimes \bar k}(\Gr\otimes \bar k)$ corresponds to the restriction functor $\on{Rep}({^L}G)\to\on{Rep}(\hat G)$.
\end{thm}
\begin{proof}
See \cite{RiZh} or \cite[\S 5.5]{Z16}.
\end{proof}

\subsection{From geometric to classic Satake isomorphism}
\label{S: geom v.s. classic Sat}

We need to recall how to deduce the classical Satake isomorphism from the geometric Satake isomorphism via the sheaf-function dictionary. In this subsection, we fix a square root $q^{1/2} \in \overline \QQ_\ell$ and hence accordingly  a half Tate twist $\Ql(\frac{1}{2})$ on which the geometric $q$-Frobenius acts by multiplying with $q^{-1/2}$.

\subsubsection{Review of the classical Satake isomorphism}
There are already some good references for split groups, for example \cite{Gro}. However, as we need the unramified groups in an essential way, we include a more detailed discussion. 

Let us denote by $K=G(\mO)\subset G(F)$. We choose an embedding $B\subset G$ but the Satake isomorphism will be independent of the choice. Recall that $T=B/U$ is the abstract Cartan. Let $2\rho$ be the half sum of positive roots of $G$ with respect to $B$. Let $S\subset T$ denote the maximal split subtorus. For a coweight $\la \in \XX_\bullet(T)$, let $d_\la=\langle \rho,\la \rangle\in \frac{1}{2}\bZ$. Recall we denote by $\sigma$ the $q$-power Frobenius element, so that $\xcoch(S)=\xcoch(T)^\sigma$. Let $W_0=N_G(S)/Z_G(S)$ denote the relative Weyl group of $G$, which is a subgroup of $W=N_G(T)/T$

Let $\delta_B: T(F)\to\bZ[q,q^{-1}]$ denote the modular character, i.e. $\delta_B(t)=|\det(\Ad_t | \Lie U)|$. Explicitly, $\delta_B|_{T(\mO)}=1$ and $\delta_B(\varpi^\la)=q^{\langle 2\rho,\la\rangle}$ for $\la\in \xcoch(S)$. With the fixed square root $q^{1/2}$, let $\delta^{1/2}_B: T(F)\to\bZ[q^{1/2},q^{-1/2}]^\times$ denote the square root of $\delta_B$, i.e. $\delta^{1/2}_B(\varpi^\la)=(q^{1/2})^{\langle 2\rho,\la\rangle}=: q^{d_\la}$.

We normalize the Haar measure on $G(F)$ (and on $U(F)$) so that the volume of $K$ (and $U(\mO)$) is one.
Let $H_G$ denote the spherical Hecke algebra of $G$ with respect to $K$, i.e. the space of $K$-bi-invariant, compactly supported, locally constant $\bZ$-valued functions on $G(F)$, with the multiplication given by the convolution
\begin{equation}
\label{E: convolution for functions}
(f*g)(x)=\int_{G(F)}f(xy^{-1})g(y)dy.
\end{equation}
If $G=T$ is a torus, then 
$H_T$ is canonically isomorphic to the group algebra $\bZ[\xcoch(T)^\sigma]$ of the lattice $\xcoch(T)^\sigma$, by sending the characteristic function of $\varpi^\la T(\mO)$ to $e^\la$ for $\la\in\xcoch(T)^\sigma$.
Recall the Satake transform
\[\on{CT}: H_G\otimes \bZ[q^{1/2},q^{-1/2}]\to H_T\otimes\bZ[q^{1/2},q^{-1/2}],\quad f\mapsto \on{CT}(f), \] \[\on{CT}(f)(t):= \delta_B^{1/2}(t)\int_{U(F)}f(tu)du.\]
Then the classical Satake isomorphism says that the Satake transform is injective, and induces a canonical isomorphism 
\begin{equation}
\label{E: classical Sat}
\on{CT}: H_G\otimes\bZ[q^{1/2},q^{-1/2}]\cong H_T^{W_0}\otimes\bZ[q^{1/2},q^{-1/2}]\cong \bZ[q^{1/2},q^{-1/2}][\xcoch(T)^\sigma]^{W_0}.
\end{equation}
See \cite[\S 4.2]{cartier}. Note that in \emph{loc. cit.}, the Satake isomorphism was stated as an isomorphism of $\bC$-algebras. But a closer look at the proof shows that it is an isomorphism over $\bZ[q^{\pm 1/2}]$.

\subsubsection{Reinterpretation in terms of the dual group}
Let us reinterpret the classical Satake isomorphism in terms of the dual group. As we shall see, in fact there are two interpretations.

Recall that $\Gal(\bar k/k)$ acts on $(\hat G,\hat B,\hat T,\hat X)$, and we consider the Langlands dual group ${^L}G:=\hat G\rtimes \Gal(\bar k/k)$ as a pro-algebraic group.
Let $\phi=\sigma^{-1}$ denote the geometric $q$-Frobenius element. 
Note that the inclusion $S\subset T$ induces $\hat T\to\hat S$ identifying $\hat S$ as the quotient $\hat T/(1-\sigma)\hat T=(\hat{T}\times \phi)/\on{Int}\hat{T}$. In addition, $\xcoch(T)^\sigma$  the character group of $\hat S$.
Recall that the following natural maps are bijective (cf. Gantmacher \cite[Theorem 14]{gantmacher} and \cite[Lemmas~6.4 and 6.5]{borel})
\begin{equation*}\label{parameter}
\hat{S}/W_0\stackrel{\simeq}{\leftarrow} (\hat{T}\phi)/\on{Int} N_0\stackrel{\simeq}{\to } (\hat{G}\phi)_{ss}/\on{Int}\hat G. 
\end{equation*}
Therefore, by restricting a function on $\hat G\phi$ to $\hat T\phi$, there is an isomorphism
\[\Gamma([\hat G\phi/\hat G],\mO)=\mO_{\hat G}^{c_\phi(\hat G)}=\mO_{\hat S}^{W_0}=\Ql[\xcoch(T)^\sigma]^{W_0}.\]
Combining with \eqref{E: classical Sat}, there is a unique isomorphism 
\begin{equation}
\label{spectral Sat}
\Sat^{cl}:=\Sat^{cl}_G: \Gamma([\hat G\phi/\hat G],\mO)\cong H_G\otimes\Ql,
\end{equation} called the \emph{Satake isomorphism}, making the following diagram commute
\begin{equation}
\label{E: comp G and T}
\xymatrix{
\Gamma([\hat G\phi/\hat G],\mO)\ar[rr]^-\cong_-{\Sat^{cl}_G}\ar[d]_\Res && H_G\otimes\Ql  \ar[d]_{\on{CT}}\\
\Gamma([\hat T\phi/\hat T],\mO) \ar[rr]^-\cong_-{\Sat^{cl}_T} && H_T\otimes\Ql.
}
\end{equation}

Let  $R({^L}G)=K_0(\on{Rep}({^L}G))$ denote the representation ring of ${^L}G$.  If $V$ is an algebraic representation of ${^L}G$, let $\chi_V$ denote its character and $[V]$ its class in $R({^L}G)$.
Then there is a natural map
\[R({^L}G)\to \Gamma([(\hat G\phi)/\hat G], \mO),\quad [V]\mapsto \chi_V|_{\hat{G}\phi} \]
By abuse of notations, we will also use $\Sat^{cl}$ to denote the induced composition map 
$$\Sat^{cl}: R({^L}G)\to  \Gamma([\hat G\phi/\hat G],\mO) \cong H_G\otimes\Ql.$$

\begin{rmk}
\label{R: geom Sat, arith Frob v.s. geom Frob}
Note that there is another natural map 
\begin{equation}
\label{E: Sat, arith Frob}
\Sat^{cl'}: \Gamma([\hat G\sigma/\hat G],\mO)\cong \Ql[\xcoch({T})^\sigma]^{W_0}\cong H_G\otimes\Ql.
\end{equation}
By restriction of a character $\chi_V$ to $\hat G\times\sigma\subset {^L}G$, we thus also obtain a map from $R({^L}G)$ to $H_G\otimes\Ql$, denoted by the same notation
\[\Sat^{cl'}: R({^L}G)\to  \Gamma([\hat G\sigma/\hat G],\mO) \cong H_G\otimes\Ql.\]
However, $\Sat^{cl}$ and $\Sat^{cl'}$ are different in general. Indeed, note that the following diagram is commutative
\begin{equation*}
\label{E: two Sat isom related}
\xymatrix@R=10pt{
& \Gamma([\hat G\phi/\hat G],\mO)\ar^{\on{inv}}[dd]\ar_-\cong^-{\Sat^{cl}}[r] & \Ql[\xcoch(T)^\sigma]^{W_0}\ar^{\la\mapsto -\la}[dd]\\
R({^L}G)\ar[ur]\ar[dr]&&\\
& \Gamma([\hat G\sigma/\hat G],\mO)\ar_-\cong^-{\Sat^{cl'}}[r] & \Ql[\xcoch(T)^\sigma]^{W_0}
}
\end{equation*}
where $\on{inv}: \Gamma([\hat G\phi/\hat G],\mO)\cong \Gamma([\hat G\sigma/\hat G],\mO)$ is induced by  the map $\hat G\phi\mapsto \hat G\sigma, \ \ g\phi\mapsto \sigma(g^{-1})\sigma$. Therefore,
$$\Sat^{cl}([V])=\Sat^{cl'}([V^*]).$$
It is $\Sat^{cl}$ that is directly related to the geometric Satake via Grothendieck's sheaf-function dictionary. However, $\Sat^{cl'}$ will appear in Theorem \ref{T:Spectral action}. This leads to usual switch from $V$ to $V^*$ for Galois representations arising from the cohomology of Shimura varieties.
\end{rmk}

\subsubsection{From geometric to classical Satake isomorphism}
Now we review how to deduce the classical Satake isomorphism from the geometric Satake isomorphism (see also \cite{RiZh}, \cite[\S 5.6]{Z16} where the case $F=k((\varpi))$ was discussed). 

For a $k$-point $x$ of $\Gr$, and a geometric point $\bar x$ over it, let $\phi_x$ denote the geometric Frobenius in $\Gal(k(\bar x)/k)$.
Recall that the Grothendieck fonctions-faisceaux dictionary attaches to every $\mA\in \on{P}_{L^+G}(\Gr)$ a function 
\[f_\mA: \Gr(k)\to \Ql,\quad f_\mA(x)=\sum_i(-1)^i\tr(\phi_x, \on{H}^i_{\bar x}(\mA)),\]
Since $\mA$ is $L^+G$-equivariant, we can regard $f_\mA\in H_G\otimes_\bZ\Ql$. In addition, it follows from definition that
$f_{\mA_1\star\mA_2}=f_{\mA_1}*f_{\mA_2}$ (e.g. see \cite[Lemma 5.6.1]{Z16}).
So $\mA\mapsto f_\mA$ defines
a map of algebras 
$$\Tr: K_0(\on{P}^0_{L^+G}(\Gr))\to H_G\otimes\Ql.$$  
On the other hand, taking the K-ring of the equivalence in Theorem \ref{T: geom Sat upgraded} gives a canonical isomorphism
\begin{equation*}
K_0(\Sat):R({^L}G)\to K_0(\on{P}^0_{L^+G}(\Gr_G)).
\end{equation*} 
Composing these two maps, we thus obtain
\[\on{Sat}^{\on{s-f}}: R({^L}G)\to H_G\otimes\Ql.\]
\begin{prop}
\label{P: classical Sat=geometric Sat}
$\Sat^{\on{s-f}}=\Sat^{cl}$.
\end{prop}
\begin{proof}
We sketch the proof and refer to \cite{RiZh} for more details. First, one shows that the map
\[R({^L}G)\to K_0(\on{P}^0_{L^+G}(\Gr))\to H_G\otimes\Ql\]
factors through 
\[R({^L}G)\to \Gamma([(\hat G\phi)/\hat G],\mO)\to H_G\otimes\Ql.\]
Indeed, the kernel of the map $[V]\mapsto \chi_V|_{\hat G\phi}$ is the ideal generated by elements of
the form $[V\otimes \psi]-\psi(\phi)[V]$, where $V\in
\on{Rep}({^L}G)$ and $\psi: \Gal(\bar k/k)\to \Ql^\times$
is a character factoring through a finite quotient.  On
the other hand, the kernel of the map $\Tr: K_0(\on{P}^0_{L^+G}(\Gr))\to H_G\otimes\Ql$ is
the ideal generated by elements of the form
$[\calA\otimes\calL]-\tr(\phi,\calL)[\calA]$, where $\calA\in
\on{P}^0_{L^+G}(\Gr_G)$, and $\calL$ is a rank one local system on
$\Spec k$, pure of weight zero. But it is clear that these two
ideals match under $K_0(\Sat)$. The claims follows.

To finish to prove of the proposition, first note that this is clear if $G=T$ is a torus. For general $G$, one observes that the fiber functor
decomposes as a direct sum of weight functors \cite[\S 3]{MV} and \cite[Corollary 2.10]{Z}:
$$\mC\mT: \on{P}_{L^+G}^0(\Gr_G)\to \on{P}_{L^+T}^0(\Gr_T),\quad \mC\mT(\mA)=\bigoplus_{\la} \on{H}_c^{d_\la}(S_\la, \mA)(d_\la).$$ 
Note that under the
sheaf-function dictionary, this corresponds the Satake transform
$\on{CT}: H_G\otimes\Ql\to H_T\otimes\Ql$. Therefore, both $\Sat^{\on{s-f}}$ and $\Sat^{cl}$
are uniquely determined the commutative diagram \eqref{E: comp G and T}.
The proposition for general $G$ follows. 
\end{proof}

\section{Irreducible components of some affine Deligne-Lusztig varieties}
\label{Sec:affine DLV}
This section is devoted to studying the irreducible components of certain affine Deligne-Lusztig varieties in equal and mixed characteristic.

We will continue to use notations from the previous section: let  $F$ be a local field (of either equal or mixed characteristic), with ring of integers $\calO$ and residue field $k =\FF_q$, and let $G$ be an \emph{unramified} reductive group over $\calO$.  Let $\mO_L=W_{\mO}(\bar k)$ be the completion of the maximal unramified extension of $F$ and $L=\mO_L[1/\varpi]$.
Let $(\hat G, \hat B, \hat T)$ be its based Langlands dual group over $\overline \QQ_\ell$, constructed from the geometric Satake as in the previous section (see Remark \ref{R: birth of the dual group}).

For a perfect $k$-algebra $R$, and a $G$-torsor $\mE$   over $D_{R}$, let ${^\sigma}\mE$ denote the $G$-torsor $(\sigma\otimes \id)^*\mE$ on $W(R)\otimes_{W(k)} \mO$ as the pullback of $\mE$ by the Frobenius $\sigma\otimes \id$. 

\subsection{Definition of ADLVs}
\begin{definition}
For $b\in G(L)$ and $\mu$ a dominant coweight, we define the \emph{(closed) affine Deligne-Lusztig variety} $X_\mu(b)$ as the closed subvariety inside the affine Grassmannian:
\[
X_\mu(b)=\{g\in LG/L^+G\mid g^{-1}b \sigma(g)\in \overline{L^+G\varpi^\mu L^+G}\}.
\]
In terms of the moduli problem, for every $R$, $X_\mu(b)(R)$ parametrizes commutative diagrams of modifications of $G$-torsors on $D_R$,
\begin{equation}
\label{E:moduli problem of Xmu(b)}
\xymatrix{
{}^\sigma\calE_1 \ar@{-->}[d]^{\sigma(\beta_0)} \ar@{-->}[r]^{\beta_1} & \calE_{1}
\ar@{-->}[d]^{\beta_0}
\\
{}^\sigma\calE^0 \ar@{-->}[r]^-{b\cdot } & \calE^0,
}
\end{equation}
where
$b\cdot $ is the modification of the trivial $G$-torsor ${}^\sigma\calE^0 \cong \calE^0$ given by multiplication by $b$, and $\beta_1$ has relative position $\preceq \mu$.
\end{definition}

Note that by the definition, the following diagram is Cartesian
\begin{equation}
\label{E:cartesian DL}
\xymatrix@R=15pt{
X_{\mu}(b)\ar[r]\ar@{_{(}->}[d]  &\Gr\tilde\times\Gr_{ \mu}\ar[d]^{\pr\times m}
& \eqref{E:moduli problem of Xmu(b)} \ar@{|->}[r]\ar@{|->}[d] & \{ {}^\sigma\calE_1 \stackrel{\beta_1}{\dashrightarrow} \calE_1 \stackrel{\beta_0}{\dashrightarrow} \calE_0\}\ar@{|->}[d]
\\
\Gr\ar[r]^-{1\times b\sigma} &\Gr\times \Gr, &
\{ \calE_1 \stackrel{\beta_0}{\dashrightarrow} \calE_0\} \ar@{|->}[r] & \big\{\calE_1 \stackrel{\beta_0}{\dashrightarrow} \calE_0,\ {}^\sigma\calE_1 \stackrel{b\sigma(\beta_0)}{\dashrightarrow} \calE_0\big\}.
}
\end{equation}
In turn, \eqref{E:cartesian DL} can be served as an alternative definition of affine Deligne-Lusztig varieties. Using this point of view, one can replace $\Gr_\mu$ in \eqref{E:cartesian DL} by $\Gr_\mmu$ and therefore obtain an iterated version of affine Deligne-Lusztig varieties.

Let $J_b$ be the twisted centralizer of $b$, i.e. the algebraic group over $F$, whose value on an $F$-algebra $R$ is given by
\begin{equation}
\label{E: twisted centralizer J}
J_b(R)=\{g\in G(R\otimes_F L)\mid g^{-1}b\sigma(g)=b\}.
\end{equation}
Then $J_b(F)$ acts on $X_\mu(b)$ naturally by left multiplication. 

Finally, recall that the pair $(X_\mu(b),J_b)$ depends only on the $\sigma$-conjugacy class of $b$ up to isomorphism.

\subsubsection{The set $B(G,\mu)$}
Let $B(G)$ denote the set of $\sigma$-conjugacy classes of $G(L)$. To study $B(G)$, Kottwitz attached to every element $[b]\in B(G)$ two invariants:\footnote{We remark that Kottwitz' results hold more generally for not necessarily quasi-split groups. But we limit ourselves to only considering unramified groups in the following discussion.} the first is the \emph{Newton polygon} of $b$, which is a $\sigma$-invariant conjugacy class of maps 
$$\nu_b\in( \Hom(\bD,G_L)/G_L)^\sigma \cong \XX_\bullet(T)_\QQ^{+, \sigma},$$ 
where $\bD$ is the pro-torus with character group $\bQ$; the second is an element $$\kappa_G(b)\in \pi_1(G)_\sigma=\xch(Z(\hat{G})^\sigma).$$ In addition, Kottwitz proved that $\nu_b$ and $\kappa_G(b)$ together determine the $\sigma$-conjugacy class $[b]$ uniquely. In other words, the map
\[B(G)\to \XX_\bullet(T)_\QQ^{+, \sigma}\times \pi_1(G)_\sigma,\quad [b]\mapsto (\nu_b,\kappa_G(b))\]
is injective. The set $B(G)$ naturally forms a poset with $[b]\leq [b']$ if $\kappa_G(b)=\kappa_G(b')$ and $\nu_b\leq\nu_{b'}$.
Here, the inequality $\nu_b \leq \nu_{b'}$ means that $\nu_{b'} - \nu_b$ is a non-negative rational linear combination of simple coroots of $G$.
We call an element $[b]\in B(G)$ \emph{basic} if $\nu_b$ is central, i.e. it factors as
$$\nu_b:\bD\to Z_G \subseteq G.$$ 
Or equivalently, $\langle \rho, \nu_b\rangle =0$.
By definition, basic elements are minimal elements in $B(G)$ with respect to the above partial order. It follows from \cite[\S 5]{Koisocry} that all minimal elements in $B(G)$ are basic.

For a dominant coweight $\mu$ of $G$, we define $B(G, \mu) \subset B(G)$ as
the subset consisting of those
$[b] \in B(G)$ such that
\begin{equation}
\label{E:condition nonemptyness Xmub}
\nu_b\leq \bar\mu, \quad \textrm{and} \quad \kappa_G(b)=[\mu] \mbox{ in } \pi_1(G)_\sigma,
\end{equation}
where 
\begin{equation}\label{Newton point}
\bar\mu=\frac{1}{m}\sum_{i=0}^{m-1} \sigma^i(\mu),
\end{equation} and $m$ is the degree of a splitting field of $G$ over $F$.  For a coweight $\mu$, $[\mu]$ denotes its image under $\xcoch(T)\to \pi_1(G)\to \pi_1(G)_\sigma$.
By \cite[\S 5]{Koisocry}, inside $B(G,\mu)$ there is always a unique basic element. 

Recall the following result, which is a combination of \cite[Theorem 4.2]{RR} and \cite{Gashi} (see also \cite{Win} and \cite[Proposition 5.6.1]{GHKR}). 
\begin{thm}
\label{L:nonemptyness of Xmub}
The variety $X_\mu(b)$ is non-empty if and only if $[b]\in B(G,\mu)$.
\end{thm}

\subsubsection{Passing to adjoint groups}
To study ADLV, it is more convenient to assume that $G$ is of adjoint type. We first recall that $\pi_0(LG)$ is canonically isomorphic to $\pi_1(G)$. We denote the map
$LG\to \pi_0(LG)\cong \pi_1(G)$ by $w$. 
Note that if $\la\in \xcoch(T)$, then $w(\varpi^\la)=[\la]$ is the image of the natural projection $\xcoch(T)\to \pi_1(G)$.
Note that if $gL^+G\in X_\mu(b)$, then $w(g)^{-1}w(b)w(\sigma(g))=w(\varpi^\mu)$. So we have an equality $(\sigma-1)w(g)=[\mu]-w(b)$ in $\pi_1(G)$. In \cite{CKV}, the authors denote
\[c_{b,\mu}\pi_1(G)^\sigma:=\{\ga\in \pi_1(G)\mid (\sigma-1)\ga=[\mu]-w(b)\},\]
which is a $\pi_1(G)^\sigma$-coset. The above observations imply that there is the following commutative diagram
\[\begin{CD}
X_\mu(b)@>>>\Gr\\
@VVV@VVV\\
c_{b,\mu}\pi_1(G)^\sigma @>>>\pi_1(G).
\end{CD}\]
Now, let $G\to G'$ be a central isogeny of unramified groups. For a pair $(\mu,b)$, let $(\mu',b')$ denote the induced pair for $G'$. The map $G\to G'$ induces a natural map $X_\mu(b)\to X_{\mu'}(b')$. 
The following lemma was contained in \cite[\S 2.4]{CKV}.
\begin{lem}
\label{L: red to adj}
The natural map $X_\mu(b)\to X_{\mu'}(b')$ induces a Cartesian diagram
\[\begin{CD}
X_\mu(b) @>>> X_{\mu'}(b')\\
@VVV@VVV\\
c_{b,\mu}\pi_1(G)^\sigma @>>> c_{b',\mu'}\pi_1(G')^\sigma.
\end{CD}\]
\end{lem}

\subsubsection{The variety $X_{\mu,\nu}(b)$}
\label{SS:variety Xmunub}
In this paper, we also need to consider the intersection of ADLVs with Schubert varieties in the affine Grassmannian. 
Let $b\in G(L)$ and $\mu,\nu$ two dominant coweights, we define
\begin{equation}
\label{E: fin piece of ADLV}
X_{\mu,\nu}(b)= X_\mu(b)\cap \Gr_\nu,
\end{equation}
i.e. $X_{\mu,\nu}(b)$ classifies those diagrams as in \eqref{E:moduli problem of Xmu(b)}, with the extra condition that $\inv(\beta_0)\preceq \nu$. Alternatively, we can define it via the Cartesian diagram, similar to \eqref{E:cartesian DL}
\begin{equation}
\label{E: fin cartesian DL}
\begin{CD}
X_{\mu,\nu}(b)@>>> \Gr_{\nu}\tilde\times\Gr_{\mu}\\
@VVV@VV\pr_1\times m_{\nu,\mu}V\\
\Gr_\nu@>1\times b\sigma >> \Gr_{\nu}\times \Gr.
\end{CD}
\end{equation}
Similarly, for a sequence $\mmu$ of dominant coweights, we can define $X_{\mmu,\nu}(b)$ by replacing $\Gr_\mu$  \eqref{E: fin cartesian DL} by $\Gr_\mmu$.

Note that $X_{\mu,\nu}(b)$ is equipped with an action of $J_b(F)\cap G(\mO_L)$, and the pair $(X_{\mu,\nu}(b), J_b(F)\cap G(\mO_L))$ depends on $b$ up to $\sigma$-conjugacy by elements of $G(\mO_L)$. So it is natural to introduce the set $A(G)$, which is the quotient of $G(L)$ by $G(\mO_L)$ by the $\sigma$-conjugation action. There is a natural map (the Newton map)
\begin{equation}
\label{E: SetNewton}
\mN: A(G)\to B(G),
\end{equation}
sending an element in $G(L)$ up to $\sigma$-conjugacy by $G(\mO_L)$ to its underlying $\sigma$-conjugacy class.
In fact $A(G)$ is exactly the set of $\bar k$-points of the moduli of local $G$-shtukas $\Sht^\loc$ introduced in \S \ref{S: Moduli of loc Sht}. We denote $A(G,\mu)$ to be the preimage of $B(G,\mu)$ under $\mN$. Note that $A(G,\mu)$ can be explicitly presented as the quotient of $\overline{G(\mO_L)\varpi^\mu G(\mO_L)}$ by $\sigma$-conjugation by $G(\mO_L)$.

\subsection{Unramified elements in $B(G,\mu)$}
\label{S: unram element}
In this subsection we give a description of the set of unramified elements in $B(G,\mu)$ in terms of the dual group $\hat G$. Although this is just a small subset of $B(G,\mu)$, it already contains many interesting examples (see Remark \ref{classification of minuscule}). 

\subsubsection{The set of unramified elements}
\label{SS:unramified elements}
We first give a definition of unramified elements in $B(G)$, different from the one given in the introduction (but the two definitions are the same by Corollary \ref{C: two def of unramified} below). For this purpose, we need to choose a rational Borel $B\subset G$ and fix an embedding of the abstract Cartan $T\subset B$.
We call an element $b\in B(G)$ unramified if it is contained in the image of the natural map $B(T)\to B(G)$, and denote the set of unramified elements in $B(G)$ by $B(G)_{\on{unr}}$. Note that since any two such choices of embeddings are rationally conjugate, this set is independent of the choices.\footnote{Recall that the set of basic elements in $B(G)$, on the other hand, is the image of $B(T')\to B(G)$ for some elliptic maximal $F$-torus of $G$, see \cite[Proposition 5.3]{Koisocry}.}
 
Recall that we denote by $\xcoch(T)_\sigma$ the $\sigma$-coinvariants of $\xcoch(T)$. For $\la\in\xcoch(T)$, let $\la_\sigma$ denote its projection to $\xcoch(T)_\sigma$. The relative Weyl group $W_0$ of $G$ acts on $\xcoch(T)_\sigma$.

\begin{lem}
\label{L: unramified elements}
The canonical map $\xcoch(T)\to B(G),\ \la\mapsto [\varpi^\la]$ induces a bijection
\[\xcoch(T)_\sigma/W_0\cong  B(G)_{\on{unr}}.\]
\end{lem}
\begin{proof}
Since the map ${\xcoch(T)}_\sigma\to B(T), \ \la\mapsto [\varpi^\la]$ is bijective, the map in the lemma is clearly surjective. Let $\la_1,\la_2\in\xcoch(T)$, and write $b_i=[\varpi^{\la_i}]$. Assume that $b_1=b_2$. Then since $\kappa_G(b_1)=\kappa_G(b_2)$, $\la_1-\la_2=(\sigma-1)\nu+\ga$ for some $\nu\in\xcoch(T)$ and $\ga\in Q^\vee$, where $Q^\vee\subset \xcoch(T)$ denotes the coroot lattice of $G$. Up to $\sigma$-conjugation, we may replace $\la_1$ by $\la_1-(\sigma-1)\nu$, so assume that $\la_1-\la_2=\ga$.
Since $\nu_{b_1}=\nu_{b_2}$, after conjugation of $\la_2$ by an element in $W_0$ we may further assume that $\bar{\la}_1=\bar{\la}_2$ (where $\bar{\la}_j$ is defined as in \eqref{Newton point}).  Therefore, $\sum_{i=0}^{m-1} \sigma^i(\ga)=0$. Since $\sigma$ permutes a basis of the coroot lattice (i.e. the set of the simple coroots), $\ga=(\sigma-1)\delta$ for some coroot $\delta$. It follows that the image of $\la_1$ and $\la_2$ in ${\xcoch(T)}_\sigma/W_0$ coincide, i.e., the map ${\xcoch(T)}_{\sigma}/W_0\to B(G)_{\on{unr}}$ is also injective.
\end{proof}

To continue, we recall some combinatorics of $\xcoch(T)_\sigma$. Recall that $\Delta\subset \xch(T)$ (resp. $\Delta^\vee\subset \xcoch(T)$) denotes the set of simple roots (resp. coroots) of $G$. Then we can define the set of dominant elements in $\xcoch(T)_\sigma$ as
$$\xcoch(T)_\sigma^+=\Big\{\la_\sigma\in\xcoch(T)_\sigma\;\Big|\; \Big\langle \la,\sum_{i=0}^{m-1} \sigma^i(\al)\Big\rangle\geq 0 \mbox{ for every } \al\in \Delta\Big\}.$$
The natural map $\xcoch(T)_\sigma^+\to \xcoch(T)_\sigma/W_0$ is an isomorphism.
In addition, the partial order $``\preceq"$ on $\xcoch(T)$ descends to a partial order on $\xcoch(T)_\sigma$. Explicitly, for $\la_\sigma, \mu_\sigma\in\xcoch(T)_\sigma$,  $\la_\sigma\preceq \mu_\sigma$ if $\mu_\sigma-\la_\sigma$ is a positive \emph{integral} combination of elements in $\Delta^\vee_\sigma:=\on{Im}(\Delta^\vee\to \xcoch(T)_\sigma)$. Note the there is a natural partial order preserving injective map 
$$ \xcoch(T)_\sigma^+\to \xcoch(T)_{\bQ}^{\sigma,+}\times \pi_1(G)_\sigma, \quad \mu_\sigma\mapsto (\bar{\mu}, [\mu]).$$
We can reformulate the above lemma as follows.
\begin{lem}
\label{L: unramified elements parameterization}
The following diagram is commutative, and all maps in the diagram preserves the partial order
\[\xymatrix{
B(G)_{\on{unr}} \ar[r]^-\cong \ar[d] & \xcoch(T)_\sigma^+ \ar[d]\\
B(G) \ar[r] & \xcoch(T)_{\bQ}^{\sigma,+}\times \pi_1(G)_\sigma.
}\]
\end{lem}
\begin{cor}
\label{L: unramified elementsII}
Under the identifications $\xcoch(T)_\sigma^+\cong \xcoch(T)_\sigma/W_0\cong B(G)_{\on{unr}}$ as above, 
$$B(G)_{\on{unr}}\cap B(G,\mu)=\{\la_\sigma\in\xcoch(T)_\sigma^+\mid \la_\sigma\preceq\mu_\sigma\}.$$
\end{cor}

\begin{rmk}\label{R:center connect}
Note that the natural map $\xcoch(T)^+\to \xcoch(T)^+_\sigma$ may not be surjective.
For example, let $G=
(\Res_{E/F}\SL_2)/\mu_2$, where $E/F$ is an unramified quadratic extension and $\mu_2\subset \Res_{E/F}\mu_2=Z(\Res_{E/F}\SL_2)$. The coweight lattice of $G_\ad=\Res_{E/F}\on{PGL}_2$ is canonically identified with $\bZ^2$ in the usual way such that $\bZ^2_{\geq 0}$ is the semigroup of dominant coweights and that the Frobenius $\sigma$ acts on $\bZ^2$ by permuting the two factors of $\bZ^2$. 
Then $\xcoch(T)$ is identified with the subgroup of $\bZ^2=\xcoch(T_\ad)$ generated by $(2,0), (0,2)$ and $(1,1)$. Now, let $\tau=(1,-1)\in\xcoch(T)$. Then $\tau_\sigma\in\xcoch(T)_\sigma^+$ but this element is not in the image of $\xcoch(T)^+$.

 The failure of the surjectivity of the above map will cause some complications in the sequel. However, if $Z_G$ is connected, there exists fundamental coweights $\omega_i\in\xcoch(T)$ such that $\langle \omega_i,\alpha_j\rangle=\delta_{ij}$ for all $\alpha_j\in \Delta$. Then it follows that the natural map $\xcoch(T)^+\to \xcoch(T)^+_\sigma$ is surjective.
\end{rmk}

Now let $\hat G^\sigma$ denote the $\sigma$-fixed points of $\hat G$. Although this is a possibly non-connected reductive group, its representation theory behaves as if the connected reductive groups:
First $\hat G^\sigma$ contains the diagonalizable subgroup $\hat T^\sigma$, whose character group is ${\xcoch(T)}_\sigma$; Second, the combinatoric structure on $\xcoch(T)_\sigma$ as discussed above allows one to talk about the highest weight of a finite dimensional representation of $\hat G^\sigma$ (with respect to $\hat T^\sigma$). Using the fact that the natural map $\hat T^\sigma/\hat T^{\sigma,\circ}\to \hat G^\sigma/\hat G^{\sigma,\circ}$ is an isomorphism (see \cite[Lemma 4.6]{Z14}), the usual highest weight theory for representations of a connected reductive group extends to $\hat G^\sigma$ without change (see \cite[Lemma 4.10]{Z14}). In particular, for every $\mu_\sigma\in \xcoch(T)^+_\sigma$, there is a unique irreducible representation of $\hat G^\sigma$ of highest weight $\mu_\sigma$, and every irreducible representation is of this form. Given $\la_\sigma,\mu_\sigma\in \xcoch(T)^+_\sigma$, $\la_\sigma\preceq \mu_\sigma$ if and only if $\la_\sigma$ appears as a weight for the representation $V_{\mu_\sigma}$.

Using the above observations, we can show the following.
\begin{lem}
\label{C:nonempty DL}
The unramified element $[\varpi^\tau]$ belongs to $B(G,\mu)$ if and only if the following subspace of $V_\mu$
\begin{equation}
\label{E:Vmmu tau}
\bigoplus_{\lambda \in \tau + (\sigma-1)\XX_\bullet} V_\mu(\lambda)
\end{equation}
is nonzero.
\end{lem}
\begin{proof}
To prove the lemma, we are free to replace $\tau$ by $w(\tau)$ for some $w\in W_0$. Therefore, we may assume that $\tau_\sigma\in \xcoch(T)_\sigma^+$.
The above space \eqref{E:Vmmu tau} is just the $\tau_\sigma$-weight space of $V_\mu|_{\hat G^\sigma}$, which has $\mu_\sigma$ as its unique highest weight. Therefore, \eqref{E:Vmmu tau} is non-zero if and only if $\tau_\sigma\preceq \mu_\sigma$, which by Corollary \ref{L: unramified elementsII} is equivalent to $[\varpi^\tau]\in B(G,\mu)$.
\end{proof}

\begin{rmk}
The above criterion shows that if $[\varpi^\tau]\in B(G,\mu)$, then $[\varpi^{w(\tau)}]\in B(G,\mu)$ for every $w\in W$ (although $\tau$ and $w(\tau)$ may not be $\sigma$-conjugate if $w\not\in W_0$).
\end{rmk}

\subsubsection{Unramified basic elements}
Next, we discuss when $B(G,\mu)\cap B(G)_{\on{unr}}$ contains the basic element. Recall that we have defined
\[\Lambda=\Big\{\la\in \xcoch(T)\;\Big|\; \sum_{i=0}^{m-1} \sigma^i(\la)\in\xcoch(Z_G)\Big\}\quad\textrm{and} \quad V_\mu^\Tate=\bigoplus_{\la\in\La}V_\mu(\la).\]
\begin{lem}\label{L:Tate weight}
There exists some $\tau\in\Lambda$ such that $V_\mu^\Tate=\bigoplus_{\la\in\tau+(\sigma-1)\xcoch}V_\mu(\la)$.
\end{lem}
\begin{proof}
Let $\la_1,\la_2\in \La$ be two weights appearing in $V_\mu$. Then $\varpi^{\la_i}$ represents the basic element in $B(G,\mu)$. Therefore, $\la_1=w(\la_2)+(\sigma-1)\nu$ for some $w\in W_0$ and $\nu\in\xcoch(T)$ by Lemma \ref{L: unramified elements}. Note that $w(\la_2)-\la_2$ is an element in the coroot lattice of $G$ and 
$$\sum_{i=0}^{m-1} \sigma^i(w(\la_2)-\la_2)=w(\sum_{i=0}^{m-1}\sigma^i(\la_2))-\sum_{i=0}^{m-1}\sigma^i(\la_2)=0.$$ Therefore, $w(\la_2)-\la_2=(\sigma-1)\nu'$ for some $\nu'$ in the coroot lattice. Therefore, $\la_1-\la_2=(\sigma-1)(\nu+\nu')$. The lemma follows.
\end{proof}

\begin{prop}\label{unramified basic}
The following are equivalent.
\begin{enumerate}
\item[(1)\;]  The basic element in $B(G,\mu)$ is unramified;
\item[(2)\;] $V_\mu^\Tate\neq 0$;
\item[(3)\;] $\on{def}_G(b)=0$ for the basic element $b\in B(G,\mu)$.
\item[(3)'] $\dim X_\mu(b)= \langle \rho, \mu\rangle$ for any $b$ representing the basic element of $B(G,\mu)$.
\item[(4)\;] $\mu_\ad$, regarded as a character of $\hat T_\s$, is trivial on $Z(\hat{G}_\s)^{\Ga_F}$.
\end{enumerate}
If these conditions hold, and in addition $Z_G$ is connected, then the basic element in $B(G,\mu)$ can be represented as $\varpi^\tau$ for some $\tau\in\xcoch(Z_G)$.
\end{prop}
Note that the last statement may fail if $Z_G$ is not connected. See Remark \ref{R:center connect}.
\begin{proof}
The equivalence of (1) and (2) follows from Lemma \ref{C:nonempty DL} and Lemma \ref{L:Tate weight}. The equivalence of (3) and (3)' follows from the dimension formula of ADLVs in the affine Grassmannian \cite{Ham,Z} (and the fact the Newton point $\nu_b$ of a basic element is central).

Now if $[\varpi^\tau]$ is basic, then $J_{\varpi^\tau}=G$ and therefore, $\on{def}_G(b)=0$. Conversely, $\on{def}_G(b)=0$ implies that $b_\ad$ is $\sigma$-conjugate to $1$ in $G_\ad(L)$. Since $G_\ad(L)$ is generated by the image of $G(L)\to G_\ad(L)$ and $T_\ad(L)$, $b$ is $\sigma$-conjugate to an unramified element. If in addition $Z_G$ is connected, then $G(L)\to G_\ad(L)$ is surjective and therefore $b$ is $\sigma$-conjugate to $\varpi^\tau$ in $G(L)$ for some $\tau\in\xcoch(Z_G)$.

Finally, the identity element represents the basic element in $B(G_\ad,\mu_\ad)$ if and only if $[\mu_\ad]=0$ in $\pi_1(G)_{\Ga_F}=(Z(\hat G_\s)^{\Ga_F})^D$. Therefore, (3) and (4) are equivalent.  
\end{proof}

\begin{remark}\label{classification of minuscule}

Note that if $G$ is a split adjoint group, a dominant coweight satisfying conditions in Proposition~\ref{unramified basic} is necessarily in the coroot lattice, and therefore cannot be minuscule. On the other hand, if $G$ is non-split, there are minuscule coweights satisfying these conditions.
Here is a list of all such minuscules $\mu$ for simple adjoint groups. The group $G$ is necessarily quasi-split but non-split form of type $\mathsf{A},\mathsf{D},\mathsf{E}_6$.  In the following two tables, $\omega_i$ is the fundamental coweight such that $\langle\omega_i,\al_j\rangle=\delta_{ij}$ for simple roots $\{\al_j\}$, and our numeration of simple roots follows from Bourbaki \cite[pp 250--275]{Bourbaki}.
\medskip
\begin{center}
\begin{tabular}{|c | c| c| c|}
\hline
Type of $G$ & Minuscule coweight $\mu$  & $\dim V_\mu$ &$\dim V_\mu^{\Tate}$ \\ 
\hline
${^2}\mathsf A_n$, $n>1$ odd & $\omega_{i}, i=2,4,\ldots, n-1$  & $\begin{pmatrix} n+1 \\ i \end{pmatrix}$ & $\begin{pmatrix} (n+1)/2 \\ i/2 \end{pmatrix}$\\ 
\hline
${^2}\mathsf A_n$, $n$ even & $\omega_i, i=1,\ldots,n$ & $\begin{pmatrix} n+1 \\ i \end{pmatrix}$ & $\begin{pmatrix} n/2 \\ \lfloor i/2\rfloor \end{pmatrix}$ \\
\hline
${^2}\mathsf D_n$ & $\omega_1$  &  $2n$ & $2$ \\  
\hline
${^3}\mathsf D_4$ & $\omega_1,\omega_3,\omega_4$ &  $8$ & $2$\\
\hline
${^2}\mathsf E_6$ & $\omega_1,\omega_6$ & $27$ & 3
\\ \hline
\end{tabular}
\end{center}
From the table, we see that the dimension of $V_\mu^\Tate$ can be large in the unitary group case, even for minuscule cocharacters. 
Many more examples that satisfy the equivalent conditions of Proposition \ref{unramified basic} come from the cases when $G$ is a restriction of scalar along an unramified extension. We give some examples.
 
Let $G=\Res_{F'/F}G'$ where $F'/F$ is a degree $d$ unramified extension and $G'$ an absolute simple adjoint group over $F'$. 
We identify 
$$\xcoch(T)=\prod_{\varphi\in\Hom_F(F', \overline{F})}\xcoch(T')$$ in the usual way. Under this identification, we write $\mu\in\xcoch(T)$ as $\{ \mu_\varphi\}\in \prod_{\varphi:F'\to \overline{F}}\xcoch(T')$. On the dual group side,
$\hat G=\prod_{\varphi: F'\to \overline{F}} \hat{G'}$ and the highest weight representation $V_\mu$ of $\hat G$ is identified with the representation  $\boxtimes_{\varphi} V_{\mu_{\varphi}}$ of $\prod_{\varphi}\hat G'$. For simplest, we assume that $G'$ is split. Then it is easy to see
\[V_\mu^\Tate=\big(\bigotimes_{\varphi}V_{\mu_{\varphi}}\big)(0).\]
In particular,
\medskip
\begin{center}
\begin{tabular}{|c | c| c| c| c|}
\hline
Type of $G'$ &  $[F':F]$ & Minuscule coweight $\mu=\{ \mu_\varphi\}$  & $\dim V_\mu$ & $\dim V_\mu^{\Tate}$ \\ 
\hline
$A_1$ & $2g$ & $\mu_{\varphi}=\omega$  & $2^g$ & $\begin{pmatrix} 2g \\ g \end{pmatrix}$\\ 
\hline
$B_n$ & $2$ & $\mu_{\varphi}=\omega_1$ & $4n^2$ & $2n$ \\
\hline
$C_n$ & $2$ & $\mu_{\varphi}=\omega_n$  &  $2^{2n}$ & $2^n$ \\  
\hline
$D_n$ & $2$ & $\mu_{\varphi}=\omega_1$ &  $4n^2$ & $2n$
\\ \hline
\end{tabular}
\end{center}
\medskip

Note that all pairs $(G,\mu)$ in the above two tables, except ${^2}E_6$, can appear as the local component of a Shimura datum of abelian type.
\end{remark}

We also mention a corollary, which shows the two definitions of the unramified elements, one given in \S\ref{SS:unramified elements} and one given here, are the same.
\begin{cor}
\label{C: two def of unramified}
The following are equivalent.
\begin{enumerate}
\item $b\in B(G)$ can be represented as $\varpi^\la$ for some $\la\in\xcoch(T)$.
\item $\on{def}_G(b)=0$, i.e. the $F$-rank of $G$ and $J_b$ coincide.
\end{enumerate} 
\end{cor}
\begin{proof}
Clearly (1) implies (2). Conversely, assume $\on{def}_G(b)=0$. Since $b$ is $\sigma$-conjugate to some basic element in a Levi subgroup of $G$, we can apply Proposition \ref{unramified basic} to this Levi subgroup to conclude (choose any $\mu$ such that $b\in B(G,\mu)$).
\end{proof}

\subsubsection{Twisted centralizer for unramified elements}\label{tcen}
We  recall some facts of the twisted centralizer for the unramified elements. We fix $\mO$-embeddings $T\subset B\subset G$, so elements in $B(G)_{\on{unr}}$ can be represented by $\varpi^\tau$ for $\tau\in \xcoch(T)$. We fix such a representative. For simplicity, we write $J_\tau=J_{\varpi^\tau}$.  It is isomorphic to a Levi subgroup of $G$, and naturally contains $T$ as a maximal torus. With respect to this torus,
\begin{equation}\label{E: Deltatau}
\Delta_\tau=\Big\{\al\in\Delta\; \Big|\; \langle \sum_{i=0}^{m-1} \sigma^i(\al),\tau\rangle=0\Big\}\subset \Delta
\end{equation}
form a set of simple roots of $J_\tau$. Let $\Phi^+_\tau\subset\Phi^+$ denote the corresponding set of positive roots. The relative Weyl group of $J_\tau$ (for the torus $T$) then is
\[(W_\tau)_0:=\{w\in W_0\mid w(\nu_{\varpi^\tau})=\nu_{\varpi^\tau}\}.\]

Let $B_\tau$ and $U_\tau$ denote the corresponding Borel and unipotent radical with respect to this set of simple roots. Then
\begin{equation}\label{E: Utau}
U_\tau(F)=J_{\varpi^\tau}(F)\cap U(L)=\{u\in U(L)\mid \varpi^{\tau}\sigma(u)\varpi^{-\tau}=u\}, \quad B_\tau(F)=J_\tau(F)\cap B(L)=T(F)U_\tau(F).
\end{equation}
 We have the Bruhat decomposition
 \begin{equation}\label{E: bruhatJtau}
 J_\tau(F)=\bigcup_{w\in (W_\tau)_0}B_\tau(F)\dot{w} B_\tau(F),
 \end{equation}
where $\dot{w}$ is a lifting of $w$ to $G(L)$ satisfying $\sigma(\dot{w})\dot{w}^{-1}=\varpi^{w(\tau)-\tau}$.

Now, if we assume further that $\tau\in\xcoch(T)^+$,\footnote{As explained in Remark~\ref{R:center connect}, the map $\xcoch(T)^+ \to \xcoch(T)^+_\sigma$ need not be surjective. So this is a genuine assumption.} then \eqref{E: Deltatau} simplifies as
$$\Delta_\tau=\{\al\in\Delta\mid \langle \sigma^i\al,\tau\rangle=0 \mbox{ for all } i\}.$$ 
In this case, $J_\tau$ has a natural hyperspecial integral structure, i.e, the $\mO$-subgroup scheme of $G$ generated by $\{T,U_{\pm\al}, \al\in\Phi^+_\tau\}$. Then 
\[J_\tau(F)=B_\tau(F)J_\tau(\mO)=\bigcup_{\chi\in \xcoch(T)^\sigma} U_\tau(F)\varpi^\chi J_\tau(\mO)\]
is an Iwasawa decomposition of $J_\tau$.

\subsection{ADLVs via semi-infinite orbits.}
\label{S: ADLV and semiinfinite}
In this subsection, we fix dominant coweights $\mu$ and an unramified element $[\varpi^\tau]$ in $B(G,\mu)$. Then $X_\mu(\varpi^\tau) \neq \emptyset$ by Theorem~\ref{L:nonemptyness of Xmub}. By the discussion of the previous subsection, we may assume that $\tau_\sigma\in \xcoch(T)_\sigma^+$.
We revisit the arguments in \cite[\S 2]{GHKR}. A more detailed analysis (together with a generalization to unramified groups) will allow us to give a description of the irreducible components of $X_\mu(\varpi^\tau)$ up to the action of $B_\tau(F)$. 
To simplify the notation, we write $X_\mu(\tau)$ instead of $X_\mu(\varpi^\tau)$.

\subsubsection{}
Let $\sigma$ denote the $q$-Frobenius.  
Let $S_\mu\subset \Gr\otimes k'$ be a semi-infinite orbit, where $k'$ is the field of the definition for $\mu$. Then we have the following diagram
\begin{equation}\label{q-Frob S}
\xymatrix@C=40pt{
S_\mu\ar[r]\ar[d]&S_{\sigma(\mu)}\ar[r]\ar[d]&S_\mu\ar[d]\\
\Gr\otimes k'\ar^{\sigma\otimes 1}[r]& \Gr\otimes k'\ar^{1\otimes \sigma}[r]&\Gr\otimes k'.
}\end{equation}
We denote the map $S_\mu\to S_{\sigma(\mu)}$ by $\sigma$ for simplicity (since it is induced by the $q$-Frobenius endomorphism of $\Gr$). 

We define 
$$Y_\nu: =  S_\nu\cap X_\mu(\varpi^\tau),$$ which is a locally closed sub ind-scheme of $X_\mu(\tau)$. It fits into the following Cartesian diagram
\begin{equation}\label{first}
\xymatrix@C=20pt{
Y_\nu  \ar[rr]  \ar@{_{(}->}[d] && (S_{\nu}\tilde\times S_{\lambda})\cap (\Gr\tilde\times\Gr_{\mu}) \ar[d]^{\pr_1 \times m} \\
S_\nu \ar[rr]^-{1\times \varpi^\tau\sigma} && S_\nu\times S_{\tau+\sigma(\nu)},
}
\end{equation}
where $\la=\tau+\sigma(\nu)-\nu$.

Since $(S_{\nu}\tilde\times S_{\lambda})\cap(\Gr\tilde\times\Gr_{\mu}) = S_\nu \tilde \times (S_\lambda \cap \Gr_\mu)$ by \eqref{factorizationII},
every $\bbb\in \MV_\mu(\la)$ gives a closed subset $Y_\nu^\bbb = Y_\nu^\bbb(\tau)$ of $Y_\nu$ that fits into the Cartesian diagram
\begin{equation}\label{refined}
\xymatrix@C=20pt{
\!\!\!\!\!\!\!\!\!\!\!\!\!\!\!\!\!\!\!\!Y^\bbb_\nu  = Y^\bbb_\nu(\tau)  \ar[rr]  \ar@{_{(}->}[d] && S_\nu \tilde \times (S_\lambda \cap \Gr_\mu)^\bbb\ar[d]^{\pr_1 \times m} \\
S_\nu \ar[rr]^-{1\times \varpi^\tau\sigma} && S_\nu\times S_{\tau+\sigma(\nu)}.
}
\end{equation}

We quickly mention a couple of simple properties of $Y_\nu^\bbb$.

\begin{lem}
\label{L:easy properties of Ynubtau}
We fix $\bbb\in \MV_\mu(\la)$ with $\la = \tau+\sigma(\nu)-\nu$.
\begin{enumerate}

\item Let $\chi\in \xcoch(T)^\sigma$ be a $\sigma$-invariant coweight of $T$. Then multiplication (on the left) by $\varpi^\chi\in J_\tau(F)$ induces an isomorphism $Y_\nu^\bbb(\tau)\cong Y_{\nu+\chi}^\bbb(\tau)$.

\item Let $\tau'=\tau+\delta- \sigma(\delta)$. Then the isomorphism $X_\mu(\tau)\to X_\mu(\tau'),\  gL^+G\mapsto \varpi^\delta gL^+G$ induces an isomorphism $Y_\nu^\bbb(\tau)\cong Y_{\nu+\delta}^\bbb(\tau')$. 

\item Let $G\to G'$ be a central isogeny, and $(\mu',\tau',\nu')$ be the image of $(\mu,\tau,\nu)$ under the map $\xcoch(T)\to \xcoch(T')$. Then the map $\Gr_G \to \Gr_{G'}$ induces a natural bijection $\MV_\mu(\lambda) \cong \MV_{\mu'}(\lambda')$ and let $\bbb'$ denote the image of $\bbb$. Then the map $\Gr_G \to \Gr_{G'}$ induces a natural isomorphism $Y_{\nu'}^{\bbb'}(\tau') \cong Y_{\nu}^\bbb(\tau)$.

\end{enumerate}
\end{lem}
\begin{proof}
Each of the statement is clear from the Cartesian diagram \eqref{refined} as the natural morphism induces isomorphisms on all terms in that diagram.
\end{proof}

Now we simply write $Y_\nu^\bbb$ for $Y_\nu^\bbb(\tau)$.
The following proposition gives a key property of $Y_\nu^\bbb$. 
\begin{prop}\label{P: prop of Ynu}
We fix $\bbb\in \MV_\mu(\la)$.
\begin{enumerate} 
\item The subset $Y_\nu^\bbb$ is non-empty if and only if $\la=\tau+\sigma(\nu)-\nu$. 
\item The group 
$$U_\tau(F):=J_\tau(F)\cap U(L)=\{u\in U(L)\mid \varpi^\tau\sigma(u)\varpi^{-\tau}=u\}$$ 
acts transitively on the set of irreducible components of $Y_\nu^\bbb$. Moreover, all irreducible components of $Y_\nu^\bbb$ are of dimension $\langle\rho,\mu-\tau\rangle$.
\end{enumerate}
\end{prop}
\begin{proof}
The arguments below essentially repeat \cite[\S 2]{GHKR}. 
Using Lemma~\ref{L:easy properties of Ynubtau}(2)(3), we may pass to the adjoint group of $G$ and therefore (after a $\sigma$-conjugation) assume that $\tau$ is dominant.

We consider the $L^+U$-torsors $\pi_\lambda$ defined in \eqref{E:triviallization of U(O) torsor} and its variants $\nu^\nu L^+U \varpi^{-\nu}$-torsor $\pi'_\nu$:\footnote{One could alternatively present the proof using the $L^+U$-torsors $\pi_\nu$ and $\pi_\lambda$, but the analogous Lemma~\ref{Lang map} will require some additional effort.}
\[
\xymatrix@R=0pt{
\pi_\la:LU \ar[r] & S_\la &
\pi'_\nu:LU \ar[r] & S_\nu
\\
\qquad u \ar@{|->}[r] &  \varpi^\la u,&
\qquad u \ar@{|->}[r] &   u\varpi^\nu.
}
\]

By definition \eqref{refined}, an element $u\in LU$ belongs to $\widetilde Y_\nu^\bbb: = \pi'^{-1}_\nu(Y^{\bbb}_\nu)$ if and only if  there exists $x\in (S_\la\cap \Gr_\mu)^\bbb$ such that 
$$\pi'_\nu( u)x= \varpi^\tau\sigma(\pi'_\nu(u)),$$ 
or equivalently, 
$$x=\pi_\la(\varpi^{-\la}\varpi^{-\nu}u^{-1}\varpi^{\tau}\sigma(u)\varpi^{\sigma(\nu)})=\pi_\la(\varpi^{-\tau-\sigma(\nu)}u^{-1}\varpi^{\tau}\sigma(u)\varpi^{-\tau} \varpi^{\tau+\sigma(\nu)}).$$ In other words, if we define the morphisms
\begin{equation}
\label{E:ftau}
f_\tau: LU\to LU,\quad u\mapsto u^{-1}\varpi^{\tau}\sigma(u)\varpi^{-\tau}, \textrm{  and}\end{equation}
\[c_{-\tau-\sigma(\nu)}: LU\to LU,\quad u\mapsto \varpi^{-\tau-\sigma(\nu)} u \varpi^{\tau+\sigma(\nu)},\]
then 
\begin{equation}\label{E: local model for Y}
\widetilde Y_\nu^\bbb: = \pi'^{-1}_\nu(Y^{\bbb}_\nu)=f_\tau^{-1}\big(c^{-1}_{-\tau-\sigma(\nu)}(\pi_\la^{-1}((S_\la\cap\Gr_\mu)^{\bbb})) \big),
\end{equation}
or equivalently, we have the following Cartesian diagram
\begin{equation}
\label{E:tilde Ynub Cartesian diagram}
\xymatrix{
Y_\nu^\bbb \ar[d]
& \ar[l] \widetilde Y_\nu^\bbb \ar[d] \ar[rrr] &&& (S_\la\cap\Gr_\mu)^{\bbb} \ar[d]
\\
S_\nu &\ar[l]_{\pi'_\nu} LU \ar[r]^-{f_\tau} & LU \ar[rr]^-{\pi_\lambda\circ c_{-\tau-\sigma(\nu)}} && S_\lambda.
}
\end{equation}
The key of the proof lies in the fact that the map $f_\tau$ behaves like the Lang map for an algebraic group defined over a finite field. More precisely, we have the following.

\begin{lem}\label{Lang map}
Let $\tau$ be a dominant coweight.
\begin{enumerate}
\item The morphism $f_{\tau}$ in \eqref{E:ftau} restricts to a ``pro-\'etale" surjective morphism $f_\tau: L^+U \to L^+U$ with Galois group $U_\tau(\calO)$.
\item
The morphism $f_\tau: LU \to LU$ is an inductive limit of ``pro-\'etale" surjective maps, and $U_{\tau}(F)$
acts simply-transitively on the fibers of $f_{\tau}$ by left multiplication. 
\end{enumerate}
\end{lem}
\begin{proof}
(1) Note that $\varpi^{\tau}L^+U^{(r)}\varpi^{-\tau}\subseteq L^+U^{(r)}$ for every principle congruence subgroup $L^+U^{(r)}$ of $L^+U$ (see \eqref{E:cong subgroup} for the definition). It follows that $f_\tau$ induces a morphism $f_\tau: L^+U \to L^+U$ and a morphism \[
f_{\tau, r}: L^rU \to L^rU, \quad u \mapsto u^{-1} \varpi^\tau \sigma(u) \varpi^{-\tau}
\]
for each $r$. But $f_{\tau, r}$ is naturally defined before passing to the perfection and its \'etaleness follows from the fact that the induced tangent map is multiplication by $-1$. So  $f_{\tau, r}$ is an  \'etale covering with the Galois group $U_\tau(\mO/\varpi^r)$. Passing to the inverse limit shows that $f_\tau: L^+U \to L^+U$ is a ``pro-\'etale" cover
with Galois group $U_\tau(\mO)$. 

(2) For a dominant coweight $\chi \in \XX_\bullet(T)^{+,\sigma}$, Let $\mathrm{conj}_\chi: LU \to LU$ be the conjugation map $u\mapsto \varpi^\chi u \varpi^{-\chi}$. It is clear that $f_\tau \circ \mathrm{conj}_\chi = \mathrm{conj}_\chi \circ f_\tau$ forms a Cartesian square. 
Therefore (1) implies that $f_\tau: \varpi^{-\chi}L^+U \varpi^{\chi} \to \varpi^{-\chi}L^+U \varpi^{\chi}$ is a ``pro-\'etale" map with Galois group $\varpi^{-\chi}U_\tau(\calO) \varpi^\chi$. Taking the limit over $\chi \in \XX_\bullet(T)^{+,\sigma}$ proves (2).
\end{proof}

Now we see that $\widetilde Y_\nu^\bbb$ is the pullback of $(S_\lambda \cap \Gr_\mu)^\bbb$ along a sequence of surjective maps. Part (1) of the proposition is clear. In addition, since $(S_\la\cap\Gr_\mu)^\bbb$ is irreducible, $U_\tau(F)$ acts transitively on the set of irreducible components of $\widetilde Y_\nu^\bbb$, and hence acts transitively on the set of irreducible components of $Y^\bbb_\nu$. 

It remains to show that each irreducible component $Y_\nu^\bbb$ is of dimension $\langle\rho,\mu-\tau \rangle$. For this, we need a ``truncated version" of \eqref{E:tilde Ynub Cartesian diagram}.
We pick $r \in \ZZ$ sufficiently large so that $\varpi^{-\nu}L^+U^{(r)} \varpi^\nu \subseteq L^+U$ and $c_{-\tau-\sigma(\nu)} ( L^+U^{(r)}) \subseteq L^+U$. Then \eqref{E:tilde Ynub Cartesian diagram} induces the following Cartesian diagram
\begin{equation}
\label{E:tilde Ynub Cartesian diagram truncated}
\xymatrix{
Y_\nu^\bbb \ar[d]
& \ar[l] \widetilde Y_\nu^{\bbb,(r)} \ar[d] \ar[rrr] &&&(S_\la\cap\Gr_\mu)^{\bbb} \ar[d]
\\
S_\nu &\ar[l]_-{\pi'_\nu} LU / L^+U^{(r)} \ar[r]^-{f_\tau} & LU/ L^+U^{(r)} \ar[rr]^-{\pi_\lambda\circ c_{-\tau-\sigma(\nu)}} && S_\lambda.
}
\end{equation}
Note that $\pi'_\nu$ is smooth of dimension $r \dim U$, $f_\tau$ is \'etale, and $\pi_\lambda \circ c_{-\tau-\sigma}(\nu)$ is smooth of dimension $\dim L^+U / c_{-\tau-\sigma(\nu)} ( L^+U^{(r)})$. So we have
\begin{align*}
\dim Y_\nu^\bbb \;=\;&
\dim (S_\la\cap\Gr_\mu)^{\bbb}  +\dim \pi_\lambda - \dim \pi'_\nu
\\
\;=\;& \langle \rho, \mu+\lambda\rangle + \big( r\dim U - \langle \rho, \tau+\sigma(\nu)\rangle \big) - r\dim U = \langle \rho, \mu-\tau \rangle. \qedhere
\end{align*}
\end{proof}

\begin{cor}
\label{C: irr comp up to B}
For $\bbb\in \cup_{\la\in\tau+(\sigma-1)\xcoch}\MV_\mu(\la)$, 
let $X^{\bbb}_\mu(\tau)$ denote the closure of $\cup_{\nu\in\xcoch} Y_{\nu}^\bbb$ in $X_\mu(\tau)$. Then 
\begin{enumerate}
\item $X^{\bbb}_\mu(\tau)$ is a union of certain irreducible components of $X_\mu(\tau)$. 
\item $X^{\bbb}_\mu(\tau)\neq X^{\bbb'}_\mu(\tau)$ if $\bbb\neq \bbb'$, and $X_\mu(\tau)=\cup_{\bbb} X_\mu^\bbb(\tau)$.
\item The group $B_\tau(F)$ acts transitively on the set of irreducible components of $X_\mu^\bbb(\tau)$.
\item  $X_\mu^\bbb(\tau)$ is equi-dimensional of dimension $\langle \rho, \mu-\tau\rangle$.
\end{enumerate}
\end{cor}
\begin{proof}
Since semi-infinite orbits form a partition of the affine Grassmannian into locally closed subsets, $\{Y^\bbb_\nu\}$ form a partition of $X_\mu(\tau)$. Proposition \ref{P: prop of Ynu}(2) says that every $Y^\bbb_\nu$ is pure of dimension $\langle\rho,\mu-\tau\rangle$; so is $X_\mu^\bbb(\tau)$. It follows that if $Z$ is an irreducible component of $Y^\bbb_\nu$, then its closure $\bar{Z}$ is an irreducible component of $X_\mu(\tau)$. Therefore, $X^{\bbb}_\mu(\tau)$ is a union of certain irreducible components of $X_\mu(\tau)$, and $X_\mu(\tau)=\cup_{\bbb} X_\mu^\bbb(\tau)$. 

We claim that, if $\bbb\neq \bbb'$, then $Y_\nu^\bbb\neq Y_\nu^{\bbb'}$ Indeed, if $\bbb\in\MV_{\mu}(\la)$ and $\bbb'\in \MV_{\mu}(\la')$ for $\la\neq \la'$, then $Y_\nu^\bbb$ and $Y_{\nu}^{\bbb'}$ cannot be both nonempty by Lemma~\ref{L:easy properties of Ynubtau}(1). If $\bbb,\bbb'\in \MV_{\mu}(\la)$, then $\pi'^{-1}_{\nu}(Y_{\nu}^\bbb)\neq \pi'^{-1}_{\nu}(Y_{\nu}^{\bbb'})$ by \eqref{E: local model for Y}. Therefore, $X^{\bbb}_\mu(\tau)\neq X^{\bbb'}_\mu(\tau)$ if $\bbb\neq \bbb'$. 

Finally, by Proposition \ref{P: prop of Ynu}(2), $B_\tau(F)$ acts transitively on the set of irreducible components of $\cup_\nu Y_\nu^\bbb$, and therefore  acts transitively on the set of irreducible components of $X_\mu^\bbb(\tau)$. 
\end{proof}

We end the subsection with the following observations.

\begin{lem}
\label{L: invariance of Xmub}
We fix $\bbb\in \MV_\mu(\la)$ with $\la\in \tau+(\sigma-1)\xcoch$ and let $X_\mu^\bbb(\tau)$ be defined as above.
\begin{enumerate}
\item Let $G\to G'$ be a central isogeny, and $(\mu',\tau')$ be the image of $(\mu,\tau)$ under the map $\xcoch(T)\to \xcoch(T')$. Then one can replace $X_\mu(b)$ and $X_{\mu'}(b')$ by $X^\bbb_\mu(\tau)$ and $X^\bbb_{\mu'}(\tau')$ in the diagram in Lemma \ref{L: red to adj} and the new diagram is still Cartesian.

\item Let $\tau'=\tau+\delta- \sigma(\delta)$. Then the isomorphism $X_\mu(\tau)\to X_\mu(\tau'),\  gL^+G\mapsto \varpi^\delta gL^+G$ induces an isomorphism $X^{\bbb}_\mu(\tau)\cong X^{\bbb}_\mu(\tau')$. 
\end{enumerate}
\end{lem}
\begin{proof}
The map $\Gr_G\to \Gr_{G'}$ induces a natural bijection $\MV_\mu(\la)=\MV_{\mu'}(\la')$ for any $\la$, and the isomorphism $Y_\nu^\bbb(\tau)\cong Y_{\nu'}^\bbb(\tau')$ by Lemma~\ref{L:easy properties of Ynubtau}(3). Then (1) follows.

(2) follows from Lemma~\ref{L:easy properties of Ynubtau}(2) immediately.
\end{proof}

\subsection{Irreducible components of some ADLVs}
\label{S: irr comp of ADLVs}
In this subsection, we study in more details the irreducible components of $X_\mu(\tau)$. We continue to assume that $\tau_\sigma\in \xcoch(T)^+_\sigma$. We fix $\bbb\in \MV_\mu(\la)$, where $\la\in \tau+(\sigma-1)\xcoch$, and let $X_\mu^\bbb(\tau)$ be the union of certain irreducible components of $X_\mu(\tau)$ as constructed in the previous subsection.  

\begin{hypothesis}
\label{H:p not 23}
In this section, we assume $k\neq \bF_2$ if any of the simple factors of $J_\tau$ contains a factor of type $\mathsf{B}_n,\mathsf{C}_n,\mathsf{F}_4$, and $k\neq \bF_2,\bF_3$ if  it contains $\mathsf{G}_2$. Note that this does not depend on the choice of $\tau$ in its $\sigma$-conjugacy class because all $J_\tau$'s are isomorphic. 
\end{hypothesis}

\begin{rmk}
This Hypothesis will only be used to invoke Lemma~\ref{L: commutator}. On the other hand, we believe that the restriction of the characteristic is not necessary to the validity of Proposition~\ref{P: irrkey}, but currently we are unable to remove this extra assumption.
\end{rmk}

We need the following lemma. Recall from \S \ref{S:geom realization of G-crystal} that the set $\MV_\mu$ has a $\hat G$-crystal structure and therefore every $\bbb\in\MV_\mu(\la)$ can be attached a collection of non-negative integers $\{\varepsilon_\al(\bbb), \al\in\Delta\}$.
\begin{lem}
\label{L: finding best tau}
Assume that $Z_G$ is connected.
\begin{enumerate}
\item Let $\la\in\xcoch(T)$ and assume that $\la_\sigma=\tau_\sigma$. Then there exists a dominant coweight $\nu$ such that $\la+\nu-\sigma(\nu)$ is dominant, and that $\bbb$ lies in the image of the map $i^{\MV}_\nu:\bS_{(\nu,\mu)\mid \la+\nu}\to \MV_\mu(\la)$ defined in Lemma~\ref{L:Sat vs MV}.
\item Among all $\nu$'s satisfying the above property, there is a ``minimal" $\nu_\bbb$, unique up to adding by an element in $\xcoch(Z_G)$. Here ``minimality" means that for any other $\nu$ satisfying the above property, $\nu-\nu_\bbb$ is dominant. In addition, for every $\al\in\Delta$, at least one of the following inequalities is an equality
\[
\langle \alpha, \nu_\bbb \rangle \geq \varepsilon_\alpha(\bbb), \quad \langle \alpha_{\sigma(\alpha)}, \nu_\bbb \rangle \geq \varepsilon_{\sigma(\alpha)}(\bbb),\  \dots,\ \langle \alpha_{\sigma^{d-1}(\alpha)}, \nu_\bbb \rangle \geq \varepsilon_{\sigma^{d-1}(\alpha)}(\bbb),
\]
where $d$ is the cardinality of the $\sigma$-orbit of $\alpha$.
\end{enumerate}
\end{lem}
\begin{proof}
By \eqref{E:condition to realize as a homomorphism crystal} and the compatibility Proposition~\ref{comb decom}, $\bbb$ belongs to the image of the map $i^{\MV}_\nu$ if and only if 
\begin{equation}
\label{E: ineq1}
\langle \nu,\al\rangle \geq  \varepsilon_\al(\bbb) \quad \textrm{for all } \alpha \in \Delta.
\end{equation} On the other hand, $\la+\nu-\sigma(\nu)$ is dominant if and only if, for all simple roots $\alpha$ of $G$,
\begin{align*}
0  & \leq  \langle \lambda, \alpha \rangle + \langle \nu, \alpha \rangle - \langle \sigma(\nu), \alpha \rangle
\\
&= \phi_\alpha(\bbb) - \varepsilon_\alpha(\bbb) + \langle \nu, \alpha\rangle - \langle \nu, \sigma^{-1}(\alpha)\rangle.
\end{align*}
This is equivalent to, for every
$\sigma$-orbit $\{\al,\sigma(\al)\ldots,\sigma^d(\al)=\al\}$ of simple roots of $G$
\begin{equation}
\label{E:ineq2}
\langle \nu, \sigma^j(\alpha) \rangle  \geq \langle \nu, \sigma^{j-1}(\alpha) \rangle + \varepsilon_{\sigma^j(\alpha)}(\bbb) -  \phi_{\sigma^j(\alpha)}(\bbb), \quad j=1,2,\ldots,d.
\end{equation}
Since $Z_G$ is connected, there exist coweights $\{\omega_\alpha\}_{\alpha \in \Delta}$ such that $\langle \omega_{\al},\beta\rangle=\delta_{\al\beta}$.
Since 
$$\sum_j \big(\varepsilon_{\sigma^j(\al)}(\bbb)-\phi_{\sigma^{j-1}(\al)}(\bbb) \big)=\langle -\la, \sum_j \al_j\rangle=\langle -\tau, \sum_j\al_j\rangle\leq 0,$$
one can use $\{\omega_\al\}$ to modify $\nu$ such that $\{\langle \nu, \sigma^j(\alpha)\rangle\}$ satisfy \eqref{E: ineq1} and \eqref{E:ineq2}.

Now, let $\nu_0$ be a  coweight satisfying \eqref{E: ineq1} and \eqref{E:ineq2}. Then we can choose $\nu_\bbb \in \nu_0+\sum_{\alpha \in \Delta} \ZZ \omega_\alpha$ such that 
\begin{equation*}
\label{E:minimal condition}
\langle \alpha, \nu_\bbb\rangle = \max_{j=0, \dots, d-1}\big\{ \varepsilon_{\sigma^j(\alpha)}(\bbb) + \langle \alpha, \nu_0\rangle - \langle \alpha_{\sigma^j(\alpha)}, \nu_0\rangle \big\}.
\end{equation*}
Such $\nu_\bbb$ is minimal (in the sense of the lemma) among all coweights satisfying \eqref{E: ineq1} and \eqref{E:ineq2}, and clearly two choices of $\nu_\bbb$ differ by an element in $\xcoch(Z_G)$.
\end{proof}

\begin{rmk}
\label{Ca:minimal condition}
\begin{enumerate}
\item Let $\nu = \nu_\bbb$ be a (choice of) minimal coweight as above, \eqref{E: ineq1} may not be an equality for all $\al\in\Delta$. 

\item
For a (choice of) minimal weight $\nu = \nu_\bbb$ above, let $\tau_\bbb=\la+\nu_\bbb-\sigma(\nu_\bbb)$. Note that $\tau_\bbb$ is well-defined up to adding an element in $(1-\sigma)\xcoch(Z_G)$ (so in particular if $G$ is adjoint, $\tau_\bbb$ is uniquely determined by $\bbb$).  
In any case, it belongs to  $ \tau+(1-\sigma)\xcoch(T)$, and therefore $\varpi^{\tau_\bbb}\in A(G,\mu)$ is a canonical representative of the unramified element $[b]=[\varpi^\tau]\in B(G,\mu)$ ``adapted" to $\bbb$. For example, if $[b]$ is basic, then $\tau_\bbb$ is central and $\varpi^{\tau_\bbb}$ usually called superspecial.

\item 
But also note that the class of $\tau_\bbb$ in $\xcoch(T)/(1-\sigma)\xcoch(Z_G)$ depends on $\bbb$ in general.
Therefore, if the set $\sqcup_{\la_\sigma=\tau_\sigma}\MV_\mu(\la)$ contains more than one element, there may not be a canonical representative of $[b]$.  See \cite{XZ} for an example.
\end{enumerate}
\end{rmk}

Our goal is to study the irreducible components of $X_\mu^\bbb(\tau)$. By Lemma \ref{L: invariance of Xmub}(2), $X_\mu^\bbb(\tau)\cong X_\mu^\bbb(\tau_\bbb)$. Therefore, after $\sigma$-conjugacy, we may replace $\tau$ by $\tau_\bbb$. In particular, $\tau$ is assumed to be dominant.
Then by \S\ref{tcen}, $J_\tau$ (and  $U_\tau$) has a natural $\mO$-structure in this case. We denote by $J_{\tau,x_0}(\mO)\subset J_{\tau}(F)$ the corresponding hyperspecial subgroup.

Let $\bba\in \bS_{(\nu,\mu)\mid \tau+\sigma(\nu)}$ such that $\bbb=i^\MV_\nu(\bba)$. By Lemma~\ref{L:Sat vs MV}, such $\bba$ is unique and inside $\Gr_\nu\times\Gr_{\tau+\sigma(\nu)}$,
$$\Gr_{(\nu,\mu)\mid \tau+\sigma(\nu)}^{0,\mathbf{a}}\cap \big((S_\nu\cap \Gr_\nu)\times (S_{\nu+\la}\cap \Gr_{\tau+\sigma(\nu)})\big)=(S_\nu\cap \Gr_\nu)\tilde\times(S_\la\cap \Gr_\mu)^\bbb.$$ 
Let us define a closed subset $X_{\mu,\nu}^\bba(\tau)\subset X_{\mu,\nu}(\tau)$ via the Cartesian square
\begin{equation}
\label{E: possible irr comp}
\begin{CD}
X_{\mu,\nu}^\bba(\tau)@>>>\Gr_{(\nu,\mu)\mid \tau+\sigma(\nu)}^{0,\bba}\\
@VVV@VVV\\
\Gr_{\nu}@>1\times\varpi^\tau\sigma>>\Gr_{\nu}\times\Gr_{\tau+\sigma(\nu)}.
\end{CD}
\end{equation}
Compare with the definition of $X_{\mu,\nu}(\tau)$ in \eqref{E: fin cartesian DL}, we replace $\Gr_{(\nu,\mu)\mid \tau+\sigma(\nu)}^{0}$ by one of its irreducible component $\Gr_{(\nu,\mu)\mid \tau+\sigma(\nu)}^{0,\bba}$. 

\begin{thm}
\label{C: unique comp}
Assume that $Z_G$ is connected. Let $\bbb\in \MV_\mu(\la)$. Let $\nu=\nu_\bbb$ and $\tau=\tau_\bbb$ as in Lemma \ref{L: finding best tau}.
Then scheme $X_{\mu,\nu}^\bba(\tau)$ has a unique irreducible component $X_\mu^{\bbb,x_0}(\tau)$\footnote{Here the superscript $x_0$ indicates that the stabilizer group of this irreducible component in $X_\mu(\tau)$ under the $J_\tau(F)$-action is $J_{\tau,x_0}(\calO)$.} of dimension $\langle\rho, \mu-\tau\rangle$ (all other possible irreducible components of would have dimension strictly smaller that $\langle\rho, \mu-\tau\rangle$).
In addition, $X_{\mu,\nu}^\bba(\tau)\cap \mathring{\Gr}_{\nu}$ is irreducible. In particular,
$$\mathring{X}_\mu^{\bbb,x_0}(\tau):= X_\mu^{\bbb,x_0}(\tau)\cap\mathring{\Gr}_{\nu}=X_{\mu,\nu}^\bba(\tau)\cap \mathring{\Gr}_{\nu}.$$ 
All $X_{\mu,\nu}^\bba(\tau), X_\mu^{\bbb,x_0}(\tau), \mathring{X}_\mu^{\bbb,x_0}(\tau)$ admit an action by $J_{\tau,x_0}(\mO)$.
\end{thm}

\begin{rmk}
(1) If $\tau\in\xcoch(Z_G)$ and $\nu$ is minuscule, then $\Gr_{\nu}=\mathring{\Gr}_{\nu}=(\bar G/\bar P_\mu)^\pf$. It follows that in this case
$$X_{\mu,\nu}^\bba(\tau)=X_\mu^{\bbb,x_0}(\tau)$$ 
is (the perfection) of a (closed) finite Deligne-Lusztig variety in $\bar G/\bar P_\mu$, and with the action of $J_{\tau,x_0}(\mO)$ factoring through the usual action of $J_{\tau,x_0}(\mO/\varpi)$. 

(2) However, if $\nu$ is not minuscule, the action $J_{\tau,x_0}(\mO)$ on $\mathring{X}_\mu^{\bbb,x_0}(\tau)$ does not factor through $J_{\tau,x_0}(\mO/\varpi)$. We refer to \cite{XZ} for an example.

(3) We do not know whether $X_{\mu,\nu}^\bba(\tau)$ is irreducible in general. 
\end{rmk}

The following proposition is the key step towards Theorem \ref{C: unique comp}.

\begin{prop}\label{P: irrkey}
The scheme $Y_\nu^\bbb\cap\Gr_\nu= X_{\mu,\nu}^\bba(\tau)\cap S_\nu$ is geometrically irreducible of dimension $\langle\rho,\mu-\tau\rangle$.
\end{prop}
\begin{proof}
The Cartesian diagram \eqref{refined} gives
\[
\xymatrix@C=20pt{
Y^\bbb_\nu \cap \Gr_\nu  \ar[rr]  \ar@{_{(}->}[d] && (S_\nu \cap \Gr_\nu) \tilde \times (S_\lambda \cap \Gr_\mu)^\bbb\ar[d]^{\pr_1 \times m} \\
S_\nu \cap \Gr_\nu \ar[rr]^-{1\times \varpi^\tau\sigma} && (S_\nu \cap \Gr_\nu)\times (S_{\tau+\sigma(\nu)} \cap \Gr_{\tau+\sigma(\nu)}).
}
\]

Putting $\widetilde Y_\nu^{\bbb, x_0} : = \pi'^{-1}_\nu(Y_\nu^\bbb \cap \Gr_\nu)$, we claim that it fits into the following Cartesian diagram analogous to \eqref{E:tilde Ynub Cartesian diagram}
\begin{equation}
\label{E:Z}
\xymatrix{
Y_\nu^\bbb \cap \Gr_\nu \ar[d]
& \ar[l] \widetilde Y_\nu^{\bbb, x_0} \ar[d] \ar[rrr] &&& (S_\la\cap\Gr_\mu)^{\bbb} \ar[d]
\\
S_\nu  \cap \Gr_\nu &\ar[l]_-{\pi'_\nu} L^+U \ar[r]^-{f_\tau} & L^+U \ar[rr]^-{\pi_\lambda \circ c_{-\tau-\sigma(\nu)}} && S_\lambda.
}
\end{equation}
Indeed, we recall that $S_{\nu}\cap \Gr_{\nu}=L^+U\varpi^{\nu}L^+U/L^+U$. From this, we can deduce that
\begin{equation}
\label{E:pinu inverse of S cap Gr}
LU \times_{\pi'_\nu, S_\nu} (S_{\nu}\cap \Gr_{\nu}) = \pi'^{-1}_\nu(S_\nu \cap \Gr_\nu) =  L^+U,
\end{equation}
and the same holds if we replace $\nu$ by $\tau+\sigma(\nu)$. It follows immediately that the left square of \eqref{E:Z} is Cartesian.
To see the right square of \eqref{E:Z} is Cartesian, it suffices to show that $c_{-\tau-\sigma(\nu)}^{-1} \pi_\lambda^{-1}((S_\la\cap\Gr_\mu)^{\bbb})$ is contained in $L^+U$.
Indeed, if an element $u \in LU$ satisfies $\pi_\lambda(c_{-\tau-\sigma(\nu)}(u)) \in (S_\la\cap\Gr_\mu)^{\bbb}$, or equivalently
\[
\varpi^{-\nu} u\varpi^{\tau+\sigma(\nu)} \bmod L^+U \in (S_\la\cap\Gr_\mu)^{\bbb},
\]
then we know that
\[
\varpi^\nu \cdot \varpi^{-\nu} u\varpi^{\tau+\sigma(\nu)} \bmod L^+U \in (S_\nu \cap \Gr_\nu) \tilde \times (S_\la\cap\Gr_\mu)^{\bbb} \xrightarrow{m_{\nu, \mu}} S_{\tau+\sigma(\nu)} \cap \Gr_{\tau + \sigma(\nu)},
\]
which forces $u \in L^+U$ by \eqref{E:pinu inverse of S cap Gr} (with $\nu$ replaced by $\tau+\sigma(\nu)$). This concludes the proof of the claim that \eqref{E:Z} is Cartesian.

As in the proof of Proposition~\ref{P: prop of Ynu}, we may choose a congruence subgroup $L^+U^{(r)}$ for some $r$ large enough so that $\varpi^{-\nu}L^+U^{(r)} \varpi^\nu \subseteq L^+U$ and $c_{-\tau-\sigma(\nu)} ( L^+U^{(r)}) \subseteq L^+U$. Then passing to the quotient of \eqref{E:Z} by $L^+U^{(r)}$ gives a Cartesian diagram
\begin{equation}
\label{E:Z quot r}
\xymatrix{
Y_\nu^\bbb \cap \Gr_\nu \ar[d]
& \ar[l] \widetilde Y_\nu^{\bbb, x_0, r}  \ar[d] \ar[rrr] &&& (S_\la\cap\Gr_\mu)^{\bbb} \ar[d]
\\
S_\nu  \cap \Gr_\nu &\ar[l]_-{\pi'_\nu} L^rU \ar[r]^-{f_\tau} & L^rU \ar[rr]^-{\pi_\lambda \circ c_{-\tau-\sigma(\nu)}} & & S_\lambda.
}
\end{equation}
Since this $\pi'_\nu$ is (perfectly) smooth, it is enough to show that $\widetilde Y_\nu^{\bbb, x_0, r}$ is irreducible.
Moreover, since $\pi_\lambda \circ c_{-\tau-\sigma(\nu)}$ is smooth, the fiber product
\[
Z_r: = L^rU \times_{\pi_\lambda \circ c_{-\tau-\sigma(\nu)}, S_\lambda} (S_\lambda \cap \Gr_\mu)^\bbb
\]
is irreducible, and $\widetilde Y_\nu^{\bbb, x_0, r}$ is an \'etale cover of $Z_r$ with Galois group $U_\tau(\mO/\varpi^r)$. Therefore, it is enough to show that the induced representation
\[\rho:\Gal(\bar \eta/\eta)\to U_\tau(\mO/\varpi^r)\]
is surjective, where $\eta$ is the generic point of $Z_r\otimes \bar k$.

We will make use of the following easy group theoretical fact.
\begin{lem}\label{L: surj nil}
Let $f:\Ga_1\to \Ga_2$ be a group homomorphism. Assume that $\Ga_2$ is a nilpotent group. Then $f$ is surjective if and only if the abelianization map $f^{\on{ab}}: \Ga_1^{\on{ab}}\to \Ga_2^{\on{ab}}$ is surjective.
\end{lem}
\begin{proof}
The ``only if" direction is clear. We prove the other direction. As usual, we denote by $[a,b]=aba^{-1}b^{-1}$ and $a^b=bab^{-1}$. 

Let $\ga_1(\Ga_i)=\Ga_i$ and $\ga_j(\Ga_i)=[\ga_{j-1}(\Ga_i),\Ga_i]$ be the lower central series. It is enough to show that $\Ga_1/\ga_j(\Ga_1)\to \Ga_2/\ga_j(\Ga_2)$ is surjective for every $j$ by induction.
The case $j=2$ is our assumption. Suppose that this is true for $j$. Then it is enough to show that 
$$\ga_{j}(\Ga_1)/\ga_{j+1}(\Ga_1)\to \ga_{j}(\Ga_2)/\ga_{j+1}(\Ga_2)$$ is surjective. It is enough to show that elements in  $\gamma_j(\Gamma_2)$ of the form $[a,b]$ with $a \in \gamma_{j-1}(\Gamma_2)$ and $b \in \Gamma_2$ are contained in the image of the above map.

By inductive hypothesis, there exists $\tilde{a}\in \ga_{j-1}(\Ga_1)$ and $\tilde{b}\in\Ga_1$ such that $f(\tilde a)=am, f(\tilde b)=bn$, with $m,n\in\ga_j(\Ga_2)$. Then 
$$f([\tilde a,\tilde b])=[am,bn]=[m,b]^a[a,b][am,n]^b=[a,b] \mod \ga_{j+1}(\Ga_2).$$  
The lemma then follows.
\end{proof}
Therefore, it is enough to show that $\rho:\Gal(\bar \eta/\eta)\to U_\tau(\mO/\varpi^r)^{\on{ab}}$ is surjective. The next lemma gives a description of $ U_\tau(\mO/\varpi^r)^{\on{ab}}$ in most cases.

\begin{lem}\label{L: commutator}
Let $U_H$ be the unipotent radical of a Borel subgroup $B_H$ of an unramified reductive group $H$ over $\mO$, and for $i \geq 1$, let  $\Phi_{H,i}$ denote the set of (absolute) positive roots $\alpha$ of height $i$, i.e. $\langle\alpha,\rho_H^\vee \rangle =i$, where $\rho_H^\vee$ is the half of the sum of positive coroots. Let $U_{H, i}$ denote the group generated by the coroot groups $U_\alpha$ with $\alpha \in \Phi_{H,i} \cup \Phi_{H,i+1} \cup \cdots$. Then we obtain a canonical filtration 
$U_H=U_{H,1}\supset U_{H,2}\supset\cdots$ by normal subgroups. 

Assume that $k\neq \bF_2$ if $H_\ad$ has a factor of type $\mathsf{B}_n,\mathsf{C}_n,\mathsf{F}_4$, and $k\neq \bF_2,\bF_3$ if $H_\ad$ has a factor of type $\mathsf{G}_2$,  then for every $\mO$-algebra $R$,
\[\ga_2(U_H(R)))=U_{H,2}(R).\]
\end{lem}
\begin{proof}
The inclusion $\ga_2(U_H(R))) \subseteq U_{H,2}(R)$ is clear. We prove the other inclusion.
We are free to pass to a central isogeny of $H$, so we may assume that $H$ is adjoint. Then we may assume that $H$ is absolutely simple over $\mO$. 
For a relative coroot $a$, let $U_{H,a}$ denote the corresponding root group.
Recall that each $\sigma$-orbit in $\Phi_H^+: = \cup_{i>0} \Phi_{H,i}$ corresponds to a positive relative root of $H$.  We will say the height of a positive relative root to mean the height of the corresponding absolute roots.
Note that every non-simple relative root $a$ of $H$ is either non-reduced (i.e. $a/2$ is also a root) or is  contained in a sub root system of rank $2$. Then to show $U_{H,a}(R)\subset \ga_2(U_H(R))$, we may assume $H$ itself is of type $\mathsf{BC}_1$, or of relative rank two. Explicitly, it is one of the following cases: $\on{PGL}_3, \on{PSO}_5, \on{G}_2, \on{PU}_n, n=3,4,5$, ${^3}\mathsf{D}_4$.

We also point out that, to show that $U_{H,2}(R) \subseteq \gamma_2(U_H(R))$, it is enough to show that, for each relative root $a$ of height $i$, there exist relative roots $a_1$ and $a_2$ of heights $i_1$ and $i_2$, such that $a=a_1+a_2$ (thus $i=i_1+i_2$) and 
\begin{equation}
\label{E:commutator of a1 and a2} [U_{H,a_1}(R)U_{H,i_1+1}(R)/U_{H,i_1+1}(R), U_{H,a_2}(R)U_{H,i_2+1}(R)/U_{H,i_2+1}(R)] = U_{H,a}(R)U_{H, i+1}(R) / U_{H, i+1}(R).
\end{equation}

The calculations of the commutator $[x,y]$ for $x\in U_{H,a}(R)$ and $y\in U_{H,b}(R)$ in split case are contained in SGA 3 XXIII, Section 3. Using these formulas and the argument as \cite[Lemma 7]{How} give our claim in the split case.

Next we assume that $H=\on{PU}_n$ with $n=3,4,5$ for the unramified quadratic extension $E/F$.
Let $\alpha_1, \dots, \alpha_{n-1}$ denote the absolute simple roots, and then each subset of consecutive numbers $I \subseteq\{1, \dots, n-1\}$ corresponds to an absolute positive root. 
Let $x_{\alpha_i}: \GG_a \to U_{\alpha_i}$ denote the corresponding root morphisms over $\calO_E$, normalized so that $\sigma(x_{\alpha_i}(r)) = x_{\alpha_{n-i}}(\sigma(r))$ for $r \in \calO_E$.
Then we can normalize other root morphisms so that $[x_{\alpha_I}(r), x_{\alpha_J}(r')] = 0$ unless the concatenation of $I$ and $J$ (resp. $J$ and $I$) is exactly another consecutive sequence, in which case  $[x_{\alpha_I}(r), x_{\alpha_J}(r')] =  x_{\alpha_{I \cup J}}(rr')$ (resp. $[x_{\alpha_I}(r), x_{\alpha_J}(r')] =  x_{\alpha_{J \cup I}}(-rr')$).
\begin{enumerate}
\item When $n=3$, for the relative root $a$ corresponding to $\alpha_{1,2} =\alpha_1+\alpha_2$, a straightforward computation shows that\footnote{The relative root group corresponding to $\{\alpha_1, \alpha_2\}$ should be $x_{\alpha_1}(r)x_{\alpha_2}(\sigma(r))x_{\alpha_{1,2}}(-\tfrac 12 r\sigma(r))$ but its image in $U_{\{\alpha_1, \alpha_2\}}(R) / U_2(R)$ is the same as $x_{\alpha_1}(r)x_{\alpha_2}(\sigma(r))$.
Explicitly when $n=3$, if we write $\on{U}_3=\{ A\in \GL_3(\mO_{E})\mid (\overline{A}^{t})^{-1}JA=J\}$ with the anti-diagonal unit matrix $J_3$. The commutator subgroup of $\on{PU}_3$ is given by
\[
\begin{pmatrix}1 & r & s\\  & 1 & -\sigma(r) \\  & & 1 \end{pmatrix}\begin{pmatrix}1 & r' & s'\\  & 1 & -\sigma(r')\\  & & 1 \end{pmatrix} \begin{pmatrix}1 & -r & \sigma(s)\\  & 1 & \sigma(r) \\  & & 1 \end{pmatrix}\begin{pmatrix}1 & -r' & \sigma(s')\\  & 1 & \sigma(r') \\  & & 1 \end{pmatrix}=\begin{pmatrix}1 &  & r'\sigma(r)-r\sigma(r')\\  & 1 &  \\  & & 1 \end{pmatrix},
\] where $r\sigma(r)+s+\sigma(s)=0, r'\sigma(r')+s'+\sigma(s')=0$.}
\begin{align*}
&[x_{\alpha_1}(r)x_{\alpha_2}(\sigma(r)), x_{\alpha_1}(r')x_{\alpha_2}(\sigma( r'))] 
\\
=\;& x_{\alpha_1}(r)x_{\alpha_2}(\sigma(r))x_{\alpha_1}(r')x_{\alpha_2}(\sigma( r'))x_{\alpha_2}(-\sigma( r')) x_{\alpha_1}(-r')x_{\alpha_2}(-\sigma(r)) x_{\alpha_1}(-r)
\\
=\;& x_{\alpha_1}(r)x_{\alpha_1}(r')x_{\alpha_2}(\sigma(r))
x_{\alpha_{1,2}}(-\sigma(r)r')
x_{\alpha_2}(\sigma( r'))x_{\alpha_2}(-\sigma( r')) x_{\alpha_2}(-\sigma(r))x_{\alpha_1}(-r')x_{\alpha_{1,2}}(r'\sigma(r)) x_{\alpha_1}(-r)
\\
=\;& x_{\alpha_{1,2}}( r'\sigma( r) -r\sigma( r') ) .
\end{align*}
It follows that $U_{\alpha_{1,2}}(R) \subseteq \gamma_2(U_H(R))$. (Note that this even holds if $p=2$.) 

\item
When $n=4$, the relative positive roots of height $ \geq 2$ correspond to the $\sigma$-orbits $\{\alpha_{1,2}, \alpha_{2,3}\}$ and $\{\alpha_{1,2,3}\}$.
We use the following commutator relations (whose computation we omit)\footnote{Note that the root group for the relative root $\{\alpha_{1,2},\alpha_{2,3}\}$ is $x_{\alpha_{1,2}}(s) x_{\alpha_{2,3}}(-\sigma(s))$ \emph{with} the negative sign.}
\[
[x_{\alpha_1}(r) x_{\alpha_3}(\sigma( r)), x_{\alpha_2}(r')] = x_{\alpha_{1,2}}(rr') x_{\alpha_{2,3}}(-\sigma(r)r') ,
\]
\[
[x_{\alpha_1}(r)x_{\alpha_3}(\sigma(r)),
x_{\alpha_{1,2}}(r') x_{\alpha_{2,3}}(-\sigma( r'))] = x_{\alpha_{1,2,3}}(-\sigma(r)r'-r\sigma(r')).
\]
This verifies \eqref{E:commutator of a1 and a2} for the relative roots $\{\alpha_{1,2},\alpha_{2,3}\}$ and $\alpha_{1,2,3}$.

\item
When $n=5$, the relative positive roots of height $\geq 2$ correspond to the $\sigma$-orbits $\{\alpha_{1,2}, \alpha_{3,4}\}$, $\{\alpha_{2,3}\}$, $\{\alpha_{1,2,3}, \alpha_{2,3,4}\}$, and $\{\alpha_{1,2,3,4}\}$.
The relative root corresponding to $\{\alpha_{2,3}\}$ (resp. $\{\alpha_{1,2,3,4}\}$) is contained in an embedded $\on{PU}_2$ with absolute simple roots $\alpha_2$ and $\alpha_3$ (resp. $\alpha_{1,2}$ and $\alpha_{3,4}$), which is covered in (1). The following verifies \eqref{E:commutator of a1 and a2} for the other two relative positive roots.
\[
[x_{\alpha_1}(r) x_{\alpha_4}(\sigma( r)), x_{\alpha_2}(r')x_{\alpha_3}(\sigma(r'))] = x_{\alpha_{1,2}}(rr') x_{\alpha_{3,4}}(-\sigma(rr')) ,
\]
\[
[x_{\alpha_1}(r)x_{\alpha_4}(\sigma(r)),
x_{\alpha_{2,3}}(r')] = x_{\alpha_{1,2,3}}(rr')x_{\alpha_{2,3,4}}(-\sigma(r)r').
\]
\end{enumerate}

Finally, we assume that $H={^3}\mathsf{D}_4$ with respect to the cubic unramified extension $E/F$. Let $\alpha_1, \alpha_2, \alpha_3$, and $\alpha_4$ denote the four absolute simple roots; so that $\sigma$ permutes $\alpha_2, \alpha_3, \alpha_4$. 
The relative positive roots of height  $\geq 2$ correspond to the $\sigma$-orbits $\{\alpha_{1,2}, \alpha_{1,3},\alpha_{1,4}\}$, $\{\alpha_{1,2,3}, \alpha_{1,2,4}, \alpha_{1,3,4}\}$, $\{\alpha_{1,2, 3,4}\}$, and $\{\alpha^*_{1,2,3,4} = \alpha_{1,2,3,4} + \alpha_1\}$.
The relative positive root corresponding to $\{\alpha^*_{1,2,3,4}\}$ (resp. $\{\alpha_{1,2,3}, \alpha_{1,2,4}, \alpha_{1,3,4}\}$) is contained in an embedded $\PGL_2$ (resp. $\Res_{E/F}\PGL_2$) with simple roots $\alpha_1$ and $\alpha_{1,2,3,4}$ (resp. $\{\alpha_1, \alpha_2, \alpha_3\}$ and $\{\alpha_{1,2}, \alpha_{1,3}, \alpha_{1,4}\}$), which is covered in the split group case.
To deal with the other two relative roots, we normalize root morphisms so that 
$
\sigma(x_{\alpha_i}(r)) = x_{\sigma(\alpha_i)}(\sigma(r))$ and $x_{\alpha_{1,i}}(1) = [x_{\alpha_1}(1), x_{\alpha_i}(1)]$ for $i=2,3,4$;  $x_{\alpha_{1,i,j}} = [x_{\alpha_{1,i}}(1), x_{\alpha_j}(1)]$ for pairs $(i,j) = (2,3), (3,4), (4,2)$; and $x_{\alpha_{1,2,3,4}}(1) = [x_{\alpha_{1,2,3}}(1), x_{\alpha_4}(1)] = [x_{\alpha_{1,3,4}}(1), x_{\alpha_2}(1)] =  [x_{\alpha_{1,4,2}}(1), x_{\alpha_3}(1)]$. The following verifies \eqref{E:commutator of a1 and a2} for the other relative positive roots.
\[
[x_{\alpha_1}(r) , x_{\alpha_2}(r')x_{\alpha_3}(\sigma(r'))x_{\alpha_4}(\sigma^2(r'))] = x_{\alpha_{1,2}}(rr') x_{\alpha_{1,3}}(r\sigma(r'))x_{\alpha_{1,4}}(r\sigma^2(r')),
\]
\[
[x_{\alpha_2}(r')x_{\alpha_3}(\sigma(r'))x_{\alpha_4}(\sigma^2(r')),
x_{\alpha_{1,2,3}}(r')x_{\alpha_{1,3,4}}(\sigma(r'))x_{\alpha_{1,4,2}}(\sigma^2(r'))] = x_{\alpha_{1,2,3,4}}\big((1+\sigma+\sigma^2)(-r\sigma(r')) \big).
\]
This concludes the proof of the lemma.
\end{proof}

 We have a natural projection
\[\pr: L^rU\to L^r(U/U_2)\to L^r(U_\tau/U_{\tau,2}).\]
Let
$$\bL:  L^r(U_\tau/U_{\tau,2})\to  L^r(U_\tau/U_{\tau,2}),\quad \bL(u)=u^{-1}\sigma(u)$$
be (the perfection of) the usual Lang map for the \emph{commutative} algebraic group $L^r(U/U_2)$. Then we have the following commutative, in which the right square is Cartesian
\[\xymatrix{
L^rU\ar_{f_\tau}[d]\ar[r]& L^r(U/U_2)\ar[r]\ar[d]_{f_{\tau,2}}&  L^r(U_\tau/U_{\tau,2})\ar^\bL[d]\\
L^rU \ar[r] & L^r(U/U_2)\ar[r] & L^r(U_\tau/U_{\tau,2}), \ar@{}[lu]|\Box
}\]
where $f_{\tau,2}: L^r(U/U_2) \to L^r(U/U_2)$ is the morphism induced by $f_\tau$.

Let $Z_{r,2}$ be the closure of the image of $Z_r$ under the projection $L^rU\to L^r(U/U_{2})$, whose generic point we denote by $\eta_2$. It follows from this commutative diagram and
 Lemma \ref{L: commutator} that we have a commutative diagram
\[
\xymatrix{
\Gal(\bar \eta / \eta) 
\ar[d] \ar[r]^-\rho & U_\tau(\calO/ \varpi^r) \ar[d]
\\
\Gal(\bar \eta_2 / \eta_2) \ar[r]^-{\rho_2} & (U_\tau/ U_{\tau, 2})(\calO/\varpi^r) \cong U_\tau (\calO/\varpi^r)^\mathrm{ab},\!\!\!\!\!\!\!\!\!\!\!\!\!\!\!\!\!\!\!\!\!\!\!\!\!\!\!\!\!\!\!\!\!\!\!\!\!\!\!\!
}
\]
where the bottom arrow is induced by the Lang map $\LL$. By Lemma \ref{L: surj nil}, to prove the surjectivity of $\rho$, it is enough to show the surjectivity of $\Gal(\bar\eta/\eta)\to\Gal(\bar\eta_2/\eta_2)$ and the surjectivity of $\rho_2$.
The former follows from the fact that the map $Z_r \to Z_{r,2}$ generically has geometrically connected fiber (Proposition~\ref{P: geom prop of MV cycle}). For the latter,  we need to show that $Z_{r,2}$ is big enough. We make use of the following lemma.

\begin{lem}
We have $Z_{r,2}= \prod L^+U_{\al_i}^{(\langle\nu,\al_i\rangle-\varepsilon_{\al_i}(\bbb))}/ L^+U_{\al_i}^{(r)}\subset L^r(U/U_{2})$.
\end{lem}
Note that by our assumption of $\nu$ and  \eqref{E:condition to realize as a homomorphism crystal}, $\langle\nu,\al_i\rangle-\varepsilon_{\al_i}(\bbb)\geq 0$. This is consistent with the fact that $Z_r\subset L^rU$.
\begin{proof}
By Proposition \ref{P: geom prop of MV cycle}, the closure of the projection of $\pi^{-1}_\la((S_\la\cap \Gr_\mu)^\bbb)\subset LU$ to $L(U/U_2)$ is $\prod_i \pi_{M_i,\la}^{-1}(S_{M_i,\la}\cap \Gr_{M_i,\la+\varepsilon_i(\bbb)})$, where $\pi_{M_i,\la}: LU_{M_i}=LU_{\al_i}\to S_{M,\la_i}$ is the projection \eqref{E:triviallization of U(O) torsor} for $M_i$.
By \eqref{E: MV sl2} and $\phi_i(\bbb)-\varepsilon_i(\bbb)=\langle\la,\al_i\rangle$, we have
$$\pi_{M_i,\la}^{-1}(S_{\la}\cap \Gr_{M_i,\la+\varepsilon_i(\bbb)})=L^+U^{(-\phi_i(\bbb))}_{\al_i}.$$
Now use $Z_r=\varpi^{\sigma(\nu)+\tau} \pi^{-1}_\la((S_\la\cap \Gr_\mu)^\bbb) \varpi^{-\sigma(\nu)-\tau} / L^+U^{(r)}$ and $\sigma(\nu)+\tau=\nu+\la$ to conclude.
\end{proof}

Note that so far we have not use the fact that $\nu=\nu_\bbb$. Now, assuming this, we see that in each $\sigma$-orbit $\{\al,\sigma(\al),\ldots,\sigma^d(\al)=\al\}$ of $\Delta$, there is some $\sigma^i(\al)$ such that $\langle\nu_\bbb,\sigma^i(\al)\rangle=\varepsilon_{\sigma^i(\al)}(\bbb)$, by Lemma \ref{L: finding best tau} (2). In particular, $L^rU_{\sigma^i(\al)}\subset Z_{r,2}$. 

We need a last lemma.

\begin{lem}Let $H_0$ be a connected affine algebraic group over $\bF_{q^r}$ and let $H=\Res_{\bF_{q^r}/\bF_q}H_0$ be its Weil restriction to $\bF_q$. Let $H_0\to H\otimes \bF_{q^r}$ be the canonical embedding.
Let $\bL_{H}: H\to H,\ h\mapsto h^{-1}\sigma(h)$ denote the Lang map of $H$, base changed to $\bF_{q^r}$, where $\sigma$ is the $q$-Frobenius of $H$. 
Then $\bL_H: \bL_H^{-1}(H_0)\to H_0$ can be identified with the Lang map for $H_0$. In particular, $\bL_H^{-1}(H_0)$ is irreducible so the composed map
\[\pi^{\et}_1(H_0)\to \pi^{\et}_1(H)\to H_0(\bF_{q^r}).\]
is surjective.
\end{lem}
\begin{proof}This follows by definition.
\end{proof}

Now we finish the proof.
Applying this lemma to  each $\sigma$-orbit $\{\al,\sigma(\al),\ldots,\sigma^d(\al)=\al\}$ of $\Delta_\tau$ (using the fact that $Z_{r,2}$ contains $L^rU_{\sigma^i(\alpha)}$ for some $i$), we see that the image of $\rho_2$ contains the relative root group corresponding this $\sigma$-orbit. Since this is true for all $\sigma$-orbit, we conclude that $\rho_2$ is surjective. This concludes the proof of Proposition~\ref{P: irrkey}.
\end{proof}

\subsubsection{Proof of Theorem \ref{C: unique comp}}
Let $\overline{Y_\nu^\bbb\cap \Gr_\nu}$ denote the closure of $Y_\nu^\bbb\cap \Gr_\nu$. By Proposition~\ref{P: irrkey}, it is an irreducible component of $X_\mu^\bbb(\tau)$.
Recall that that $Y_\nu^\bbb\cap\Gr_\nu= X_{\mu,\nu}^\bba(\tau)\cap S_\nu$. Therefore, for every $g\in J_\tau(\mO)$,
$$(Y_\nu^\bbb\cap\Gr_\nu)\cap g(Y_\nu^\bbb\cap\Gr_\nu)= X_{\mu,\nu}^\bba(\tau)\cap S_\nu\cap gS_\nu.$$
Recall that $S_\nu\cap\Gr_\nu$ is open in $\Gr_\nu$. It follows that $(Y_\nu^\bbb\cap\Gr_\nu)\cap g(Y_\nu^\bbb\cap\Gr_\nu)$ is open in $Y_\nu^\bbb\cap\Gr_\nu$. 
In particular, if this intersection is non-empty, then $g\overline{Y_\nu^\bbb\cap \Gr_\nu}=\overline{Y_\nu^\bbb\cap \Gr_\nu}$ by the irreducibility of $Y_\nu^\bbb\cap\Gr_\nu$ established in Proposition~\ref{P: irrkey}.

Now, by Lemma \ref{L: basic prop of MV}, $\varpi^\nu\in Y^\bbb_\nu\cap\Gr_\nu$. Let 
$$\on{St}_{J_{\tau,x_0}(\mO)}(\varpi^\nu)=\{g\in J_{\tau,x_0}(\mO)\mid g\varpi^\nu=\varpi^\nu \mod G(\mO_L)\}.$$
It follows from the above claim that $\overline{Y_\nu^\bbb\cap \Gr_\nu}$ is invariant under the action $\on{St}_{J_{\tau,x_0}(\mO)}(\varpi^\nu)$. On the other hand, $\overline{Y_\nu^\bbb\cap \Gr_\nu}$ is also invariant under the action of $U_\tau(\mO)$ (since so is $Y^\bbb_\nu\cap \Gr_\nu$). Since $\on{St}_{J_{\tau,x_0}(\mO)}(\varpi^\nu)$ and $U_\tau(\mO)$ generate $J_\tau(\mO)$ as a group, we see that $\overline{Y_\nu^\bbb\cap \Gr_\nu}$ is invariant under the action by $J_{\tau,x_0}(\mO)$.

Next, we show that $X_{\mu,\nu}^\bba(\tau)$ contains a unique irreducible component of dimension $\langle\rho,\mu-\tau\rangle$. Since $\Gr_\nu=\sqcup_{\nu'\preceq \nu}(S_{\nu'}\cap \Gr_\nu)$, it is enough to show that $X_{\mu,\nu}^\bba(\tau)\cap S_{\nu'}$ has dimension strictly smaller that $\langle \mu-\tau\rangle$ if $\nu'\prec \nu$. But by definition, $X_{\mu,\nu}^\bba(\tau)\cap S_{\nu'}$ fits into the following Cartesian diagram
\[\begin{CD}
X_{\mu,\nu}^\bba(\tau)\cap S_{\nu'}@>>> \Gr_{(\nu,\mu)\mid \tau+\sigma(\nu)}^{0,\bba}\cap (S_{\nu'}\cap \Gr_{\nu})\times (S_{\tau+\sigma(\nu')}\cap \Gr_{\tau+\sigma(\nu)})\\
@VVV@VVV\\
S_{\nu'}\cap \Gr_{\nu}@>\id\times \varpi^\tau\sigma>>(S_{\nu'}\cap \Gr_{\nu})\times (S_{\tau+\sigma(\nu')}\cap \Gr_{\tau+\sigma(\nu)}).
\end{CD}\]
We need the following lemma.
\begin{lem}
Let $\nu'\prec \nu$. Then the intersection $\Gr_{(\nu,\mu)\mid \tau+\sigma(\nu)}^{0,\bba}\cap (S_{\nu'}\cap \Gr_{\nu})\times (S_{\tau+\sigma(\nu')}\cap \Gr_{\tau+\sigma(\nu)})$ in $\Gr_\nu\times \Gr_{\tau+\sigma(\nu)}$ has dimension $<\langle\rho, \nu+\tau+\sigma(\nu')+\mu\rangle$.
\end{lem}
\begin{proof}
Let $(S_{\tau+\sigma(\nu')}\cap \Gr_{\tau+\sigma(\nu)})^\bbc$ be an irreducible component of $S_{\tau+\sigma(\nu')}\cap \Gr_{\tau+\sigma(\nu)}$.
Recall from the proof of Lemma \ref{L: decomp MV into Satake} that the preimage of $(S_{\tau+\sigma(\nu')}\cap \Gr_{\tau+\sigma(\nu)})^\bbc\cap \mathring{\Gr}_{\tau+\sigma(\nu)}$ in $\Gr_{(\nu,\mu)\mid \tau+\sigma(nu)}^{0,\bba}$ is an open subset $(S_{\eta_1}\times \Gr_\nu)^{\bbc_1}\tilde\times (S_{\eta_2}\times \Gr_\la)^{\bbc_2}$, where $\eta_1+\eta_2=\tau+\sigma(\nu')$, and $(\bba,\bbc)\mapsto (\bbc_1,\bbc_2)$ is the map defined in Lemma \ref{L: decomp MV into Satake}, which is compatible with Littelmann's decomposition theorem (see Proposition \ref{comb decom})

We identify the set of MV cycles with the set of $\hat G$-crystals via Proposition \ref{bijectionI}. Note that $\bbc\in \MV_{\tau+\sigma(\nu)}(\tau+\sigma(\nu'))$ is obtained by the unique element in $\MV_{\tau+\sigma(\nu)}(\tau+\sigma(\nu))$ by a series of root operators $f_\al$s.
Since one of the equalities in Lemma \ref{L: finding best tau}(2) holds, by the definition of the root operators acting on the set of $\hat G$-crystals, we see that 
$$\nu'\prec \eta_1,\quad \eta_2\prec \la.$$ This implies that the intersection in the lemma has dimension $<\langle\rho, \nu+\eta_1+\mu+\eta_2 \rangle=\langle\rho, \nu+\tau+\sigma(\nu')+\mu\rangle$. 
\end{proof}

Now, by the argument similar to Proposition \ref{P: prop of Ynu} (in particular, a diagram similar to \eqref{E:tilde Ynub Cartesian diagram truncated}), we see that
\begin{align*}\dim X_{\mu,\nu}^\bba(\tau)\cap S_{\nu'}& = \dim \Gr_{(\nu,\mu)\mid \tau+\sigma(\nu)}^{0,\bba}\cap (S_{\nu'}\cap \Gr_{\nu})\times (S_{\tau+\sigma(\nu')}\cap \Gr_{\tau+\sigma(\nu)})- \dim (S_{\tau+\sigma(\nu')}\cap \Gr_{\tau+\sigma(\nu)})\\
&< \langle\rho, \nu+\tau+\sigma(\nu')+\mu\rangle - \langle\rho, \tau+\sigma(\nu')+\tau+\sigma(\nu)\rangle\\
&=\langle\rho,\mu-\tau\rangle.
\end{align*}
We have finished the proof that $X_{\mu}^{\bbb,x_0}(\tau):=\overline{Y_\nu^\bbb\cap \Gr_\nu}$ is the unique irreducible component of $X_{\mu,\nu}^\bba(\tau)$ of dimension $\langle\rho,\mu-\tau\rangle$.

Finally, we show that $X_{\mu,\nu}^\bba(\tau)\cap\mathring{\Gr}_\nu$ is irreducible. 
Let 
\begin{equation}
\label{E: open Satake cycle}
\mathring{\Gr}^{0,\bba}_{(\nu,\mu)\mid \tau+\sigma(\nu)}:=\Gr^{0,\bba}_{(\nu,\mu)\mid \tau+\sigma(\nu)}\cap (\mathring{\Gr}_\nu\times \mathring{\Gr}_{\tau+\sigma(\nu)}).
\end{equation}
Let $F$ be the fiber over $\varpi^\nu$ of the projection  $\mathring{\Gr}^{0,\bba}_{(\nu,\mu)\mid \tau+\sigma(\nu)}\to \mathring{\Gr}_{\nu}$. This is identified with an irreducible locally closed subset of $\Gr_\mu\cap \varpi^{-\nu}\mathring{\Gr}_{\tau+\sigma(\nu)}\subset \Gr$. Let $F^{(\infty)}$ denote its preimage in $LG\to \Gr$.

Recall that $\mathring{\Gr}_{\nu}$ is a single $L^+G$-orbit given by $g\mapsto g\varpi^\nu$. Then the preimage of $X_{\mu,\nu}^\bba(\tau)\cap\mathring{\Gr}_\nu$ in $L^+G$ is 
\[\{g\in L^+G\mid g^{-1}\varpi^\tau\sigma(g)\in \varpi^\nu F^{(\infty)}\varpi^{-\sigma(\nu)}\subset LG.\]
Note that $J_{\tau,x_0}(\mO)$ acts transitively along the fibers of the map $L^+G\to LG, g\mapsto g^{-1}\varpi^\tau\sigma(g)$. Since  $\varpi^\nu F^{(\infty)}\varpi^{-\sigma(\nu)}$ is irreducible, it follows that $J_{\tau,x_0}(\mO)$ acts transitively on the set of irreducible components of $X_{\mu,\nu}^{\bba}(\mO)\cap\mathring{\Gr}_{\nu}$. Combining with the fact that $X_{\mu}^{\bbb,x_0}(\tau)\cap \mathring{\Gr}_{\nu}\supset Y^\bbb_\nu\cap \Gr_\nu$ is irreducible in $X_{\mu,\nu}^\bba(\tau)\cap \mathring{\Gr}_{\nu}$ of dimension $\langle\rho,\mu-\tau\rangle$, we see that
\[X_{\mu}^{\bbb,x_0}(\tau)\cap \mathring{\Gr}_{\nu}=X_{\mu,\nu}^\bba(\tau)\cap \mathring{\Gr}_{\nu},\]
which is irreducible.
The proof of Theorem \ref{C: unique comp} now is complete.

\begin{thm}
\label{C: irr comp up to J}
Let $b=\varpi^\tau$ be an unramified element, and assume that $\tau_\sigma$ is dominant.
Consider the action of $J_\tau(F)$ on the set of irreducible components of $X_\mu(\tau)$.
\begin{enumerate}
\item 
The stabilizer of each irreducible component is a hyperspecial subgroup.
\item
The subset $X^\bbb_\mu(\tau)\subset X_\mu(\tau)$ is invariant under the action of $J_\tau(F)$. The group $J_\tau(F)$ acts transitively on the set of irreducible components of $X^\bbb_\mu(\tau)$.
\item 
Assume that $Z_G$ is connected.
There is a canonical $J_\tau(F)$-equivariant bijection between the set of irreducible components of $X^{\bbb}_\mu(\tau)$ and 
$$\sqcup_{\la\in \tau+(1-\sigma)\xcoch}\MV_\mu(\la)\times \bH\bS_b,\quad (\bbb,x)\mapsto X_{\mu}^{\bbb,x}(\tau)$$
where $\bH\bS_b$ denote the set of hyperspecial subgroups of $J_\tau(F)$, and $X_{\mu}^{\bbb,x}(\tau)$ is the irreducible component in $X_\mu^{\bbb}(\tau)$ whose stabilizer is the hyperspecial subgroup of $J_\tau(F)$ corresponding to $x$.
\end{enumerate}
\end{thm}
\begin{proof}
To proof (1), first assume that $Z_G$ is connected. Then by Proposition \ref{P: irrkey}, in each $X_\mu^\bbb(\tau)$ there is an irreducible component whose stabilizer in $J_\tau(F)$ is hyperspecial. Then the statement follows from Corollary \ref{C: irr comp up to B}(3).
In particular, if $G$ is of adjoint type, the statement holds. The general case follows from Lemma \ref{L: red to adj}.

Recall that $X^\bbb_\mu(\tau)$ is $B_\tau(F)$-invariant by Corollary \ref{C: irr comp up to B}. On the other hand, the stabilizer of each irreducible component is hyperspecial in $J_\tau(F)$. Therefore (2) follows from the Iwasawa decomposition if $J_\tau(F)$ and Corollary \ref{C: irr comp up to B}.

If $Z_G$ is connected, then all hyperspecial subgroups of $J_\tau$ are conjugate. For every $\bbb\in \MV_\mu(\la)$, we constructed $X_{\mu}^{\bbb,x_0}(\tau)$, and other irreducible components are given by $g\mapsto gX_\mu^{\bbb,x_0}(\tau)$ for $g\in J_\tau(F)$, with stabilizer group $gJ_{\tau,x_0}(\mO)g^{-1}$. Part (3) follows.
\end{proof}

\begin{rmk}
In general if $Z_G$ is not connected, it seems that $J_\tau$ does not have a natural conjugacy class of hyperspecial subgroup. But Theorem \ref{C: irr comp up to J} implies that each $\bbb\in \sqcup_{\la\in \tau+(1-\sigma)\xcoch}\MV_\mu(\la)$ gives such a conjugacy class, which contains the stabilizers of irreducible components in $X_\mu^\bbb(\tau)$. Note that different $\bbb$'s may give different conjugacy classes.
\end{rmk}

\section{The moduli of local shtukas}

The moduli spaces of shtukas were invented by Drinfeld as the function field analogue of the Shimura varieties. However, the moduli space of shtukas can be defined in a much greater generality. In particular, there are notions of moduli spaces of iterated shtukas and partial Frobenius between them, which play a crucial role in realizing the Langlands correspondence over the function fields (see the works of Drinfeld \cite{drinfeld}, L. Lafforgue \cite{L.La}, and more recently V. Lafforgue \cite{La}). Unfortunately, so far it is not clear how to construct their Shimura variety counterparts. In this section, we pass to the local situation to give a uniform construction of the moduli spaces of (iterated) local shtukas, for groups over equal or mixed characteristic. In addition, we construct various correspondences between them, including the partial Frobenius operator. Finally, we define certain category $\on{P}^{\on{Corr}}(\Sht^\loc)$ of perverse sheaves, which will be the main player of the next section.

Let $\mathbf{Aff}^\pf_{k}$ denote the category of perfect rings over $k$, equipped with the fpqc topology. We continue to use the notations from Section~\ref{S: geom Sat}, in particular \eqref{ramified Witt vector} and \eqref{discs}.
In this section, unless otherwise specified, schemes/algebraic spaces/algebraic stacks mean schemes/algebraic spaces/algebraic stacks in $\mathbf{Aff}_k^\pf$. The foundations on perfect algebraic geometry we needed in this paper are recalled in the Appendix \ref{ASS:perfect AG}, and the results on cohomological correspondences are recalled in \ref{Sec:cohomological correspondence}.

\subsection{The local Hecke stack}

\begin{definition}
\label{D:local Hecke stack setup}
Let $\mmu=(\mu_1,\ldots,\mu_t)$ be a sequence of dominant coweights (of $G$). We define the \emph{local Hecke stack} as the prestack 
$$\Hk^\loc_\mmu = \Hk^\loc_{\mu_1, \dots, \mu_t}:= [L^+G\backslash \Gr_\mmu]$$ 
which assigns for every perfect $k$-algebra $R$, the groupoid of chains of modifications of $G$-torsors over $D_R =\Spec W_\calO(R)$
\begin{equation}\label{pt of hk}
\mE_t\stackrel{\beta_t}\dashrightarrow \mE_{t-1}\stackrel{\beta_{t-1}}\dashrightarrow\cdots\stackrel{\beta_2}\dashrightarrow \mE_1\stackrel{\beta_1}\dashrightarrow \mE_0
\end{equation}
of relative positions $\preceq \mu_t,\ \ldots,$ and $\preceq \mu_1$, respectively. (Comparing to the moduli problem of $\Gr_{\mmu}$ in \eqref{E:sequence of isogenies}, we drop the trivialization of $\calE_0$.)

Similarly, we define 
$$\Hk_{\la_\bullet\mid\mmu}^{0,\loc}:=[L^+G\backslash \Gr_{\la_\bullet\mid\mmu}^0],\  \ \textrm{and}\ \Hk_{\nu_\bullet; \la_\bullet\mid \mmu; \xi_\bullet}^{0,\loc}:=[L^+G\backslash \Gr_{\nu_\bullet;\la_\bullet\mid\mmu;\xi_\bullet}^0]$$ 
for sequences of dominant coweights $\lambda_\bullet, \mmu, \nu_\bullet, \xi_\bullet$. 
In particular, taking the quotient of \eqref{E: Sat corr} by the left $L^+G$-action gives the \emph{Satake correspondence} on local Hecke stacks
\[\Hk^\loc_{\la_\bullet}\xleftarrow{h_{\la_\bullet}^{\leftarrow}} \Hk_{\la_\bullet\mid\mmu}^{0,\loc}\xrightarrow{h^\rightarrow_\mmu} \Hk_{\mmu}^{\loc}.\]
\end{definition}

The above definitions are very convenient for the later geometric constructions, but one cannot directly apply the $\ell$-adic formalism to it due to the quotient by an infinite group. For this  reason, we need a finite dimensional quotient replacement of it.

\begin{definition}\label{d:loc Hk}
Let $\mmu=(\mu_1,\ldots,\mu_t)$ be a sequence of dominant coweights and let $m$ be a $\mmu$-large integer (see Definition \ref{D:acceptable pair}). We define the \emph{$m$-restricted local Hecke stack} to be the stack
\[\Hk_\mmu^{\loc(m)} = \Hk_{\mu_1, \dots, \mu_t}^{\loc(m)}:= [L^mG\backslash \Gr_\mmu].\]
We will sometimes write $\Hk_\mmu^{\loc(\infty)}$ for $ \Hk_\mmu^\loc$.
Similarly, if $m$ is both $(\nu_\bullet, \lambda_\bullet, \xi_\bullet)$-large and $(\nu_\bullet, \mu_\bullet, \xi_\bullet)$-large, we can define
$$\Hk_{\la_\bullet\mid\mmu}^{0,\loc(m)}:=[L^mG\backslash \Gr_{\la_\bullet\mid\mmu}^0],\quad 
 \Hk_{\nu_\bullet; \la_\bullet\mid \mmu; \xi_\bullet}^{0,\loc(m)}:=[L^mG\backslash \Gr_{\nu_\bullet;\la_\bullet\mid\mmu;\xi_\bullet}^0]
$$
In particular, the Satake correspondence \eqref{E: Sat corr} descends to
\[\Hk^{\loc(m)}_{\la_\bullet}\xleftarrow{h_{\la_\bullet}^{\leftarrow}} \Hk_{\la_\bullet\mid\mmu}^{0,\loc(m)}\xrightarrow{h^\rightarrow_\mmu} \Hk_{\mmu}^{\loc(m)}.\] 
\end{definition}

\subsubsection{Convolution map} The convolution map $\Gr_{\mu_\bullet}\to\Gr_{|\mmu|}$ induces
\begin{equation}
\label{E: conv local Hk}
m: \Hk^{\loc(m)}_{\nu_\bullet,\mu_\bullet,\xi_\bullet}\to \Hk^{\loc(m)}_{\nu_\bullet,|\mmu|,\xi_\bullet},
\end{equation}
called the convolution map, which is perfectly proper. Here $m$ (which is allowed to be $\infty$) is $(\nu_\bullet,\mu_\bullet,\xi_\bullet)$-large.

\begin{notation}
\label{N:ext and res}
We write $\bfB L^+G$ (resp. $\bfB L^mG$ for $m \in \ZZ_{\geq0}$) for the moduli stack that classifies for every perfect $k$-algebra $R$ the groupoid of $G$-torsors over $D_R$ (resp. $D_{m,R}$).
As a stack, it is the same as $[L^+G \backslash \mathrm{pt}]$ (resp. $[L^mG \backslash \mathrm{pt}]$). 
For $m\leq m'$, we have quotient maps $L^+G\to L^{m'}G\to L^mG$ which induce
\[
\bfB L^+G\xrightarrow{\res_{m'} =\res^\infty_{m'}} \bfB L^{m'}G \xrightarrow{\res^{m'}_m} \bfB L^m G.
\]
We call them the \emph{restriction maps} (restricting the $G$-torsor from $D$ to $D_{m'}$ and from $D_{m'}$ to $D_m$). 
Clearly, for $m\leq m'\leq m''$,
\begin{equation}
\label{E:ex and q compatibility}
\res_{m}^{m'} \circ \res_{m'}^{m''} = \res_{m}^{m''},
\end{equation}
where we allow $m''$ to be $\infty$.
Since $\res_{m}^{m'}$ is an $L^{m'-m}G^{(m)}$-gerbe (for $m'<\infty$), it is perfectly smooth.
\end{notation}

\subsubsection{Two torsors over restricted local Hecke stacks}
\label{SS:torsors over Hecke stack}
Let $\mmu = (\mu_1, \dots, \mu_t)$ be a sequence of dominant coweights. There are natural morphisms 
$$t_\leftone, t_\rightone: \Hk^\loc_\mmu \to \bfB L^+G$$ 
sending \eqref{pt of hk} to the \emph{last} torsor $\calE_t$ and the \emph{first} torsor $\calE_0$ in the sequence, respectively.
In particular, $t_\rightone$ comes from taking the quotient of the structure sheaf $\Gr_\mmu \to \mathrm{pt}$ by the left $L^+G$-action.

Similarly, if $(m,n)$ is a pair of non-negative integers and $m-n$ is $\mmu$-large, by writing
\[
[L^{m}G\backslash\Gr_\mmu]\cong [L^{m}G\backslash\Gr^{(n)}_\mmu/L^nG]\xrightarrow{t_\leftone\times t_\rightone} \bfB L^{n}G\times \bfB L^mG,
\]
where $\Gr_\mmu^{(n)}$ is the $L^nG$-torsor over $\Gr_\mmu$ as defined after \eqref{E:sequence of isogenies}, 
we see that over $\Hk_{\mmu}^{\loc(m)}$, the first map $t_\rightone$ defines the $L^{m}G$-torsor 
\begin{equation}\label{torsor2}
\Gr_\mmu \to [L^{m}G\backslash\Gr_\mmu],
\end{equation}
and the second map $t_\leftone$ defines the $L^{n}G$-torsor
\begin{equation}\label{torsor1}
[L^{m}G\backslash\Gr_{\mmu}^{(n)}]\to  [L^{m+n}G\backslash\Gr_\mmu].
\end{equation}
We write $\calE_{\rightone}|_{D_{m}}$ and $\calE_{\leftone}|_{D_n}$ for these two \emph{canonical} torsors. Note that for $m\leq m'$ and $n\leq n'$ ($(m',n')=(\infty,\infty)$ allowed) such that both $m'-n'$ and $m-n$ are $\mmu$-large, there
is a restriction map
$$\res^{m'}_m: \Hk^{\loc(m')}_{\mmu}\to \Hk^{\loc(m)}_{\mmu}$$ such that
the following diagram is commutative and the right square is Cartesian.
\begin{equation}
\label{E:restriction and torsor commutative}
\xymatrix{
\bfB L^{n'}G \ar[d]^{\res^{n'}_n} &
\Hk_\mmu^{\loc(m')} \ar[l]_-{t_\leftone} \ar[r]^-{t_\rightone} \ar[d]^{\res^{m'}_{m}} &  \bfB L^{m'} G \ar[d]^{\res^{m'}_{m}}
\\ 
\bfB L^n G & \Hk_\mmu^{\loc(m)} \ar[l]_-{t_\leftone} \ar[r]^-{t_\rightone} &  \bfB L^{m} G.
}
\end{equation}
In addition, these restrictions maps  satisfy the natural compatibility condition \eqref{E:ex and q compatibility}.

\begin{rmk}
\label{R: pullbacktwotorsor}
On $\Hk_{\la_\bullet|\mmu}^{0,\loc(m)}$, the pullbacks of the $L^{m}G$-torsor $\mE_\rightone|_{D_{m}}$ and the $L^{n}G$-torsor $\mE_\leftone|_{D_{n}}$ on $\Hk_{\mu_\bullet}^{\loc(m)}$, which we denote by the same notation, are canonically isomorphic to the pullbacks of $\mE_\rightone|_{D_{m}}$ and $\mE_\leftone|_{D_{n}}$ from $\Hk_{\la_\bullet}^{\loc(m)}$, which we denote by $\calE'_\rightone|_{D_{m+n}}$ and $\calE'_\leftone|_{D_n}$, respectively.
\end{rmk}

\subsubsection{Perverse sheaves on restricted local Hecke stack}
\label{SS:Perv(Hk)}
Now we define the category of perverse sheaves on $\Hk^\loc$, by essentially repeating the construction of $\on{P}_{L^+G\otimes \bar k}(\Gr\otimes\bar k)$.
First recall from Notation \ref{N:star pullback}, for a perfectly smooth morphism $f:X \to Y$ between algebraic stacks of relative dimension $d$, we write $f^\star: = f^*[d](\frac d2)$; it takes perverse sheaves to perverse sheaves. Given  $m\leq m'$ two $\mmu$-large integers ($m'\neq \infty$), let $\Res^{m'}_m:=(\res^{m'}_m)^\star$. If $m\geq 1$, since $\res^{m'}_m$ is an $L^{m'-m}G^{(m)}$-gerbe and $L^{m'-m}G^{(m)}$ is the perfection of a unipotent group, the pullback
\[
\Res^{m'}_m: \on{P}(\Hk_{\mmu}^{\loc(m)}) \longto  \on{P}(\Hk_{\mmu}^{\loc(m')})
\]
induces an equivalence of the corresponding categories of perverse sheaves.
An inverse equivalence is given by the push-forward $(\res^{m'}_{m})_!$ up to a shift and a twist. Then similar to \eqref{E: def of Sat cat}, we can define
\begin{equation}
\label{E:definition of P(Hkloc)}
\mathrm{P}(\Hk^\loc_{\bar k}): = \bigoplus_{\boldsymbol \zeta \in \pi_1(G)} \on{P}(\Hk^\loc_{\bbzeta}), \quad \on{P}(\Hk^\loc_{\bbzeta})=\varinjlim_{(\mu,m) \in \bbzeta\times \bZ_{\geq \langle\mu,\al_h\rangle}} \mathrm{P}(\Hk_\mu^{\loc(m)})
\end{equation}
where the connecting functor for the  limit above is given the fully faithful embedding
\[
\xymatrix@C=35pt{
\on{P}(\Hk_{\mu}^{\loc(m)})\ar[r]^-{\Res^{m'}_{m}}_-\cong & \on{P}(\Hk_{\mu}^{\loc(m')})\ar[r]^-{i_{\mu, \mu', *} } & \on{P}(\Hk_{\mu'}^{\loc(m')})}.
\]
Again, note that the limit is filtered. Note that since an $(L^mG\otimes \bar k)$-equivariant perverse sheaf on $\Gr_\mu$ descends to a sheaf on $[L^{m}G \backslash \Gr_\mu]$, after shifted by $m \dim G$ and twisted by $\frac {m}{2} \dim G$, there is a natural equivalence of categories
\begin{equation}
\label{E:P of Gr = P of Hk fin}
\mathrm{P}_{L^mG}(\Gr_\mu)\xrightarrow{\;\cong\;}  
\mathrm{P}(\Hk_\mu^{\loc(m)} ),
\end{equation}
which, after passing to the limit and the direct sum over $\pi_1(G)$, gives rise to the natural equivalence of categories.
\begin{equation}
\label{E:P of Gr = P of Hk}
\mathrm{P}_{L^+G\otimes \bar k}(\Gr\otimes\bar k) \xrightarrow{\;\cong\;} \mathrm{P}(\Hk^\loc_{\bar k}).
\end{equation}

\subsubsection{Breaking up local Hecke stacks}
\label{SS:break hecke}
Let $\la_\bullet$ and $\mmu$ be two sequences of dominant coweights. Assume that  $m_1-m_2$ is $\la_\bullet$-large and $m_2-n$ is $\mmu$-large. By \eqref{E:Hecke factorization induction}, there is a canonical isomorphism
\[[L^{m_1}G\backslash \Gr_{\la_\bullet,\mmu}^{(n)}]\cong [L^{m_1}G\backslash \Gr_{\la_\bullet}^{(m_2)}]\times^{L^{m_2}G}\Gr_{\mmu}^{(n)},\]
which induces an isomorphism
\begin{equation}
\label{E: break hecke}
\Hk_{\la_\bullet,\mmu}^{\loc(m_1)}\cong \Hk_{\mu_\bullet}^{\loc(m_2)}\times_{t_\rightone,\bfB L^{m_2}G, t_\leftone}\Hk^{\loc(m_1)}_{\la_\bullet}.
\end{equation}
In other words, giving an $R$-point of $\Hk_{\la_\bullet,\mmu}^{\loc(m_1)}$ is equivalent to giving an $R$-point $x_1$ of $\Hk_{\la_\bullet}^{\loc(m_1)}$ and an $R$-point $x_2$ of $\Hk^{\loc(m_2)}_{\mmu}$, together with an isomorphism of $L^{m_2}G$-torsors $\mE_{x_1,\leftone}|_{D_{m_2, R}}\simeq \mE_{x_2,\rightone}|_{D_{m_2, R}}$. 

In a special case, the isomorphism \eqref{E: break hecke} induces a perfectly smooth morphism
\begin{equation}
\label{break chain}
\Hk^{\loc(m_1)}_{\lambda_\bullet,\mmu}\to \Hk_{\lambda_\bullet}^{\loc(m_1)}\times \Hk_{\mmu}^{\loc(m_2)}\to  \Hk_{\lambda_\bullet}^{\loc(m_1-m_2)}\times \Hk_{\mmu}^{\loc(m_2)},
\end{equation}
where the first map is induced from \eqref{E: break hecke} (which is perfectly smooth) and the second map is $ \res^{m_1}_{m_1-m_2}\times\id$ (which is also perfectly smooth).

\subsection{Moduli of local shtukas}
\label{S: Moduli of loc Sht}
In this subsection, we define various moduli of local (iterated) shtukas as prestacks and various correspondences between them. 

\begin{dfn}\label{lsht}
Let $\mmu=(\mu_1,\ldots,\mu_t)$ be a sequence of dominant coweights.  Let $R$ be a perfect $k$-algebra. An $R$-family of \emph{local (iterated) shtukas} of singularities bounded by $\mmu$ consists of
\begin{itemize}
\item a point of $\Hk_\mmu^\loc$, i.e. a sequence of modifications of $G$-torsors over $D_R$ of relative positions $(\preceq\mu_t,\preceq\mu_{r-1},\ldots, \preceq \mu_1)$ as in \eqref{pt of hk}, and
\item an isomorphism $\psi: {^\sigma}\mE_t\cong \mE_0$. (Here the notation $^\sigma \mE_t$ is defined at the beginning of Section~\ref{Sec:affine DLV}.)
\end{itemize}
The moduli space that assigns every perfect $k$-algebra $R$ the groupoid of $R$-families of local Shtukas of singularities bounded by $\mmu$ is denoted by $\Sht_\mmu^\loc:=\Sht_{\mu_1,\ldots,\mu_t}^{\loc}$.
In other words, $\Sht_\mmu^\loc$ classifies sequences of modifications 
\[
\calE_t \stackrel{\beta_t}{\dashrightarrow} \calE_{t-1} \stackrel{\beta_{r-1}}{\dashrightarrow} \cdots \stackrel{\beta_2}{\dashrightarrow} \calE_1 \stackrel{\beta_1}{\dashrightarrow}  {}^\sigma \calE_t.
\]
For such an $S$-point $x \in \Sht_\mmu^\loc(S)$, we write $\calE_{x,\leftarrow}$ for $\calE_t$ and $\calE_{x,\to}$ for $\calE_0$.

Note that by definition, there is a forgetful map of prestacks 
$$\varphi^\loc :\Sht_\mmu^\loc\longto \Hk_\mmu^\loc.$$ 
In fact, we can write $\Sht_\mmu^\loc$ as a Cartesian product of $\Hk_\mmu^\loc$ with a Frobenius graph
\begin{equation}
\label{E:mS as Cartesian product}
\xymatrix@C=45pt{
\Sht_\mmu^\loc \ar[r]^{\varphi^\loc} \ar[d]
&   \Hk_\mmu^\loc \ar[d]^-{ t_\leftone\times t_\rightone }
\\ \bfB L^+G \ar[r]^-{1\times\sigma} & \bfB L^+G \times \bfB L^+G.
}
\end{equation}

Take two additional (not necessarily non-empty) sequences $\nu_\bullet,\xi_\bullet$ of dominant coweights. 
We similarly define $\Sht_{\la_\bullet\mid\mmu}^{0,\loc}$ and $\Sht_{\nu_\bullet;\la_\bullet\mid\mmu;\xi_\bullet}^{0,\loc}$ and the \emph{Satake correspondence}
\[\Sht_{\la_\bullet}^{\loc} \xleftarrow{h^{\leftarrow}_{\la_\bullet}} \Sht_{\la_\bullet|\mmu}^{0,\loc}\xrightarrow{h^\rightarrow_{\mmu}} \Sht_\mmu^{\loc}.
\]
For example, $\Sht_{\nu_\bullet;\la_\bullet\mid\mmu;\xi_\bullet}^{0,\loc}$ classifies the following commutative diagrams of modifications.
\[
\xymatrix@C=35pt@R=20pt{
&& \calE'_{s+a} \ar@{-->}[r]^-{\beta'_{s+a}}
& \cdots \ar@{-->}[r]^-{\beta'_{a+1}} & \calE'_a
\\
\calE_{t+a+b} \ar@{-->}[r]^-{\beta_{t+a+b}} & \cdots \ar@{-->}[r]^{\beta_{t+a+1}} &
\calE_{t+a} \ar@{=}[u] \ar@{-->}[r]^{\beta_{t+a}}  & \cdots \ar@{-->}[r]^{\beta_{a+1}} & \calE_a \ar@{=}[u] \ar@{-->}[r]^{\beta_a} & \cdots \ar@{-->}[r]^-{\beta_1} & \calE_0 = {}^\sigma \calE_{t+a+b}.
}
\]
\end{dfn}

\begin{rmk}
\label{R: var and gen of loc Sht}
(1) In mixed characteristic, a local shtuka $\mE\dashrightarrow {^\sigma}\mE$ over $\Spec R$ is nothing but an $F$-crystal with $G$-structures. In particular, If $G=\GL_n$ and $\mu=\omega_i$ is the $i$th fundamental coweight of $\GL_n$, by a theorem of Gabber, $\Sh^\loc_\mu$ can be regarded as the moduli of $p$-divisible groups of height $n$ and dimension $n-i$ over $k$. Therefore, moduli of local shtukas in mixed characteristic can be regarded as vast generalizations of moduli of $p$-divisible groups in characteristic $p$.

(2) What we just defined are local shtukas with singularities at the closed point $s\in D$.  One can also define local shtukas with singularities at the generic point $\eta\in D$, or even with singularities moving along $D$. 
In mixed characteristic, local shtukas with singularities along $D$ are closely related the Breuil-Kisin modules and their moduli spaces are recently constructed by Scholze (\cite{SW}).
\end{rmk}

\begin{remark} 
\label{R:shtukas E to sigma E}
Our convention to define local iterated shtukas using modifications from $\calE_t$ to $^\sigma \calE_t$ agrees \cite{Va, La},
but is incompatible with the definition of affine Deligne-Lusztig varieties in Section~\ref{Sec:affine DLV} (following most literatures on Shimura varieties), where the direction of the modification is from $^\sigma \calE$ to $\calE$.  This will lead a switching from $\mu$ to $\mu^*$ when relating the two definitions (see 
\eqref{E: Nm for loc Sht} and \eqref{E: fiber of hleft}). 
The reason for our choice is limited by the need of considering restricted local shtukas in the next subsection. When comparing that with the de Rham homology of the universal abelian varieties on Shimura varieties, the Grothendieck-Messing theory forces us to use a deeper truncation on $^\sigma \calE$ than on $\calE$.  As we will see in Definition~\ref{tlsht} that the deeper truncation must appear at the target of the modifications, and hence our choice of convention.
\end{remark}

\begin{ex}
\label{Ex: discrete Sht}
If $\mmu=\mu=0$, $\Sht^{\loc}_{0}=\bfB G(\mO)$ is the the classifying stack of the profinite group $G(\mO)$. 
\end{ex}
 
\begin{ex}
\label{Ex:loc Sht and pdiv}
Let $\mO=W(k)$. Let $G=\GL_n$, and $\mu=\omega_i=1^i0^{n-i}$ the $i$th fundamental coweight. Then by a theorem of Gabber, $\Sht^\loc_{\omega_i}$ is isomorphic to the prestack of $p$-divisible groups of dimension $i$, and height $n$ over $k$.
\end{ex}

\begin{definition}
\label{D:partial Frobenius}
In addition to the Satake correspondences, there is the following important  \emph{partial Frobenius morphism} between the moduli of local iterated shtukas, given by
\begin{equation}
\label{hpf}
\xymatrix@R=0pt{
F_\mmu\colon \Sht^\loc_{\mu_1, \dots, \mu_t} \ar[r] & \Sht_{\sigma(\mu_t),\mu_1,\ldots,\mu_{t-1}}^\loc \\ \big( \mE_{t}\stackrel{\beta_{t}} {\dashrightarrow}\mE_{t-1}\stackrel{\beta_{t-1}}{\dashrightarrow}\cdots\stackrel{\beta_2} {\dashrightarrow}\mE_1\stackrel{\beta_1}{\dashrightarrow}{}^\sigma\mE_{t} \big) \ar@{|->}[r] & \big( \mE_{t-1}\stackrel{\beta_{t-1}}{\dashrightarrow}\cdots \stackrel{\beta_2}{\dashrightarrow}\mE_1 \stackrel{\beta_1}{\dashrightarrow}{}^\sigma\mE_{t} \stackrel{\sigma(\beta_{t})}{\dashrightarrow}{}^\sigma \mE_{t-1} \big).
}
\end{equation}
To motivate later definition of the (inverse) partial Frobenius between moduli of restricted local Shtukas.  we explain in steps the inverse $F_{\mmu}^{-1}$ map on $S$-points for perfect affine scheme $S = \Spec R$ as follows.
\begin{align*}
&\Sht_{\sigma(\mu_t), \mu_1, \dots, \mu_{t-1}} (S)
\\
\cong\ & \big\{ (x, \psi)\; \big|\; x \in \Hk_{\sigma(\mu_t), \mu_1, \dots, \mu_{t-1}} (S) \textrm{ and } \psi:{}^\sigma\calE_{x,\leftone} \simeq \calE_{x,\rightone} \big\} && (\textrm{Definition~\ref{lsht}})
\\
\cong\ & \Bigg\{\! (x_1,x_2,\psi', \psi)\; \Bigg|\! \begin{array}{l} x_1 \in \Hk_{\sigma(\mu_t)} (S) ,\ x_2 \in \Hk_{\mu_1, \dots, \mu_{t-1}} (S),\\ \psi':\calE_{x_1, \leftone} \simeq \calE_{x_2, \rightone}, \textrm{ and}\\ \psi:{}^\sigma\calE_{x_2,\leftone} \simeq \calE_{x_1,\rightone}
\end{array}\!\! \Bigg\} && \Bigg(\!\! \begin{array}{l} \textrm{Use \S\ref{SS:break hecke}}\\ \textrm{to replace $x$}
\\
\textrm{by $(x_1,x_2, \psi')$}
\end{array}\!
\!\Bigg)
\\
\cong\ & \Bigg\{\! (x'_1,x_2,\psi', \psi'')\; \Bigg|\! \begin{array}{l} x'_1 \in \Hk_{\mu_t} (S) ,\ x_2 \in \Hk_{\mu_1, \dots, \mu_{t-1}} (S),\\ \psi':{}^\sigma\calE_{x'_1, \leftone} \simeq \calE_{x_2, \rightone},\textrm{ and}\\ \psi'':\calE_{x_2,\leftone} \simeq \calE_{x'_1,\rightone}
\end{array}\!\! \Bigg\} && \bigg(\!\! \begin{array}{l}x'_1 = \sigma^{-1}(x_1)\\ \psi'' = \sigma^{-1}(\psi)\end{array}\!\! \bigg)
\\
\cong \ &
\bigg\{ 
(x', \psi'')\;\bigg|\!\!
\begin{array}{l}
x' \in \Hk_{\mu_1,\dots, \mu_t}(S) \textrm{ and}\\
\psi':{}^\sigma \calE_{x',\leftone} \simeq \calE_{x',\rightone}
\end{array}\!\! \bigg\}
&& \!\!\!\!\!\!\bigg(\!\!\! \begin{array}{l}
\textrm{Use \S\ref{SS:break hecke} to glue} \\ (x_2,x_1, \psi'')\textrm{ into }x' \ 
\end{array}\!\!\!
\bigg)
\\
\cong \ & \Sht_{\mu_1, \dots, \mu_t}(S). && (\textrm{Definition~\ref{lsht}})
\end{align*}

Similarly, for dominant coweights $\lambda_\bullet, \mmu,\nu$, there is a natural partial Frobenius morphism
\[
F_{\la_\bullet\mid \mmu;\nu}\colon
\Sht^{0,\loc}_{\la_\bullet\mid \mmu;\nu} \to \Sht^{0,\loc}_{\sigma(\nu);\la_\bullet\mid \mmu}.
\]
\end{definition}
The morphism $F_\mmu$ (and $F_{\la_\bullet\mid \mmu;\nu}$) is an isomorphism of prestacks (i.e. as functors over \emph{perfect} $k$-algebras).

\begin{remark}
In the equal characteristic case, the partial Frobenius morphism $F_{\mmu}$ can be defined for the imperfect versions of local shtukas, which then is a just universal homeomorphism. 
\end{remark}

\begin{definition}
Let $\lambda_\bullet$ and $\mmu$ be two sequences of dominant coweights.
We are interested in the following correspondence of prestacks 
\[\xymatrix{&\Sht_{\la_\bullet|\mmu}^{\loc}\ar_{\overleftarrow{h}^\loc_{\lambda_\bullet}}[dl]\ar^{\overrightarrow{h}^\loc_{\mmu}}[dr]&\\
\Sht_{\la_\bullet}^\loc&&\Sht_{\mmu}^\loc,
}\]
where $\Sht_{\la_\bullet\mid\mmu}^{\loc}$ is the prestack classifying, for each perfect ring $R$, the following commutative diagram of modifications of $G$-torsors over $D_R$:
\begin{equation}\label{qi for diff type}
\xymatrix@C=35pt{
\mE'_s \ar@{-->}[r]^{\beta'_s} \ar@{-->}[d]_\beta & \cdots  \ar@{-->}[r]^{\beta'_2} & \calE'_1 \ar@{-->}[r]^{\beta'_1} & {}^\sigma\mE'_s \ar@{-->}[d]^{\sigma(\beta)}\\
\mE_t \ar@{-->}[r]^{\beta_t} & \cdots  \ar@{-->}[r]^{\beta_2} & \calE_1 \ar@{-->}[r]^{\beta_1} & {}^\sigma\mE_t,
}
\end{equation}
such that the top row (resp. bottom row) defines an $R$-point of $\Sht_{\lambda_\bullet}^\loc$ (resp. $\Sht_{\mu_\bullet}^\loc$).
We call such a diagram a \emph{Hecke correspondence} from the top row to the bottom row.

If we further require $\beta$ to have relative position $\preceq \nu$, we get a closed sub-prestack $\Sht_{\la_\bullet|\mmu}^{\nu,\loc}$ of $\Sht_{\la_\bullet|\mmu}^{\loc}$. Note that if $\nu=0$, this reduces to the previously defined Satake correspondence so our notations are consistent. 
\end{definition}

\begin{rmk}
One should regard $\Sht_{\la_\bullet|\mmu}^{\loc}$  as the Hecke correspondence between  the moduli of local shtukas associated to different (sequences of) coweights. It is non-empty if and only if $B(G,-|\la_\bullet|)\cap B(G,-|\mmu|)\neq \emptyset$.
\end{rmk}

\subsubsection{Points of $\Sht^\loc$ and the Newton map}
Recall that we introduced the set $A(G)$, which is the quotient of $G(L)$ by $\sigma$-conjugation of elements in $G(\mO_L)$.
We define the bijection
\begin{equation}
\label{E: group rep of Sht}
\Sht^{\loc}_{\mu}(\bar k)\cong A(G,\mu)
\end{equation}
as follows: each $\bar k$-point $x\in \Sht^\loc_\mu$, defines a local shtuka
$$x^*\mE\dashrightarrow x^*({^\sigma}\mE)$$ over $\Spec \bar k$: by trivializing $x^*\mE$ over $\mO_L$ as a $G$-torsor, the map $x^*({^\sigma}\mE)\dashrightarrow x^*\mE$\footnote{Note that the arrow is reversed. See also Remark~\ref{R:shtukas E to sigma E}.} is given by an element $b_x\in \overline{G(\mO_L)\varpi^{\mu^*}G(\mO_L)}$, which changes as $g^{-1}b_x\sigma(g)$ for $g\in G(\mO_L)$ when the trivialization of $x^*\mE$ changes. Therefore, $[b_x]$ is a well-defined element in $A(G,\mu)$. Conversely, given an element in $A(G,\mu)$, represented by an element in $b\in \overline{G(\mO_L)\varpi^{\mu^*}G(\mO_L)}$, we can define a local shtuka structure on the trivial $G$-torsor $\mE^0$ as $b\sigma: \mE^0\dashrightarrow {^\sigma}\mE^0$. This gives the bijection.

Then we may reinterpret \eqref{E: SetNewton} as the Newton map
\begin{equation}
\label{E: Nm for loc Sht}
\mN: \Sht^\loc_\mu(\bar k)\to B(G,\mu^*)
\end{equation}

According to \cite{RR}, for every $b\in B(G,\mu^*)$, there is a locally closed substack $\Sht^\loc_{\mu,b}\subset \Sht^{\loc}_\mu$ such that $\Sht^\loc_{\mu,b}(\bar k)=\mN^{-1}(b)$.

\subsubsection{Relation to affine Deligne-Lusztig varieties}
\label{SS:relation of Hecke of mS and ADLV}
We have the following lemma.
\begin{lem}
\label{L: Rep of Hk of Sht}
The morphism $\overleftarrow{h}^\loc_{\la_\bullet}: \Sht^{\nu,\loc}_{\la_\bullet \mid \mmu}\to \Sht^\loc_{\la_\bullet}$ is representable by a perfectly proper perfect scheme, and therefore $\overleftarrow{h}^\loc_{\la_\bullet}: \Sht^{\loc}_{\la_\bullet \mid \mmu}\to \Sht^\loc_{\la_\bullet}$ is ind-representable and ind-proper.
\end{lem}
\begin{proof}
For simplicity, we assume that both $\la_\bullet=\la$ and $\mmu=\mu$ are is single coweight.
Let $\Spec R\to  \Sht^\loc_{\la}$, giving by a local shtuka $\beta': \mE'\dashrightarrow {^\sigma}\mE'$. Then $\Spec R \times_{\Sht^\loc_{\la}} \Sht^\loc_{\la\mid \mu}$ classifies for every $x:\Spec R'\to \Spec R$, a modification $\beta: x^*\mE'\dashrightarrow \mE$ such that $\inv(\beta)\preceq\nu$ and $\inv(\sigma(\beta)\beta'\beta^{-1})\preceq \mu$. All possible modifications $\beta$ are represented by the relative Schubert variety $\Gr_\nu$ over $\Spec R$, and the condition $\inv(\sigma(\beta)\beta'\beta^{-1})\preceq \mu$ defines a closed subscheme of it.
\end{proof}

We describe fibers of $\overleftarrow{h}^\loc_{\la_\bullet}$ over a $\kappa$-point $x$ of $\Sht_{\la_\bullet}^{\loc}$, where $\kappa$ is a finite extension (or an algebraic closure) of $k$. Let $b_x\in A(G,\mu)$ be the element as in \eqref{E: group rep of Sht}. Then
\begin{equation}
\label{E: fiber of hleft}
(\overleftarrow{h}^\loc_{\lambda_\bullet})^{-1}(x)\simeq X_{\mu_\bullet^*}(b_x)
\end{equation}
by sending $\beta$ to $\beta^{-1}$.
Similarly, the fiber $(\overleftarrow{h}^\loc_{\lambda_\bullet})^{-1}(x)\cap \Sht_{\la_\bullet|\mmu}^{\nu,\loc}$ is then identified with $X_{\mu_\bullet^*,\nu^*}(b_x)$. 

In particular, when $\lambda_\bullet = \tau \in \xcoch(Z_G)$ is a single central coweight and $\mmu = \mu$ is a single coweight, $\Sht^\loc_\tau$ is identified with $\bfB J_\tau(\mO) =\bfB G(\mO)$ as prestacks, and we get the following commutative diagram with the left square Cartesian
\begin{equation}\label{rigid corr}
\xymatrix{
\on{pt}\ar[d]& X_{\mu^*,\nu^*}(\tau^*)\ar[d]\ar[dr]\ar[l]&\\
\Sht_\tau^\loc& \Sht_{\tau|\mu}^{\nu,\loc}\ar[l]\ar[r]& \Sht_\mu^{\loc}.
}.
\end{equation}

\subsubsection{Composition of isogenies of local shtukas}
There is a natural composition morphism
\begin{equation}
\label{E:composition}
\on{Comp}^{\loc}: \Sht_{\kappa_\bullet |\lambda_\bullet}^{\loc} \times_{\Sht_{\lambda_\bullet}^\loc} \Sht_{\lambda_\bullet|\mu_\bullet}^{\loc} \longto \Sht_{\kappa_\bullet |\mu_\bullet}^{\loc}.
\end{equation}
More precisely, a point of $\Sht_{\kappa_\bullet |\lambda_\bullet}^{\loc} \times_{\Sht_{\lambda_\bullet}^\loc} \Sht_{\la_\bullet|\mu_\bullet}^{\loc}$ is represented by the following two commutative squares glued along the middle row, and the map is given by  forgetting the middle row but only remembering the composed arrows in left and right columns.
\begin{equation}
\label{E: comp of HkS}
\xymatrix@C=35pt{
\mE''_r \ar@{-->}[r]^{\beta''_r} \ar@{-->}[d]^{\ga'}\ar@/_20pt/@{-->}[dd]_{\ga\ga'}& \cdots  \ar@{-->}[r]^{\beta''_2} & \calE''_1 \ar@{-->}[r]^{\beta''_1} & {}^\sigma\mE''_r \ar@{-->}[d]_{\sigma(\ga')}\ar@/^20pt/@{-->}[dd]^-{\sigma(\ga\ga')}\\
\mE'_s \ar@{-->}[r]^{\beta'_s} \ar@{-->}[d]^{\ga} & \cdots  \ar@{-->}[r]^{\beta'_2} & \calE'_1 \ar@{-->}[r]^{\beta'_1} & {}^\sigma\mE'_s\ar@{-->}[d]_{\sigma(\ga)}\\
\mE_t \ar@{-->}[r]^{\beta_t} & \cdots \ar@{-->}[r]^{\beta_2} & \calE_1 \ar@{-->}[r]^{\beta_1} & {}^\sigma\mE_t
}
\end{equation}
Note that $\on{Comp}^\loc$ maps $\Sht_{\kappa_\bullet |\lambda_\bullet}^{\nu',\loc} \times_{\Sht_{\lambda_\bullet}^\loc} \Sht_{\lambda_\bullet|\mu_\bullet}^{\nu,\loc}$ to $\Sht_{\kappa_\bullet |\mu_\bullet}^{\nu+\nu',\loc}$.

\medskip
In later sections, we need to replace $\Sht^\loc_\mmu$ by its restricted version and to upgrade the previous constructions into cohomological correspondences. For this purpose, we need to discuss other disguises of the Hecke correspondences and their compositions, via Satake correspondences and partial Frobenius morphisms.

The following is a simple yet important observation.
\begin{lem}\label{decomp1}
For any dominant coweight $ \eta$ such that $\eta\succeq |\lambda_\bullet|+\sigma(\nu)$ \emph{or} $\eta \succeq |\mmu|+\nu$, we have the following commutative diagram of prestacks
\begin{equation}
\label{E:pentagon factorization}
\xymatrix@C=10pt{
&&
\Sht_{\la_\bullet|\mmu}^{\nu,\loc}  \ar[dl] \ar[dr]\\
&  \Sht_{\la_\bullet|(\sigma(\nu^*), \eta)}^{0,\loc} \ar[dl] \ar[d]  && \Sht_{(\eta,\nu^*)|\mmu}^{0,\loc} \ar[d] \ar[dr]
\\
\Sht_{\la_\bullet}^\loc & \Sht_{\sigma(\nu^*), \eta}^\loc \ar[rr]^-{F^{-1}_{\eta, \nu^*}} &&  \Sht_{\eta,\nu^*}^\loc  & \Sht_{\mmu}^\loc,
}
\end{equation}
where the pentagon is Cartesian when composing $F^{-1}_{\eta, \nu^*}$ with the right vertical arrow.\footnote{We could have stated the lemma in terms of the partial Frobenius morphism $F_{\eta, \nu^*}$ (as it is an isomorphism). But we shall later consider a restricted analogue of this lemma (see Definition~\ref{D:mS lambda mu truncated}), where only the analogue of the \emph{inverse} of the partial Frobenius morphism is defined.}
\end{lem}
In other words, the Hecke correspondence is the composition of a Satake correspondence, the partial Frobenius, and another Satake correspondence.
\begin{proof}
Recall that $\Sht^{\nu,\loc}_{\la_\bullet\mid\mmu}$ classifies the following commutative diagram of modifications of $G$-torsors
\[
\xymatrix@C=35pt{
\mE'_s \ar@{-->}[r]^{\beta'_s}\ar@{-->}[rrrd]^-{\alpha} \ar@{-->}[d]_\beta & \cdots  \ar@{-->}[r]^{\beta'_2} & \calE'_1 \ar@{-->}[r]^{\beta'_1} & {}^\sigma\mE'_s \ar@{-->}[d]^{\sigma(\beta)}\\ 
\mE_t \ar@{-->}[r]_{\beta_t} & \cdots  \ar@{-->}[r]_{\beta_2} & \calE_1 \ar@{-->}[r]_{\beta_1} & {}^\sigma\mE_t,
}
\]
where comparing to \eqref{qi for diff type}, we added the diagonal arrow $\alpha$ which is forced to be equal to $\sigma(\beta)\circ \beta'_1 \circ \dots, \circ \beta'_s  = \beta_1 \circ \cdots \circ \beta_t \circ \beta $. 
Since $\eta\succeq |\lambda_\bullet|+\sigma(\nu)$ or $\eta \succeq |\mmu|+\nu$ by assumption, automatically we have $\on{Inv}(\alpha) \preceq \eta$.
This rectangular shape diagram is the same as the two commutative triangles glued along the diagonal map $\alpha$.
We know that
\begin{itemize}
\item the upper right triangle represents a point of $\Sht_{\la_\bullet|(\sigma(\nu^*), \eta)}^{0,\loc}$, and
\item the lower left triangle represents a point of $\Sht_{(\eta, \nu^*)|\mmu}^{0,\loc}$.
\end{itemize}
The lemma follows from this.
\end{proof}

The same trick leads the following two lemmas.
\begin{lem}
\label{decomp2}
Let $\lambda_\bullet,\mu_\bullet,\nu_\bullet,\xi_\bullet,\theta$ be dominant coweights such that $\theta\succeq |\la_\bullet|+|\mu^*_\bullet|$ or $\theta\succeq |\xi_\bullet|+|\nu_\bullet^*|$. Then there is a canonical isomorphism
\begin{equation}
\label{E:decomp2}
\Sht^{0,\loc}_{(\la_\bullet,\nu_\bullet)\mid(\mu_\bullet,\xi_\bullet)}
\cong \Sht^{0,\loc}_{\la_\bullet \mid (\mu_\bullet,\theta);\nu_\bullet}
\times_{\Sht^{\loc}_{\mu_\bullet,\theta,\nu_\bullet}}
\Sht^{0,\loc}_{\mu_\bullet; (\theta,\nu_\bullet)\mid \xi_\bullet}.
\end{equation}
\end{lem}
\begin{proof}
The left hand side classifies the following commutative diagram of modifications of $G$-torsors
\[
\xymatrix{
\mE'_{s+a} \ar@{-->}[r]^-{\beta'_{s+a}} \ar@{=}[d] & \cdots  \ar@{-->}[r]^-{\beta'_{s+1}} & \calE'_s \ar@{-->}[r]^-{\beta'_s} & \cdots \ar@{-->}[r]^-{\beta'_1} & \calE'_0={}^\sigma\mE'_{s+a} \ar@{=}@<2ex>[d]\\ 
\mE_{t+b} \ar@{-->}[r]^-{\beta_{t+b}} & \cdots  \ar@{-->}[r]^-{\beta_{t+1}} &\mE_t \ar@{-->}[r]^-{\beta_{t}} & \cdots \ar@{-->}[r]^-{\beta_{1}} & \calE_0 = {}^\sigma\mE_{t+b}.
}
\]
The right hand side classifies the following commutative diagram of modifications of $G$-torsors
\[
\xymatrix{
&& \calE'_s
\ar@{=}[d]
 \ar@{-->}[rr]^-{\beta'_s} && \cdots \ar@{-->}[r]^-{\beta'_1} & \calE'_0={}^\sigma\mE_{s+a} \ar@{=}@<2ex>[d]
\\
\mE'_{s+a} \ar@{-->}[r]^-{\beta'_{s+a}} \ar@{=}[d] & \cdots  \ar@{-->}[r]^-{\beta'_{s+1}} & \calE'_s \ar@{-->}[r]^-{\gamma}
&
\calE_t\ar@{=}[d]
\ar@{-->}[r]^-{\beta_t} & \cdots \ar@{-->}[r]^-{\beta_{1}} & \calE_0 = {}^\sigma\mE_{t+b} 
\\
\mE_{t+b} \ar@{-->}[r]^-{\beta_{t+b}} & \cdots  \ar@{-->}[rr]^-{\beta_{t+1}} &&\mE_t,
}
\]
where the top two rows define a point of $\Sht^{0,\loc}_{\la_\bullet \mid (\mu_\bullet,\theta);\nu_\bullet}$, the bottom two rows define a point of $\Sht^{0,\loc}_{\mu_\bullet; (\theta,\nu_\bullet)\mid \xi_\bullet}$, and they are glued along the middle row which defines a point of $\Sht^{\loc}_{\mu_\bullet,\theta,\nu_\bullet}$.

The two diagrams above differ by the modification $\gamma: \calE'_s \dashrightarrow \calE_t$ which is required to have relative position $\preceq \theta$ (which is superfluous if $\theta\succeq|\la_\bullet|+ |\mu^*_\bullet|$ or $\theta\succeq |\xi_\bullet|+|\nu_\bullet^*|$). Then the isomorphism of the lemma is obtained by adding/removing the modification $\gamma$.
\end{proof}
\begin{lem}
\label{decomp3}
Let $\lambda_\bullet, \mu_\bullet,\nu_\bullet,\zeta, \zeta_1,\zeta_2$ be dominant coweights such that $\zeta\succeq |\la_\bullet|+|\mu_\bullet^*| + |\nu_\bullet^*|$ or $\zeta\succeq \zeta_1+ \zeta_2$. Then there is a canonical isomorphism
\begin{equation}
\label{E:decomp3}
\Sht^{0,\loc}_{\la_\bullet\mid (\mu_\bullet,\zeta, \nu_\bullet)}
\times_{\Sht^\loc_{\mu_\bullet,\zeta, \nu_\bullet}}
\Sht^{0,\loc}_{\mu_\bullet; \zeta\mid (\zeta_1,\zeta_2);\nu_\bullet}
\cong \Sht^{0,\loc}_{\la_\bullet\mid(\mu_\bullet,\zeta_1, \zeta_2, \nu_\bullet)}.
\end{equation}
\end{lem}
\begin{proof}
The right hand side classifies the following commutative diagram of modifications of $G$-torsors
\[
\xymatrix@C=35pt{
\mE'_s \ar@{-->}[r]^{\beta'_s} & \cdots  \ar@{-->}[rrr]&&& \cdots  \ar@{-->}[r]^{\beta'_2} & \calE'_1 \ar@{-->}[r]^-{\beta'_1} & \mE'_0 = {}^\sigma\mE'_s
\\
\mE_{t+a+2} \ar@{-->}[r]^{\beta''_{a}} \ar@{=}[u] & 
\cdots
 \ar@{-->}[r]^{\beta''_{1}} &
\mE_{t+2} \ar@{-->}[r]^{\ga_{2}} &
\mE_{t+1} \ar@{-->}[r]^{\ga_{1}} & \calE_t \ar@{-->}[r]^-{\beta_t} & \cdots \ar@{-->}[r]^-{\beta_1} & \mE_0 = {}^\sigma\mE_{t+a+2}. \ar@{=}@<-3ex>[u]
}
\]
The left hand side classifies the following commutative diagram of modifications of $G$-torsors
\[
\xymatrix@C=35pt{
\mE'_s \ar@{-->}[r]^{\beta'_s} & \cdots  \ar@{-->}[rrr]&&& \cdots  \ar@{-->}[r]^{\beta'_2} & \calE'_1 \ar@{-->}[r]^-{\beta'_1} & \mE'_0 = {}^\sigma\mE'_s
\\
\mE_{t+a+2} \ar@{-->}[r]^{\beta''_{a}} \ar@{=}[u] & 
\cdots
 \ar@{-->}[r]^{\beta''_{1}} &
\mE_{t+2} \ar@{-->}[rr]^{\ga} && \calE_t \ar@{-->}[r]^-{\beta_t} & \cdots \ar@{-->}[r]^-{\beta_1} & \mE_0 = {}^\sigma\mE_{t+a+2}, \ar@{=}@<-3ex>[u]
\\
&& \ar@{=}[u]
\mE_{t+2} \ar@{-->}[r]^{\ga_{2}} &
\mE_{t+1} \ar@{-->}[r]^{\ga_{1}} & \calE_t\ar@{=}[u]
}
\]
where the top two rows define a point of $\Sht^{0,\loc}_{\la_\bullet\mid (\mu_\bullet, \zeta, \nu_\bullet)}$, bottom two rows define a point of $\Sht^{0,\loc}_{\mu_\bullet; \zeta|(\zeta_1,\zeta_2); \nu_\bullet}$, and they are glued along the middle row which defines a point of $\Sht^\loc_{\mu_\bullet, \zeta, \nu_\bullet}$.

The two diagrams above differ by the modification $\gamma: \calE_{t+2} \dashrightarrow \calE_t$ which is required to have relative position $\preceq \zeta$ (which is superfluous if $\zeta\succeq \zeta_1+\zeta_2$ or $\zeta\succeq |\la_\bullet|+|\mu_\bullet^*| + |\nu_\bullet^*|$). Then the isomorphism of the lemma is obtained by adding/removing the modification $\gamma$.
\end{proof}

\begin{lem}
\label{L: comm Sat and pFrob}
Let $\lambda_\bullet, \mu_\bullet, \nu$ be dominant coweights.
The following diagram is Cartesian
\begin{equation}
\label{E:comm Sat and pFrob}
\xymatrix@C=40pt{
\Sht^{0,\loc}_{\sigma(\nu);\la_\bullet\mid \mmu} \ar[r]^-{F^{-1}_{\la_\bullet\mid \mmu;\nu}} \ar[d] & \Sht^{0,\loc}_{\la_\bullet\mid \mmu;\nu} \ar[d]\\
\Sht^{\loc}_{\sigma(\nu),\la_\bullet} \ar[r]^-{F^{-1}_{\lambda_\bullet,\nu}} & \Sht\Sht^{\loc}_{\la_\bullet, \nu}.
}
\end{equation}
\end{lem}
\begin{proof}
Clear.
\end{proof}

Now we can give another interpretation of $\on{Comp}^\loc$.
\begin{lem}
\label{L:decomposition of Comp}
The composition map $\on{Comp}^\loc$ is the composition of the following maps
\begin{small}
\begin{align*}
& \ \Sht^{\nu',\loc}_{\kappa_\bullet\mid\lambda_\bullet}\times_{\Sht^{\loc}_{\lambda_\bullet}}\Sht^{\nu,\loc}_{\lambda_\bullet\mid \mu_\bullet}   
\\
\cong \ &\  \Sht^{0,\loc}_{\kappa_\bullet\mid (\sigma(\nu'^*),\eta')}\times_{\Sht^\loc_{\eta',\nu'^*}}\Sht^{0,\loc}_{(\eta',\nu'^*)\mid \lambda_\bullet}\times_{\Sht^\loc_{\lambda_\bullet}}\Sht^{0,\loc}_{\lambda_\bullet\mid(\sigma(\nu^*), \eta)}\times_{\Sht^\loc_{\eta,\nu^*}}\Sht^{0,\loc}_{(\eta,\nu^*)\mid \mu_\bullet} & (\mathrm{Lemma}\ \ref{decomp1})
\\
\xrightarrow{(1)}  &\ \Sht^{0,\loc}_{\kappa_\bullet\mid (\sigma(\nu'^*),\eta')}\times_{\Sht^\loc_{\eta',\nu'^*}}\Sht^{0,\loc}_{(\eta',\nu'^*)\mid(\sigma(\nu^*), \eta)}\times_{\Sht^\loc_{\eta,\nu^*}}\Sht^{0,\loc}_{(\eta,\nu^*)\mid \mu_\bullet}   
\\
\cong \ &\      \Sht^{0,\loc}_{\kappa_\bullet\mid (\sigma(\nu'^*),\eta')}
\times_{\Sht^\loc_{\eta',\nu'^*}}
\Sht^{0,\loc}_{\eta'\mid (\sigma(\nu^*),\theta);\nu'^*}
\times_{\Sht^\loc_{\sigma(\nu^*),\theta,\nu'^*}}
\Sht^{0,\loc}_{\sigma(\nu^*); (\theta,\nu'^*)\mid\eta}
\times_{\Sht^\loc_{\eta,\nu^*}}\Sht^{0,\loc}_{(\eta,\nu^*)\mid \mu_\bullet} & (\mathrm{Lemma}\ \ref{decomp2})
\\
\cong    \  & \  \Sht^{0,\loc}_{\kappa_\bullet\mid (\sigma(\nu'^*),\eta')}\times_{\Sht^\loc_{\sigma(\nu'^*), \eta'}}
\Sht^{0,\loc}_{\sigma(\nu'^*);\eta'\mid (\sigma(\nu^*),\theta)}
\times_{\Sht^\loc_{\theta,\nu'^*,\nu^*}}
\Sht^{0,\loc}_{ (\theta,\nu'^*)\mid\eta;\nu^*} \times_{\Sht^\loc_{\eta,\nu^*}}\Sht^{0,\loc}_{(\eta,\nu^*)\mid \mu_\bullet}   & (\mathrm{Lemma}\ \ref{L: comm Sat and pFrob})\\
\cong \ & \ \Sht^{0,\loc}_{\kappa_\bullet \mid(\sigma(\nu'^*),\sigma(\nu^*) , \theta)} \times_{\Sht^\loc_{\theta,\nu'^*,\nu^*}} \Sht^{0,\loc}_{(\theta,\nu'^*,\nu^*)\mid \mu_\bullet} & (\mathrm{Lemma}\ \ref{decomp3})  
\\
 \xrightarrow{(2)} & \ \Sht^{0,\loc}_{\kappa_\bullet \mid (\sigma(\nu^*+\nu'^*), \theta )}\times_{\Sht^{\loc}_{\theta,\nu^*+\nu'^*}}\Sht^{0,\loc}_{(\theta,\nu^*+\nu'^*)\mid \mu_\bullet} 
 \\
\cong \ & \Sht^{\nu+\nu',\loc}_{\kappa_\bullet\mid \mu_\bullet }.   & (\mathrm{Lemma}\ \ref{decomp1})
\end{align*}
\end{small}
Here we choose $\eta'\succeq |\kappa_\bullet| + \sigma(\nu')$, $\eta\succeq |\la_\bullet|+\sigma(\nu)$, and $\theta\succeq \eta'+\sigma(\nu)$. The map $(1)$ is induced by the composition of Satake correspondence
\[\Sht^{0,\loc}_{(\eta',\nu'^*)\mid \lambda_\bullet}\times_{\Sht^\loc_{\lambda_\bullet}}\Sht^{0,\loc}_{\lambda_\bullet\mid(\sigma(\nu^*), \eta)} \to \Sht^{0,\loc}_{(\eta',\nu'^*)\mid(\sigma(\nu^*), \eta)},\]
and the map $(2)$ is induced by the convolution of moduli of local shtukas
\[
\Sht^{0,\loc}_{\kappa_\bullet \mid(\sigma(\nu'^*),\sigma(\nu^*) , \theta)} \to \Sht^{0,\loc}_{\kappa_\bullet \mid(\sigma(\nu^*+\nu'^*) , \theta)}
,\quad \Sht^{\loc}_{\theta,\nu'^*,\nu^*} \to \Sht^{\loc}_{\theta,\nu^*+\nu'^*},
\]
\[
\textrm{and}\quad \Sht^{0,\loc}_{(\theta,\nu'^*,\nu^*)\mid \mu_\bullet}  \to \Sht^{0,\loc}_{(\theta,\nu^*+\nu'^*)\mid \mu_\bullet} .
\]
In particular, $\mathrm{Comp}^\loc $ is perfectly proper.
\end{lem}
\begin{proof}
The fiber product of the Hecke correspondences in the first row can be decomposed by Lemma~\ref{decomp1}, and is the same as the moduli prestack that classifies the following commutative diagram of modifications of $G$-torsors
\[
\xymatrix@C=35pt{
\mE''_r \ar@{-->}[r]^{\beta''_r} \ar@{-->}[d]_{\ga'} \ar@{-->}[rrrd]^-{\alpha'}& \cdots  \ar@{-->}[r]^{\beta''_2} & \calE''_1 \ar@{-->}[r]^{\beta''_1} & {}^\sigma\mE''_r \ar@{-->}[d]^{\sigma(\ga')}\\
\mE'_s \ar@{-->}[rrrd]^-{\alpha} \ar@{-->}[r]^{\beta'_s} \ar@{-->}[d]_{\ga} & \cdots  \ar@{-->}[r]^{\beta'_2} & \calE'_1 \ar@{-->}[r]^{\beta'_1} & {}^\sigma\mE'_s\ar@{-->}[d]^{\sigma(\ga)}\\
\mE_t \ar@{-->}[r]^{\beta_t} & \cdots \ar@{-->}[r]^{\beta_2} & \calE_1 \ar@{-->}[r]^{\beta_1} & {}^\sigma\mE_t.
}
\]
The map (1) amounts to forgetting the maps in the middle row $\calE'_s \stackrel{\beta'_s}{\dashrightarrow} \cdots  \stackrel{\beta'_1}{\dashrightarrow} {}^\sigma \calE'_s$.
Next, we use Lemma~\ref{decomp2} to rewrite the moduli problem as the commutative diagram
\begin{equation}
\label{E:composition isogeny proof}
\xymatrix@C=35pt{
\mE''_r \ar@{-->}[r]^{\beta''_r} \ar@{-->}[d]_{\ga'} \ar@{-->}[rrrdd]^-{\delta} \ar@{-->}[rrrd]^-{\alpha'}& \cdots  \ar@{-->}[r]^{\beta''_2} & \calE''_1 \ar@{-->}[r]^{\beta''_1} & {}^\sigma\mE''_r \ar@{-->}[d]^{\sigma(\ga')}\\
\mE'_s \ar@{-->}[rrrd]^-{\alpha} \ar@{-->}[d]_{\ga} & & & {}^\sigma\mE'_s\ar@{-->}[d]^{\sigma(\ga)}\\
\mE_t \ar@{-->}[r]^{\beta_t} & \cdots \ar@{-->}[r]^{\beta_2} & \calE_1 \ar@{-->}[r]^{\beta_1} & {}^\sigma\mE_t
}
\end{equation}
by adding the modification $\delta: \calE''_r\dashrightarrow {}^\sigma \calE_t$.
The four triangles (from top to bottom) correspond to the four factors (from left to right) in the product of line $4$.
After that, we apply Lemma~\ref{L: comm Sat and pFrob} twice to get
\begin{align*}
&
- \times_{\Sht^\loc_{\eta',\nu'^*}}
\Sht^{0,\loc}_{\eta'\mid (\sigma(\nu^*),\theta);\nu'^*}
\times_{\Sht^\loc_{\sigma(\nu^*),\theta,\nu'^*}}
\Sht^{0,\loc}_{\sigma(\nu^*); (\theta,\nu'^*)\mid\eta}
\\
\cong\ & - \times_{\Sht^\loc_{\sigma(\nu'^*), \eta'}}
\Sht^{0,\loc}_{\sigma(\nu'^*);\eta'\mid (\sigma(\nu^*),\theta)}
\times_{\Sht^\loc_{\sigma(\nu^*),\theta,\nu'^*}}
\Sht^{0,\loc}_{\sigma(\nu^*); (\theta,\nu'^*)\mid\eta}  &&\textrm{Lemma~\ref{L: comm Sat and pFrob} on the first product}
\\
\cong\ &- \times_{\Sht^\loc_{\sigma(\nu'^*), \eta'}}
\Sht^{0,\loc}_{\sigma(\nu'^*);\eta'\mid (\sigma(\nu^*),\theta)}
\times_{\Sht^\loc_{\theta,\nu'^*,\nu^*}}
\Sht^{0,\loc}_{ (\theta,\nu'^*)\mid\eta;\nu^*} . &&\textrm{Lemma~\ref{L: comm Sat and pFrob} on the second product}
\end{align*}
This does not change the moduli problem above, but only change how the gluing works with the inverse partial Frobenius morphism.
Indeed, in line 4, the first and the third fiber product each involves an inverse partial Frobenius; in line 5, only the second product involves the inverse partial Frobenius, applied twice.

Going from line 5 to line 6, we apply Lemma~\ref{decomp3}, which amounts to forgetting the modifications $\alpha'$ and $\alpha$ from the diagram \eqref{E:composition isogeny proof}. We then arrive at the gluing of two big upper right and big lower left triangles.
Finally, we use Satake correspondence pushforward (2) to replace modifications $\gamma$ and $\gamma'$ by their composite $\gamma\circ \gamma'$, and hence the  modifications $\sigma(\gamma)$ and $\sigma(\gamma')$ by their composite $\sigma(\gamma\circ \gamma')$.
We end up with the following diagram of modifications
\[
\xymatrix@C=35pt{
\mE''_r \ar@{-->}[r]^{\beta''_r} \ar@{-->}[d]_{\ga\ga'} \ar@{-->}[rrrd]^-{\delta}& \cdots  \ar@{-->}[r]^{\beta''_2} & \calE''_1 \ar@{-->}[r]^{\beta''_1} & {}^\sigma\mE''_r \ar@{-->}[d]^{\sigma(\ga\ga')}\\
\mE_t \ar@{-->}[r]^{\beta_t} & \cdots \ar@{-->}[r]^{\beta_2} & \calE_1 \ar@{-->}[r]^{\beta_1} & {}^\sigma\mE_t,
}
\]
which is the same as $ \Sht^{\nu+\nu',\loc}_{\kappa_\bullet\mid \mu_\bullet }$ by Lemma~\ref{decomp1}.

We finally remind the reader that the map $\mathrm{Comp}^\loc$ is defined, on the level of moduli spaces, to modify the diagram \eqref{E: comp of HkS} by removing the middle row and replacing the left and right vertical arrows by their composites. This is exactly what we just explained in the composition of the series of maps of the lemma. The factorization of the map $\mathrm{Comp}^\loc$ is clear from this.

The perfect properness of $\mathrm{Comp}^\loc$ follows from the perfect properness of the maps (1) and (2), which are in turn consequences of perfect properness of convolution products and Satake correspondences for affine Grassmannians.
\end{proof}


\subsection{Moduli of restricted local shtukas}
\label{S: moduli of res loc Sht}

The definition of the moduli of local (iterated) shtukas in the previous subsection is convenient for various geometric constructions we need.
However, as defined in this way, we cannot directly apply the usual $\ell$-adic formalism to it. As mentioned in Example \ref{Ex: discrete Sht}, if $\mmu$ is zero or more generally a cocharacter $\tau\in \xcoch(Z_G)$, then 
$\Sht_\tau^\loc\cong \bfB G(\mO)$ as prestacks, which is not perfectly of finite presentation. 
So we need an approximation of $\Sht_\mmu^\loc$.

\begin{dfn}\label{tlsht}
Let $\mmu$ be a sequence of dominant coweights, and let $(m,n)$ be a pair of non-negative integers such that $m-n$ is $\mmu$-large integer.
The moduli stack $\Sht_\mmu^{\loc(m,n)}$ of $(m,n)$-\emph{restricted local iterated shtukas} is defined as the stack that classifies for every perfect $k$-algebra $R$, 
\begin{itemize}
\item an $R$-point of $ \Hk^{\loc(m)}_\mmu$,  and 
\item an isomorphism of $L^nG$-torsors over $\Spec R$
$$\psi:{^\sigma} (\mE_\leftone|_{D_{n}}) \simeq(\mE_\rightone|_{D_{m}})|_{D_{n}},$$ 
where  $\mE_\rightone|_{D_{m}}$ and $\mE_\leftone|_{D_{n}}$ are the canonical $L^{m}G$-torsor and the $L^{n}G$-torsor over $\Hk_\mmu^{\loc(m)}$ given by \eqref{torsor2} and \eqref{torsor1}, respectively.
\end{itemize}
Forgetting the isomorphism $\psi$ defines a natural morphism
\begin{equation}
\label{E:restriction morphism}
\varphi^{\loc(m,n)}: \Sht_\mmu^{\loc(m,n)} \to \Hk_\mmu^{\loc(m)},
\end{equation}
which exhibits $\Sht_{\mmu}^{\loc(m,n)}$ as an $\Aut(^{\sigma}\mE_\leftone|_{D_{n}})$-torsor over $\Hk_\mmu^{\loc(m)}$. So $\varphi^{\loc(m,n)}$ is a perfectly smooth morphism of relative dimension $n \dim G$. In particular, $\Sht_\mmu^{\loc(m,0)}=\Hk_\mmu^{\loc(m)}$.
Similarly to \eqref{E:mS as Cartesian product}, we can write $\Sht_\mmu^{\loc(m,n)}$ as a Cartesian product of $\Hk_\mmu^{\loc(m)}$ with a Frobenius graph
\begin{equation}
\label{E:mS restricted as a Cartesian product}
\xymatrix@C=90pt{
\Sht_\mmu^{\loc(m,n)} \ar[r]^{\varphi^{\loc(m,n)}} \ar[d]
&  \Hk_\mmu^{\loc(m)} \ar[d]^-{ t_\leftone\times \res^{m}_{n}\circ t_\rightone }
\\ \bfB L^nG \ar[r]^{1\times \sigma} & \bfB L^nG \times \bfB L^nG.
}
\end{equation}

Let $\nu_\bullet,\xi_\bullet$ be two additional (not necessarily non-empty) sequences of dominant coweights. 
We similarly define $\Sht_{\lambda_\bullet\mid\mmu}^{0, \loc(m,n)}$ (resp.  $\Sht_{\nu_\bullet;\lambda_\bullet\mid\mmu;\xi_\bullet}^{0, \loc(m,n)}$) to be the stack classifying, for each perfect $k$-algebra $R$,
\begin{itemize}
\item
an $R$-point of $\Hk_{\lambda_\bullet|\mmu}^{0,\loc(m)}$ ( resp. $\Hk_{\nu_\bullet;\lambda_\bullet|\mmu;\xi_\bullet}^{0,\loc(m)}$) ;
\item
and an isomorphism $\psi: {}^\sigma(\calE_\leftone|_{D_{n}}) \simeq (\calE_\rightone|_{D_{m}})|_{D_{n}}$ (or equivalently $ \psi: {}^\sigma(\calE'_\leftone|_{D_{n}})\simeq (\calE'_\rightone|_{D_{m}})|_{D_{n}}$) of $L^nG$-torsors (see Remark \ref{R: pullbacktwotorsor}).
\end{itemize}

There is a Satake correspondence for restricted local shtukas that is compatible with the Satake correspondence of restricted local Hecke stacks, as shown by the following Cartesian diagram.
\begin{equation}
\label{E:S and Hk Satake compatibility}
\xymatrix@C=40pt{
\Sht_{\la_\bullet}^{\loc(m,n)} \ar[d]^{\varphi^{\loc(m,n)}}
& \ar[l]_{h^{\leftarrow}_{\la_\bullet}} \Sht_{\la_\bullet|\mmu}^{0,\loc(m,n)}\ar[r]^{h^\rightarrow_{\mmu}}  \ar[d]^{\varphi^{\loc(m,n)}}&  \Sht_\mmu^{\loc(m,n)} \ar[d]^{\varphi^{\loc(m,n)}}
\\
\Hk_{\la_\bullet}^{\loc(m)}
& \ar[l]_{h^{\leftarrow}_{\la_\bullet}} \Hk_{\la_\bullet|\mmu}^{0,\loc(m)}\ar[r]^{h^\rightarrow_{\mmu}} & \Hk_\mmu^{\loc(m)}.
}
\end{equation}
In particular, all horizontal maps are perfectly proper, as this is true for the Satake correspondences \eqref{E: Sat corr} between Schubert varieties.
\end{dfn}

The convolution map of local Hecke stacks induces \eqref{E: conv local Hk} induces 
\begin{equation}
\label{E: conv local Sht}
m: \Sht^{\loc(m,n)}_{\nu_\bullet,\mmu,\la_\bullet}\to \Sht^{\loc(m,n)}_{\nu_\bullet,|\mmu|,\la_\bullet},
\end{equation}
which is perfectly proper

\subsubsection{Group theoretic description of $\Sht_\mmu^{\loc(m,n)}$}
\label{SS:description of mS in terms of group elements}
We give an explicit description of $\Sht_{\mmu}^{\loc(m,n)}$ as a quotient stack. 
Consider the fiber product
\begin{equation}
\label{E:framed local shtukas}
\Sht_{\mmu}^{\loc(m,n),\Box}:=\Sht_{\mmu}^{\loc(m,n)}\times_{\Hk^{\loc(m)}_\mmu}\Gr_{\mmu},
\end{equation}
which is the stack classifying an $R$-point $x=(\mE_t \dashrightarrow \cdots \dashrightarrow  \mE_0 = \mE^0)$ of $ \Gr_\mmu$ together with an isomorphism 
$
\psi: {}^\sigma(\mE_t|_{D_{n}}) \simeq \mE_0|_{D_{n}} = \mE^0|_{D_{n}}
$ of $L^nG$-torsors. This is an $L^{m}G$-torsor over $\Sht_{\mmu}^{\loc(m,n)}$. Sending $\psi$ to $\sigma^{-1}(\psi)$ gives rise to a trivialization of $\calE_t |_{D_n}$, and hence a natural isomorphism
\[
\xymatrix@R=0pt{
\Sht_{\mmu}^{\loc(m,n),\Box}\ar[r]^-\cong & \Gr_{\mmu}^{(n)}\\
(x, \psi) \ar@{|->}[r] & (x, \sigma^{-1}(\psi): \calE_t|_{D_{n}} \simeq {}^{\sigma^{-1}}\calE^0|_{D_{n}} \cong \calE^0|_{D_{n}}).
}
\]
In addition, the action of $L^{m}G$ on $\Sht_\mmu^{\loc(m,n),\Box}$ (to recover $\Sht_\mmu^{\loc(m,n)}$) can be identified with the $L^{m}G$-action on $\Gr_\mmu^{(n)}$ given by a twisted conjugation
\begin{equation}
\label{E:twisted action of LmG on framed local shtukas}
c_{\sigma^{-1}}(g)(y) = gy \sigma^{-1}(\pi_{m,n}(g^{-1})) ,\quad \textrm{for }y\in \Gr_\mmu^{(n)},\ g\in L^{m}G,
\end{equation}
where $\pi_{m,n}:L^{m}G \to L^n G$ is the natural map (see Definition~\ref{D:local Hecke stack setup}).
Taking the quotient with respect to this action, we obtain a canonical isomorphism  
\begin{equation}
\label{E:quotient expression of Sloc}
\Sht_\mmu^{\loc(m,n)}\cong[\Gr_\mmu^{(n)}/c_{\sigma^{-1}} L^{m}G].
\end{equation}

By a similar discussion, we can show that
$$\Sht_{\la_\bullet\mid\mmu}^{0,\bba,\loc(m,n)} \cong [\Gr^{0,\bba,(n)}_{\la_\bullet\mid\mmu} / c_{\sigma^{-1}} L^{m}G] \quad \textrm{and}\quad \Sht_{\nu_\bullet; \la_\bullet\mid\mmu; \xi_\bullet}^{0,\loc(m,n)} \cong [\Gr_{\nu_\bullet}\tilde\times \Gr_{\la_\bullet\mid\mmu}\tilde\times \Gr_{\xi_\bullet}^{0,(n)}/ c_{\sigma^{-1}} L^{m}G].$$

\begin{remark}
In the equal characteristic case, the affine Grassmannians and the affine Deligne-Lusztig varieties admit canonical deperfections. So we can define an imperfect version of the moduli of restricted local iterated shtukas. For those, a trivialization of $\calE_\rightone$ for a truncated local shtukas does not give a trivialization $\sigma^{-1}(\psi)$ of $\calE_\leftone|_{D_{n}}$ as we cannot use the inverse Frobenius map. We would need to consider the moduli space that classifies trivializations of both $\calE_\rightone|_{D_{m}}$ and $\calE_\leftone|_{D_{n}}$.
\end{remark}

\begin{ex}
\label{Ex:mStau mn}
When $\mu=0$ or more generally $\mu=\tau\in\xcoch(Z_G)$, we have $\Gr_\tau^{(n)}\cong L^{n}G$ and the action \eqref{E:twisted action of LmG on framed local shtukas} reduces to the $\sigma^{-1}$-conjugacy action of $L^{m}G$ on $L^{n}G$, 
$$(g,h)\mapsto \pi_{m,n}(g)h\sigma^{-1}(\pi_{m,n}(g)^{-1}).$$ 
By Lang's theorem, 
$$\Sht^{\loc(m,n)}_\tau\cong \bfB L^{m}_nG,$$ 
where $L^{m}_nG$ is the preimage of the discrete group $G(\mO/\varpi^n)=L^nG(k)\subset L^nG$ under the natural projection $\pi_{m,n}: L^{m}G \to L^nG$.
\end{ex}

\begin{rmk}
The readers can skip this remark, which will only be used in Remark \ref{R:Sh to Gzip}. 

As mentioned in Remark \ref{R: var and gen of loc Sht}, moduli of local shtukas can be regarded as a generalization of the moduli of $p$-divisible groups with $G$-structures. It is natural to expect that the moduli of restricted local shtukas generalize moduli of truncated Barsotti-Tate groups. It turns out that they are indeed closely related but the relation is slightly more complicated than the unrestricted case. Here we only explain the relation between $\Sht_\mu^{\loc(2,1)}$ for minuscule $\mu$ and the moduli of $1$-truncated  Barsotti-Tate groups with $G$-structures, or more precisely, the moduli of $G$-zips as defined in \cite{PWZ}. The relation between $\Sht_\mu^{\loc(m,n)}$ with the moduli of $n$-truncated Barsotti-Tate groups with $G$-structures is similar.

Let $P_\mu$ be the parabolic subgroup of $G$ generated by root subgroups $U_{\al}$ for those $\langle\al,\mu\rangle\leq 0$, and $P_{-\mu}$ the opposite parabolic subgroup. Let $U_\mu$ and $U_{-\mu}$ be the unipotent radical of $P_\mu$ and $P_{-\mu}$ respectively, and let $L_\mu=P_\mu\cap P_{-\mu}$ be the Levi subgroup.
Similar to \eqref{E: Hodge to fil}, we have a perfectly smooth $L^+G\times \bar G^\pf$-morphism of relative dimension $\dim U_\mu$
$$ \Gr_\mu^{(1)}\to ((\bar G/\bar U_\mu\times \bar G/\bar U_{-\mu})/\bar L_\mu)^\pf,\quad  g_1\varpi^\mu g_2\mapsto (\bar g_1 \mod \bar U^\pf_\mu, \ g^{-1}_2 \mod \bar U^\pf_{-\mu}) \mod \bar L^\pf_\mu,$$
where $g_1\in L^+G$, $g_2\in \bar G^\pf=L^1G$, and $\bar g_1= g_1 \mod L^+G^{(1)}$. Note that this map intertwines the action of $L^2G$ on $\Gr_{\mu}^{(1)}$ given in \eqref{E:twisted action of LmG on framed local shtukas}, and the left action of $\bar G$ on $((\bar G/\bar U_\mu\times \bar G/\bar U_{-\mu})/\bar L_\mu)^\pf$ given by $g\cdot (g_1,g_2)=(gg_1, \sigma^{-1}(g)g_2)$.
It follows that there is a perfectly smooth morphism of relative dimension $-\dim P_\mu$
\[\Sht^{\loc(2,1)}_\mu\to [\bar G^\pf\backslash(G/\bar U_\mu\times \bar G/ \bar U_{-\mu})^\pf/\bar L^\pf_\mu].\]
Note that the codomain of this map is nothing but the perfection of the moduli space $G\on{-Zip}_\mu$ of $G$-zips.
We thus obtain the following lemma.
\begin{lem}
\label{Ex:Sht and GZip}
Assume that $\mu$ is minuscule. There is a natural perfectly smooth map
\[\Sht_\mu^{\loc(2,1)}\to G\on{-Zip}^\pf_\mu.\]
of relative dimension $-\dim P_\mu$.
\end{lem}
\end{rmk}

\begin{construction}
\label{Cons:restriction morphism}
Let $(m',n')$ and $(m,n)$ be two pairs of non-negative integers such that $m\leq m'$, $n\leq n'$ are both $m'-n'$ and $m-n$ are $\mmu$-large ($(m',n')=(\infty,\infty)$ allowed).
We construct the \emph{restriction morphism}
\begin{equation}\label{trunc}
\res^{m',n'}_{m,n}: \Sht^{\loc(m',n')}_\mmu\longto \Sht_\mmu^{\loc(m,n)}
\end{equation}
as follows.
We think of (restricted) local shtukas as the fiber product of the local Hecke stack and the Frobenius graph over the classifying space of (restricted) torsors.
By \eqref{E:ex and q compatibility} and the right square of \eqref{E:restriction and torsor commutative}, the following diagram is commutative
\begin{equation}
\label{E:restriction vs torsor for both t0 and tr}
\xymatrix@C=60pt{
\Hk_\mmu^{\loc(m')} \ar[r]^-{t_{\leftone}\times t_\rightone } \ar[d]_{\res^{m'}_{m}} & \bfB L^{n'}G \times \bfB L^{m'}G \ar[d]_{\res^{n'}_{n} \times \res^{m'}_{m}} \ar[r]^{\id\times \res^{m'}_{n'}} &\bfB L^{n'}G \times \bfB L^{n'}G\ar^{\res^{n'}_n\times\res^{n'}_n}[d]
\\
\Hk_\mmu^{\loc(m)} \ar[r]^-{t_{\leftone}\times t_\rightone} & \bfB L^{n}G \times \bfB L^{m}G \ar[r]^-{\id\times \res^m_n} & \bfB L^{n}G \times \bfB L^nG.
}
\end{equation}
This means that, starting with an $R$-point $x$ of $\Hk_\mmu^{\loc(m')}$, the canonical $L^{m'}G$-torsor \eqref{torsor2} and the $L^{n'}G$-torsor \eqref{torsor1} associated to $x$ restrict to the corresponding torsors associated to the point $\res^{m'}_m(x)$.
Thus, an isomorphism $\psi:{^\sigma}(\calE_\leftone|_{D_{n',R}})  \simeq (\calE_\rightone|_{D_{m',R}})|_{D_{n',R}}$ (which will lift $x$ to an $R$-point of $\Sht_\mmu^{\loc(m',n')}$) will naturally give an isomorphism $\psi|_{D_n}: {}^\sigma(\mE_\leftone|_{D_{n,R}})\simeq (\mE_\rightone|_{D_{m,R}})|_{D_{n,R}}$ (which will lift $\res^{m'}_m(x)$ to an $R$-point of $\Sht_\mmu^{\loc(m,n)}$). This defines the needed map $\res_{m,n}^{m',n'}$.  Equivalently, we may define the map as the fiber product of the Cartesian diagram with the Frobenius graph:
\begin{align*}
\res_{m,n}^{m',n'}\colon
\Sht_\mmu^{\loc(m',n')}& \stackrel{\eqref{E:mS restricted as a Cartesian product}}\cong \Hk_\mmu^{\loc(m')}  \times_{t_{\leftone}\times \res^{m'}_{n'}\circ t_\rightone, \, \bfB L^{n'} G \times \bfB L^{n'}G,\, 1\times \sigma } \bfB L^{n'}G \\
&\xrightarrow{\res^{m'}_{m} \times \res^{n'}_{n}}
\Hk_\mmu^{\loc(m)}\times_{ t_\leftone\times \res^{m}_{n}\circ t_\rightone, \, \bfB L^n G \times \bfB L^nG,\, 1\times \sigma } \bfB L^nG \stackrel{\eqref{E:mS restricted as a Cartesian product}}\cong \Sht_\mmu^{\loc(m,n)}.
\end{align*}
Note that $\res^{m,n}_{m,0}$ is nothing but the previously introduced map $\varphi^{\loc(m,n)}$ in \eqref{E:restriction morphism}. In addition, if $(m',n')=(\infty,\infty)$, we denote
\begin{equation}
\label{E: restriction map 2}
\res_{m,n}=\res^{\infty,\infty}_{m,n}: \Sht^\loc_{\mmu}\to \Sht^{\loc(m,n)}_{\mmu}.
\end{equation}
Note that by construction, for $m_1\leq m_2\leq m_3$, $n_1\leq n_2\leq n_3$ such that $m_i-n_i$ is $\mmu$-large ($(m_3,n_3)=(\infty,\infty)$ being allowed),
\begin{equation}
\label{true-to-trun}
\res^{m_2,n_2}_{m_1,n_1}\circ\res^{m_3,n_3}_{m_2,n_2}=\res^{m_3,n_3}_{m_1,n_1}.
\end{equation}
Moreover, the restriction morphisms are compatible with the Satake correspondences, i.e., we have the following Cartesian diagram ($(m',n')=(\infty,\infty)$ allowed).
\begin{equation}
\label{E:shtukas restriction Cartesian with Satake correspondences}
\xymatrix{
\Sht_{\lambda_\bullet}^{\loc(m',n')} \ar[d]^{\res_{m,n}^{m',n'}} & \ar[l]
\Sht_{\la_\bullet|\mmu}^{0,\loc(m',n')} \ar[r] \ar[d]^{\res_{m,n}^{m',n'}} & \Sht_\mmu^{\loc(m',n')} \ar[d]^{\res_{m,n}^{m',n'}}\\		\Sht_{\la_\bullet}^{\loc(m,n)} & \ar[l] \ar[r]  \Sht_{\la_\bullet|\mmu}^{0,\loc(m,n)}&   \Sht_\mmu^{\loc(m,n)}
}
\end{equation}
\end{construction}

\begin{ex}
When $\tau\in \xcoch(Z_G)$ is a central cocharacter, the restriction map $\res_{m,n}: \Sht_\tau^\loc\cong\bfB G(\mO)\to \Sht_\tau^{\loc(m,n)}\cong\bfB L^m_nG$ is induced by the natural map $G(\mO)\to L^m_nG$.
\end{ex}

\begin{lem}
\label{L:Satake correspondence cartesian}
There is the following commutative diagram with all squares Cartesian
\begin{equation}\label{E:Satake correspondence cartesian}
\xymatrix@C=20pt{
\Sht^{0,\loc}_{\kappa_\bullet\mid \la_\bullet}\times_{\Sht^{\loc}_{\la_\bullet}}\Sht^{0,\loc}_{\la_\bullet\mid \mu_\bullet}\ar[r] \ar[d] & \Sht^{0,\loc(m,n)}_{\kappa_\bullet\mid \la_\bullet}\times_{\Sht^{\loc(m,n)}_{\la_\bullet}}\Sht^{0,\loc(m,n)}_{\la_\bullet\mid \mu_\bullet} \ar[d] \ar[r] &\Hk^{0,\loc(m)}_{\kappa_\bullet\mid \la_\bullet}\times_{\Hk^{\loc(m)}_{\la_\bullet}}\Hk^{0,\loc(m+n)}_{\la_\bullet\mid \mu_\bullet} \ar[d] \\
\Sht^{0,\loc}_{\kappa_\bullet\mid \mu_\bullet}\ar[r] & \Sht^{0,\loc(m,n)}_{\kappa_\bullet\mid \mu_\bullet}\ar[r] &  \Hk^{0,\loc(m)}_{\kappa_\bullet\mid \mu_\bullet}
}\end{equation}
\end{lem}
\begin{proof}
Clear from the definition.
\end{proof}

\subsubsection{Category of perverse sheaves on local shtukas}
\label{SS:Perv(Sloc)}
We construct $\mathrm{P}(\Sht_{\bar k}^\loc)$ the category of perverse sheaves on the moduli of local shtukas (base changed to $\bar k$), in a way similar to \S\ref{SS:Perv(Hk)}.

Let $\mmu$ is a sequence of dominant coweights. For two pairs $(m',n'), (m,n)$ of non-negative integers, satisfying $m \leq m', n\leq n'$ and both $m'-n'$ and $m-n$ are $\mmu$-large ($m'\neq \infty$), let $\Res^{m',n'}_{m,n}=(\res^{m',n'}_{m,n})^\star$ be the shifted pullback of $\ell$-adic sheaves. By \eqref{true-to-trun}, there is a canonical isomorphism of functors
\[\Res^{m',n}_{m,n}\circ\Res^{m',n'}_{m',n}\cong \Res^{m',n'}_{m,n}.\]
Same as $\Res^{m'}_m$, the functor $\Res^{m',n}_{m,n}$ induces an equivalence of categories if $m\geq 1$. On the other hand, $\res^{m',n'}_{m',n}$ is a torsor under $\Ker\big( \Aut(\calE_\rightone|_{D_{n'}}) \to \Aut(\calE_\rightone|_{D_{n}}) \big)$, which is again the perfection of a unipotent (non-constant) group scheme if $n\geq 1$, therefore, $\Res^{m',n'}_{m',n}$ is fully faithful.

Then we can define the category of perverse sheaves on $\Sht^\loc$  as the direct sum of over the filtered limits
\begin{equation}
\label{E:definition of P(Sloc)}
\mathrm{P}(\Sht_{\bar k}^\loc): = \bigoplus_{\bbzeta \in \pi_1(G)}\mathrm{P}(\Sht_{\bbzeta}^\loc),\quad  \mathrm{P}(\Sht_{\bbzeta}^\loc)=\varinjlim_{(\mu, m,n)} \mathrm{P}(\Sht_\mu^{\loc(m, n)}),
\end{equation}
where the limit is taken over the triples $\{(\mu, m, n) \in \bbzeta \times \ZZ^2_{\geq 0} \mid m-n \mbox{ is }\mu\mbox{-large}\}$, with the product partial order, and 
the connecting functor is given by the following composite of fully faithful functors
\[
\mathrm{P}(\Sht_{\mu}^{\loc(m, n)}) 
\xrightarrow{\Res^{m',n'}_{m,n}}
\mathrm{P}(\Sht_{\mu}^{\loc(m', n')})
\xrightarrow{i_{\mu,\mu',*}}
\mathrm{P}(\Sht_{\mu'}^{\loc(m' n')}), 
\]
and where $i_{\mu, \mu'}: \Sht_{\mu}^{\loc(m', n')} \to \Sht_{\mu'}^{\loc(m', n')}$ is the natural closed embedding. 
Note that compared the situation in \S\ref{SS:Perv(Hk)}, although the first functor is no longer an equivalence, these connecting functors still satisfy natural compatibility conditions given by proper or smooth base change so the limit indeed makes sense.


For each dominant coweight $\mu$ and a pair $(m,n)$ such that $m-n$ is $\mu$-large, we have a natural pullback functor
\begin{equation}
\label{E:functor P(Hk) to P(S)fin}
\Phi^{\loc(m,n)}:=\Res^{m,n}_{m,0}: \on{P}(\Hk_\mu^{\loc(m)})\to \on{P}(\Sht_\mu^{\loc(m,n)}),
\end{equation}
which commutes the above connecting morphism by \eqref{true-to-trun} and the proper smooth base change. 
Taking the limit and taking the direct sum over $\pi_1(G)$, we obtain a well-defined functor
\begin{equation}
\label{E:functor P(Hk) to P(S)}
\Phi^{\loc}: \mathrm{P}(\Hk_{\bar k}^\loc) \to \mathrm{P}(\Sht_{\bar k}^\loc).
\end{equation}

\begin{rmk}
\label{R: more object}
Note that there are many objects in $\mathrm{P}(\Sht_{\bar k}^\loc)$ that do not come from $\mathrm{P}(\Hk_{\bar k}^\loc)$ under the pullback $\Phi^{\loc}$. For example, recall from Example \ref{Ex:mStau mn} that $\on{Sht}_0^{\loc(n,n)}=\bB G(\mO/\varpi^n)$. Therefore, every representation $\rho$ of the finite group $G(\mO/\varpi^n)$ defines a local system $\mL_\rho$ on $\on{Sht}_0^{\loc(m,n)}$, and therefore an object  in $\mathrm{P}(\Sht_0^{\loc(n,n)})\to \mathrm{P}(\Sht_{\bar k}^\loc)$. This object does not lie in the essential image of $\Phi^\loc$ (as soon as $n>0$).
\end{rmk}

\begin{construction}
\label{Cons:inverse F restricted}
A crucial ingredient we need later is the existence of a ``partial Frobenius morphism" between the moduli of \emph{restricted} local iterated shtukas 
that is compatible with the $F_\mmu$ defined in \eqref{hpf} via the restriction morphism \eqref{trunc}.
In fact, for technical reasons (i.e. to apply the formalism of cohomological correspondences), we will instead construct a restricted version of the \emph{inverse} of the partial Frobenius $F_\mmu$.

We fix a sequence of dominant coweights $\mmu=(\mu_1,\ldots,\mu_t)$ as above.  A quadruple of non-negative integers $(m_1,n_1,m_2,n_2)$ is said $\mmu$-\emph{acceptable} if 
\begin{enumerate}
\item
$m_1-m_2=n_1-n_2$ are $\mu_t$-large (or equivalently $\sigma(\mu_t)$-large)\footnote{In fact it is enough to require both $m_1-m_2$ and $n_1-n_2$ are $\mu_t$-large to define $F_{\mmu}^{-1}$. But this extra generality is not needed in the paper.}, 
\item 
$m_2-n_1$ is $(\mu_1, \dots, \mu_{t-1})$-large.
\end{enumerate} 
In particular, $m_1-n_1$ is $\mmu$-large. We regard $(\infty,\infty,\infty,\infty)$ to be $\mmu$-acceptable for any $\mmu$.

Now we fix a $\mmu$-acceptable $(m_1,n_1,m_2,n_2)$ and construct a natural morphism
\begin{equation}
\label{E:truncated partial Frob}
F_\mmu^{-1}: \Sht_{\sigma(\mu_t), \mu_1, \dots, \mu_{t-1}}^{\loc(m_1,n_1)} \to \Sht_{\mu_1, \dots, \mu_t}^{\loc(m_2,n_2)},
\end{equation}
which resembles Definition~\ref{D:partial Frobenius}.
For a perfect affine scheme $S = \Spec R$, the map on the $S$-points is given as follows.
\begin{align*}
&\Sht_{\sigma(\mu_t), \mu_1, \dots, \mu_{t-1}}^{\loc(m_1,n_1)} (S)
\\
\cong\ & \big\{ (x, \psi)\; \big|\; x \in \Hk_{\sigma(\mu_t), \mu_1, \dots, \mu_{t-1}}^{\loc(m_1)} (S) \textrm{ and } \psi:{}^\sigma\calE_{x,\leftone}|_{D_{n_1,R}} \simeq \calE_{x,\rightone}|_{D_{n_1,R}} \big\} && (\textrm{Definition~\ref{tlsht}})
\\
\cong\ & \Bigg\{\! (x_1,x_2,\psi', \psi)\; \Bigg|\! \begin{array}{l} x_1 \in \Hk_{\sigma(\mu_t)}^{\loc(m_1)} (S) ,\ x_2 \in \Hk_{\mu_1, \dots, \mu_{t-1}}^{\loc(m_2)} (S),\\ \psi':\calE_{x_1, \leftone}|_{D_{m_2,R}} \simeq \calE_{x_2, \rightone}|_{D_{m_2,R}}, \textrm{ and}\\ \psi:{}^\sigma\calE_{x_2,\leftone}|_{D_{n_1,R}} \simeq \calE_{x_1,\rightone}|_{D_{n_1,R}}
\end{array}\!\! \Bigg\} && \Bigg(\!\! \begin{array}{l} \textrm{Use \S\ref{SS:break hecke}}\\ \textrm{to replace $x$}
\\
\textrm{by $(x_1,x_2, \psi')$}
\end{array}\!
\!\Bigg)
\\
\cong\ & \Bigg\{\! (x'_1,x_2,\psi', \psi'')\; \Bigg|\! \begin{array}{l} x'_1 \in \Hk_{\mu_t}^{\loc(m_1)} (S) ,\ x_2 \in \Hk_{\mu_1, \dots, \mu_{t-1}}^{\loc(m_2)} (S),\\ \psi':{}^\sigma\calE_{x'_1, \leftone}|_{D_{m_2,R}} \simeq \calE_{x_2, \rightone}|_{D_{m_2,R}},\textrm{ and}\\ \psi'':\calE_{x_2,\leftone}|_{D_{n_1,R}} \simeq \calE_{x'_1,\rightone}|_{D_{n_1,R}}
\end{array}\!\! \Bigg\} && \bigg(\!\! \begin{array}{l}x'_1 = \sigma^{-1}(x_1)\\ \psi'' = \sigma^{-1}(\psi)\end{array}\!\! \bigg)
\\
\to\ & \Bigg\{\! (x''_1,x_2,\psi''', \psi'')\; \Bigg|\! \begin{array}{l} x''_1 \in \Hk_{\mu_t}^{\loc(n_1)} (S) ,\ x_2 \in \Hk_{\mu_1, \dots, \mu_{t-1}}^{\loc(m_2)} (S),\\ \psi''':{}^\sigma\calE_{x''_1, \leftone}|_{D_{n_2,R}} \simeq \calE_{x_2, \rightone}|_{D_{n_2,R}},\textrm{ and}\\ \psi'':\calE_{x_2,\leftone}|_{D_{n_1,R}} \simeq \calE_{x''_1,\rightone}|_{D_{n_1,R}}
\end{array}\!\! \Bigg\} && \bigg(\!\! \begin{array}{l}x''_1 = \res^{m_1}_{n_1}(x'_1)\\ \psi''' = \psi'|_{D_{n_2,R}}\end{array}\!\! \bigg)
\\
\cong \ &
\bigg\{ 
(x', \psi'')\;\bigg|\!\!
\begin{array}{l}
x' \in \Hk_{\mu_1,\dots, \mu_t}^{\loc(m_2)}(S) \textrm{ and}\\
\psi''':{}^\sigma \calE_{x',\leftone}|_{D_{n_2,R}} \simeq \calE_{x',\rightone}|_{D_{n_2,R}}
\end{array}\!\! \bigg\}
&& \!\!\!\!\!\!\bigg(\!\!\! \begin{array}{l}
\textrm{Use \S\ref{SS:break hecke} to glue} \\ (x_2,x''_1, \psi'')\textrm{ into }x' \ 
\end{array}\!\!\!
\bigg)
\\
\cong \ & \Sht_{\mu_1, \dots, \mu_t}^{\loc(m_2,n_2)}(S). && (\textrm{Definition~\ref{tlsht}})
\end{align*}
Note that the condition that $m_1-m_2$ are $n_1-n_2$ are $\mu_t$-large ensured that $\psi', \psi'''$ in the previous line makes sense, namely $\mE_{x_1,\leftone}|_{D_{m_2,R}}$ and $\calE_{x''_1, \leftone}|_{D_{n_2,R}}$ is canonically defined; the condition that $m_2-n_1$ is $(\mu_1, \dots, \mu_t)$-large guarantees that the gluing in the last step is valid.

Similar to the construction above, for dominant coweights $\lambda_\bullet,\mmu,\nu$ and  $(m_1,n_1,m_2,n_2)$ which is $\lambda_\bullet$-acceptable and $\mmu$-acceptable,
there is a natural inverse partial Frobenius morphism
\[
F^{-1}_{\la_\bullet\mid \mmu;\nu}\colon
\Sht^{0,\loc(m_1,n_1)}_{\sigma(\nu);\la_\bullet\mid \mmu} \to \Sht^{0,\loc(m_2,n_2)}_{\la_\bullet\mid \mmu;\nu} 
\]
\end{construction}

\begin{lem} 
The map $F_{\mmu}^{-1}$ is equidimensionally perfectly smooth, of relative dimension $0$.\footnote{Note that a smooth morphism between stacks of relative dimension $0$ need not be \'etale.}
\end{lem}
\label{L: rel dim pFrob}
\begin{proof}
The only non-isomorphic arrow is given by a base change of the map $\res^{m_1}_{n_1}$ (which gives a gerbe over a twist of $L^{m_1-n_1}G^{(n_1)}$) and a restriction of an isomorphism of torsors (which gives a torsor over a twist of $L^{m_2-n_2} G^{(n_2)}$).
\end{proof}

\begin{lem}
\label{L:property of res partial Frob}
\begin{itemize}
\item[(1)]
The following diagram is commutative (but not a Cartesian)
\begin{equation}\label{E:pf as pf}
\xymatrix@C=40pt{
\Sht_{\sigma(\mu_t), \mu_1,\ldots,\mu_{t-1}}^{\loc(m_1,n_1)} \ar[d]_{F_\mmu^{-1}} \ar[rr]^-{\varphi^{\loc(m_1, n_1)}} && \Hk_{\sigma(\mu_t), \mu_1,\ldots,\mu_{t-1}}^{\loc(m_1)}\ar[r]^-{\eqref{break chain}} & \Hk_{\sigma(\mu_t)}^{\loc(n_1)}\times \Hk_{\mu_1,\ldots,\mu_{t-1}}^{\loc(m_2)}
\\
\Sht_{\mu_1,\ldots,\mu_{t}}^{\loc(m_2,n_2)}\ar[rr]^-{\varphi^{\loc(m_2, n_2)}} && \Hk_{\mu_1,\ldots,\mu_{t}}^{\loc(m_2)} \ar[r]^-{\eqref{break chain}} & \Hk_{\mu_t}^{\loc(n_1)}\times  \Hk_{\mu_1,\ldots,\mu_{t-1}}^{\loc(m_2)} \ar[u]_{ \sigma \times \id}.
}
\end{equation}
\item[(2)]
Let $(m'_1,n'_1,m'_2,n'_2)\geq (m_1,n_1,m_2,n_2)$ be two $\mmu$-acceptable quadruples ($(m'_1,n'_1,m'_2,n'_2)=(\infty,\infty,\infty,\infty)$ being allowed),
the following diagram is commutative (but not Cartesian in general)
\begin{equation}
\label{E: pFrob v.s. res. pFrob}
\xymatrix@C=40pt{
\Sht_{\sigma(\mu_t), \mu_1,\ldots,\mu_{t-1}}^{\loc(m'_1,n'_1)} \ar[r]^-{F_\mmu^{-1}}
\ar[d]_{\res^{m'_1,n'_1}_{m_1,n_1}}
 & \Sht_{\mu_1,\ldots,\mu_t}^{\loc(m'_2,n'_2)} \ar[d]^{\res^{m'_2,n'_2}_{m_2,n_2}}\\
\Sht_{\sigma(\mu_t), \mu_1,\ldots,\mu_{t-1}}^{\loc(m_1,n_1)} \ar[r]^-{F_\mmu^{-1}} & \Sht_{\mu_1,\ldots,\mu_{t}}^{\loc(m_2,n_2)}.
}
\end{equation}
\end{itemize}
\end{lem}
\begin{proof}
Both follow by the construction.
\end{proof}

\begin{remark}
\label{R: group const of pFrob}
We reinterpret the above construction in terms of group elements. 
Pick a (non-increasing) sequence of integers $(m_0 = m', m_1 = m, m_2, \dots, m_{t-1},m_t =n', m_{t+1}= n)$ such that $m_i-m_{i+1}$ is $\mu_i$-large for $i=1, \dots, t-1$. (This is possible because $m-n'$ is $(\mu_1, \dots, \mu_{t-1})$-large.)
 
Consider the action of $\prod_{i=1}^{t+1} L^{m_i}G$ on $\prod_{i=1}^t \Gr_{\mu_i}^{(m_{i+1})}$ given by  
\begin{multline}\label{Hecke framing}
(g_1,\ldots,g_t,g_{t+1})\cdot (h_1L^+G^{(m_2)},h_2L^+G^{(m_3)},\ldots,h_tL^+G^{(m_{t+1})})\\
=(g_1h_1g_2^{-1}L^+G^{(m_2)},g_2h_2g_3^{-1}L^+G^{(m_3)},\ldots, g_th_tg_{t+1}^{-1}L^+G^{(m_{t+1})}).
\end{multline}
The space $\Gr_\mmu^{(m_{t+1})}$ is the quotient of $\prod_{i=1}^t \Gr_{\mu_i}^{(m_{i+1})}$ by the subgroup $\prod_{i=2}^t L^{m_i}G\subseteq \prod_{i=1}^{t+1} L^{m_i}G$.
To get the moduli of truncated local shtukas, we consider an embedding $c_{\sigma^{-1}}^{\bullet}: \prod_{i=1}^t L^{m_i}G \to \prod_{i=1}^{t+1} L^{m_i}G$ given by
\[c_{\sigma^{-1}}^{\bullet}(g_1,\ldots,g_t)=(g_1,\ldots,g_t, \sigma^{-1}(\pi_{m_1,m_{t+1}}(g_1))).\]
Then we have (compare \S\ref{SS:description of mS in terms of group elements})
\[\Sht_\mmu^{\loc(m_1,m_{t+1})}\cong \Big[ \prod_{i=1}^t \Gr_{\mu_i}^{(m_{i+1})} \Big/ c_{\sigma^{-1}}^\bullet\Big(\prod_{i=1}^t L^{m_i}G \Big)   \Big].\]

Now, it is straightforward to check that the following pair of morphisms between products of Grassmannians and groups are equivariant for the corresponding actions:
\[
\xymatrix@R=5pt{
(g_0,\ldots,g_{t-1}) \ar@{|->}[r] & (g_1,\ldots,g_{t-1}, \sigma^{-1}(\pi_{m_0,m_t}(g_0)))
\\
\prod_{i=0}^{t-1} L^{m_i}G \ar[r] & \prod_{i=1}^t L^{m_i}G
\\
\\
\ar@(ur,ul)[]^{\qquad\qquad c_{\sigma^{-1}}^{\bullet}}&\ar@(ur,ul)[]^{\qquad\qquad c_{\sigma^{-1}}^{\bullet}}&
\\
\Gr_{\sigma(\mu_t)}^{(m_1)} \times \prod_{i=1}^{t-1} \Gr_{\mu_i}^{(m_{i+1})}\ar[r]  & \prod_{i=1}^{t}\Gr_{\mu_i}^{(m_{i+1})}
\\
(h_0L^+G^{(m_1)}, h_1L^+G^{(m_2)},\dots,h_{t-1}L^+G^{(m_t)})\ar@{|->}[r] & (h_1L^+G^{(m_2)},\dots,h_{t-1}L^+G^{(m_t)},\sigma^{-1}(h_0)L^+G^{(m_{t+1})}).
}
\]
Taking this quotients by the equivariant action gives the inverse partial Frobenius morphism
\[
F^{-1}_\mmu: \Sht_{\sigma(\mu_t),\mu_1,\ldots,\mu_{t-1}}^{\loc(m_0,m_t)}\to \Sht_{\mu_1,\ldots,\mu_{t}}^{\loc(m_1,m_{t+1})},
\] which is canonically independent of the choice of numbers $m_i$.
\end{remark}

Lemma \ref{decomp1} suggests that we can define a restricted version of $\Sht_{\la_\bullet|\mmu}^{\nu,\loc}$ to be certain fiber product (because a direct definition is unavailable). Here is the precise definition.
\begin{definition}
\label{D:mS lambda mu truncated}
Let $\la_\bullet$, $\mmu$ be two sequences of dominant coweights, and $\nu$ a dominant coweight.
Let $(m_1,n_1,m_2,n_2)$ be a quadruple that is $(|\la_\bullet|+\nu, \nu)$-acceptable and $(|\mmu|+\nu,\nu)$-acceptable.
We choose a dominant coweight $\eta $ such that
\begin{itemize}
\item[(a)]
either $\eta\succeq |\la_\bullet|+\sigma(\nu)$ or $\eta\succeq |\mmu|+\nu$, and
\item[(b)]  $m_2-n_1$ is $\eta$-large.
\end{itemize}
In particular, Condition (b) above and Condition (1) in Construction \ref{Cons:inverse F restricted} imply that $m_1-n_1=m_2-n_2$ is $(\sigma(\nu^*),\eta)$-large and $(\eta, \nu^*)$-large.
For example, we may take $\eta = |\la_\bullet|+\sigma(\nu)$ or $|\mmu|+\nu$.

We define the \emph{Hecke correspondence for restricted local shtukas} $\Sht_{\lambda_\bullet|\mmu}^{\nu,\loc(m_1,n_1)}$ so that the pentagon in the following diagram is Cartesian (when composing the left vertical arrow with $F^{-1}_{\eta, \nu^*}$)
\begin{equation}\label{deompf1}
\xymatrix@C=15pt{
&&
\Sht_{\lambda_\bullet|\mmu}^{\nu,\loc(m_1,n_1)}  \ar[dl] \ar[dr]\\
&  \Sht_{\lambda_\bullet|(\sigma(\nu^*),\eta)}^{0,\loc(m_1,n_1)} \ar[dl] \ar[d]  && \Sht_{(\eta, \nu^*)|\mmu}^{0,\loc(m_2,n_2)} \ar[d] \ar[dr]
\\
\Sht_{\lambda_\bullet}^{\loc(m_1,n_1)} & \Sht_{\sigma(\nu^*),\eta}^{\loc(m_1,n_1)} \ar[rr]^-{F^{-1}_{\eta, \nu^*}} && \Sht_{\eta, \nu^*}^{\loc(m_2,n_2)} & \Sht_{\mmu}^{\loc(m_2,n_2)}.
}
\end{equation}
The algebraic stack $\Sht_{\lambda_\bullet|\mmu}^{\nu,\loc(m_1,n_1)} $ is independent of the choice of the coweight $\eta$ (as long as it satisfies conditions (a) and (b)). 
In addition, the following Lemma \ref{corr trun2} (see also Lemma \ref{L: torsor over S-corr}) shows that it is also independent of $(m,n)$ (as soon as the conditions (1) and (2) are met), which justifies our notation. 
\end{definition}

The relation between  $\Sht_{\la_\bullet|\mmu}^{\nu,\loc(m_1,n_1)}$ as $(m_1,n_1,m_2,n_2)$ varies is as follows.
\begin{lem}\label{corr trun2}
Assume that $(m'_1,n'_1,m'_2,n'_2)\geq (m_1,n_1,m_2,n_2)$ are two quadruples that are both $(|\la_\bullet|+\nu,\nu)$-acceptable and $(|\mmu|+\nu,\nu)$-acceptable ($(m'_1,n'_1,m'_2,n'_2)=(\infty,\infty,\infty,\infty)$ being allowed).
The following diagram is commutative. Moreover, all squares are Cartesian except the one marked with the letter \emph{``X"}.
\begin{equation}
\label{E:cycle construction diagram full}
\xymatrix@C=17pt{
\Sht_{\lambda_\bullet}^{\loc(m'_1,n'_1)}\ar_{\res^{m'_1,n'_1}_{m_1,n_1}}[d]& \Sht_{\lambda_\bullet|(\sigma(\nu^*), \eta)}^{0,\loc(m'_1,n''_1)}\ar[l]\ar[d]
&\Sht_{\lambda_\bullet|\mmu}^{\nu,\loc(m'_1,n'_1)}\ar_{\res^{m'_1,n'_1}_{m_1,n_1}}[d]\ar[r]\ar[l]\ar@{}[dr]|{\mathrm{X}}&\Sht_{(\eta,\nu^*)|\mmu}^{0,\loc(m'_2,n'_2)} \ar[d]\ar[r]& \Sht_\mmu^{\loc(m'_2,n'_2)}\ar^{\res^{m'_2,n'_2}_{m_2,n_2}}[d]
\\
\Sht_{\lambda_\bullet}^{\loc(m_1,n_1)}&\Sht^{0,\loc(m_1,n_1)}_{\lambda_\bullet|(\sigma(\nu^*), \eta)}\ar[l]&\Sht_{\lambda_\bullet|\mmu}^{\nu,\loc(m_1,n_1)}\ar[l]\ar[r]&\Sht_{(\eta,\nu^*)|\mmu}^{0,\loc(m_2,n_2)}\ar[r]&\Sht_\mmu^{\loc(m_2,n_2)}.
}
\end{equation}
\end{lem}
\begin{proof}
The commutativity of the diagram is clear. The left square and the right square are Cartesian by \eqref{E:shtukas restriction Cartesian with Satake correspondences}.
That the second square is also Cartesian follows from a purely formal argument. To reduce the load of notations, we put
\[
A = \Sht_{\lambda_\bullet|\mmu}^{\nu,\loc(m_1,n_1)},\ B = \Sht^{0,\loc(m_1,n_1)}_{\lambda_\bullet|(\sigma(\nu^*), \eta)}, \ C =  \Sht_{(\eta,\nu^*)|\mmu}^{0,\loc(m_2,n_2)}, \textrm{ and }D= \Sht_{\eta,\nu^*}^{\loc(m_2,n_2)}.
\]
Similar, we write $A'$, $B'$, $C'$, and $D'$ when we replace $(m_1,n_1,m_2,n_2)$ by $(m'_1,n'_1,m'_2,n'_2)$.
Then in the following commutative diagram
\[
\xymatrix@R=10pt@C=5pt{
& A' \ar[rr]\ar'[d][dd]\ar[dl]
& & B' \ar[dd]^(.45){\!F^{-1}}\ar[dl]
\\
A \ar[rr]\ar[dd]
& & B \ar[dd]^(.3){ \!F^{-1}}
\\
& C' \ar'[r][rr]\ar[dl]
& & D'\ar[dl]
\\
C \ar[rr]
& & D 
}
\]
the front and the back faces are Cartesian by Definition~\ref{D:mS lambda mu truncated} (or by Proposition~\ref{decomp1} if $(m'_1,n'_1,m'_2,n'_2)=(\infty,\infty,\infty,\infty)$), and the bottom face is Cartesian by \eqref{E:shtukas restriction Cartesian with Satake correspondences}. (But note that the right face involving the two partial Frobenius morphisms is not Cartesian.) Now the lemma follows from Lemma \ref{L: cartesian}.
\end{proof}

\subsubsection{Explicit description of $\Sht_{\la_\bullet|\mmu}^{\nu,\loc(m_1,n_1)}$}
\label{SS: exp des of res hk}
Let us analyze the morphism 
\begin{equation}
\label{C:properness of cycles}
\overleftarrow{h}^{\loc(m_1,n_1)}_{\la_\bullet}:\Sht_{\lambda_\bullet|\mmu}^{\nu, \loc(m_1,n_1)} \to \Sht_{\lambda_\bullet}^{\loc(m_1,n_1)},
\end{equation}
parallel to the discussions in \S \ref{SS:relation of Hecke of mS and ADLV}. First, since both morphisms 
$$\Sht^{0,\loc(m_1,n_1)}_{\lambda_\bullet|(\sigma(\nu^*), \eta)} \to \Sht_{\lambda_\bullet}^{\loc(m_1,n_1)},\quad \mbox{and} \quad \Sht^{0,\loc(m_2,n_2)}_{(\eta,\nu^*)|\mmu} \to \Sht^{\loc(m_2,n_2)}_{\eta,\nu^*}$$ are perfectly proper (by the comments after \eqref{E:S and Hk Satake compatibility}),  $\overleftarrow{h}^{\loc(m_1,n_1)}_{\la_\bullet}$ is perfectly proper.  This is the restricted counterpart of Lemma \ref{L: Rep of Hk of Sht}.

\begin{remark}
However, the morphism 
$\Sht_{\lambda_\bullet|\mmu}^{\nu, \loc(m_1,n_1)} \to \Sht_{\mu_\bullet}^{\loc(m_2,n_2)}$ is not perfectly proper (because the restricted inverse partial Frobenius is not perfectly proper). This is the reason that there is an asymmetry in the diagram \eqref{E:cycle construction diagram full} above.
\end{remark}

Note that since $\Sht_{\la_\bullet|\mmu}^{\nu,\loc(m_1,n_1)}=\Sht_{\la_\bullet}^{\loc(m_1,n_1)}\times_{\Sht_{|\la_\bullet|}^{\loc(m_1,n_1)}}\Sht_{|\la_\bullet|||\mmu|}^{\nu,\loc(m_1,n_1)}\times_{\Sht_{|\mmu|}^{\loc(m,n)}}\Sht_{\mmu}^{\loc(m_2,n_2)}$, to simplify notations, we assume that $\la_\bullet=\la$ and $\mmu=\mu$ are single coweights in the following lemma. Recall the $L^{m_1}G$-torsor $\Sht_\la^{\loc(m_1,n_1),\Box}\simeq \Gr_\la^{(n_1)}\to \Sht_\la^{\loc(m_1,n_1)}$ from \eqref{E:framed local shtukas}. Let 
$$\Sht^{\nu,\loc(m_1,n_1),\Box}_{\la\mid\mu}:=\Sht^{\nu,\loc(m_1,n_1)}_{\la\mid\mu}\times_{\Sht_\la^{\loc(m_1,n_1)}}\Sht_{\la}^{\loc(m_1,n_1),\Box}$$ 
be the pullback of the $L^{m_1}G$-torsor to $\Sht^{\nu,\loc(m_1,n_1)}_{\la\mid\mu}$.

\begin{lem}
\label{L: torsor over S-corr}
There is a natural $L^{m_1}G$-equivariant isomorphism
\begin{equation}\label{E: torsor over S-corr}
\Sht^{\nu,\loc(m_1,n_1),\Box}_{\la\mid\mu}\cong\big\{(g_1L^+G^{(n_1)},g_2L^+G)\in \Gr_\la^{(n_1)}\times \Gr_{\sigma(\nu^*)}\mid g_2^{-1}g_1\sigma^{-1}(g_2)\in [L^+G\backslash\Gr_{\mu}^{(\infty)} /L^+G]\big\}
\end{equation}
such that $\overleftarrow{h}^{\loc(m_1,n_1),\Box}_{\la}: \Sht^{\nu,\loc(m_1,n_1),\Box}_{\la\mid\mu}\to \Sht_\la^{\loc(m_1,n_1),\Box}$ is identified with the projection to the first factor,
where $L^{m_1}G$ acts on $\Gr_\la^{(n_1)}\times \Gr_{\sigma(\nu^*)}$ as follows: it acts on  $\Gr_\la^{(n_1)}$ as in \eqref{E:twisted action of LmG on framed local shtukas}, and on  $\Gr_{\sigma(\nu^*)}$ by the usual left multiplication.
\end{lem}
\begin{proof}
We may let $\eta=\la+\sigma(\nu)$ in the definition of $\Sht^{\nu,\loc(m_1,n_1)}_{\la\mid\mu}$.
First note that $\Sht^{\nu,\loc(m_1,n_1)}_{\la\mid\mu}\to \Sht_{\la|(\sigma(\nu^*),\la+\sigma(\nu))}^{0,\loc(m_1,n_1)}$ is a closed embedding (since $\Sht_{(\lambda+\sigma(\nu), \nu^*)|\mu}^{0,\loc(m_2,n_2)}\to \Sht_{\lambda+\sigma(\nu), \nu^*}^{\loc(m_2,n_2)}$ is a closed embedding).

Next, by the discussion from \S \ref{SS:description of mS in terms of group elements}, over $\Sht_\la^{\loc(m_1,n_1)}\leftarrow\Sht_{\la|(\sigma(\nu^*),\la+\sigma(\nu))}^{0,\loc(m_1,n_1)}\to \Sht_{\sigma(\nu^*),\la+\sigma(\nu))}^{\loc(m_1,n_1)}$, there are $L^{m_1}G$-torsors
\[\Gr_\la^{(n_1)}\leftarrow \Gr_{\la|(\sigma(\nu^*),\la+\sigma(\nu))}^{0,(n_1)}\to \Gr_{\sigma(\nu^*),\la+\sigma(\nu)}^{(n_1)}.\]
In addition, there is a canonical isomorphism 
$$ \Gr_\la^{(n_1)}\times \Gr_{\sigma(\nu^*)}\cong \Gr_{\la|(\sigma(\nu^*),\la+\sigma(\nu))}^{0,(n_1)}, $$
$$(\beta_1: \mE_1\dashrightarrow \mE^0, \epsilon_{n_1}: \mE_1|_{D_{n_1}}\simeq \mE^0|_{D_{n_1}},\beta_2: \mE_2\dashrightarrow \mE^0)\to (\xymatrix{
\mE_1 \ar@/^10pt/@{-->}[rr]^-{\beta_1} \ar_{\beta_2^{-1}\beta_1}@{-->}[r] & \mE_2 \ar_{\beta_2}@{-->}[r]& \mE^0}, \epsilon_{n_1}).$$
It is clear that under this isomorphism, the action of $L^{m_1}G$ on the right hand side is identified with the action on the left hand side as described in the lemma. Now the lemma follows from the construction of $F^{-1}_{\la+\sigma(\nu), \nu^*}$ (in particular in terms of group elements as in Remark \ref{R: group const of pFrob}).
\end{proof}

\begin{cor}
\label{C:expclit mStautau truncated}
Assume that $\la=\tau\in\xcoch(Z_G)$ is central.
There is a canonical isomorphism
\[
\Sht_{\tau|\mu}^{\nu,\loc(m_1,n_1)} \cong \big[ L_{n_1}^{m_1}G \backslash X_{\mu^*,\nu^*}(\tau^*)\big],
\]
where $L^{m_1}_{n_1}G$ is the group defined in Example~\ref{Ex:mStau mn}. Moreover, if $(m'_1,n'_1,m'_2,n'_2)\geq (m_1,n_1,m_2,n_2)$ is another $(\tau+\nu,\nu)$-acceptable and $(\mu+\nu,\nu)$-acceptable quadruple, we have the following natural Cartesian diagram.
\[
\xymatrix{
\big[L^{m'_1}_{n'_1}G  \backslash \mathrm{pt}\big] \cong \Sht_\tau^{\loc(m'_1,n'_1)} \ar@<35pt>[d] & \ar[l] \Sht_{\tau|\mu}^{\nu,\loc(m'_1,n'_1)}  \cong \big[L^{m'_1}_{n'_1}G \backslash X_{\mu^*,\nu^*}(\tau^*)\big] \ar@<-50pt>[d]\\
\big[L^{m_1}_{n_1}G \backslash \mathrm{pt}\big] \cong\Sht_{\tau}^{\loc(m_1,n_1)} & \ar[l] \Sht_{\tau|\mu}^{\nu, \loc(m_1,n_1)} \cong \big[L^{m_1}_{n_1}G \backslash X_{\mu^*,\nu^*}(\tau^*) \big],
}
\]
where the action of $L^{m_1}_{n_1}G$ on $X_{\mu^*,\nu^*}(\tau^*)$ is the $\sigma^{-1}$-twist of the usual left multiplication. Here $(m'_1,n'_1,m'_2,n'_2)=(\infty,\infty,\infty,\infty)$ is allowed, in which case $L^{\infty}_{\infty}G$ is interpreted as $G(\mO)$, and the isomorphism $[G(\mO)\backslash \on{pt}]\cong\Sht^\loc_{\tau}$ comes from \eqref{rigid corr}.
\end{cor}
\begin{proof} Note that in \eqref{E: torsor over S-corr}, 
\begin{multline*}
\big\{(g_1L^+G^{(n_1)},g_2L^+G)\in \Gr_\la^{(n_1)}\times \Gr_{\sigma(\nu^*)}\mid g_2^{-1}g_1\sigma^{-1}(g_2)\in [L^+G\backslash\Gr_{\mu}^{(\infty)} /L^+G]\big\} \cong \\
\big\{(g_1L^+G^{(n_1)}, g_2L^+G)\in \Gr_\la^{(n_1)}\times \Gr_{\nu^*}\mid g_2^{-1}g_1^{-1}\sigma(g_2)\in  [L^+G\backslash\Gr_{\mu^*}^{(\infty)} /L^+G]\big\},
\end{multline*}
where the isomorphism is induced by $\sigma: \Gr_{\nu^*}\simeq\Gr_{\sigma(\nu^*)}$. The corollary follows.
\end{proof}

The final task of this subsection is to \emph{define} the composition of Hecke correspondences between restricted local shtukas, analogous to Lemma~\ref{L:decomposition of Comp}. To slightly reduce the load of notation, we focus on the case that we will need later, namely when $\kappa_\bullet$, $\la_\bullet$ and $\mu_\bullet$ are single coweights. The general case can be treated in a similar way.
We first state the restricted versions of Lemmas~\ref{decomp2}, \ref{decomp3}, and 
\ref{L: comm Sat and pFrob}.

\begin{lem}
\label{L:technical lemma for comp loc restricted}
\begin{enumerate}
\item
Let $\lambda,\mu,\nu,\xi,\theta$ be dominant coweights such that $\theta\succeq \la+\mu^*$ or $\theta\succeq \xi+\nu^*$, and let $(m,n)$ be a pair of non-negative integers such that $m-n$ is $(\mu,\theta, \nu)$-large, $(\la,\nu)$-large and $(\mu,\xi)$-large. Then there is a canonical isomorphism
\begin{equation}
\label{E:decomp2 restricted}
\Sht^{0,\loc(m,n)}_{(\la,\nu)\mid(\mu,\xi)}
\cong \Sht^{0,\loc(m,n)}_{\la \mid (\mu,\theta);\nu}
\times_{\Sht^{\loc(m,n)}_{\mu,\theta,\nu}}
\Sht^{0,\loc(m,n)}_{\mu; (\theta,\nu)\mid \xi}.
\end{equation}
Moreover, the isomorphism obtained by replacing $(m,n)$ by $(m',n')\geq (m,n)$ is the pullback of this isomorphism via the restriction map $\res^{m',n'}_{m,n}$. In particular  \eqref{E:decomp2} is the pull back of \eqref{E:decomp2 restricted} via $\res_{m,n}$.

\item Let $\lambda, \mu,\nu,\zeta, \zeta_1,\zeta_2$ be dominant coweights such that $\zeta\succeq \la+\mu^* + \nu^*$ or $\zeta\succeq \zeta_1+ \zeta_2$, and let $(m,n)$ be a pair of non-negative integers such that $m-n$ is  $\lambda$-large, $(\mu,\zeta,\nu)$-large, and $(\mu,\zeta_1, \zeta_2,\nu)$-large, then
\begin{equation}
\label{E:decomp3 restricted}
\Sht^{0,\loc(m,n)}_{\la\mid (\mu,\zeta, \nu)}
\times_{\Sht^{\loc(m,n)}_{\mu,\zeta, \nu}}
\Sht^{0,\loc(m,n)}_{\mu; \zeta\mid (\zeta_1,\zeta_2);\nu}
\cong \Sht^{0,\loc(m,n)}_{\la_\bullet\mid(\mu,\zeta_1, \zeta_2, \nu)}.\end{equation}
Moreover, the isomorphism obtained by replacing $(m,n)$ by $(m',n')\geq (m,n)$ is the pullback of this isomorphism via the restriction map $\res^{m',n'}_{m,n}$. In particular  \eqref{E:decomp3} is the pull back of \eqref{E:decomp3 restricted} via $\res_{m,n}$.

\item Let $\lambda_\bullet, \mmu, \nu$ be dominant coweights and let $(m_1,n_1,m_2,n_2)$ be a quadruple of nonnegative integers that is $(|\la_\bullet|+\nu,\nu)$-acceptable and $(|\mmu|+\nu,\nu)$-acceptable
The following natural diagram is Cartesian
\begin{equation}
\label{E:decomp4 restricted}
\xymatrix@C=40pt{
\Sht^{0,\loc(m_1,n_1)}_{\sigma(\nu);\la_\bullet\mid \mmu} \ar[r]^-{F^{-1}_{\la_\bullet\mid \mmu;\nu}} \ar[d] & \Sht^{0,\loc(m_2,n_2)}_{\la_\bullet\mid \mmu;\nu} \ar[d]\\
\Sht^{\loc(m_1,n_1)}_{\sigma(\nu),\la_\bullet} \ar[r]^-{F^{-1}_{\lambda_\bullet,\nu}} & \Sht^{\loc(m_2,n_2)}_{\la_\bullet, \nu}.
}
\end{equation}
Moreover, the diagram obtained by replacing $(m_1,n_1,m_2,n_2)$ by $(m'_1,n'_1,m'_2,n'_2)$ is compatible with  \eqref{E:decomp4 restricted} via the restriction map $\res^{m'_1,n'_1}_{m_1,n_1}$ and $\res^{m'_2,n'_2}_{m_2,n_2}$. In particular,  \eqref{E:decomp4 restricted} is compatible with \eqref{E:comm Sat and pFrob} via $\res_{m_1,n_1}$ and $\res_{m_2,n_2}$.
\end{enumerate}
\end{lem}
\begin{proof}
(1) and (2) essentially follow from the same proof of Lemmas~\ref{decomp2} and \ref{decomp3}, respectively. More precisely, the proofs work to give corresponding statements when $\Sht^\loc$ is replaced by $\Gr$ (with appropriate decorations). Taking the quotient by $L^mG$ under the $\sigma$-conjugacy action as in \S \ref{SS:description of mS in terms of group elements} gives (1) and (2). (3) is clear from the definition.
\end{proof}

\begin{proposition}
\label{P:pushforward is local}
Let $\kappa, \lambda,\mu, \nu,\nu'$ be dominant coweights and let $(m_1,n_1,m_2,n_2,m_3,n_3)$ be a sextuple of non-negative integers satisfying
\begin{itemize}
\item $(m_1,n_1,m_2,n_2)$ is $(\kappa+\nu',\nu')$-acceptable and $(\lambda+\nu',\nu')$-acceptable;
\item $(m_2,n_2,m_3,n_3)$ is $(\la+\nu,\nu)$-acceptable and $(\mu+\nu,\nu)$-acceptable.
\end{itemize}
The following series of maps define a natural perfectly proper morphism
\begin{equation}
\label{E:composition restricted}
\on{Comp}_{\kappa, \lambda, \mu}^{\mathrm{loc}(m_1,n_1)}: \Sht_{\kappa|\la}^{\nu',\loc(m_1,n_1)} \times_{\Sht_{\la}^{\loc(m_2,n_2)}} \Sht_{\la|\mu}^{\nu,\loc(m_2,n_2)} \longto \Sht_{\kappa|\mu}^{\nu+\nu',\loc (m_1,n_1)}.
\end{equation}
\begin{small}
\begin{align*}
& \ \Sht^{\nu',\loc(m_1,n_1)}_{\kappa\mid\lambda}\times_{\Sht^{\loc(m_1,n_1)}_{\lambda}}\Sht^{\nu,\loc(m_2,n_2)}_{\lambda\mid \mu}   
\\
\cong \ &\  \Sht^{0,\loc(m_1,n_1)}_{\kappa\mid (\sigma(\nu'^*),\eta')}\times_{\Sht^{\loc(m_2,n_2)}_{\eta',\nu'^*}}\Sht^{0,\loc(m_2,n_2)}_{(\eta',\nu'^*)\mid \la}\times_{\Sht^{\loc(m_2,n_2)}_{\lambda}}\Sht^{0,\loc(m_2,n_2)}_{\lambda\mid(\sigma(\nu^*), \eta)}\times_{\Sht^{\loc(m_3,n_3)}_{\eta,\nu^*}}\Sht^{0,\loc(m_3,n_3)}_{(\eta,\nu^*)\mid \mu} & \!\!\!\!\!\!\!\!\!\! (\mathrm{Defn}\ \ref{D:mS lambda mu truncated})
\\
\xrightarrow{(1)}  &\ \Sht^{0,\loc(m_1,n_1)}_{\kappa\mid (\sigma(\nu'^*),\eta')}\times_{\Sht^{\loc(m_2,n_2)}_{\eta',\nu'^*}}\Sht^{0,\loc(m_2,n_2)}_{(\eta',\nu'^*)\mid(\sigma(\nu^*), \eta)}\times_{\Sht^{\loc(m_3,n_3)}_{\eta,\nu^*}}\Sht^{0,\loc(m_3,n_3)}_{(\eta,\nu^*)\mid \mu}   
\\
\cong \ &\      \Sht^{0,\loc(m_1,n_1)}_{\kappa\mid (\sigma(\nu'^*),\eta')}
\times_{\Sht^{\loc(m_2,n_2)}_{\eta',\nu'^*}}
\Sht^{0,\loc(m_2,n_2)}_{\eta'\mid (\sigma(\nu^*),\theta);\nu'^*}
\times_{\Sht^{\loc(m_2,n_2)}_{\sigma(\nu^*),\theta,\nu'^*}}
\Sht^{0,\loc(m_2,n_2)}_{\sigma(\nu^*); (\theta,\nu'^*)\mid\eta}
\times_{\Sht^{\loc(m_3,n_3)}_{\eta,\nu^*}}\Sht^{0,\loc(m_3,n_3)}_{(\eta,\nu^*)\mid \mu} & \eqref{E:decomp2 restricted}
\\
\cong    \  & \  \Sht^{0,\loc(m_1,n_1)}_{\kappa\mid (\sigma(\nu'^*),\eta')}\times_{\Sht^{\loc(m_1,n_1)}_{\sigma(\nu'^*), \eta'}}
\Sht^{0,\loc(m_1,n_1)}_{\sigma(\nu'^*);\eta'\mid (\sigma(\nu^*),\theta)}
\times_{\Sht^{\loc(m_3,n_3)}_{\theta,\nu'^*,\nu^*}}
\Sht^{0,\loc(m_3,n_3)}_{ (\theta,\nu'^*)\mid\eta;\nu^*} \times_{\Sht^{\loc(m_3,n_3)}_{\eta,\nu^*}}\Sht^{0,\loc(m_3,n_3)}_{(\eta,\nu^*)\mid \mu}   & \eqref{E:decomp4 restricted}
\\
\cong \ & \ \Sht^{0,\loc(m_1,n_1)}_{\kappa \mid(\sigma(\nu'^*),\sigma(\nu^*) , \theta)} \times_{\Sht^{\loc(m_3,n_3)}_{\theta,\nu'^*,\nu^*}} \Sht^{0,\loc(m_3,n_3)}_{(\theta,\nu'^*,\nu^*)\mid \mu} & \eqref{E:decomp3 restricted}  
\\
 \xrightarrow{(2)} & \ \Sht^{0,\loc(m_1,n_1)}_{\kappa \mid (\sigma(\nu^*+\nu'^*), \theta )}\times_{\Sht^{\loc(m_3,n_3)}_{\theta,\nu^*+\nu'^*}}\Sht^{0,\loc(m_3,n_3)}_{(\theta,\nu^*+\nu'^*)\mid \mu} 
 \\
\cong \ & \Sht^{\nu+\nu',\loc(m_1,n_1)}_{\kappa\mid \mu }   &\!\!\!\!\!\!\!\!\!\! (\mathrm{Defn}~\ref{D:mS lambda mu truncated})
\end{align*}
\end{small}Here we choose $\eta'\succeq \kappa + \sigma(\nu')$, $\eta\succeq \la+\sigma(\nu)$, and $\theta\succeq \eta'+\sigma(\nu)$ such that $m_2-n_1$ is $\eta'$-large, $m_3-n_2$ is $\eta$-large, and $m_3-n_1$ is $\theta$-large.
The map $(1)$ is induced by the composition of Satake correspondence
\[\Sht^{0,\loc(m_2,n_2)}_{(\eta',\nu'^*)\mid \lambda}\times_{\Sht^{\loc(m_2,n_2)}_{\lambda}}\Sht^{0,\loc(m_2,n_2)}_{\lambda\mid(\sigma(\nu^*), \eta)} \to \Sht^{0,\loc(m_2,n_2)}_{(\eta',\nu'^*)\mid(\sigma(\nu^*), \eta)},\]
and the map $(2)$ is induced by the convolution map of moduli of local shtukas \eqref{E: conv local Sht}
\begin{small}
\[
\Sht^{0,\loc(m_1,n_1)}_{\kappa \mid(\sigma(\nu'^*),\sigma(\nu^*) , \theta)} \to \Sht^{0,\loc(m_1,n_1)}_{\kappa \mid(\sigma(\nu^*+\nu'^*) , \theta)}
,\   \Sht^{\loc(m_3,n_3)}_{\theta,\nu'^*,\nu^*} \to \Sht^{\loc(m_3,n_3)}_{\theta,\nu^*+\nu'^*},\ 
\textrm{and}\ \Sht^{0,\loc(m_3,n_3)}_{(\theta,\nu'^*,\nu^*)\mid \mu}  \to \Sht^{0,\loc(m_3,n_3)}_{(\theta,\nu^*+\nu'^*)\mid \mu} .
\]
\end{small}

Moreover, if $(m'_1,n'_1,m'_2,n'_2,m'_3,n'_3)\geq (m_1,n_1,m_2,n_2,m_3,n_3)$ is another sextuple satisfying the same conditions as above ($(m'_1,n'_1,m'_2,n'_2,m'_3,n'_3)=(\infty,\infty,\infty,\infty,\infty,\infty)$ being allowed), then the following diagram is commutative.
\begin{equation}
\label{E:projection localization truncated} 
\xymatrix@C=20pt{
&&\Sht_{\kappa|\mu}^{\nu+\nu',\loc(m'_1,n'_1)} 
\ar[ddd]\ar[dll]\ar[dr]
\\
\Sht^{\loc(m'_1,n'_1)}_\kappa \ar[d]^{\res^{m'_1,n'_1}_{m_1,n_1}} &\ar[l]\ar[d] \Sht_{\kappa\mid\la}^{\nu',\loc(m'_1,n'_1)} \times_{\Sht_\la^\loc(m'_2,n'_2)} \Sht_{\la\mid\mu}^{\nu,\loc(m'_2,n'_2)} \ar[l] \ar[rr] \ar[ur]_-{\mathrm{Comp}^{\loc(m'_1,n'_1)}}&& \Sht^\loc_\mu \ar[d]_{\res^{m'_3,n'_3}_{m_3,n_3}}
\\
\Sht_\kappa^{\loc(m_1,n_1)}
&\ar[l]\Sht_{\kappa|\la}^{\nu',\loc(m_1,n_1)} \times_{\Sht_\la^{\loc(m_2,n_2)}} \Sht_{\la|\mu}^{\nu,\loc(m_2,n_2)} \ar[rd]^-{\mathrm{Comp}^{\loc(m_1,n_1)}} \ar[rr]
&&\Sht_\mu^{\loc(m_3,n_3)}
\\
&& \Sht_{\kappa|\mu}^{\nu+\nu', \loc(m_1,n_1)}\ar[ull]\ar[ru].
}
\end{equation}
In addition, the middle trapezoid is Cartesian.
\end{proposition}
\begin{rmk}Note that the product $\Sht_{\kappa|\la}^{\nu',\loc(m_1,n_1)} \times_{\Sht_\la^{\loc(m_2,n_2)}} \Sht_{\la|\mu}^{\nu,\loc(m_2,n_2)}$ is independent of $(m_2,n_2)$.
\end{rmk}

\begin{proof}[Proof of Proposition~\ref{P:pushforward is local}]
The construction of the map $\mathrm{Comp}^{\loc(m_1,n_1)}$ is clear as explained at each step, and details can be found in the proof of Lemma~\ref{L:decomposition of Comp} (for the unrestricted version). We will only give a little more detail on the isomorphism between line 4 and line 5 (as it involves changing the depth of the restriction).
Indeed, we have natural isomorphisms of functors.

\begin{align*}
&
- \times_{\Sht^{\loc(m_2,n_2)}_{\eta',\nu'^*}}
\Sht^{0,\loc(m_2,n_2)}_{\eta'\mid (\sigma(\nu^*),\theta);\nu'^*}
\times_{\Sht^{\loc(m_2,n_2)}_{\sigma(\nu^*),\theta,\nu'^*}}
\Sht^{0,\loc(m_2,n_2)}_{\sigma(\nu^*); (\theta,\nu'^*)\mid\eta}
\\
\cong\ & - \times_{\Sht^{\loc(m_1,n_1)}_{\sigma(\nu'^*), \eta'}}
\Sht^{0,\loc(m_1,n_1)}_{\sigma(\nu'^*);\eta'\mid (\sigma(\nu^*),\theta)}
\times_{\Sht^{\loc(m_2,n_2)}_{\sigma(\nu^*),\theta,\nu'^*}}
\Sht^{0,\loc(m_2,n_2)}_{\sigma(\nu^*); (\theta,\nu'^*)\mid\eta}  &&\textrm{\eqref{E:decomp4 restricted} on the first product}
\\
\cong\ &- \times_{\Sht^{\loc(m_1,n_1)}_{\sigma(\nu'^*), \eta'}}
\Sht^{0,\loc(m_1,n_1)}_{\sigma(\nu'^*);\eta'\mid (\sigma(\nu^*),\theta)}
\times_{\Sht^{\loc(m_3,n_3)}_{\theta,\nu'^*,\nu^*}}
\Sht^{0,\loc(m_3,n_3)}_{ (\theta,\nu'^*)\mid\eta;\nu^*}  &&\textrm{\eqref{E:decomp4 restricted} on the second product}.
\end{align*}

The perfect properness of $\mathrm{Comp}^{\loc(m_1,n_1)}$ follows from the perfect properness of the maps (1) and (2), which are in turn consequences of perfect properness of convolution products and Satake correspondences for affine Grassmannians.

The diagram~\eqref{E:projection localization truncated} is clearly commutative by the definition of $\on{Comp}^{\loc(m_1,n_1)}$ (and the description of $\on{Comp}^\loc$ via Lemma \ref{L:decomposition of Comp} if $(m'_1,n'_1,m'_2,n'_2,m'_3,n'_3)=(\infty,\infty,\infty,\infty,\infty,\infty)$) . The trapezoid in \eqref{E:projection localization truncated} is Cartesian because pulling back along the restriction morphisms $\res^{m'_1,n'_1}_{m_1,n_1}: \Sht_\kappa^{\loc(m'_1,n'_1)}\to \Sht_\kappa^{\loc(m_1,n_1)}$ of the series of maps for $(m_1,n_1,m_2,n_2,m_3,n_3)$ exactly gives the series of maps for $(m'_1,n'_1,m'_2,n'_2,m'_3,n'_3)$. For example, applying the base change $\Sht^{\loc(m'_1,n'_1)}_{\kappa} \times_{\Sht^{\loc(m_1,n_1)}_\kappa} -$ to line 2, and repeated applying \eqref{E:shtukas restriction Cartesian with Satake correspondences} and Lemma \ref{corr trun2} to each of the product from the left to right gives
\begin{tiny}
\begin{align*}
&
\Sht^{\loc(m'_1,n'_1)}_{\kappa} \times_{\Sht^{\loc(m_1,n_1)}_\kappa} 
\Sht^{0,\loc(m_1,n_1)}_{\kappa\mid (\sigma(\nu'^*),\eta')}\times_{\Sht^{\loc(m_2,n_2)}_{\eta',\nu'^*}}\Sht^{0,\loc(m_2,n_2)}_{(\eta',\nu'^*)\mid \la}\times_{\Sht^{\loc(m_2,n_2)}_{\lambda}}\Sht^{0,\loc(m_2,n_2)}_{\lambda\mid(\sigma(\nu^*), \eta)}\times_{\Sht^{\loc(m_3,n_3)}_{\eta,\nu^*}}\Sht^{0,\loc(m_3,n_3)}_{(\eta,\nu^*)\mid \mu}
\\
\cong \ & 
\Sht^{0,\loc(m'_1,n'_1)}_{\kappa\mid (\sigma(\nu'^*),\eta')}\times_{\Sht^{\loc(m'_1,n'_1)}_{\eta',\nu'^*}}\Sht^{0,\loc(m'_2,n'_2)}_{(\eta',\nu'^*)\mid \la}\times_{\Sht^{\loc(m_2,n_2)}_{\lambda}}\Sht^{0,\loc(m_2,n_2)}_{\lambda\mid(\sigma(\nu^*), \eta)}\times_{\Sht^{\loc(m_3,n_3)}_{\eta,\nu^*}}\Sht^{0,\loc(m_3,n_3)}_{(\eta,\nu^*)\mid \mu}
\\
\cong \ & 
\Sht^{0,\loc(m'_1,n'_1)}_{\kappa\mid (\sigma(\nu'^*),\eta')}\times_{\Sht^{\loc(m'_1,n'_1)}_{\eta',\nu'^*}}\Sht^{0,\loc(m'_2,n'_2)}_{(\eta',\nu'^*)\mid \la}\times_{\Sht^{\loc(m'_2,n'_2)}_{\lambda}}\Sht^{\loc(m'_2,n'_2)}_{\lambda}\times_{\Sht^{\loc(m_2,n_2)}_{\lambda}}\Sht^{0,\loc(m_2,n_2)}_{\lambda\mid(\sigma(\nu^*), \eta)}\times_{\Sht^{\loc(m_3,n_3)}_{\eta,\nu^*}}\Sht^{0,\loc(m_3,n_3)}_{(\eta,\nu^*)\mid \mu}
\\
\cong \ &
\Sht^{0,\loc(m'_1,n'_1)}_{\kappa\mid (\sigma(\nu'^*),\eta')}\times_{\Sht^{\loc(m'_1,n'_1)}_{\eta',\nu'^*}}\Sht^{0,\loc(m'_2,n'_2)}_{(\eta',\nu'^*)\mid \la}\times_{\Sht^{\loc(m'_2,n'_2)}_{\lambda}}\Sht^{0,\loc(m'_2,n'_2)}_{\lambda\mid(\sigma(\nu^*), \eta)}\times_{\Sht^{\loc(m'_3,n'_3)}_{\eta,\nu^*}}\Sht^{0,\loc(m'_3,n'_3)}_{(\eta,\nu^*)\mid \mu}.
\end{align*}
\end{tiny}
The rest can be checked similarly.
\end{proof}

\begin{lem}
\label{L:cocycle condition for compositions}
Let $\mu_1, \mu_2, \mu_3, \mu_4, \nu_1, \nu_2, \nu_3$ be dominant coweights and $(m_1, n_1, m_2,n_2, m_3,n_3, m_4,n_4)$ be an octuple of non-negative integers such that
\begin{itemize}
\item
$m_i-m_{i+1}=n_i - n_{i+1}$ is $\nu_i$-large for $i = 1, 2,3$,
\item
$m_i-n_{i+1}$ is $(\mu_i, \nu_i)$-large and $(\mu_{i+1}, \nu_i)$-large.
\end{itemize}
Then the following diagram commutes
\begin{tiny}
\[
\xymatrix@C=70pt{ 
\Sht_{\mu_1|\mu_2}^{\nu_1,\loc(m_1,n_1)} \times_{\Sht_{\mu_2}^{\loc(m_2,n_2)}} \Sht_{\mu_2|\mu_3}^{\nu_2,\loc(m_2,n_2)}  \times_{\Sht_{\mu_3}^{\loc(m_3,n_3)}} \Sht_{\mu_3|\mu_4}^{\nu_3,\loc(m_3,n_3)} 
\ar[r]^-{\on{Comp}_{\mu_1,\mu_2, \mu_3}^{\mathrm{loc}(m_1,n_1)} \times \id} \ar[d]^-{\id \times \on{Comp}_{\mu_2,\mu_3, \mu_4}^{\mathrm{loc}(m_2,n_2)}} & \Sht_{\mu_1|\mu_3}^{\nu_1+\nu_2,\loc(m_1,n_1)}  \times_{\Sht_{\mu_3}^{\loc(m_3,n_3)}} \Sht_{\mu_3|\mu_4}^{\nu_3,\loc(m_3,n_3)} \ar[d]^-{\on{Comp}_{\mu_1,\mu_3, \mu_4}^{\mathrm{loc}(m_1,n_1)}}
\\
\Sht_{\mu_1|\mu_2}^{\nu_1,\loc(m_1,n_1)} \times_{\Sht_{\mu_2}^{\loc(m_2,n_2)}} \Sht_{\mu_2|\mu_4}^{\nu_2+\nu_3,\loc(m_2,n_2)}  \ar[r]^-{\on{Comp}_{\mu_1,\mu_2, \mu_4}^{\mathrm{loc}(m_1,n_1)}}
& \Sht_{\mu_1|\mu_4}^{\nu_1+\nu_2+\nu_3,\loc(m_1,n_1)}.
}
\]
\end{tiny}
\end{lem}
\begin{proof}
Using the commutative diagram below (in which the superscripts $\loc(m_?, n_?)$ are omitted), we see that the composition $\mathrm{Comp}_{\mu_1, \mu_2, \mu_4}^{\loc(m_1, n_1)} \circ (\id \times \mathrm{Comp}_{\mu_2, \mu_3, \mu_4}^{\loc (m_2, n_2)})$ is the composition from upper left to lower left by first go to the upper right and then lower right. This is the same as the downwards composition directly from the upper left to lower left. By a symmetric argument, the composition of the downwards left column is the same as the composition $\mathrm{Comp}_{\mu_1, \mu_3, \mu_4}^{\loc(m_1, n_1)} \circ ( \mathrm{Comp}_{\mu_1, \mu_2, \mu_3}^{\loc (m_1, n_1)} \times \id)$. This then proves the lemma.
\begin{center}
\begin{sideways}
\begin{tiny}
$$
\xymatrix@R=40pt@C=-120pt{
\Sht_{\mu_1|\mu_2}^{\nu_1}  \underset{\Sht_{\mu_2}}{\times}
\Sht_{\mu_2|\mu_3}^{\nu_2}\underset{\Sht_{\mu_3}}\times
\Sht_{\mu_3|\mu_4}^{\nu_3} 
\ar@<-20pt>[d]^{\textrm{Prop.\ref{P:pushforward is local}}(1)} \ar[rr]^-{\textrm{Prop.\ref{P:pushforward is local}}(1)}
&&
\Sht^{\nu_1}_{\mu_1|\mu_2} \underset{\Sht_{\mu_2}}\times \Sht^0_{\mu_2|(\sigma(\nu_2^*), \eta_2)} \underset{\Sht_{\eta_2, \nu_2^*}} \times \Sht^0_{(\eta_2, \nu_2^*)|(\sigma(\nu_3^*), \eta_3)}\underset{\Sht_{\eta_3, \nu_3^*}} \times \Sht^0_{(\eta_3, \nu_3^*)|\mu_4} \ar@<20pt>[d]^\cong\ar[dll]_{\textrm{Prop.\ref{P:pushforward is local}}(1)}\!\!\!\!\!\!\!\!\!\!\!\!\!\!\!\!\!\!\!\!
\\ \!\!\!\!\!\!\!\!\!\!\!\!\!\!\!\!\!\!\!\!\!\!\!\!\!\!\!\!\!\!\!\!\!\!\! \!\!\!\!\!\!\!\!\!\!\!\!\!\!\!\!\!\!\!\!\!\!\!\!\!\!\!\!\!\!\!\!\!\!\!
\Sht^0_{\mu_1|(\sigma(\nu_1^*), \eta_1)} \underset{\Sht_{\eta_1, \nu_1^*}} \times \Sht^0_{(\eta_1, \nu_1^*)|(\sigma(\nu_2^*), \eta_2)}\underset{\Sht_{\eta_2, \nu_2^*}} \times \Sht^0_{(\eta_2, \nu_2^*)|(\sigma(\nu_3^*), \eta_3)}\underset{\Sht_{\eta_3, \nu_3^*}} \times \Sht^0_{(\eta_3, \nu_3^*)|\mu_4} \ar@<-20pt>[dd]^-\cong \ar[dr]^-\cong
&& \Sht_{\mu_1|\mu_2}^{\nu_1} \underset{\Sht_{\mu_2}}{\times}\Sht^0_{\mu_2|((\sigma(\nu_2^*), \sigma(\nu_3^*), \eta_2+\eta_3)} \underset{\Sht_{\eta_2+\eta_3, \nu_2^*, \nu_3^*}}{\times} \Sht^0_{(\eta_2+\eta_3, \nu_2^*, \nu_3^*)|\mu_4}\!\!\!\!\!\!\!\!\!\!\!\!\!\!\!\!\!\!\!\! \ar@<20pt>[dd]^{\textrm{Prop.\ref{P:pushforward is local}}(2)} \ar[dl]_{\textrm{Prop.\ref{P:pushforward is local}}(1)}
\\
& \Sht^0_{\mu_1|(\sigma(\nu_1^*), \eta_1)} \underset{\Sht_{\eta_1, \nu_1^*}}\times 
\Sht^0_{(\eta_1, \nu_1^*)|(\sigma(\nu_2^*),\sigma(\nu_3^*), \eta_2+\eta_3)} \underset{\Sht_{\eta_2+\eta_3,\nu_2^*, \nu_3^*}}{\times} \Sht_{(\eta_2+\eta_3,\nu_2^*, \nu_3^*)|\mu_4}^0 \ar[dl]^-\cong \ar[dd]^{\textrm{Prop.\ref{P:pushforward is local}}(2)}
\\
\!\!\!\!\!\!\!\!\!\!\!\!\!\!\!\!\!\!\!\!\!\!\!\!\!\!\!\!\!\!\!\!\!\!\!\!\!\!\!\!\!\!\!\!\!\!\!\!\!\!\!\!\!\!\!\!\!\!\!\!\!\!\!\!\!\!\!\!\!\!\!\!\!\!\!\!\!\!\!\!\! \Sht^0_{\mu_1|(\sigma(\nu_1^*),\sigma(\nu_2^*), \sigma(\nu_3)^*, \eta_1+\eta_2+\eta_3)} \underset{\Sht_{\eta_1+\eta_2+\eta_3, \nu_1^*, \nu_2^*, \nu_3^*}}\times \Sht^0_{(\eta_1+\eta_2+\eta_3, \nu_1^*, \nu_2^*, \nu_3^*)|\mu_4} \ar@<-20pt>[ddd]^{\textrm{Prop.\ref{P:pushforward is local}}(2)}  \ar[dr]_{\textrm{Prop.\ref{P:pushforward is local}}(2)}
&
& \ar[dl]_{\textrm{Prop.\ref{P:pushforward is local}}(1)}
\qquad \qquad \Sht_{\mu_1|\mu_2}^{\nu_1} \underset{\Sht_{\mu_2}}\times \Sht^{\nu_2+\nu_3}_{\mu_2|\mu_4}\ar@<20pt>[ddd]^{\textrm{Prop.\ref{P:pushforward is local}}(1)}
\\
&
\Sht^0_{\mu_1|(\sigma(\nu_1^*), \eta_1)} \underset{\Sht_{\eta_1, \nu_1^*}}\times 
\Sht^0_{(\eta_1, \nu_1^*)|(\sigma(\nu_2+\nu_3)^*, \eta_2+\eta_3)}\underset{\Sht_{\eta_2+\eta_3, \nu_2^*+\nu_3^*}}\times \Sht^0_{(\eta_2+\eta_3, \nu_2^*+\nu_3^*)|\mu_4} \ar[d]^\cong
\\
& \Sht_{\mu_1|(\sigma(\nu_1^*),\sigma(\nu_2+\nu_3)^*, \eta_1+\eta_2+\eta_3)}^0 \underset{\Sht_{ \eta_1+\eta_2+\eta_3, \nu_1^*, \nu_2^*+\nu_3^*}}{\times}
\Sht_{( \eta_1+\eta_2+\eta_3, \nu_1^*, \nu_2^*+\nu_3^*)|\mu_4}^0
\ar[dl]_{\textrm{Prop.\ref{P:pushforward is local}}(2)} &
\\
\!\!\!\!\!\!\!\!\!\!\!\!\!\!\!\!\!\!\!\!
\!\!\!\!\!\!\!\!\!\!\!\!\!\!\!\!\!\!\!\!
\!\!\!\!\!\!\!\!\!\!\!
\Sht_{\mu_1|(\sigma(\nu_1+\nu_2+\nu_3)^*, \eta_1+\eta_2+\eta_3)} \underset{\Sht_{\eta_1+\eta_2+\eta_3, \nu_1^*+ \nu_2^*+ \nu_3^*}}\times \Sht^0_{(\eta_1+\eta_2+\eta_3, \nu_1^*+ \nu_2^*+ \nu_3^*)|\mu_4}
&&
\Sht_{\mu_1|(\sigma(\nu_1^*), \eta_1)}^0 \underset{\Sht_{\eta_1, \nu_1^*}}{\times}
\Sht_{(\eta_1, \nu_1^*)|\mu_2}^0\underset{\Sht_{\mu_2}}\times \Sht^{\nu_2+\nu_3}_{\mu_2|\mu_4} \ar[ul]^-\cong
}
$$
\end{tiny}
\end{sideways}
\end{center}
\end{proof}

\newpage
\subsection{The category $\on{P}^{\on{Corr}}(\Sht^\loc_{\bar k})$}
\label{S: cat PCorrSht}

We now define the category $\rmP^\mathrm{Corr}(\Sht^\loc_{\bar k})$. It has the same objects as $\rmP(\Sht_{\bar k}^\loc)$. And we just need to explain the space of morphisms and compositions of them as follows.

\subsubsection{Definition of morphisms} Recall from \S\ref{SS:Perv(Sloc)} that every object of $\rmP(\Sht_{\bar k}^\loc)$ (or equivalently every object of $\rmP^\mathrm{Corr}(\Sht^\loc_{\bar k})$) is a direct sum of \emph{connected objects} $\calF$, i.e. objects that appear in the limit of $\rmP(\Sht_\mu^{\loc(m,n)})$ over $(\mu,m, n) \in \bbzeta \times \ZZ^2_{\geq 0}$ for some $\bbzeta \in \pi_1(G)$. Each connected object $\calF$ is realized as a perverse sheaf $\calF_{\mu}^{(m,n)} \in \rmP(\Sht_\mu^{\loc(m,n)})$ for some triple $(\mu,m,n)$ such that $m-n$ is $\mu$-large. For another triple $(\mu',m',n')\geq (\mu,m,n)$, we identify $\calF_{\mu}^{(m,n)}$ with $\mF_{\mu'}^{(m',n')}:=i_{\mu,\mu',*}\Res^{m',n'}_{m,n}\mF_{\mu}^{(m,n)}$.

Recall that for a septuple $(\mu_1,\mu_2,\nu,m_1,n_1,m_2,n_2)$ such that  $(m_1,n_1,m_2,n_2)$ is $(\mu_1+\nu,\nu)$-acceptable and $(\mu_2+\nu,\nu)$-acceptable,
there is the correspondence 
$$\Sht^{\loc(m_1,n_1)}_{\mu_1}\leftarrow \Sht^{\nu,\loc(m_1,n_1)}_{\mu_1\mid\mu_2}\to \Sht^{\loc(m_2,n_2)}_{\mu_2}$$
from Definition \ref{D:mS lambda mu truncated}. If $(\mu'_1,\mu'_2,\nu',m'_1,n'_1,m'_2,n'_2)\geq (\mu_1,\mu_2,\nu,m_1,n_1,m_2,n_2)$ is another septuple satisfying the similar acceptability conditions, there is the following diagram
\begin{equation}
\label{E: loc Hecke corr changing level}
\begin{CD}
\Sht_{\mu_1}^{\loc(m_1,n_1)}@<\res^{m'_1,n'_1}_{m_1,n_1}<< \Sht_{\mu_1}^{\loc(m'_1,n'_1)}@>i_{\mu_1,\mu'_1}>> \Sht_{\mu'_1}^{\loc(m'_1,n'_1)}\\
@AAA@AAA@AAA\\
\Sht_{\mu_1\mid\mu_2}^{\nu,\loc(m_1,n_1)}@<\res<<\Sht_{\mu_1\mid\mu_2}^{\nu,\loc(m'_1,n'_1)}@>i>>\Sht_{\mu'_1\mid\mu'_2}^{\nu',\loc(m'_1,n'_1)}\\
@VVV@VVV@VVV\\
\Sht_{\mu_2}^{\loc(m_2,n_2)}@<\res^{m'_2,n'_2}_{m_2,n_2}<<\Sht_{\mu_2}^{\loc(m'_2,n'_2)}@>i_{\mu_2,\mu'_2}>>\Sht_{\mu'_2}^{\loc(m'_2,n'_2)}
\end{CD}
\end{equation}
where arrows to the left are equidimensional, perfectly smooth, of the same relative dimension, and arrows to the right are closed embeddings. Note that the upper left square is Cartesian by Lemma \ref{corr trun2} and all arrows in the upper right square are perfectly proper.


We define the space $\mathrm{Mor}_{\rmP^\mathrm{Corr}(\Sht^\loc)} (\calF_1, \calF_2)$ of morphisms between two connected objects $\calF_1, \calF_2$ ($\calF_i\in \on{P}(\Sht_{\bbzeta_i}^\loc)$) to be
\begin{equation}
\label{E: hom space of PcorrSht}
\bigoplus_{\bbzeta \in \pi_1(G)} \varinjlim_{(\mu_1,\mu_2,\nu,m_1,n_1,m_2,n_2)} \mathrm{Corr}_{\Sht_{\mu_1|\mu_2}^{\nu, \loc(m_1,n_1)}} \big( (\Sht_{\mu_1}^{\loc(m_1,n_1)}, \calF_{1,\mu_1}^{(m_1,n_1)}),\,( \Sht_{\mu_2}^{\loc(m_2,n_2)},\calF_{2,\mu_2}^{(m_2,n_2)}) \big),
\end{equation}
where the colimit is taken over all (partially ordered) sextuples $(\mu_1,\mu_2,\nu,m_1,n_1,m_2,n_2)$ satisfying the acceptability conditions as above (so that the space $\Sht_{\mu_1|\mu_2}^{\nu,\loc(m_1,n_1)}$ is well-defined) and $\mu_1\in\bbzeta_1,\mu_2\in\bbzeta_2,\nu\in\bbzeta$.
The connecting map of the filtered colimit is given by $i_!\circ\res^\star$, 
where $\res^\star$ is the shifted smooth pullback of cohomological correspondences along the map $\res$ in \eqref{E: loc Hecke corr changing level}  (in the sense of \S \ref{ASS:smooth pullback correspondence}, in particular Notation \ref{N: shifted pull back coh cor}),  and $i_!$ is the pushforward of cohomological correspondences along the map $i$ (in the sense of \S \ref{ASS:pushforward correspondence}) in \eqref{E: loc Hecke corr changing level}. Note that for $(m''_1,n''_1)\geq (m'_1,n'_1)$, the following diagram is Cartesian
\[\begin{CD}
\Sht_{\mu_1\mid \mu_2}^{\nu,\loc(m''_1,n''_1)}@>i>>\Sht_{\mu'_1\mid \mu'_2}^{\nu',\loc(m''_1,n''_1)}\\
@V\res VV@VV\res V\\
\Sht_{\mu_1\mid \mu_2}^{\nu,\loc(m'_1,n'_1)}@>i>>\Sht_{\mu'_1\mid\mu'_2}^{\nu',\loc(m'_1,n'_1)}
\end{CD}\]
Therefore, by Lemma \ref{AL:pushforward pullback compatibility}, Lemma \ref{L: comp of pull coh corr} and Lemma \ref{L: comp of push coh corr}, these connecting maps compose and the filtered colimit is well-defined.

Finally, we define the morphisms between general objects in $\rmP^\mathrm{Corr}(\Sht^\loc)$ by extending the definition on the connected ones by linearity.

\subsubsection{Composition of morphisms}
We now define the composition of morphisms. By linearity, it suffices to define this on the connected objects and we may assume that the given morphisms $c_1: \calF_1 \to \calF_2$ and $c_2: \calF_2 \to \calF_3$ are realized as correspondences (as opposed to linear combinations of correspondences)
\begin{align*}
c^{(\mu_1,\mu_2,\nu_1,m_1,n_1,m_2,n_2)}_1 \in \mathrm{Corr}_{\Sht_{\mu_1|\mu_2}^{\nu_1,\loc(m_1,n_1)}}\big( (\Sht_{\mu_1}^{\loc(m_1, n_1)}, \calF_{1,\mu_1}^{(m_1,n_1)}),\ (\Sht_{\mu_2}^{\loc(m_2, n_2)}, \calF_{2,\mu_2}^{(m_2,n_2)}) \big)&, \quad\textrm{and}\\
c^{(\mu_2,\mu_3,\nu_2,m_2,n_2,m_3,n_3)}_2 \in \mathrm{Corr}_{\Sht_{\mu_2|\mu_3}^{\nu_2,\loc(m_2,n_2)}}\big( (\Sht_{\mu_2}^{\loc(m_2, n_2)},\ \calF_{2,\mu_2}^{(m_2,n_2)}) \,\ (\Sht_{\mu_3}^{\loc(m_3, n_3)}, \calF_{3,\mu_3}^{(m_3,n_3)})\big)&,
\end{align*}
where the septuples satisfy the conditions as above.
According to the formalism of the composition of cohomological correspondences (see Definition~\ref{AD:correspondences}), the composition $c^{(\mu_2,\mu_3,\nu_2,m_2,n_2,m_3,n_3)}_2 \circ 
c^{(\mu_1,\mu_2,\nu_1,m_1,n_1,m_2,n_2)}_1$ is supported on the following correspondence
\[
\Sht_{\mu_1}^{\loc(m_1,n_1)} \leftarrow \Sht_{\mu_1|\mu_2}^{\nu_1,\loc(m_1, n_1)} \times_{ \Sht_{\mu_2}^{\loc(m_2,n_2)}}
\Sht_{\mu_2|\mu_3}^{\nu_2,\loc(m_2, n_2)}  \to  \Sht_{\mu_3}^{\loc(m_3,n_3)}.
\]

We then define the composition $c_2\circ c_1$ to be the cohomological correspondence from $\calF_1$ to $\calF_3$ so that $(c_2\circ c_1)^{(\mu_1,\mu_3,\nu_1+\nu_2,m_1,n_1,m_3,n_3)}$ is given by the push forward of $c^{(\mu_2,\mu_3,\nu_2,m_2,n_2,m_3,n_3)}_2 \circ 
c^{(\mu_1,\mu_2,\nu_1,m_1,n_1,m_2,n_2)}_1$ along the perfectly proper morphism
\[
\Sht_{\mu_1|\mu_2}^{\nu_1,\loc(m_1, n_1)} \times_{ \Sht_{\mu_2}^{\loc(m_2,n_2)}}
\Sht_{\mu_2|\mu_3}^{\nu_2,\loc(m_2, n_2)} \xrightarrow{\mathrm{Comp}^{\loc(m_1,n_1)}}\Sht_{\mu_1|\mu_3}^{\nu_1+\nu_2,\loc(m_1, n_1)},
\]
in the sense of \S\ref{ASS:pushforward correspondence}. To show that $c_2\circ c_1$ is well-defined, we need to check that if 
$$(\mu'_1,\mu'_2,\mu'_3,\nu'_1,\nu'_2,m'_1,n'_1,m'_2,n'_2,m'_3,n'_3)\geq (\mu_1,\mu_2,\mu_3,\nu_1,\nu_2,m_1,n_1,m_2,n_2,m_3,n_3)$$ 
(in the product partial order), then the connecting map in the colimit \eqref{E: hom space of PcorrSht}, 
$$i_!(\res^\star((c_2\circ c_1)^{(\mu_1,\mu_3,\nu_1+\nu_2,m_1,n_1,m_3,n_3)}))=(c_2\circ c_1)^{(\mu'_1,\mu'_3,\nu'_1+\nu'_2,m'_1,n'_1,m'_3,n'_3)}.$$
First, we can apply Lemma \ref{AL:pullback compatible with composition} and Lemma \ref{AL:pushforward compatible with composition} to see that under the pullback along $\res$ and the pushforward along $i$ (in \eqref{E: loc Hecke corr changing level}),  $c^{(\mu_2,\mu_3,\nu_2,m_2,n_2,m_3,n_3)}_2 \circ 
c^{(\mu_1,\mu_2,\nu_1,m_1,n_1,m_2,n_2)}_1$ maps to $c^{(\mu'_2,\mu'_3,\nu'_2,m'_2,n'_2,m'_3,n'_3)}_2 \circ 
c^{(\mu'_1,\mu'_2,\nu'_1,m'_1,n'_1,m'_2,n'_2)}_1$. Since the middle trapezoid in \eqref{E:projection localization truncated}  is Cartesian by Proposition \ref{P:pushforward is local}, we can apply Lemma \ref{AL:pushforward pullback compatibility} to conclude.

Finally, the composition satisfies the natural associativity law $(c_3\circ c_2) \circ c_1 \cong c_3 \circ (c_2 \circ c_1)$ by Lemma~\ref{L:cocycle condition for compositions} and Lemma \ref{L: comp of push coh corr}. The construction of $\on{P}^{\on{Corr}}(\Sht^\loc_{\bar k})$ now is complete.

\subsubsection{Description of morphisms}
We describe the endomorphism ring of one object in $\on{P}^{\on{Corr}}(\Sht^\loc_{\bar k})$. 
Recall by \eqref{E:P of Gr = P of Hk} and \eqref{E:functor P(Hk) to P(S)}, we have a functor
\[\on{P}_{L^+G\otimes \bar k}(\Gr\otimes \bar k)\cong \on{P}(\Hk^\loc_{\bar k})\xrightarrow{\Phi^\loc}\on{P}(\Sht^\loc_{\bar k}).\]
Let $\delta_{\mathbf{1}}\in\on{P}(\Sht^\loc_{\bar k})$  denote the image of $\IC_0$.
Recall from Example \ref{Ex:mStau mn},  $\Sht_0^{\loc(m,n)}\cong [L^m_nG\backslash \on{pt}]$ is perfectly smooth. Then $\delta_{\mathbf{1}}$ can be realized as 
$$\delta_{\mathbf{1}}^{(m,n)}:=\Ql[(m-n)\dim G](\frac{m-n}{2}\dim G)\in \on{P}(\Sht_{0,\bar k}^{\loc(m,n)})$$ for every $m\geq n$.

\begin{prop}
\label{P: endo of unit}
The endomorphism ring $\End_{\on{P}^{\on{Corr}}(\Sht^\loc_{\bar k})}(\delta_{\mathbf{1}})$ is isomorphic to the spherical Hecke algebra $H_G\otimes\Ql=C_c(G(\mO)\backslash G(F)/G(\mO))$.
\end{prop}
\begin{proof}
We first understand
\[\Sht^{\loc(m_1,n_1)}_{0}\xleftarrow{\overleftarrow{h}_0^{\loc(m_1,n_1)}} \Sht^{\nu,\loc(m_1,n_1)}_{0\mid 0}\rightarrow \Sht^{\loc(m_2,n_2)}_0\]
for $\nu\in\xcoch(T)^{+,\sigma}$.

Note that the affine Deligne-Lusztig variety $X_{0}(0)\cong \Gr(k)$ is discrete
and for , $X_{0,\nu^*}(0)\cong \Gr_{\nu^*}(k)=\Gr_{\nu^*}^{(\infty)}(k)/G(\mO)$. Choose $n$ to be $\nu$-large. By Corollary \ref{C:expclit mStautau truncated},  the left arrow of the above correspondence is identified with
\[[L^{m_1}_{n_1}G\backslash \on{pt}]\leftarrow [L^{m_1}_{n_1}G\backslash \Gr_{\nu^*}^{(\infty)}(k)/G(\mO)],\]
with the map is induced by the natural projection $\Gr_{\nu^*}^{(\infty)}(k)/G(\mO)\to \on{pt}$. It follows from the same considerations as in \S \ref{cl vs cc} (see also Example \ref{Ex:examples of correspondences} (5)) that 
\begin{eqnarray*}
 & &\mathrm{Corr}_{\Sht_{0|0}^{\nu, \loc(m_1,n_1)}} \big( (\Sht_{0}^{\loc(m_1,n_1)},\delta_{\mathbf{1}}^{(m_1,n_1)}),( \Sht_{0}^{\loc(m_2,n_2)},\delta_{\mathbf{1}}^{(m_2,n_2)}) \big)\\
&\cong &\on{H}_0^{\on{BM}}([L^{m_1}_{n_1}G\backslash \Gr_{\nu^*}^{(\infty)}(k)/G(\mO)]\otimes\bar k)  \cong C(\Gr_{\nu^*}(k))^{G(\mO/\varpi^{n_1})}=C(G(\mO)\backslash \Gr_{\nu^*}^{(\infty)}(k)/G(\mO)),
\end{eqnarray*}
where $C(-)$ denotes the space of $\Ql$-valued functions. More precisely, giving such a function $f: G(F)\to \Ql$, the correspondence $c^{(0,0,\nu,m_1,n_1,m_2,n_2)}$ at $[g]\in [L^{m_1}_{n_1}G\backslash \Gr_{\nu^*}^{(\infty)}(k)/G(\mO)]$ is given by multiplying $f(g): \Ql\to \Ql$.

Note that if $(\nu,m_1,n_1,m_2,n_2)\leq (\nu',m'_1,n'_1,m'_2,n'_2)$, the connecting map in \eqref{E: hom space of PcorrSht} under the above isomorphism is identified with
$$C(G(\mO)\backslash \Gr_{\nu^*}^{(\infty)}(k)/G(\mO))\to C(G(\mO)\backslash \Gr_{\nu'^*}^{(\infty)}(k)/G(\mO)),$$ 
where we regard a function on $\Gr_{\nu^*}^{(\infty)}(k)$ as a function on $\Gr_{\nu'^*}^{(\infty)}(k)$ by extension by zero. Therefore, as a vector space, 
$\End_{\on{P}^{\on{Corr}}(\Sht^\loc_{\bar k})}(\delta_{\mathbf{1}})$ is identified with the space of bi-$G(\mO)$-invariant, compactly supported functions on $G(F)$. We also need to identify the ring structure.

By the same proof as Lemma \ref{L: torsor over S-corr} and Corollary \ref{C:expclit mStautau truncated}, one can identify \eqref{E:composition restricted} ($\kappa=\la=\mu=0$) with
\[[L^{m_1}_{n_1}\backslash (\Gr_{\nu'^*}\tilde\times\Gr_{\nu^*})(k)]\to [L^{m_1}_{n_1}\backslash \Gr_{\nu'^*+\nu^*}(k)],\]
which is induced by the convolution map $\Gr_{\nu'^*}\tilde\times\Gr_{\nu^*}\to \Gr_{\nu'^*+\nu}$ (defined in \eqref{E: convolution map}). 
Now, let $c_1^{(0,0,\nu',m_1,n_1,m_2,n_2)}$ (resp. $c_2^{(0,0,\nu,m_2,n_2,m_3,n_3)}$) be a cohomological correspondence from $\delta_{\mathbf{1}}^{(m_1,n_1)}$ to $\delta_{\mathbf{1}}^{(m_2,n_2)}$ (resp. $c$ from $\delta_{\mathbf{1}}^{(m_2,n_2)}$ to $\delta_{\mathbf{1}}^{(m_3,n_3)}$), corresponding to $f_1\in C(G(\mO)\backslash \Gr_{\nu'^*}^{(\infty)}(k)/G(\mO))$ (resp. $f_2\in C(G(\mO)\backslash \Gr_{\nu^*}^{(\infty)}(k)/G(\mO))$), under the above identification. The composition $c_2^{(0,0,\nu,m_2,n_2,m_3,n_3)}\circ c_1^{(0,0,\nu',m_1,n_1,m_2,n_2)}$ is supported on $\Sht_{0\mid 0}^{\nu',\loc(m_1,n_1)}\times_{\Sht_0^{\loc(m_2,n_2)}}\Sht_{0\mid 0}^{\nu,\loc(m_2,n_2)}$, which again by the same considerations as in \S \ref{cl vs cc} and Example \ref{Ex:examples of correspondences} (5), corresponding to a function $G(\mO)$-invariant function on $\Gr_{\nu'^*}\tilde\times\Gr_{\nu^*}(k)$.
It is clear that for $(g_1,g_2)\in \Gr^{(\infty}_{\nu'^*}(k)\times\Gr^{(\infty)}_{\nu^*}(k)$, this function is $f_1(g_1)f_2(g_2)$. Therefore, the function on $\Gr_{\nu'^*+\nu^*}(k)$ corresponding to 
$$\on{Comp}^{\loc(m_1,n_1)}_!(c_2^{(0,0,\nu,m_2,n_2,m_3,n_3)}\circ c_1^{(0,0,\nu',m_1,n_1,m_2,n_2)})$$ 
is the convolution $f_1* f_2$ (defined in \eqref{E: convolution for functions}). The proposition follows.
\end{proof}

\begin{rmk}
\label{R: endo generalization}
(1) The same statement and argument hold if one replace $0$ by a central coweight $\tau\in\xcoch(Z_G)$. Namely, $\End_{\on{P}^{\on{Corr}}(\Sht^\loc_{\bar k})}(\Phi^\loc(\IC_\tau))=H_G\otimes\Ql$.

(2) Recall that in Remark \ref{R: more object}, for every $n>0$ we have the object $\delta_\rho$ for a representation $\rho$ of $G(\mO/\varpi^n)$. If $\rho$ is the regular representation of $G(\mO/\varpi^n)$, then by the same argument, one can show that 
$$\End_{\on{P}^{\on{Corr}}(\Sht^\loc_{\bar k})}(\delta_\rho)\cong C_c(K_n\backslash G(F)/K_n),$$ where $K_n=L^+G^{(n)}(k)$ is the $n$th principal congruence subgroup, and  $ C_c(K_n\backslash G(F)/K_n)$ denote the corresponding Hecke algebra.
\end{rmk}


\section{Cohomological correspondences between moduli of local shtukas}
\setcounter{equation}{0}
\setcounter{subsubsection}{0}
In this section, all (restricted) Hecke stacks and (restricted) moduli of local shtukas are considered over $\bar k$.
We explain the proof of Theorem~\ref{T:periodic geometric Satake}. 
Let us restate it here. 
\begin{thm}
\label{T:Spectral action}
\begin{enumerate}
\item There exists a functor $S: \on{Coh}^{\hat G}_{fr}(\hat G\sigma) \to \mathrm{P}^\mathrm{Corr}(\Sht^\loc_{\bar k})$ making the following diagram commutative.
\begin{equation}
\label{E:periodic geometric Satake}
\xymatrix@C=40pt{
\mathrm{Rep}(\hat G)
\ar[r]^-{\mathrm{Sat}} \ar[d] & \mathrm{P}(\Hk^\loc_{\bar k}) \ar[d]^{\Phi^{\loc}}
\\
\Coh^{\hat G}_{fr}(\hat G\sigma) \ar[r]^-{S} & \mathrm{P}^\mathrm{Corr}(\Sht^\loc_{\bar k}).
}
\end{equation}
Here the top arrow is the composition of the usual geometric Satake correspondence and its equivalence from \eqref{E:P of Gr = P of Hk}, and the right vertical arrow is the pullback of perverse sheaves as constructed in \eqref{E:functor P(Hk) to P(S)}. $\Coh^{\hat G}_{fr}(\hat G\sigma)$ is defined as the full subcategory of coherent sheaves $[\hat G\sigma/\hat G]$ spanned by
the pullback of coherent sheaves along the natural projection $[\hat G\sigma/\hat G]\to \bfB \hat G$ (see Notation~\ref{N:Rep(G) to Vect(G)} and the left arrow is the pullback functor.  
\item $S(\mO_{[\hat G\sigma/\hat G]})=\delta_{\mathbf{1}}$ and the map
\[S: \bfJ:= \Gamma([G\sigma/\hat G],\mO)\to \End_{\mathrm{P}^\mathrm{Corr}(\Sht^\loc_{\bar k})}(\delta_{\mathbf{1}})\cong H_G\otimes\Ql\]
coincides with the Satake isomorphism.
\end{enumerate}
\end{thm}
Here the isomorphism $\End_{\mathrm{P}^\mathrm{Corr}(\Sht^\loc_{\bar k})}(\delta_{\mathbf{1}})\cong H_G\otimes\Ql$ in Part (2) is from Proposition \ref{P: endo of unit}, and the Satake isomorphism comes from \eqref{E: Sat, arith Frob}. 

To prove the theorem, we will upgrade the correspondences \eqref{deompf1} between moduli of restricted local shtukas to cohomological correspondences.
The essential computation is to verify the compatibility of this functor with respect to the composition law of morphisms. In special cases, we describe these cohomological correspondences more explicitly, by relating them to the cycle class maps or to the Hecke operators. In particular, we give a simpler proof of V. Lafforgue's $S=T$ theorem \cite{La}\footnote{In fact, Varshavsky (\cite{varS=T}) already gave a simpler proof of the $S=T$ theorem, and our proof is a further simplification of Varshavsky's proof.}, which is the content of Part (2) of the theorem.

For the generalities of cohomological correspondences between perfect stacks, we refer to \S \ref{Sec:cohomological correspondence}.


\subsection{Preparations}

We keep the notation as in the previous section.

\begin{notation}
\label{N:Rep(G) to Vect(G)}
Recall that the Langlands dual group $\hat G$ carries an action of (arithmetic) $q$-Frobenius $\sigma$.
Consider the conjugation action of $\hat G$ on $\hat G\sigma\subset {^L}G=\hat G\rtimes\langle\sigma\rangle$, or equivalently the twisted conjugation action $c_\sigma$ of $\hat G$ on $\hat G$ 
given by $c_\sigma(g)(h) = gh\sigma(g^{-1})$.
We use $\mathrm{Coh}^{\hat G}(\hat G\sigma)$ to denote the category of coherent sheaves over the stack $[(\hat G\sigma) / \hat G]$, or equivalently coherent sheaves over $\hat G \sigma$ equipped with an equivariant $\hat G$-action.

For each algebraic representation $V$ of $\hat G$, we may associate a vector bundle over $\hat G \sigma$ whose global section is $\calO_{\hat G} \otimes V$ with a $\hat G$-action given by the twisted conjugation $c_\sigma$ on the first factor and natural action on the latter, that is, for $g \in \hat G$ and $f: \hat G \to V$, we have 
\[
g \bullet f(h) = g(f(g^{-1}h\sigma(g))), \quad h \in \hat G.
\]
We write $\widetilde V  \in \Coh^{\hat G}(\hat G\sigma)$ for the resulting equivariant bundle.
This defines a natural \emph{faithful} functor $\mathrm{Rep} (\hat G) \to \mathrm{Coh}^{\hat G}(\hat G\sigma)$, sending $V$ to $\widetilde V$. More abstractly, this is the pullback functor for coherent sheaves along the natural projection $[(\hat G\sigma)/\hat G]\to \bfB \hat G$.
\end{notation}

\subsubsection{Hom spaces in $\on{Coh}^{\hat G}(\hat G \sigma)$}
\label{SS:homs in Vect}

To benefit the readers, we make explicit the above functor on morphisms. For a representation $V$ of $G$, let $\bfJ(V)$ denote the space of $V$-valued functions $f$ on $\hat G$ such that $f(gh\sigma(g^{-1})) = g(f(h))$ for $g,h \in \hat G$, i.e. $\bfJ(V)= (\mO_{\hat G}\otimes V)^{\hat G}$.

For two representations $V_1, V_2 \in \mathrm{Rep}(\hat G)$, the hom space $\Hom_{\on{Coh}^{\hat G}(\hat G\sigma)}(\widetilde{V_1}, \widetilde{V_2})$ is the subspace of the hom space of coherent sheaves $\widetilde{V_1} \to \widetilde{V_2}$ on $\hat G\sigma$, consisting of homomorphisms that respect the  $\hat G$-actions. In other words,
\begin{equation}
\label{E:hom in Vect = J}
\mathrm{Hom}_{\Coh^{\hat G}(\hat G \sigma)} \big( \widetilde V_1, \widetilde V_2 \big) \cong \Hom_{\calO_{\hat G \sigma}} \big(\calO_{\hat G\sigma} \otimes V_1, \calO_{\hat G\sigma} \otimes V_2 \big)^{\hat G} \cong
\big( \calO_{\hat G\sigma} \otimes V_1^* \otimes V_2 \big)^{\hat G} \cong \bfJ(V_1^*\otimes V_2).
\end{equation}
Moreover, for a homomorphism $\bba \in \Hom_{\hat G}(V_1, V_2)$, its image $\widetilde\bba$ is the constant function with value $\bba \in V_1^* \otimes V_2$.

\begin{notation}
We write $\mathrm{Coh}^{\hat G}_{fr}(\hat G\sigma)$ for the full subcategory of  $\mathrm{Coh}^{\hat G}(\hat G\sigma)$ spanned by those $\widetilde V$ for $V \in \mathrm{Rep}(\hat G)$.
\end{notation}

To define the functor $S$, we need first make some preparations.

\begin{notation}
We fix a half Tate twist $(1/2)$.

Let $\Sat: \Rep(\hat G)\to \on{P}_{L^+G\otimes \bar k}(\Gr\otimes \bar k)$ be the geometric Satake equivalence.
For $V \in \Rep(\hat G)$,
we write $\Gr_V$ for the support of $\Sat(V)$ on $\Gr$. Explicitly, if $\{V_{\nu_i}\}$ are the Jordan-H\"older factors of $V$, then $\Gr_V = \cup_i \Gr_{\nu_i}$. Note that $\Gr_V$ may not be geometrically connected.
An integer $m$ is \emph{$V$-large} if it is $\nu_i$-large for all $i$.
For such an integer, we put $\Hk_V^{\loc(m)} = [L^mG \backslash \Gr_V]$ and we write  $\Sat(V)^{\loc(m)}$ for the descent of $\Sat(V)$ from $\Gr_V$ (via the equivalence \eqref{E:P of Gr = P of Hk fin}).  It represents the object $\Sat(V) \in \rmP(\Hk^{\loc})$. To simplify notations, sometimes we use $\Sat(V)^{\loc(m)}$ to denote the pair $(\Hk_V^{\loc(m)},\Sat(V)^{\loc(m)})$, if no confusion is likely to arise.

More generally, for a representation $V_\bullet = V_1\boxtimes\cdots\boxtimes V_t$ of representations of $\hat G^t$, we define $\Gr_{V_\bullet} = \Gr_{V_1}\tilde\times \cdots \tilde\times \Gr_{V_t}$ to be the support of the external twisted product 
$$\Sat(V_\bullet):= \Sat(V_1) \tilde \boxtimes \cdots \tilde \boxtimes \Sat(V_t).$$ An integer $m$ is \emph{$V_\bullet$-large} if it can be written as a sum $m=m_1+\cdots +m_t$ such that each $m_i$ is $V_i$-large. For such an integer, we write $\Hk_{V_\bullet}^{\loc(m)}=[L^mG\backslash \Gr_{V_\bullet}]$ and
$\Sat(V_\bullet)^{\loc(m)}  \in \rmP(\Hk_{V_\bullet}^{\loc(m)})$ for the descent of  $\Sat(V_\bullet)$ 
to $\Gr_{V_\bullet}$ (after shifted by $-m\dim G$ and twisted by $(-\frac m2\dim G)$).
Moreover,  $\Sat(V_\bullet)^{\loc(m)}$ is isomorphic to the $\star$-pullback (in the sense of Notation \ref{N:star pullback}) of $\boxtimes_i\Sat(V_i)^{\loc(m_i)}$ along
the perfectly smooth map $\Hk^{\loc(m)}_{V_\bullet}\to \prod_i\Hk^{\loc(m_i)}_{V_i}$ from \eqref{break chain}.

We can similarly define $\Gr^0_{V_\bullet\mid W_\bullet}=\Gr_{V_\bullet}\times_{\Gr}\Gr_{W_\bullet}$ and $\Hk^{0,\loc(m)}_{V_\bullet\mid W_\bullet}=[L^mG\backslash \Gr^0_{V_\bullet\mid W_\bullet}]$. We will occasionally need 
$$\Hk^{\loc(m)}_{U_\bullet; V_\bullet\mid W_\bullet; U'_\bullet}=[L^mG\backslash \Gr_{U_\bullet}\tilde\times\Gr_{V_\bullet\mid W_\bullet}\tilde\times \Gr_{U'_\bullet}]\subset\Hk^{\loc(m)}_{U_\bullet\boxtimes V_\bullet\boxtimes U'_\bullet\mid U_\bullet\boxtimes W_\bullet\boxtimes U'_\bullet}.$$

For a pair of nonnegative integer $(m,n)$, similarly to Definition~\ref{tlsht},  we define the restricted local shtukas $\Sht_{V_\bullet}^{\loc(m,n)}$ and $\Sht_{V_\bullet |W_\bullet}^{0,\loc(m,n)}$ (if $m-n$ is $V_\bullet$-large and $W_\bullet$-large), and $\Sht^{\loc(m,n)}_{U_\bullet; V_\bullet\mid W_\bullet; U'_\bullet}\subset \Sht^{\loc(m,n)}_{U_\bullet\boxtimes V_\bullet\boxtimes U'_\bullet\mid U_\bullet\boxtimes W_\bullet\boxtimes U'_\bullet}$ if $m-n$ is $U_\bullet\boxtimes V_\bullet \boxtimes U'_\bullet$-large and $U_\bullet\boxtimes W_\bullet\boxtimes U'_\bullet$-large.
There is the morphism
\[
\varphi^{\loc(m,n)} : \Sht_{V_\bullet}^{\loc(m,n)} \to \Hk_{V_\bullet}^{\loc(m)}.
\]
For an algebraic representation $V \in \mathrm{Rep}(\hat G)$, we define 
\begin{equation}
\label{E:tilde Sat on objects}
S(\widetilde V): = \Phi(\Sat(V)) \in \mathrm{P}(\Sht^\loc),
\end{equation}
where $\Phi$ is defined in \eqref{E:functor P(Hk) to P(S)}.
More precisely, for any pair $(m,n)$ such that $m-n$ is $V$-large and $n>0$, $S(\widetilde V)$ is represented by the perverse sheaf
\[
S(\widetilde V)^{\loc(m,n)}: = \Phi^{\loc(m,n)} ( \Sat(V)^{\loc(m)} ) \in \rmP(\Sht_V^{\loc(m,n)}) ,
\]
where $\Phi^{\loc(m,n)}$ is defined in \eqref{E:functor P(Hk) to P(S)fin}, namely the $\star$-pullback (in the sense of Notation \ref{N:star pullback}) of sheaves along $\varphi^{\loc(m,n)}$ . Similarly as above, to simplify notations sometimes we use $S(\widetilde V)^{\loc(m,n)}$ to denote the pair $(\Sht_V^{\loc(m,n)},S(\widetilde V)^{\loc(m,n)})$, if no confusion is likely to arise.

More generally, if $V_\bullet = V_1\boxtimes\cdots\boxtimes V_t$ is a representation of $\hat G^t$, then for a pair of non-negative integers $(m,n)$ with $n>0$ and $m-n$ $V_\bullet$-large, we denote
\[
S(\widetilde{V}_\bullet)^{\loc(m,n)} 
 : = \Phi^{\loc(m,n)} (\Sat(V_\bullet)^{\loc(m)}) \in \rmP(\Sht_{V_\bullet}^{\loc(m,n)}).
\]

Now for a representation $V_\bullet = V_1\boxtimes\cdots\boxtimes V_t$ of $\hat G^t$ and a representation $W$ of $\hat G$. We say a quadruple of nonnegative integers $(m_1,n_2,m_1,n_1)$ is $(V_\bullet\boxtimes W)$-acceptable if 
\begin{itemize}
\item $m_1-m_2=n_1 - n_2$ is $W$-large;
\item $m_2-n_1$ is $V_\bullet$-large.
\end{itemize}
Then we can adapt Construction~\ref{Cons:inverse F restricted} to define an analogous inverse partial Frobenius morphism
\[
F_{V_\bullet\boxtimes W}^{-1}: \Sht_{\sigma W\boxtimes V_\bullet}^{\loc(m_1,n_1)} \to \Sht_{V_\bullet\boxtimes W}^{\loc(m_2,n_2)}.
\]
Here $\sigma W$ is the $\sigma$-twist of $W$  as in Notation~\ref{N:sigma W}.

For representations $V_1, V_2$ and $W$ of $\hat G$, and nonnegative integers $(m_1, m_2, n_1, n_2)$ that is $((V_1\otimes W)\boxtimes W^*)$-acceptable and $((V_2\otimes W)\boxtimes W^*)$-acceptable, similarly to Definition~\ref{D:mS lambda mu truncated}, we define
\[
\Sht_{V_1|V_2}^{W, \loc(m_1,n_1)} : = \Sht_{V_1|\sigma W^*\boxtimes (\sigma W \otimes V_1)}^{0, \loc(m_1,n_1)} \times_{
\Sht_{(\sigma W \otimes V_1)\boxtimes W^*}^{\loc(m_2,n_2)}} \Sht_{(\sigma W \otimes V_1)\boxtimes W^*|V_2}^{0,\loc(m_2,n_2)}.
\]
\end{notation}

\begin{remark}
The newly defined spaces $\Hk_{V_\bullet}^{\loc(m)}$ and $\Sht_{V_\bullet}^{\loc(m,n)}$ are simply unions of $\Hk_{\mmu}^{\loc(m)}$  and $\Sht_{\mmu}^{\loc(m,n)}$. So we are free to use most of the results from the previous section.
\end{remark}

\subsubsection{Satake cohomological correspondences}
Recall that the geometric Satake (Theorem~\ref{T:geom Satake} and Corollary~\ref{C: rep to cor}) in particular gives a natural isomorphism between homomorphisms of $\hat G$-representations and cohomological correspondences of perverse sheaves on $\Gr$ supported on the Satake correspondences.
This isomorphism \eqref{E: rep to cor} descends to a canonical isomorphism
\begin{equation}\label{E: from Rep to Hecke corr}
{\mathscr C}^{\loc(m)}: \Hom_{\hat G}(V_{\bullet}, W_\bullet)\cong \on{Corr}_{\Hk_{V_\bullet|W_\bullet}^{0,\loc(m)}} \big(\Sat(V_\bullet)^{\loc(m)}, \Sat(W_\bullet)^{\loc(m)}\big).
\end{equation}
Here and below, we regard a $\hat G^t$-representation as a $\hat G$-representation by the diagonal restriction along $\hat G\subset \hat G^t$.

\begin{rmk}
\label{R:tensor versus box tensor}
There are different ways to interpret a tensor product representation of $\hat G$ as the diagonal restriction of an exterior tensor product representation of some power of $\hat G$. For example, we may either regard $W_1\otimes W_2\otimes W_3$ as the restriction $W_1\boxtimes W_2\boxtimes W_3$ as a representation of $\hat G^3$, or as the restriction of $(W_1\otimes W_2)\boxtimes W_3$ as a representation of $\hat G^2$. Therefore, an element $\bba\in\Hom(V, W_1\otimes W_2\otimes W_3)$ will induce a correspondence
\[{\mathscr C}^{\loc(m)}(\bba)\in\on{Corr}_{\Hk_{V\mid W_1\boxtimes W_2\boxtimes W_3}^{\loc(m)}}((\Hk_V^{\loc(m)},\Sat(V)^{\loc(m)}), (\Hk_{W_1\boxtimes W_2\boxtimes W_3}^{\loc(m)}, \Sat(W_1\boxtimes W_2\boxtimes W_3)^{\loc(m)}),\]
and a correspondence
\[{\mathscr C}^{\loc(m)}(\bba)\in\on{Corr}_{\Hk_{V\mid (W_1\otimes W_2)\boxtimes W_3}^{\loc(m)}}((\Hk_V^{\loc(m)},\Sat(V)^{\loc(m)}), (\Hk_{(W_1\otimes W_2)\boxtimes W_3}^{\loc(m)}, \Sat((W_1\otimes W_2)\boxtimes W_3)^{\loc(m)}).\]
The latter is the pushforward of the former along the convolution map
$$\Hk_{W_1\boxtimes W_2\boxtimes W_3}^{\loc(m)}\to \Hk_{(W_1\otimes W_2)\boxtimes W_3}^{\loc(m)}$$
introduced in \eqref{E: conv local Hk}.
In below, readers should be able to tell from the notations that how we regard a tensor product representation of $\hat G$ as the restriction of an exterior tensor product.
\end{rmk}

Pulling back along the natural restriction morphism $\varphi^{\loc(m,n)}$ \eqref{E:restriction morphism}, we obtain correspondences on local shtukas.
\begin{lem}
\label{L:Satake correspondence shtukas}
Let $(m,n)$ be a pair such that $n>0$ and $m-n$ is $V_\bullet$-large and $W_\bullet$-large. There is a natural homomorphism
\begin{equation}
\label{E:geometric Satake truncated mS}
\scrC^{\loc(m,n)}: 
\Hom_{\hat G}(V_{\bullet},W_\bullet) \longto \on{Corr}_{\Sht_{V_\bullet|W_\bullet}^{0,\loc(m,n)}}\big( S(\widetilde{V_{\bullet}})^{\loc(m,n)},S(\widetilde{W_{\bullet}})^{\loc(m,n)}\big),
\end{equation}
such that the following diagram is commutative (if $m-n$ is also $U_\bullet$-large), 
\begin{equation}\label{E: comp corr Sht}
\xymatrix{
\Hom_{\hat G}(U_{\bullet},V_{\bullet})\otimes \Hom_{\hat G}(V_{\bullet},W_{\bullet}) \ar[r] \ar[d] &  \on{Corr}_{\Sht_{U_\bullet|V_\bullet}^{0,\loc(m,n)}}\big(S(\widetilde{U_\bullet}),S(\widetilde{V_\bullet})\big) \otimes \on{Corr}_{\Sht_{V_\bullet|W_\bullet}^{0,\loc(m,n)}}\big(S(\widetilde{V_\bullet}),S(\widetilde{W_\bullet})\big) \ar[d]\\ 
\Hom_{\hat G}(U_{\bullet},W_{\bullet}) \ar[r]& \on{Corr}_{\Sht_{U_\bullet|W_\bullet}^{0,\loc(m,n)}}\big(S(\widetilde{U_\bullet}), S(\widetilde{W_\bullet})\big),
}
\end{equation}
where we omit $\loc(m,n)$ in the superscripts, the right vertical map is given by the pushforward of the composition of cohomological correspondences along the perfectly proper morphism
\[
\on{Comp}^{\loc(m,n)}: \Sht_{U_\bullet|V_\bullet}^{0,\loc(m,n)} \times_{\Sht_{V_\bullet}^{\loc(m,n)}} \Sht_{V_\bullet|W_\bullet}^{0,\loc(m,n)} \longto \Sht_{U_\bullet|V_\bullet}^{0,\loc(m,n)}.
\]
as in the middle vertical arrow of  \eqref{E:Satake correspondence cartesian}.
\end{lem}
\begin{proof}
Recall the commutative diagram from \eqref{E:S and Hk Satake compatibility}, where the vertical maps between the second and the third arrows are smooth. Then $\scrC^{\loc(m,n)}$ is the pullback of \eqref{E: from Rep to Hecke corr} along these by the construction in \S \ref{ASS:smooth pullback correspondence}.

To prove the \eqref{E: comp corr Sht} is commutative, we use the right Cartesian square of \eqref{E:Satake correspondence cartesian} to reduce the question to $\mathscr{C}^{\loc(m)}$ via Lemma \ref{AL:pushforward pullback compatibility}, which in turn follows from Corollaries \ref{SS:cycles of geometric Satake} and \ref{C: rep to cor}.
\end{proof}

\begin{ex}\label{Ex: creation and annilhilation}
For a representation $W_\bullet$ of $\hat G^t$, there is a canonical unit map 
$$\delta_{W_\bullet}: \mathbf{1}\to W_\bullet^*\otimes W_\bullet$$ and a counit map 
$$e_{W_\bullet}: W_\bullet\otimes W_\bullet^*\to \mathbf{1}.$$ The corresponding cohomological correspondence 
$$\scrC^{\loc(m,n)}(\delta_{W_\bullet}\otimes \id_{V_\bullet}): S(\widetilde{V_\bullet})^{\loc(m,n)}\to S(\widetilde{W^*_\bullet}\boxtimes(\widetilde{W_\bullet}\otimes \widetilde{V_\bullet}))^{\loc(m,n)}$$ and 
$$\scrC^{\loc(m,n)}(  \id_{V_\bullet}\otimes e_{W_\bullet} ): S((\widetilde{V_\bullet}\otimes \widetilde{W_\bullet})\boxtimes \widetilde{W_\bullet^*})^{\loc(m,n)}\to S(\widetilde{V_\bullet})^{\loc(m,n)}$$ 
are the local version of V. Lafforgue's creation and annihilation operators. 

The reason we regard $W_\bullet^*\otimes W_\bullet\otimes V_\bullet$ as the diagonal restriction of the $\hat G^2$-representation $W_\bullet^*\boxtimes(W_\bullet\otimes V_\bullet)$ rather than the diagonal restriction of the $\hat G^3$-representation $W_\bullet^*\boxtimes W_\bullet\boxtimes V_\bullet$ is to make the support of the cohomological correspondence defined in Construction \ref{Cons:CW} below equal to $\Sht_{V_1\mid V_2}^{W,\loc(m_1,n_1)}$.
\end{ex}

\begin{rmk}
\label{E: true support}
A homomorphism $\bba \in \Hom_{\hat G}(V_{\bullet}, W_{\bullet})$ induces a homomorphism 
\[
\id_{U_\bullet}\otimes\bba\otimes\id_{U'_\bullet} \in \Hom_{\hat G}(U_\bullet\otimes V_\bullet\otimes U'_\bullet, U_{\bullet}\otimes W_\bullet\otimes U'_\bullet).
\]
In this case, the correspondence $\scrC^{\loc(m,n)}(\id_{U_\bullet}\otimes\bba\otimes\id_{U'_\bullet})$ is the pushforward of a natural cohomological correspondence supported on $\Sht_{U_\bullet; V_\bullet\mid W_\bullet; U'_\bullet}^{0,\loc(m,n)}$.
Indeed, it is enough to prove this if $\bba$ belongs to the Satake basis. 
But this follows from the fact that the cohomological correspondence from $(\Gr_{U_\bullet\boxtimes V_\bullet\boxtimes U'_\bullet},\Sat(U_\bullet\boxtimes V_\bullet\boxtimes U'_\bullet))$ to $(\Gr_{U_\bullet\boxtimes W_\bullet\boxtimes U'_\bullet}, \Sat(U_\bullet\boxtimes W_\bullet\boxtimes U'_\bullet))$ given by $\id_{U_\bullet}\otimes\bba\otimes\id_{U'_\bullet}$ (under the isomorphism \eqref{E: rep to cor}) is the pushforward of a natural cohomological correspondence supported on $\Gr_{U_\bullet}\tilde{\times}\Gr_{V_\bullet\mid W_\bullet}^{0,\bba}\tilde\times\Gr_{U'_\bullet}$.

In the sequel, when we explicitly write a $\hat G$-homomorphism $U_\bullet\otimes V_\bullet\otimes U'_\bullet\to U_{\bullet}\otimes W_\bullet\otimes U'_\bullet$ as $\id_{U_\bullet}\otimes\bba\otimes\id_{U'_\bullet}$, we regard $\scrC^{\loc(m,n)}(\id_{U_\bullet}\otimes\bba\otimes\id_{U'_\bullet})$ as a cohomological correspondence supported on $\Sht_{U_\bullet; V_\bullet\mid W_\bullet; U'_\bullet}^{0,\loc(m,n)}$. 
\end{rmk}

The inverse partial Frobenius morphism $F^{-1}_{V_\bullet\boxtimes W}: \Sht_{\sigma W,V_\bullet}^{\loc(m_1,n_1)}\to \Sht_{V_\bullet, W}^{\loc(m_2,n_2)}$ between moduli of restricted shtukas can be also upgraded to a cohomological correspondence.
\begin{lem}
\label{L:partial Frobenius cohomological correspondence}
If $V_\bullet$ is a representations of $\hat G^t$ and $W$ is a representation of $\hat G$, then there is a natural cohomological correspondence
\begin{equation}
\label{E:F pulls back IC to IC}
\DD\Ga_{F^{-1}_{V_\bullet}\boxtimes W}^*:\big( \Sht_{\sigma W\boxtimes V_\bullet}^{\loc(m_1,n_1)}, S(\widetilde{\sigma W}\boxtimes \widetilde{V_\bullet})^{\loc(m_1,n_1)} \big) \longto \big(\Sht_{V_\bullet\boxtimes W}^{\loc(m_2,n_2)},S(\widetilde{V_\bullet}\boxtimes\widetilde{W})^{\loc(m_2,n_2)} \big).
\end{equation}
\end{lem}
\begin{proof}
We construct $\DD\Ga_{F^{-1}_{V_\bullet\boxtimes W}}^*$ as follows.
We rewrite commutative diagram \eqref{E:pf as pf} as
\[\xymatrix{
\Sht_{\sigma W\boxtimes  V_{\bullet}}^{\loc(m_1,n_1)}\ar@{=}[r] \ar[d] & \Sht_{\sigma W\boxtimes  V_{\bullet}}^{\loc(m_1,n_1)}\ar[d] \ar[r]^-{F^{-1}_{V_\bullet\boxtimes W}}
& \Sht_{V_{\bullet}\boxtimes W}^{\loc(m_2,n_2)} \ar[d]
\\
\Hk^{\loc(n_1)}_{\sigma W}\times \Hk^{\loc(m_2)}_{V_\bullet}& \ar[l]_{\sigma\times\id} \Hk^{\loc(n_1)}_{W}\times \Hk^{\loc(m_2)}_{V_{\bullet}} \ar@{=}[r] &
 \Hk^{\loc(m_2)}_{V_{\bullet}}\times \Hk^{\loc(n_1)}_{W},
}\]
where vertical maps are perfectly smooth. By \eqref{E: Gal equiv Sat}, for every representation $W$ of $\hat G$ and a $W$-large integer $m$, there is a canonical isomorphism 
\[\sigma^*\Sat(\sigma W)^{\loc(m)}\cong \Sat(W)^{\loc(m)},\]
as perverse sheaves on $\Hk_{W}^{\loc(m)}$, which by Example \ref{Ex:examples of correspondences} (2) defines a natural correspondence
\begin{small}
\[
\Ga_{\sigma\times\id}^*:(\Hk_{\sigma W}^{\loc(m)}\times\Hk_{V_\bullet}^{\loc(m')},\Sat(\sigma W)^{\loc(m)}\boxtimes\Sat(V_\bullet)^{\loc(m')}) \to (\Hk_{W}^{\loc(m)}\times\Hk_{V_\bullet}^{\loc(m')},\Sat(W)^{\loc(m)}\boxtimes\Sat(V_\bullet)^{\loc(m')}),
\]
\end{small}supported on the second row.
Then $\DD\Ga_{F^{-1}_{V_\bullet\boxtimes W}}^*$ is the pullback of the correspondence $\Ga_{\sigma\times\id}^*$ via the construction \S \ref{ASS:smooth pullback correspondence}.
\end{proof}

Later we will make use of the following two lemmas.
\begin{lem}
\label{L: double pFrob to single}
The pushforward of $\DD\Ga_{F^{-1}_{V_\bullet\boxtimes W_1\boxtimes W_2}}\circ \DD\Ga_{F^{-1}_{\sigma W_2\boxtimes V_\bullet\boxtimes W_1}}: S(\widetilde{\sigma W_1}\boxtimes\widetilde{\sigma W_2}\boxtimes \widetilde{V_\bullet})\to S(\widetilde{V_\bullet}\boxtimes \widetilde{W_1}\boxtimes \widetilde{W_2})$ along the convolution map
\[\begin{CD}
\Sht^{\loc(m_1,n_1)}_{\sigma W_1\boxtimes \sigma W_2\boxtimes V_\bullet}@= \Sht^{\loc(m_1,n_1)}_{\sigma W_1\boxtimes \sigma W_2\boxtimes V_\bullet}@>>>\Sht^{\loc(m_3,n_3)}_{V_\bullet\boxtimes W_1\boxtimes W_2}\\
@VVV@VVV@VVV\\
\Sht^{\loc(m_1,n_1)}_{\sigma (W_1\otimes W_2)\boxtimes V_\bullet}@= \Sht^{\loc(m_1,n_1)}_{\sigma (W_1\otimes W_2)\boxtimes V_\bullet}@>>> \Sht^{\loc(m_3,n_3)}_{V_\bullet\boxtimes (W_1\otimes W_2)}
\end{CD}\]
is equal to $\DD\Ga_{F^{-1}_{V_\bullet\boxtimes (W_1\otimes W_2)}}$.
\end{lem}
\begin{proof}By definition and Lemma \ref{AL:pullback compatible with composition}, $\DD\Ga_{F^{-1}_{V_\bullet\boxtimes W_1\boxtimes W_2}}\circ \DD\Ga_{F^{-1}_{\sigma W_2\boxtimes V_\bullet\boxtimes W_1}}$ is equal to the pullback of the correspondence $\Ga_{\sigma_{W_1}\times\sigma_{W_2}\times \id_{V_\bullet}}^*$, where 
$$\sigma_{W_1}\times\sigma_{W_2}\times\id_{V_\bullet}: \Hk_{W_1}^{\loc(m_1)}\times\Hk_{W_2}^{\loc(m_2)}\times\Hk_{V_\bullet}^{\loc(m')}\to \Hk_{\sigma W_1}^{\loc(m_1)}\times\Hk_{\sigma W_2}^{\loc(m_2)}\times\Hk_{V_\bullet}^{\loc(m')}.$$
Note that the map $\Sht^{\loc(m_1,n_1)}_{W_1\boxtimes  W_2\boxtimes V_\bullet}\to \Hk_{W_1}^{\loc(m_1)}\times\Hk_{W_2}^{\loc(m_2)}\times\Hk_{V_\bullet}^{\loc(m')}$ factors through $\Sht^{\loc(m_1,n_1)}_{W_1\boxtimes  W_2\boxtimes V_\bullet}\to \Hk_{ W_1\boxtimes W_2}^{\loc(m_1)}\times\Hk_{V_\bullet}^{\loc(m')}$. Then $\DD\Ga_{F^{-1}_{V_\bullet\boxtimes W_1\boxtimes W_2}}\circ \DD\Ga_{F^{-1}_{\sigma W_2\boxtimes V_\bullet\boxtimes W_1}}$ is equal to the pullback of the correspondence $\Ga_{\sigma_{W_1\boxtimes W_2}\times \id_{V_\bullet}}^*$, where 
$$\sigma_{W_1\boxtimes W_2}\times\id_{V_\bullet}: \Hk_{W_1\boxtimes W_2}^{\loc(m_1)}\times\Hk_{V_\bullet}^{\loc(m')}\to \Hk_{\sigma (W_1\boxtimes W_2)}^{\loc(m_1)}\times\Hk_{V_\bullet}^{\loc(m')}.$$
Note that the following diagram is Cartesian
\[\begin{CD}
\Sht^{\loc(m,n)}_{W_1\boxtimes  W_2\boxtimes V_\bullet}@>>> \Sht^{\loc(m,n)}_{(W_1\otimes  W_2)\boxtimes V_\bullet}\\
@VVV@VVV\\
\Hk_{W_1\boxtimes W_2}^{\loc(m)}\times \Hk_{V_\bullet}^{\loc(m')}@>>>\Hk^{\loc(m)}_{W_1\otimes W_2}\times \Hk_{V_\bullet}^{\loc(m')}.
\end{CD}
\]
Now apply Lemma \ref{AL:pushforward pullback compatibility} to conclude.
\end{proof}

\begin{lem}
\label{L: Sat corr comm with pFrob}
(1) Let $\bba\in\Hom_{\hat G}(U_\bullet, V_\bullet)$, and regard $\scrC^{\loc(m,n)}(\bba\otimes\id_W)$ (resp. $\scrC^{\loc(m,n)}(\id_{\sigma W}\otimes \bba$) as a cohomological correspondence supported on $\Sht^{0,\loc(m,n)}_{U_\bullet\mid V_\bullet;W}$ (resp. $\Sht^{0,\loc(m,n)}_{\sigma W; U_\bullet\mid V_\bullet}$) as in Remark \ref{E: true support}. Then
\[\scrC^{\loc(m,n)}(\bba\otimes\id_W)\circ \DD\Ga_{F^{-1}_{U_\bullet\boxtimes W}}= \DD\Ga_{F^{-1}_{V_\bullet\boxtimes W}}\circ \scrC^{\loc(m,n)}(\id_{\sigma W}\otimes \bba).\]

(2) Similarly, let $\bbb\in\Hom_{\hat G}(W, W')$. Then
\[\scrC^{\loc(m,n)}(\id_{U_\bullet}\otimes \bbb)\circ  \DD\Ga_{F^{-1}_{U_\bullet\boxtimes W}}=  \DD\Ga_{F^{-1}_{U_\bullet\boxtimes W'}}\circ \scrC^{\loc(m,n)}(\id_{U_\bullet}\otimes\sigma\bbb)\]
\end{lem}
\begin{proof}
We prove (1) and the proof of (2) is similar.
By Lemma \ref{L:technical lemma for comp loc restricted}(3), the supports of this two compositions of cohomological correspondences coincide. Now by definition and by Lemma \ref{AL:pullback compatible with composition}, the statement follows from the evident equality
\[(\scrC^{\loc(m)}(\bba)\times\id_{W})\circ \Ga_{\id_{U_\bullet}\times \sigma}^*= \Ga_{\id_{V_\bullet}\times\sigma}^*\circ(\scrC^{\loc(m)}(\bba)\times\id_{\sigma W})\]
as cohomological correspondences from $( \Hk_{\sigma W}^{\loc(m')} \times \Hk_{U_\bullet}^{\loc(m)},\Sat(\sigma W)^{\loc(m')})\boxtimes \Sat(U_\bullet)^{\loc(m)} $ to  $(\Hk_{W}^{\loc(m')}\times \Hk_{V_\bullet}^{\loc(m)}, \Sat(W)^{\loc(m')}\boxtimes 
\Sat(V_\bullet)^{\loc(m)})$.
\end{proof}

\subsection{Cohomological correspondences of perverse sheaves on the moduli of local shtukas}

\subsubsection{Sketch of the construction of $S$}
\label{SS:proof of periodic geometric Satake}
We now construct the natural functor  
\[
S: \Coh^{\hat G}_{fr}(\hat G \sigma) \to \rmP^\mathrm{Corr}(\Sht^\loc)
\]
so that the diagram \eqref{E:periodic geometric Satake} is commutative and hence prove Theorem~\ref{T:periodic geometric Satake}.

Every object in $\Coh^{\hat G}_{fr}(\hat G \sigma)$ is of the form $\widetilde V$ for some $V \in \Rep(\hat G)$. We define 
\[
S(\widetilde V): = \Phi (\Sat(V)) \in \rmP^\mathrm{Corr}(\Sht^\loc),
\]
which is represented by $\Sat(V)^{\loc(m,n)} \in \rmP(\Sht_V^{\loc(m,n)})$ for any pair $(m,n)$ of non-negative integers such that $m$ is positive and $m-n$ is $V$-large.
It is clear from the definition that the diagram \eqref{E:periodic geometric Satake} of functors is commutative for objects.
\medskip
We are left to define the functor $S$ on the morphisms, that is, for each pair of representations $V_1, V_2 \in \Rep(\hat G)$ a natural homomorphism
\begin{equation}
\label{E:widetilde Sat on hom}
S: \mathrm{Hom}_{\Coh_{fr}^{\hat G}(\hat G\sigma)} \big(\widetilde{V_1}, \widetilde{V_2} \big) \stackrel{\eqref{E:hom in Vect = J}}\cong \bfJ(V_1^*\otimes V_2) \longto \mathrm{Mor}_{\rmP^\mathrm{Corr}(\Sht^\loc)}\big(S(\widetilde{V_1}), S(\widetilde{V_2})\big)
\end{equation}
which is compatible with compositions. Moreover, we need to show the commutativity of functors in \eqref{E:periodic geometric Satake} for morphisms. We overview the construction and the proof below, and leave the details to later in this subsection.

For representations $W, V_1, V_2 \in \Rep(\hat G)$, there is a natural homomorphism
\[
\Xi_{W}: \Hom_{\hat G}(\sigma W \otimes V_1 \otimes W^*, V_2)\to (\calO_{\hat G\sigma} \otimes V_1^* \otimes V_2)^{\hat G} = \bfJ(V_1^* \otimes V_2)=\Hom_{\Coh_{fr}^{\hat G}(\hat G\sigma)}(\widetilde{V_1},\widetilde{V_2})
\]
sending $\bba \in \Hom_{\hat G}(\sigma W \otimes V_1 \otimes W^*, V_2)$ to the function
\begin{equation}
\label{E:XiW2}
\Xi_W(\bba): g \mapsto \sum_i \bba (g\cdot e_i \otimes e_i^* ) \in \Hom_{\overline \QQ_\ell}(V_1, V_2) \cong V_1^* \otimes V_2,
\end{equation}
where $\{e_i\}_i$ is a basis of $W  (=\sigma W)$, and $\{e_i^*\}_i$ its dual basis.
Moreover, it follows from the Peter-Weyl theorem that every element of $\bfJ(V_1^* \otimes V_2)$ is in the image of some $\Xi_W$. 
So to define $S$ on morphisms, we will first define in Construction~\ref{Cons:CW} a natural map
\[
\scrC_W: \Hom_{\hat G}(\sigma W \otimes V_1 \otimes W^*, V_2) \longto \mathrm{Mor}_{\rmP^\mathrm{Corr}(\Sht^\loc)}\big(S(\widetilde{V_1}), S(\widetilde{V_2})\big)
\]
for each representation $W$ of $\hat G$. Then for every element $\bbc \in \bfJ(V_1^* \otimes V_2)$, we write it as $\bbc = \Xi_{W} (\bba)$ for some representation $W$ and some element $\bba \in \Hom_{\hat G}(\sigma W \otimes V_1 \otimes W^*, V_2)$, and define 
\begin{equation}
\label{E:definition of tilde Sat on morphisms}
S(\bbc): = \scrC_W(\bba) \in \mathrm{Mor}_{\rmP^\mathrm{Corr}(\Sht^\loc)}\big(S(\widetilde{V_1}), S(\widetilde{V_2})\big).
\end{equation}
We will show that
\begin{itemize}
\item (Lemma~\ref{C:tilde Sat well defined}) the definition of
$S(\bbc)$ is independent of the choice of the representation $W$ and the element $\bba$, and
\item (Lemma~\ref{L:CW compatible with tensor}) this definition of $S$ on morphisms is compatible with compositions of morphisms.
\end{itemize}
Therefore, we obtained a well-defined functor $S: \Coh^{\hat G}_{fr}(\hat G \sigma) \to \rmP^\mathrm{Corr}(\Sht^\loc)$.
Finally, we will check in Lemma~\ref{L:commutativity of periodic geometric Satake diagram} that the diagram \eqref{E:periodic geometric Satake} of functors is commutative for morphisms.

To sum up, after we construct the map $\scrC_W$ below, the proof of Theorem~\ref{T:periodic geometric Satake} (1) is divided into Lemmas~\ref{C:tilde Sat well defined}, \ref{L:CW compatible with tensor}, and \ref{L:commutativity of periodic geometric Satake diagram}.
\hfill $\Box$

\begin{construction}
\label{Cons:CW}
We now define the natural map $\scrC_W$. For this, we take a quadruple $(m_1,n_1,m_2,n_2)$ that is $(V_1\otimes W)\boxtimes W^*$-acceptable and $(V_2\otimes W)\boxtimes W^*$-acceptable. Given an element $\bba \in \Hom_{\hat G}(\sigma W \otimes V_1 \otimes W^*, V_2)$, we define $\scrC_W(\bba) \in \mathrm{Mor}_{\rmP^\mathrm{Corr}(\Sht^\loc)}\big(S(\widetilde{V_1}), S(\widetilde{V_2})\big)$ to be the morphism represented by the cohomological correspondence in
\[
\mathrm{Corr}_{\Sht_{V_1\mid V_2}^{W,\loc(m_1,n_1)}} \big( S(\widetilde{V_1})^{\loc(m_1,n_1)}, 
S(\widetilde{V_2})^{\loc(m_2,n_2)} \big)
\]
given by the following composition,
\begin{align*}
 S(\widetilde{V_1})^{\loc(m_1,n_1)}  \xrightarrow{\scrC^{\loc(m_1, n_1)}(\delta_{\sigma W} \otimes \id_{V_1})}  & S(\widetilde{\sigma W^*}\boxtimes (\widetilde{\sigma W}\otimes\widetilde{V_1}))^{\loc(m_1,n_1)}&
\\
\xrightarrow{\DD \Gamma^*_{F^{-1}_{(\sigma W \otimes V_1)\boxtimes W^*}} \textrm{ of \eqref{E:F pulls back IC to IC}}}  & S((\widetilde{\sigma W} \otimes \widetilde{V_1})\boxtimes\widetilde{W^*})^{\loc(m_2,n_2)} \xrightarrow{\scrC^{\loc(m_2, n_2)}(\bba)} S(\widetilde{V_2})^{\loc(m_2,n_2)} .
\end{align*}
The definition of $\scrC_W(\bba)$ as a morphism in $\mathrm{Mor}_{\rmP^\mathrm{Corr}(\Sht^\loc)}\big(S(\widetilde{V_1}), S(\widetilde{V_2}) \big)$ is independent of the choice of $(m_1,n_1, m_2, n_2)$ in a natural way: if $(m'_1,n'_1, m'_2, n'_2)\geq (m_1,n_1,m_2,n_2)$ is another quadruple of nonnegative integers that is $(V_1\otimes W)\boxtimes W^*$-acceptable and $(V_2\otimes W)\boxtimes W^*$-acceptable, then $\scrC_W^{\loc(m'_1,n'_1)}(\bba)$ is the smooth pullback (in the sense of \S\ref{ASS:smooth pullback correspondence}) of $\scrC_W^{\loc(m_1,n_1)}(\bba)$ along the following diagram
\begin{equation}
\label{E: change of m,n in Hk corr}
\xymatrix{
\Sht_{V_1}^{\loc(m'_1,n'_1)} \ar[d] & \ar[l] \Sht_{V_1|V_2}^{W, \loc(m'_1,n'_1)} \ar[r] \ar[d]  & \Sht_{V_2}^{\loc(m'_2,n'_2)} \ar[d] 
\\
\Sht_{V_1}^{\loc(m_1,n_1)}& \ar[l] \Sht_{V_1|V_2}^{W, \loc(m_1,n_1)} \ar[r] & \Sht_{V_2}^{\loc(m_2,n_2)},
}
\end{equation}
where the vertical arrows are perfectly smooth and the left square is Cartesian. This follows from Lemma \ref{AL:pullback compatible with composition}, the Cartesian diagram \eqref{E:shtukas restriction Cartesian with Satake correspondences}, and the commutative diagram \eqref{E: pFrob v.s. res. pFrob}.
\end{construction}

We need to show the definition of $S$ on morphisms using $\scrC_W$ is independent of choices. First, we have
\begin{lem}
\label{L:alternative CW}
Let $\bba'$ denote the image of $\bba$ under the isomorphism $\Hom(\sigma W\otimes V_1\otimes W^*, V_2)\cong \Hom (V_1, \sigma W^*\otimes V_2\otimes W)$. Then $\scrC_W(\bba)$ can also be computed as
\begin{align*}
 S(\widetilde{V_1})^{\loc(m_1,n_1)}  \xrightarrow{\scrC^{\loc(m_1, n_1)}(\bba')}  & S(\widetilde{\sigma W^*}\boxtimes (\widetilde{V_2}\otimes \widetilde{W}))^{\loc(m_1,n_1)}&
\\
\xrightarrow{\DD \Gamma^*_{F^{-1}_{(V_2\otimes W)\boxtimes W^*}} \textrm{ of \eqref{E:F pulls back IC to IC}}}  & S((\widetilde{V_2}\otimes \widetilde{W})\boxtimes\widetilde{W^*})^{\loc(m_2,n_2)} \xrightarrow{\scrC^{\loc(m_2, n_2)}(\id_{V_2}\otimes e_W)} S(\widetilde{V_2})^{\loc(m_2,n_2)} .
\end{align*}
\end{lem}
\begin{proof}
Indeed, let $\bba''$ be the image of $\bba$ under the isomorphism $\Hom(\sigma W\otimes V_1\otimes W^*, V_2)\cong \Hom(\sigma W\otimes V_1, V_2\otimes W)$. Then the lemma follows from the following commutative diagram (where we omit subscripts involving $\loc(m_1,n_1)$ and $\loc(m_2,n_2)$).
\[\xymatrix@C=30pt{
S(\widetilde{V_1})\ar^-{\scrC(\delta_{\sigma W}\otimes\id_{V_1})}[rr]\ar_-{\scrC(\bba')}[drr]&& S(\widetilde{\sigma W^*}\boxtimes(\widetilde{\sigma W}\otimes \widetilde{V_1}))\ar^-{\bD\Gamma^*_{F^{-1}}}[r]\ar^-{\scrC(\id_{\sigma W^*}\otimes \bba'')}[d] & S((\widetilde{\sigma W}\otimes\widetilde{V_1})\boxtimes \widetilde{W^*})\ar^-{\scrC(\bba)}[rr]\ar^{\scrC(\bba''\otimes\id_W)}[d]&& S(\widetilde{V_2}),\\
&& S(\widetilde{\sigma W^*}\boxtimes(\widetilde{V_2}\otimes \widetilde{W}))\ar^{\bD\Gamma^*_{F^{-1}}}[r]& S((\widetilde{V_2}\otimes\widetilde{W})\boxtimes \widetilde{W^*}) \ar_-{\ \ \scrC(\id_{V_2}\otimes e_W)}[urr]
}\]
where the left and the right triangles are commutative due to equalities
\[(\id_{\sigma W^*}\otimes \bba'')\circ  (\delta_{\sigma W}\otimes \id_{V_1})= \bba' \textrm{\quad and\quad} (\id_{V_2}\otimes e_W)\circ (\bba''\otimes \id_{W^*})=\bba,\]
and the middle square is commutative due to Lemma \ref{L: Sat corr comm with pFrob}.
\end{proof}
\begin{rmk}
\label{R:alternative CW}
The above proof is similar to an argument in \cite[Lemma 10.1]{La}.
In fact, the proof implies the followings slightly more general statements. Let $f_1\otimes f_2: W_1\otimes W_2\to W'_1\otimes W'_2$ be a homomorphism of $\hat G\times \hat G$-modules. Let $\bbb\in\Hom_{\hat G}(V_1,\sigma W_1\otimes V_2\otimes W_2)$ and $\bbb'\in\Hom_{\hat G}(V_2 \otimes W'_2\otimes W'_1 ,V_2)$. Then
\begin{equation}
\label{E:change across Frobenius} \scrC(\bbb'\circ (\id_{V_2}\otimes f_1\otimes f_2))\circ \bD\Gamma^*_{F^{-1}}\circ \scrC(\bbb)=  \scrC(\bbb')\circ \bD \Gamma^*_{F^{-1}}\circ \scrC((\sigma f_2\otimes\id_{V_2}\otimes f_1)\circ \bbb)\end{equation}
as elements in $\Hom_{\on{P}^{\on{Corr}}(\Sht^{\loc})}(S(\widetilde{V_1}),S(\widetilde{V_2}))$. Here, as above we omit subscripts involving $\loc(m_1,n_1)$ and $\loc(m_2,n_2)$
\end{rmk}


\begin{lemma}
\label{C:tilde Sat well defined}
The definition of $S$ on morphisms in \eqref{E:definition of tilde Sat on morphisms} is independent of auxiliary choices. More precisely, for $V_1, V_2 \in \Rep(\hat G)$ and $\bbc \in J(V_1^* \otimes V_2)$, if we write $\bbc = \Xi_W(\bba)$ for some representation $W$ and some element $\bba \in\Hom(\sigma W\otimes V_1\otimes W^*,V_2) $, then the cohomological correspondence $\scrC_W(\bba)    \in \mathrm{Mor}_{\rmP^\mathrm{Corr}(\Sht^\loc)}\big(S(\widetilde{V_1}), S(\widetilde{V_2})\big)$ defined in Construction~\ref{Cons:CW} depends only on $\bbc$ but not on the particular choice of $W$ and $\bba$.
\end{lemma}

\begin{proof}
We consider $\mO_{\hat G}$ as an ind-object in $\on{Rep}(\hat G)$ via the \emph{right translation}. Then for every representation $U$ of $\hat G$, the action map $\hat G\times U\to U$ induces a $\hat G$-equivariant map 
\[
\on{act}_U: U\to  \mO_{\hat G}\otimes\underline{U}, \qquad u \mapsto \on{act}_U(u)(g): = gu,
\] where $\underline{U}$ means the underlying vector space of $U$ equipped with the \emph{trivial} $\hat G$-action.
We have a similar $\hat G$-equivariant map 
\[
m_U:\underline{U^*}\otimes U\to \mO_{\hat G}, \qquad (u^*,u)\mapsto m_U(u^*,u)(g): =u^*(gu).
\]
Combining the two, we have a natural homomorphism of $\hat G \times \hat G$-representations
\[
\sigma W^*\otimes V_2\otimes W\xrightarrow{\on{act}_{\sigma W^*}} \underline{W^*}\otimes \sigma\mO_{\hat G}\otimes V_2\otimes W\xrightarrow{m_W} \sigma \mO_{\hat G}\otimes V_2\otimes \mO_{\hat G}.
\]

Let $\bba'$ denote the image of $\bba$ under the isomorphism $\Hom(\sigma W\otimes V_1\otimes W^*, V_2)\cong \Hom (V_1, \sigma W^*\otimes V_2\otimes W)$.
The map $ \hat G \times \hat G \to \hat G\sigma$ sending $(g,h) \mapsto \sigma(g)^{-1} \sigma(h) \sigma$ induces a natural map $d_\sigma: \calO_{\hat G \sigma} \to \sigma\calO_{\hat G} \otimes \calO_{\hat G}$, intertwining the (twisted) conjugation action on $\calO_{\hat G \sigma}$ and the diagonal action on $\sigma \calO_{\hat G} \otimes \calO_{\hat G}$.
Let $d_\sigma(\bbc')$ denote the image of $\bbc$ under the following map
\[
\bfJ(V_1^*\otimes V_2) \cong (\calO_{\hat G \sigma} \otimes V_1^* \otimes V_2)^{\hat G} \xrightarrow{d_\sigma} (\sigma\calO_{\hat G} \otimes \calO_{\hat G} \otimes V_1^*\otimes V_2 \otimes )^{\hat G} \cong \Hom_{\hat G}(V_1, \sigma\calO_{\hat G} \otimes V_2 \otimes \calO_{\hat G}).
\]
It is straightforward  to check that the following two diagrams are  tautologically commutative.
\[
\xymatrix@C=50pt{
V_1 \ar@{=}[d] \ar[r]^-{\bba'} &
\sigma W^* \otimes V_2\otimes W  \ar[d]^{m_W \circ \on{act}_{\sigma W^*}} &
V_2 \otimes W \otimes W^* \ar[d]_{m_W \circ \on{act}_{W^*}} \ar[r]^-{\id_{V_2} \otimes e_W} & V_2 \ar@{=}[d]
\\
V_1  \ar[r]^-{d_\sigma(\bbc')} & \sigma \calO_{\hat G}\otimes V_2 \otimes \calO_{\hat G}  &  V_2 \otimes \calO_{\hat G} \otimes \calO_{\hat G} \ar[r]^-{\mathrm{ev}_{(1,1)}} & V_2,
}
\]
where $\mathrm{ev}_{(1,1)}$ is the evaluation map at $(1,1) \in \hat G \times \hat G$. Indeed, for $v_1 \in V_1$, its has the same image
\[
(g,h)\mapsto \sum_i  \big\langle \bba'(v_1), \sigma(g)^{-1} h^{-1} \cdot e_i \otimes e_i^*\big\rangle
\] under both $d_\sigma(\bbc')$ and $m_W \circ \on{act}_{\sigma W^*} \circ \bba'$, where $\{e_i\}_i$ is a basis of $W(=\sigma W)$ and $\{
e_i^*\}_i$ its dual basis.

Now apply Remark \ref{R:alternative CW} above to the case
\[
W_1 \otimes W_2 := W\otimes W^*, \  W'_1 \otimes W'_2: = \calO_{\hat G} \otimes \calO_{\hat G}, \  f_1\otimes f_2 : = m_W \circ \on{act}_{W^*}, \  \bbb := \bba', \  \textrm{and} \  \bbb': = \mathrm{ev}_{(1,1)}.
\]
We deduce that
\begin{align*}
\scrC_W(\bba)& \stackrel{\textrm{Lemma~\ref{L:alternative CW}}}= \scrC(\id_{V_2} \otimes e_W) \circ \DD \Gamma^*_{F^{-1}_{(V_2 \otimes W) \boxtimes W^*}} \circ \scrC(\bba')
\\
& \stackrel{\textrm{Remark~\ref{R:alternative CW}}}= \scrC(\mathrm{ev}_{(1,1)}) \circ \DD \Gamma^*_{F^{-1}_{(V_2 \otimes \calO_{\hat G}) \boxtimes \calO_{\hat G}}} \circ \scrC(d_\sigma(\bbc')).
\end{align*}
But the last expression depends only on $\bbc$ so the lemma follows.
Of course, the map $d_\sigma(\bbc'): V_1\to \sigma\mO_{\hat G}\otimes V_2\otimes \mO_{\hat G}$ only lands in some finite dimensional submodule $\sigma W_2\otimes V_2\otimes W_1$ so more precisely we replace $\sigma\mO_{\hat G}\times \mO_{\hat G}$ by $\sigma W_2\otimes W_1$ in the above expression and choose appropriate $(m_1,n_1,m_2,n_2)$ to define the cohomological correspondences.
\end{proof}

\begin{rmk}
Of course, one can just define $S(\bbc)$ via the last expression, which then is independent of any choices. But in practice, e.g. in \S \ref{Sec:S=T}, we sometimes need to express $\bbc$ as $\Xi_W(\bba)$ and realize $S(\bbc)$ as $\scrC_W(\bba)$.
\end{rmk}

\begin{lem}
\label{L:CW compatible with tensor}
For representations $V_1$, $V_2$, $V_3$, $W_1$, $W_2$ of $\hat G$, we have the following commutative diagram
\begin{small}
\begin{equation}
\label{E:composition commutes with CW}
\xymatrix@C=10pt{
\Hom_{\Coh_{fr}^{\hat G}(\hat G\sigma)}(\widetilde{V_1},\widetilde{V_2}) \otimes \Hom_{\Coh_{fr}^{\hat G}(\hat G\sigma)}(\widetilde{V_2},\widetilde{V_3}) \ar[r] & \Hom_{\Coh_{fr}^{\hat G}(\hat G\sigma)}(\widetilde{V_1},\widetilde{V_3})
\\
\ar[u]_{\Xi_{W_1} \otimes \Xi_{W_2}}
\Hom_{\hat G}(\sigma W_1\otimes V_1 \otimes W_1^* , V_2) \otimes \Hom_{\hat G}(\sigma W_2\otimes V_2 \otimes W_2^* , V_3) \ar[d]^{\scrC_{W_1} \otimes \scrC_{W_2}} \ar[r] &  
\ar[u]_{\Xi_{W_2 \otimes W_1}}
\Hom_{\hat G}(\sigma W_2 \otimes \sigma W_1 \otimes V_1\otimes W_1^* \otimes W_2^* , V_3)
\ar[d]^{\scrC_{W_2 \otimes W_1}}
\\
\mathrm{Mor}_{\rmP^\mathrm{Corr}(\Sht^\loc)}\big(S(\widetilde{V_1}), S(\widetilde{V_2})\big) \otimes \mathrm{Mor}_{\rmP^\mathrm{Corr}(\Sht^\loc)}\big(S(\widetilde{V_2}), S (\widetilde{V_3})\big) \ar[r] & \mathrm{Mor}_{\rmP^\mathrm{Corr}(\Sht^\loc)}\big(S(\widetilde{V_1}), S(\widetilde{V_3})\big),
}
\end{equation}
\end{small}
\begin{itemize}
\item
where the top horizontal arrow is given by the composition of morphisms in $\Coh^{\hat G}_{fr}(\hat G\sigma)$ or equivalently natural products of functions on $\hat G\sigma$,
\item
the bottom horizontal arrow is the composition of morphisms in $\rmP^\mathrm{Corr}(\Sht^\loc)$, and
\item
the middle horizontal arrow is given by sending $\bba_1 \otimes \bba_2$ to the homomorphism
\begin{equation}
\label{E:composition of representatives}
\sigma W_2 \otimes \sigma W_1\otimes V_1 \otimes W_1^* \otimes W_2^*  \xrightarrow{\id_{\sigma W_2} \otimes \bba_1 \otimes \id_{W_2^*} } \sigma W_2\otimes V_2 \otimes W_2^*\xrightarrow{\; \bba_2\;}V_3.
\end{equation}
\end{itemize}
In particular, the definition of $S$ on morphisms defined in \eqref{E:definition of tilde Sat on morphisms} is compatible with composition of morphisms, i.e. for $\bba_1 \in \bfJ(V_1^* \otimes V_2)$ and $\bba_2 \in \bfJ(V_2^* \otimes V_3)$, we have
\begin{equation}
\label{E:compatibility of widetilde Sat with composition}
S(\bba_2 \circ \bba_1) = S(\bba_2) \circ S(\bba_1)   \quad \mbox{ in} \quad \mathrm{Mor}_{\rmP^\mathrm{Corr}(\Sht^\loc)}\big(S(\widetilde {V_1}), S(\widetilde {V_3})\big).
\end{equation}
\end{lem}
\begin{proof}
Once we establish the commutativity of the diagram \eqref{E:composition commutes with CW}, the compatibility \eqref{E:compatibility of widetilde Sat with composition} of $\widetilde \Sat$ with composition of morphisms is clear, as $\scrC_{W_1} \otimes \scrC_{W_2}$ (resp. $\scrC_{W_2\otimes W_1}$) factors through $\Xi_{W_1} \otimes \Xi_{W_2}$ (resp. $\Xi_{W_2\otimes W_1}$).

We only establish the commutativity of the bottom square of \eqref{E:composition commutes with CW}, as the commutativity of the top square follows from the similar idea but is much easier. 
We choose $(m_1,n_1,m_2,n_2,m_3,n_3)$ such that
\begin{itemize}
\item $(m_1,n_1,m_2,n_2)$ is $(V_1\otimes W_1)\boxtimes W_1$-acceptable and $(V_2\otimes W_1)\boxtimes W_1$-acceptable;
\item $(m_2,n_2,m_3,n_3)$ is $(V_2\otimes W_2)\boxtimes W_2$-acceptable and $(V_3\otimes W_2)\boxtimes W_2$-acceptable.
\end{itemize}

For $i=1,2$, choose $\bba_i \in \Hom_{\hat G}(\sigma W_i \otimes V_i \otimes W_i^*, V_{i+1})$, 
and let
\[
\bba: = \bba_2 \circ (\id_{\sigma W_2} \otimes \bba_1 \otimes \id_{W_2^*}) \in \Hom_{\hat G}(\sigma W_2 \otimes \sigma W_1 \otimes V_1 \otimes W_1^* \otimes W_2^*, V_3)
\]
as in \eqref{E:composition of representatives}. We need to show that 
\begin{equation}
\label{E:equality of S operator}
\scrC_{W_2}^{\loc(m_2,n_2)}(\bba_2) \circ 
\scrC_{W_1}^{\loc(m_1,n_1)}(\bba_1) = \scrC_{W_2\otimes W_1}^{\loc(m_1,n_1)}(\bba)
\end{equation}
as a cohomological correspondence from $S(\widetilde{V_1})^{\loc(m_1,n_1)}$ to $S(\widetilde{V_3})^{\loc(m_3,n_3)}$.
By the definition of composition law in the category $\rmP^\mathrm{Corr}(\Sht_{\bar k}^\loc)$, the left hand side is
given by first forming the composition  of cohomological correspondences (Definition~\ref{AD:correspondences}) $\scrC_{W_2}^{\loc(m_2,n_2)}(\bba_2)$ and $ 
\scrC_{W_1}^{\loc(m,n_1)}(\bba_1)$
which is supported on $\Sht_{V_1|V_2}^{W_1,\loc(m_1,n_1)} \times_{\Sht_{V_2}^{\loc(m_2,n_2)}} \Sht_{V_2|V_3}^{W_2,\loc(m_2,n_2)}$, and then pushing it forward (\S\ref{ASS:pushforward correspondence}) along the perfectly proper composition map \eqref{E:composition restricted}
\[
\mathrm{Comp}^{\loc(m_1,n_1)} : \Sht_{V_1|V_2}^{W_1,\loc(m_1,n_1)} \times_{\Sht_{V_2}^{\loc(m_2,n_2)}} \Sht_{V_2|V_3}^{W_2,\loc(m_2,n_2)} \to \Sht_{V_1|V_3}^{W_2 \otimes W_1,\loc(m_1,n_1)}
.\]
We shall compute this pushforward following the list of isomorphisms and maps given in Proposition~\ref{P:pushforward is local}.

To simplify the notations, we omit the superscripts $\loc(m_i,n_i)$ everywhere in the following proof; this should not cause any confusion.
Our proof relies on the following diagram of cohomological correspondence.
\begin{tiny}
\begin{equation}
\label{E:composition of S operator}
\xymatrix{
\ar@{=}[ddd]\ar[dr]S(\widetilde{V_1}) \ar[r]^-{\scrC(\delta_{\sigma W_1} \otimes \id_{V_1})} 
& 
S\big(\widetilde{\sigma W_1^*} \boxtimes (\widetilde{\sigma W_1} \otimes \widetilde {V_1}) \big) \ar[r]^-{\DD \Gamma_{F^{-1}}^*} \ar@<-15pt>[d]^{\scrC( \id_{\sigma W_1^*} \otimes \delta_{\sigma W_2} \otimes \id_{\sigma W_1^* \otimes V_1 } )} 
& S\big((\widetilde{\sigma W_1} \otimes \widetilde{V_1}) \boxtimes \widetilde {W_1^*}\big) \ar@<15pt>[d]_{\scrC(\delta_{W_2}\otimes \id_{\sigma W_1 \otimes V_1 }\otimes \id_{ W_1^*})} \ar[rr]^{\scrC(\bba_1)}
\ar[drr]^{\scrC(\bba')}
&& S(\widetilde{V_2})
\ar[d]^{\scrC(\bba'_2)}
\\
& \ar[dr] \ar@<-15pt>[dd]S\big(\widetilde{\sigma W_1^*} \boxtimes \widetilde{\sigma W_2^*} \boxtimes( \widetilde{\sigma W_2} \otimes \widetilde{\sigma W_1} \otimes \widetilde{V_1})\big) \ar[r]^-{\DD \Gamma_{F^{-1}}^*} 
& S\big( \widetilde{\sigma W_2^*} \boxtimes( \widetilde{\sigma W_2} \otimes \widetilde{\sigma W_1} \otimes \widetilde{V_1})\boxtimes \widetilde{ W_1^*} \big) \ar[rr]_-{\scrC( \id_{\sigma W_2^*} \otimes \bba'')} \ar@<15pt>[d]^-{\DD \Gamma_{F^{-1}}^*}  \ar@{}[rru]|(0.35){\mathbf{(A)}}
&& S\big(\widetilde{\sigma W_2^*} \boxtimes (\widetilde {V_3} \otimes \widetilde{W_2})\big)
\ar[d]^-{\DD \Gamma_{F^{-1}}^*} 
\\
&& S\big((\widetilde{\sigma W_2} \otimes \widetilde{\sigma W_1} \otimes \widetilde{ V_1})\boxtimes \widetilde{ W_1^*}\boxtimes \widetilde{ W_2^*} \big) \ar[rr]^-{\scrC(\bba'' \otimes \id_{W_2^*})}
\ar@<15pt>[d]\ar[drr]
&& S\big(  (\widetilde {V_3} \otimes \widetilde{W_2})\boxtimes \widetilde{W_2^*}\big) \ar[d]_-{\scrC(\id_{V_3} \otimes e_{W_2})} 
\\
S(\widetilde{V_1})\ar[r]_-{\scrC(\delta_{\sigma W_1\otimes\sigma W_2} \otimes \id_{V_1})\qquad } & S\big((\widetilde{\sigma W_1^*} \otimes \widetilde{\sigma W_2^*}) \boxtimes( \widetilde{\sigma W_2} \otimes \widetilde{\sigma W_1} \otimes \widetilde{V_1})\big) \ar[r]_-{\DD \Gamma_{F^{-1}}^*} & S\big((\widetilde{\sigma W_2} \otimes \widetilde{\sigma W_1} \otimes \widetilde{ V_1})\boxtimes (\widetilde{ W_1^*}\otimes \widetilde{ W_2^*}) \big) \ar_-{\scrC(\bba)}[rr]&& S(\widetilde{V_3})}.
\end{equation}
\end{tiny}Here $\bba'$ (resp. $\bba''$) denotes the image of $\bba$ under the natural identification of $\Hom(\sigma W_2 \otimes \sigma W_1 \otimes V_1 \otimes W_1^* \otimes W_2^*, V_3)$ with $\Hom(\sigma W_1 \otimes V_1 \otimes W_1^*, \sigma W_2^* \otimes V_3 \otimes W_2)$ (resp. $\Hom(\sigma W_2^*\sigma W_1 \otimes V_1 \otimes W_1^*, V_3 \otimes W_2)$).

The composition of the first row is $\scrC_{W_1}(\bba_1)$ by Construction~\ref{Cons:CW} and the composition of the right column is $\scrC_{W_2}(\bba_2)$ by Lemma~\ref{L:alternative CW}. The bottom line of the diagram is $\scrC_{W_2\otimes W_1}(\bba)$. We will explain that when moving from the upper right of to lower left of the diagram, one exactly pushes the composition of $\scrC_{W_1}(\bba_1)$ and $\scrC_{W_2}(\bba_2)$ as cohomological correspondences along the morphism defined in Proposition~\ref{P:pushforward is local}). By definition, this is exactly
the composition $\scrC_{W_2}(\bba_2) \circ \scrC_{W_1}(\bba_1)$ in the category $\on{P}^{\Corr}(\Sht)$. It then follows that $\scrC_{W_2}(\bba_2) \circ \scrC_{W_1}(\bba_1)=\scrC_{W_2\otimes W_1}(\bba)$ in $\on{P}^{\Corr}(\Sht)$.

First, in the triangle in the upper right corner, since $\bba'  = \bba'_2 \circ \bba_1$, by Lemma~\ref{L:Satake correspondence shtukas}, the pushforward of the composition $\scrC(\bba'_2)$ and $\scrC(\bba_1)$ as cohomological correspondences along the map (1) in \eqref{E:composition restricted} is $\scrC(\bba')$. Next, the triangle $\mathbf{(A)}$ is commutative, by Lemma~\ref{L:technical lemma for comp loc restricted}(1) and Lemma \ref{L:Satake correspondence shtukas}. Here we use the convention as in Remark \ref{E: true support} so the support of $\scrC(\delta_{\sigma W_2} \otimes \id_{\sigma W_1 \otimes V_1} \otimes \id_{W_1^*})$ (resp. $\scrC(\id_{\sigma W^*_2} \otimes \bba'')$) is
$$\Sht^{0, \loc(m_2,n_2)}_{ ({\sigma W_1} \otimes {V_1})\mid  {\sigma W_2^*}\boxtimes({\sigma W_2} \otimes {\sigma W_1} \otimes {V_1}) ; {W_1^*}},\quad (\mbox{resp. } \Sht^{0,\loc(m_2,n_2)}_{{\sigma W_2^*}; ({\sigma W_2} \otimes {\sigma W_1} \otimes {V_1}, W_1^*)\mid ({V_3} \otimes {W_2})}).)$$
The two squares involving inverse partial Frobenii are commutative by Lemma~\ref{L: Sat corr comm with pFrob}. 

Now we have the cohomological correspondences along the diagonal of the diagram. We claim its pushforward along  the map (2) in Proposition~\ref{P:pushforward is local}) is exactly the cohomological correspondences along the bottom row. Indeed, the left trapezoid and right triangles are explained by Remark~\ref{R:tensor versus box tensor} and the middle pentagon follows from Lemma~\ref{L: double pFrob to single}. 
This completes the proof of the Lemma.
\end{proof}

\begin{lem}
\label{L:commutativity of periodic geometric Satake diagram}
For a morphism $\bba \in \Hom_{\hat G}(V_1, V_2)$ between two irreducible representations of $\hat G$, we have an equality of morphisms
\[
\Phi^\loc(\Sat(\bba)) = S(\widetilde \bba ) .
\]
In other words, the diagram \eqref{E:periodic geometric Satake} of functors is commutative for morphisms.
\end{lem}
\begin{proof}

For a morphism $\bba \in \Hom_{\hat G}(V_1, V_2)$ between two irreducible representations of $\hat G$, $\widetilde\bba \in \Hom_{\Coh^{\hat G}_{fr}(\hat G\sigma)}(\widetilde{V_1}, \widetilde{V_2})$ is represented by $\Xi_{\mathbf {1}}(\bba)$.
Since $W = \boldsymbol{1}$ in this case, we can take $n_1=n_2$ to be an arbitrary non-negative integer $n$ in Construction~\ref{Cons:CW}. Then $S(\widetilde \bba) = \scrC_{\boldsymbol{1}} (\bba)$ is defined to be the composite
\[
\scrC^{\loc(m,n)}( \bba) \circ \DD\Gamma_{F_{V_1\boxtimes\boldsymbol{1}}^{-1}} \circ \scrC^{\loc(m,n)}(\delta_{\boldsymbol{1}}\otimes \id_{V_1}).
\]
But $W = \boldsymbol{1}$, so the right two maps are trivial and therefore, $S(\widetilde \bba) = \scrC_{ \boldsymbol{1}}(\bba)$ is simply equal to the cohomological Satake correspondence $\mathscr{C}^{\loc(m,n)}
(\bba)$.
Clearly, if we go right and then down on the diagram \eqref{E:periodic geometric Satake}, the morphism $\Phi^{\loc}(\Sat(\bba))$ will be exactly the one represented by $\mathscr{C}^{\loc(m,n)}
(\bba)$. The lemma is proved.
\end{proof}

\subsection{Cohomology class of the correspondence}
\label{Sec:S=T}

In this subsection, we  describe the image of $S$ on morphisms explicitly in some special cases, by relating them with cycle class maps or the usual Hecke operators. This will in particular imply Theorem \ref{T:Spectral action} (2), which essentially is the $S=T$ theorem \cite[Proposition~0.16]{La} of V. Lafforgue.

Now let $\tau\in\xcoch(Z_G)$ be central and $\mu$ minuscule.\footnote{The assumption on $\mu$ is not necessary. But it simplifies the discussion and will be sufficient for our application.} For $\bba\in \Hom_{\hat G}(\sigma V_{\nu}\otimes V_\tau\otimes V_\nu^*, V_\mu)$, recall that we defined $\scrC_{V_\nu}(\bba)$ in Construction~\ref{Cons:CW}. Concretely, choose $(m_1,n_1,m_2,n_2)$ that is $(\nu+\tau,\nu)$-acceptable and $(\nu+\mu,\nu)$-acceptable, then $\scrC_{V_\nu}(\bba)$ is realized by a cohomological correspondence
$$\scrC_{V_\nu}^{\loc(m_1,n_1)}(\bba)\in \on{Corr}_{\Sht^{\nu,\loc(m_1,n_1)}_{\tau\mid \mu}}(S(\widetilde{V_\tau})^{\loc(m_1,n_1)}, S(\widetilde{V_\mu})^{\loc(m_2,n_2)}).$$

Recall that from Corollary \ref{C:expclit mStautau truncated}, 
$$\Sht^{\nu,\loc(m_1,n_1)}_{\tau\mid \mu}\cong [L^{m_1}_{n_1}G\backslash X_{\mu^*,\nu^*}(\tau^*)]=:Y,$$ which is of dimension $\leq \langle\rho,\mu-\tau\rangle=\langle\rho,\mu\rangle$.
Thus
\[\on{Corr}_{Y}(S(\widetilde{V_\tau})^{\loc(m_1,n_1)}, S(\widetilde{V_\mu})^{\loc(m_2,n_2)})\cong \on{H}^{-\langle2\rho,\mu\rangle}(Y,\omega_Y(-\langle\rho,\mu\rangle))\cong \on{H}^{\on{BM}}_{\langle2\rho,\mu\rangle}(X_{\mu^*,\nu^*}(\tau^*))^{G(\mO/\varpi^{n_1})},\]
where the first isomorphism follows from the same considerations as in \S \ref{cl vs cc}. Since $\dim X_{\mu^*,\nu^*}(\tau^*)\leq\langle\rho,\mu\rangle$, the space $\on{H}^{\on{BM}}_{\langle2\rho,\mu\rangle}(X_{\mu^*,\nu^*}(\tau^*))$ has a basis given by irreducible components of $X_{\mu^*,\nu^*}(\tau^*)$ of dimension $\langle\rho, \mu\rangle$. Therefore, we may regard $\scrC_{V_\nu}^{\loc(m_1,n_1)}(\bba)$ as an element in $\on{H}^{\on{BM}}_{\langle2\rho,\mu\rangle}(X_{\mu^*,\nu^*}(\tau^*))$. Note that this element  is independent of the choices of $(m_1,n_1,m_2,n_2)$ by \eqref{E: change of m,n in Hk corr}, as soon as the choice is $(\nu+\tau,\nu)$-acceptable and $(\nu+\mu,\nu)$-acceptable. We thus also write this element as $\scrC_{V_\nu}(\bba)$.
In this subsection, we make this element explicit in two special cases.

Now, we choose $n$ large enough and consider  the following commutative diagram of correspondences 
\begin{tiny}
\begin{equation}\label{E: cal coh corr to local}
\xymatrix@C=15pt{
\on{pt}&\ar[l] \Gr_{\sigma(\nu^*)}\ar^-{\id\times\varpi^{\tau^*}}[r]\ar@{}[dr]|{\mathrm{(A)}}& \Gr_{\sigma(\nu^*)}\times\Gr_{\tau^*+\sigma(\nu^*)}& \Gr_{\nu^*}\times\Gr_{\tau^*+\sigma(\nu^*)}\ar_-{\sigma\times\id}[l]\ar@{}[dr]|{\mathrm{(B)}}&\ar[l] \Gr^0_{(\nu^*,\mu^*)\mid\tau^*+\sigma(\nu^*)}\ar[r]& \on{pt}\\
\Gr_\tau^{(n)}\ar[d]\ar[u]&\ar[l] (\Gr_{\sigma(\nu^*)}^{(2n)}\times\Gr_{\tau+\sigma(\nu)}^{(n)})_\tau\ar[r]\ar[d]\ar[u]\ar@{}[dr]|{\mathrm{(C)}}& \Gr_{\sigma(\nu^*)}^{(2n)}\times\Gr_{\tau+\sigma(\nu)}^{(n)}\ar[r]\ar[d]\ar[u]\ar@{}[dr]|{\mathrm{(D)}}&\Gr_{\tau+\sigma(\nu)}^{(n)}\times\Gr_{\nu^*}\ar[d]\ar[u]\ar@{}[dr]|{\mathrm{(E)}}&\ar[l](\Gr_{\tau+\sigma(\nu)}^{(n)}\times\Gr_{\nu^*})_\mu\ar[r]\ar[d]\ar[u]&\Gr_\mu\ar[d]\ar[u]\\
\Sht_{\tau}^{\loc(3n,n)} & \Sht_{\tau|(\sigma(\nu^*), \tau+\sigma(\nu))}^{0, \loc(3n,n)}\ar[l]\ar[r]& \Sht_{\sigma(\nu^*), \tau+\sigma(\nu)}^{\loc(3n,n)}\ar^{F^{-1}}[r]&\Sht_{ \tau+\sigma(\nu),\nu^*}^{\loc(2n,0)}&\ar[l] \Sht_{ (\tau+\sigma(\nu),\nu^*)|\mu}^{0,\loc(2n,0)}\ar[r]& \Sht_{\mu}^{\loc(2n,0)}. 
}
\end{equation}
\end{tiny}

This diagram requires some explanation. 
The commutative square $(\mathrm{D})$ follows from Remark \ref{R: group const of pFrob}.
We define $(\Gr_{\sigma(\nu^*)}^{(2n)}\times\Gr_{\tau+\sigma(\nu)}^{(n)})_\tau$ and $(\Gr_{\nu^*}\times\Gr_{\tau+\sigma(\nu)}^{(n)})_\mu$ so that the commutative square $(\mathrm{C})$ and $(\mathrm{E})$ are Cartesian. Explicitly,
$$(\Gr_{\sigma(\nu^*)}^{(2n)}\times\Gr_{\tau+\sigma(\nu)}^{(n)})_\tau=\{(g_1L^+G^{(2n)},g_2L^+G^{(n)})\in \Gr_{\sigma(\nu^*)}^{(2n)}\times\Gr_{\tau+\sigma(\nu)}^{(n)} \mid g_1g_2\in \varpi^\tau L^+G^{(n)}\}\cong \Gr_{\sigma(\nu^*)}^{(2n)},$$ 
$$(\Gr_{\tau+\sigma(\nu)}^{(n)}\times\Gr_{\nu^*})_\mu=\{(g_1L^+G^{(n)},g_2L^+G)\in \Gr_{\tau+\sigma(\nu)}^{(n)}\times\Gr_{\nu^*}\mid g_1g_2\in L^+G\varpi^\mu L^+G\}.$$ It follows that the commutative square $(\mathrm{B})$ is also Cartesian.

Note that since $\mu$ is minuscule, the $\star$-pullback (in the sense of Notation \ref{N:star pullback}) of $\IC_{\nu^*}$ along the perfectly smooth morphism $\Gr_{\nu^*}\tilde\times\Gr_{\mu^*}\to \Gr_{\nu^*}$ is $\IC_{\nu^*}\tilde\boxtimes\IC_{\mu^*}$ on $\Gr_{\nu^*,\mu^*}$. Then
\begin{eqnarray*}
&&\Hom_{\hat G}(V_{\sigma(\nu)}\otimes V_{\tau}\otimes V_{\nu^*},V_{\mu})\\
&\cong&\Hom_{\hat G}(V_{\nu^*}\otimes V_{\mu^*},V_{\tau^*}\otimes V_{\sigma(\nu^*)})\\
&\stackrel{\Sat}{\cong}& \on{Corr}_{\Gr^0_{(\nu^*,\mu^*)\mid\tau^*+\sigma(\nu^*)}}\big((\Gr_{\nu^*,\mu^*},\IC_{\nu^*}\tilde\boxtimes\IC_{\mu^*}),(\Gr_{\tau^*+\sigma(\mu^*)},\IC_{\tau^*+\sigma(\nu^*)})\big)\\
&\cong& \on{Corr}_{\Gr^0_{(\nu^*,\mu^*)\mid\tau^*+\sigma(\nu^*)}}\big((\Gr_{\nu^*},\IC_{\nu^*}[-\langle 2\rho, \mu\rangle](-\langle \rho, \mu\rangle) ),(\Gr_{\tau^*+\sigma(\mu^*)},\IC_{\tau^*+\sigma(\nu^*)})\big)\\
&\cong &\on{Corr}_{\Gr^0_{(\nu^*,\mu^*)\mid\tau^*+\sigma(\nu^*)}}\big((\Gr_{\nu^*}\times \Gr_{\tau^*+\sigma(\nu^*)},\IC_{\nu^*}\boxtimes  \IC_{\tau^*+\sigma(\nu^*)}), (\on{pt},\Ql[-\langle2\rho,\mu\rangle](-\langle\rho,\mu\rangle))\big),
\end{eqnarray*}
where the last isomorphism follows from Lemma \ref{L:sharp correspondence}. We denote by $\Sat(\bba)^\sharp$ the image of $\bba\in \Hom_{\hat G}(V_{\sigma(\nu)}\otimes V_{\tau}\otimes V_{\nu^*},V_{\mu})$ under this series of isomorphisms.

Note that all vertical arrows in \eqref{E: cal coh corr to local} are cohomologically smooth by Lemma \ref{L: reversing modification}, and we can apply the formalism of pullback of cohomological correspondences as discussed in \S \ref{ASS:smooth pullback correspondence}.
\begin{itemize}
\item The $\langle 3n\dim G\rangle$-shift of the smooth pullback of $\scrC^{\loc(3n)}(\id_{\tau} \otimes \delta_{\sigma(\nu^*)})$ to the middle row is equal to the $\langle n\dim G\rangle$-shift of the smooth pullback of  
$$\bD((\Ga_{\varpi^{\tau^*}}^*)^\sharp): (\on{pt},\Ql)\to (\Gr_{\sigma(\nu^*)}\times\Gr_{\tau^*+\sigma(\nu^*)},\IC_{\sigma(\nu^*)}\boxtimes\IC_{\tau^*+\sigma(\nu^*)}),$$ 
where $\Gamma_{\varpi^{\tau^*}}^*$ is the cohomological correspondence associated to the canonical isomorphism $(\varpi^{\tau^*})^*\IC_{\tau^*+\sigma(\nu^*)}\cong \IC_{\sigma(\nu^*)}$ as from Example \ref{Ex:examples of correspondences} (2), $(\Ga_{\varpi^{\tau^*}}^*)^\sharp$ is obtained from $\Ga_{\varpi^{\tau^*}}^*$ by Lemma \ref{L:sharp correspondence}; 
\item The shifted smooth pullback of $\bD F^{-1}$ from the bottom to the middle row is equal to the shifted smooth pullback of $\Gamma_{\sigma\times\id}^*$ from the top to the middle row.
\item The shifted smooth pullback of $\scrC^{\loc(2n)}(\bba)$  from the bottom to the middle row is equal to the shifted smooth pullback of $\Sat(\bba)^\sharp$ from the top to the middle row.
\end{itemize}
It follows from Lemma \ref{AL:pullback compatible with composition} that the shifted smooth pullback of 
$$\scrC^{\loc(3n,n)}_{V_\nu}(\bba)= \scrC^{\loc(2n,n)}(\bba)\circ \bD F^{-1}_{} \circ \scrC^{\loc(3n,n)}(\id_{\tau}\otimes \delta_{\sigma(\nu^*)})$$ 
to a cohomological correspondence
\begin{equation}
\label{E: corr in middle row}
(\Gr_{\tau}^{(n)},\Ql\langle n\dim G\rangle)\to (\Gr_{\mu},\Ql[\langle 2\rho,\mu\rangle](\langle 2\rho,\mu\rangle))
\end{equation}
is equal to the shifted smooth pullback of
\[\scrC^\Gr_{V_\nu}(\bba):=\Sat(\bba)^\sharp\circ \Gamma_{\sigma\times\id}^*\circ \bD((\Ga_{\varpi^{\tau^*}}^*)^\sharp).\]
Let us denote the support of compositions of the correspondences in the upper row (resp. middle row, lower row) of the diagram by $C^u$ (resp. $C$, $C_l$).
Then
\[C:=\{(g_1L^+G^{(2n)},g_2L^+G^{(n)})\in \Gr_{\sigma(\nu^*)}^{(2n)}\times\Gr_{\tau+\sigma(\nu)}^{(n)} \mid g_1g_2\in \varpi^\tau L^+G^{(n)}, g_2\sigma^{-1}(g_1)\in L^+G\varpi^\mu L^+G\}\]
The map $C\to C^u$ is induced by $g_1L^+G^{(2n)}\mapsto \sigma^{-1}g_1L^+G$. This realizes $C$ as an $L^{2n}G$-torsor over $C^u\cong X_{\mu^*,\nu^*}(\tau^*)\subset \Gr_{\nu^*}$. 

On the other hand, by the proof of Corollary \ref{C:expclit mStautau truncated}, further composing this map with 
$$X_{\mu^*,\nu^*}(\tau^*)\to [ L^{3n}_n G \backslash X_{\mu^*,\nu^*}(\tau^*)]\cong C_l$$ 
gives the perfectly smooth morphism $C\to C_l$. 

Now, the cohomological correspondence \eqref{E: corr in middle row} is given by a class in $c\in\on{H}^{\on{BM}}_{\langle2\rho,\mu\rangle+4n\dim G}(C)$, and the cohomological correspondence $\scrC_{V_\nu}^\Gr(\bba)$ is given by a class in $\on{H}^{\on{BM}}_{\langle2\rho,\mu\rangle}(C^u)= \on{H}^{\on{BM}}_{\langle2\rho,\mu\rangle}(X_{\mu^*,\nu^*}(\tau^*))$, by the same considerations as in \S \ref{cl vs cc}. The above observations imply that under the canonical isomorphisms
\[\on{H}^{\on{BM}}_{\langle2\rho,\mu\rangle}(X_{\mu^*,\nu^*}(\tau^*))\cong\on{H}^{\on{BM}}_{\langle2\rho,\mu\rangle+4n\dim G}(C)\cong \on{H}^{\on{BM}}_{\langle2\rho,\mu\rangle}(X_{\mu^*,\nu^*}(\tau^*)),\] 
$\scrC_{V_\nu}(\bba)=c=\scrC^\Gr_{V_\nu}(\bba)$. We thus have proven

\begin{lem}
The class $\scrC_{V_\nu}(\bba)\in \on{H}^{\on{BM}}_{\langle2\rho,\mu\rangle}(X_{\mu^*,\nu^*}(\tau^*))$ is equal to the class $\scrC^\Gr_{V_\nu}(\bba)\in \on{H}^{\on{BM}}_{\langle2\rho,\mu\rangle}(X_{\mu^*,\nu^*}(\tau^*))$.
\end{lem}

Assume that 
$$\bba\in\Hom_{\hat G}(\sigma V_{\nu}\otimes V_{\tau}\otimes V_{\nu^*}, V_{\mu})\cong \Hom(V_{\nu^*}\otimes V_{\mu^*}, V_{\sigma(\nu^*)}\otimes V_{\tau^*})$$ is given by the fundamental class of a Satake cycle $\Gr^{0,\bba}_{(\nu^*,\mu^*)\mid \sigma(\nu^*)+\tau^*}$ via Proposition \ref{SS:cycles of geometric Satake}, and assume that
$X_{\mu^*,\nu^*}^{\bba}(\tau)$ as defined via \eqref{E: possible irr comp}. Then it follows from definition that $\scrC_{V_{\nu}}(\bba)=\scrC^\Gr_{V_\nu}(\bba)$ in fact lies in the image of $\on{H}^{\on{BM}}_{\langle2\rho, \mu\rangle}(X^\bba_{\mu^*,\nu^*}(\tau^*))\to \on{H}^{\on{BM}}_{\langle2\rho, \mu\rangle}(X_{\mu^*,\nu^*}(\tau^*))$.

Now we specialize to the cases we need.

\begin{prop}
\label{P: corr vs cl for min a}
Assume that $\bbb=i^{\MV}(\bba)\in \MV_\mu(\tau+\sigma(\nu)-\nu)$ satisfies Lemma \ref{L: finding best tau}. Then
$\scrC_{V_{\nu}}(\bba)\in \on{H}^{\on{BM}}_{\langle2\rho, \mu\rangle}(X_{\mu^*,\nu^*}(\tau^*))$ is given by the fundamental class of $X_{\mu^*}^{\bbb,x_0}\subset X_{\mu^*,\nu^*}(\tau^*)$.
\end{prop}
\begin{proof}
By Theorem \ref{C: unique comp}, $X^\bba_{\mu^*,\nu^*}(\tau^*)$ has a unique irreducible component of dimension $\langle\rho,\mu\rangle$, namely $X_{\mu^*}^{\bbb,x_0}(\tau^*)$. So $\scrC_{V_{\nu}}(\bba)$ is some multiple of the fundamental class of $X_{\mu^*}^{\bbb,x_0}(\tau^*)$.
Since the intersection of this component with $\mathring{\Gr}_{\nu^*}$ is non-empty (so open dense in $X_{\mu^*}^{\bbb,x_0}(\tau^*)$). Then we can restrict every cohomological class to the smooth open locus. The smooth part of the top row of \eqref{E: cal coh corr to local} can be identified with intersection of the diagonal of $\mathring{\Gr}_{\sigma(\nu^*)}\times\mathring{\Gr}_{\sigma(\nu^*)}$ with the following correspondence: 
$$C=\mathring{\Gr}^{0,\bba}_{(\nu^*,\mu^*)\mid \tau^*+\sigma(\nu^*)}\xrightarrow{c_1\times c_2}\mathring{\Gr}_{\sigma(\nu^*)}\times \mathring{\Gr}_{\sigma(\nu^*)},$$ where $\mathring{\Gr}^{0,\bba}_{(\nu^*,\mu^*)\mid \tau^*+\sigma(\nu^*)}$ is as in \eqref{E: open Satake cycle}, and
$$c_1: \mathring{\Gr}^{0,\bba}_{(\nu^*,\mu^*)\mid \tau^*+\sigma(\nu^*)}\to \mathring{\Gr}_{\nu^*}\xrightarrow{\sigma} \mathring{\Gr}_{\sigma(\nu^*)},\ \ c_2: \mathring{\Gr}^{0,\bba}_{(\nu^*,\mu^*)\mid \tau^*+\sigma(\nu^*)}\to \mathring{\Gr}_{\sigma(\nu^*)+\tau^*}\xrightarrow{\varpi^{\tau}}\mathring{\Gr}_{\sigma(\nu^*)}.$$ The cohomological correspondence then is identified with $\delta_{\Ql}\circ u^\sharp$, where
$u: c_1^*\Ql[\langle2\rho,\mu\rangle](\langle\rho,\mu\rangle) \to c_2^!\Ql$ is given by the fundamental class of $\mathring{\Gr}^{0,\bba}_{(\nu^*,\mu^*)\mid \tau^*+\sigma(\nu^*)}$. Choose a deperfection and use the fact that the differential of Frobenius is zero, we see $C$ intersects with the diagonal properly smoothly. The proposition now follows from Lemma \ref{L: Shifted trace smooth}.
\end{proof}

In the second case, we assume that $\mu=\tau$, $\sigma(\nu)=\nu$. Let 
$\bba\in\Hom(V_{\nu}\otimes V_{\tau}\otimes V_{\nu^*},V_\tau)$ be the homomorphism induced by $\bbe_{\nu}: V_{\nu}\otimes V_{\nu^*}\to \mathbf{1}$.
Then $\Gr^{0,\bba}_{(\nu^*,\mu^*)\mid \sigma(\nu^*)+\tau^*}=\Gr^{0}_{(\nu^*,\mu^*)\mid \sigma(\nu^*)+\tau^*}$ is isomorphic to $\Gr_{\nu^*}$, and $$X_{\mu^*,\nu^*}(\tau^*)=\Gr_{\nu^*}(k)$$
is the set of rational points of $\Gr_{\nu^*}$.
The top row of \eqref{E: cal coh corr to local} can be identified with
\[\on{pt}\leftarrow \Gr_{\nu^*}\xrightarrow{\Delta}\Gr_{\nu^*}\times\Gr_{\nu^*}\xleftarrow{\sigma\times\id}\Gr_{\nu^*}\times\Gr_{\nu^*}\xleftarrow{\Delta}\Gr_{\nu^*}\to\on{pt},\]
and the cohomological correspondence can be identified with $\delta_{\IC_{\nu^*}} \circ \Ga^*_{\sigma\times\id}\circ e_{\IC_{\nu^*}}$, where $e_{\IC_{\nu^*}}$ and $\delta_{\IC_{\nu^*}}$ are as in Example \ref{Ex:examples of correspondences}(4).
Then Lemma \ref{P: GL trace formula}, together with Proposition \ref{P: classical Sat=geometric Sat}, implies the following result, which gives Theorem \ref{T:Spectral action} (2) (by Remark \ref{R: geom Sat, arith Frob v.s. geom Frob}).

\begin{prop}
\label{P: local S=T}
The class
$\scrC_{V_{\nu}}^{\loc(m_1,n_1)}(\bbe)\in \on{H}^{\on{BM}}_{0}(X^\bbe_{\mu^*,\nu^*}(\tau^*))=C(\Gr_{\nu^*}(k))$ is the function $\Sat^{cl}([V_{\nu^*}])$, i.e.
its value at $x\in \Gr_{\nu^*}(k)$ is
\[\tr(\phi_x\mid \Sat(V_{\nu^*})_{\bar x}).\]
\end{prop}

\section{Application to Shimura varieties}
In this section, we prove our main theorem.
\begin{convention}
Our convention for Hodge structure, weight maps, and Artin reciprocity map all follow \cite{Deligne}. But the reciprocity morphism of \cite[\S2.2.3]{Deligne} is changed to be the one \emph{without taking the inverse}, as pointed out by \cite{Milne}.
\end{convention}

\subsection{The canonical integral model of Shimura varieties of Hodge type}
In this subsection, we recall the canonical integral models of Shimura varieties of Hodge type and some related constructions, mostly following \cite{Ki1}.
\subsubsection{The canonical integral model}
Let $G$ be a connected reductive group over $\bQ$, and let $X$ be a $G(\bR)$-conjugacy class of homomorphisms $h:\bS\to G_\bR$. Recall that the pair $(G,X)$ is called a \emph{Shimura datum} if for some (and therefore every) $h\in X$,
\begin{enumerate}
\item[(SV1)] under $\Ad\circ h:\bS\to\GL(\fg)$, $\fg$ acquires a Hodge structure of type $(-1,1), (0,0), (1,-1)$,
\item[(SV2)] $\ad h(i)$ is a Cartan involution of $G^\ad_\bR$, and
\item[(SV3)] $G^\ad$ does not contain a $\bQ$-factor whose real points are compact.
\end{enumerate}
In this work, we will also consider groups that do not satisfy (SV3). A reductive group $G$ over $\bQ$ together with a set $X$ of a $G(\bR)$-conjugacy classes of $h:\bS\to G_{\bR}$ satisfying (SV1), (SV2) will be called a \emph{weak Shimura datum}.

The map $w_h:\bG_m\subset \bS\stackrel{h}{\to} G_\bR$ factors through the center of $G_{\bR}$, and is therefore independent of the choice of $h$. We denote it by $w$ and call it the \emph{weight homomorphism}.

For $G=\GSp(V,\psi)$, the similitude group of a symplectic $\bQ$-vector space, we use $\frakH^\pm$ (instead of $X$) to denote the Siegel (upper and lower) half spaces of complex structures on $V$ such that $\psi(\cdot, J\cdot)$ is (positive or negative) definite.\footnote{We used the notation $\psi$ earlier to in the definition of local $G$-shtukas, but there should be no confusion.} Recall that a (weak) Shimura datum $(G, X)$ is called of \emph{Hodge type} if there is an embedding $\rho:G\to \GSp(V,\psi)$ that induces an embedding $X\to \frakH^\pm$ given by $h\mapsto \rho\circ h$.  

Let us fix the isomorphism $\bS_\bC \cong \bG_m\times\bG_m$ as usual (see \eqref{E: Delignetorus}).
Let $\bG_m\to \bS_\bC$ denote the inclusion into the first factor.
The conjugacy class $\{h\}$ induces a conjugacy class $\mu_h: \bG_m\to \bS_\bC\to G_\bC$, which is the same as a conjugacy class of 1-parameter subgroups of $G$ over $\rat\subset\bC$. This conjugacy class is called the \emph{Shimura cocharacter}, and the field of its definition is called the \emph{reflex field}, denoted by $E=E(G,X)$. In the sequel, we write $\mu$ for $\mu_h$ to simplify notation. 

Let $K\subset G(\bA_f)$ be a (sufficiently small) open compact subgroup. Let $\mathbf{Sh}_K(G,X)$ denote the corresponding Shimura variety defined over $E$. We will fix a prime $p>2$, such that $K_p$ is a hyperspecial subgroup of $G(\bQ_p)$, i.e. $K_p=\underline{G}(\bZ_p)$ for some reductive group extension $\underline{G}$ of $G$ to $\bZ_{(p)}$. 
Let $v$ be a place of $E(G,X)$ over $p$. Let $\mO_{E,(v)}$ denote the ring of integers of $E$ localized at $v$. Recall the following theorem of Kisin \cite{Ki1} and Vasiu \cite{Vasiu}.

\begin{thm}
If $(G,X)$ is of Hodge type, there is a smooth integral canonical model $\mathscr S_K(G,X)$ of $\mathbf{Sh}_K(G,X)$, defined over $\mO_{E,(v)}$. 
\end{thm}
\begin{rmk}
In fact, the above theorem is more precise.
Write $K=K_pK^p$, and consider the pro-scheme $\mathscr S_{K_p}(G,X)=\underleftarrow\lim_{K^p}\mathscr S_{K_pK^p}(G,X)$. 
It satisfies the extension property: for any regular, formally smooth $\mO_{E,(v)}$-scheme $T$, a map $S \otimes E \to \mathscr S_{K_p}(G,X)$ extends to a $T$-point of $\mathscr S_{K_p}(G,X)$. This property uniquely characterizes $\mathscr S_{K_p}(G,X)$.
\end{rmk}

The construction of $\mathscr S_K(G,X)$ (in the case of Hodge type) is as follows. First, if $(G',X')=(\GSp(V,\psi),\frakH^\pm)$ with $\dim V=2g$, then $\mathbf{Sh}_{K}(G',X')$ is the moduli space of principally polarized abelian varieties with a $K'$-level structure. It has a canonical extension to an integral model $\mathscr A_{g,K'}$ smooth over $\bZ_{(p)}$ as the moduli space of principally polarized abelian schemes, whose definition will be recalled in \S \ref{SS: moduli of abelian scheme} below.
In general, if $(G,X)$ is of Hodge type, one can chooses a faithful representation $i:G\to \GSp(V,\psi)$ as above such that $i(X)\subset \frakH^{\pm}$, that $\underline G(\bZ_p)=G(\bQ_p)\cap \GSp(V_{\bZ_p},\psi)$ for a self-dual $\ZZ_p$-lattice $V_{\bZ_p}$ in $V_{\bQ_p}$, and that $K^p=G(\bA_f^p)\cap {K'}^p$, where ${K'}^p\subseteq \GSp(V, \psi)(\AAA_f^p)$ is a prime-to-$p$ level structure for the Siegel moduli space (cf. \cite[Lemma~2.3.1]{Ki1}).  Then $i$ induces a closed embedding $i:\mathbf{Sh}_K(G,X)\to \mathbf{Sh}_{K'}(G',X')\otimes E$ defined over $E$. We write $\mathscr S^-_K(G,X)$ for the closure of $\mathbf{Sh}_K(G,X)$ in  $\mathscr A_{g,K'}$, and define $\mathscr S_K(G,X)$ to be its normalization. By \cite[Proposition~2.3.5]{Ki1}, the completion at each closed point of $\mathscr S^-_K(G,X)$ has \emph{formally smooth} irreducible components over $\calO_{E,v}$. So the morphism
\begin{equation}
\label{E:locally closed embedding}
i:\mathscr S_K(G,X) \to \mathscr S_K^-(G,X) \subseteq \mathscr A_{g, K'}
\end{equation}
induces a surjective map of the completed local rings $\widehat\mO_{\mathscr A_{g,{K'}},i(x)}\to \widehat\mO_{\mathscr S_K(G,X),x}$ at every point $x\in \mathscr S_K(G,X)$.
In what follows, we will use $\mathscr S_K$ to denote the base change of $\scrS_K(G, X)$ to the completion $\mO_{E,v}$. 

\subsubsection{Abelian schemes up to prime-to-$p$ isogeny}
\label{SS: moduli of abelian scheme}
We recall the moduli interpretation of $\mathscr A_{g, K'}$. 
First, recall that for every (quasi-compact) scheme $T$, there is the category $\on{AV}^p$ of abelian schemes on $T$ up to prime-to-$p$ isogenies: objects are abelian schemes over $T$, and morphisms are the homomorphisms of abelian schemes tensored with $\bZ_{(p)}$. Isomorphisms in $\on{AV}^p$ are called \emph{prime-to-$p$ quasi-isogenies}.
Given $A$, a \emph{polarization} $\la$ of $A$ is a morphism in $\Hom_{\on{AV}^p}(A,A^\vee)$, such that some of its multiple $n\la$ is induced by an ample line bundle. A  \emph{principal polarization} is a polarization that is an isomorphism in this category. Then we have the groupoid of principally polarized abelian scheme on $T$ up to prime-to-$p$ isogeny: objects are pairs $(A,\la)$, consisting of an abelian scheme over $T$ up to prime-to-$p$ isogeny and a principal polarization; morphisms from $(A,\la)$ to $(A',\la')$ are prime-to-$p$ quasi-isogenies $f:A\to A'$ such that $f^*\la'=c\la$ for some locally constant function $c$ on $T$ with values in $\bZ_{(p)}^\times$.

Let $(A,\la)$ be a principally polarized abelian scheme of dimension $g$ on $T$ up to prime-to-$p$ isogeny. Then the prime-to-$p$ rational Tate module 
$$\mV_{\on{et}}(A)^p:=\underleftarrow\lim_{p\neq n}A[n]\otimes \bQ$$ 
is a $\bA_f^p$-local system on $T$ equipped with an alternating pairing $\mV_{\on{et}}(A)^p\otimes\mV_{\on{et}}(A)^p\to \bA_f^p(1)$. Let $\on{Isom}(V\otimes \bA_f^p, \mV_{\on{et}}(A)^p)$ denote the $G'(\bA_f^p)$-torsor of isomorphisms between $V\otimes \bA_f^p$ and $\mV_{\on{et}}(A)^p$ that respect to the natural pairs up to a $(\bA_f^p)^\times$-scalar\footnote{See \S \ref{R: trivialize Fcrys} below for a more canonical formulation.}. Then a \emph{$K'^p$-level structure} of $(A,\la)$ is a section $\eta$ of the \'etale sheaf $\on{Isom}(V\otimes \bA_f^p, \mV_{\on{et}}(A)^p)/{K'^p}$.

Finally, $\mathscr A_{g,K'}$ is the moduli functor which assigns every $T$ the groupoid, whose objects are triples $(A,\la,\eta)$, where $(A,\la)$ is a principally polarized abelian scheme on $T$ up to prime-to-$p$ isogeny, and $\eta$ is a $K'^p$-level structure, and whose morphisms from $(A,\la,\eta)$ to $(A',\la',\eta')$ are morphisms $f:(A,\la)\to (A',\la')$ as defined above such that $\mV_\et(f)^p\circ \eta=\eta'$. If $K'^p$ is sufficiently small, this groupoid is discrete (i.e. there is no non-trivial automorphism).

\subsubsection{\'Etale, de Rham, and crystalline tensors}
By construction, there is an abelian scheme 
$$h:A\to \mathscr S_K(G,X)$$ 
which is the pullback of the universal one on $\mathscr{A}_{g,K'}$.
For our purpose, we need to recall the construction of a collection of tensors in various cohomology of $A$, following \cite{Ki1}. 

\begin{notation}
\label{N: rigid abelian tensor}
Let $\mC$ be a rigid tensor additive category equipped with an invertible object $\bT$ (in the sequel $\mC$ will be categories of ``local systems" in various settings, and
$\bT$ will be the Tate object in the corresponding category). That is, $ \bT\otimes\bT^\vee\to \mathbf{1}$ is an isomorphism where $\bT^\vee$ is the dual of $\bT$ and $\mathbf{1}$ is the unit object. We denote by $X(n)=X\otimes \bT^{\otimes n}$ if $n\geq 0$ and $X(n)=X\otimes(\bT^\vee)^{\otimes n}$ if $n<0$. If $X$ is an object in $\mC$, equipped with an alternating pairing $\psi: X\otimes X\to \bT$ which induces an isomorphism $X\cong X^\vee(1)$,
we denote 
$$X^\otimes=\bigoplus_{r\in\bN} X^{\otimes 2r}(-r)\cong \bigoplus_{r\in\bN} (X^\vee)^{\otimes 2r}(r),$$
as an algebra object in the ind-completion of $\mC$. Note that one can give a natural grading on $X^\otimes$ so that the direct summand $X^{\otimes 2r}(-r)\cong (X^\vee)^{\otimes 2r}(r)$ is of degree $r$.

Now assume that $\mC$ is $R$-linear, where $R$ is some commutative ring with unit, and let $\mT$ be a (not necessarily commutative) graded $R$-algebra. An object $X$ of $\mC$ is said to be \emph{equipped with a $\mT$-structure} if it is given a morphism of graded algebra (in the ind-completion of $\mC$)
\[s_X: \mT\otimes \mathbf{1}\to X^\otimes.\]
Given two objects $X$ and $Y$ in $\mC$, both equipped with a $\mT$-structure, we denote the subspace of 
$$\Hom((X,s_X),(Y,s_Y))\subset \Hom_\mC(X,Y), \quad \mbox{resp. } \on{Isom}((X,s_X),(Y,s_Y))\subset\on{Isom}_\mC(X,Y)$$ consisting of those morphisms (resp. isomorphisms) $f:X\to Y$ such that $s_Y:= f^\otimes\circ s_X$.
\end{notation}

Let $V_{\bZ_{(p)}}$ be a $\bZ_{(p)}$-lattice of $V$ whose $p$-adic completion is $V_{\bZ_p}\subset V_{\bQ_p}$ so $\psi$ induces a perfect pairing of $V_{\bZ_{(p)}}$. 
Let $\bZ_{(p)}(1)$ be the rank one $\bZ_{(p)}$-module on which $\GSp(V_{\bZ_{(p)}},\psi)$ acts through the similitudes $\GSp(V_{\bZ_{(p)}},\psi)\to \bG_m$. The symplectic form is regarded as a perfect pairing of $\on{GSp}(V_{\bZ_{(p)}},\psi)$-representations
\[\psi: V_{\bZ_{(p)}}\otimes V_{\bZ_{(p)}}\to \bZ_{(p)}(1).\]
We take $\mC$ to be the category of representations of $\GSp(V_{\bZ_{(p)}},\psi)$ on finite free $\bZ_{(p)}$-modules and $\bT=\bZ_{(p)}(1)$.
The following lemma slightly refines \cite[Proposition 1.3.2]{Ki1}. 
\begin{lem}
The group scheme  $\underline G\subseteq \GSp(V_{\bZ_{(p)}},\psi)$ is the schematic centralizer of a finitely graded $\bZ_{(p)}$-subalgebra $\mT\subseteq (V_{\bZ_{(p)}}^{\otimes})^{\underline G}$.
\end{lem}

\begin{notation}
\label{N: tensor s}
In the sequel, we denote the natural inclusion $\mT\subset V_{\bZ_{(p)}}^\otimes$ by $s$. We assume that the element in $V_{\bZ_{(p)}}^\vee\otimes V_{\bZ_{(p)}}^\vee(1)$ induced by the alternating pairing $\psi$ belongs to $\mT$ and is denoted by $\psi$ as well.

If $R$ is a $\bZ_{(p)}$-algebra, we write $V_R:=V_{\bZ_{(p)}}\otimes_{\bZ_{(p)}}R$ and $R(1)=\bZ_{(p)}(1)\otimes_{\bZ_{(p)}}R$.
\end{notation}

In below, let 
\[\mH_{\on{B}}(A):= R^1h_{\bC,*}\bQ,\quad \mH_{\on{et}}(A):= R^1h_{\on{et},*}\bA_f,\quad \mH_{\dR}(A):= R^1h_{\dR,*}\Omega^\bullet\]
denote the first relative Betti, \'etale, and de Rham cohomology of $A$ (over $E$). The $\ell$-component of $\mH_{\on{et}}(A)$ is denoted by $\mH_{\on{et}}(A)_\ell$, and the prime-to-$p$ component is denoted by $\mH_{\on{et}}(A)^p$. Their duals, the first relative homology, are denoted by $\mV_{\on{B}}(A), \mV_{\on{et}}(A), \mV_{\dR}(A)$. We will apply Notation \ref{N: rigid abelian tensor} in Betti, \'etale, or de Rham setting. The Tate objects $\bQ_{\on{B}}(1), \bQ_{\on{et}}(1), \bQ_{\on{dR}}(1)$ in each setting are defined as in \cite[\S 1]{DMOS}. The $\ell$-component and the prime-to-$p$ component of $\bQ_{\on{et}}(1)$ are denoted by $\bQ_{\on{et}}(1)_\ell$ and $\bQ_{\on{et}}(1)^p$, respectively.

The representation $V$ of $G$ defines a Betti local system $\mV$ on $\Sh_K(G,X)(\bC)$
\[\mV= V\times^{G(\bQ)}(X\times G(\bA_f)/K),\]
which is canonically isomorphic to $\mV_{\on{B}}(A)$ of $A$ (because such canonical isomorphism exists over $\mathbf{Sh}_{K'}(\on{GSp}(V), \frakH^\pm)$). 
Since elements in $\mT$ are $G(\bQ)$-invariant, they define global sections of $\mV^\otimes$, and therefore induce
$$s_{\on{B}}: \mT\otimes\mathbf{1}\to \mH_{\on{B}}(A)^\otimes.$$ 
Note that, $s_{B,x}(\mT)$ is in the subalgebra of Hodge classes of $\on{H}^1(A_x,\bQ)^\otimes$, for every $x\in\Sh_K(G,X)(\bC)$.

By \cite[Lemma 2.2.1]{Ki1}, there is a map of \'etale local systems
$$s_{\on{et}}: \mT\otimes \mathbf{1}\to \mH_{\on{et}}(A)^\otimes,$$
which matches $s_B$ under the comparison isomorphism 
$ \mH_{\on{et}}(A)_\ell|_{\Sh_K(G,X)_\bC} \cong \mH_{\on{B}}(A)\otimes \bQ_\ell$.
In particular, for every extension $F/E$ in $\bC$, and an $F$-point $x$ of $\Sh_K(G,X)$, the algebra $s_{\on{et},x}(\mT)$ is in the subalgebra of Tate classes of $\on{H}^1_{\on{et}}(A_{x, \bar F},\bQ_\ell)^\otimes$.

The relative de Rham cohomology $\mH_\dR(A)$ is a vector bundle on $\Sh_K(G,X)$ equipped with a flat connection $\nabla$ and a decreasing filtration 
$$\mH_\dR(A) = \on{Fil}^0\mH_\dR(A)\supset \on{Fil}^1\mH_\dR(A)=(\Lie A/\Sh_K(G,X))^\vee\supset \on{Fil}^2\mH_\dR(A)=0.$$ 
By Deligne's theory of absolute Hodge cycles (cf. \cite{DMOS} and \cite[Corollary 2.2.2]{Ki1}), there is
\[s_\dR: \mT\otimes\mathbf{1}\to \mH_\dR(A)^\otimes,\]
that match $s_{\on{B}}$ under the comparison isomorphism $\mH_\dR(A)\otimes\bC\cong \mH_{\on{B}}(A)\otimes \bC$. In addition, $\on{Im}(s_\dR)\subset\on{Fil}^0\mH_\dR(A)^\otimes$.

Next, we move to the integral model.
If $\ell\neq p$, $\mH_{\on{et}}(A)_\ell$ exists as an \'etale local system on $\scrS_K(G,X)$, and $s_{\ell}: \mT\otimes\mathbf{1}\to \mH_{\on{et}}(A)_\ell$ extends to the whole $\scrS_K(G,X)$, still denoted by $s_{\ell}$. The de Rham cohomology $\mH_\dR(A)$ extends to a vector bundle with a flat connection and a decreasing filtration on $\scrS_K(G,X)$. It was proved in \cite[Corollary 2.3.9]{Ki1} that $s_\dR$ extends integrally as well, and denoted by the same notation. In addition, integrally $\on{Im}(s_\dR)\subset\on{Fil}^0\mH_\dR(A)^\otimes$.

Let $k_v$ be the residue field of $\mO_{E,v}$. We recall the definition of the rigid tensor category $F$-Crys$(T)$ of $F$-crystals on a (quasi-compact) $k_v$-scheme $T$. An object is a locally free crystal $\bD$ on the big crystalline site $(T/\mO_{E,v})_{\on{CRIS}}$, equipped with an quasi-isogeny  (the Frobenius map)
$$F: \sigma^*\bD\dashrightarrow \bD,$$
i.e. an element $F\in\Hom_{(T/\mO_{E,v})_{\on{CRIS}}}(\sigma^*\bD, \bD)\otimes\bQ$ such that there exists some element $V\in \Hom_{(T/\mO_{E,v})_{\on{CRIS}}}(\bD,\sigma^*\bD)\otimes\bQ$ (the Verschiebung) such that $VF=FV=p$\footnote{Of course, replacing $V$ by $p^{-1}V$, one can require $VF=FV=1$. We choose this convention so it is compatible with literatures on Dieudonn\'e crystals of $p$-divisible groups.}. Morphisms in $F$-Crys$(T)$ are morphisms of locally free crystals, compatible with $F$.

The unit object $\mathbf{1}$ is the structure sheaf of $(T/\mO_{E,v})_{\on{CRIS}}$, equipped with $F$ sending $1$ to $1$. The Tate object $\calO_{\on{crys}}(1)$ is the structure sheaf of $(T/\mO_{E,v})_{\on{CRIS}}$, equipped with $F$ sending $1$ to $p$.

Now, let $A_{k_v}$ be the mod $p$ fiber of the universal abelian variety, and $\bD(A_{k_v})$ its contravariant Dieudonn\'e crystal. This is an $F$-crystal.  In addition, $\bD(A_{k_v})(\scrS_{K,k_v})$ (the value of $\bD(A_{k_v})$ at the trivial PD thickening $\scrS_{K,k_v}\stackrel{\id}{\to}\scrS_{K,k_v}$) is canonically isomorphic to $\mH_\dR(A)|_{\scrS_{K,k_v}}$ and therefore is equipped with a decreasing filtration. It follows that $\bD(A_{k_v})^\otimes(\scrS_{K,k_v})$ is canonically isomorphic to $\mH_\dR(A)^\otimes|_{\scrS_{K,k_v}}$ and therefore is also equipped with a decreasing filtration.
Since $\scrS_K$ is smooth over $\mO_{E,v}$, It follows that
the horizontal map $s_\dR:\mT\to \mH_\dR(A)^\nabla$ induces
\[s_{0}: \mT\otimes\mathbf{1}\to \bD(A)^\otimes,\]
where $\bD(A)^\otimes$ is now considered as an object in (the ind-completion) of $F$-Isoc$(\scrS_{K, k_v})$.
In addition, $s_{0}(t)(\scrS_{K,k_v})\in \on{Fil}^0(\bD(A_{k_v})^\otimes(\scrS_{K,k_v}))$ for $t\in \mT$.

It follows that every characteristic $p$ point $x$ of $\scrS_{K}$, of residue field $\kappa$, is canonically equipped with a set of Tate classes $s_{\ell,x}(\mT)\subset \on{H}^1(A_{x,\bar \kappa},\bQ_\ell)^\otimes$, and a set of crystalline classes $s_{0,x}(\mT)\subset \bD(A_x)^\otimes$. 

\begin{rmk}
In fact, in \cite{Ki1}, Kisin first constructed $s_{0,x}$ for every characteristic $p$ point $x$ via the theory of Breuil-Kisin module in order to study the local structure of $\scrS_K$, and deduced as a corollary that $s_\dR$ extends integrally.
\end{rmk}

\subsubsection{The de Rham and crystalline $G$-torsors}
Since $\underline{G}\otimes \bZ_p$ is quasi-split, there is a representative of $\{\mu\}$ defined over $E_v$, and
up to conjugacy, we can assume that it extends to $\mu:\bG_m\to \underline{G}\otimes\mO_{E,v}$. Let $\on{Fil}^\bullet (V_{\bZ_p}^\vee\otimes\mO_{E,v})=(\on{Fil}^0\supset\on{Fil}^1\supset 0)$ be the decreasing filtration on $V_{\bZ_p}^\vee\otimes\mO_{E,v}$ induced by $\mu$ so that the subgroup
$P_\mu\subset \GL(V_{\bZ_p}^\vee)\times\GL(\bZ_p(1))$ that preserves $s$ and the filtration $\on{Fil}^\bullet$,
is a parabolic subgroup of $\underline{G}\otimes\mO_{E,v}$ determined by $\mu$. By abuse of notation, we still use $\underline G$ to denote the base change $\underline{G}\otimes\mO_{E,v}$. 

Let us also construct the local model diagram for $\scrS_K$. Note that local model diagram for Shimura varieties of abelian type with parahoric level structure has been established in \cite{KP} (under some mild restriction of $G_{\bQ_p}$).
In the case when $K_p$ is hyperspecial, this is quite easy so we include a proof for completeness. First, we have
\begin{lem}\label{P-torsor}
The (fppf) sheaf 
$$\mE_{\dR}:=\on{Isom}(((V_{\bZ_p}^\vee\oplus\bZ_p(1))\otimes \mO_{\mathscr S_K},s\otimes 1),(\mH_\dR(A)\oplus\bQ_{\dR}(1),s_{\dR})),$$
of isomorphisms of vector bundles $V^\vee_{\bZ_p}\otimes\mO_{\mathscr S_K}\simeq \mH_\dR(A)$ and $\bZ_p(1)\otimes \mO_{\mathscr S_K}\simeq \bQ_{\dR}(1)$ that send $s\otimes 1$ to $s_\dR$,
is a $\underline G$-torsor on $\scrS_K$, and the sheaf
$$\mP:=\on{Isom}(((V_{\bZ_p}^\vee\oplus \bZ_p(1))\otimes \mO_{\mathscr S_K}, s\otimes 1, \on{Fil}^\bullet ),(\mH_\dR(A)\oplus\bQ_{\dR}(1), s_{\dR}, \on{Fil}^\bullet ))$$ is a $P_\mu$-torsor on $\mathscr S_K$. There is a canonical isomorphism $\underline G\times^{P_\mu}\mP=\mE_{\dR}$. 
\end{lem}
\begin{proof}(1) We just consider $\mE_{\dR}$, and the statement for $\mP$ is similar (by taking account of the filtration).

We regard $\mE_{\dR}$ as a closed subscheme of the $\GL(V_{\bZ_p}^\vee)\times\GL(\bZ_p(1))$-torsor 
$$\on{Isom}((V_{\bZ_p}^\vee\oplus\bZ_p(1))\otimes\mO_{\scrS_K},\mH_\dR(A)\oplus\bQ_\dR(1)).$$ There is an action of $\underline G$ on $\mE_{\dR}$ and it is clear that the map
\[\underline G\times_{\mathscr S_K}\mE_{\dR}\to \mE_{\dR}\times_{\mathscr S_{K}}\mE_{\dR}\]
is an isomorphism. So it remains to show that $\mE_{\dR}$ is faithfully flat over $\mathscr S_K$. Note that when base change from $\mO_{E,v}$ to $\bC$, $\mE_{\dR}$ is clearly a $G$-torsor. So
we can pass to a formal neighbourhood $\hat{U}_x$ at a closed point $x$ in the special fiber of $\mathscr S_K$. Let $\kappa$ denote the residue field of $x$. Let $A_x$ denote the fiber of the above mentioned abelian variety at $x$.

Attached to $x$ there is a reductive group $G_x\subset \GL(\bD(A_x))\times \GL(\calO_{\on{crys}}(1))$ which is the schematically stabilizer of $s_{0}(\mT)$. It was shown in \cite[\S 1]{Ki1} that there exists an isomorphism $V_{\bZ_p}^\vee\otimes W(\kappa)\simeq \bD(A_x)$ intertwining $s: \mT\to L^\otimes$ and $s_0:\mT\to \bD(A_x)^\otimes$. Therefore, $G_x$ is a reductive group isomorphic to $\underline G\otimes_{\mO_{E,v}}W(\kappa)$.

Let $R_{G_x}$ denote the $G_x$-adapted deformation ring of $A_x[p^\infty]$ as in (\cite[\S 1.5]{Ki1} and \cite[\S 1.1]{Ki2}). Then by \cite[Proposition 2.3.5]{Ki1}, there is a natural isomorphism $\hat{U}_x\cong \on{Spf}R_{G_x}$ and the restriction of $(\mH_\dR(A),s_{\dR},\nabla)$ to $\hat{U}_x$ is identified with $(M=\bD(A_x)\otimes_{W(\kappa)} R_{G_x}, s_{0}\otimes 1,\nabla)$ on $\on{Spf} R_{G_x}$. It follows from the existence of an isomorphism $\bD(A_x)\simeq V_{\bZ_p}^\vee\otimes W(\kappa)$ as above that the restriction of $\mE_{\on{dR}}$ to $\hat{U}_x$ is a $\underline G$-torsor, so is faithfully flat over $\hat{U}_x$. 
\end{proof}

\begin{rmk}
In fact, $\mE_\dR$ is defined over $\mathscr S_K(G,X)$.
\end{rmk}

Now, let $\widetilde{\mathscr S_K}$ denote the scheme over $\mathscr S_K$ of the trivialization of the $\underline G$-torsor $\mE_\dR$. 
\begin{cor}\label{local model}
There is the following local model diagram
\[\mathscr S_K\stackrel{\pi}{\longleftarrow} \widetilde{\mathscr S_K}\stackrel{\tilde\varphi}{\longto} \underline{G}/P_\mu,\]
where $\pi$ is a $\underline G$-torsor and $\tilde\varphi$ is $\underline G$-equivariant smooth of relative dimension $\dim G$. Moreover, if $U\to \widetilde{\mathscr S_K}$ is a morphism such that the induced map $U\to \mathscr S_K$ is \'etale, the induced map $U\to \underline G/P_\mu$ is \'etale.  
\end{cor}
\begin{proof}
By the previous lemma, the trivial $\underline G$-torsor $\mE_\dR\times _{\mathscr S_K}\widetilde{\mathscr S_K}$ is induced by the $P_\mu$-torsor $\mP\times_{\mathscr S_K}\widetilde{\mathscr S_K}$. This defines the map  $\tilde\varphi: \widetilde{\mathscr S_K}\longto \underline{G}/P_\mu$. It follows from the construction that $\tilde\varphi$ is $\underline G$-equivariant. It remains to prove the last statement (which implies that $\tilde\varphi$ is smooth of relative dimension $\dim G$).

The last statement is clear for the characteristic zero points. Now let $y$ be a characteristic $p$ point of $U$ mapping to $\tilde x\in \widetilde{\mathscr S_K}$ and $x \in \mathscr S_K$. Let $\bar{x}$ denote the image of $\tilde x$ in $\underline{G}/P_\mu$.
Then the completed local ring $\hat{V}_{\bar x}$ at $\bar{x}$ can be identified with the deformation ring $R_{G_x}$  in such a way that the natural filtration on $L^\vee\otimes \hat{V}_{\bar x}$ is identified with the Hodge filtration of $\bD(A_x)\otimes_{W(\kappa)}R_{G_x}$.
Since $\hat{U}_x\simeq \on{Spf} R_{G_x}$, the \'etaleness of $U$ over $\scrS_K(G, X)$ at $y$ (or equivalently $\hat U_y \simeq \hat U_x$) implies that the induced map $U \to \underline G/P_\mu$ takes $\hat U_y$ isomorphically to $\on{Spf} R_{G_x} \simeq \on{Spf}\hat V_{\bar x}$. This verifies the last statement.
\end{proof}

Another application of Lemma~\ref{P-torsor} is as follows.
\begin{cor}\label{Fcrys on Sh}
The sheaf 
\[R\mapsto \on{Isom}\Big(\big((V_{R}^\vee\oplus R(1)),s\big), \big(\bD(A)(R)\oplus \calO_{\on{crys}}(1)(R),s_0\big)\Big)\]
 on $(\scrS_{K,k_v}/\mO_{E,v})_{\on{CRIS}}$, where the isomorphisms are isomorphisms of $R$-modules $V_R^\vee\simeq \bD(A)(R)$ and $R(1)\simeq \calO_{\on{crys}}(1)(R)$ sending $s$ to $s_0$, defines a $\underline G$-torsor $\mE_{\cris}$ on $(\scrS_{K,k_v}/\mO_{E,v})_{\on{CRIS}}$.
In addition, there is a map of $\underline G$-torsors $\mE_{\cris}\to {^\sigma}\mE_{\cris}$, whose relative position is $ \sigma(\mu)$.
\end{cor}

Note that by definition there is a canonical isomorphism $\mE_{\cris}(\mathscr S_K\otimes\bF_p)\cong \mE_\dR\otimes \bF_p$ (cristalline-de Rham comparison).
\begin{proof}
If $x$ is an $R$-point of $\scrS_{K,k_v}$ and $\tilde R\to R$ a PD thickening, we (locally) lift $x$ to $\tilde x:\Spec \tilde R\to \scrS_{K_x}$. Then $\mE_{\cris}(\tilde R)=\tilde x^*\mE_\dR$, which is independent of the lifting.
The $F$-crystal structure $$\bD(A_x)^\vee\to \bD(A_x^{(p)})^\vee=\sigma^*\bD(A_x)^\vee$$ induces a quasi-isogeny $x^*\mE_{\cris}\to x^*({^\sigma}\mE_{\cris})$. Note that we use the dual of contravariant Dieudonn\'e modules to define the $F$-crystal with $\underline G$-structures, as by definition $\underline G$ is considered a subgroup of $\GSp(V_{\bZ_p},\psi)\otimes\mO_{E,v}$, where $V$ corresponds to homology (as opposed to the cohomology) of the universal abelian variety.

To prove the last statement,
we can work pointwise. So assume $x$ is an $\overline\bF_p$-point. The relative position of
\[x^*((\sigma^{-1})^*\mE_{\cris})\to x^*\mE_{\cris}\]
is given by $\mu$ since the morphism on the associated bundle 
is the Verschiebung
\[V: (\sigma^{-1})^*\bD(A_x)^\vee\to \bD(A_x)^\vee=\mE_x\times^{\underline G} V_{\bZ_p}.\qedhere
\]
\end{proof}
\begin{example}
We give an example to justify our computation of the relative position:
let $F$ be a cubic totally real field and $E$ a CM extension of $F$.
Let $p$ be a prime inert in $F$ and splits as $\frakp \bar \frakp$ in $E$.
Consider the unitary group $G = G(U(1,2) \times U(0,3)\times U(0,3))$.
Let $\tau_1, \tau_2, \tau_3$ be the $p$-adic/complex embeddings corresponding to $\frakp$.
Under the usual identification of $G_\CC$ with $\GL(3)^{\times 3} \times \GG_m$ (where the first three factors use associated complex embeddings $\tau_1, \tau_2, \tau_3$), $\mu(z)$ is $\diag\{z, 1, 1\}$, $\mathrm{id}$, and $\mathrm{id}$ on the first three factors.
Then $\dim \Lie(A)_{\tau_i}$ is equal to $1,0,0$ for $i = 1,2,3$, respectively.
It follows that the dimension of the cokernel of $V_{\tau_i}: \DD(A_x)^\vee_{\tau_i} \to (\sigma^*\DD(A_x)^\vee)_{\tau_i}$ has dimension $0, 1, 0$, respectively.
From this, we see that the associated $G$-torsor map $\calE_\cris \to {}^\sigma \calE_\cris$ has relative position $\sigma(\mu)$.
\end{example}

\subsubsection{Level structure} 
\label{SS: level structure}
Let $x\in \scrS_K(G,X)(T)$ be a point. Its image on $\mathscr A_{g,K'^p}$ is represented by a triple $(A_x,\la_x,\eta')$, as explained in \S \ref{SS: moduli of abelian scheme}. In particular, 
$$\eta'\in \Gamma(T, \on{Isom}((V\oplus\bQ(1))\otimes\bA_f^p, \mV_{\on{et}}(A)^p\oplus\bQ_{\on{et}}(1)^p)/K'^p).$$ But as explained in \cite[\S 3.4.2]{Ki1}, the level structure $\eta'$ comes from a canonical section
\[\eta_x\in \Gamma\big(T,\on{Isom}\big(((V\oplus\bQ(1))\otimes \bA_f^p,s\otimes 1), (\mV_{\on{et}}(A)^p\oplus\bQ_{\on{et}}(1)^p,s_{\on{et},x})\big)\big/K^p \big).\]
Here $\on{Isom}(((V\oplus\bQ(1))\otimes \bA_f^p,s\otimes 1), (\mV_{\on{et}}(A)^p\oplus\bQ_{\on{et}}(1)^p,s_{\on{et},x}))$ is the subsheaf of $\on{Isom}((V\oplus\bQ(1))\otimes \bA_f^p,\mV_{\on{et}}(A)^p\oplus\bQ_{\on{et}}(1))$, which intertwines $s: \mT\subset V^\otimes$ and $s_{\on{et},x}: \mT\otimes \mathbf{1}\to (\mH_{\on{et}}(A)^p)^\otimes=(\mV_{\on{et}}(A)^p)^\otimes$. Note that this is a $G(\bA_f^p)$-torsor so that the quotient by $K^p$ makes sense.

\subsection{The mod $p$ fiber of $\mathscr S_K(G,X)$}
We first relate the special fibers of the integral models of Shimura varieties recalled above to the moduli of local shtukas via Proposition \ref{smoothness}. Then we recollect some results about the Rapoport-Zink uniformization of the basic Newton stratum.

\subsubsection{Perfection of the characteristic $p$ fiber of $\mathscr S_K$}
\label{SS: universal G sht}
Now we concentrate on mod $p$ fiber. Let $k_v$ denote the residue field of $\calO_{E,v}$ and fix $\bar k_v$ an algebraic closure  of $k_v$. We denote by 
$$\Sh_{\mu,K}:=(\mathscr S_K\otimes k_v)^\pf$$ 
the perfection of the special fiber of $\mathscr S_K$, or just $\Sh_\mu$ if the level structure $K$ is clear from the context. We also write $\Sh_{\mu,K_p}$ for $\underleftarrow\lim_{K^p}\Sh_{\mu,K_pK^p}$.
As suggested by the proof of Proposition \ref{Fcrys on Sh}, it is convenient to twist the $F$-crystal $\mE_{\cris}$ to consider
\[\mE:= (\sigma^{-1})^*\mE_{\cris}.\]
Then we have the $F$-crystal of $\underline G$-torsors (i.e. a local $\underline G$-shtuka, see Remark \ref{R: var and gen of loc Sht})
\[\beta:\mE\dashrightarrow {^\sigma}\mE\]
over $\Sh_{\mu,K}$, of relative position $\mu$.

\begin{rmk}
\label{R: trivialize Fcrys}
Note that by our definition of $\mE$, it only exists over $\Sh_{\mu,K}$, but not over $\mathscr S_K$. On the other hand, 
from the theory of Breuil-Kisin module, it is also natural to define $\mE$ as $(\sigma^{-1})^*\mE_{\cris}$. Namely, 
Let $x$ be
a $\kappa$-point of $\Sh_{\mu,K}$ and choose a lifting $\tilde x$ of $x$ to a characteristic zero point. Let 
$$\frakM_{\tilde x}=\frakM(T_pA_{\tilde x}^\vee(-1))$$ 
be the Kisin module associated to the crystalline representation of the $(-1)$-twist of the $p$-adic Tate module of the dual abelian variety $A_{\tilde x}^\vee$ (which is isomorphic to $V_{\bZ_p}^\vee$). Then $s_{p,\tilde x}: \mT\to (T_pA^\vee_{\tilde x}(-1))^\otimes$ induces $s_{\on{BK}}: \mT\to \frakM_{\tilde x}^\otimes$. 
Recall that $\frakM_{\tilde x}$ is a free $\frakS=W(\kappa)[[u]]$-module and $\varphi^*(\frakM_{\tilde x}/u\frakM_{\tilde x})$ is identified with $\bD(A_x)$, where $\varphi:\frakS\to\frakS$ is the Frobenius. Therefore,
$x^*\mE$ is the $\underline G\otimes W(\kappa)$-torsor of isomorphisms $V_{\bZ_p}^\vee\otimes W(\kappa) \simeq \frakM_{\tilde x}/u\frakM_{\tilde x}$ that sends $s(\mT)$ to  the tensors $s_{\on{BK}}(\mT) \mod u$ up to a common scalar.
\end{rmk}

\subsubsection{The map $\loc_p$} By only remembering the local $\underline G$-shtuka $\psi:\mE\dashrightarrow {^\sigma}\mE$, we obtain a morphism of prestacks
\[\loc_p: \Sh_{\mu}\to \Sht_{\mu}^\loc.\]
In the Siegel case, $\loc_p$ is the perfection of the morphism sending an abelian variety to its underlying $p$-divisible group. According to Serre-Tate theory, this morphism (before perfection) is formally \'etale. However formal \'etaleness is not a good notion for morphisms between perfect schemes, so we will not make use of it. Instead, we will compose it with the morphism from $\Sht_{\mu}^\loc$ to the moduli of restricted local shtukas. Let $(m,n)$ be a pair of non-negative integers such that $m-n$ is $\mu$-large. We set
\[\loc_p(m,n):\Sh_{\mu}\xrightarrow{\loc_p} \Sht_{\mu}^\loc\xrightarrow{\res_{m,n}} \Sht_{\mu}^{\loc(m,n)}.\]

The following result establishes the relation between the special fibers of Shimura varieties and the moduli of local shtukas defined earlier.
Recall the definition of perfectly smooth morphisms in Appendix \ref{ASS:perfect AG}. 
\begin{prop}\label{smoothness}
The morphism $\loc_p(m,n)$ is perfectly smooth.
\end{prop}
\begin{proof}
Let $\Sh_{\mu}^{(m,n),\Box}$ denote the $L^{m}\underline G$-torsor over $\Sh_{\mu}$ that classifies, for each perfect ring $R$, an $R$-point $x$ of $\Sh_{\mu}$ and a trivialization of the $L^{m}\underline G$-torsor ${^\sigma}\mE_x|_{D_{m}}\simeq \mE^0|_{D_{m}}$ over $D_{m,R}$. Recall that we have defined its local version, namely the moduli of framed restricted local shtukas $ \Sht_\mu^{\loc(m,n),\Box}$ in \eqref{E:framed local shtukas}, which is canonically isomorphic to $\Gr_\mu^{(n)}$. We have the following canonical isomorphism
\begin{equation} \Sh_{\mu}^{(m,n),\Box}\cong\Sh_{\mu,K}\times_{\Sht_{\mu}^{\loc(m,n)}} \Sht_{\mu}^{\loc(m,n),\Box}.
\end{equation}
So to prove that $\loc_p(m,n)$ is perfectly smooth, it suffices to check that  the map
\[
\loc_p^\Box(m,n)\colon\Sh_{\mu}^{(m,n),\Box}\longto \Sht^{\loc(m,n),\Box}_{\mu}\]
is perfectly smooth.

Let $x^\Box$ be a $\bar k_v$-point of $\Sh_{\mu}^{(m,n),\Box}$, whose image in $\Sh_{\mu}$ is denoted by $x$. By  \cite[Corollary A.27]{Z}, we can choose a small \'etale neighborhood $a: U = \Spec R \to \overline{\scrS}_K$ of $x$ such that both the pullback of the $L^{m}G$-torsor ${}^\sigma \calE|_{D_{m}}$ to $U$ and the pullback of the $L^nG$-torsor $\calE|_{D_n}$ to $U$ are trivial.
We fix one such trivialization $\epsilon_{m}: {}^\sigma \calE|_{D_{m,R}} \cong \calE^0|_{D_{m,R}}$, which is equivalent to giving a lift $a^{(m)}: U \to \Sh_{\mu,K}^{(m,n),\Box}$ of $a$.
We put
\[
a^{(m,n)}_\loc:= \loc_p^\Box(m,n)\circ a^{(m)}: U \to \Sh_{\mu}^{(m,n),\Box} \to \Sht_\mu^{\loc(m,n),\Box} \cong \Gr_\mu^{(n)}.
\]
Then the natural map
\[
\xymatrix@R=0pt{
U \times L^{m}\underline G \ar[r]& \Sh_{\mu,K}^{(m,n),\Box}
\\
(u,g) \ar@{|->}[r] & (a^{(m)}(u), \sigma(g)\epsilon_{m})
}
\]
is \'etale and gives an \'etale chart near at the point $x^\Box$.
So it suffices to show that the composition
\begin{equation}
\label{E:local chart for locp(m,n)}
U \times L^{m}\underline G \longto \Sh_{\mu}^{(m,n),\Box}\xrightarrow{\loc_p^\Box(m,n)}  \Sht_\mu^{\loc(m,n),\Box} \cong  \Gr_\mu^{(n)}
\end{equation}
is perfectly smooth.
In explicit terms, this map is given by
\begin{equation}
\label{E:explicit aloc(m,n)Box}
(u, g) \mapsto \sigma(g)\cdot a_\loc^{(m,n)}(u)\cdot \pi_{m,n}(g)^{-1}.
\end{equation}
See \eqref{E:twisted action of LmG on framed local shtukas}. (Note that $g$ in \eqref{E:twisted action of LmG on framed local shtukas} is $\sigma(g)$ here.)

We note that the perfection of the local model diagram in Corollary~\ref{local model} mod $p$ is exactly
\[
\Sh_{\mu} \leftarrow \Sh_{\mu}^{(1,0),\Box} \to (\underline G /P_\mu \otimes k_v)^\pf \cong \Gr_\mu,
\]
where the second arrow may be identified with the map $\loc_p^\Box(1,0) :\Sh_{\mu}^{(1,0),\Box} \to \Sht_\mu^{\loc(1,0),\Box} \cong \Gr_\mu$.
The lift $a^{(m)}:U \to \Sh_{\mu}^{(m,n),\Box}$ naturally gives gives a map $a^{(1)}: U \to \Sh_{\mu}^{(1,0),\Box}$.
The natural compatibility of localization map gives the following commutative diagram
\[
\xymatrix{
U\times L^{m}\underline G \ar[r]\ar[d]_{\mathrm{proj}_1} & \Sh_{\mu,K}^{(m,n),\Box} \ar[rr]^-{\loc_p^\Box(m,n)}&& \Gr_\mu^{(n)} \ar[d]
\\
U \ar[r]^-{a^{(1)}} & \Sh_{\mu,K}^{(1,0),\Box} \ar[rr]^-{\loc_p^\Box(1,0)}& & \Gr_\mu.
}
\]
The composition of the bottom arrow, namely $a^{(1,0)}_\loc$, is \'etale by Corollary~\ref{local model}.

This now has become a purely local question: knowing an \'etale map $a_\loc^{(1,0)}: U \to \Gr_\mu$ lifts to $a_\loc^{(m,n)}: U \to \Gr_\mu^{(n)}$, we shall prove that the associated map $U \times L^{m} \underline G \to \Gr_\mu^{(n)}$ given by formula \eqref{E:explicit aloc(m,n)Box} is perfectly smooth.
This is equivalent to show that the induced map
\begin{equation}
\label{E:U L^mG to UGrn}
U \times L^{m}\underline G \longto U \times_{\Gr_\mu} \Gr_\mu^{(n)}
\end{equation}
is perfectly smooth.  To prove this, we need to use the imperfect version.

Now, let $L^m_p \underline G$ denote the usual Greenberg realization of $\underline G\otimes \bZ/p^m$ so that $L^m_p \underline G$ is a smooth algebraic group over $\bF_p$ whose perfection is $L^{m}\underline G$ (cf. \cite[\S 1]{Z}). 
The \'{e}tale map $a:U\to \Gr_{\mu}$ descends uniquely to an \'{e}tale morphism $a':U'\to \underline G/P_\mu\otimes k_v$.
If we write $\calE_1\dashrightarrow \calE_0 = \calE^0$ denote the universal modification over $\Gr_\mu$, the $L^n\underline G$-torsor $\calE_1|_{D_{n,R}}$ is trivial over $U$.
Thus, if we write $U'^{(n)}$ for the trivial $L^n_p \underline G$-torsor over $U'$, then
\[
(U'^{(n)})^\pf \cong U \times_{\Gr_\mu} \Gr_\mu^{(n)}.
\]
Now, the map \eqref{E:U L^mG to UGrn} is the perfection of the map 
\begin{equation}
\label{E:nonperfect U' to U'(n)}
\xymatrix@R=0pt{
U' \times L^{m}_p \underline G \ar[r] & U'^{(n)}
\\
(u,g) \ar@{|->}[r] & \sigma(g) a_\loc^{(m,n)}(u) \pi'_{m,n}(g)^{-1},
}
\end{equation} where $\pi'_{m,n}: L^{m}_p \underline G \to L^n_p \underline G$ is the natural projection.
On the level of tangent space, we can ignore the left multiplication by $\sigma(g)$. Hence, the infinitesimal action of $L^{m}_p \underline G$ on $U'^{(n)}$ is entirely along the fiber direction of $U'^{(n)} \to U'$.
So the map \eqref{E:nonperfect U' to U'(n)} is smooth, and hence \eqref{E:U L^mG to UGrn} is perfectly smooth.
This concludes the proof of the proposition.
\end{proof}
\begin{rmk}
\label{R:Sh to Gzip}
The readers can skip this remark.
Recall from Lemma \ref{Ex:Sht and GZip} that there is a natural perfectly smooth morphism $\Sht_\mu^{\loc(2,1)}\to G\on{-Zip}_\mu$. It is easy to see that the composition
\[\Sh_\mu\xrightarrow{\loc_p(2,1)} \Sht_\mu^{\loc(2,1)}\to G\on{-Zip}^\pf_\mu\]
coincides with the perfection of the map defined in \cite{ZhangChao} (which in the case of $G=\GSp(V,\psi)$ is the map sending an abelian variety $A$ to its $p$-torsion $A[p]$). The following is a corollary of Proposition \ref{smoothness}.
\begin{cor}
The morphism $\Sh_\mu\to G\on{-Zip}^\pf_\mu$ is perfectly smooth.
\end{cor}
In fact, by an argument similar to the proof of Proposition \ref{smoothness}, one can show that the map $\scrS_{K,k_v}\to G\on{-Zip}_\mu$ is already smooth before taking the perfection. See \cite{ZhangChao}.
\end{rmk}

\subsubsection{Newton stratification}
\label{SS: NS for Shimura}
Recall the definition of the Newton map \eqref{E: Nm for loc Sht} for the moduli of local shtukas. By composition, one obtains the Newton map for Shimura varieties
\[\mN: \Sh_{\mu,K}(\bar k_v)\stackrel{\loc_p}{\longto} \Sh_{\mu}^\loc(\bar k_v)\stackrel{\mN}{\longto} B(G_{\bQ_p},\mu^*).\]
For $[b]\in B(G_{\bQ_p},\mu^*)$, let $\mN_b = \mN_{[b]} := \mN^{-1}([b])$ denote the corresponding \emph{Newton stratum}. It is known (e.g. \cite[Section~5.2]{Wortmann}) that $\mN_b$ is a locally closed subset of $\Sh_{\mu,K}$ and its closure
$\overline\mN_b$ is contained in the union $ \bigcup_{b'\preceq b}\mN_{b'}$.
In particular, if $[b]\in B(G_{\bQ_p},\mu^*)$ is basic,  $\mN_b$ is a closed subset, denoted by $\Sh_{\mu,K,b}$ or just $\Sh_{\mu,b}$ if the level structure $K$ is clear from the context.

\begin{lem}
The basic Newton stratum $\Sh_{\mu,K,b}$ is non-empty.
\end{lem}
It has been proved by Lee \cite{Lee} that all Newton strata are non-empty (and we learned that Kisin-Madapusi Pera-Shin proved more general statement under weaker assumptions of $G_{\bQ_p}$). Since the case of basic Newton stratum is particularly simple, we include an argument here for the basic Newton stratum for readers' convenience.
\begin{proof}
We first claim that there is a maximal torus $T$ of $G$, defined over $\bQ$, such that both $T(\bR)$ and $T(\bQ_p)$ are anisotropic modulo the centers of $G(\bR)$ and $G(\bQ_p)$, respectively. Namely, we apply weak approximation to find a strongly regular semisimple element $\ga$ defined over $\bQ$ such that $\ga_\infty$ and $\ga_p$ are elliptic. Then $T$ is defined as the centralizer of $\ga$.

Next, we claim that there exists $h\in X$ such that $h(\bC^\times)\subset T(\bR)$. Indeed, let $h'$ be a point of $X$. Then $h'(\bC^\times)$ is contained in a maximal torus $T'\subset G_\bR$, anisotropic modulo center. Therefore, there is some $g\in G(\bR)$ such that $gT'g^{-1}=T_{\bR}$. So $h=gh'g^{-1}$ satisfies the required property.
In particular, the pair $(T,h)$ defines a special point of $X$. By choosing appropriate field of definition and tame level structure, we obtain a point $\tilde{x}$ in $\Sh_K(G,X)$, with a reduction to a point $x$ in $\Sh_\mu$, such that the Mumford-Tate group of the universal abelian variety $A_{\tilde x}$ at $\tilde x$ is contained in $T$. 
Then $\mN(x)$ is contained in the image of $B(T)\to B(G)$. Since $T_{\bQ_p}$ is elliptic, we know that this image must be basic in $B(G)$ (cf. \cite[Proposition 5.3]{Koisocry}).
\end{proof}

\subsubsection{The group $I_x$} \label{SS: the group I}
We assume that $x$ is a $\kappa$-point of $\Sh_{\mu,K}$, where $\kappa\supset k_v$ is a finite extension, $\sharp\kappa=q=p^r$, and
let $x$ be a $\kappa$-point of $\scrS_K$. We denote $W_x=W(\kappa)$ and $K_x=W_x[1/p]$. Let $\bar x$ be the induced $\bar \kappa$-point over $x$. We similarly have $W_{\bar x}$ and $K_{\bar x}$.
 Let $\pi_{A_x}$ denote the $q$-Frobenius endomorphism of $A_x$. Recall that by construction, we have a collection of tensors $s_{\on{et},x}: \mT\to \on{H}^1(A_{\bar x},\bQ_\ell)^\otimes$ and $s_{0,x}:\mT\to \bD(A_x)^\otimes$. Recall we assume that $\mT$ contains $\psi\in (V_{\bZ_{(p)}}^\vee\otimes V_{\bZ_{(p)}}^\vee(1))^{\underline G}$ corresponds to the alternating pairing.
 Then $(s_{\on{et},x}(\psi), s_{0,x}(\psi))$ is induced by a principal polarization $\la_x$ of $A_x$.


We recall the construction of some groups 
\[I_{\ell/\kappa },\ I_{\ell}, \ I_{p/\kappa }, \ I_{p},\ I_{/\kappa},\ I,\]
from \cite[\S 2.1.2]{Ki2}. First, 
let 
$$I_{\ell/\kappa}\subset \on{GL}(\on{H}^1(A_{\bar x},\bQ_\ell))\times \GL(\bQ_{\on{et}}(1)_\ell)$$ 
be the $\bQ_\ell$-subgroup, consisting of elements that fix $s_{\ell,x}$ and commute with the action of $\pi_{A_x}$. Replacing $\kappa$ by a degree $n$ extension $\kappa_n$ one obtains $I_{\ell/\kappa_n}$. Then sequence $\{I_{\ell/\kappa_n}\}$ is increasing and stabilizes to a $\bQ_\ell$-subgroup $I_{\ell}\subset\GL(\on{H}^1(A_{\bar x},\bQ_\ell))\times \GL(\bQ_{\on{et}}(1)_\ell)$. 

Next, $I_{p/\kappa}$ is the $\bQ_p$-group whose $R$-points consist of $W_x\otimes_{\bZ_p} R$-linear automorphisms in 
$$\GL(\bD(A_x)\otimes_{\bZ_p} R)\times \GL(\calO_{\on{crys}}(1)(W_x)\otimes R)$$ 
that fix $s_{0,x}\otimes 1$ and commute with the Frobenius map $F: \sigma^*\bD(A_x)\to \bD(A_x)$. Note that if $R$ is a $K_x$-algebra, $I_{p/\kappa}(R)\subset \GL(\bD(A)\otimes_{W_x}R)\times \GL_1(R)$ is the subgroup of elements that fix $s_0\otimes 1$ and commute with $\ga_p=F^r$.
By replacing $\kappa$ by $\kappa_n$, one obtains $I_{p/\kappa_n}$. Then $I_{p/\kappa_n}\otimes K_{\bar x}$  is the centralizer of $\ga_p^n$ in $\GL(\bD(A)\otimes_{W_x}K_{\bar x})$. Therefore,  $I_{p/\kappa_n}$ stabilizes to a $\bQ_p$-group $I_p$. Note that if we choose an isomorphism $\bD(A_x)\simeq V_{\bZ_p}^\vee\otimes W_x$ and $\calO_{\on{crys}}(1)(W_x)\simeq \bZ_p(1)\otimes W_x$ that sends $s_{0,x}$ to $s\otimes 1$ (i.e. a trivialization of $x^*\mE$), then $F$ is represented by $b_x\sigma$ for an element $b_x\in G(W_x)\mu^*(p)G(W_x)$. Then 
\[I_{p/\kappa}(R)=\{g\in G(W_x\otimes R)\mid g^{-1}b_x\sigma(g)=b_x\}\subset J_{b_x}(R),\]
where $J_{b_x}$ is the twisted centralizer of $b_x$ as defined in \eqref{E: twisted centralizer J}. In particular, $I_p\subset J_{b_x}$.

Finally, we define
$I_{/\kappa}$ to be the group whose $R$-points are elements of $(\End(A_x)\otimes R)^\times \times R^\times$ that preserve the tensors $s_{\ell,x}$  for ${ \ell\neq p}$ and $s_{0,x}$. 
Define $I$ similarly by replacing $(A_x, (s_{\on{et},x},s_{0,x}))$ by $(A_{\bar x}, (s_{\on{et},\bar x},s_{0,\bar x}))$. Note that if we denote by $f\mapsto f^*$ denote the Rosati involution on $\End(A_{\bar x})\otimes \bQ$ induced by the polarization on $A_x$, then $I(R)\subset \{g\in (\End(A_{\bar x})\otimes R)^\times\mid gg^*\in R^\times\}$. In particular, there is a central $\bG_m\subset I$ given by the scalar multiplication, and $(I/\bG_m)(\bR)$ is compact. Also note that the Frobenius endomorphism  $\pi_{A_x}$ of $A_x$ is an element in $I_{/\kappa}(\bQ)$.

Note that $I$ only depends on $\bar x$. Conversely, we can attach to each $\bar k_v$-point of $\Sh_{\mu,K}$ the group $I$, since every such point is indeed defined over some finite field, 
To emphasize the dependence of $I$ on the base point $\bar x$, we also denote it by $I_{\bar x}$.

Note that by definition, for every $\ell$ (including $\ell=p$), there are natural inclusions 
\[I_{/\kappa}\otimes\bQ_\ell\to I_{\ell/\kappa},\quad I\otimes\bQ_\ell\to I_{\ell},\]
which turn out to be isomorphisms by \cite[Corollary 2.3.2]{Ki2}. In addition, it was proved in \cite[Corollary~2.3.5]{Ki2} that $I$ is the inner form of a Levi subgroup $I_0$ of $G$. More precisely, by choosing an isomorphism $\on{H}^1(A_{\bar x},\bQ_\ell)\oplus \bQ_{\on{et}}(1)_\ell\simeq (V^\vee\oplus \bQ(1))\otimes\bQ_\ell$ compatible with $s_{\on{et},x}$ and $s$, and an isomorphism $\bD(A_x)\oplus \calO_{\on{crys}}(1)(W_x)\simeq (V_{\bZ_p}^\vee\oplus\bZ_p(1))\otimes W_x$ compatible with $s_{0,x}$ and $s\otimes 1$, we may regard $\pi_A$ as an element $\ga_\ell\in G(\bQ_\ell)$  (including $\ell=p$), whose conjugacy class (twisted conjugacy when $\ell = p$) is independent of the choice. Then by \cite[Corollary 2.3.1]{Ki2}, there is an element $\ga_0\in G(\bQ)$ which is 
\begin{itemize}
\item conjugate to $\ga_\ell$ in $G(\bQ_\ell)$ for $\ell\neq p$;
\item is stably conjugate to $\ga_p$ in $G(\overline{\bQ}_p)$;
\item elliptic in $G(\bR)$.
\end{itemize} 
Let $I_0\subset G$ be the centralizer of the neutral connected component of the closed subgroup generated by $\ga_0$. Then it was proved in \cite[Corollary 2.3.5]{Ki2} that $I$ is an inner form of $I_0$.

\subsubsection{Isogeny classes}
\label{SS:base point}
We say two $\bar k_v$-points $x$ and $y$ of $\Sh_{\mu,K}$ are in the same isogeny class if and only if there is a quasi-isogeny $f:A_x\to A_{y}$ that takes $(s_{\on{et},x},s_{0,x})$  to $(s_{\on{et},y},s_{0,y})$ up to a common scalar multiple $c\in \bQ^\times$. Note that this in particular implies that $f^*\la_y=c\la_x$. We denote the isogeny class containing $x$ by $\mathscr I_x$. It is clear that such a map $f: A_x\to A_y$ induces an isomorphism $I_x\simeq I_y$.

Let $x \in \Sh_{\mu,K}$ be a $\kappa$-point, and put $W_x = W(\kappa)$. Let $K_x=W_x[1/p]$.
If we fix an isomorphism $\bD(A_x)\simeq V_{\bZ_p}^\vee\otimes W_x$ sending $\{s_{0,x}\}$ to $\{s\otimes 1\}$, or equivalently fix a trivialization of $x^*\mE$ (see Remark \ref{R: trivialize Fcrys}), we get an honest element $b_x\in [b_x]\cap G(W_x) \mu^*(p) G(W_x)\subset G(K_x)$. Note that $b_x$ is well-defined up to $\sigma$-conjugacy by an element in $G(W_x)$. Let $X_{\mu^*}(b_x)$ be the corresponding ADLV.
We shall construct a natural morphism 
\[X_{\mu^*}(b_x)\to \Sh_{\mu,K}\otimes \bar \kappa\]
whose image are in an isogeny class.
The map on the level of closed points was constructed in \cite[Proposition 1.4.4]{Ki2}. 
Recall that the space $X_{\mu^*}(b_x)$ can be viewed as a reduced version of the Rapoport-Zink space. Indeed, it was proved in \cite[\S 3.2]{Z} that $X_{\mu^*}(b_x)$ is the perfection of the Rapoport-Zink space constructed in \cite{Kim} and \cite{HP2}. Then one can deduce the above desired morphism from \cite{Kim} and \cite{HP2}. However, we prefer to give a direct construction, which does not rely on the existence of the Rapoport-Zink space as a formal scheme. It is based on the following lemma, whose proof is inspired by an argument in \cite{HP2}.

\begin{lem}
\label{L: from point to family}
Let $R$ be a reduced $k_v$-algebra of finite type, and let $f: \Spec R\to \mathscr A_{g,K'}$ be a morphism given by a principally polarized abelian scheme up to prime-to-$p$ isogeny $(B,\mu^*)$. 
Let 
\[\iota_{\on{et}}: \mT\otimes\mathbf{1}\to (\mH_{\on{et}}(B)^p)^{\otimes}, \qquad \iota_{0}: \mT\otimes\mathbf{1}\to \bD(B)^{\otimes}\]
be two graded algebra homomorphisms. Assume that for every closed point $x\in\Spec R$, there is a point $y\in \scrS_K$ over $f(x)$ such that $(B_x, \iota_{\on{et},x},\iota_{0,x})$ is isomorphic to $(A_y, s_{\on{et},y},s_{0,y})$ in the sense that there is a prime-to-$p$ quasi-isogeny $f:B_x\to A_y$ sending $(s_{\on{et},y},s_{0,y})$ to $(\iota_{\on{et},x},\iota_{0,x})$ (by \cite[Corollary 1.3.11]{Ki2}, such $y$ if exists is unique). Then there is a unique map $\Spec R\to \scrS_K$ such that $(B,\iota_{\on{et}},\iota_{0})$ is the pullback of $(A,s_{\on{et}},s_0)$ on $\scrS_K$.
\end{lem}
\label{L: from pt to family}
\begin{proof}
We define a scheme $Y$ by the following Cartesian diagram
\[\begin{CD}
Y@>j_Y>> \scrS_K\\
@V\pr_Y VV@VVV\\
\Spec R@>>> \mathscr A_{g,K'}
\end{CD}\]
Then $Y$ is a scheme finite over $\Spec R$, and therefore is affine of finite type over $k_v$.  
By definition, there is a well-defined principally polarized abelian scheme $(A_Y,\la_Y)$ on $Y$, equipped with two crystalline tensors
\[j_Y^*(s_{0}): \mT\otimes\mathbf{1}\to \bD(A_Y)^\otimes, \quad \pr_Y^*(\iota_{0}): \mT\otimes\mathbf{1}\to \bD(A_Y)^\otimes.\]
Let $Y^\circ$ denote the union of connected components of $Y$ on which there is a $\bar k_v$-point $y$ such that $s_{0,j_Y(y)}=\iota_{0,\pr_Y(y)}$. Then by \cite[Lemma~5.10]{Keerthi}, 
\[j_Y^*(s_{0})|_{Y^\circ}=  \pr_Y^*(\iota_{0})|_{Y^\circ}.\]
In other words,  $Y^\circ(\bar k_v)$ is the set of  points $y \in Y(\bar k_v)$ such that $s_{0,j_Y(y)} = \iota_{0,\pr_Y(y)}$.

Note that by assumption, $Y^\circ\to \Spec R$ is bijective on $\bar k_v$-points.
On the other hand, recall that by definition, the map $\scrS_K\to \mathscr A_{g,K'}$ is finite, and induces surjective maps between completed local rings at every point. Then $Y^\circ\to \Spec R$ is a finite morphism of affine schemes of finite type over $k_v$, and induces bijection on closed points and closed immersion when completed at every points. In addition, $R$ is reduced. Therefore, it is an isomorphism.

Therefore, we obtain the desired map $\Spec R\cong Y^\circ \to \scrS_K$. It follows from the construction that $\iota_0$ is the pullback of $s_0$. The Chebotarev density implies that the pullback of $s_{\on{et}}$ is $\iota_{\on{et}}$.
\end{proof}

Now recall that  $X_{\mu^*}(b_x)$ classifies local $\underline G$-Shtukas $\mE'\dashrightarrow {^\sigma}\mE'$ of relative position $\preceq \mu$ and a modification $\beta:\mE'\dashrightarrow x^*\mE$ of local $\underline G$-Shtukas. Then 
\[\bD^\vee:= (\mE'\times^{\underline G}V_{\bZ_p}^\vee)^\vee\]
admits a structure of an $F$-crystal with $p\bD^\vee\subset {^\sigma}\bD^\vee\subset\bD^\vee$. 
Since the Dieudonn\'e functor is an equivalence of categories over perfect schemes, $\DD^\vee$ corresponds to a  $p$-divisible group $\mathscr G$ over $X_{\mu^*}(b_x)$ and the modification $\beta: \calE' \dashrightarrow x^*\calE$ corresponds to a quasi-isogeny
\begin{equation}
\label{E:isogeny of p-div group} \mathscr G \to
A_x[p^\infty] \times_\kappa X_{\mu^*}(b_x).
\end{equation}
From this, we get an abelian variety $A'$ over $X_{\mu^*}(b_x)$ together with a $p$-power quasi-isogeny 
\[A' \to
A_x \times_\kappa X_{\mu^*}(b_x),
\]
whose induced map on the $p$-divisible groups is \eqref{E:isogeny of p-div group}. It follows from the construction that the map
\[\mT\otimes \mathbf{1}\xrightarrow{s_{0,x}\otimes \id} (\bD(A_x)^\otimes\otimes \mathbf{1})[1/p]\cong \bD(A')^\otimes[1/p]\]
factors as $\mT\otimes \mathbf{1}\to \bD(A')^\otimes$, denoted by $s'_0$.

The abelian variety $A'$ comes equipped with a prime-to-$p$ principal polarization as well as a $K^p$-level structure, inherited from $A_x$.  Using the universal property of Siegel moduli space, we have a morphism
\[X_{\mu^*}(b_x) \longto \mathscr A_{g,{K'}^p}\otimes \bar \kappa.\]
\begin{lem}
\label{L: map of ADLV to Shimura}
The above map lifts to a unique morphism
\[
\iota_x: 
X_{\mu^*}(b_x)\longto \Sh_{\mu,K}\otimes \bar \kappa,\]
such that for each closed point $y \in X_{\mu^*}(b_x)$, under the identification $\DD(A'_{y}) \cong \DD(A_{\iota_x(y)})$, we have
\[
s'_{0, y} = s_{0, \iota_x(y)}.
\]
\end{lem}
\begin{proof}
The uniqueness will follow from \cite[Corollary 1.3.11]{Ki2}, as two maps between perfect schemes coincide if they coincide at the level of points. Therefore, to prove the existence, we can work locally around each point $y$ in $X_{\mu^*}(b_x)$. Let $U$ be an open neighborhood of $y$, which is assumed to be perfectly of finite presentation, and $U\to (\mathscr A_{g, K'^p}\otimes \bar \kappa)^\pf$ be the morphism. Then  we may assume that $U$ is the perfection of a reduced affine scheme $U_0$ of finite type over $\kappa$, and assume that the abelian scheme $A'$ on $U$ with the quasi-isogeny $A'\to A_x\times U$ is the pullback of a quasi-isogeny of abelian schemes $A_0\to A_x\times U_0$ on $U_0$. 
By \cite[Proposition 1.4.4]{Ki2}, the assumptions of Lemma \ref{L: from point to family} are satisfied so we can apply it to conclude.
\end{proof}

Note that the map from Lemma \ref{L: map of ADLV to Shimura} is compatible for changing the prime-to-$p$ level and for the tame Hecke algebra action so we can pass to the inverse limit we obtain a $G(\bA_f^p)$-equivariant morphism
\begin{equation}
\label{E:over uniformization}
\iota_x: X_{\mu^*}(b_x)\times G(\bA_f^p)\longto \Sh_{\mu,K_p}\otimes \bar \kappa.
\end{equation} 
Taking quotient by $K^p$ gives
\begin{equation}
\label{E:over uniformization2}
\iota_x: X_{\mu^*}(b_x)\times G(\bA_f^p)/K^p\longto \Sh_{\mu,K}\otimes \bar \kappa.
\end{equation}
It follows from the construction that the image of \eqref{E:over uniformization2} is the isogeny class $\mathscr I_x$ that $x$ belongs to.

\subsubsection{Rapoport-Zink uniformization}
Now let $x$ be a $\bar k_v$-point of $\Sh_{\mu,K_p}$. Then there is a canonical isomorphism $\on{H}^1(A_{\bar x},\bQ_\ell)\oplus \bQ_{\ell,\on{et}}(1)\simeq (V^\vee\oplus \bQ(1))\otimes\bQ_\ell$ for $\ell\neq p$, and therefore there is a canonical injection $I_x(\bQ)\subset G(\bA_f^p)$. In addition, a choice of the trivialization of $x^*\mE$ gives $b_x$, and an embedding $I_x(\bQ)\subset J_{b_x}(\bQ_p)$.
It is clear that the map \eqref{E:over uniformization2} induces
\begin{equation}
\label{E:RZ uniformization special}
I_x(\bQ)\backslash X_{\mu^*}(b_x)\times G(\bA_f^p)/K^p\to \Sh_{\mu,K}\otimes \bar k_v.\end{equation}
By \cite[Proposition 2.1.3]{Ki2}, this map is injective at the level of $\bar k_v$-points. Then by an argument as in \cite[Chap. 6]{RZ} (see also \cite{Kim}), $I_x(\bQ)\to J_{b_x}(\bQ_p)\times\prod G(\bA_f^p)$ is a discrete subgroup, and the quotient in \eqref{E:RZ uniformization special} on left hand side is representable (if $K^p$ is sufficiently small). In addition, it is independent of the choice of the trivialization of $x^*\mE$ up to a canonical isomorphism. We call this map the Rapoport-Zink uniformization of an isogeny class.

Recall that in general $I_x\otimes \bQ_p\subset J_{b_x}$, but the inclusion can be strict\footnote{We thank Rong Zhou to point this out to us.}. However,
\begin{lem}
\label{L: Ix for basic}
If $[b_x]\in B(G_{\bQ_p},\mu^*)$ is basic, then $I_{x}$ is an inner form of $G$. In particular, $I_x\otimes \bQ_p\cong J_{b_x}$ is an inner form of $G_{\bQ_p}$.
\end{lem}
If $(G,X)$ is of PEL type, this was contained in \cite[Theorem 6.30]{RZ}, which relies on Lemma 6.28 of \emph{loc. cit.}. Our proof, although looks different in disguise, is essentially the same.
\begin{proof}
We assume that $x$ is defined over $\kappa$ and $[\kappa:\bF_p]=r$. Recall we write $W_x=W(\kappa)$ and $K_x=W_x\otimes\bQ$. Set $\gamma_{x,p} = b_x\sigma_p(b_x) \cdots \sigma_p^{r-1}(b_x) \in G(K_x)$. We need to show that some power of $\gamma_{x,p}$ is in the center of $G$. It then follows that some power of $\ga_0$ from \S \ref{SS: the group I} is in the center of $G$. Therefore $I_0=G$ and $I_x$ is an inner form of $G$.

First, since $I_x$ is an inner form of a Levi subgroup of $G$, $Z_G$ canonically embeds into $I_x$ as a subgroup. Let $T$ denote its neutral connected component. 

Recall that we have a decomposition $\bD(A_x)\otimes K_x=\oplus_i N_i$, where $N_i$ is an isoclinic $F$-crystal, which is the union of $F$-stable $W_x$-lattices $M$ such that $F^{s_i}M=p^{r_i}M$. Replacing $s_i$ by a multiple, we may assume all $s_i$ are equal, denoted by $s$. The Newton cocharacter $\nu=\nu_{b_x}$ of $b_x$ is defined so that $(s\nu)(p)$ acts on $N_i$ as multiplication by $p^{r_i}$. 
Since $b_x$ is basic, $s\nu$ is a cocharacter of $T$. Then $(s\nu)(p^{-1})$ acts on $\bD(A_x)$, commutes with the action of $F$, and $(s\nu)(p^{-1})F^s$ preserves a lattice in $\bD(A_x)\otimes K_x$.

We regard $(s\nu)(p)\in T(\bA_f)$ whose $p$-component is $(s\nu)(p)$ and prime-to-$p$ component is $1$.  Since $T(\bQ)\backslash T(\bA_f)/H'$ is always finite for any open compact subgroup $H'$ of $T(\bA_f)$, we may replace $s$ by a multiple of it and find some $t\in T(\bQ)$ such that $t'=(s\nu)(p)t$ is contained in some open compact subgroup of $T(\bA_f)$.

Now we consider $g=\pi^s_{A_x}t^r\in I_x(\bQ)$. Note that $g$ is contained in some open compact subgroup of $H\subset I_x(\bA_f)$. Indeed, if $\ell\neq p$, $\pi_{A_x}$ is contained in some open compact subgroup of $I_x(\bQ_\ell)$ (as the Frobenius is an automorphism of the Tate module), therefore $g$ is also contained in some open compact subgroup $H_\ell\subset I_x(\bQ_\ell)$. 
On the other hand, after localized at $p$, $g=\ga_{x,p}^s (rs\nu)(p^{-1}) t'^r$ that acts on $\bD(A_x)\otimes K_x\simeq V^\vee_{\bZ_p}\otimes K_x$. As mentioned above, $\ga_{x,p}^s (rs\nu)(p^{-1})$ fixes a lattice. Therefore $g$ lies in some open compact $H_p\subset I_x(\bQ_p)$.  

Finally, note that $I_x(\bQ)\cap H\subset I_x(\bA_f)$  is a finite group since  $(I_x/\bG_m)(\bR)$ is compact. It follows that some power $g^{d}$ is the identity. Therefore, $\ga_{p,x}^{sd}=t^{-rd}$, which is central.
\end{proof}

\begin{cor}
Assume that $[b_x]\in B(G_{\bQ_p},\mu^*)$ is basic. Then there is a unique inner twist $\Psi_x: G\to I_x$, which is trivial at $\ell\neq p,\infty$, is given by $[(b_x)_{\ad}]\in B(G_{\ad,\bQ_p})_{\on{basic}}\cong H^1(\bQ_p, G_{\ad})$ at $p$, and such that $[(I_{\bR}, \Psi_\bR)]$ is the unique compact modulo center inner form of $G$ at $\infty$.
\end{cor}
\begin{proof}This is the combination of Lemma \ref{L: Ix for basic} and \cite[Corollary 2.3.5]{Ki2}.
\end{proof}

\begin{cor}
\label{C: RZ uniformization basic}
The basic locus $\Sh_{\mu,K_p,b}$ contains a unique isogeny class, and the map \eqref{E:RZ uniformization special} (for a $\bar k_v$-point $x \in \Sh_{\mu, K_p, b}$) factors through the basic locus, and gives the following isomorphism (over $\bar k_v$)
\[I_x(\bQ)\backslash X_{\mu^*}(b_x)\times G(\bA_f^p)/K^p\to \on{Sh}_{\mu,K,b}.\] 
\end{cor}
As explained in \cite{RZ} (and \cite{Kim} in the Hodge type case), the isomorphism is in fact defined over some finite field. We omit the general discussion here. The special case when $b$ is unramified will be essentially treated in Proposition \ref{exotic Hecke}.
\begin{proof}This follows from the previous corollary and \cite[Proposition 4.4.13]{Ki2}. Namely, since $I_x$ is an inner form of $G$, the group defined in \cite[(4.4.9)]{Ki2} is trivial. Then this map is bijective at the level of $\bar k_v$-points, and the proof of Lemma \ref{L: map of ADLV to Shimura} implies that it is an isomorphism.
\end{proof}

We will be particularly interested in a special case. Recall the notion of unramified elements in $B(G_{\bQ_p})$ as discussed in \S \ref{S: unram element}.

\begin{cor}
\label{L:Ix is G'}
If the basic element of $B(G_{\bQ_p},\mu^*)$ is unramified, then the inner twist $G\to I_x$ (for a $\bar k_v$-point $x \in \Sh_{\mu, K_p, b}$) is trivial at all finite places, and $I_{x,\bR}$ is the unique compact modulo center inner form of $G_\bR$.
\end{cor}

\begin{rmk}
Recall that by Corollary \ref{inner mu} and Proposition \ref{unramified basic}, there exists a unique inner form $G'$ of $G$ such that $G'_{\bR}$ is compact modulo center and $G'(\bA_f)\simeq G(\bA_f)$. The above lemma gives another proof of this result in the special case when $(G,X)$ is a Shimura datum of Hodge type. 
\end{rmk}

\subsection{Cohomological correspondences between Shimura varieties}
In this subsection, we apply Theorem \ref{T:Spectral action} to construct cohomological correspondences between Shimura varieties. This in particular induces a spectral action on the cohomology of Shimura varieties by the ring of regular functions of the stack of Satake parameters, which is the Shimura variety counterpart of V. Lafforgue's $S$-operators.

\subsubsection{The correspondence}
\label{SS: Corr Sh}

Let $(G_1,X_1)$ and $(G_2,X_2)$ be two Shimura data such that $G_{1,\bA_f}\simeq G_{2,\bA_f}$. 
As explained in the paragraph before Lemma \ref{JL:inner form}, $G_1$ and $G_2$ are inner forms.
Let $\{\mu_i\}$ denote the conjugacy class of Hodge cocharacters determined by $X_i$, regarded as a dominant character of $\hat T$. Let $E_i\subset \bC$ be the reflex field of $(G_i,X_i)$. 

As explained in Corollary \ref{C: position of two real}, the Shimura data defines a canonical inner twist $\Psi_\bR: G_1\to G_2$ over $\bC$, whose image under $\al_{G,\infty}$ is $\bar{\mu}_{1,\ad}-\bar{\mu}_{2,\ad}$. We assume 
$$\mu_{1,\ad}|_{Z(\hat G_\s)^{\Ga_{\bQ}}}=\mu_{2,\ad}|_{Z(\hat G_\s)^{\Ga_{\bQ}}}.$$ 
Then by Corollary \ref{inner mu}, this inner twist comes from a global inner twist $\Psi: G_1\to G_2$ over $\overline\bQ$, which is trivial at all finite places, i.e.
 \[\Psi=\on{Int}(h)\circ \theta\]
for some $\theta: G_{1,\bA_f}\simeq G_{2,\bA_f}$ and $h\in G_{2,\ad}(\overline\bA_f)$.  Recall that such global inner twist, if exists, is necessarily unique (see Corollary \ref{inner mu}).
Examples of such pairs $(G_i,X_i), i=1,2$ can be found in Example \ref{JL:ex}.

We fix $K_i\subset G_i(\bA_f)$ an open compact subgroup which is sufficiently small such that $\theta K_1=K_2$.
Let $p$ be a prime such that $K_{1,p}$ (and therefore $K_{2,p}$) is hyperspecial. 
Let $\underline G_i$ be the integral model of $G_{i,\bQ_p}$ over $\bZ_{p}$ determined by $K_{i,p}$. We have the isomorphism $\theta: \underline G_1\simeq \underline G_2$. 

Let $v$ be a place of $E=E_1E_2$ lying over $p$. 
Let $\Sh_{\mu_i,K_i}$ denote the mod $p$ fiber of the canonical integral model for $\Sh_{K_i}(G_i,X_i)$, base changed to $k_{v}$.

We need a hypothesis, which should always be true and can be proved unconditional in many cases (in particular the cases we are considering). The discussion in \S \ref{SS: coh corr Sh} will depend only on this hypothesis.
\begin{hypothesis}
\label{H: Hecke diagram}
\begin{enumerate}
\item 
The Shimura variety $\Sh_{K_i}(G_i, X_i)$ has a canonical integral model over $\mO_{E,(v)}$. Let $\Sh_{\mu_i,K_i}$ (or $\Sh_{\mu_i}$ for simplicity if the level structure is clear from the context) denote the perfection of the mod $p$ fiber of the model. Then there is a morphism $\loc_p:\Sh_{\mu_i,K_i}\to \Sht^\loc_{\mu_i}$ such that the composition 
$$\loc_p(m,n):=\res_{m,n}\circ \loc_p:\Sh_{\mu_i,K_i}\to \Sht^{\loc(m,n)}_{\mu_i}$$ 
is perfectly smooth for any pair $(m,n)$ such that $m-n$ is $\mu_i$-large.
\item 
Assume that 
\begin{equation}
\label{E: exotic p}
\mu_{1}|_{Z(\hat G)^{\Ga_{\bQ_p}}}=\mu_{2}|_{Z(\hat G)^{\Ga_{\bQ_p}}}.
\end{equation}
Then there exists a perfect ind-scheme $\Sh_{\mu_1|\mu_2}$ fitting into the following commutative diagram, with both squares Cartesian,
\begin{equation}
\label{E:global to local cartesian1}
\xymatrix{
\Sh_{\mu_1,K_1} \ar_{\loc_p}[d] & \ar_{\overleftarrow{h}_{\mu_1}}[l] \Sh_{\mu_1|\mu_2} \ar^{\overrightarrow{h}_{\mu_2}}[r] \ar[d] & \Sh_{\mu_2,K_2} \ar^{\loc_p}[d] \\
\Sht_{\mu_1}^\loc & \ar_{\overleftarrow{h}_{\mu_1}^\loc}[l] \Sht_{\mu_1|\mu_2}^{\loc} \ar^{\overrightarrow{h}_{\mu_2}^\loc}[r]& \Sht_{\mu_2}^\loc.
}
\end{equation}
\end{enumerate}
\end{hypothesis}
Note that Part (1) of the hypothesis holds for Shimura data of Hodge type, as discussed in the previous subsection. In addition, Part (1) also holds for weak Shimura data $(G,X)$ such that $G_\bR$ is compact modulo center. Namely, in this case $\Sh_K(G,X)=G(\bQ)\backslash G(\bA_f)/K$ is a set with an action of $\Gal(\overline\bQ/E)$ unramified at $p$, and the Shimura cocharacter $\mu\in\xcoch(Z_G)$ is central. Then the existence of the canonical integral model and the map $\loc_p$ is clear.

We concentrate on Part (2) of the hypothesis. It seems possible to construct the correspondence \eqref{E: pHecke corr} in a very general situation. However, to do so it seems best to work in the framework of Shimura varieties of abelian type.
To avoid some significant digression, in below we will establish it in some special cases, which suffices for our applications, and leave the general cases in a future work.
But let us first make a few remarks on this hypothesis. 
\begin{rmk}
\label{R: exotic Hecke}
(1) If $(G_1,X_1)=(G_2,X_2)=(G,X)$, $\Sh_{\mu\mid \mu}$ is the perfection of the mod $p$ fiber of a natural integral model of the Hecke correspondence
\begin{equation}
\label{E: pHecke corr}
\xymatrix@C=-20pt{
&G(\bQ)\backslash X\times G(\bA_f^p)/K^p\times G(\bQ_p)\times^{K_p}G(\bQ_p)/K_p\ar^{\overleftarrow{h}}[ld]\ar_{\overrightarrow{h}}[rd]&\\
G(\bQ)\backslash X\times G(\bA_f)/K && G(\bQ)\backslash X\times G(\bA_f)/K.}
\end{equation}
If $(G_1,X_1)\neq (G_2,X_2)$ but \eqref{E: exotic p} holds, then $\Sh_{\mu_1\mid \mu_2}$ can be regarded as ``exotic Hecke correspondences" between mod $p$ fibers of \emph{different} Shimura varieties. These correspondences cannot be lifted to characteristic zero, and give a large class of characteristic $p$ cycles on Shimura varieties.

(2) Condition \eqref{E: exotic p} is equivalent to 
\[\Hom_{\Coh^{\hat G}_{fr}(\hat G\sigma_p)}(\widetilde{V_{\mu_1}},\widetilde{V_{\mu_2}})\neq 0.\]
It is also the necessary and sufficient condition to ensure $\Sht^\loc_{\mu_1\mid \mu_2}\neq \emptyset$.
\end{rmk}

\subsubsection{Examples of exotic Hecke correspondences \eqref{E:global to local cartesian1}}
\label{SS:exotic Hecke G1 discrete}
We explain that \eqref{E:global to local cartesian1} exists in many cases.

If both $\mathbf{Sh}_{K_1}(G_1,X_1)$ and $\mathbf{Sh}_{K_2}(G_2,X_2)$ are Shimura sets, the top row of \eqref{E:global to local cartesian1} extends naturally to $\mO_{E, v}$. For this, we need to check the Shimura's reciprocity map for weak Shimura data. The geometric Frobenius $\phi_v$ acts on $\mathbf{Sh}_{K_{ip}}(G_i,X_i)$ by Hecke translation $p^{\mathrm{Nm}_{E_v/\QQ_p}(\mu_i)}$. The condition \eqref{E: exotic p} ensures that $\mathrm{Nm}_{E_v/\QQ_p}(\mu_1-\mu_2)$ is trivial on $Z(\hat G)$, which surjects onto the maximal abelian quotient of $\hat G$. This implies that $p^{\mathrm{Nm}_{E_v/\QQ_p}(\mu_1)} = p^{\mathrm{Nm}_{E_v/\QQ_p}(\mu_2)} $ as an element of $Z_G(\QQ_p)$. So \eqref{E:global to local cartesian1} exists in this case.

Next, we assume that $\mathbf{Sh}_{K_1}(G_1,X_1)=G_1(\bQ)\backslash G_1(\bA_f)/K_1$ is a discrete set and $\mathbf{Sh}_{K_2}(G_2,X_2)$ is of Hodge type. Then $G_1$ is compact modulo center at the infinite place and $\mu_1$ is central. Then \eqref{E: exotic p} exactly implies that the basic element in $B(G_{\bQ_p},\mu_2^*)$ can be represented by $p^{\mu^*_1}$. We may choose a $\bar k_v$-point $x\in\Sh_{\mu_2, K}$ and a trivialization of $x^*\mE$ such that $b_x=p^{\mu_1^*}$. 
Recall that in this case the group $I_x$ is equipped with an inner twist $\Psi: G\to I_x$, which is trivial at all finite places, and is compact modulo center at the infinite place. In particular, there is an isomorphism $G_1\simeq I_x$, compatible with the inner twist. We fix such inner twist.
We may reinterpret the Rapoport-Zink uniformization as follows.

\begin{prop}\label{exotic Hecke}
In the following diagram, both squares are Cartesian.
\begin{equation}
\label{E:exotic Hecke discrete to Hodge}
\xymatrix{
G_1(\QQ) \backslash G_1(\AAA_f) /K \ar[d] & \ar[l]_-{\overleftarrow{h}} \Sh_{\mu_1|\mu_2}: =  G_1(\bQ)\backslash G(\bQ_p)\times^{G(\bZ_p)} X_{\mu_2^*}(p^{\mu_1^*})\times G(\bA_f^p)/K^p \ar[r]^-{\overrightarrow{h}} \ \ar[d] &\Sh_{\mu_2,K}\ar[d]^-{\loc_p} \\
\Sht_{\mu_1}^\loc & \ar_-{\overleftarrow{h}^\loc}[l]\Sht_{\mu_1\mid\mu_2}^{\loc} \ar^-{\overrightarrow{h}^\loc}[r] & \Sht_{\mu_2}^\loc.
}\end{equation}
Moreover, there is a $k_v$-scheme structure on $G_1(\bQ)\backslash G(\bQ_p)\times^{G(\bZ_p)} X_{\mu_2^*}(p^{\mu_1^*})\times G(\bA_f^p)/K^p$ so that the top row of \eqref{E:exotic Hecke discrete to Hodge} is defined over $k_v$.
\end{prop}
\begin{proof}
Note that by \eqref{E: fiber of hleft}, $(\overrightarrow{h}^\loc)^{-1}(\loc_p(x))=X_{\mu_1^*}(b_x)=X_{\mu_1^*}(p^{\mu_1^*})=G(\bQ_p)/G(\bZ_p)$. Then by Corollary \ref{C: RZ uniformization basic},
\begin{align*}
\Sht_{\mu_1|\mu_2}^{\loc}\times_{\Sht_{\mu_2}^{\loc}}\Sh_{\mu_2,K} &=\Sht_{\mu_1|\mu_2}^{\loc}\times_{\Sht_{\mu_2}^{\loc}}\Sh_{\mu_2,K,b} \\
& = 
\Sht_{\mu_1|\mu_2}^{\loc}\times_{\Sht_{\mu_2}^{\loc}}\big( I_x(\QQ) \backslash X_{\mu_2^*}(b_x)\times G(\bA_f^p)/K^p\big)
\\
& \cong G_1(\bQ)\backslash G(\bQ_p)\times^{G(\bZ_p)} X_{\mu_2^*}(p^{\mu_1^*})\times G(\bA_f^p)/K^p = \Sh_{\mu_1|\mu_2}.
\end{align*}
Therefore, we have the right Cartesian square.  The left square is Cartesian because the left square of \eqref{rigid corr} is Cartesian.

Now we show that the top row of \eqref{E:exotic Hecke discrete to Hodge} is defined over $k_v$ in an appropriate sense.
We go into the proof of Corollary~\ref{C: RZ uniformization basic} when $b_x = p^{\mu_1^*}$.
Recall that the point $x \in \Sh_{\mu_2, K_p}$ gives rise to a tuple $(A_x,\eta_x, s_{\et, x}, s_{0,x})$ such that $b_x = p^{\mu_1^*}$ under some trivialization of $x^*\calE$. A point $y = (g_p, y_0, g^p)$ in the space $\Sh_{\mu_1|\mu_2}$ gives rise to a modification $\beta_y: \calE'_y \xrightarrow{\beta_{y_0}} x^*\calE \xrightarrow{g_p}x^*\calE$ of local $\underline G$-shtukas, which further gives rise to $p$-power quasi-isogenies $\beta_y: A'_y \to A_x$, and a tame level structure $\eta'_y$ on $A'_y$ given by the image of the section $\eta_x g^p$ under the isomorphism
\begin{align*}
\on{Isom}\big((V\otimes \AAA_f^p \oplus \AAA_f^p,& s \otimes 1), (\calV_\et(A)^p \oplus \AAA_f^p(1), s_{\et,x})\big) / K^p
\\ &\xrightarrow{\beta_y^{-1}}
\on{Isom}\big((V\otimes \AAA_f^p \oplus \AAA_f^p(1),s \otimes 1), (\calV_\et(A'_y)^p \oplus \AAA_f^p, \beta_y^{-1}(s_{\et,x}))\big)  / K^p.
\end{align*}
Then the image of this point $y$ under $\overrightarrow{h}$ in $\Sh_{\mu_2,K}$ is represented by $(A'_y, \eta'_y, \beta_y^{-1}(s_{\et, x}), \beta_y^{-1}(s_{0,x}))$.

We define an action of Frobenius $\sigma_v$ on $ \Sh_{\mu_1|\mu_2}$ by sending such a point $(g_p, y_0, g^p)$ to
\[
\big(p^{\NN_v(\mu_1^*)}g_p, \sigma_v(x), g^p \big), \qquad \textrm{for\quad}\NN_v(\mu_1^*): = \sum_{i = 1}^{[k_v:\FF_p]} \sigma^i(\mu_1^*).
\]
This is compatible via $\overleftarrow{h}$ with the Frobenius $\sigma_v$ action on $\Sh_{\mu_1}= G_1(\QQ) \backslash G_1(\AAA_f)/K$ determined by the Shimura reciprocity law.
On the other hand, the image of $z: = (p^{\NN_v(\mu_1^*)}g_p, \phi_v(x), g^p)$ in $\Sh_{\mu_2,K}$ under $\overrightarrow{h}$ is represented by the abelian variety $ A'_z$, equipped with the $p$-power isogeny $\beta_{z}: A'_z \to A_x$ corresponding to the modification
\begin{equation}
\label{E:Shimura reciprocity on A}
\beta_{z}:
(1\times \sigma_v)^*(\calE'_y)  \xrightarrow{\sigma_v(\beta_y)}
(1\times \sigma_v)^*(x^*\calE) \xrightarrow{p^{\NN_v(\mu_1^*)}} x^*\calE
\end{equation}
of local $\underline G$-shtukas. In particular, we have a natural isomorphism $A'_{z} \cong\sigma_v^*(A'_y)  $. The last map in \eqref{E:Shimura reciprocity on A} is given by multiplication by $p^{\NN_v(\mu_1^*)}$ because the natural map ${}^\sigma \calE\to \calE$ is the multiplication by $p^{\mu_1^*}$ (under the chosen trivialization).
The tame level structure $\eta'_{z}$ on $A'_{z}$ is given by the image of $\eta_x g^p$ under the isomorphism induced by $\beta_{z}^{-1}$, which is the same as the image under the following sequence of isomorphisms
\[
\xymatrix{
\eta_x g^p \ar@{}[r]|-\in \ar@{|->}[d] & \on{Isom}\big((V\otimes \AAA_f^p \oplus \AAA_f^p,s \otimes 1), (\calV_\et(A)^p \oplus \AAA_f^p, s_{\et,x})\big) / K^p
\ar[d]_{\sigma_v}
\\ \sigma_v(\eta_x g^p) \ar@{}[r]|-\in \ar@{|->}[d] &
\on{Isom}\big((V\otimes \AAA_f^p \oplus \AAA_f^p,s \otimes 1), (\calV_\et(\sigma_v^*(A_x))^p \oplus \AAA_f^p, \sigma_v(s_{\et,x}))\big)  / K^p \ar[d]^{\sigma_v(\beta_y^{-1})}\!\!\!\!\!\!\!\!\!\!\!\!
\\ \sigma_v(\beta_y^{-1})(\sigma_v(\eta_x g^p)) \ar@{}[r]|-\in &
\on{Isom}\big((V\otimes \AAA_f^p \oplus \AAA_f^p,s \otimes 1), (\calV_\et(\sigma_v^*(A'_y))^p \oplus \AAA_f^p, \sigma_v(\beta_y^{-1}(s_{\et,x})))\big)  / K^p. \!\!\!\!\!\!\!\!\!\!\!\!\!\!
}\quad
\]
This exactly says that
$$\overrightarrow{h}(p^{\NN_v(\mu_1^*)}g_p, \phi_v(x), g^p) = \sigma_v\big(\overrightarrow{h}(g_p, \phi_v(x), g^p)\big)$$
for the natural action $\sigma_v$ on $\Sh_{\mu_2,K}$. This means that the top row of \eqref{E:exotic Hecke discrete to Hodge} can be defined over $k_v$, using the given Frobenius structure.
\end{proof}

Finally, we assume that both $(G_1,X_1)$ and $(G_2,X_2)$ are Shimura data of Hodge type. We discuss the case $(G_1,X_1)=(G_2,X_2)=(G,X)$.
\begin{prop}
\label{P:Shmu mu}
There is an ind-scheme $\Sh_{\mu\mid\mu}$ and morphisms $\overleftarrow{h},\overrightarrow{h}:\Sh_{\mu\mid\mu}\to \Sh_\mu$ that fit into the diagram \eqref{E:global to local cartesian1} with both squares Cartesian.
\end{prop}
\begin{proof}
Let $\overleftarrow{\Sh}_{\mu\mid\mu}:=\Sh_\mu\times_{\Sh^\loc_\mu,\overleftarrow{h}^\loc}\Sh^\loc_{\mu\mid\mu}$. The natural projection $\overleftarrow{\Sh}_{\mu\mid\mu}\to \Sh_\mu$ is denoted by $\overleftarrow{h}$. Similarly we have $\overrightarrow{h}:\overrightarrow{\Sh}_{\mu\mid\mu}: = \Sh^\loc_{\mu\mid\mu} \times_{\overrightarrow{h}^\loc, \Sh_\mu^\loc} \Sh_\mu \to\Sh_\mu$.  We claim that there is another map 
$$\overrightarrow{h}: \overleftarrow{\Sh}_{\mu\mid\mu}\to \Sh_\mu$$ that induces a map $\overleftarrow{\Sh}_{\mu\mid\mu}\to\overrightarrow{\Sh}_{\mu\mid\mu}$. Indeed, $\overleftarrow{\Sh}_{\mu\mid\mu}$ classifies a point $x$ of $\Sh_\mu$ and a modification of local $\underline G$-Shtukas $\mE'\dashrightarrow x^*\mE$. As argued in the paragraphs before Lemma \ref{L: map of ADLV to Shimura}, this provides an abelian scheme $A'$ over $\overleftarrow{\Sh}_{\mu\mid\mu}$ with a $p$-power quasi-isogeny $\overleftarrow{h}^*A_x\to A'$ such that the crystalline tensors $s_{0,x}:\mT\otimes\mathbf{1}\to \bD(A_x)^\otimes$ induces $s'_0:\mT\otimes\mathbf{1}\to \bD(A')^\otimes$. In addition, $A'$ comes equipped with a prime-to-$p$ principal polarization as well as a $K^p$-level structure, inherited from $A_x$. 
Then $A'$ defines a morphism $\overrightarrow{h}':\overleftarrow{\Sh}_{\mu\mid\mu}\to \mathscr A_{g,K'^p}$. Note that if $y$ is a closed point $\Sh_\mu$, the fiber $\overleftarrow{h}^{-1}(y)\cong X_{\mu^*}(b_y)$. Therefore, pointwise, $\overrightarrow{h}'$ lifts to $\Sh_\mu$ matching the $s'_0$ and $s_0$ by \cite[Proposition 1.4.4]{Ki2}. Now by the same argument as Lemma \ref{L: map of ADLV to Shimura} (using Lemma \ref{L: from point to family}), we see $\overrightarrow{h}'$ lifts to $\overrightarrow{h}: \overleftarrow{\Sh}_{\mu\mid\mu}\to \Sh_\mu$ which induces $\overleftarrow{\Sh}_{\mu\mid\mu}\to \overrightarrow{\Sh}_{\mu\mid\mu}$.

By the same argument, we obtain $\overleftarrow{h}:\overrightarrow{\Sh}_{\mu\mid\mu}\to\Sh_\mu$ that induces $\overrightarrow{\Sh}_{\mu\mid\mu}\to \overleftarrow{\Sh}_{\mu\mid\mu}$. Clearly, these two maps are inverse to each other. The proposition follows.
\end{proof}

\begin{rmk}
For our purpose, the existence of such a self-correspondence between the mod $p$ fiber of the Shimura variety is enough. But it is not difficult to show that $\Sh_{\mu\mid\mu}$ is the perfection of the mod $p$ fiber of a natural integral model of the $p$-power Hecke correspondence \eqref{E:global to local cartesian1}. This will be discussed in details in another occasion.
\end{rmk}

\subsubsection{Exotic Hecke correspondence for PEL type Shimura varieties}
\label{SS:exotic Hecke PEL}
There is another case where \eqref{E:global to local cartesian1} can be constructed relatively easily.

Let $(B,*, \calO_B, V_i, (-,-)_i, \Lambda_i )$ (for $i =1,2$) be two integral PEL type Shimura data as in 
\cite{Kopoints} that are unramified at $p$. Assume that $V_1(\AAA_f) \simeq V_2(\AAA_f)$ as skew-Hermitian $B$-modules.
Let $G_i$ denote the group of automorphisms of the skew-Hermitian $B$-modules $V_i$ (with similitudes), which is assumed to be connected (in particular the type (D) case is excluded).
Assume moreover that the group $\ker^1(\QQ, G_i)$ is trivial so that the Shimura variety $\scrS_K(G_i, X_i)$ is exactly the moduli space of tuples $(A, \lambda, i,  \eta)$ satisfying the Kottwitz condition as in \cite[\S 5]{Kopoints}.

\begin{prop}
\label{P:exotic Hecke PEL}
The correspondence \eqref{E:global to local cartesian1} exists between $\Sh_{\mu_1, K}$ and $\Sh_{\mu_2, K}$.
\end{prop}
\begin{proof}
The proof is similarly that of Proposition~\ref{P:Shmu mu}, except using the PEL type moduli interpretation of Shimura varieties $\Sh_{\mu_1, K}$ and $\Sh_{\mu_2,K}$. Namely, a point of $\Sh_{\mu_1,K}\times_{\Sht^\loc_{\mu_1}}\Sht^\loc_{\mu_1\mid \mu_2}$ corresponds to $(A,\lambda,i, \eta)$ and a modification  $\mE'\dashrightarrow x^*\mE$ of the local Shtuka, where $x^*\mE$ arises from $(A[p^\infty],\la,i)$ (via the Dieudonn\'e theory). As in Proposition~\ref{P:Shmu mu}, this modification allows one to produce another abelian scheme $A'$ equipped with a polarization $\la'$, an endomorphism by $\mO_B$ and a level structure $\eta'$ inherited from $A$. Since the singularity of $\mE'$ is of type $\mu_2$,tThe Kottwitz condition on $A'$ is the one imposed for the Shimura variety $\Sh_{\mu_2}$.
This defines a map $\Sh_{\mu_1,K}\times_{\Sht^\loc_{\mu_1}}\Sht^\loc_{\mu_1\mid \mu_2}\to \Sh_{\mu_2}$. One can define the map in the other direction by the same argument. Then as in the proof of Proposition~\ref{P:Shmu mu}, we obtain  \eqref{E:global to local cartesian1}.
We leave the details  as an exercise for readers. 
\end{proof}

\begin{remark}
In the proposition above, it is possible that $\Sh_{\mu_1|\mu_2}$ is empty (when $\Sh_{\mu_1|\mu_2}^\loc$ is). For example, if $B=E$ is a CM field with $*$ being the complex conjugation, $V_i$ is a Hermitian space with signature $(p_{i, \tau}, q_{i, \tau})$ at $\tau : F=E^{*=1} \to \RR$ (where $p_{i,\tau}+q_{i,\tau}=n$). Then at an inert prime $p$ of $E$, $\Sh_{\mu_1|\mu_2}^\loc$ is non-empty if and only if $(-1)^{\sum_{\tau}p_{1,\tau}q_{1, \tau}} = (-1)^{\sum_{\tau}p_{2,\tau}q_{2, \tau}}$, which is equivalent to \eqref{E: exotic p}. (Also compare with Example \ref{JL:ex} (ii)).
\end{remark}

\subsubsection{The cohomological correspondence}
\label{SS: coh corr Sh}

Let $(G_i,X_i), i=1,2,3$ be three Shimura data, equipped with isomorphisms $G_{1,\bA_f}\simeq G_{2,\bA_f}\simeq G_{3,\bA_f}$, and a level structure $K$. Let $p$ be an unramified prime, such that \eqref{E: exotic p} holds for each pair of them. 
We assume Hypothesis \ref{H: Hecke diagram} holds for each pair of them. Let $\mH^p$ be the prime-to-$p$ Hecke algebra.
We also fix a square root $\sqrt{p}$ in $\Ql$, i.e. a half Tate twist $\Ql(1/2)$. To simplify notations, for a smooth variety $X$ of pure dimension $d$ over $\bar k_v$, we write $\langle d\rangle=[d](d/2)$ as usual, and
$$\on{H}_c(X)=\on{H}_c^*(X, \Ql\langle d \rangle).$$

Recall that for every representation $V$ of $\hat G$, there is a vector bundle $\widetilde V$ on the stack $\hat G\sigma/\hat G$ of unramified Langlands parameters. Let $d_i=\langle2\rho,\mu_i\rangle=\dim \Sh_K(G_i,X_i)$.
Write $V_{i}=V_{\mu_i}$. Let $S(\widetilde{V_{i}})\in \on{P}^{\on{Corr}}(\Sht^{\loc})$, which can be represented as a perverse sheaf $S(\widetilde{V_i})^{\loc(m,n)}$ on $\Sht_{\mu_i}^{\loc(m,n)}$ for some pair non-negative integers $(m,n)$ with $m$ being positive and $m-n$ being $\mu_i$-large. Note that since $\mu_i$ is minuscule, the pullback $\loc_p(m,n)^\star S(\widetilde{V_{i}})$ to $\Sh_{\mu_i}$ is canonically isomorphic to $\Ql\langle d_i\rangle$.

We first construct a map
\begin{equation}
\label{E: exotic Hecke on coh}
\Hom_{\on{P}^{\on{Corr}}(\Sht^\loc)}\big(S(\widetilde{V_{1}}),S(\widetilde{V_{2}}) \big)\to \Hom_{\mH^p}\big(\on{H}_c(\Sh_{\mu_1, \bar k_v}), \on{H}_c(\Sh_{\mu_2, \bar k_v})\big)
\end{equation}
as follows. Choose a dominant coweight $\nu$ and a quadruple $(m_1, n_1, m_2, n_2)$ that is $(\mu_1+\nu,\nu)$-acceptable and $(\mu_2+\nu,\nu)$-acceptable. Then we have
\begin{equation}
\label{E:global to local cartesian2}
\xymatrix{
\Sh_{\mu_1} \ar[d] & \ar[l] \Sh^\nu_{\mu_1|\mu_2} \ar[r] \ar[d]_{\loc_p^\nu} & \Sh_{\mu_2} \ar[d] \\
\Sht_{\mu_1}^\loc\ar[d] & \ar[l]\ar[d] \Sht_{\mu_1|\mu_2}^{\nu,\loc} \ar[r]\ar[d] \ar@{}[dr]|{\mathrm X}  & \Sht_{\mu_2}^\loc\ar[d]\\
\Sht_{\mu_1}^{\loc(m_1,n_1)}  &\ar[l] \Sht_{\mu_1|\mu_2}^{\nu,\loc(m_1,n_1)} \ar[r]& \Sht_{\mu_2}^{\loc(m_2,n_2)}
}
\end{equation}
Note that all squares except the one marked with $``\mathrm{X}"$ are Cartesian: that the upper two squares being Cartesian follows from Hypothesis \ref{H: Hecke diagram}(2) and that the lower left square is Cartesian follows from Lemma \ref{corr trun2}.
The map $\loc_p(m_i,n_i): \Sh_{\mu_i}\to \Sht^{\loc(m_i,n_i)}_{\mu_i}$ are perfectly smooth by Proposition \ref{smoothness}. Therefore, the composition of the vertical maps in the middle column, denoted by $\loc^\nu_p(m_1,n_1)$, is also perfectly smooth.
Also recall that the map $\Sh_{\mu_1\mid \mu_2}^\nu\to \Sh_{\mu_1}$ is perfectly proper.

By the formalism of pullback cohomological correspondences along perfectly smooth morphisms (\S \ref{ASS:smooth pullback correspondence}), and proper pushforward of cohomological correspondences (\S \ref{ASS:pushforward correspondence}), we thus obtain a map
\begin{align}
\label{E: exotic Hecke on coh constr}
&\ \on{Corr}_{\Sht_{\mu_1\mid\mu_2}^{\nu,\loc(m_1,n_1)}}\big(S(\widetilde{V_1})^{\loc(m_1,n_1)},  S(\widetilde{V_2})^{\loc(m_2,n_2)}\big)\\
\nonumber\xrightarrow{\loc^\nu_p(m_1,n_1)^\star}&\ \on{Corr}_{\Sh_{\mu_1\mid\mu_2}^{\nu}}\big((\Sh_{\mu_1},\Ql\langle d_1\rangle),(\Sh_{\mu_2}, \Ql\langle d_2\rangle)\big)  \\
\nonumber\stackrel{\on{H}_c}{\longto} & \ \Hom_{\mH^p}\big(\on{H}_c(\Sh_{\mu_1, \bar k_v}), \on{H}_c(\Sh_{\mu_2, \bar k_v})\big).
\end{align}
Note that  it follows from the diagram \eqref{E: change of m,n in Hk corr} that $\loc^\nu_p(m_1,n_1)^\star$ is compatible with changing $(m_1,n_1,m_2,n_2)$ in a natural way. If $\nu\preceq \nu'$ and if $(m_1, n_1, m_2, n_2)$ is $(\mu_1+\nu', \nu')$-acceptable and is $(\mu_2+\nu', \nu')$-acceptable, then proper smooth base change implies the commutativity of the following diagram
\begin{small}
\[\xymatrix{
\on{Corr}_{\Sht_{\mu_1\mid\mu_2}^{\nu,\loc(m_1,n_1)}}\big(S(\widetilde{V_1})^{\loc(m_1,n_1)},  S(\widetilde{V_2})^{\loc(m_2,n_2)}\big)\ar[r] \ar[d] & \on{Corr}_{\Sht_{\mu_1\mid\mu_2}^{\nu',\loc(m_1,n_1)}}\big(S(\widetilde{V_1})^{\loc(m_1,n_1)},  S(\widetilde{V_2})^{\loc(m_2,n_2)}\big) \ar[d]\\
\on{Corr}_{\Sh_{\mu_1\mid\mu_2}^{\nu}}\big((\Sh_{\mu_1},\Ql\langle d_1\rangle),(\Sh_{\mu_2}, ,\Ql\langle d_2\rangle)\big)\ar[r] & \on{Corr}_{\Sh_{\mu_1\mid\mu_2}^{\nu'}}\big((\Sh_{\mu_1},\Ql\langle d_1\rangle),(\Sh_{\mu_2},\Ql\langle d_2\rangle)\big).
}\]
\end{small}Therefore, when composed with $\on{H}_c$,  by  Lemma \ref{L: comp of push coh corr}, the composition map \eqref{E: exotic Hecke on coh constr} factors through
$\Hom_{\on{P}^{\on{Corr}}(\Sht^\loc)}(S(\widetilde{V_{1}}),S(\widetilde{V_{2}}))$ and gives the desired \eqref{E: exotic Hecke on coh}.

Now assume we have three Shimura data of Hodge type $(G_i,X_i)$, and we assume that Hypothesis \ref{H: Hecke diagram} holds for each pair out of three.
\begin{lem}
The following diagram is commutative.
\begin{footnotesize}
\[\xymatrix{
\Hom_{\on{P}^{\on{Corr}}(\Sht^\loc)}
\big(S(\widetilde{V_{1}}),S(\widetilde{V_{2}}) \big)\otimes \Hom_{\on{P}^{\on{Corr}}(\Sht^\loc)}\big(S(\widetilde{V_{2}}),S(\widetilde{V_{3}}) \big) \ar[r]  \ar[d]& \Hom_{\on{P}^{\on{Corr}}(\Sht^\loc)}\big(S(\widetilde{V_{1}}),S(\widetilde{V_{3}})\big)\ar[d]
\\
\Hom_{\mH^p}\big(\on{H}_c^{*}(\Sh_{\mu_1, \bar k_v}), \on{H}_c^{*}(\Sh_{\mu_2, \bar k_v})\big)\otimes \Hom_{\mH^p}\big(\on{H}_c^{*}(\Sh_{\mu_2, \bar k_v}), \on{H}_c^{*}(\Sh_{\mu_3, \bar k_v})\big)\ar[r] & \Hom_{\mH^p}\big(\on{H}_c^{*}(\Sh_{\mu_1, \bar k_v}), \on{H}_c^{*}(\Sh_{\mu_3, \bar k_v})\big).
}\]
\end{footnotesize}
\end{lem}
\begin{proof}
First, it is clear that from the hypothesis,
there exists $\on{Comp}: \Sh_{\mu_1\mid \mu_2}^{\nu}\times_{\Sh_{\mu_2}}\Sh_{\mu_2\mid\mu_3}^{\nu'}\to \Sh_{\mu_1\mid\mu_3}^{\nu+\nu'}$ making
the following diagram commutative.
\begin{equation}
\label{E:projectionComp Sh to Sht}
\xymatrix@C=20pt{
&&\Sh_{\mu_1|\mu_3}^{\nu+\nu'} 
\ar[ddd]\ar[dll]\ar[dr]
\\
\Sh_{\mu_1} \ar[d]^{\loc_p} &\ar[l]\ar[d] \Sh_{\mu_1\mid\mu_2}^{\nu} \times_{\Sh_{\mu_2}} \Sh_{\mu_2\mid\mu_3}^{\nu'} \ar[l] \ar[rr] \ar[ur]_-{\mathrm{Comp}}&& \Sh_{\mu_3} \ar[d]_{\loc_p}
\\
\Sht_{\mu_1}^{\loc}
&\ar[l]\Sht_{\mu_1|\mu_2}^{\nu,\loc} \times_{\Sht_{\mu_2}^{\loc}} \Sht_{\mu_2|\mu_3}^{\nu',\loc} \ar[rd]^-{\mathrm{Comp}^{\loc}} \ar[rr]
&&\Sht_{\mu_3}^{\loc}
\\
&& \Sht_{\mu_1|\mu_3}^{\nu+\nu', \loc}\ar[ull]\ar[ru].
}
\end{equation}
In addition, the middle trapezoid is a Cartesian.

Combining the above diagram with the diagram in Proposition \ref{P:pushforward is local}, we obtain a diagram similar to \eqref{E:projectionComp Sh to Sht}, with the bottom two rows replaced by moduli of restricted shtukas. Then lemma follows from the compatibility of pullback and pushforward of cohomological correspondence as in 
Lemma \ref{AL:pushforward pullback compatibility}, and the definition of compositions of morphisms in $\on{P}^{\on{Corr}}(\Sht^\loc)$ and Lemma \ref{AL:pushforward compatible with composition}.
\end{proof}

Now combining the above considerations and Theorem \ref{T:Spectral action}, we see that there is a canonical map
\begin{equation}
\label{E: spetral action1}
\on{Spc}:\Hom_{\Coh^{\hat G}_{fr}(\hat G\sigma_p)}(\widetilde{V_i},\widetilde{V_j})\to \Hom_{\mH^p}\big(\on{H}_c(\Sh_{\mu_i,\bar k_v}),\on{H}_c(\Sh_{\mu_j,\bar k_v})\big)
\end{equation}
compatible with compositions. In particular, the ring 
\[\bfJ:= \Gamma([\hat G\sigma_p/\hat G],\mO)=\Ql[\hat G]^{c_\sigma(\hat G)}\]
of regular functions on the stack acts on every $\on{H}_c(\Sh_{\mu_i,\bar k_v})$, and the above map is upgraded to a $\bfJ$-module homomorphism
\[\on{Spc}: \Hom_{\Coh^{\hat G}_{fr}(\hat G\sigma_p)}(\widetilde{V_i},\widetilde{V_j})\to \Hom_{\mH^p\otimes \bfJ}\big(\on{H}_c(\Sh_{\mu_i,\bar k_v}),\on{H}_c(\Sh_{\mu_j,\bar k_v})\big),\]
compatible with compositions.

This action is the Shimura variety counterpart of V. Lafforgue's $S$-operator. Recall that the classical Satake isomorphism is a canonical isomorphism
\[\bfJ\cong \bfJ_{\hat T}\cong \Ql[\xcoch(T)^\sigma]^{W_0}\cong C_c(G(\bZ_p)\backslash G(\bQ_p)/G(\bZ_p)).\]
On the other hand, 
the Hecke algebra $C_c(G(\bZ_p)\backslash G(\bQ_p)/G(\bZ_p))$ naturally acts on $\on{H}_c(\Sh_{\mu,\bar k_v})$: the existence of smooth toroidal compactification (\cite{Lan13} \cite{Keerthi15}) implies that after choosing a specialization map $\bar k_v\to \bar E_v$, there is an isomorphism
\[\on{H}_c(\Sh_{\mu,\bar k_v})\simeq \on{H}_c(\Sh_K(G,X)_{\bar E_v}),\] 
and since the action of the Hecke algebra $C_c(G(\bZ_p)\backslash G(\bQ_p)/G(\bZ_p))$ on the right hand side commutes with the action of the Galois group, such action translates to an action of the left hand side.

\begin{conjecture}
\label{Conj: S=T}
The action of $\bfJ$ on $\on{H}_c(\Sh_{\mu,\bar k_v})$ constructed above coincides with the usual Hecke algebra action, via the Satake isomorphism.
\end{conjecture}

This conjecture would be the analogue of V. Lafforgue's $S=T$ theorem in the Shimura variety setting. In this direction, we have
\begin{prop}
\label{P: S=T for Sh set}
If $\Sh_K(G,X)$ is a Shimura set, then Conjecture \ref{Conj: S=T} holds.
\end{prop}
\begin{proof}
Let $f\in \bfJ$, given by the restriction of a character $\chi_V$ of a representation of ${^L}G$ to $\hat G\sigma$. As discussed in Example \ref{Ex:examples of correspondences}(5), in this case the cohomological correspondence $\on{Spc}(f)$ is given by a function on 
$$\Sh_{\mu\mid\mu}= G(\bQ)\backslash( G(\bA^p_f)/K^p\times G(\bQ_p)\times^{G(\bZ_p)}G(\bQ_p)/G(\bZ_p)),$$ 
which in turn is the pullback of a function on $\Sht^\loc_{\mu\mid\mu}= G(\bZ_p)\backslash G(\bQ_p)/G(\bZ_p)$ by construction. But from Proposition \ref{P: local S=T}, this function is 
$$\Sat^{cl'}([V])=\Sat^{cl}([V^*]).$$ 
The proposition follows.
\end{proof}

\subsection{Cycle classes of the basic Newton stratum}
\label{}
In this subsection, we  combine all the previous results to compute the determinant of the intersection matrix of cycles coming from the irreducible components of the basic locus of the special fiber of the Shimura variety, and prove the main Theorem~\ref{T:main theorem}.

We keep the notation as in the previous sections.
In addition, we  assume that $V_{\mu^*}^{\Tate_p} \neq 0$. Recall from Proposition \ref{unramified basic}, the basic element $b\in B(G_{\bQ_p},\mu^*)$ is unramified. Let $\la_b\in\xcoch(T)_{\sigma_p}$ be the element associated to $b$ via Lemma \ref{L: unramified elements parameterization}.
Let 
$$\MV_{\mu^*}^{\Tate_p}:= \bigsqcup_{\la, \la_{\sigma_p}=\la_b} \MV_{\mu^*}(\la).$$
By Corollary \ref{inner mu}  and Proposition \ref{unramified basic}, there is the inner form $G\to G'$, which is trivial at all finite places and such that $G'_\bR$ is compact modulo center.

\subsubsection{Irreducible components of basic locus}
Let $\Sh_{\mu,b}\subset \Sh_\mu$ denote the basic Newton stratum.

Let $x$ be a $\bar k_v$-point in the basic locus of $\Sh_{\mu,K_p}$ so we obtain $\loc_p(x)\in\Sht^\loc_{\mu}(\bar k_v)= A(G_{\bQ_p},\mu^*)$ as from \S \ref{SS: NS for Shimura}. Let $b_x\in G(L)$ be a representative of $\loc_p(x)$.
Recall that we have a canonical decomposition
\[X_{\mu^*}(b_x)= \bigcup_{\bbb\in \MV_{\mu^*}^\Tate} X_{\mu^*}^\bbb(b_x),\]
and $J_{b_x}(\bQ_p)\cong G(\bQ_p)$ acts on each $X_{\mu^*}^\bbb(b_x)$ (Theorem \ref{C: irr comp up to J}). 
Therefore, by the Rapoport-Zink uniformization (Corollary \ref{C: RZ uniformization basic}), we can decompose the basic Newton stratum $\Sh_{\mu,b}$ as the union of
\[\Sh_{\mu,b}^\bbb:=  I_x(\bQ)\backslash X_{\mu^*}^\bbb(b_x)\times G(\bA_f^p)/K^p,\]
for $\bbb\in \MV_{\mu^*}^\Tate$.
Note that this decomposition is independent of the choice of $x$ and the representative $b_x$.

Applying Theorem \ref{C: irr comp up to J}, we obtain
\begin{prop}
Every irreducible component of $\Sh_{\mu,K,b}\otimes \bar k_v$ is of dimension $d/2$. There is a bijection between the set of irreducible components of $\Sh_{\mu,K,b}\otimes \bar k_v$ and $G'(\bQ)\backslash G'(\bA_f)/K\times \MV_{\mu^*}^{\Tate_p}$.
\end{prop}

\subsubsection{Cycle classes coming from irreducible components of basic locus}

The cycle class map induces (see \S \ref{cl vs cc})
\begin{equation}\label{E: basic cycle class}
\on{cl}(\bbb):=\cl(\Sh_{\mu,b}^\bbb): \rmH_d^{\mathrm{BM}}(\Sh^\bbb _{\mu,b,\bar k_v}) \cong C\big( G'(\QQ) \backslash  G'(\bA_f)/K\big)\to \rmH_c^d(\Sh_{\mu,\bar k_v}, \Ql(d/2)).
\end{equation}
Putting them together, we have
\begin{equation}
\label{E: all Tate cycle}
\cl(\Sh_{\mu,b,\bar k_v}): \bigoplus_{\bbb}\rmH_d^{\mathrm{BM}}(\Sh^\bbb _{\mu,b,\bar k_v}, \bQ_\ell)\to\rmH_c^d(\Sh_{\mu,\bar k_v}, \Ql(d/2)) 
\end{equation}

On the other hand, we may represent the basic element of $B(G_{\bQ_p},\mu^*)$ as $p^{\tau^*}$ for some $\tau^*\in\xcoch(Z_G)$ (under our assumption that $Z_G$ is connected). Note that then \eqref{E: exotic p} holds, and we have
\begin{equation}
\label{E: JL from basic}
\on{JL}_{\tau,\mu}: \on{H}_c(\Sh_{\tau,\bar k_v})\otimes_{\bfJ}\Hom(\widetilde{V_{\tau}},\widetilde{V_{\mu}})\to  \rmH_c^d(\Sh_{\mu,\bar k_v}, \Ql(d/2))
\end{equation}
from \eqref{E: spetral action1}. We relate these two maps.

For each $\bbb\in\MV_{\mu^*}^\Tate$,  recall that by Lemma \ref{L: finding best tau}, we can find $\tau^*_{\bbb}\in\xcoch(Z_G)$ and a dominant cocharacter $\nu^*_\bbb$ such that
\begin{itemize}
\item $\bbb$ belongs to the image of $i_{\nu_\bbb^*}^{\MV}:\bS_{(\nu^*_\bbb,\mu^*)\mid \tau^*_{\bbb}+\sigma(\nu_\bbb^*)}\to \MV_{\mu^*}(\tau^*_{\bbb}+\sigma(\nu_\bbb^*)-\nu_\bbb^*)$, and
\item $\nu^*_\bbb$ in minimal among all possible $\nu^*$'s such that the above condition holds.
\end{itemize}
Then $\bba:=(i_{\nu_\bba^*}^{\MV})^{-1}(\bbb)$ determines a Satake cycle $\Gr_{(\nu^*_\bbb,\mu^*)\mid \tau_{\bbb}^*+\sigma(\nu^*_\bbb)}^{0,\bba}$, which via Proposition \ref{SS:cycles of geometric Satake} and the geometric Satake, determines
\[\bba_{\on{in}}\in \Hom(V_{\nu^*_\bbb}\otimes V_{\mu^*},V_{\tau^*_\bbb}\otimes V_{\sigma(\nu^*_\bbb)})\cong \Hom(V_{\sigma(\nu_\bbb)}\otimes V_{\tau_\bbb}\otimes V_{\nu_\bbb^*},V_{\mu}),\]
and therefore an element $\Xi_{V_{\nu_\bbb}}(\bba_{\on{in}})\in \Hom(\widetilde{V_{\tau_\bbb}},\widetilde{V_{\mu}})$ (see \eqref{E:XiW2}). But by Proposition \ref{P: corr vs cl for min a}, we have the following commutative diagram
\[
\xymatrix{
\on{H}_c(\Sh_{\tau_\bbb,\bar k_v})\ar^\cong [rr]\ar_{\on{JL}_{\tau_\bbb,\mu}(\bba_{\on{in}})}[dr]&& \on{H}^{\on{BM}}_{d}(\Sh^\bbb_{\mu,b,\bar k_v})\ar^{\on{cl}(\bbb)}[dl]\\
&\on{H}_c^{d}(\Sh_{\mu,\bar k_v},\Ql(d/2))&.
}
\]
In particular
\[\on{Im}\on{cl}(\bbb)=\on{Im}\on{JL}_{\tau_\bbb,\mu}(\bba_{\on{in}}).\]

\begin{rmk}
\label{R: subtle depends on bbb}
Again we remark that the above chosen $\tau_\bbb\in\xcoch(Z_G)$ depends on $\bbb$.
\end{rmk}

Dually, there is the map
\begin{equation}
\label{E: Res map dual to cycle}
\rmH_c^d(\Sh_{\mu,\bar k_v}, \Ql(d/2))\to \on{H}^0(\Sh_{\tau_\bbb,\bar k_v}),
\end{equation}
dual to the cycle class map, see \S \ref{cl vs cc}. In addition, the Satake cycle $\Gr_{(\nu^*_\bbb,\mu^*)\mid \tau_{\bbb}+\sigma(\nu^*_\bbb)}^{0,\bba}$ also induces
\[\bba_{\on{out}}\in \Hom(V_{\tau^*_\bbb}\otimes V_{\sigma(\nu^*_\bbb)}, V_{\nu^*_\bbb}\otimes V_{\mu^*})\cong \Hom(V_{\sigma(\nu^*_\bbb)}\otimes V_{\mu}\otimes V_{\nu_\bbb},V_{\tau_\bbb})\]
Then \eqref{E: Res map dual to cycle} coincides with $\on{JL}_{\mu,\tau_\bbb}(\bba_{\on{out}})$.
We have proven the following lemma.
\begin{lem}Let $\{\bbb_1,\ldots,\bbb_r\}$ be the elements in $\MV_{\mu^*}^{\Tate_p}$. 
The intersection matrix of the cycle classes from \eqref{E: all Tate cycle} is given by a matrix with $(i,j)$-entry
\begin{equation}
\label{E: intersection matrix G}
\on{JL}_{\mu,\tau_{\bbb_j}}(\bba_{j,\on{out}})\circ \on{JL}_{\tau_{\bbb_i},\mu}(\bba_{i,\on{in}})
\end{equation}
\end{lem}

As the matrix of an endomorphism of the vector bundle $\oplus_i \widetilde{V_{\tau_{\bbb_i}}}$ on $[\hat G\sigma_p/\hat G]$, its determinant $\mD$ is a regular function on the stack $[\hat G\sigma_p/\hat G]$. Let $\hat G'=\hat G_\der\times Z(\hat G)^\circ$, and $\Delta= \hat G_\der\cap Z(\hat G)^\circ$. Note that $\sigma_p$ acts on $\hat G'$ and $\Delta$.
The exact sequence 
$1\to \Delta\to \hat G'\to \hat G\to 1$ induces an isomorphism
$$\Gamma([\hat G\sigma_p/\hat G],\mO)\to \Gamma([\hat G'\sigma_p/\hat G'],\mO)^{\Delta_{\sigma_p}},$$ 
where $\Delta_{\sigma_p}=\Delta/(1-\sigma_p)\Delta$ is the group of coinvariants. It is enough to calculate $\mD$ as a function on $[\hat G'\sigma_p/\hat G']$. 
Let $\tau\in Z(\hat G)^\circ$ be the central character of $V_{\mu}$. Restriction of $\la=\tau_{\bbb_i}+(\sigma-1)\nu_{\bbb_i}$ to $Z(\hat G)^\circ$ gives
$\tau=\tau_{\bbb_i}+\sigma(\eta_i)-\eta_i$, where $\eta_i=\nu_{\bbb_i}|_{Z(\hat G)^\circ}$, which can be regarded as a character of $\hat G'$.
It follows that for $G'$, we can use a single $\tau$ for all $\bbb\in \MV^\Tate_{\mu^*}$ in Lemma \ref{L: finding best tau}. In addition, over $\hat G'\sigma_p/\hat G'$, we have the isomorphism
$$\widetilde{V_{\tau}}\simeq \widetilde{V_{\tau_{\bbb_i}}}, \quad f\mapsto fe^{\eta_i}.$$ 
Denote this map by $\bbc_i$. 
The collection $\{\bba_{i,\on{in}}\circ \bbc_i : \widetilde{V_\tau}\to\widetilde{V_{\mu}}\}$ form a basis of $\Hom_{[\hat G'\sigma_p/\hat G']}(\widetilde{V_\tau},\widetilde{V_{\mu}})$, and a similarly dual statement holds. Therefore, the determinant of the intersection matrix \eqref{E: intersection matrix G} is identified with the determinant of the map
\[\Hom_{[\hat G'\sigma_p/\hat G']}(\widetilde{V_\tau},\widetilde{V_{\mu}})\otimes_{\bfJ_{\hat G'}} \Hom_{[\hat G'\sigma_p/\hat G']}(\widetilde{V_\mu},\widetilde{V_{\tau}})\to \bfJ_{\hat G'},\]
which can be rewritten as the pairing of $\bfJ_{\hat G'}$-modules
\begin{equation}
\label{E: final form, intersection matrix}
\bfJ_{\hat G'}(V_\mu\otimes V_{\tau}^*)\otimes_{\bfJ_{\hat G'}} \bfJ_{\hat G'}(V_\tau\otimes V_\mu^*)\to \bfJ_{\hat G'}.
\end{equation}

Finally, we (re)state and prove the main theorem of the paper.
\begin{thm}
\label{T:main theorem again}
Assume that $(G,X)$ is of Hodge type and the center $Z_G$ is connected. Let $K\subset G(\bA_f)$ be a (small enough) open compact subgroup.
Let $p>2$ be a prime, such that $K_p$ is hyperspecial and $V_{\mu^*}^{\Tate_p}\neq 0$. 
Then:
\begin{enumerate}
\item There is a unique inner form $G'$ of $G$, that is trivial at all finite places, and is compact modulo center at the infinite place.

\item $\Sh_{\mu,b}$ is pure of dimension $\frac{d}{2}$. In particular, $d$ is always an even integer. 
There is an $\mH_K$-equivariant isomorphism
$${\rm H}^{\rm BM}_{d}(\Sh_{\mu, b, \overline \FF_v})\cong C(G'(\bQ)\backslash G'(\bA_f)/K,\Ql)\otimes_{\bfJ} V_{\mu^*}^{\Tate_p}.$$
Here via the isomorphism $G'(\bA)\simeq G(\bA)$, we regard $K$ as an open compact subgroup of $G'(\bA_f)$, and we use $C(G'(\bQ)\backslash G'(\bA_f)/K,\Ql)$ to denote the space of $\Ql$-valued functions on the finite set $G'(\bQ)\backslash G'(\bA_f)/K$. 

\item 
Let $\pi_f$ be an irreducible module of $\calH_K$, and let 
$${\rm H}^{\rm BM}_{d}(\Sh_{\mu,b, \overline \FF_v})[\pi_f]=\Hom_{\mH_K}(\pi_f,{\rm H}^{\rm BM}_{d}(\Sh_{\mu,b, \overline \FF_v}) \otimes \Ql)\otimes \pi_f$$ be the $\pi_f$-isotypical component. Then the cycle class map
\[\on{cl}:{\rm H}^{\rm BM}_{d}(\Sh_{\mu, b, \overline \FF_v})\otimes\Ql\to {\rm H}^{d}_{\et,c}\big(\Sh_{\mu,\overline \FF_v},\overline\bQ_\ell(d/2)\big)\]
restricted to ${\rm H}^{\rm BM}_{d}(\Sh_{\mu, b, \overline \FF_v})[\pi_f]$ is injective if the Satake parameter of $\pi_{f,p}$ (the component of $\pi_f$ at $p$) is general with respect to $V_\mu$.

\item Assume that $\mathbf{Sh}_K(G,X)$ is a Kottwitz arithmetic variety (i.e. those considered in \cite{Kolambda}), or assume that $G_\der$ is simply-connected and anisotropic, and there is a place $p\neq p'$ such that $\pi_{f,p'}$ is an unramified twist of Steinberg representation. Then the $\pi^p_f$-isotypical component of the cycle class map $\on{cl}$ is surjective to 
$$\sum_{\pi_{p}}T^d(\pi_p\pi_f^p,\Ql)\otimes\pi_{p}\pi^p_f$$ 
if the Satake parameters of $\{\pi_{p}\}$ are all strongly general with respect to $V_\mu$. 
In particular, the Tate conjecture holds for these $\pi^p_f$.
\end{enumerate}
\end{thm}
\begin{proof}
Part (1) follows from the combination of Corollary \ref{inner mu} and Proposition \ref{unramified basic}. Part (2) follows from Theorem \ref{C: RZ uniformization basic} and Theorem \ref{C: irr comp up to J}. Since we can write the intersection matrix of the cycle classes as \eqref{E: final form, intersection matrix},
Part (3) follows Theorem \ref{T:intro Chevalley restriction}. Finally, Part (4) follows from Part (3), Theorem \ref{JL:mult} in the first situation and Corollary \ref{C: comparison, stable case} in the second situation, and the following lemma.
\begin{lem}
\label{L:strongly general again}
Let $\ga\phi_p\in \hat T\rtimes \Gal(\FF_{p^m}/\FF_p) \subset \hat G\rtimes \Gal(\FF_{p^m}/\FF_p)$ be an element whose image in $\widehat{Z_G^\circ} \rtimes \Gal(\FF_{p^m}/\FF_p)$ has finite image. Then 
\[
V_{\mu^*}^{\Tate_p}\subseteq \bigcup_{j\geq 1} (V_{\mu^*})^{r_{\mu^*}((\ga\phi_p)^{j[v:p]})},
\]
and if $\ga$ is strongly general in the sense of Definition~\ref{D: general parameter}, the above inclusion is an isomorphism.
\end{lem}
\begin{proof}
We can rewrite $V^{\Tate_p}$ alternatively as
\[
V^{\Tate_p} =  \bigoplus_{\lambda \in \XX^\bullet(\widehat{Z_{ G}^\circ}^{\sigma})} V|_{\hat G^\sigma}(\lambda).
\]
For $\lambda \in \XX^\bullet(\widehat{Z_G^\circ}^\sigma)$, $r_{\mu^*}((\gamma\phi_p)^{jm})$ acts trivially on $V|_{\hat G^\sigma}(\lambda)$ because  $\gamma\phi_p$ has finite image in $\widehat{Z_G^\circ}$. This implies the inclusion \eqref{E:VTate inclusion}. On the other hand, if $\lambda \in \XX^\bullet (\hat T^\sigma)$ and it does not  factor through $\widehat{Z_G^\circ}^\sigma$, $V|_{\hat G^\sigma}(\lambda)$ has no invariants under $r_{\mu^*}((\gamma\phi_p)^{jm})$ for any $j >1$ when $\gamma \sigma$ is strongly general with respect to $V$. In this case, the inclusion \eqref{E:VTate inclusion} is an equality.
\end{proof}
\end{proof}

\appendix

\section{Preliminaries from algebraic geometry.}
\subsection{(Pre)stacks in $\Aff^\pf$.}
\label{ASS:perfect AG}
Let $k$ be a field. Recall that in the usual formalism of algebraic geometry, one considers the category of $k$-algebras $\Aff_k$ as the test objects. The category of $\Aff_k$ can be endowed with various topology (Zariski, \'etale, fpqc, etc.).
Then one can define the category of schemes (resp. algebraic spaces, resp. algebraic stacks) as certain full subcategory of the category of sheaves (resp. stacks) on $\Aff_k$. We can mimic this approach to do algebraic geometry over perfect rings.

Let $k$ be a perfect field of characteristic $p>0$.
We denote by $\Aff^\pf_k$ the category of perfect $k$-algebras. 

\begin{dfn}A \emph{presheaf} is a covariant functor from $\Aff^\pf_k$ to the category of sets. A \emph{prestack} is a covariant $2$-functor from $\Aff^\pf_k$ to the $2$-category of groupoids. By regarding a set as a discrete groupoid, we regard a presheaf as a prestack.
\end{dfn}
We remark that we considered (pre)stacks in the paper, like $\Hk_\mu^\loc$ and $\Sht_\mu^\loc$. But we only used $\ell$-adic formalism for algebraic stacks, which we define in the sequel.

We can equip $\Aff_k^\pf$ with Zariski, \'etale or fpqc topology.\footnote{Note that however we avoid fppf topology since perfect rings over $k$ are almost never of finite type.} Then we have the notion of sheaves (resp. stacks) in $\Aff_k^\pf$. 
A sheaf in fpqc topology is sometimes also called a space. Note that all the mentioned topology are subcanonical. 

In \cite{BS}, Bhatt-Scholze also defined the $v$-topology on $\Aff_k^\pf$, which is finer that any of the above mentioned topology but is still subcanonical. We will not make use of this topology in the paper.

\begin{dfn}
An affine scheme $\Spec R$ in $\Aff_k^\pf$ is a presheaf on $\Aff_k^\pf$ of the form $\Hom_{\Aff_k^\pf}(R,-)$. 
A scheme in $\Aff_k^\pf$ is Zariski sheaf that admits a Zariski cover by affine schemes.
\end{dfn}
As mentioned above, affine schemes are Zariski sheaves and therefore are schemes. The category of schemes in $\Aff_k^\pf$ admits finite products. One can define the usual Zariski, \'etale, fpqc topology on a scheme. 
In addition, it makes sense to define schematic morphisms of presheaves in $\Aff_k^\pf$. 
\begin{dfn}
An \emph{algebraic space} $X$ in $\Aff^\pf_k$ is an \'etale sheaf such that the diagonal is schematic, and there is an \'etale surjection from a scheme $U\to X$.
\end{dfn}

The product of two algebraic spaces in $\Aff_k^\pf$ exists as an algebraic space in $\Aff_k^\pf$.
As usual, it makes sense to define representable morphisms (by algebraic spaces) of presheaves in $\Aff_k^\pf$. For a representable morphism of presheaves, the usual topological notions for algebraic spaces make sense, such as quasi-compact, quasi-separated, separated, connected, irreducible, etc.

\begin{rmk}
The category of schemes in $\Aff^\pf_k$ is a full subcategory of the category of algebraic spaces in $\Aff_k^\pf$.  The latter is a full subcategory of the category of fpqc sheaves on $\Aff_k^\pf$ (\cite[Tag03W8]{stack}). In fact, algebraic spaces in $\Aff_k^\pf$ are even $v$-sheaves in the sense of \cite{BS}.
\end{rmk}

\begin{dfn}
Let $\mF$ be a presheaf on $\Aff_k$. Its \emph{perfection}, denoted by $\mF^\pf$, is its restriction to $\Aff_k^\pf\subset \Aff_k$.
\end{dfn}
It is easy to see that the perfection of a scheme (resp. algebraic space) in $\Aff_k$ is represented by a scheme (resp. algebraic space) in $\Aff_k^\pf$.

On the other hand, recall in usual algebraic geometry, there is another definition of perfection of a scheme (resp. algebraic space) in $\Aff_k$:
a scheme $X$ (resp. an algebraic space) in $\Aff_k$ is called perfect if its Frobenius endomorphism $\sigma_X$ is an isomorphism. Given a scheme (resp. algebraic space) $X$, the perfection of $X$ is defined as 
$$X^\pf:=\underleftarrow\lim_{\on{\sigma}_X}X.$$ The functor $X\mapsto X^\pf$ is right adjoint to the forgetful functor from the category of perfect schemes (resp. algebraic spaces) to the category of all schemes (resp. algebraic spaces).
The following lemma is essentially \cite[Lemma A.10]{Z}, which explains that the two notions of perfections are essentially equivalent.
\begin{lem}The functor from the category of perfect schemes (resp. perfect algebraic spaces) in $\Aff_k$ to the category of schemes (resp. algebraic spaces) in $\Aff^\pf_k$ given by restriction is an equivalence.
\end{lem}

From now on, by abuse of language, we also call a scheme (resp. algebraic spaces) in $\Aff^\pf_k$ a perfect scheme (resp. perfect algebraic space).

\begin{dfn}
A perfect algebraic space $X$ is \emph{perfectly locally of finite type} if there exists an \'etale affine cover $\{U_i\}$ of $X$ such that each $U_i$ is the perfection of an affine scheme of finite type over $k$. It is called \emph{perfectly of finite type} if it is perfectly locally of finite type and quasi-compact. It is called \emph{perfectly of finite presentation} (\emph{pfp} for short) if it is perfectly of finite type and quasi-separated (i.e. diagonal is quasi-compact). 

A morphism $f: X\to Y$ between pfp perfect algebraic spaces is called \emph{perfectly proper} if it is separated and universally closed.
\end{dfn}

Recall the following results (\cite[\S A]{Z}).
\begin{prop}
\label{P: deperfection}
\begin{enumerate}
\item A perfect algebraic space is pfp if and only if it is the perfection of an algebraic space of finite presentation over $k$.
\item Any morphism $f:X\to Y$ between two pfp perfect algebraic spaces is the perfection of a morphism $f':X'\to Y'$ between two algebraic spaces (in $\Aff_k$) of finite presentation.
\item A perfectly proper morphism between pfp perfect algebraic spaces is the perfection of a proper morphism between to algebraic spaces of finite presentation.
\end{enumerate}
\end{prop}
For a pfp algebraic space $X$, an algebraic space $X'$ of finite presentation such that $X={X'}^\pf$ is called a \emph{deperfection} of $X$.

\begin{dfn}
Let $f:X\to Y$ be a map between two algebraic spaces. We say that $f$ is \emph{perfectly smooth} at $x\in X$ if there exists an \'{e}tale atlas $U\to X$ at $x$ and an \'etale atlas $V\to Y$ at $f(x)$, such that the map $U\to Y$ factors as $U\stackrel{h}{\to} V\to Y$ and $h$ factors as $h=\pr \circ h'$, where $h':U\to V\times(\bA^n)^\pf$ is \'{e}tale and $\pr:V\times (\bA^n)^\pf\to V$ is the projection. We say that $f$ is \emph{perfectly smooth} if it is perfectly smooth at every point of $X$. We say that $X$ is perfectly smooth (at $x$) if $X\to \Spec k$ is perfectly smooth (at $x$).
\end{dfn}

Now we define algebraic stacks. First, it makes sense to define the notion of prestacks in $\Aff_k^\pf$, and stacks with respect to \'etale or fpqc topology of $\Aff_k^\pf$.

\begin{dfn}
An \emph{algebraic stack} $X$ in $\Aff^\pf_k$ is a stack with respect to the fpqc topology, such that (i) the diagonal is represented by an algebraic space in $\Aff^\pf_k$; (ii) there exists a perfectly smooth surjective map $U\to X$ from an algebraic space. Such a $U$ is called a \emph{smooth atlas}.
\end{dfn}

\begin{dfn}
An algebraic stack $X$ in $\Aff^\pf_k$ is called \emph{quasi-compact} if there exists a smooth atlas $U\to X$ with $U$ quasi-compact.

An algebraic stack $X$ in $\Aff^\pf_k$ is called \emph{locally of finite type} (resp. \emph{perfectly of finite type}, resp. \emph{pfp})
if there exists one smooth atlas $U\to X$, where $U$ is perfectly locally of finite type (resp. perfectly of finite type, resp. pfp).

A morphism of algebraic stack $f:X\to Y$ in $\Aff^\pf_k$ is called 
\begin{itemize}
\item \emph{separated} if the diagonal map $X\to X\times_YX$ (which is always representable by an algebraic space) is perfectly proper;
\item \emph{quasi-separated} if the diagonal is quasi-compact and quasi-separated;
\item \emph{DM} if the diagonal is \'etale.
\end{itemize}
\end{dfn}

\begin{dfn}
A morphism $f: X\to Y$ of pfp algebraic stacks in $\Aff_k^\pf$ is called \emph{perfectly proper} if $f$ is separated, perfectly of finite type, and universally closed.
\end{dfn}

\begin{dfn} A morphism $f: X\to Y$ of pfp algebraic stacks in $\Aff_k^\pf$ is called \emph{perfectly smooth} if there is a perfectly smooth surjective morphism $U\to X$ from an algebraic space $U$ such that the composition $U\to X\to Y$ is perfectly smooth.
\end{dfn}

\begin{ex}
Let $G$ be a perfect algebraic group (i.e. the perfection of an algebraic group) over $k$. Then $\bfB G\to \Spec k$ is perfectly smooth.
\end{ex}

As before, for a prestack $\mF$ in $\Aff_k$, its restriction to $\Aff_k^\pf\subset \Aff_k$ is called the \emph{perfection of $\mF$}. It is clear that the perfection of algebraic stacks in $\Aff_k$ are algebraic stacks in $\Aff_k^\pf$.
Note however, for a usual algebraic stack $X$ in $\Aff_k$, the Frobenius endomorphism $\sigma_X$ may not be representable and therefore the inverse limit $\underleftarrow\lim_{\sigma_X}X$ may not exist as an algebraic stack in $\Aff_k$.

\subsubsection{The $\ell$-adic formalism}
Let $\ell\neq p$. As explained in \cite[\S A.3]{Z}, for a pfp algebraic space $X$ in $\Aff_k^\pf$,  the $\ell$-adic derived category $D(X, \Ql)$ is well-defined, 
and there exists the six operation formalism. Thanks to Proposition \ref{P: deperfection}, no additional effort is needed to construct this theory.
In addition, the proper base change isomorphism (for perfectly proper morphisms) and the smooth base change isomorphism (for perfectly smooth morphisms) hold. In particular, the category of perverse sheaves $\on{P}(X)$ is well-defined, and is the heart of a t-structure of $D(X,\Ql)$.

We can endow the category of pfp algebraic spaces with smooth topology (the covering maps being perfectly smooth surjective maps). Then as usual, perverse sheaves form a stack of abelian categories with respect to this topology. Therefore, given a pfp algebraic stack $X$ in $\Aff_k^\pf$, the category of perverse sheaves $\on{P}(X)$ on $X$ is well-defined.

Note that as explained in \cite{LZ1, LZ2}, the six operation formalism and the base change isomorphism (and their enhancements to the $\infty$-categorical level) can be expressed as certain functor from a simplicial set (built out of separated perfectly of finite type schemes over $k$) to the category of presentable $\infty$-categories sending $X$ to $D(X,\Ql)$.

Now the same techniques as developed in \cite{LO1, LO2, LZ1, LZ2} allows one to define the  $\ell$-adic derived category $D(X, \Ql)$ for pfp algebraic stacks in $\Aff_k^\pf$. For a smooth atlas $U\to X$ with $U$ a pfp algebraic space, let $U^{\bullet/X}\to X$ denote its \v Cech nerve. Then $D(X,\Ql)$ is the limit (as $\infty$-categories) of $D(U^{\bullet/X},\Ql)$. As usual, $\on{P}(X)$ forms the heart of a t-structure of $D(X,\Ql)$.
In fact, we just need quotient stack in the paper. If $X=[Y/G]$, then we can define $D(X,\overline\bQ_\ell)$ as the equivariant derived category $D_G(Y,\overline\bQ_\ell)$ and $\on{P}(X)$ as $\on{P}_G(Y)$ (up to a shift).

\begin{notation}
\label{N: BM homology}
For a perfect pfp algebraic space $X$ over $k$, we denote by 
$$\on{H}^{\on{BM}}_i(X_{\bar k}):= \on{H}^{-i}(X_{\bar k},\omega_X(-i/2))$$ the $i$th \emph{Borel-Moore homology} of $X_{\bar k}$. Note that our definition is different from the standard one by a Tate twist. The reason is with this normalization, if $\dim X=d$, then $\on{H}^{\on{BM}}_{2d}(X_{\bar k})$ is the $\Ql$-vector space with a basis given by its irreducible components of dimension $d$.
\end{notation}

\subsubsection{The trace map}

Recall that as explained in \cite[\S A.3]{Z}, if $X$ is a pfp algebraic space of dimension $d$, then the fundamental class 
$$[X]\in \on{H}^{\on{BM}}_{2d}(X_{\bar k})\cong \on{H}_c^{2d}(X_{\bar k},\Ql(d))^*$$ 
of $X$ (which is the sum of the fundamental classes of the irreducible components of $X_{\bar k}$ of dimension $d$) defines the trace map
$\Tr:\on{H}_c^{2d}(X_{\bar k}, \Ql(d))\to \Ql$.
As in \cite{LZ1,LZ2}, this trace map extends to pfp algebraic stacks. In addition, if $f:X\to Y$ a equidimensional perfectly smooth morphism of pfp algebraic stacks, of relative dimension $d$, then the fiberwise trace maps can be organized into a family, i.e. there is the trace map
\[\Tr_f: R^{2d}f_!\Ql(d)\to \Ql,\]
which induces the fiberwise trace map by restriction. By the projection formula, it induces a natural map $f^*[2d](d)\to f^!$ which is an isomorphism. In practice, it is convenient to introduce the following condition on morphisms, which is weaker than perfectly smooth but share the same cohomological properties.
\begin{dfn}
\label{D: coh smooth}
An equidimensional morphism $f: X\to Y$ of pfp algebraic stacks in $\Aff_k^\pf$, of relative dimension $d$, is called \emph{cohomologically smooth} if there is a trace map
\[\Tr_f: R^{2d}f_!\Ql(d)\to \Ql,\]
which restricts to the usual trace map along each geometric fiber $f_y: X_{\bar y}\to {\bar y}$ for $y\in Y$, and which induces an isomorphism $f_y^*[2d](d)\cong f_y^!$.
\end{dfn}
Note that trace map, if exists, is unique. In addition, cohomologically smooth morphisms are stable under base change. Part (1) of the following statement justifies the terminology.
\begin{prop}
(1) Let $f:X\to Y$ be cohomologically smooth. Then the trace map induces $f^*[2d](d)\to f^!$, which is an isomorphism.
(2) The composition of cohomologically smooth maps are cohomologically smooth.
\end{prop}
\begin{proof}
Part (1) is clear and Part (2) follows from Part (1) directly.
\end{proof}

Here is an easy but useful criterion for a morphism to be cohomologically smooth. 
\begin{lem}
\label{L: criterion coh smooth}
\begin{enumerate}
\item Let $f: X\to Y$ be a morphism of perfectly smooth pfp schemes, and assume that each fiber of $f$ is perfectly smooth, equidimensional of dimension $d$. Then $f$ is cohomologically smooth.
\item Let $f:X\to Y$ be a morphism of pfp schemes. Assume that there is a perfectly proper morphism $g: \tilde Y\to Y$ with connected geometric fibers such that the base change $\tilde X:=X\times_Y\tilde Y\to\tilde Y$ is cohomologically smooth, then $f$ is cohomologically smooth.
\end{enumerate}
\end{lem}
Note that Part (2) of the lemma essentially says that cohomological smoothness can be checked locally in the $v$-topology introduced in \cite{BS}. 
\begin{proof}
(1) Since $X$ and $Y$ are perfectly smooth, 
$$f^!\Ql[2d](d)\cong f^!\omega_Y[2d-2d_Y](d-d_Y)\cong \omega_X[-2d_X](-d_X)\cong \Ql,$$ where $d_X=\dim X$, $d_Y=\dim Y$, and $\omega_X$ and $\omega_Y$ are the dualizing sheaves. 
This isomorphism induces $f_!\Ql[2d](d)\to \Ql$ which in turn induces $\Tr_f: R^{2d}f_!\Ql(d)\to \Ql$ since $R^{i}f_!\Ql(d)=0$ for $i>2d$. It follows from the construction that fiberwise, $\Tr_f$ restricts to the usual trace map. Since fibers are perfectly smooth, $f_y^*[2d](d)\cong f_y^!$.

(2) Let use denote $\tilde X\to X$ by $g_{X}$ and $\tilde Y\to Y$ by $g_{Y}$. Let $\tilde f: \tilde X\to \tilde Y$ be the base change of $f$. Then we have
\[g_Y^*f_!\Ql[2d](d)\cong \tilde f_!\Ql[2d](d)\xrightarrow{\Tr_{\tilde f}} \Ql,\]
which induces $f_!\Ql[2d](d)\to (g_Y)_*\Ql$. Since $f_!\Ql[2d](d)$ lives in cohomological degree $\leq 0$ and $(g_Y)_*\Ql$ in cohomological degree $\geq 0$. This gives 
$$\Tr_f:R^{2d}f_!\Ql(d)\to R^0(g_Y)_*\Ql=\Ql$$ 
since the fibers of $g_Y$ are geometrically connected.  It is clear that $\Tr_f$ restricts to the fiberwise trace map. The isomorphism $f_y^*[2d](d)\cong f_y^!$ is clear.
\end{proof}

We would like to have a base change homomorphism (but not necessarily an isomorphism) in non-Cartesian case.
\begin{dfn}
\label{D: basechangeable}
A commutative square of perfect pfp algebraic stacks 
\begin{equation}
\label{E: basechangeable}
\begin{CD}
D@>p>>C\\
@VbVV@VVaV\\
Y@>f>>X
\end{CD}
\end{equation}
is called \emph{base changeable} if the induced morphism $h:D\to C\times_XY$ is representable by perfectly proper algebraic spaces.
\end{dfn}

The following lemma is straightforward.
\begin{lem}
\label{L: criterion basechangeable}
\begin{enumerate}
\item A commutative diagram as \eqref{E: basechangeable} is base changeable if
\begin{itemize}
\item either $p$ is representable by perfectly proper algebraic spaces and $f$ is separated;
\item or $b$ is representable by perfectly proper algebraic spaces and $a$ is separated.
\end{itemize}
\item In the commutative diagram
\begin{equation}
\label{E: two base changeable}
\begin{CD}
E@>q>>D@>p>>C\\
@VcVV@VbVV@VVaV\\
Z@>g>>Y@>f>>X
\end{CD}
\end{equation}
if both inner squares are base changeable, so is the outer rectangle.
\end{enumerate}
\end{lem}

This name is justified by the following fact.

\begin{lem}
\label{L:proper base change}
Assume we have a base changeable commutative diagram of perfect pfp algebraic stacks
as \eqref{E: basechangeable}.
Then there is a natural base change homomorphism
\begin{equation}
\label{E: base change isom}
\on{BC}^*_!: a^*f_!\to p_!b^*,
\end{equation}
which is an isomorphism if \eqref{E: basechangeable} is Cartesian.
If $p$ and $f$ are separated, it fits into the commutative diagram 
\[\begin{CD}
a^*f_!@>>> p_!b^*\\
@VVV@VVV\\
a^*f_*@>>>p_*b^*,
\end{CD}\]
where the bottom arrow is the natural adjunction.

In addition, given \eqref{E: two base changeable} with both squares base changeable, then the base change homomorphism $a^*(fg)_!\to (pq)_!c^*$ is equal to the composition 
$$a^*(fg)_!=a^*f_!g_!\xrightarrow{\on{BC}_!^*} f_!b^*g_!\xrightarrow{\on{BC}_!^*} f_!g_!c^*=(fg)_!c^*.$$ 
Similarly, $(fg)^*a_!\to c_!(pq)^*$ is equal to the composition
$$(fg)^*a_!=g^*f^*a_!\xrightarrow{\on{BC}_!^*} g^*b_!f^*\xrightarrow{\on{BC}_!^*} c_!g^*f^*=c_!(fg)^*.$$
\end{lem}
\begin{proof} This follows from the previously mentioned properties of the base change isomorphism, together with the natural adjunction $\id\to h_*h^*\cong h_!h^*$ where $h: D\to C\times_XY$.
\end{proof}

\begin{rmk}
\label{R: another adjunction}
Given a diagram \eqref{E: basechangeable}, base changeable or not, there is always a homomorphism
\begin{equation}
\label{E: shrek adjunction}
 p_!b^!\to a^!f_!
 \end{equation}
by adjunction. In addition, given \eqref{E: two base changeable}, the map $(pq)_!c^!\to a^!(fg)_!$ is equal to
\[(pq)_!c^!=p_!q_!c^!\to p_!b^!g_!\to a^!f_!g_!=a^!(fg)_!,\]
and the map  $c_!(pq)^!\to (fg)^!a_!$ is equal to
\[c_!(pq)^!=c_!q^!p^!\to g^!b_!p^!\to g^!f^!a_!=(fg)^!a_!.\]
\end{rmk}

Now assume that \eqref{E: basechangeable} is Cartesian. Then the base change  isomorphism $\on{BC}^*_!$ also induces (by adjunction)
\begin{equation}
\label{E: base change shrek}
\on{BC}^{*!}: b^*f^!\to p^!a^*.
\end{equation}
If $f$ is  equidimensionally cohomologically smooth of relative dimension $d$, since the trace map is compatible with the base change, \eqref{E: base change shrek} coincides with
\begin{equation}
\label{E: base change shrek adjunction}
b^*f^!\cong b^*f^*[2d](d) = p^*a^*[2d](d) \cong p^!a^*
\end{equation}

\subsection{Review of cohomological correspondence}
\label{Sec:cohomological correspondence}
We recall the formalism of cohomological correspondences, partially following \cite{Vas}.
We will assume that $k$ is a perfect field of characteristic $p>0$. Below, stacks are algebraic stacks, perfectly of finite presentation in $\bf{Aff}_k^\pf$.\footnote{Of course, the same discussions will carry through when we work with algebraic stacks in $\Aff_k$. } We write $\mathrm{pt}$ for $\Spec k$.   We fix $\ell\neq p$, and fix a half Tate twist $\Ql(1/2)$. As usual, we write $\langle d\rangle$ for $[d](d/2)$. 

\begin{notation}
\label{N:star pullback}
For an equidimensional  cohomologically smooth morphism $f:X \to Y$ between algebraic stacks of relative dimension $d$, we write $f^\star: = f^*\langle d\rangle$; it takes perverse sheaves to perverse sheaves.
\end{notation}

If $X$ is a stack, let $D^*_c(X)\subset D(X,\Ql) (*=+,-,b,\emptyset)$ denote its constructible derived categories of $\ell$-adic sheaves. Let $\omega_X\in D^b_c(X)$ denote the dualizing sheaf on $X$. For $\mF\in D_c^b(X)$, let $\bD\mF=R\underline{\Hom}(\mF,\omega_X)$ denote the Verdier dual. By definition, there are canonical (unit and counit) maps 
\begin{equation}
\label{E: VD counit}
e: \mF\otimes \bD\mF\to \omega_X,\quad \delta: \Ql\to \mF\otimes^! \bD\mF:= \Delta^!(\mF\boxtimes\bD\mF)
\end{equation}

\begin{definition}
\label{AD:correspondences}
Let $(X_i,\mF_i)$ for $i=1,2$ be two pairs, where $X_i$ are stacks, and $\mF_i\in D(X_i,\Ql)$. A \emph{cohomological correspondence} $(C,u):(X_1,\mF_1)\to (X_2,\mF_2)$ is a stack $C\xrightarrow{c_1\times c_2} X_1\times  X_2$, together with a morphism $u:c_1^*\mF_1\to c_2^!\mF_2$. We call $C$ the \emph{support} of $(C,u)$. For simplicity, $(C,u)$ is sometimes denoted by $C$ or by $u$. There is an obvious notion of (iso)morphisms between two cohomological correspondences $(C,u)$ and $(C',u')$ from $(X_1,\mF_1)$ to $(X_2,\mF_2)$. The set of isomorphism classes of cohomological correspondences from $(X_1,\mF_1)$ to $(X_2,\mF_2)$ supported on $C$ is denoted by $\on{Corr}_C((X_1,\mF_1),(X_2,\mF_2))$ and by definition
\[\on{Corr}_C((X_1,\mF_1),(X_2,\mF_2))\cong \Hom_C(c_1^*\mF_1,c_2^!\mF_2).\]
When $c_1$ is proper, a cohomological correspondence $u$ will naturally induce a homomorphism $\on{H}_c(u): \on{H}^*_c(X_1\otimes\bar k, \calF_1) \to \on{H}^*_c(X_2\otimes \bar k, \calF_2)$,
 given as the composition of the following maps
\[
\on{H}^*_c(X_1\otimes \bar k, \calF_1) \to \on{H}^*_c(C\otimes \bar k, c_1^*\calF_1) \xrightarrow{\on{H}^*_c(u)} \on{H}^*_c(C\otimes \bar k, c_2^!\calF_2) \to \on{H}^*_c(X_2\otimes \bar k, \calF_2).
\]
In fact, this is a special case of the push-forward cohomological correspondences to be reviewed below.

If $(C,u):(X_1,\mF_1)\to (X_2,\mF_2)$ and $(D,v):(X_2,\mF_2)\to (X_3,\mF_3)$ are two cohomological correspondences, we define their composition to be $(C\times_{X_2}D,v\circ u):(X_1,\mF_1)\to (X_3,\mF_3)$, where $v\circ u$ is the following composition
\[p^*c_1^*\mF_1\xrightarrow{p^*u} p^*c_2^!\mF_2\xrightarrow{\on{BC}^{*!}} q^!d_1^*\mF_2\xrightarrow{q^!v} q^!d_2^!\mF_3,\]
$\on{BC}^{*!}$ is the base change isomorphism as in \eqref{E: base change shrek}, and $p, q$ are the projections from $C\times_{X_2}D$ to $C$ and to $D$ respectively.
For the induced homomorphism on the cohomology groups, we have $\on{H}_c(v\circ u) = \on{H}_c(v) \circ \on{H}_c(u)$, as maps $\on{H}^*_c(X_1\otimes\bar k, \calF_1) \to \on{H}^*_c(X_3\otimes \bar k, \calF_3)$, see Lemma \ref{AL:pushforward compatible with composition} below.
\end{definition}

Here are examples we use in the paper.

\begin{ex}
\label{Ex:examples of correspondences}
\begin{enumerate}

\item For a morphism $f: X\to Y$ and $\mF \in D(X,\Ql)$, there is a cohomological correspondence $(X\xleftarrow{\id}X\xrightarrow{f} Y, u:\mF\to f^!f_!\mF)$ from $(X,\mF)$ to $(Y,f_!\mF)$, called the \emph{pushforward correspondence}, and denoted by $(\Ga_f)_!$ for simplicity.

\item For a morphism $f: X\to Y$ and $\mF \in D(Y,\Ql)$, there is a natural cohomological correspondence ($Y\xleftarrow{f} X\xrightarrow{\id} X, u:= \mathrm{id}:f^*\mF\to f^*\mF)$ from $(Y,\mF)$ to $(X,f^*\mF)$, called the \emph{pullback correspondence}, and denoted by $\Ga_f^*$ for simplicity.

\item If $(C,u): (X_1,\mF_1)\to (X_2,\mF_2)$ is a cohomological correspondence, with $\mF_i\in D_c^b(X_i)$, then $(C,\bD u): (X_2, \bD \mF_2)\to (X_1,\bD \mF_1)$ is a cohomological correspondence, called the \emph{dual correspondence}.

\item Let $(X,\mF)$ be a pair with $\mF\in D_c^b(X)$. Then \eqref{E: VD counit} gives a cohomological correspondence and its dual
$$e_\mF\in \on{Corr}_{X}((X\times X, \mF\boxtimes\bD\mF),(\on{pt}, \Ql)),\quad \delta_\mF\in \on{Corr}_{X}((\on{pt}, \Ql), (X\times X, \mF\boxtimes\bD\mF)).$$

\item Assume that $k=\bar k$. Let $X_1, X_2$ be perfectly smooth algebraic spaces, of pure dimension, and let $X_1\leftarrow C\to X_2$ be a correspondence. Then 
\begin{equation}
\label{E:corr and cycle}
\begin{split}
\on{Corr}_C((X_1,\Ql\langle d_1 \rangle), (X_2, \Ql\langle d_2\rangle))&=\Hom_{D^b_c(C)}(\Ql\langle d_1\rangle,\omega\langle d_2-2\dim X_2\rangle)\\
&=\on{H}^{\on{BM}}_{2\dim X_2+d_1-d_2}(C).
\end{split}
\end{equation}
So if $2\dim C=2\dim X_2+d_1-d_2$, $\on{Corr}_C((X_1,\Ql\langle d_1 \rangle), (X_2, \Ql\langle d_2\rangle))$ can be regarded as space of functions on the set of irreducible components of $C$ of maximal dimension. In particular,
If $X_1,X_2,C$ are finite sets, regarded as varieties over $k$,  $\on{Corr}_C((X_1,\Ql),(X_2,\Ql))$ can be identified with the space of $\Ql$-valued functions on $C$. 
\end{enumerate}
\end{ex}

We will also make use of the following construction.
\begin{lem}
\label{L:sharp correspondence}
Assume that $\mF_i\in D_c^b(X_i)$.
The set of cohomological correspondences from $(X_1,\mF_1)$ to $(X_2,\mF_2)$ is canonically bijective to the set of cohomological correspondences from $(X_1\times X_2,\mF_1\boxtimes\bD\mF_2)$ to $(\mathrm{pt}, \overline \QQ_\ell)$. 
We denote the this bijection as $(C, u) \leftrightarrow(C, u)^\sharp$ or simply $u\leftrightarrow u^\sharp$, called the \emph{sharp correspondence}.
Moreover, $(\DD u)^\sharp$ is the same as $u^\sharp$ if we naturally identify $X_1 \times X_2$ with $X_2 \times X_1$.

\end{lem}
\begin{proof}
Explicitly, given $(C,u)$, $u^\sharp$ can be constructed as
\begin{equation}
\label{E:sharp correspondence explicit}
(c_1 \times c_2)^*\calF_1 \boxtimes \DD \calF_2 \cong c_1^*\calF_1 \otimes c_2^*\DD\calF_2 \xrightarrow{u \otimes \id} c_2^! \calF_2 \otimes \DD c_2^! \calF_2 \xrightarrow{ \eqref{E: VD counit}}
\omega_C.
\end{equation}
It is clear that taking the dual on $u$ does not change the sharp correspondence (after swapping the two factors).
\end{proof}

\begin{ex}
\label{Ex:sharp correspondence}
(1) Under the sharp correspondence, the identity cohomological correspondence $\id$ becomes $\id^\sharp=e_\mF\in \on{Corr}_X((X\times X, \mF\boxtimes\bD\mF),(\on{pt},\Ql))$ from Example \ref{Ex:examples of correspondences}(4).

(2) In case when $X_1 =X_2 = \mathrm{pt}$ and $\calF_1 = \calF_2 = \overline \QQ_l$, $u$ and $u^\sharp$ both viewed as homomorphisms $\overline \QQ_{\ell, C} \to \omega_C$ are the same.
\end{ex}

\subsubsection{Pushforward of cohomological correspondences}
\label{ASS:pushforward correspondence}

Now assume that we have the following commutative diagram
\[\begin{CD}
X_1@<c_1<< C@>c_2>> X_2\\
@Vf_1VV@VVfV@VVf_2V\\
Y_1@<d_1<<D@>d_2>>Y_2
\end{CD}\]
and let $(C, u:c_1^*\mF_1\to c_2^!\mF_2)$ be a cohomological correspondence from $(X_1,\mF_1)$ to $(X_2,\mF_2)$. 
Assume that the left commutative square is base changeable (Definition \ref{D: basechangeable}), then we have 
\[d_1^*(f_1)_!\mF_1\xrightarrow{\on{BC}^*_!} f_!c_1^*\mF_1\stackrel{f_!u}{\to} f_!c_2^!\mF_2\xrightarrow{\eqref{E: shrek adjunction}} d_2^!(f_2)_!\mF_2.\]
By abuse of notation, we denote the above map still by $f_!(u)$. Then $(D,f_!(u))$ is a cohomological correspondence from $(Y_1,(f_1)_!\mF_1)$ to $(Y_2,(f_2)_!\mF_2)$. We call this the \emph{pushforward of the cohomological correspondence}.

\begin{rmk}
If $f_1$ is separated and $f$ is representable by perfectly proper algebraic spaces (so the left square is base changeable), then $f_!(u)$ factors as the composition of the natural map $(f_1)_!\mF_1\to (f_1)_*\mF_1$ and
\[d_1^*(f_1)_*\mF_1\to f_*c_1^*\mF_1\to f_!c_2^!\mF_2\to d_2^!(f_2)_!\mF_2.\]
The latter map is the pushforward of cohomological correspondences considered in \cite{Illusie}. However, in most places of this paper, we are in the situation where $d_1$ is separated and $c_1$ is representable by perfectly proper algebraic spaces.
\end{rmk}

In particular, if $(Y_1\xleftarrow{d_1} D\xrightarrow{d_2} Y_2)$ is $(S=S=S)$ for some base scheme $S$, and $c_1$ is proper (so the left square is base changeable), we get a homomorphism of $\ell$-adic sheaves
\[f_!(u): (f_1)_!\mF_1\to (f_2)_!\mF_2.\]
More specifically, if $S=\Spec k$ is algebraically closed, and $c_1$ is proper, we obtain the previously mentioned
$\on{H}_c(u): \on{H}_c^*(X_1,\mF_1)\to \on{H}_c^*(X_2,\mF_2)$.

The pushforward of cohomological correspondences behaves well under compositions.

\begin{lem}
\label{L: comp of push coh corr}
Let
\[\begin{CD}
X_1@<c_1<< C@>c_2>> X_2\\
@Vf_1VV@VVfV@VVf_2V\\
Y_1@<d_1<<D@>d_2>>Y_2\\
@Vg_1VV@VVgV@VVg_2V\\
Z_1@<e_1<<E@>e_2>>Z_2
\end{CD}\]
be a commutative diagram, such that the left two squares are base changeable. Let $u:c_1^*\mF_1\to c_2^!\mF_2$ be a cohomological correspondence from $(X_1,\mF_1)$ to $(X_2,\mF_2)$. Then the left outer rectangle is base changeable and
$(E, (gf)_!(u))$ is equal to $(E,g_!(f_!(u)))$ as cohomological correspondences from $(Z_1,(g_1f_1)_!\mF_1\cong(g_1)_!(f_1)_!\mF_1)$ to $(Z_2,(g_2f_2)_!\mF_2\cong(g_2)_!(f_2)_!\mF_2)$.
\end{lem}
\begin{proof}Using Lemma \ref{L:proper base change} and Remark \ref{R: another adjunction}, this is straightforward.
\end{proof}

Before proceeding, let us state is usual lemma, which will be used several times below.
\begin{lem}
\label{L: cartesian}
In a cubic commutative diagram as below, 
\[
\xymatrix@R=10pt@C=5pt{
& A' \ar[rr]\ar'[d][dd]\ar[dl]
& & B' \ar[dd]\ar[dl]
\\
A \ar[rr]\ar[dd]
& & B \ar[dd]
\\
& C' \ar'[r][rr]\ar[dl]
& & D'\ar[dl]
\\
C \ar[rr]
& & D 
}
\]
if the front, the back, and the bottom faces are Cartesian, so is the top face.
\end{lem}
\begin{proof}
Go through the definitions.
\end{proof}

The pushforward cohomological correspondence is compatible with composition of cohomological correspondences.
\begin{lem}
\label{AL:pushforward compatible with composition}
Suppose we have a commutative diagram
\[
\begin{CD}
X_1@<c_1<< C@>c_2>> X_2@<c'_1<< C'@>c'_2>> X_3\\
@Vf_1VV@VVfV@VVf_2V@VVf'V@VVf_3V\\
Y_1@<d_1<<D@>d_2>>Y_2@<d'_1<<D'@>d'_2>>Y_3
\end{CD}
\]
with the first and the third square base changeable, $c_1,  c'_1$ proper and $d_1,d'_1$ separated, and let $u: c_1^*\calF_1 \to c_2^!\calF_2$ and $v: c'^*_1\calF_2 \to c'^!_2 \calF_3$ be cohomological correspondences from $(X_1, \calF_1)$ to $(X_2, \calF_2)$ and  from $(X_2, \calF_2)$ to $(X_3, \calF_3)$, respectively.
Set $ \widetilde C := C \times_{X_2} C'$ and $\widetilde D: = D \times_{Y_2} D'$ and let $\tilde f: \widetilde C \to \widetilde D$ denote the naturally induced morphism.
Then the following statements hold.
\begin{enumerate}
\item The right square in the diagram 
\[\begin{CD}
X_1@<c_1<<C@<p_C<<\tilde C \\
@Vf_1VV@VVfV@VV\tilde fV\\
Y_1@<d_1<<D@<P_D<<\tilde D
\end{CD}\]
is base changeable. Note that this ensures that $\tilde f_!(v \circ u)$ is well defined by Lemma \ref{L: criterion basechangeable} (1).
\item We have an equality of cohomological correspondences from $(Y_1, (f_1)_!\calF_1)$ to $(Y_3, (f_3)_!\calF_3)$ supported on $\widetilde D$:
\[
\tilde f_!(v\circ u) = f'_!(v) \circ f_!(u).
\]
\end{enumerate}
\end{lem}
\begin{proof}
(1) Using Lemma \ref{L: cartesian}, one sees that $(C\times_D\tilde D)\simeq C\times_{X_2}(X_2\times_{Y_2}D')$. Since $C'\to X_2\times_{Y_2}D'$ is perfectly proper by definition, $\tilde C\to C\times_D\tilde D$ is also perfectly proper.

(2) Let $\tilde c_1= c_1p_C$ and $\tilde d_1=d_1 p_D$. 
We write $q_C: \widetilde C \to C'$, $\tilde c_2= c_2q_C$, $q_D: \widetilde D \to D'$, and $\tilde d_2=d_2 q_D$.  Proving the lemma amounts to proving the commutativity of the following diagram.
\begin{tiny}
\[
\xymatrix@C=15pt{
\tilde d_1^* f_{1,!}\calF_1 \ar@{=}[d] \ar^{\on{BC}_!^*}[r] &
\tilde f_!\tilde c_1^*\calF_1 \ar@{=}[r] &
\tilde f_! p_C^*c_1^*\calF_1 \ar[r]^{u} & \tilde f_! p_C^* c_2^!\calF_2 \ar[r] & \tilde f_! q_C^! c'^*_1 \calF_2
\ar[r]^{v} \ar[dr] & \tilde f_! q_C^! c'^!_2 \calF_3 \ar[dr]
\ar@{=}[r] & \tilde f_!\tilde c^!_2 \calF_3 \ar[r] & \tilde d_2^!f_{3,!}\calF_3 \ar@{=}[d]
\\
p_D^*d_1^*f_{1,!}\calF_1 \ar^{\on{BC}_!^*}[r] & p_D^*f_! c_1^*\calF_1 \ar[r]^-u \ar^{\on{BC}_!^*}[ur] & p_D^*f_! c_2^!\calF_2 \ar[r] \ar[ru] & p_D^*d_2^!f_{2,!}\calF_2 \ar[r] & q_D^! d'^*_1 f_{2,!}\calF_2 \ar[r] & q_D^! f'_! c'^*_1\calF_2 \ar[r]^-v & 
q_D^! f'_!c'^!_2\calF_3 \ar[r] & 
q_D^! d'^!_2 f_{3,!}\calF_3.
}
\]
\end{tiny}

The commutativity of the left pentagon follows from Lemma \ref{L:proper base change}, and the commutativity of the right pentagon follows from Remark \ref{R: another adjunction}.
We remain to explain the commutativity of the middle hexagon. First notice that in the following cubic commutative diagram, the top and the bottom faces are Cartesian, and the left and the right faces are base changeable.
\[
\xymatrix@R=10pt@C=5pt{
& \tilde C\ar^{q_C}[rr]\ar'[d][dd]\ar_{p_C}[dl]
& & C' \ar^{f'}[dd]\ar[dl]
\\
C \ar[rr]\ar_f[dd]
& & X_2 \ar[dd]
\\
& \tilde D \ar'[r][rr]\ar[dl]
& & D'\ar^{d'_1}[dl]
\\
D \ar_{d_2}[rr]
& & Y_2 
}
\]
Applying Lemma \ref{L:proper base change} to the composition $\tilde C\to D'$ and $C\to Y_2$, we obtain
\[
\xymatrix@R=10pt{
& q_{D,!}\tilde f_!p_C^* \ar@{=}[r] &f'_!q_{C,!}p_C^*\ar^{\on{BC}_!^*}_\cong[dr] 
\\
q_{D,!}p_D^*f_! \ar^{\on{BC}_!^*}[ur] \ar^\cong_{\on{BC}_!^*}[dr] &&& f'_!c'^*_1c_{2,!}.
\\
&d'^*_1d_{2,!}f_! \ar@{=}[r]& d'^*_1f_{2,!}c_{2,!}\ar_{\on{BC}_!^*}[ur] 
}
\]
By adjunction of pairs $(q_{D,!}, q_D^!)$ and $(c_{2,!}, c_2^!)$, we obtain
\[
\xymatrix@R=10pt{
& \tilde f_!p_C^*c_2^! \ar[r] &q_D^!f'_!q_{C,!}p_C^*c_2^!\ar[dr] 
\\
p_D^*f_!c_2^! \ar[ur] \ar[dr] &&& q_D^!f'_!c'^*_1.
\\
&q_D^!d'^*_1d_{2,!}f_!c_2^! \ar[r]& q_D^!d'^*_1f_{2,!}\ar[ur] 
}
\]
It remains to notice that both diagrams
\[\xymatrix{
\tilde f_!p_C^*c_2^!\ar[r]\ar_{\on{BC}^{*!}}[d] & q_D^!f'_!q_{C,!}p_C^*c_2^!\ar[d] &  p_D^*f_!c_2^!\ar[r]\ar_{\eqref{E: shrek adjunction}}[d]& q_D^!d'^*_1d_{2,!}f_!c_2^!\ar[d]  \\
\tilde f_!q_C^!c'^*_1\ar^{\eqref{E: shrek adjunction}}[r] & q_D^!f'_!c'^*_1 &  p_D^*d_2^!f_{2,!}\ar^{\on{BC}^{*!}}[r] & q_D^!d'^*_1f_{2,!}
}\]
are commutative.
\end{proof}

\subsubsection{Pullback of cohomological correspondences}
\label{ASS:smooth pullback correspondence} 
Assume that we have the following commutative diagram
\[\begin{CD}
Y_1@<d_1<< D@>d_2>> Y_2\\
@Vf_1VV@VVfV@VVf_2V\\
X_1@<c_1<<C@>c_2>>X_2
\end{CD}\]
and assume that $f$ is cohomologically smooth, equidimensional, of relative dimension $d_f$. Given $u:c_1^*\mF_1\to c_2^!\mF_2$ so that $(C,u)$ is a cohomological correspondence from $(X_1,\mF_1)$ to $(X_2,\mF_2)$, then we have a cohomological correspondence $(D, f^*(u))$ from $(Y_1, f_1^*\calF_1)$ to $(Y_2, f_2^!\calF_2\langle-2d_f\rangle)$ defined as follows:
\[
f^*(u)\colon d_1^*f_1^*\mF_1\cong f^*c_1^*\mF_1\xrightarrow{f^*u} f^*c_2^!\mF_2\cong f^!c_2^!\mF_2\langle -2d_f\rangle \cong d_2^!f_2^!\mF_2\langle -2d_f\rangle.\]
We call this the \emph{smooth pullback of the cohomological correspondence}. 

\begin{notation} 
\label{N: shifted pull back coh cor}
If $f$, $f_1$, and $f_2$ are all smooth and equidimensional of dimension $d = d_f = d_{f_1}=d_{f_2}$, then we shift $f^*\langle d\rangle$ to get the cohomological correspondence from $f_1^\star \mF_1$ to $f_2^\star \mF_2$, denoted by $f^\star(u)$, and called the \emph{shifted} smooth pullback.
\end{notation}

\begin{rmk}
There are another two situations one can pullback of cohomological correspondences. Namely, if the right inner square of the above diagram is Cartesian, then for a $(C,u)$ is a cohomological correspondence from $(X_1,\mF_1)$ to $(X_2,\mF_2)$, we can define
\[f^*(u)\colon d_1^*f_1^*\mF_1= f^*c_1^*\mF_1\xrightarrow{f^*u} f^*c_2^!\mF_2\xrightarrow{\on{BC}^{*!}} d_2^! f_2^*\mF_2.\]
which, by \eqref{E: base change shrek adjunction}, coincides with the above defined smooth pullback of the cohomological correspondence if in addition $f$ is cohomologically smooth.
Similarly, if the left inner square of the above diagram is Cartesian, then we have
\[f^!(u)\colon d_1^*f_1^!\mF_1\xrightarrow{\on{BC}^{*!}} f^!c_1^*\mF_1\xrightarrow{f^!u}f^!c_2^!\mF_2= d_2^!f_2^!\mF_2,\]
which coincides with the above defined smooth pullback of the cohomological correspondence if in addition $f$ is cohomologically smooth.
The following Lemma \ref{L: comp of pull coh corr} and Lemma \ref{AL:pullback compatible with composition} have similar counterparts in these situations. But we will not need them in this paper.
\end{rmk}

\begin{lem}
\label{L: comp of pull coh corr} 
Let
\[\begin{CD}
Z_1@<e_1<< E@>e_2>> Z_2\\
@Vf_1VV@VVfV@VVf_2V\\
Y_1@<d_1<<D@>d_2>>Y_2\\
@Vg_1VV@VVgV@VVg_2V\\
X_1@<c_1<<C@>c_2>>X_2
\end{CD}\]
be a commutative diagram, such that $f$ (resp. $g$) are cohomologically smooth, equidimensional of dimension $d_f$ (resp. $d_g$), and let $u:c_1^*\mF_1\to c_2^!\mF_2$ be a cohomological correspondence from $(X_1,\mF_1)$ to $(X_2,\mF_2)$. Then $$(gf)^*(u)= f^*(g^*(u))$$ as cohomological correspondences from $(g_1f_1)^*\mF_1\cong f_1^*(g_1^*\mF_1)$ to $(g_2f_2)^!\mF_2\langle -2d_{gf}\rangle\cong f_2^!(g_2^!\mF_2\langle -2d_g\rangle)\langle -2d_f\rangle$.
\end{lem}
\begin{proof}
Straightforward.
\end{proof}
Similar to pushforward cohomological correspondences, smooth pullback of cohomological correspondence is also compatible with composition of cohomological correspondences in certain situations.
\begin{lem}
\label{AL:pullback compatible with composition}
Suppose we have a commutative diagram
\[
\begin{CD}
Y_1@<d_1<<D@>d_2>>Y_2@<d'_1<<D'@>d'_2>>Y_3\\
@Vf_1VV@VVfV@VVf_2V@VVf'V@VVf_3V\\
X_1@<c_1<< C@>c_2>> X_2@<c'_1<< C'@>c'_2>> X_3
\end{CD}
\]
with $f$, $f_2$ cohomologically smooth equidimensional, of relative dimension $d_f$, $d_{f_2}$. Assume that the third inner square is Cartesian.
Set $ \widetilde C := C \times_{X_2} C'$ and $\widetilde D: = D \times_{Y_2} D'$, and let $\tilde f:\tilde D\to\tilde C$ be the natural map.
Let $u: c_1^*\calF_1 \to c_2^!\calF_2$ and $v: c'^*_1\calF_2 \to c'^!_2 \calF_3$ be cohomological correspondences from $(X_1, \calF_1)$ to $(X_2, \calF_2)$ and  from $(X_2, \calF_2)$ to $(X_3, \calF_3)$, respectively. Then cohomological correspondence $\tilde f^*(v\circ u): (Y_1, f_1^*\calF_1) \to (Y_3, f_3^!\calF_3\langle 2d_f\rangle )$ is equal to
\[
(Y_1, f_1^*\calF_1) \xrightarrow{f^*(u)} \big(Y_2, f_2^!\calF_2\langle 2d_f\rangle \cong f_2^*\calF_2\langle 2d_f-2d_{f_2}\rangle \big) \xrightarrow{f'^*(v)} \big(Y_3, f_3^!\calF_3\langle 2d_f\rangle \big).
\]
\end{lem}
\begin{rmk}
Varshavsky (\cite{varS=T}) proved a stronger statement, where instead of requiring that the third square is Cartesian, it imposes the assumption that both $f'$ and $\tilde f$ are equidimensional smooth of relative dimension $d_{f'}$ and $d_{\tilde f}$, and that $d_{\tilde f}+d_{f_2}=d_f+d_{f'}$. For the applications in this paper, the weaker statement is enough.
\end{rmk}
\begin{proof}
Since the third square is Cartesian,  $\widetilde D = D \times_{Y_2}D' \cong D \times_{X_2} C' = D \times_C C \times_{X_2}C' = D \times_C \widetilde C$. So $\tilde f: \widetilde D \to \widetilde C$ is smooth of relative dimension $d_{f_2}$. 
We write $p_C: \widetilde C \to C$, $q_C: \widetilde C \to C'$, $\tilde c_1: \widetilde C \to X_1$, $\tilde c_2: \widetilde C \to X_3$, $p_D: \widetilde D \to D$, $q_D: \widetilde D \to D'$, $\tilde d_1 : \widetilde D \to Y_1$, and $\tilde d_2: \widetilde D \to Y_3$ for the natural maps.
The decomposition of the cohomological correspondences is equivalently to the commutativity of the following diagram
\begin{tiny}
\[
\xymatrix@C=15pt{
p_D^* d_1^*f_1^* \calF_1  \ar@{=}[r] \ar@{=}[d] & \tilde f^*p_C^*c_1^* \calF \ar[r]^u & \tilde f^* p_C^*c_2^!\calF_2 \ar^\cong[r] & \tilde f^!p_C^*c_2^!\mF_2\langle-2d_{f}\rangle\ar[r]& \tilde f^!q_C^!c'^*_1 \calF_2\langle-2d_{f}\rangle \ar[r]^v & \tilde f^!q_C^!c'^!_2\calF_3\langle-2d_{f}\rangle  \ar@{=}[r] & \tilde d_2^! f_3^!\calF_3\langle-2d_{f}\rangle\ar@{=}[d]
\\
p^*_Df^*c_1^*\calF_1 \ar[r]^u & p^*_D f^*c_2^!\calF_2 \ar^\cong[r] \ar@{=}[ru] &  p_D^*f^!c_2^!\mF_2\langle-2d_f\rangle\ar^\cong[ur]\ar@{=}[r]&p^*_Dd_2^!f_2^!\calF_2\langle-2d_f\rangle\ar[r] & q_D^!d'^*_1f_2^!\calF_2\langle-2d_f\rangle \ar^\cong[r] & q_D^!f'^!c'^*_1\calF_2\langle-2d_{ f}\rangle \ar[r]^v\ar@{=}[lu] & q_D^!f'^!c'^!_2\calF_2\langle-2d_{f}\rangle.
}
\]\end{tiny}Only the commutativity of the middle trapezoid needs explanation. Note that in the following cubic commutative diagram, the top and the bottom faces are Cartesian, and all vertical maps are perfectly smooth
\[
\xymatrix@R=10pt@C=5pt{
& \tilde D\ar^{q_D}[rr]\ar'[d][dd]\ar_{p_D}[dl]
& & D' \ar^{f'}[dd]\ar[dl]
\\
D \ar[rr]\ar_f[dd]
& & Y_2 \ar[dd]
\\
& \tilde C \ar'[r][rr]\ar[dl]
& & C'\ar[dl]
\\
C \ar_{c_2}[rr]
& & X_2 
}
\]
Let $\phi = c_2\circ f = f_2\circ d_2: D \to X_2$ and $\psi = f'\circ q_D = q_C \circ \tilde f$. Then we need to show that the two maps
\[
p_D^*\phi^!  \xrightarrow{\cong} \tilde f^!p_C^*c_2^! \xrightarrow{\on{BC}^{*!}} \psi^!c'^*_1, \quad p_D^*\phi^! \xrightarrow{\on{BC}^{*!}} q_D^!d'^*_1f_2^!\xrightarrow{\cong} \psi^! c'^*_1
\]
are the same. Since the right face of the cube (i.e. the third inner square of the diagram in the lemma) is Cartesian, the isomorphisms $p_D^*\phi^!  \cong \tilde f^!p_C^*c_2^!$ and $q_D^!d'^*_1f_2^!\cong \psi^! c'^*_1$ coincide with the base change map $\on{BC}^{!*}$ (see \eqref{E: base change shrek adjunction}). Then 
the claims follows from the fact that the composition of the base change maps is the base change map. 
\end{proof}

Finally, the smooth pullback of cohomological correspondence is compatible with pushforward correspondence in the Cartesian case (the base change for cohomological correspondences).
\begin{lem}
\label{AL:pushforward pullback compatibility}
Assume that we have the following commutative diagram
\[
\xymatrix{
& X'_1 \ar[dl]_{\pi_{X_1}} \ar'[d]^(0.6){f'_1}[dd] && C' \ar[rr]^{c'_2}\ar[ld]_(0.6){\pi_C}\ar[ll]_{c'_1}\ar'[d]^(0.6){f'}[dd]
& & X'_2 \ar[dl]_(0.6){\pi_{X_2}} \ar[dd]^{f'_2}
\\
X_1 \ar[dd]_{f_1} & & C \ar[rr]_(0.6){c_2}\ar[ll]^(0.35){c_1}\ar[dd]^(0.3)f
& & X_2 \ar[dd]^(0.3){f_2}
\\
& Y'_1 \ar[dl]^(0.4){\pi_{Y_1}} && D' \ar'[r]_(0.75){d'_2}[rr]\ar[ld]^(0.4){\pi_D}\ar'[l][ll]^(0.3){d'_1}
& & Y'_2 \ar[dl]^{\pi_{Y_2}}
\\
Y_1 && D \ar[rr]^{d_2}\ar[ll]_{d_1}
& & Y_2 ,
}
\]
in which $f_1,f'_1,f,f'$ are representable by perfectly proper algebraic spaces, the middle sectional square with vertices $C, D, C', D'$ is Cartesian, and $\pi_C$ and $\pi_D$ are smooth of dimension $d$.
Let $u: c_1^*\calF_1 \to c_2^!\calF_2$ be a cohomological correspondence from $(X_1, \calF_1)$ to $(X_2, \calF_2)$.
Then we have the following commutative diagram of cohomological correspondences
\[
\xymatrix{
(Y'_1, \pi_{Y_1}^*f_{1,!}\calF_1) \ar[d] \ar[rr]^-{\pi_D^*f_!(u)} & & \big(Y'_2, \pi_{Y_2}^!f_{2,!}\calF_2\langle-2d\rangle\big)
\\
(Y'_1,f'_{1,!} \pi_{X_1}^*\calF_1) \ar[rr]^-{f'_!\pi_C^*(u)} & & \big(Y'_2,f'_{2,!} \pi_{X_2}^!\calF_2\langle-2d\rangle\big). \ar[u]
}
\]
\end{lem}
\begin{proof}
For simplicity, we write $\widetilde \calF_2$ for $\calF_2\langle-2d\rangle$. The condition on properness implies that both $f$ and $f'$ are proper. The lemma amounts to the commutativity of the following diagram
\begin{small}
\[
\xymatrix@C=15pt{ \ar[d]
d'^*_1\pi_{Y_1}^*f_{1,!}\calF_1 \ar@{=}[r] &  \pi_D^*d_1^*f_{1,!}\calF_1 \ar[r]\ar[d]  & \pi_D^*f_! c_1^*\calF_1 \ar[d]^\cong \ar[r]^u & \pi_D^*f_!c_2^!\calF_2 \ar^\cong[r]\ar[d]^\cong & \pi_D^!f_!c_2^!\widetilde \calF_2 \ar[r]  & \pi_D^!d_2^!f_{2,!}\widetilde \calF_2 \ar@{=}[r] & d'^!_2 \pi_{Y_2}^!f_{2,!}\widetilde \calF_2
\\
d'^*_1f'_{1,!}\pi_{X_1}^*\calF_1 \ar[r] &  f'_{!}c'^*_1\pi_{X_1}^*\calF_1 \ar@{=}[r] & f'_!\pi_C^*c_1^*\calF_1 \ar[r]^u & f'_!\pi_C^*c_2^!\calF_2 \ar^\cong[r] & f'_!\pi_C^!c_2^!\widetilde \calF_2 \ar@{=}[r] \ar[u]   & f'_!c'^!_2\pi_{X_2}^!\widetilde \calF_2 \ar[r] \ar[u] & d'^!_2f'_{2,!}\pi_{X_2}^! \widetilde \calF_2. \ar[u]
}
\]
\end{small}
The left two squares are commutative by Lemma \ref{L:proper base change} and the right two squares commutative by Remark \ref{R: another adjunction}.
The third square is clearly commutative, and the the commutativity of the fourth square follows from the smooth proper base change. 
\end{proof}

\subsubsection{Cycle class maps and cohomological correspondences}\label{cl vs cc}
Assume $k=\bar k$.
Let $f:X\to Y$ be a morphism of pfp algebraic spaces with $Y$ smooth. Set $d_X=\dim X$, $d_Y=\dim Y$, and $d=\dim Y-\dim X$. 
If $X$ is proper over $k$, we have the usual cycle class map as
\[\cl_X: \rmH^{\on{BM}}_{2d_X}(X)\to \rmH^{2d}_c(Y,\bQ_\ell(d)). \]
On the other hand, let $S$ be the set of irreducible components of $X$ of dimension $d_X$, and let
$\tilde X$ denote the disjoint union of irreducible components of $X$ of dimension $d_X$. Then there is a correspondence $S\leftarrow\tilde X\to Y$. 
By \eqref{E:corr and cycle}
\[
\on{Corr}_{\tilde X}\big((S,\bQ_\ell), (Y,\bQ_\ell\langle 2d\rangle)\big)=\rmH^{\on{BM}}_{2d_X}(\tilde X) \cong \rmH^{\on{BM}}_{2d_X}( X).
\]
In particular, the fundamental class $[\tilde X]\in \rmH^{\on{BM}}_{2d_X}(\tilde X)$ (which is the sum of the fundamental classes of each irreducible component of $\tilde X$) defines a cohomological correspondence from $(S,\bQ_\ell)\to (Y,\bQ_\ell[2d](d))$, denoted by $c_{\tilde X}$. If $X$ is proper, it induces
$$\on{H}(c_{\tilde X}): \on{H}^0(S,\bQ_\ell)\to  \rmH^{2d}_c(Y,\bQ_\ell(d)).$$
Note the $\on{H}^0(S,\bQ_\ell)$ is just the space of $\bQ_\ell$-valued functions on the set $S$, and
we have the natural identification $\on{H}^0(S,\bQ_\ell)=\rmH^{\on{BM}}_{2d_X}(X),\ f\mapsto \sum_{s\in S}f(s)[X_s]$, where $X_s$ is the irreducible component labelled by $s$. It follows from definition that under this identification,
\begin{equation}
\label{E: corr vs cycle class}
\cl_X=\on{H}(c_{\tilde X}).
\end{equation}

We also have the dual correspondence 
$$\bD c_{\tilde X}: \big(Y,\bQ_\ell\langle2d_X\rangle \big)\to (S,\bQ_\ell).$$ 
If the map $X\to Y$ is proper, it induces
$$ \rmH(\bD c_{\tilde X}): \rmH^{2d_X}_c(Y,\bQ_\ell(d_X))\to \on{H}^0(S,\bQ_\ell).$$ 

Now if $X$ and $X'$ are perfectly proper algebraic spaces and $Y$ is a separated perfectly smooth algebraic space such that $\dim X+\dim X'=\dim Y$, and if we have morphisms $X\to Y$ and $X'\to Y$, then the composition gives a map
\[\on{H}(\bD c_{\tilde {X'}})\circ \on{H}(c_{\tilde{X}}): \on{H}^0(S)\to \on{H}^0(S').\]
which can be represented by a kernel function $I: S\times S'\to \Ql$ such that 
$$\on{H}(\bD c_{\tilde{X'}})\circ \on{H}(c_{\tilde{X}})(f)(s')=\sum_{s\in S} I(s,s')f(s').$$
we will call 
$I$ the \emph{intersection number} of $X$ and $X'$ for the following reasons.
If $Y$ is a perfectly proper scheme, and $X_s$ and $X'_{s'}$ are closed subschemes of $Y$, then $\rmH(\bD c_{\tilde X'})$ is dual to $\rmH(c_{\tilde X'})$ under the Poincar\'e duality $\rmH^{2d_{X'}}_c(Y,\bQ_\ell(d_{X'}))=\rmH^{2d_{X'}}(Y,\bQ_\ell(d_{X'}))\cong \rmH^{2d_X}(Y,\bQ_\ell(d_X))$. It follows that $I(s,s')$ is the intersection number of $X_s$ and $X'_{s'}$ (as algebraic cycles on $Y$), and therefore is an integer. 
If $Y$ is not necessarily proper, but admits a perfectly an open embedding into a perfectly proper and smooth scheme $\overline{Y}$, and $X, X'$ are proper over $k$. Then the following diagram is commutative
\[\xymatrix{
\on{H}^0(S,\bQ_\ell)\ar^{\on{H}(c_{\tilde X})}[r]\ar[dr] & \rmH^{2d}_c(Y,\bQ_\ell(d))\ar^{\on{H}(\bD c_{\tilde X'})}[dr]\ar[d]&\\
&\on{H}^{2d}(\overline{Y},\bQ_\ell(d))\ar[r]& \on{H}^0(S',\bQ_\ell)
}\]
It follows that $I(s,s')$ is the intersection number of $X_s$ and $X'_{s'}$ in $\overline{Y}$. (Note however that this number is independent of the choice of the compactification $\overline{Y}$.)

\subsubsection{Trace formula} 
We will consider the following situation. Assume $k=\bar k$.
Let $c:=(c_1,c_2):C\to X\times X$ be a pfp self-correspondence of the pfp algebraic space $X$, and let $\on{Fix}(c):= X\times_{\Delta, X\times X}C$ be the fixed point of the correspondence.
Let $\mF\in D^b_c(X)$. We construct a map
\begin{equation}
\label{E: shifted trace map}
\on{Corr}_C((X,\mF\langle d\rangle),(X,\mF))\to \on{H}^{\on{BM}}_{d}(\on{Fix}(c)).
\end{equation}
as follows. Let $u: c_1^*\mF\langle d\rangle \to c_2^!\mF$ be a cohomological correspondence, which by Lemma \ref{L:sharp correspondence} can be regarded as $u^\sharp: (X\times X, \mF\boxtimes \bD\mF)\to (\on{pt},\Ql\langle -d\rangle)$. 
 Then the support of the correspondence
$$ (\on{pt},\Ql)\xrightarrow{\delta_\mF} (X\times X,\mF\boxtimes \bD\mF)\xrightarrow{u^\sharp}(\on{pt},\Ql\langle-d\rangle)$$
is $\on{Fix}(c)$ and by \eqref{E:corr and cycle} is given by a class of $\on{H}^{\on{BM}}_{d}(\on{Fix}(c))$.  Note that if $d=0$, \eqref{E: shifted trace map} recovers the trace map defined by Varshavsky \cite[\S 1.2.2]{Vas}.

\begin{rmk}
If $f:X\to X'$ is representable by perfectly proper algebraic spaces, then the sharp correspondence $u\leftrightarrow u^\sharp$ is compatible with the pushforward of cohomological correspondence. Then it follows from Lemma \ref{AL:pushforward compatible with composition} that \eqref{E: shifted trace map} is compatible with the pushforward along the proper morphism, giving the Lefschetz trace formula for correspondences.
\end{rmk}

The following lemma follows directly from  the purity.
\begin{lem}
\label{L: Shifted trace smooth}
Assume that $X,C$ are irreducible perfectly smooth, and $d=\dim C-\dim X$.
Assume that  $\on{Fix}(c)\subset C$ is perfectly smooth of codimension $d$, i.e. $C$ intersects with the diagonal properly smoothly. Assume that $u: c_1^*\Ql\langle 2d\rangle \to c_2^!\Ql$ is given by the fundamental class $[C]$. Then the image of $u$ under \eqref{E: shifted trace map} is the fundamental class $[\on{Fix}(c)]$.
\end{lem}

Another case we need is as follows, which is due to Braverman-Varshavsky \cite{BV}. 
\begin{lem}
\label{P: GL trace formula}
Let $(X,\mF)$ be a pair over $k=\bF_q$, with $\mF\in D_c^b(X)$. Let $u=\Ga_{\sigma_q}^*$ be the cohomological correspondence. Then $\delta_\mF\circ u^\sharp$ is supported on $X(\bF_q)$, given by the function on $X(\bF_q)$ (as in Example \ref{Ex:examples of correspondences} (5) or \eqref{E: shifted trace map}) defined as 
$$f_{\mF}(x)=\tr(\phi_x, \mF_{\bar x}).$$
\end{lem}

\end{document}